%% file: main.tex
\documentclass[10pt, a4paper, notitlepage]{article}
\pdfoutput=1

\usepackage{biblatex}
\input{packages.tex}
\usepackage{hyperref}
\input{style_general.tex}
\input{notation_general.tex}

\newcommand{\paperone}{\cite{Paper-I}}
\newcommand{\papertwoA}{\cite{Paper-IIA}}
\newcommand{\papertwoB}{\cite{Paper-IIB}}
\newcommand{\cA}{\mathcal{A}}

\title{Deformed Mirror Symmetry for Punctured Surfaces}
\author{Raf Bocklandt and Jasper van de Kreeke}
\date{15 March 2023}

\begin{document}

\maketitle
\begin{abstract}
Mirror symmetry originally envisions a correspondence between deformations of the A-side and deformations of the B-side. In this paper, we achieve an explicit correspondence in the case of punctured surfaces.

The starting point is the noncommutative mirror equivalence $ \Gtl Q ≅ \mf(\Jac \mirQ, ℓ) $ for a punctured surface $ Q $. We pick a deformation $ \Gtl_q Q $ which captures a large part of the deformation theory and includes the relative Fukaya category. To find the corresponding deformation of $ \mf(\Jac \mirQ, ℓ) $, we deform work of Cho-Hong-Lau which interprets mirror symmetry as Koszul duality. As result we explicitly obtain the corresponding deformation $ \mf(\Jac_q \mirQ, ℓ_q) $ together with a deformed mirror functor $ \Gtl_q Q \isoto \mf(\Jac_q \mirQ, ℓ_q) $.

The bottleneck is to verify that the algebra $ \Jac_q \mirQ $ is indeed a (flat) deformation of $ \Jac \mirQ $. We achieve this by deploying a result of Berger-Ginzburg-Taillefer on deformations of CY3 algebras, which however requires the relations to be homogeneous. We show how to replace this homogeneity requirement by a simple boundedness condition and obtain flatness of $ \Jac_q \mirQ $ for almost all $ Q $.

We finish the paper with examples, including a full treatment of the 3-punctured sphere and 4-punctured torus. With the help of our computations in \papertwoB, we describe $ \Jac_q \mirQ $ explicitly. It turns out that the deformed potential $ ℓ_q $ is still central in $ \Jac_q \mirQ $, in contrast to the popular slogan that central elements do not survive under deformation.
\end{abstract}

\tableofcontents

\newcommand{\paperthreeacknowledgements}{%
\subsection*{Acknowledgements}
This project was supervised by the first author and is part of the second author's PhD thesis. The second author thanks the participants of his 2020 Summer Camp on Derived Categories, Stability Conditions and Deformations. The project was supported by the NWO grant “Algebraic methods and structures in the theory of Frobenius manifolds and their applications” (TOP1.17.012).}

\input{filetree.tex}

\printbibliography

\end{document}

%% file: packages.tex
\usepackage{bbm}

\usepackage{etoolbox}
\usepackage{tikz}
\usetikzlibrary{calc}
\usetikzlibrary{cd}
\usetikzlibrary{decorations.markings}
\usetikzlibrary{decorations.pathreplacing}
\usetikzlibrary{decorations.pathmorphing}
\usetikzlibrary{decorations.text}
\usetikzlibrary{arrows.meta}
\usetikzlibrary{arrows}
\usetikzlibrary{positioning}

\usepackage{geometry} 
\usepackage{tocloft} 
\usepackage{fancyhdr} 
\usepackage{longtable}
\usepackage{subcaption}
\usepackage{enumitem}
\usepackage{amssymb}
\usepackage{amsmath}
\usepackage{stackrel} 
\usepackage{uniinput}
\usepackage{amsthm}
\usepackage{stmaryrd}
\usepackage{mathtools}
\usepackage{fouridx}
\usepackage{xparse}
\usepackage{bm}
\usepackage{aliascnt}
\usepackage{import}

%% file: style_general.tex
\theoremstyle{definition}
\newtheorem{theorem}{Theorem}
\let\emph\textbf

\newaliascnt{remark}{theorem}
\newaliascnt{definition}{theorem}
\newaliascnt{proposition}{theorem}
\newaliascnt{lemma}{theorem}
\newaliascnt{corollary}{theorem}
\newaliascnt{example}{theorem}
\newaliascnt{convention}{theorem}
\newaliascnt{todo}{theorem}

\newtheorem{remark}[remark]{Remark}
\newtheorem{definition}[definition]{Definition}
\newtheorem{proposition}[proposition]{Proposition}
\newtheorem{lemma}[lemma]{Lemma}
\newtheorem{corollary}[corollary]{Corollary}
\newtheorem{example}[example]{Example}
\newtheorem{convention}[convention]{Convention}

\aliascntresetthe{remark}
\aliascntresetthe{definition}
\aliascntresetthe{proposition}
\aliascntresetthe{lemma}
\aliascntresetthe{corollary}
\aliascntresetthe{example}
\aliascntresetthe{convention}
\aliascntresetthe{todo}

\numberwithin{equation}{section}
\numberwithin{figure}{section}
\numberwithin{table}{section}
\makeatletter\let\c@table\c@figure\makeatother

\numberwithin{theorem}{section}
\numberwithin{remark}{section}
\numberwithin{definition}{section}
\numberwithin{proposition}{section}
\numberwithin{lemma}{section}
\numberwithin{corollary}{section}
\numberwithin{example}{section}
\numberwithin{convention}{section}
\numberwithin{todo}{section}

\DeclareFieldFormat*{url}{}
\DeclareFieldFormat*{doi}{}
\DeclareFieldFormat*{issn}{}
\DeclareFieldFormat[misc]{url}{\mkbibacro{URL}\addcolon\space\url{#1}}
\DeclareFieldFormat[online]{url}{\mkbibacro{URL}\addcolon\space\url{#1}}

\setlist{topsep=0.5em, itemsep=0em}
\renewcommand{\arraystretch}{1.3}
\newcommand{\morearraystretch}{\renewcommand{\arraystretch}{1.8}}
\SetLabelAlign{parleftalign}{\parbox[t]\textwidth{#1\par\mbox{}}}


%% file: notation_general.tex
\newcommand{\Id}{\operatorname{Id}}
\newcommand{\id}{\mathrm{id}}
\newcommand{\coid}{\mathrm{id}^*}
\newcommand{\Hom}{\operatorname{Hom}}
\newcommand{\Ker}{\operatorname{Ker}}
\newcommand{\End}{\operatorname{End}}
\newcommand{\Ob}{\operatorname{Ob}}
\newcommand{\cat}[1]{\mathcal{#1}}

\newcommand{\htensor}{\widehat{\otimes}}
\newcommand{\pre}{\operatorname{pre}}

\newcommand{\Tw}{\operatorname{Tw}}
\newcommand{\Mod}{\operatorname{Mod}}
\newcommand{\Add}{\operatorname{Add}}
\newcommand{\MC}{\operatorname{MC}}

\newcommand{\Gtl}{\operatorname{Gtl}}
\newcommand{\Jac}{\operatorname{Jac}}

\newcommand{\wFuk}{\operatorname{wFuk}}

\newcommand{\D}{\operatorname{D}}
\newcommand{\MF}{\operatorname{MF}}
\def\H{\operatorname{H}}
\newcommand{\HTw}{\operatorname{H}\operatorname{Tw}}

\newcommand{\mf}{\operatorname{mf}}
\newcommand{\HH}{\operatorname{HH}}
\newcommand{\HC}{\operatorname{HC}}

\newcommand{\cyc}{\operatorname{cyc}}
\newcommand{\pmat}[1]{\begin{pmatrix} #1 \end{pmatrix}}
\newcommand{\vspan}{\operatorname{span}}
\newcommand{\odd}{\operatorname{odd}}
\newcommand{\even}{\operatorname{even}}
\newcommand{\isoto}{\xrightarrow{\sim}}
\newcommand{\verylongisoto}{\xrightarrow{~~~\sim~~~}}
\newcommand{\verylongmapsto}{\xmapsto{~~~\phantom{\sim}~~~}}
\newcommand{\verylongto}{\xrightarrow{~~~\phantom{\sim}~~~}}
\newcommand{\embeds}{\hookrightarrow}
\def\Im{\operatorname{Im}}
\newcommand{\smooth}[1]{\tilde{#1}}
\newcommand{\restr}[1]{|_{#1}}

\newcommand{\IDd}{\mathsf{Disk}_{\operatorname{ID}}}

\newcommand{\CRd}{\mathsf{Disk}_{\operatorname{CR}}}

\newcommand{\DSd}{\mathsf{Disk}_{\operatorname{DS}}}

\newcommand{\DWd}{\mathsf{Disk}_{\operatorname{DW}}}

\newcommand{\disktarget}{\operatorname{t}}
\newcommand{\Abouzaid}{\operatorname{Abou}}
\newcommand{\punctures}{\operatorname{Punc}}

\newcommand{\disjoint}{\mathbin{\dot{\cup}}}
\newcommand{\ZigzagCat}{\mathbb{L}}
\newcommand{\DefZigzagCat}{\mathbb{L}_q}
\newcommand{\ncpow}[1]{\langle\!\langle #1 \rangle\!\rangle}

\newcommand{\running}{~|~}
\newcommand{\mirQ}{\check{Q}}
\newcommand{\compl}[1]{\widehat{#1}}

\newcommand{\contraction}{c}

\newcommand{\opp}{\mathsf{opp}}
\newcommand{\Perf}{\operatorname{Perf}}

\newcommand{\landau}{\mathcal{O}}

\newcommand{\matf}[4]{\mathord{\begin{tikzpicture}[baseline={(0, -0.07)}] \path (0, 0) node[anchor=east] {$ #3: #1 $} (0.6, 0) node[anchor=west] {$ #2 : #4 $}; \path[draw, ->] (0, 0.07) -- (0.6, 0.07); \path[draw, ->] (0.6, -0.07) -- (0, -0.07); \end{tikzpicture}}}
\newcommand{\dmatf}[4]{\mathord{\begin{tikzpicture}[baseline={(0, -0.07)}] \path (0, 0) node[anchor=east] {$ #1 $} (0.6, 0) node[anchor=west] {$ #2 $}; \path[draw, ->] (0, 0.07) -- (0.6, 0.07) node[midway, above] {$ #3 $}; \path[draw, ->] (0.6, -0.07) -- (0, -0.07) node[midway, below] {$ #4 $}; \end{tikzpicture}}}
\newcommand{\chlQ}{Q^{\ZigzagCat}}
\newcommand{\RefObjects}{\mathbb{L}}
\newcommand{\DefRefObjects}{\mathbb{L}_q}
\newcommand{\projects}{\twoheadrightarrow}
\newcommand{\closure}[1]{\overline{#1}}
\newcommand{\Modfd}{\operatorname{Mod}^{\mathrm{fd}}}
\newcommand{\rModfd}{\operatorname{Mod}^{\mathrm{fd}}_{\mathrm{right}}}
\newcommand{\Ext}{\operatorname{Ext}}
\newcommand{\koszul}[1]{#1^!}
\newcommand{\kdual}[1]{{#1}^\vee}

\newcommand{\Ch}{\mathrm{Ch}}

\newcommand{\arcsequence}{\operatorname{Arcs}}

\newcommand{\Proj}{\operatorname{Proj}}
\newcommand{\mfconst}{O}
\newcommand{\CYn}{CY$ n $}
\newcommand{\Tr}{\operatorname{Tr}}

\newcommand{\algopp}{\operatorname{op}}
\newcommand{\oppalg}{\operatorname{op}}
\newcommand{\env}{\operatorname{e}}
\newcommand{\RHom}{\operatorname{RHom}}
\newcommand{\Barcon}{\operatorname{B}}
\newcommand{\vdimension}{\operatorname{dim}}
\newcommand{\Mdual}[1]{#1^{\mathrm{D}}}
\newcommand{\diag}{\operatorname{diag}}

\NewDocumentCommand{\tensor}{t_}
 {%
  \IfBooleanTF{#1}
   {\tensorop}
   {\otimes}%
 }
\NewDocumentCommand{\tensorop}{m}
 {%
  \mathbin{\mathop{\otimes}\displaylimits_{#1}}%
 }


%% file: filetree.tex
\input{intro.tex}

\input{prelim-ainfty/intro.tex}
\input{prelim-ainfty/ainfty.tex}
\input{prelim-ainfty/defo.tex}
\input{prelim-ainfty/submodules.tex}
\input{prelim-ainfty/freemodules.tex}
\input{prelim-ainfty/flatvariants.tex}

\input{prelim-koszul/intro.tex}
\input{prelim-koszul/bimodules.tex}
\input{prelim-koszul/koszul.tex}
\input{prelim-koszul/prop.tex}
\input{prelim-koszul/cy.tex}
\input{prelim-koszul/cy_ord.tex}
\input{prelim-koszul/cy3.tex}
\input{prelim-koszul/correspondence.tex}
\input{prelim-koszul/chl.tex}

\input{prelim-MS/intro.tex}
\input{prelim-MS/gtl.tex}
\input{prelim-MS/zigzag.tex}
\input{prelim-MS/matrixfact.tex}
\input{prelim-MS/mf.tex}
\input{prelim-MS/ms.tex}

\input{prelim-zigzag/intro.tex}

\input{prelim-zigzag/category.tex}
\input{prelim-zigzag/deformed.tex}
\input{prelim-zigzag/products.tex}
\input{prelim-zigzag/mirobjects.tex}

\input{flatness/intro.tex}
\input{flatness/flatness_intro.tex}
\input{flatness/BGstandard.tex}
\input{flatness/setup.tex}
\input{flatness/bounded_type.tex}
\input{flatness/ideals.tex}
\input{flatness/BGbounded.tex}
\input{flatness/flat_completed.tex}
\input{flatness/flat_algebraic.tex}
\input{flatness/closedness.tex}
\input{flatness/dimer_bounded.tex}
\input{flatness/main.tex}

\input{CHL/intro.tex}
\input{CHL/roadmap.tex}
\input{CHL/classical.tex}
\input{CHL/defLG.tex}
\input{CHL/projectives.tex}
\input{CHL/defMF.tex}
\input{CHL/functor.tex}

\input{MS/intro.tex}
\input{MS/CHL.tex}
\input{MS/midpoint.tex}
\input{MS/technical.tex}
\input{MS/defsuperpot.tex}
\input{MS/defpotential.tex}
\input{MS/mirobjects.tex}
\input{MS/main.tex}

\appendix
\input{examples/intro.tex}
\input{examples/sphere.tex}
\input{examples/torus.tex}

%% file: intro.tex
\section{Introduction}
Mirror symmetry is the quest for equivalences between Fukaya categories and categories of coherent sheaves. Noncommutative mirror symmetry is the quest for equivalences between Fukaya categories and categories of sheaves on noncommutative spaces. In the present paper, our starting point is noncommutative mirror symmetry for punctured surfaces:

\begin{center}
\begin{tikzpicture}
\path (0, 0) node[align=center] (A) {\textbf{Gentle algebra} \\ $ \Gtl Q $} (10, 0) node[align=center] (B) {\textbf{Matrix factorizations} \\ $ \mf(\Jac \mirQ, ℓ) $};
\path[draw, ->] (4, 0) -- (6, 0) node[midway, below] {\cite{Bocklandt}} node[midway, above] {\scalebox{1.5}{$ \sim $}};
\end{tikzpicture}
\end{center}

In the present paper, we pick one specific deformation $ \Gtl_q Q $ of $ \Gtl Q $ and find the corresponding deformation of $ \mf(\Jac \mirQ, ℓ) $. The result is a deformed mirror equivalence:

\begin{center}
\begin{tikzpicture}
\path (0, 0) node[align=center] (A) {\textbf{Deformed gentle algebra} \\ $ \Gtl_q Q $} (10, 0) node[align=center] (B) {\textbf{Deformed matrix factorizations} \\ $ \mf(\Jac_q \mirQ, ℓ_q) $};
\path[draw, ->] (4, 0) -- (6, 0) node[midway, above] {\scalebox{1.5}{$ \sim $}};
\end{tikzpicture}
\end{center}

In what follows, we explain our quest from different perspectives. We explain the philosophy of the specific deformation $ \Gtl_q Q $, comment on the source of mirror functors from the construction of Cho, Hong and Lau \cite{CHL}, and explain a bottleneck concerning the question whether the deformed Jacobi algebra $ \Jac_q \mirQ $ is indeed a deformation of $ \Jac \mirQ $.

\paragraph*{Deformation theory}
In $ A_∞ $-deformation theory one studies possible modifications of a given $ A_∞ $-structure which keep the $ A_∞ $-relations intact. One possible line of study consists of formal (infinitesimally) curved $ A_∞ $-deformations. The base ring for such deformations is a local algebra $ B $ with a few additional properties.

An interesting question arises when one is given two equivalent $ A_∞ $-categories $ \cat C $, $ \cat D $ and tries to transfer a deformation from $ \cat C $ to $ \cat D $. As a starting point, one is given a quasi-equivalence $ F: \cat C → \cat D $ and a deformation $ \cat D_q $. Transferring the deformation $ \cat C_q $ via $ F $ then entails finding a deformation $ \cat D_q $ of $ \cat D $ together with a deformation $ F_q $ of $ F $ such that $ F_q: \cat C_q → \cat D_q $ is a functor of deformed $ A_∞ $-categories.

The difficulty in transferring $ A_∞ $-deformations lies in the character of $ A_∞ $-theory. Indeed, both the $ A_∞ $-products of $ \cat C $ and $ \cat D $ and the functor $ F $ have higher components, which make it impossible to quickly to write down the corresponding deformation $ \cat D_q $.

Nevertheless, it is known that a transfer of deformations along quasi-equivalences always exists. The clue is to interpret $ A_∞ $-deformations of $ \cat C $ as Maurer-Cartan elements of the Hochschild DGLA $ \HC(\cat C) $. The quasi-equivalence $ F: \cat C → \cat D $ then gives rise to a non-canonical $ L_∞ $-morphism $ F_*: \HC(\cat C) → \HC(\cat D) $. By applying $ F_* $ to a given deformation $ \cat C_q $, viewed as Maurer-Cartan element, one obtains the corresponding deformation $ \cat D_q $. While this abstract interpretation does make $ \cat D_q $ computable, it sets the stage for the systematic quest of deformed mirror symmetry.

\paragraph*{Gentle algebras}
A popular way of modeling wrapped Fukaya categories of punctured surfaces deploys gentle algebras\cite{Bocklandt}. In this framework, one starts from an oriented closed surface $ S $ with a finite set of punctures $ M ⊂ S $. One chooses a system $ \cA $ of arcs which connect the punctures and divide the surface into polygons. To an arc system, one can associate a so-called gentle algebra $ \Gtl \cA $, reminiscent of the classical associative gentle algebras of \cite{Assem}. The gentle algebra $ \Gtl \cA $ is actually an $ A_∞ $-category whose higher products detect the topology of the punctured surface $ S \setminus M $. It was shown in \cite{Bocklandt} that it accurately models the wrapped Fukaya category of $ (S, M) $.

Dimers are specific kinds of arc systems which suit the purposes of mirror symmetry. We shall consider the specific mirror symmetry of punctured surfaces, built in \cite{Bocklandt}. This statement of mirror symmetry entails a quasi-isomorphism $ F: \Gtl Q → \mf(\Jac \mirQ, ℓ) $. The dimer $ \mirQ $ is the so-called dual dimer of $ Q $ and can be built from $ Q $ in a combinatorical way. In contrast, the mirror functor $ F $ itself is only given non-constructively and built in an inductive way by solving cocycle equations. Whenever we are given a deformation $ \Gtl_q Q $, it would be very hard to explicitly find the corresponding deformation of $ \mf(\Jac \mirQ, ℓ) $.

\paragraph*{Deformed Fukaya categories}
Seidel \cite{Seidel-relative} has introduced the idea of deforming Fukaya categories relative to a divisor. The idea is to introduce a formal parameter $ q $ and weight every pseudoholomorphic disk by $ q^s $ where $ s $ counts the number of intersections of the disk with the divisor. In \paperone, we have transported this concept to the world of gentle algebras. The result is a deformation $ \Gtl_q Q $, in which every puncture comes with its own deformation parameter. Whenever a disk covers the punctures $ q_1, …, q_k $, the contribution of this disk is weighted by the product $ q_1 … q_k $.

We raised the hope that our candidate deformation $ \Gtl_q Q $ would be the correct way to implement Seidel's idea on the side of gentle algebras. In \papertwoB, we examined this expectation and computed a part of the $ A_∞ $-structure of the derived category $ \HTw\Gtl_q Q $. Very specifically, it concerns the subcategory $ \H\DefZigzagCat ⊂ \HTw\Gtl_q Q $ given by the zigzag paths in $ Q $. The result is an explicit description of the $ A_∞ $-structure of $ \H\DefZigzagCat $ in terms of certain types of immersed disks. While it is hard to determine values for products of non-transversal sequences in the relative Fukaya category, our description of $ \H\DefZigzagCat $ determines their values very accurately. Although our calculation is limited to the zigzag paths, we consider \papertwoB\ a crude verification that $ \Gtl_q Q $ is the correct transport of Seidel's vision to gentle algebras and can be considered a “relative wrapped Fukaya category”.

Relative Fukaya categories have already served as A-side of mirror symmetry before. For instance, Lekili and Perutz \cite{Lekili-Perutz} find a commutative mirror for the relative Fukaya category of the $ 1 $-punctured torus, apparently the first use of a relative Fukaya category in mirror symmetry. In \cite{Lekili-Polishchuk}, Lekili and Polishchuk generalize this result to the case of the $ n $-punctured torus. They depart from a finite collection of split-generators of the Fukaya category and compute part of their deformed products in the relative Fukaya category. Their mirror is then obtained by guessing the correct deformation on the B-side. Complete knowledge of the products in the relative Fukaya category or even a relative wrapped Fukaya category are not required in their approach.

\paragraph*{Mirror functors}
A rich source of mirror functors is the recent construction of Cho, Hong and Lau \cite{CHL}. Their construction associates to a given $ A_∞ $-category $ \cat C $ with a suitable subcategory $ \RefObjects ⊂ \cat C $ a Landau-Ginzburg model $ (\Jac Q^{\RefObjects}, ℓ) $ together with an $ A_∞ $-functor
\begin{equation*}
F: \cat C → \MF(\Jac \chlQ, ℓ).
\end{equation*}
The Cho-Hong-Lau construction can be applied to the category $ \cat C = \HTw\Gtl Q $ by choosing $ \RefObjects $ to be the subcategory given by so-called zigzag paths. This application yields back the original mirror symmetry for punctured surfaces $ \Gtl Q \cong \mf(\Jac \mirQ, ℓ) $ from \cite{Bocklandt}.

In the present paper, it is our aim to produce a deformation of $ \mf(\Jac Q, ℓ) $ which corresponds to $ \Gtl_q Q $. Thanks to the Cho-Hong-Lau construction, this task becomes straightforward: The first step is to deform the Cho-Hong-Lau construction. The result is a procedure which generates mirror functors of the kind $ \cat C_q → \MF(\Jac_q \chlQ, ℓ_q) $. The second step is to apply this deformed construction to the case of $ \cat C_q = \HTw\Gtl_q Q $. The result is a quasi-isomorphism of deformed $ A_∞ $-categories
\begin{equation*}
F_q: \Gtl_q Q \isoto \mf(\Jac_q \mirQ, ℓ_q).
\end{equation*}
In particular, the deformation $ \mf(\Jac_q \mirQ, ℓ_q) $ is the desired deformation of $ \mf(\Jac \mirQ, ℓ) $ corresponding to $ \Gtl_q Q $. It is possible to describe the algebra $ \Jac_q \mirQ $, the potential $ ℓ_q $ and the mirror objects $ F_q (a) $ explicitly. This requires heavy computations in the minimal model $ \HTw\Gtl_q Q $ which we have performed in \papertwoB. Thanks to these earlier computations, we offer in the present paper an explicit description of $ \Jac_q \mirQ $, $ ℓ_q $ and $ F_q (a) $ in terms of combinatorics in $ Q $.

\paragraph*{Flatness of superpotential deformations}
A bottleneck in this paper is the question whether $ \Jac_q \mirQ $ is a (flat) deformation of $ \Jac \mirQ $ as an algebra. Indeed, our deformed Cho-Hong-Lau construction leads to a definition of $ \Jac_q \mirQ $ as a mere quotient of $ ℂ\mirQ ⟦Q_0⟧ $ by certain deformed relations. A quotient by deformed relations however need not be an algebra deformation in general. The question is whether the specific case of $ \Jac_q \mirQ $ is a deformation of $ \Jac \mirQ $ nevertheless. To resolve this question, we prove a flatness result for superpotential deformations of CY3 algebras.

Our flatness result is a culmination of a long sequence of improvements in the literature. Our starting point is the work of Berger, Ginzburg and Taillefer \cite{Berger-Ginzburg, Berger-Taillefer} which concerns PBW deformations of CY3 algebras. Like all previous results, their work requires the superpotential $ W $ to be homogeneous. We translate their work to the setting of formal deformations and show that the homogeneity condition is superfluous and can be replaced by a mild boundedness condition. We obtain a flatness result for formal deformations of CY3 algebras with nonhomogeneous superpotential. In particular, it follows from this result that $ \Jac_q \mirQ $ is a flat deformation of $ \Jac \mirQ $ for almost all dimers $ \mirQ $.

Ultimately, our flatness result renders the category $ \mf(\Jac_q \mirQ, ℓ_q) $ a deformation of $ \mf(\Jac \mirQ, ℓ) $ and $ F_q $ an equivalence of deformations. This proves noncommutative mirror symmetry for punctured surfaces.

\paragraph*{Assembling deformed mirror symmetry}
The present paper is the final one in a series of three. We explain here the purpose of this series, which results have been obtained in the first papers and how we build on them in the present paper.

Our original motivation was to transport Seidel's idea of relative Fukaya categories to the world of gentle algebras and to use it as A-side in a deformed mirror symmetry for punctured surfaces. We realized that an effective way of constructing the deformation of $ \mf(\Jac \mirQ, ℓ) $ was to follow the Cho-Hong-Lau construction, inserting the deformation $ \Gtl_q Q $ as an input instead of $ \Gtl Q $. This approach requires an explicit and lengthy minimal model calculation, which is the reason we distributed the material into a series of three papers.

The first paper in the series is \paperone\ and concerns the deformation theory of the gentle algebras $ \Gtl \cA $ under certain assumptions on the arc system $ \cA $. One of the main results is a complete classification of the deformations of $ \Gtl \cA $.

The second paper in the series is \papertwoB\ and conducts all the necessary computations for applying the Cho-Hong-Lau construction. It focuses on the case of the specific deformation $ \Gtl_q Q $ and defines the category of deformed zigzag paths $ \DefZigzagCat $. By means of the deformed Kadeishvili construction from our auxiliary paper \papertwoA, it builds a minimal model $ \H\DefZigzagCat $ for $ \DefZigzagCat $. The main result is an explicit description of the deformed $ A_∞ $-structure of $ \H\DefZigzagCat $ in terms of four types of immersed disks. These four types are labeled CR, ID, DS and DW disks and they agree precisely with the immersed disks one expects form the relative Fukaya category.

In the present paper, we tie the previous calculations together. We start by deforming the Cho-Hong-Lau construction in general. Then, we apply this deformed construction to the special case of $ \Gtl_q Q $ and obtain a deformed Jacobi algebra $ \Jac_q \mirQ $ and a deformed central element $ ℓ_q $. Thanks to the second paper, we have an explicit description of the $ A_∞ $-structure on $ \H\DefZigzagCat $, giving an explicit description of $ \Jac_q \mirQ $ and $ ℓ_q $. This description is theoretically given in terms of the CR, ID, DS and DW disks from \papertwoB, but simplifies a bit in the present paper because mostly products of transversal sequences are regarded. Apart from proving that $ \Jac_q \mirQ $ is indeed a deformation of $ \Jac \mirQ $, simply plugging in the results of \papertwoB\ already finishes deformed mirror symmetry.

\paragraph*{Structure of the paper}
In \autoref{sec:3ainfty}, we review $ A_∞ $-categories and their deformations. We also introduce notation and terminology for treating algebra deformations, including the $ \mathfrak{m} $-adic topology and flatness conditions. In \autoref{sec:koszul}, we present Koszul duality and the relationship between cyclic $ A_∞ $-algebras and Calabi-Yau dg algebras. We show how to tweak Koszul duality in order to obtain $ A_∞ $-functors similar to the Cho-Hong-Lau construction. In \autoref{sec:3prelim-MS}, we review dimers, gentle algebras and mirror symmetry of punctured surfaces. In \autoref{sec:zigzag}, we review the definition of the category of deformed zigzag paths $ \DefZigzagCat $ and description of its minimal model $ \H\DefZigzagCat $ from \papertwoB. In \autoref{sec:flatness}, we investigate deformations of Jacobi algebras given by deformations of the superpotential. We also consider the specific case $ \Jac \mirQ $ of Jacobi algebras of dimers. In \autoref{sec:CHL}, we motivate and review the Cho-Hong-Lau construction. We provide an explicit deformed construction and resolve a few technicalities. In \autoref{sec:MS}, we apply the deformed Cho-Hong-Lau construction to the specific case of $ \Gtl_q Q $. We provide explicit descriptions of the deformed Jacobi algebra $ \Jac_q \mirQ $, the central element $ ℓ_q $ and the deformed matrix factorizations $ F_q (a) $. In \autoref{sec:3examples}, we work out deformed mirror symmetry for the examples of the 3-punctured sphere and a 4-punctured torus.

\paperthreeacknowledgements

%% file: prelim-ainfty/intro.tex
\section{Preliminaries on $ A_∞ $-categories}
\label{sec:3ainfty}
In this section, we recollect background material on $ A_∞ $-categories and fix notation. In \autoref{sec:3ainfty-ainfty}, we recall $ A_∞ $-categories, their functors, twisted completion and minimal models. In \autoref{sec:3ainfty-defo}, we recall completed tensor products, deformations of $ A_∞ $-categories and their functors. We very briefly comment on the construction of twisted completion and minimal model for $ A_∞ $-deformations from \papertwoA. In \autoref{sec:prelim-submodules}, we introduce specific terminology and properties for submodules of $ B \htensor X $ as preparation for \autoref{sec:flatness}. In \autoref{sec:3ainfty-free}, we recall $ \mathfrak{m} $-adically free modules. In \autoref{sec:3ainfty-variants}, we examine variants of our flatness condition for ideals.

%% file: prelim-ainfty/ainfty.tex
\subsection{$ A_∞ $-categories}
\label{sec:3ainfty-ainfty}
In this section we recall $ A_∞ $-categories, completed tensor products, $ A_∞ $-deformations and functors between $ A_∞ $-deformations. The material is standard and can for instance be found in \cite{Bocklandt-book}. Throughout we work over an algebraically closed field of characteristic zero and write $ ℂ $.

\begin{definition}
A ($ ℤ $- or $ ℤ/2ℤ $-graded, strictly unital) \emph{$ A_∞ $-category} $ \cat C $ consists of a collection of objects together with $ ℤ $- or $ ℤ/2ℤ $-graded hom spaces $ \Hom(X, Y) $, distinguished identity morphisms $ \id_X ∈ \Hom^0(X, X) $ for all $ X ∈ \cat C $, together with multilinear higher products
\begin{equation*}
μ^k: \Hom(X_k, X_{k+1}) ¤ … ¤ \Hom(X_1, X_2) → \Hom(X_1, X_{k+1}), \quad k ≥ 1
\end{equation*}
of degree $ 2-k $ such that the $ A_∞ $-relations and strict unitality axioms hold: For every compatible morphisms $ a_1, …, a_k $ we have
\begin{align*}
& \sum_{0 ≤ j < i ≤ k} (-1)^{‖a_n‖ + … + ‖a_1‖} μ(a_k, …, μ(a_i, …, a_{j+1}), a_j, …, a_1) = 0, \\
& μ^2 (a, \id_X) = a, ~ μ^2 (\id_Y, a) = (-1)^{|a|} a, ~ μ^{≥3} (…, \id_X, …) = 0.
\end{align*}
\end{definition}


Next we recall the additive completion $ \Add\cat C $ of an $ A_∞ $-category $ \cat C $. This category consists of formal sums of shifted objects. The hom space between two objects consists of matrices of morphisms between the summands.

\begin{definition}
\label{def:3ainfty-ainfty-Add}
Let $ \cat C $ be an $ A_∞ $ category with product $ μ_{\cat C} $. The additive completion $ \Add \cat C $ of $ \cat C $ is the category of formal sums of shifted objects of $ \cat C $:
\begin{equation*}
A_1 [k_1] ⊕ … ⊕ A_n [k_n].
\end{equation*}
The hom space between two such objects $ X = \bigoplus A_i [k_i] $ and $ Y = \bigoplus B_i [m_i] $ is
\begin{equation*}
\Hom_{\Add\cat C} (X, Y) = \bigoplus_{i, j} \Hom_{\cat C} (A_i, B_j) [m_j - k_i].
\end{equation*}
Here $ [-] $ denotes the right-shift. The products on $ \Add\cat C $ are given by multilinear extensions of
\begin{equation*}
μ_{\Add \cat C}^k (a_k, …, a_1) = (-1)^{\sum_{j < i} \Vert a_i \Vert l_j} μ_{\cat C}^k (a_k, …, a_1).
\end{equation*}
Here each $ a_i $ lies in some $ \Hom(X_i[k_i], X_{i+1} [k_{i+1}]) $. The integer $ l_i $ denotes the difference $ k_{i+1} - k_i $ between the shifts and the degree $ ‖a_i‖ $ is the degree of $ a_i $ as element of $ \Hom_{\cat C} (X_i, X_{i+1}) $.
\end{definition}

Next we recall the twisted completion $ \Tw\cat C $ of an $ A_∞ $-category $ \cat C $. The objects of this category are virtual chain complexes of objects of $ \cat C $:

\begin{definition}
\label{def:3ainfty-ainfty-Tw}
A \emph{twisted complex} in $ \cat C $ is an object $ X ∈ \Add\cat C $ together with a morphism $ δ ∈ \Hom^1_{\Add\cat C} (X, X) $ of degree $ 1 $ such that $ δ $ is strictly upper triangular and satisfies the Maurer-Cartan equation:
\begin{equation*}
\MC(δ) ≔ μ^1 (δ) + μ^2 (δ, δ) + … = 0.
\end{equation*}
We may refer to the morphism $ δ $ as the \emph{twisted differential}. Note that the upper triangularity ensures that this sum is well-defined. The \emph{twisted completion} of $ \cat C $ is the $ A_∞ $-category $ \Tw\cat C $ whose objects are twisted complexes. Its hom spaces are the same as for the additive completion:
\begin{equation*}
\Hom_{\Tw\cat C} (X, Y) = \Hom_{\Add\cat C} (X, Y).
\end{equation*}
The products on $ \Tw \cat C $ of $ \cat C $ are given by embracing with $ δ $'s:
\begin{equation*}
μ_{\Tw \cat C}^k (a_k, …, a_1) = \sum_{n_0, …, n_k ≥ 0} μ_{\Add \cat C} (\underbrace{δ, …, δ}_{n_k}, a_k, …, a_1, \underbrace{δ, …, δ}_{n_0}).
\end{equation*}
\end{definition}


A functor between two $ A_∞ $-categories is a mapping which matches the products of the two categories:

\begin{definition}
Let $ \cat C $ and $ \cat D $ be $ A_∞ $-categories. Then a \emph{functor} $ F: \cat C → \cat D $ of $ A_∞ $-categories consists of a map $ F: \Ob(\cat C) → \Ob(\cat D) $ together with for every $ k ≥ 1 $ a degree $ 1-k $ multilinear map
\begin{equation*}
F^k: \Hom_{\cat C} (X_k, X_{k+1}) ¤ … ¤ \Hom_{\cat C} (X_1, X_2) → \Hom_{\cat C} (FX_1, FX_{k+1})
\end{equation*}
such that the $ A_∞ $-functor relations hold:
\begin{multline*}
\sum_{0 ≤ j < i ≤ k} (-1)^{‖a_j‖ + … + ‖a_1‖} F(a_k, …, a_{i+1}, μ(a_i, …, a_{j+1}), a_j, …, a_1) \\
= \sum_{\substack{l ≥ 0 \\ 1 = j_1 < … < j_l ≤ k}} μ(F(a_k, …, a_{j_l}), …, F(…, a_{j_2}), F(…, a_{j_1})).
\end{multline*}
The functor $ F $ is an \emph{isomorphism} if $ F: \Ob(\cat C) → \Ob(\cat D) $ is a bijection and $ F^1: \Hom_{\cat C} (X, Y) → \Hom_{\cat D} (FX, FY) $ is an isomorphism for all $ X, Y ∈ \cat C $. The functor $ F $ is a \emph{quasi-isomorphism} if $ F: \Ob(\cat C) → \Ob(\cat D) $ is a bijection and $ F^1: \Hom_{\cat C} (X, Y) → \Hom_{\cat D} (FX, FY) $ is a quasi-isomorphism of complexes for every $ X, Y ∈ \cat C $.
\end{definition}

\begin{definition}
When $ F: \cat C → \cat D $ and $ G: \cat D → \cat E $ are $ A_∞ $-functors, then their composition is given by $ GF: \Ob(\cat C) → \Ob(\cat E) $ on objects and
\begin{equation*}
(GF)(a_k, …, a_1) = \sum G(F(a_k, …), …, F(…, a_1)).
\end{equation*}
\end{definition}


Let us recall minimal models and their notation as follows:

\begin{definition}
An $ A_∞ $-category $ \cat C $ is \emph{minimal} if $ μ^1_{\cat C} = 0 $. A \emph{minimal model} of $ \cat C $ is any minimal $ A_∞ $-category $ \cat D $ together with a quasi-isomorphism $ F: \cat D → \cat C $. A minimal model of $ \cat C $ is generically denoted $ \H\cat C $.
\end{definition}

By the famous Kadeishvili theorem, every $ A_∞ $-category has a minimal model. In fact, a minimal model can be constructed semi-explicitly by sums over trees.

%% file: prelim-ainfty/defo.tex
\subsection{Deformations of $ A_∞ $-categories}
\label{sec:3ainfty-defo}
In this section, we recall deformations of $ A_∞ $-categories. We follow \papertwoA\ where also more detail can be found. We start by recalling completed tensor products. Then we recall $ A_∞ $-deformations and their functors. We comment very briefly on the construction of the twisted completion and minimal models for $ A_∞ $-deformations from \papertwoA.

We recall now completed tensor products $ B \htensor X $ with $ B $ a local ring and $ X $ a vector space. The letter $ B $ will always denote a local ring with extra properties. We have decided to give this a name:

\begin{definition}
A \emph{deformation base} is a complete local Noetherian unital $ ℂ $-algebra $ B $ with residue field $ B/\mathfrak{m} = ℂ $. The maximal ideal is always denoted $ \mathfrak{m} $.
\end{definition}

\begin{remark}
By the Cohen structure theorem, every deformation base is of the form $ ℂ⟦x_1, …, x_n⟧ / I $ with $ I $ denoting some ideal.
\end{remark}

If $ X $ is a vector spaces, then $ B \htensor X = \lim (B/\mathfrak{m}^k \tensor X) $ denotes the completed tensor product over $ ℂ $. For simplicity, we write $ \mathfrak{m}^k X $ to denote the infinitesimal part $ \mathfrak{m}^k X = \mathfrak{m}^k \htensor X ⊂ B \htensor X $. Recall that $ B \htensor X $ is a $ B $-module and comes with the $ \mathfrak{m} $-adic topology, which turns $ B \htensor X $ into a sequential Hausdorff space. For convenience, we may from time to time use expressions like $ x = \landau(\mathfrak{m}^k) $ to indicate $ x ∈ \mathfrak{m}^k X $.

\begin{definition}
A map $ φ: B \htensor X → B \htensor Y $ is \emph{continuous} if it is continuous with respect to the $ \mathfrak{m} $-adic topologies. A map $ φ: (B \htensor X_k) ¤ … ¤ (B \htensor X_1) → B \htensor Y $ is \emph{continuous} if for every $ 1 ≤ i ≤ k $ and every sequence of elements $ x_1, …, \hat x_i, …, x_k $ the map
\begin{equation*}
μ(x_k, …, -, …, x_1): B \htensor X_i → B \htensor Y
\end{equation*}
is continuous.
\end{definition}

\begin{remark}
Every element in $ B \htensor X $ can be written as a series $ \sum_{i = 0}^∞ m_i x_i $. Here $ m_i $ is a sequence of elements $ m_i ∈ \mathfrak{m}^{→∞} $ and $ x_i $ is a sequence of elements $ x_i ∈ X $. We have used the notation $ m_i ∈ \mathfrak{m}^{→∞} $ to indicate that $ m_i ∈ \mathfrak{m}^{k_i} $ for some sequence $ (k_i) ⊂ ℕ $ with $ k_i → ∞ $.
\end{remark}

\begin{remark}
\label{rem:prelim-ainfty-continuityautomatic}
Every $ B $-linear map $ B \htensor X → B \htensor Y $ is automatically continuous (see \papertwoA~for the argument), so is every every $ B $-multilinear map $ (B \htensor X_k) ¤ … ¤ (B \htensor X_1) → B \htensor Y $. Linear maps $ X → B \htensor Y $ can be uniquely extended to $ B $-linear maps $ B \htensor X → B \htensor Y $ and multilinear maps $ X_k ¤ … ¤ X_1 → B \htensor Y $ can be uniquely extended to $ B $-multilinear maps $ (B \htensor X_k) ¤ … ¤ (B \htensor X_1) → B \htensor Y $ (see \papertwoA).
\end{remark}

\begin{remark}
\label{rem:prelim-ainfty-leadingterm}
The \emph{leading term} of a $ B $-linear map $ φ: B \htensor X → B \htensor Y $ is the map $ φ_0: X → Y $ given by the composition $ φ_0 = π φ \restr{X} $, where $ π: B \htensor Y → Y $ denotes the standard projection. If the leading term $ φ_0 $ is injective or surjective, then $ φ $ is injective or surjective itself (see \papertwoA~for the argument).
\end{remark}

We recall now $ A_∞ $-deformations. When $ \cat C $ is an $ A_∞ $-category, the idea is to model its $ A_∞ $-deformations on the collection of enlarged hom spaces $ \{B \htensor \Hom_{\cat C} (X, Y)\}_{X, Y ∈ \cat C} $. Any $ B $-multilinear product on these hom spaces is automatically continuous. Similarly, functors of $ A_∞ $-deformations will be defined as maps between tensor products of the enlarged hom spaces and will be automatically continuous as well.

$ A_\infty $-deformations of $ \cat C $ will always be allowed to have infinitesimal curvature. The reason is that only this way we get a homologically sensible notion: Whenever $ μ_q $ is an (infinitesimally) curved deformation, then $ ν = μ - μ_q $ is a Maurer-Cartan element of the Hochschild DGLA $ \HC(\cat C) $. We comment on this in more detail in \papertwoA.

\begin{definition}
Let $ \cat C $ be an $ A_∞ $ category with products $ μ $ and $ B $ a deformation base. An \emph{$ A_\infty $-deformation} of $ \cat C_q $ of $ \cat C $ consists of
\begin{itemize}
\item The same objects as $ \cat C $,
\item Hom spaces $ \Hom_{\cat C_q} (X, Y) = B \htensor \Hom_{\cat C} (X, Y) $ for $ X, Y ∈ \cat C $,
\item $ B $-multilinear products of degree $ 2 - k $
\begin{equation*}
μ_q^k: \Hom_{\cat C_q} (X_k, X_{k+1}) \tensor … \tensor \Hom_{\cat C_q} (X_1, X_2) → \Hom_{\cat C_q} (X_1, X_{k+1}), ~ k ≥ 1
\end{equation*}
\item Curvature of degree $ 2 $ for every object $ X ∈ \cat C $
\begin{equation*}
μ_{q, X}^0 ∈ \mathfrak{m} \Hom_{\cat C_q}^2 (X, X),
\end{equation*}
\end{itemize}
such that $ μ_q $ reduces to $ μ $ once the maximal ideal $ \mathfrak{m} $ is divided out, and $ μ_q $ satisfies the curved $ A_∞ $ ($ cA_∞ $) relations
\begin{equation*}
\sum_{k ≥ l ≥ m ≥ 0} (-1)^{‖a_m‖ + … + ‖a_1‖} μ_q (a_k, …, μ_q (a_l, …), a_m, …, a_1) = 0.
\end{equation*}
The deformation is \emph{unital} if the deformed higher products still satisfy the unitality axioms
\begin{equation*}
μ^2_q (a, \id_X) = a, ~ μ^2_q (\id_Y, a) = (-1)^{|a|} a, ~ μ^{≥3}_q (…, \id_X, …) = 0.
\end{equation*}
\end{definition}

It sometimes comes handy to work with deformations that include more objects than $ \cat C $ does. We fix terminology as follows:

\begin{definition}
Let $ \cat C $ be an $ A_∞ $-category. Let $ O $ be an arbitrary set of objects and $ F: O → \Ob(\cat C) $ a map. Then the \emph{object-cloned version} $ F^* \cat C $ of $ \cat C $ is the $ A_∞ $-category given by object set $ O $, hom spaces
\begin{equation*}
\Hom_{F^* \cat C} (X, Y) = \Hom_{\cat C} (F(X), F(Y)), \quad X, Y ∈ O,
\end{equation*}
and products simply given by the same composition as in $ \cat C $.
\end{definition}

An $ A_∞ $-deformation of $ \cat C $ always gives an induced deformation of $ F^* \cat C $. This provides a map of Maurer-Cartan elements $ \MC(\HC(\cat C), B) → \MC(\HC(F^* \cat C), B) $. In case $ F $ is surjective, the categories $ \cat C $ and $ F^* \cat C $ are equivalent and the map of Maurer-Cartan elements becomes a bijection after dividing out gauge equivalence. However, the map on raw Maurer-Cartan elements is not a bijection itself. After these comments, we are ready for the following terminology:

\begin{definition}
\label{def:prelim-ainfty-objectcloning}
Let $ \cat C $ be an $ A_∞ $-category and $ B $ a deformation base. Let $ O $ be an arbitrary set and $ F: O → \Ob\cat C $ a map. An \emph{object-cloning deformation} is a deformation $ \cat D_q $ of $ \cat D = F^* \cat C $. The object-cloning deformation is \emph{essentially surjective} if $ F: \Ob\cat D → \Ob\cat C $ reaches all objects of $ \cat C $ up to isomorphism.
\end{definition}

We are now ready to explain the natural extension of $ A_∞ $-functors to the deformed case.

\begin{definition}
Let $ \cat C, \cat D $ be two $ A_∞ $-categories and $ \cat C_q, \cat D_q $ deformations. A \emph{functor of deformed $ A_∞ $-categories} consists of a map $ F_q: \Ob(\cat C) → \Ob(\cat D) $ together with for every $ k ≥ 1 $ a $ B $-multilinear degree $ 1-k $ map
\begin{equation*}
F^k_q: \Hom_{\cat C_q} (X_k, X_{k+1}) ¤ … ¤ \Hom_{\cat C_q} (X_1, X_2) → \Hom_{\cat D_q} (F_q X_1, F_q X_{k+1})
\end{equation*}
and infinitesimal curvature $ F^0_{q, X} ∈ \mathfrak{m} \Hom^1_{\cat D} (F_q X, F_q X) $ for every $ X ∈ \cat C $, such that the curved $ A_∞ $-functor relations hold:
\begin{multline*}
\sum_{0 ≤ j ≤ i ≤ k} (-1)^{‖a_j‖ + … + ‖a_1‖} F_q (a_k, …, a_{i+1}, μ_q (a_i, …, a_{j+1}), a_j, …, a_1) \\
= \sum_{\substack{l ≥ 0 \\ 1 = j_1 < … < j_l ≤ k}} μ_q (F_q (a_k, …, a_{j_l}), …, F_q (…, a_{j_2}), F_q (…, a_{j_1})).
\end{multline*}
If $ \cat C_q $ and $ \cat D_q $ are strictly unital, then we say $ F_q $ is strictly unital if $ F_q^1 (\id_X) = \id_{F_q X} $ for every $ X \in \cat C $ and $ F_q^{≥2} (…, \id_X, …) = 0 $.
\end{definition}

\begin{remark}
Note that the functor $ F_q $ itself is allowed to have a curvature component. The first two curved $ A_∞ $-functor relations read
\begin{align*}
F^0_q + F^1_q (μ^0_{\cat C_q, X}) &= μ^1_{\cat D_q} (F^0_{q, X}), \\
F^1_q (μ^1_{\cat C_q} (a)) + (-1)^{‖a‖} F^2_q (μ^0_{\cat C_q, Y}, a) + F^2_q (a, μ^0_{\cat C_q, X}) &= μ^1_{\cat D_q} (F^1_q (a)) + μ^2_{\cat D_q} (F^0_{q, Y}, F^1_q (a)) \\
& \quad + μ^2_{\cat D_q} (F^1_q (a), F^0_{q, X}), \quad ∀ ~ a: X → Y.
\end{align*}
If $ F_q: \cat C_q \to \cat D_q $ is a functor of $ A_\infty $-deformations, then its leading term $ F: \cat C \to \cat D $ is automatically a functor of $ A_\infty $-categories.
\end{remark}


Twisted completions of $ A_∞ $-deformations exist. When $ \cat C_q $ is a deformation of $ \cat C $, we can form a twisted completion $ \Tw\cat C_q $, as we have elaborated in \papertwoA. This category $ \Tw\cat C_q $ is a deformation of $ \Tw\cat C $. The objects of $ \Tw\cat C $ are defined in terms of twisted differentials as well, but the twisted differentials do not satisfy the Maurer-Cartan equation with respect to the deformed product $ μ_{\cat C_q} $. Instead, the failure to satisfy the Maurer-Cartan equation is captured in the object's curvature. For more details we refer to \papertwoA.

Minimal models of $ A_∞ $-deformations exist. When $ \cat C_q $ is a deformation of $ \cat C $, we can form a minimal model $ \H\cat C_q $. This category $ \H\cat C_q $ is a deformation of $ \H\cat C $. The differential and curvature of $ \H\cat C_q $ need not vanish. Instead, $ \H\cat C_q $ carries an infinitesimal residue differential and curvature. For more details we refer to \papertwoA.

%% file: prelim-ainfty/submodules.tex
\subsection{Submodules of completed tensor products}
\label{sec:prelim-submodules}
We introduce here the notions of pseudoclosed and quasi-flat submodules of $ B \htensor X $ which we use in \autoref{sec:flatness}. We also comment on intersections between submodules.

Let us start by introducing pseudoclosed submodules. The rationale is that not all $ B $-submodules of $ B \htensor X $ are created equal: Some are closed under taking power series with increasing powers of $ \mathfrak{m} $, some are not. As preparation, we define the following notation:

\begin{definition}
\label{def:prelim-submodules-shorthand}
Let $ X $ be a vector space and $ Y ⊂ B \htensor X $ a subspace. Then we put
\begin{align*}
BY &≔ \Im(B \htensor Y → B \htensor X), \\
\mathfrak{m}^k Y &≔ \Im(\mathfrak{m}^k \htensor Y → B \htensor X).
\end{align*}
Here, the maps $ B \htensor Y → B \htensor X $ and $ \mathfrak{m}^k \htensor Y → B \htensor X $ denote the multiplication maps which send for instance $ b ¤ y ↦ by $ and $ m ¤ y ↦ my $.
\end{definition}

\begin{remark}
Explicitly, the spaces $ BY $ and $ \mathfrak{m}^k Y $ are given by elements of $ B \htensor X $ that can be written respectively as
\begin{align*}
x &= \sum_{i = 0}^∞ m_i y_i ∈ B \htensor X, \quad m_i ∈ \mathfrak{m}^{→ ∞}, ~ y_i ∈ Y, \\
x &= \sum_{i = 0}^∞ m_i y_i ∈ B \htensor X, \quad m_i ∈ \mathfrak{m}^{≥k, →∞}, ~ y_i ∈ Y.
\end{align*}
Here we use the notation $ m_i ∈ \mathfrak{m}^{→ ∞} $ to indicate that there is a sequence $ (k_i) ⊂ ℕ $ converging to $ ∞ $ such that $ m_i ∈ \mathfrak{m}^{k_i} $. The notation $ m_i ∈ \mathfrak{m}^{≥k, →∞} $ indicates that we shall have $ k_i ≥ k $.
\end{remark}

\begin{remark}
We warn that the notation $ BY $ is suggestive but should not be misunderstood. The rationale behind the notation is that $ BY $ should contain any $ B $-linear multiples of $ Y $ and power series of such elements in growing $ \mathfrak{m} $-order. We emphasize that $ BY $ is not the same as the linear span of products $ b · y $ for $ b ∈ B $ and $ y ∈ Y $. Similarly, $ BY $ is not the same as the $ \mathfrak{m} $-adic closure of this span. The analogous warning holds for $ \mathfrak{m}^k Y $. The intention of the notation is to provide foundations for \autoref{sec:flatness}.
\end{remark}

\begin{definition}
An $ B $-submodule $ M ⊂ B \htensor X $ is \emph{pseudoclosed} if $ BM ⊂ M $.
\end{definition}

\begin{example}
If $ Y \subseteq X $ is a linear subspace, then $ B \htensor Y \subseteq B \htensor X $ is pseudoclosed. In contrast, the $ B $-submodule $ B ¤ X \subseteq B \htensor X $ is not pseudoclosed if $ X $ is infinite-dimensional, since $ B (B ¤ X) = B \htensor X $.
\end{example}

Denote by $ B \cdot Y \subseteq B \htensor X $ the space (finitely) spanned by elements of the form $ by $ with $ b \in B $ and $ y \in Y $. Denote by $ \mathfrak{m}^k \cdot Y $ the space (finitely) spanned by elements of the form $ my $ with $ m \in \mathfrak{m}^k $ and $ y \in Y $. In general, the spaces $ BY $ and $ \mathfrak{m}^k Y $ are not the same as $ B \cdot Y $ and $ \mathfrak{m}^k \cdot Y $. Pseudoclosed submodules are an exception:

\begin{lemma}
\label{th:prelim-submodules-algtop}
Let $ Y ⊂ B \htensor X $ be a pseudoclosed $ B $-submodule. Then $ BY = B · Y $ and $ \mathfrak{m}^k Y = \mathfrak{m}^k · Y $.
\end{lemma}

\begin{proof}
The first statement is obvious since $ BY ⊂ Y ⊂ B · Y $. For the second statement, the idea is to exploit the Cohen structure theorem. Write $ B = ℂ⟦q_1, …, q_n⟧ / I $, and regard the maximal ideal $ \mathfrak{m} = (q_1, …, q_n) $. With this in mind, we can write any element $ x ∈ BY $ as a series
\begin{equation*}
x = \sum_{i = 0}^∞ m_i \tilde m_i y_i.
\end{equation*}
Here $ m_i $ is a monomial of degree $ k $ in the variables $ q_1, …, q_n $, the letter $ \tilde m_i $ denotes a sequence $ \tilde m_i ∈ \mathfrak{m}^{→∞} $, and $ y_i ∈ Y $. We conclude
\begin{equation*}
x = \sum_{\substack{\text{monomials } M \\ \text{ of degree } k}} M \sum_{\substack{i ≥ 0 \\ m_i = M}} \tilde m_i y_i.
\end{equation*}
The outer sum is finite. For every monomial $ M $ of degree $ k $, the inner sum is an element of $ Y $ since $ Y $ is pseudoclosed by assumption. In conclusion, every summand of the outer sum lies in $ \mathfrak{m}^k · Y $, and hence $ x ∈ \mathfrak{m}^k · Y $. We have shown that $ \mathfrak{m}^k Y ⊂ \mathfrak{m}^k · Y $. The inverse inclusion is obvious. We conclude $ \mathfrak{m}^k Y = \mathfrak{m}^k · Y $, finishing the proof.
\end{proof}

\begin{example}
A simple application is the case of the pseudoclosed submodule $ Y = BX = B \htensor X $. In this case, the lemma states $ \mathfrak{m}^k X = \mathfrak{m}^k (BX) = \mathfrak{m}^k · (BX) $ as subsets of $ B \htensor X $.
\end{example}

\begin{remark}
We interpret \autoref{th:prelim-submodules-algtop} as follows: The space $ \mathfrak{m}^k \cdot Y $ makes only reference to the $ B $-module structure, while the space $ \mathfrak{m}^k Y $ references a mixture of the $ B $-module structure with the topology of the ambient space $ B \htensor X $. For pseudoclosed modules, the topological part is already captured by the algebraic structure.
\end{remark}

We introduce now a flatness condition for $ B $-submodules of $ B \htensor X $. This flatness condition is particularly relevant in \autoref{sec:flatness}. To distinguish the notion from existing notions of flatness, we have chosen to name it quasi-flatness.

\begin{definition}
A $ B $-submodule $ M ⊂ B \htensor X $ is \emph{quasi-flat} if $ M ∩ \mathfrak{m} X ⊂ \mathfrak{m} M $.
\end{definition}

\begin{remark}
The inverse inclusion $ \mathfrak{m} M ⊂ M ∩ \mathfrak{m} X $ holds automatically if $ M $ is pseudoclosed.
\end{remark}

\begin{lemma}
\label{th:prelim-submodules-projectionsection}
Let $ M ⊂ B \htensor X $ be a $ B $-submodule. Assume $ M $ is quasi-flat and pseudoclosed. Let $ φ: π(M) → M $ be a linear section of the projection map $ π: M \to π(M) $. Then the $ B $-linear extension $ φ: B \htensor π(M) \to M $ is a $ B $-linear isomorphism.
\end{lemma}

\begin{proof}
Injectivity follows from \autoref{rem:prelim-ainfty-leadingterm}, since the leading term is the identity. For surjectivity, let $ x ∈ M $. We construct sequences $ (x_k) $ and $ (y_k) $ such that $ x = x_1 + … + x_N + y_N $ for every $ N \in \mathbb{N} $ and $ x_k ∈ φ(\mathfrak{m}^k π(M)) $ and $ y_k ∈ \mathfrak{m}^{k+1} ∩ M $.

To start with, write $ x = φ(x_1) + y_1 $ for some $ x_1 \in π(M) $ and $ y_1 ∈ M \cap \mathfrak{m} X $. For induction hypothesis, assume the sequences are given until index $ k $. Then note $ y_k ∈ M \cap \mathfrak{m}^{k+1} X $. By quasi-flatness, we get $ y_k \in \mathfrak{m}^{k+1} M $. We can then write $ y_k = x_{k+1} + y_{k+1} $ with $ x_{k+1} \in φ(\mathfrak{m}^k π(M)) $ and $ y_{k+1} \in M \cap \mathfrak{m}^{k+1} X $. This finishes the inductive construction of the sequences.

Finally, we have $ x = \sum_{k = 1}^∞ x_k ∈ φ(B \htensor π(M)) $. We have shown that $ φ $ is surjective. This finishes the proof.
\end{proof}

The following is a useful criterion to find quasi-flat modules:

\begin{proposition}
\label{th:prelim-submodules-criterion}
Let $ B $ be a deformation base and $ X $ a vector space. Let $ M_1, …, M_k ⊂ B \htensor X $ be $ B $-submodules. Then the following are equivalent:
\begin{itemize}
\item We have $ B \htensor X = M_1 ⊕ … ⊕ M_k $.
\item The modules $ M_i $ are all pseudoclosed and quasi-flat, and $ X = π(M_1) ⊕ … ⊕ π(M_k) $.
\end{itemize}
Here $ π: B \htensor X → X $ denotes the canonical projection map.
\end{proposition}

\begin{proof}
We first show that $ B \htensor X = M_1 ⊕ … ⊕ M_k $ implies that all $ M_i $ are quasi-flat and $ X = π(M_1) ⊕ … ⊕ π(M_k) $. After that, we show the converse.

For the first part, assume $ B \htensor X = M_1 ⊕ … ⊕ M_k $. Let $ 1 ≤ i ≤ n $. We first show that $ M_i $ is closed. Indeed, regard the projection $ p_i: B \htensor X → B \htensor X $ to the component $ M_i $. This map is clearly $ B $-linear, hence $ \id - p_i $ is $ B $-linear. By \autoref{rem:prelim-ainfty-continuityautomatic}, the map $ \id - p_i $ is continuous. The kernel of $ \id - p_i $ is $ M_i $ and we conclude that $ M_i $ is closed. In particular, $ M_i $ is pseudoclosed.

Next, we show that every $ M_i $ is quasi-flat. Pick any $ x ∈ \mathfrak{m} X ∩ M_i $. Since $ x ∈ \mathfrak{m} X $ and $ B \htensor X = M_1 + … + M_k $, we can write $ x = y_1 + … + y_k $ with $ y_j ∈ \mathfrak{m} M_j $. By pseudoclosedness we have $ y_j ∈ \mathfrak{m} M_j ⊂ M_j $. Since the sum $ M_1 ⊕ … ⊕ M_k $ is direct and $ x ∈ M_i $, we get $ y_j = 0 $ for $ j ≠ i $. We conclude $ x = y_i ∈ \mathfrak{m} M_i $. This proves every $ M_i $ quasi-flat.

Let us show that $ X = π(M_1) + … + π(M_k) $. Pick any $ x ∈ X $. Since $ x $ then also lies in $ B \htensor X $, write $ x = y_1 + … + y_k $ with $ y_i ∈ M_i $. We conclude $ x = π(x) = π(y_1) + … + π(y_k) $, therefore $ x ∈ π(M_1) + … + π(M_k) $. Since $ x $ was arbitrary, this shows $ X = π(M_1) + … + π(M_k) $.

Let us show that the sum $ π(M_1) + … + π(M_k) $ is direct. Assume by contradiction there is a sequence $ y_1, …, y_k $ with $ y_i ∈ M_i $ and $ π(y_1) + … + π(y_k) = 0 $. Then $ π(\sum y_i) = 0 $, in other words $ \sum y_i ∈ \mathfrak{m} X $. By assumption we have $ B \htensor X = M_1 + … + M_k $, in particuar $ \mathfrak{m} X ⊂ \mathfrak{m} (M_1 + … + M_k) $. Therefore we can write $ \sum y_i = \sum z_i $ with $ z_i ∈ \mathfrak{m} M_i $. Since the sum $ M_1 ⊕ … ⊕ M_k $ is direct, we have $ y_i = z_i $ for all $ i $. We conclude $ y_i = z_i ∈ \mathfrak{m} X $ and therefore $ π(y_i) = 0 $ for all $ i $. This shows that the sum $ π(M_1) + … + π(M_k) $ is direct. The first implication is proven, finishing the first part of the proof.

For the second part of the proof, assume every $ M_i $ is pseudoclosed and quasi-flat and $ X = π(M_1) ⊕ … ⊕ π(M_k) $. Choose a linear section $ π(M_i) \to M_i $ of the projection $ π: M_i \to π(M_i) $. According to \autoref{th:prelim-submodules-projectionsection}, the $ B $-linear extension $ φ_i: B \htensor π(M_i) \to M_i $ is an isomorphism. Add up all $ φ_i $ to arrive at the map
\begin{align*}
φ: B \htensor (π(M_1) ⊕ … ⊕ π(M_k)) &→ B \htensor X, \\
(x_1, …, x_k) &↦ φ_1 (x_1) + … + φ_k (x_k).
\end{align*}
Note we view $ π(M_1) ⊕ … ⊕ π(M_k) $ simply as vector space decomposition of $ X $. The map $ φ $ is $ B $-linear and continuous. Its leading term is by construction the identity on $ X $. Therefore $ φ $ is an isomorphism.

Using the auxiliary map $ φ $, we get that $ M_1 + … + M_k = B \htensor X $: By definition of $ φ $, its image is necessarily contained in $ M_1 + … + M_k $ and we conclude that $ M_1 + … + M_k = B \htensor X $.

Let us now show that the sum $ M_1 + … + M_k $ is direct. Pick any sequence $ x_1, …, x_k $ such that $ x_i ∈ M_i $ and $ x_1 + … + x_k = 0 $. Since $ φ_i $ surjects on $ M_i $, write $ x_i = φ_i (y_i) $ with $ y_i ∈ B \htensor π(M_i) $. We get that $ φ(y_1 + … + y_k) = 0 $. Since $ φ $ is injective, we get $ y_1 + … + y_k = 0 $, hence $ y_i = 0 $ for all $ i $. This shows $ x_i = φ_i (y_i) = 0 $ and we conclude that the sum $ M_1 ⊕ … ⊕ M_k $ is direct. This finishes the proof.
\end{proof}

We finish this section by explaining a property regarding intersections of $ B $-submodules. Whenever $ X, Y ⊂ B \htensor A $ are two subspaces, we may ask: Does it hold that
\begin{equation*}
\mathfrak{m} X ∩ \mathfrak{m} Y = \mathfrak{m} (X ∩ Y) ~?
\end{equation*}
This inclusion does not hold in general, but we shall give here the best possible variant in case one of the spaces $ X, Y $ is not deformed. Let us start with an example where the inclusion fails, as well as the example $ B = ℂ⟦q⟧ $.

\begin{example}
Regard $ B = ℂ⟦p, q⟧ $ and $ A = ℂ[X] $. Let $ X = \vspan(px) ⊂ B \htensor A $ and $ Y = \vspan(qx) ⊂ B \htensor A $. Then $ pqx $ lies in $ \mathfrak{m} X ∩ \mathfrak{m} Y $, but not in $ \mathfrak{m} (X ∩ Y) $.
\end{example}

\begin{example}
Regard $ B = ℂ⟦q⟧ $ and arbitrary $ A $. Let $ X, Y ⊂ B \htensor A $ be pseudoclosed. We claim that $ (q) X ∩ (q) Y ⊂ (q) (X ∩ Y) $. Indeed, pick $ z ∈ (q) X ∩ (q) Y $. By definition of $ (q) X $, we can write
\begin{equation*}
z = \sum q^{≥1, → ∞} x_i = q \sum q^{→∞} x_i, \quad z = \sum q^{≥1, →∞} y_i = q \sum q^{→∞} y_i.
\end{equation*}
Since $ X $ and $ Y $ are pseudoclosed, the two sums $ \sum q^{→∞} x_i $ and $ \sum q^{→∞} y_i $ lie in $ X $ and $ Y $, respectively. Write $ z = qz' $. Then both sums are equal to $ z' $, hence $ z' $ lies in the intersection of $ X $ and $ Y $. We conclude $ z = qz' ∈ q (X ∩ Y) $. This shows $ (q) X ∩ (q) Y ⊂ (q) (X ∩ Y) $.
\end{example}

Let us now make a more general statement. The idea is to keep one of the spaces $ X, Y $ non-deformed. In other words, one of them is simply of the form $ B \htensor V $ for some subspace $ V ⊂ A $. Let us make this precise as follows:

\begin{proposition}
\label{th:flatness-flatalgebraic-crude}
Let $ X $ be a vector space and $ B $ a deformation base. Let $ V ⊂ X $ be a subspace and $ M ⊂ B \htensor X $ a pseudoclosed $ B $-submodule. Then
\begin{equation*}
\mathfrak{m}^k M ∩ (\mathfrak{m}^k V + \mathfrak{m}^{k+1} X) ⊂ \mathfrak{m}^k (M ∩ (V + \mathfrak{m} X)) + \mathfrak{m}^{k+1} M.
\end{equation*}
\end{proposition}

\begin{proof}
In the first part of the proof, we illustrate the statement in case of $ B = ℂ⟦q⟧ $. In the second part, we build a commutative diagram which allows us to deduce the shape of elements of $ \mathfrak{m}^k M ∩ \mathfrak{m}^k V $ without choice of basis for $ B $. In the third part, we conclude the desired statement.

For the first step, let us illustrate the case of $ B = ℂ⟦q⟧ $. Let $ q^k x ∈ (q)^k M ∩ ((q)^k V + (q)^{k+1} X) $. In particular, we have $ q^k x ∈ (q)^k M $, hence $ x ∈ M $ since $ M $ is pseudoclosed. We also have $ q^k x ∈ (q)^k V + (q)^{k+1} X $, hence $ x ∈ ℂ⟦q⟧ V + (q) X $. Together this shows $ q^k x ∈ q^k (M ∩ (V + (q) X)) $. We conclude that in case $ B = ℂ⟦q⟧ $ the claimed statement holds.

For the second step, we build the following commutative diagram of linear maps:
\begin{equation*}
\begin{tikzcd}
\mathfrak{m}^k \htensor M \arrow[rr, "φ ~=~ π_i ¤ (π_V ∘ π_0)"] \arrow[rd, "π ∘ c"] && \frac{\mathfrak{m}^k}{\mathfrak{m}^{k+1}} ¤ \frac{X}{V} \arrow[ld, "ψ"', "\sim"] \\
& \frac{\mathfrak{m}^k X}{\mathfrak{m}^{k+1} X + \mathfrak{m}^k V}
\end{tikzcd}
\end{equation*}
Let us explain the maps. The horizontal map $ φ $ performs a projection to $ \mathfrak{m}^k / \mathfrak{m}^{k+1} $ on the first tensor factor and an inclusion of $ M $ into $ B \htensor X $ followed by projection to $ X $ and projection to $ X/V $ on the second tensor factor. The left vertical map $ π ∘ c $ consists of the inclusion $ M ⊂ B \htensor X $ on the second tensor factor, followed by multiplication with the first tensor factor and projection to the quotient by $ \mathfrak{m}^{k+1} X + \mathfrak{m}^k V $. The right vertical map $ ψ $ is induced from the multiplication map.

We claim that the diagram is commutative and $ ψ $ is an isomorphism. To see commutativity, pick an element $ m ¤ x $ with $ m ∈ \mathfrak{m}^k $ and $ x ∈ M $. Under $ φ $ it is sent to $ [m] ¤ [x] $, which under $ ψ $ is sent to $ [mx] $. Under $ π ∘ c $, the element $ m ¤ x $ is also sent to $ [mx] $. This demonstrates commutativity. To see that $ ψ $ is an isomorphism, recall that $ \mathfrak{m}^k X $ and $ \mathfrak{m}^k \htensor X $ are isomorphic by means of the splitting map $ \mathfrak{m}^k X → \mathfrak{m}^k \htensor X $. This splitting map induces a map
\begin{equation*}
\frac{\mathfrak{m}^k X}{\mathfrak{m}^{k+1} X + \mathfrak{m}^k V} → \frac{\mathfrak{m}^k}{\mathfrak{m}^{k+1}} ¤ \frac{X}{V}.
\end{equation*}
This map is an inverse of $ ψ $. This shows that $ ψ $ is an isomorphism.

For the third part of the proof, we conclude the desired inclusion. Let us start with the remark that the kernel of $ φ $ is equal to
\begin{equation}
\Ker(φ) = \mathfrak{m}^{k+1} \htensor M + \mathfrak{m}^k \htensor (M ∩ (V + \mathfrak{m} X)).
\end{equation}
Let now $ \sum_{i = 0}^∞ m_i x_i ∈ \mathfrak{m}^k M ∩ (\mathfrak{m}^k V + \mathfrak{m}^{k+1} X) $ with $ m_i ∈ \mathfrak{m}^{≥k, →∞} $ and $ x_i ∈ M $. Then
\begin{equation*}
(π ∘ c) \left(\sum m_i ¤ x_i\right) = \sum m_i π_0 (x_i) = 0.
\end{equation*}
Since $ π ∘ c = ψφ $ and $ ψ $ is injective, we get $ φ(\sum m_i ¤ x_i) = 0 $. Therefore $ \sum m_i ¤ x_i $ lies in the kernel of $ φ $, which explicitly reads
\begin{equation*}
\sum m_i ¤ x_i ∈ \mathfrak{m}^{k+1} \htensor M + \mathfrak{m}^k \htensor (M ∩ (V + \mathfrak{m} X)).
\end{equation*}
Contracting the tensors gives
\begin{equation*}
\sum m_i x_i ∈ \mathfrak{m}^{k+1} M + \mathfrak{m}^k (M ∩ (V + \mathfrak{m} X)).
\end{equation*}
Since $ \sum m_i x_i $ was arbitrarily chosen, this finishes the proof.
\end{proof}

%% file: prelim-ainfty/freemodules.tex
\subsection{$ \mathfrak{m} $-adically free modules}
\label{sec:3ainfty-free}
In this section, we recall the notion of $ \mathfrak{m} $-adically free modules and their use in $ A_∞ $-deformations. The reason is that in \autoref{sec:CHL}, it comes very handy to use $ A_\infty $-deformations modeled on $ B $-modules which are only noncanonically isomorphic to $ B \htensor \Hom_{\cat C} (X, Y) $. In this section, we first recall the definition of $ \mathfrak{m} $-adically free modules, then tie it back to quasi-flat and pseudoclosed modules. After that, we provide the definition of $ A_∞ $-deformations that makes use of $ \mathfrak{m} $-adically free modules.

We start by recalling the notion of $ \mathfrak{m} $-adically free modules. We pick the following definition, as in \cite{Yekutieli-completion}:

\begin{definition}
A $ B $-module $ M $ is \emph{$ \mathfrak{m} $-adically free} if there is a vector space $ X $ such that $ M ≅ B \htensor X $ as $ B $-modules.
\end{definition}

An abstract $ B $-module $ M $ enjoys an $ \mathfrak{m} $-adic topology given by the neighborhood basis $ x + \mathfrak{m}^k \cdot M $ for every $ x \in M $ and $ k \in \mathbb{N} $. We claim that when $ M \cong B \htensor X $, then this topology is automatically compatible with the $ \mathfrak{m} $-adic topology on $ B \htensor X $:

\begin{lemma}
\label{th:prelim-freemodules-presentingcontinuity}
Let $ M $ be a $ B $-module and $ φ: M → B \htensor X $ be a $ B $-linear isomorphism. Then $ φ $ is a homeomorphism.
\end{lemma}

\begin{proof}
By bijectivity and $ B $-linearity of $ φ $ we have $ φ(\mathfrak{m}^k · M) = \mathfrak{m}^k · (B \htensor X) $. By \autoref{th:prelim-submodules-algtop}, we have $ \mathfrak{m}^k · (B \htensor X) = \mathfrak{m}^k X $. This shows that $ φ $ is a homeomorphism.
\end{proof}

\begin{lemma}
Let $ M, N $ be $ \mathfrak{m} $-adically free $ B $-modules and $ φ: M → N $ a $ B $-linear map. Then $ φ $ is automatically continuous.
\end{lemma}

\begin{proof}
Since $ φ $ is $ B $-linear, we have $ φ(\mathfrak{m}^k \cdot M) \subseteq \mathfrak{m}^k \cdot N $. This proves $ φ $ continuous.
\end{proof}

The ad-hoc quasi-flatness condition $ M \cap \mathfrak{m} X \subseteq \mathfrak{m} M $ is related to $ \mathfrak{m} $-adic freeness. While the former is a condition that makes explicit reference to the ambient space, the latter depends only on the abstract $ B $-module structure. Both are not equivalent, but we provide here the closest tie we can get.

\begin{proposition}
Let $ M ⊂ B \htensor X $ be a $ B $-submodule. Then the following are equivalent:
\begin{itemize}
\item $ M $ is quasi-flat and pseudoclosed.
\item There is a $ B $-linear isomorphism $ B \htensor Y → M $ with injective leading term $ Y → X $.
\end{itemize}
\end{proposition}

\begin{proof}
Assume $ M $ is quasi-flat and pseudoclosed. Then according to lemma we get a $ B $-linear isomorphism $ B \htensor π(M) \isoto M $ with leading term the identity. This shows the claim, in particular $ M $ is $ \mathfrak{m} $-adically free.

Conversely, assume $ M $ is $ \mathfrak{m} $-adically free, presented by an injective leading term. Then we have an isomorphism $ φ: B \htensor Y → M ⊂ B \htensor X $. The map is automatically continuous. We show that $ M $ is pseudoclosed: Let $ \sum m_i x_i $ be a series with $ m_i ∈ \mathfrak{m}^{→∞} $ and $ x_i ∈ M $. Then write $ x_i = φ(y_i) $. We get $ \sum m_i y_i ∈ B \htensor Y $ and $ φ(\sum m_i y_i) = \sum m_i φ(y_i) = \sum m_i x_i $. This shows $ \sum m_i x_i ∈ \Im(φ) = M $. Hence $ M $ is pseudoclosed.

To show that $ M $ is quasi-flat, pick an element of $ M ∩ \mathfrak{m} X $, written in the form $ φ(x) ∈ M ∩ \mathfrak{m} X $ with $ x ∈ B \htensor Y $. We claim $ x ∈ \mathfrak{m} Y $. Write $ x = y + z $ with $ y ∈ Y $ and $ z ∈ \mathfrak{m} Y $. Then $ φ(x) = φ(y) + φ(z) $. Hence $ φ(y) ∈ \mathfrak{m} X $. In particular $ y $ vanishes under the leading term $ πφ: Y → X $. By assumption, the leading term is injective and we get $ y = 0 $. This shows $ φ(x) = φ(z) ∈ \mathfrak{m} M $. This shows $ M ∩ \mathfrak{m} X ⊂ \mathfrak{m} M $.
\end{proof}

\begin{remark}
It is not true that $ M ⊂ B \htensor X $ is quasi-flat and pseudoclosed if and only if it is $ \mathfrak{m} $-adically free. For instance, let $ X = \vspan(x_1, x_2) $ and $ B = ℂ⟦q⟧ $. Regard the space $ M = (q) x_1 + ℂ⟦q⟧ x_2 $. The module $ M $ is $ \mathfrak{m} $-adically free through the isomorphism $ B \htensor X → M $ given by $ x_1 ↦ q x_1 $ and $ x_2 ↦ x_2 $. However $ M $ is not quasi-flat since $ qx_1 ∈ ((q) X ∩ M) \setminus (q)M $.
\end{remark}

One can use $ \mathfrak{m} $-adically free modules to model $ A_\infty $-deformations. To be more precise, we have so far defined $ A_\infty $-deformations as (infinitesimally curved) $ A_\infty $-structures on the completed tensor product $ B \htensor \Hom_{\cat C} (X, Y) $. It is possible to also allow arbitrary $ \mathfrak{m} $-adically free $ B $-modules instead, under the condition that the quotient by $ \mathfrak{m} $ is $ \Hom_{\cat C} (X, Y) $. We greatly profit from this variant in \autoref{sec:CHL}.

\begin{definition}
Let $ \cat C $ be a $ ℤ $-graded (or $ ℤ/2ℤ $-graded) $ A_∞ $-category. A \emph{loose $ A_∞ $-deformation} of $ \cat C $ is a collection of $ \mathfrak{m} $-adically free $ ℤ $-graded (or $ ℤ/2ℤ $-graded) $ B $-modules $ \{\Hom_{\cat C_q} (X, Y)\}_{X, Y ∈ \cat C} $ together with $ B $-multilinear maps $ μ_q^{k≥0} $ of degree $ 2-k $ satisfying the curved $ A_∞ $-relations, together with linear isomorphisms $ ψ_{X, Y}: \Hom_{\cat C_q} (X, Y) / (\mathfrak{m} · \Hom_{\cat C_q} (X, Y)) \isoto \Hom_{\cat C} (X, Y) $ for every $ X, Y $, such that $ \cat C $ is obtained by dividing out $ \mathfrak{m} $ and identification via $ \{ψ_{X, Y}\} $.
\end{definition}

\begin{definition}
Let $ \cat C, \cat D $ be two $ A_∞ $-categories and $ \cat C_q, \cat D_q $ be loose $ A_∞ $-deformations. A \emph{functor of loose $ A_∞ $-deformations} $ F_q: \cat C_q → \cat D_q $ consists of maps $ F_q: \Ob\cat C → \Ob\cat D $ together with $ B $-multilinear maps $ F_q^{k≥0} $ such that $ F_q $ satisfies the curved $ A_∞ $-functor relations and $ F_{q, X}^0 ∈ \mathfrak{m} · \Hom_{\cat C_q} (F_q(X), F_q(X)) $ for every $ X ∈ \cat C $. The \emph{leading term} of $ F_q $ is the functor $ F: \cat C → \cat D $ obtained by dividing out $ \mathfrak{m} $ and identification via $ \{ψ_{X, Y}\} $. We may also say that $ F_q $ is a deformation of $ F $ in this case.
\end{definition}

\begin{remark}
In contrast to $ A_\infty $-deformations modeled on $ B \htensor \Hom_{\cat C} (X, Y) $, a loose $ A_\infty $-deformation does not directly give a Maurer-Cartan element of the Hochschild DGLA $ \HC(\cat C) $. Instead, one first needs to make a choice of $ B $-linear identification $ φ_{X, Y}: \Hom_{\cat C_q} (X, Y) \cong B \htensor \Hom_{\cat C} (X, Y) $ for every $ X, Y \in \cat C $. Of course, the identification needs to be compatible with $ ψ_{X, Y} $ in the sense that its leading term must be the identity when identifying $ \Hom_{\cat C_q} (X, Y) / (\mathfrak{m} · \Hom_{\cat C_q} (X, Y)) $ via $ ψ_{X, Y} $. In yet other words, the following diagram needs to commute:
\begin{equation*}
\begin{tikzcd}
\Hom_{\cat C_q} (X, Y) \arrow[r, "φ_{X, Y}"] \arrow[d, "π"] & B \htensor \Hom_{\cat C} (X, Y) \arrow[d, "π"] \\
\Hom_{\cat C_q} (X, Y) / (\mathfrak{m} · \Hom_{\cat C_q} (X, Y)) \arrow[r, "ψ_{X, Y}"] & \Hom_{\cat C} (X, Y).
\end{tikzcd}
\end{equation*}

Once choices have been made, one obtains a Maurer-Cartan element $ μ_{q, φ} \in \MC(\HC(\cat C), B) $. Let us explain why two different choices $ φ, φ' $ yield gauge-equivalent Maurer-Cartan elements: Both $ φ_{X, Y} $ and $ φ'_{X, Y} $ have leading term the identity when identified via $ ψ_{X, Y} $, hence the composition
\begin{equation*}
φ'_{X, Y} ∘ φ^{-1}_{X, Y}: B \htensor \Hom_{\cat C} (X, Y) → B \htensor \Hom_{\cat C} (X, Y)
\end{equation*}
has leading term the identity (without any identification). This shows that the two Maurer-Cartan elements $ μ_{q, φ} $ and $ μ_{q, φ'} $ are related by the strict gauge functor $ φ' ∘ φ^{-1}: (B \htensor \cat C, μ_{q, φ}) → (B \htensor \cat C, μ_{q, φ'}) $.
\end{remark}

In analogy to \autoref{def:prelim-ainfty-objectcloning}, we fix the following terminology:

\begin{definition}
\label{def:prelim-freemodules-objectcloning}
Let $ \cat C $ be an $ A_∞ $-category and $ B $ a deformation base. Let $ O $ be an arbitrary set and $ F: O → \Ob\cat C $ a map. A \emph{loose object-cloning deformation} is a loose deformation $ \cat D_q $ of $ \cat D = F^* \cat C $. The loose object-cloning deformation is \emph{essentially surjective} if $ F: \Ob\cat D → \Ob\cat C $ reaches all objects of $ \cat C $ up to isomorphism.
\end{definition}

%% file: prelim-ainfty/flatvariants.tex
\subsection{On the quasi-flatness condition}
\label{sec:3ainfty-variants}
In this section, we present two alternative ways to formulate the quasi-flatness condition that we studied in \autoref{sec:prelim-submodules}. This serves as a preparation for later use in \autoref{sec:flatness}. The two alternatives for the quasi-flatness inclusion $ M ∩ \mathfrak{m} A ⊂ \mathfrak{m} M $ read as follows:

\begin{definition}
Let $ M \subseteq B \htensor X $ be a $ B $-submodule. Then $ M $ satisfies the
\begin{itemize}
\item \emph{weak quasi-flatness inclusion} if for every $ k ≥ 1 $ we have
$ M ∩ \mathfrak{m}^k X ⊂ \mathfrak{m}^k M + \mathfrak{m}^{k+1} X $.
\item \emph{strong quasi-flatness inclusion} if for every $ k ≥ 1 $ we have
$ M ∩ \mathfrak{m}^k X ⊂ \mathfrak{m}^k M $.
\end{itemize}
\end{definition}

We claim these alternative inclusions are indeed equivalent to quasi-flatness and moreover that any quasi-flat $ B $-submodule $ M \subseteq B \htensor X $ is automatically a closed subspace of $ B \htensor X $:

\begin{proposition}
\label{th:prelim-flatvariants-equivalence}
Let $ X $ be a vector space and $ B $ a deformation base. Let $ M \subseteq B \htensor X $ be a $ B $-submodule. If $ M $ is pseudoclosed, then the following are equivalent:
\begin{itemize}
\item $ M $ is quasi-flat.
\item $ M $ is quasi-flat and closed.
\item $ M $ satisfies the weak quasi-flatness inclusion.
\item $ M $ satisfies the strong quasi-flatness inclusion.
\end{itemize}
\end{proposition}

In the following lemmas, we provide a proof of \autoref{th:prelim-flatvariants-equivalence}.

\begin{remark}
The proofs are easier to understand if one has the example of $ B = ℂ⟦q⟧ $ in mind. This case has the practical property that any element $ x ∈ \mathfrak{m} X $ can automatically be written as $ qy $ for some $ y ∈ X $. If for example it now becomes known that $ x ∈ \mathfrak{m} M $, then it is immediate that $ y ∈ M $. To see this, write $ x ∈ (q) M $ as a power series in elements of $ M $ and divide by $ q $:
\begin{align*}
qy &= x = \sum_{n = 0}^∞ q^{n+1} x_n = q · \sum_{n = 0}^∞ q^{→ ∞} x_n, \quad \text{hence} \quad y = \sum_{n = 0}^∞ q^n x_n ∈ M.
\end{align*}
\end{remark}

\begin{lemma}
\label{th:prelim-flatvariants-weaktoflat}
Let $ M \subseteq B \htensor X $ be a pseudoclosed $ B $-submodule. If $ M $ satisfies the weak quasi-flatness inclusion, then $ M $ is quasi-flat. 
\end{lemma}

\begin{proof}
It is our task to show $ M ∩ \mathfrak{m} X ⊂ \mathfrak{m} M $. Pick $ x ∈ M ∩ \mathfrak{m} X $. Iterating the weak quasi-flatness inclusion in combination with pseudoclosedness gives
\begin{equation*}
x ~∈~ M ∩ \mathfrak{m} X ~⊂~ \mathfrak{m} M + M ∩ \mathfrak{m}^2 X ~⊂~ \mathfrak{m} M + \mathfrak{m}^2 M + M ∩ \mathfrak{m}^3 X ~⊂~ ….
\end{equation*}
More precisely, write $ x = x_1 + y_1 $ with $ x_1 ∈ \mathfrak{m} M $ and $ y_1 ∈ M ∩ \mathfrak{m}^2 X $. Then write $ y_1 = x_2 + y_2 $ with $ x_2 ∈ \mathfrak{m}^2 M $ and $ y_2 ∈ M ∩ \mathfrak{m}^3 X $. Continuing this way, we obtain sequences $ (x_k) $ and $ (y_k) $ with the property that
\begin{equation*}
x = (x_1 + … + x_N) + y_N, \quad x_k ∈ \mathfrak{m}^k M, ~ y_k ∈ M ∩ \mathfrak{m}^{k+1} X.
\end{equation*}
Letting $ N → ∞ $ we get within $ B \htensor X $ that
\begin{equation*}
x = \sum_{k = 1}^∞ x_k.
\end{equation*}
In principle, the right-hand side converges within the completion of $ M $. Since we assumed that $ M $ is pseudoclosed, we can however do better: Every summand $ x_k $ lies in $ \mathfrak{m}^k M $ and summation starts at $ k = 1 $. Therefore the infinite sum lies in $ \mathfrak{m} M $. We conclude that $ x ∈ \mathfrak{m} M $. This shows $ M ∩ \mathfrak{m} X ⊂ \mathfrak{m} M $.
\end{proof}

\begin{lemma}
\label{th:prelim-flatvariants-flattostrong}
Let $ M \subseteq B \htensor X $ be a quasi-flat $ B $-submodule. Then $ M $ satisfies the strong quasi-flatness inclusion.
\end{lemma}

\begin{proof}
The proof consists of two parts: We first prove the auxiliary inclusion $ \mathfrak{m}^k M ∩ \mathfrak{m}^{k+1} X ⊂ \mathfrak{m}^{k+1} M $ for every $ k ≥ 1 $. Second, we derive the strong quasi-flatness inclusion by iterating the auxiliary inclusion.

First, let us prove the auxiliary inclusion. Denoting by $ π: B \htensor X \to X $ the standard projection, regard the subspace $ π(M) \subseteq X $. We can choose a $ B $-linear continuous map $ φ: B \htensor π(M) → M $ with leading term the identity. Then any $ x ∈ M $ can be written as $ x = φ(y) + z $ with $ z ∈ M ∩ \mathfrak{m} ⊂ \mathfrak{m} M $.

Let $ x ∈ \mathfrak{m}^k M ∩ \mathfrak{m}^{k+1} X $. Then we can write $ x = φ(y) + z $ with $ y ∈ \mathfrak{m}^k π(M) $ and $ z ∈ \mathfrak{m}^{k+1} M $. We get that $ y-φ(y) ∈ \mathfrak{m}^{k+1} X $ and $ φ(y) = x-z ∈ \mathfrak{m}^{k+1} X $. Summing up, we get $ y ∈ \mathfrak{m}^{k+1} X ∩ \mathfrak{m}^k π(M) $, hence $ y ∈ \mathfrak{m}^{k+1} π(M) $. In consequence, we have $ φ(y) ∈ \mathfrak{m}^{k+1} M $. Finally, we get $ x = φ(y) + z ∈ \mathfrak{m}^{k+1} M $. This proves the auxiliary inclusion.

Finally, we combine quasi-flatness with iterated applications of the auxiliary inclusion:
\begin{align*}
M ∩ \mathfrak{m}^k X &= (M ∩ \mathfrak{m} X) ∩ \mathfrak{m}^k X \\
&⊂ (\mathfrak{m} M ∩ \mathfrak{m}^2 X) ∩ \mathfrak{m}^k X \\
&⊂ (\mathfrak{m}^2 M ∩ \mathfrak{m}^3 X) ∩ \mathfrak{m}^k X \\
&⊂ … ⊂ \mathfrak{m}^k M.
\end{align*}
This finishes the proof.
\end{proof}

\begin{lemma}
\label{th:prelim-flatvariants-flattoclosed}
Let $ M \subseteq B \htensor X $ be a pseudoclosed and quasi-flat $ B $-submodule. Then $ M $ is closed.
\end{lemma}

\begin{proof}
Let $ \sum x_n $ be a series of elements $ x_n ∈ M $ that converges in $ B \htensor X $. Then $ x_n ∈ \mathfrak{m}^{→ ∞} X ∩ M $. By the strong quasi-flatness inclusion, we get $ x_n ∈ \mathfrak{m}^{→∞} M $. The limit of the series hence lies in $ BM $. Since $ M $ is pseudoclosed, the limit lies in $ M $. We conclude that $ M $ is closed.
\end{proof}

The combination of the above lemmas proves \autoref{th:prelim-flatvariants-equivalence}. For $ B $-submodules, pseudoclosed and quasi-flat implies closed, and closed implies pseudoclosed. However closed does not imply quasi-flat, for instance the submodule $ M = ℂ⟦q⟧x_1 + (q)x_2 ⊂ ℂ⟦q⟧ \htensor \vspan(x_1, x_2) $ is closed but not quasi-flat.

%% file: prelim-koszul/intro.tex
\section{Preliminaries on Koszul duality}
\label{sec:koszul}
In the present section, we present Koszul duality as a preparation for \autoref{sec:CHL}. Koszul duality is a phenomenon which provides a rich source of nontrivial $ A_∞ $-functors by matching $ A_∞ $-algebras and dg algebras. Classical Koszul duality involves only ordinary (associative, non-dg) algebras and makes best-effort statements on their homological properties. Modern Koszul duality concerns $ A_∞ $-algebras and dg algebras and largely recovers classical Koszul duality from a more elegant description.

\begin{center}
\begin{tabular}{@{}c@{\hspace{0.1\linewidth}}c@{}}
\textbf{A-side} & \textbf{B-side} \\\hline
{Augmented $ A_∞ $-algebra} & {DG algebra} \\
$ A = (ℂ\id ⊕ ℂX_1 ⊕ … ⊕ ℂX_n, μ_A) $ & $ \koszul A = (ℂ\ncpow{\kdual x_1, …, \kdual x_n}, \kdual{μ}_A) $ \\\hline
{Cyclic $ A_∞ $-structure} & {Calabi-Yau dg structure} \\
Degree $ n $ & Dimension $ n $ \\\hline
{$ A $-coderivation} & {Linear map} \\
$ m: T(\bar A[1]) ¤ M → T(\bar A[1]) ¤ N $ & $ \kdual m: M → N ¤ \koszul A $ \\\hline
{$ A $-module} & {Twisted complex} \\
$ (M, μ_M) $ & $ F(M) = (M \tensor \koszul A, \kdual{μ}_{M, 0}) $
\end{tabular}
\end{center}

Technically, modern Koszul duality consists of dualizing the $ A_∞ $-axioms on vector space level. The results can be described both abstractly in terms of dual vector spaces and the bar construction, as well as concretely by means of a choice of basis and a construction of dg structure on a power series algebra. In the present section, we follow the works of Ginzburg \cite{Ginzburg-CY}, Van den Bergh \cite{vdB-CY, vdB-duality}, Lu-Wu-Palmieri-Zhang \cite{LPWZ-Koszul} and \cite{Bocklandt-CY}.

In \autoref{sec:koszul-bimodules}, we recall $ A_∞ $-modules and their categories. In \autoref{sec:koszul-koszul}, we recall Koszul duality between augmented $ A_∞ $-algebras and dg algebras and the Koszul duality functor. In \autoref{sec:koszul-classical}, we recall several classical properties of Koszul duality and explain how they survive in modern Koszul duality. In \autoref{sec:koszul-cy}, we recall Calabi-Yau dg algebras. In \autoref{sec:koszul-cyord}, we focus on ordinary (non-dg) Calabi-Yau algebras. In \autoref{sec:koszul-cy3}, we focus on the special case of ordinary Calabi-Yau algebras of dimension $ n = 3 $. In \autoref{sec:koszul-correspondence}, we recall the notion of cyclic $ A_∞ $-algebras and explain the correspondence between cyclic $ A_∞ $-algebras and dg Calabi-Yau algebras via Koszul duality. In \autoref{sec:koszul-transfer}, we tweak Koszul duality statements in order to motivate the Cho-Hong-Lau construction.

\begin{remark}
In this section, we partially follow the dg sign convention rather than the $ A_∞ $-sign convention:
\begin{align*}
μ^1 (μ^2 (f, g)) &= μ^2 (μ^1 (f), g) + (-1)^{|f|} μ^2 (f, μ^1 (g)), \\
μ^2 (f, μ^2 (g, h)) &= μ^2 (μ^2 (f, g), h).
\end{align*}
The precise distribution is as follows: The $ A_∞ $-algebra $ A $ follows $ A_∞ $-signs. The dg algebra $ \koszul A $, the dg categories $ \rModfd A $ and $ \Tw\koszul A $ as well as the dg functor $ F: \rModfd A → \Tw\koszul A $ follow dg signs. We amend all signs to the $ A_∞ $-setting in \autoref{sec:CHL}.
\end{remark}

%% file: prelim-koszul/bimodules.tex
\subsection{Modules}
\label{sec:koszul-bimodules}
In this section, we recall modules over $ A_∞ $-algebras. We start by recalling the interpretation of $ A_∞ $-structures as coderivations on the bar construction. Then we recall $ A_∞ $-modules and categories of $ A_∞ $-modules.



We start by recalling the tensor coalgebra construction. We denote by $ [1] $ the left-shift.

\begin{definition}
Let $ A $ be an $ A_∞ $-algebra. Regard the \emph{tensor algebra}
\begin{equation*}
T(A[1]) = \bigoplus_{n ∈ ℕ} A[1]^{¤n}.
\end{equation*}
The canonical coproduct $ Δ: T(A[1]) → T(A[1]) ¤ T(A[1]) $ is given by
\begin{equation}
\label{eq:koszul-bimodules-coproduct}
Δ(a_k ¤ … ¤ a_1) = \sum_{0 ≤ i ≤ k} (a_i ¤ … ¤ a_1) ¤ (a_k ¤ … ¤ a_{i+1}).
\end{equation}
A \emph{coderivation} $ m: T(A[1]) → T(A[1]) $ is a linear map satisfying the co-Leibniz rule
\begin{equation*}
Δ ∘ m = (m ¤ \id + \id ¤ m) ∘ Δ.
\end{equation*}
Here the maps $ m ¤ \id $ and $ \id ¤ m $ bear the sign given by the Koszul sign rule:
\begin{align*}
(m ¤ \id)((a_k ¤ … ¤ a_1) ¤ (b_l ¤ … ¤ b_1)) &= m(a_k ¤ … ¤ a_1) ¤ (b_l ¤ … ¤ b_1) \\
(\id ¤ m)((a_k ¤ … ¤ a_1) ¤ (b_l ¤ … ¤ b_1)) &= (-1)^{‖a_1‖ + … + ‖a_k‖} (a_k ¤ … ¤ a_1) ¤ m(b_l ¤ … ¤ b_1).
\end{align*}
\end{definition}

\begin{remark}
The $ A_∞ $-product $ μ $ on $ A $ can be interpreted as a coderivation $ μ_A: T(A[1]) → T(A[1]) $ given by
\begin{equation*}
μ_A (a_k ¤ … ¤ a_1) = \sum_{0 ≤ j < i ≤ k} (-1)^{‖a_1‖ + … + ‖a_j‖} a_k ¤ … ¤ μ(a_i, …, a_{j+1}) ¤ … ¤ a_1.
\end{equation*}
The $ A_∞ $-relations for $ μ $ are equivalent to the condition $ μ_A^2 = 0 $. One easily checks that $ μ_A $ is indeed a coderivation with respect to the coproduct $ Δ $. This check explains the awkward flip used in \eqref{eq:koszul-bimodules-coproduct}. If one uses “Polish notation” $ μ(a_1, …, a_k) $ instead of $ μ(a_k, …, a_1) $, one can avoid this flip (see \cite{Seidel-bimodules, CHL}).
\end{remark}

\begin{definition}
Let $ A $ be an $ A_∞ $-algebra. Then the \emph{bar construction} $ \Barcon A $ is the dg coalgebra structure on $ T(A[1]) $ given by the canonical coproduct $ Δ $ together with the coderivation $ μ_A $.
\end{definition}


Modules over dg algebras comes with only an action map $ A ¤ M → A $ and a differential $ M → M $. When $ A $ is an $ A_∞ $-algebra, one allows the action maps to have higher components. For sake of \autoref{sec:koszul-koszul}, we restrict here to defining right $ A $-modules. Left $ A $-modules are defined analogously.

\begin{definition}
Let $ A $ be an $ A_∞ $-algebra. Then a right \emph{$ A $-module} is a graded vector space $ M $ together with a degree $ 1 $ map $ μ: M ¤ T(A[1]) → M $ of satisfying the $ A_∞ $-relations when combined with the product $ μ $ of $ A $ in a suitable way:
\begin{multline*}
\sum_{0 ≤ i ≤ k} (-1)^{‖a_1‖ + … + ‖a_i‖} μ(μ(m, a_k, …, a_{i+1}), a_i, …, a_1) \\
+ \sum_{0 ≤ j < i ≤ k} (-1)^{‖a_1‖ + … + ‖a_j‖} μ(m, a_k, …, a_{i+1}, μ(a_i, …, a_{j+1}), a_j, …, a_1) = 0.
\end{multline*}
We shall only regard unital $ A $-modules, in the sense that $ μ(m, \id) = m $ and $ μ^{≥3} (m, …, \id, …) = 0 $.
\end{definition}

An $ A $-module can be captured elegantly as a coderivation, comparable to the way that the $ A_∞ $-product on $ A $ can be captured via a coderivation:

\begin{definition}
\label{def:koszul-bimodules-Mcoderivation}
Let $ A $ be an $ A_∞ $-algebra and $ M $ a graded vector space. Then we regard the comodule map $ Δ_M: M ¤ T(A[1]) → T(A[1]) ¤ (M ¤ T(A[1]) $ given by
\begin{equation*}
Δ_M (m ¤ a_k ¤ … ¤ a_1) = \sum_{0 ≤ i ≤ k} (a_i ¤ … ¤ a_1) ¤ (m ¤ a_k ¤ … ¤ a_{i+1}).
\end{equation*}
A map $ f: M ¤ T(A[1]) → N ¤ T(A[1]) $ is a \emph{coderivation} if
\begin{equation*}
Δ_N ∘ f = (\id ¤ f) ∘ Δ_M.
\end{equation*}
In the context of coderivations, we denote by $ μ_A $ also the map $ M ¤ T(A[1]) → M ¤ T(A[1]) $ given by
\begin{equation*}
μ_A (a_k ¤ … ¤ a_1 ¤ m) = \sum_{0 ≤ j < i ≤ k} (-1)^{‖a_1‖ + … + ‖a_j‖} ~ m ¤ a_k ¤ … ¤ μ(a_i ¤ … ¤ a_{j+1}) ¤ … ¤ a_1.
\end{equation*}
\end{definition}

\begin{remark}
Whenever $ f: M ¤ T(A[1]) → N ¤ T(A[1]) $ is a coderivation, we can consider its projection to $ N $ which we denote by $ f_0: M ¤ T(A[1]) → N $. Conversely, if $ f_0: M ¤ T(A[1]) → N $ is a graded linear map, we can turn it into a coderivation $ f: M ¤ T(A[1]) → N ¤ T(A[1]) $. The precise correspondence between $ f $ and $ f_0 $ reads
\begin{equation*}
f(m ¤ a_k ¤ … ¤ a_1) = \sum_{0 ≤ i ≤ k} (-1)^{|f_0| (‖a_1‖ + … + ‖a_i‖)} f_0 (m ¤ a_k ¤ … ¤ a_{i+1}) ¤ a_i ¤ … ¤ a_1.
\end{equation*}
In terms of \autoref{def:koszul-bimodules-Mcoderivation}, a right $ A $-module is simply a coderivation $ μ_M: M ¤ T(A[1]) → M ¤ T(A[1]) $ of degree $ 1 $ such that $ (μ_M + μ_A)^2 = 0 $. The use of the letter $ μ_A $ is clearly an abuse of notation, but we expect no confusion to arise.
\end{remark}


Capturing an $ A $-module in terms of a coderivation $ μ_M $ makes it particularly straightforward to define a category of $ A $-modules:

\begin{definition}
Let $ A $ be an $ A_∞ $-algebra. Then $ \rModfd A $ is the dg category of finite-dimensional right $ A $-modules, with structure specified as follows:
\begin{itemize}
\item The hom space $ \Hom_{\rModfd A} (M, N) $ is the space of coderivations $ M ¤ T(A[1]) → N ¤ T(A[1]) $.
\item The product $ μ^2_{\rModfd A} $ is ordinary composition.
\item The differential measures failure to be a module morphism:
\begin{equation*}
μ^1_{\rModfd A} (f) = (μ_A + μ_N) ∘ f - (-1)^{|f|} f ∘ (μ_A + μ_M).
\end{equation*}
\end{itemize}
\end{definition}

\begin{remark}
It is readily checked by hand that $ μ^1_{\rModfd A} (f) $ is a coderivation if $ f $ is a coderivation.
\end{remark}

Bimodules are another important tool in homological algebra. They can be defined in a way analogous to left or right $ A $-modules, see for instance \cite[Section 2]{Seidel-bimodules}. We come back to bimodules in the dg case in \autoref{sec:koszul-cy}.




%% file: prelim-koszul/koszul.tex
\subsection{Koszul duality}
\label{sec:koszul-koszul}
In this section, we recapitulate Koszul duality as a preparation to the Cho-Hong-Lau construction. Koszul duality is a construction connecting an $ A_∞ $-algebra $ A $ with a dg algebra $ \koszul A $, its Koszul dual. Surprisingly, this construction also induces a correspondence between $ A_∞ $-modules over $ A $ and twisted complexes over $ \koszul A $:
\begin{equation*}
F: \rModfd A \verylongto \Tw\koszul A.
\end{equation*}
Our aim is to recall as fast as possible that Koszul duality produces functors. For further details we refer to \cite{LPWZ-Koszul} and \cite[Section 12.5]{Bocklandt-book}.

\begin{definition}
An \emph{augmented} $ A_∞ $-algebra is an $ A_∞ $-algebra $ A $ with a decomposition $ A = \bar A ⊕ ℂ \id $ such that $ μ(a_k, …, a_1) ∈ \bar A $ whenever $ a_1, …, a_k ∈ \bar A $.
\end{definition}

\begin{remark}
If $ A $ is an augmented $ A_∞ $-algebra, many constructions for $ A $ can be carried out by working with the “augmented” tensor coalgebra $ T(\bar A[1]) $ instead of the full tensor coalgebra $ T(A[1]) $. For instance, to define an $ A $-module it suffices to provide the map $ T(\bar A[1]) ¤ M → M $ instead of $ T(A[1]) ¤ M → M $ since we only work with unital modules.

The canonical coproduct $ Δ: T(\bar A[1]) → T(\bar A[1]) ¤ T(\bar A[1]) $ and for a graded vector space $ M $ the comodule map $ Δ_M: T(\bar A[1]) ¤ M → T(\bar A[1]) ¤ T(\bar A[1]) ¤ M $ are defined as in the non-augmented case, this time restricting to $ T(\bar A[1]) $ instead of $ T(A[1]) $.
\end{remark}

\begin{remark}
If $ M $ is a graded vector space, we denote its graded dual vector space by $ \kdual M $. The dual of $ T(\bar A[1]) $ is equal to
\begin{equation*}
\kdual{T(\bar A[1])} = \prod_{n = 0}^∞ (\kdual{\bar A[1]})^{¤ n}.
\end{equation*}
Here we make the identification that reverses the order of tensor components:
\begin{equation}
\label{eq:koszul-koszul-reverse}
\begin{aligned}
\kdual V_1 ¤ … ¤ \kdual V_k &\verylongisoto \kdual{(V_k ¤ … ¤ V_1)}, \\
φ_1 ¤ … ¤ φ_k & \verylongmapsto (-1)^{\sum_{1 ≤ s < t ≤ k} |φ_s||φ_t|} [(x_k ¤ … ¤ x_1) ↦ φ_1 (x_1) … φ_k (x_k)].
\end{aligned}
\end{equation}
The dual of a map $ m: T(\bar A[1]) → T(\bar A[1]) $ has the shape $ \kdual m: \kdual{T(\bar A[1])} → \kdual{T(\bar A[1])} $.
\end{remark}

\begin{definition}
\label{def:koszul-koszul-def}
Let $ A $ be a finite-dimensional augmented $ A_∞ $-algebra. Then its \emph{Koszul dual} $ \koszul A $ is the dg algebra given by
\begin{equation*}
\koszul A = \kdual{T(\bar A[1])}.
\end{equation*}
Upon the identification of \eqref{eq:koszul-koszul-reverse}, the product on $ \koszul A $ is defined as the standard product on $ \compl{T(\kdual{\bar A[1]})} $. The differential on $ \koszul A $ is given by
\begin{equation*}
dv = (-1)^{|v| + 1} v ∘ μ_A, \quad ∀ v ∈ \kdual{T(\bar A[1])} = \Hom_{ℂ} (T(\bar A[1]), ℂ).
\end{equation*}
Here $ μ_A: T(\bar A[1]) → T(\bar A[1]) $ denotes the product of $ A $.
\end{definition}

\begin{example}
Pick a basis $ x_1, …, x_k $ for $ A $. Then we have the dual elements $ \kdual x_i $ of degree $ |\kdual x_i| = 1 - |x_i| $. Multiplication within $ \koszul A $ can be performed simply as $ \kdual x_i · \kdual x_j = \kdual x_i \kdual x_j $. To interpret this element as element of $ \kdual{T(\bar A[1])} $, we have to apply the sign flip from \eqref{eq:koszul-koszul-reverse}:
\begin{equation*}
\kdual x_i \kdual x_j = (-1)^{|\kdual x_i| |\kdual x_j|} \kdual{(x_j ¤ x_i)}.
\end{equation*}
Here we have written $ \kdual{(x_j ¤ x_i)} $ for the element of $ \koszul A $ which sends the element $ x_j ¤ x_i $ to $ 1 $ and all other basis tensors to zero. Now write the product of $ A $ as
\begin{equation*}
μ^l (x_{i_l}, …, x_{i_1}) = \sum_{1 ≤ j ≤ k} c_{i_l, …, i_1}^j x_j.
\end{equation*}
Then
\begin{align*}
d_{\koszul A} (\kdual x_j) &= (-1)^{|\kdual x_j| + 1} \sum_{\substack{l ≥ 1 \\ 1 ≤ i_l, …, i_1 ≤ k}} c_{i_l, …, i_1}^j \kdual{(x_{i_l} ¤ … ¤ x_{i_1})} \\
&= (-1)^{|\kdual x_j| + 1} \sum_{\substack{l ≥ 1 \\ 1 ≤ i_l, …, i_1 ≤ k}} c_{i_l, …, i_1}^j (-1)^{\sum_{1 ≤ s < t ≤ l} |\kdual x_{i_s}| |\kdual x_{i_t}|} ~ \kdual x_{i_1} … \kdual x_{i_l}.
\end{align*}
\end{example}

\begin{remark}
The Koszul dual $ \koszul A $ is indeed a dg algebra. Abstractly speaking, the dual of the operator $ Δ $ is the ordinary product $ \kdual{Δ}: \kdual{T(\bar A[1])} ¤ \kdual{T(\bar A[1])} → \kdual{T(\bar A[1])} $. Dualizing $ μ_A^2 = 0 $ gives $ (\kdual{μ}_A)^2 = 0 $ and dualizing the co-Leibniz rule for $ μ_A $ with respect to $ Δ $ gives the Leibniz rule for $ μ_A $ with respect to $ \kdual{Δ} $. The signs can be checked by hand.
\end{remark}

\begin{definition}
Let $ A $ be an augmented $ A_∞ $-algebra. Let $ M $ and $ N $ be graded vector spaces and $ f: M ¤ T(\bar A[1]) → N $ a linear map. Then the \emph{Koszul transform} of $ f $ is the partial dual map $ \kdual f: M → N ¤ \koszul A $. It is a graded linear map with characterizing property
\begin{equation*}
∀ m ∈ M, ~ a ∈ T(\bar A[1]): \quad ⟨\kdual f(m), a⟩ = f(m ¤ a).
\end{equation*}
Here $ ⟨-, -⟩ $ liberally denotes the standard pairing between $ \kdual{T(\bar A[1])} $ and $ T(\bar A[1]) $, in this case as map
\begin{equation*}
⟨-, -⟩: (N ¤ \kdual{T(\bar A[1])}) ¤ T(\bar A[1]) → N.
\end{equation*}
\end{definition}

We copy \autoref{def:3ainfty-ainfty-Add} and adapt it slightly to the dg case.

\begin{definition}
Let $ D $ be a dg algebra. Then the category $ \Add D $ is the category of formal shifted sums of copies of $ D $ with hom spaces given as spaces of matrices with entries $ D $. The differential $ μ^1_{\Add D} $ and product $ μ^2_{\Add D} $ are the linear and bilinear extension of differential and product of $ D $, with a sign change. On single matrix entries, the sign change is as follows:
\begin{align*}
μ^1_{\Add D} (a) &= (-1)^{k-l} d_D a, \quad ∀ a: D[k] → D[l], \\
μ^2_{\Add D} (a, b) &= (-1)^{|a|_D (k-l)} ab, \quad ∀ a: D[l] → D[m], ~ b: D[k] → D[l].
\end{align*}
\end{definition}

\begin{definition}
Let $ D $ be a dg algebra. Then a twisted complex over $ D $ is an element $ X ∈ \Add D $ together with an element $ δ ∈ \End_{\Add D}^1 (X) $ such that $ δ $ is upper triangular and satisfies the Maurer-Cartan equation:
\begin{equation*}
μ^1_{\Add D} (δ) + μ^2_{\Add D} (δ, δ) = 0.
\end{equation*}
The differential and product on $ \Tw D $ follow the sign rule
\begin{align*}
μ^1_{\Tw D} (f) &= μ^1_{\Add D} (f) + μ^2_{\Add D} (δ, f) + (-1)^{|f|+1} μ^2_{\Add D} (f, δ), \\
μ^2_{\Tw D} (f, g) &= μ^2_{\Add D} (f, g).
\end{align*}
\end{definition}


The product $ μ^2_{\Add D} $ is comparable to a matrix product and similarly $ μ^1_{\Add D} $ acts as entry-wise differential. As preparation for the Koszul duality functor, we prove a few properties regarding the Koszul transform. We show that the Koszul transform of $ f ∘ g $ agrees with the matrix product $ μ^2_{\Add D} $ of $ \kdual f $ and $ \kdual g $ and that taking the Koszul transform of $ f ∘ μ_A $ amounts to taking the entry-wise differential $ μ^1_{\Add D} $ of $ \kdual f $. The precise statement is as follows:

\begin{lemma}
\label{th:koszul-koszul-functoriality}
Let $ A $ be a finite-dimensional augmented $ A_∞ $-algebra. Let $ f: M ¤ T(\bar A[1]) → N ¤ T(\bar A[1]) $ and $ g: L ¤ T(\bar A[1]) → M ¤ T(\bar A[1]) $ be coderivations. Then
\begin{align}
\label{eq:koszul-koszul-functoriality}
μ^2_{\Add D} (\kdual f_0, \kdual g_0) &= \kdual{(f ∘ g)_0} \\
\label{eq:koszul-koszul-dfD}
μ^1_{\Add D} (\kdual f_0) &= (-1)^{|f| + 1} \kdual{(f ∘ μ_A)_0}.
\end{align}
\end{lemma}

\begin{proof}
The simplest way to evaluate both statements is by choosing bases for $ A $, $ L $, $ M $ and $ N $. Let $ X, Y ∈ T(\bar A[1]) $ be basis elements homogeneous with respect to both tensor degree and $ ‖·‖_{\bar A[1]} $. Let $ ε_L ∈ L $ and $ ε_M ∈ M $ and $ ε_N ∈ N $ be further basis elements. Let $ f $ and $ g $ be given coderivations. They are determined solely by their first component $ f_0 $ and $ g_0 $. Since the statement is linear in $ f $ and $ g $, we may assume that $ g_0 (ε_L ¤ X) = ε_M $ and $ f_0 (ε_M ¤ Y) = ε_N $ and that $ f_0 $ and $ g_0 $ vanish on all other basis elements for $ L ¤ T(\bar A[1]) $ and $ M ¤ T(\bar A[1]) $. The Koszul transforms of $ f $, $ g $ read
\begin{align*}
\kdual f (ε_M) &= ε_N ¤ \kdual Y, \\
\kdual g (ε_L) &= ε_M ¤ \kdual X.
\end{align*}
Let us now show \eqref{eq:koszul-koszul-functoriality}. We note that $ (f ∘ g)_0 $ vanishes on all basis elements, except
\begin{equation*}
(f ∘ g)_0 (ε_L ¤ X ¤ Y) = (-1)^{|g|‖Y‖} f(g(ε_L ¤ X) ¤ Y) = (-1)^{|g||\kdual Y|} ε_N.
\end{equation*}
Therefore $ \kdual{(f ∘ g)} $ and $ \kdual f · \kdual g $ vanish on all basis elements of $ L $, except
\begin{align*}
\kdual{(f ∘ g)}(ε_L) &= (-1)^{|g||\kdual Y|} ε_N ¤ \kdual{(X ¤ Y)} \\
&= (-1)^{|g||\kdual Y| + |\kdual X||\kdual Y|} ε_N ¤ \kdual Y \kdual X \\
&= (-1)^{(|ε_L| - |ε_M|) |\kdual Y|} ε_N ¤ \kdual Y \kdual X \\
&= (-1)^{(|ε_L| - |ε_M|) |\kdual Y|} (\kdual f · \kdual g)(ε_L) \\
&= μ^2_{\Add D} (\kdual f, \kdual g) (ε_L).
\end{align*}
In the third row we have used that $ |g| + |ε_L| = |\kdual X| + |\kdual X| $. This proves \eqref{eq:koszul-koszul-functoriality}.

We now prove \eqref{eq:koszul-koszul-dfD}. It is our aim to compute the composition $ (f ∘ μ_A)_0 = f_0 ∘ μ_A $. Since $ f_0 $ vanishes on all basis elements except $ ε_M ¤ Y $ and \eqref{eq:koszul-koszul-dfD} is linear in $ μ_A $ itself, we may simply assume that $ μ_A (ε_M ¤ Z) = ε_M ¤ Y $ for some basis element $ Z ∈ T(\bar A[1]) $ and $ μ_A $ vanishes on all other basis elements. Then $ (f ∘ μ_A)_0 $ vanishes on all basis elements of $ L ¤ T(\bar A[1]) $, except
\begin{equation*}
(f ∘ μ_A)_0 (ε_M ¤ Z) = f_0 (ε_M ¤ Y) = ε_N.
\end{equation*}
We see that $ \kdual{(f ∘ μ_A)_0} $ vanishes on all basis elements of $ M $, except
\begin{align*}
\kdual{(f ∘ μ_A)_0} (ε_M) &= ε_N ¤ \kdual Z \\
&= (-1)^{|\kdual Y| + 1} ε_N ¤ d\kdual Y \\
&= (-1)^{|f| + 1} μ^1_{\Add D} (\kdual f) (ε_M).
\end{align*}
In the last row, we have used that $ |ε_M| + |f| = |ε_N| + |\kdual Y| $. This finishes the proof.
\end{proof}

\begin{remark}
In \autoref{th:koszul-koszul-functoriality}, it is essential that $ f $ and $ g $ be coderivations. For instance, $ μ_A: M ¤ T(\bar A[1]) → M ¤ T(\bar A[1]) $ is not a coderivation on its own and in fact its zeroth component $ μ_{M, 0} $ vanishes, while a composition $ f ∘ μ_A $ may again have nonvanishing zeroth component and therefore nonvanishing Koszul transform.
\end{remark}

In \autoref{th:koszul-koszul-functor}, we recall the Koszul duality functor between modules over $ A $ and twisted complexes over $ \koszul A $. The idea is to apply “Koszul transform” the action map of every module. The construction is functorial, therefore gives rise to a dg functor.

\begin{corollary}
\label{th:koszul-koszul-functor}
Let $ A $ be a finite-dimensional augmented $ A_∞ $-algebra. Then the following defines a dg functor:
\begin{align*}
F: \rModfd A &\verylongto \Tw\koszul A, \\
(M, μ_M) &\verylongmapsto (M ¤ \koszul A, \kdual{μ}_{M, 0}), \\
f &\verylongto \kdual f_0.
\end{align*}
We call $ F $ the \emph{Koszul duality functor} of $ A $.
\end{corollary}

\begin{proof}
There are three items to check. First, we show that $ \kdual{μ}_{M, 0} $ satisfies the Maurer-Cartan equation. Then, we check that $ F $ preserves differential and product. We start with the Maurer-Cartan equation:
\begin{align*}
μ^1_{\Add \koszul A} (\kdual{μ}_{M, 0}) + μ^2_{\Add \koszul A} (\kdual{μ}_{M, 0}, \kdual{μ}_{M, 0}) &= \kdual{(μ_{M, 0} ∘ μ_A)} + \kdual{(μ_{M, 0} ∘ μ_M)} \\
&= \kdual{(μ_M ∘ μ_A + μ_M ∘ μ_M + μ_A ∘ μ_M)_0} = 0.
\end{align*}
In the second row, we have used that $ (μ_A ∘ μ_M)_0 = 0 $. Next, for every coderivation $ f: M ¤ T(\bar A[1]) → N ¤ T(\bar A[1]) $ we have
\begin{align*}
F(μ^1 (f)) &= \kdual{((μ_N + μ_A) ∘ f - (-1)^{|f|} f ∘ (μ_M + μ_A))_0} \\
&= \kdual{(μ_N ∘ f)_0} - (-1)^{|f|} \kdual{(f ∘ μ_M)_0} - (-1)^{|f|} \kdual{(f ∘ μ_A)_0} \\
&= μ^2_{\Add \koszul A} (\kdual{μ}_{N, 0}, \kdual f_0) - (-1)^{|f|} μ^2_{\Add \koszul A} (\kdual f_0, \kdual{μ}_{M, 0}) + μ^1_{\Add \koszul A} (\kdual f_0) \\
&= μ^1_{\Tw\koszul A} (F(f)).
\end{align*}
If additionally $ g: L ¤ T(\bar A[1]) → M ¤ T(\bar A[1]) $ is a coderivation, then
\begin{equation*}
F(μ^2 (f, g)) = F(f ∘ g) = \kdual{(f ∘ g)_0} = μ^2_{\Add \koszul A} (\kdual f, \kdual g) = μ^2_{\Tw\koszul A} (F(f), F(g)).
\end{equation*}
This shows that $ F $ is a dg functor and finishes the proof.
\end{proof}

\begin{remark}
Strictly speaking, the object $ F(M) = (M ¤ \koszul A, \kdual{μ}_{M, 0}) $ only becomes a twisted complex upon choice of a graded basis for $ M $. Furthermore, $ \kdual{μ}_{M, 0} $ need not be an upper triangular matrix. However, if $ \bar A $ is concentrated in positive degrees, then sorting the basis elements of $ M $ in order of descending degree makes $ \kdual{μ}_{M, 0} $ upper triangular.
\end{remark}

%% file: prelim-koszul/prop.tex
\subsection{Classical Koszual duals}
\label{sec:koszul-classical}
In this section, we comment on the relations of the modern with the classical Koszul dual construction. Classical Koszul duality is namely a phenomenon known for ordinary algebras and we recall here its typical properties: First, the double Koszul dual $ \koszul{(\koszul A)} $ is $ A $ again. Second, the Koszul dual algebra is formal. Third, the Koszul algebra is the Ext algebra of its simple module $ ℂ = A / \bar A $. In the present section, we recall these statements in the classical context and recall how they translate to modern Koszul duality. A valuable source is \cite{LPWZ-Koszul}.


Koszul duality has classically been a correspondence between Koszul algebras, a class of ordinary algebras with quadratic relations:

\begin{center}
\begin{tikzpicture}
\path (0, 0) node[align=center] (A) {\textbf{Ordinary algebra} \\ $ \frac{V ¤ V}{R} = \frac{ℂ⟨X, Y⟩}{(XY-YX)} $};
\path (8, 0) node[align=center] (B) {\textbf{Ordinary algebra} \\ $ \frac{\kdual V ¤ \kdual V}{R^{\perp}} = \frac{ℂ⟨X, Y⟩}{(X^2, Y^2, XY+YX)} $};
\path[draw, <->] ($ (A.east)!0.2!(B.west) $) -- ($ (A.east)!0.8!(B.west) $) node[midway, above] {Koszul};
\end{tikzpicture}
\end{center}

The relations on either side are the “orthogonal complement” of the relations on the other side along the pairing $ (V ¤ V) ¤ (\kdual V ¤ \kdual V) → ℂ $. In particular, classical Koszul duality is an involution from the beginning.

Koszul duality for $ A_∞ $-algebras is not a one-way street either. If $ A $ is an augmented finite-dimensional $ A_∞ $-algebra, we regard the double dual $ \koszul{(\koszul A)} $. It is possible that this dg algebra is quasi-isomorphic to $ A $ itself. However, the double Koszul dual construction applies vector space duals twice so that finiteness conditions are required to match $ A $ exactly with $ \koszul{(\koszul A)} $.

A sufficient finiteness criterion can be formulated if one assumes that the $ A_∞ $-algebra $ A $ has an additional grading, also referred to as \emph{Adams grading}. The direct sum decomposition $ A = ℂ\id ⊕ \bar A $ is supposed to be compatible and the products $ μ^k $ are assumed to be homogeneous with respect to the Adams grading. One then says that $ A $ is \emph{Adams connected} if the homogeneous part of $ \bar A $ with respect to any Adams degree $ j ∈ ℤ $ is finite-dimensional and vanishes either for all $ j ≤ 0 $ or all $ j ≥ 0 $ \cite{LPWZ-Koszul}. This connectedness assumption is a sufficient finiteness criterion and ensures that the double Koszul $ \koszul{(\koszul A)} $ is quasi-isomorphic to $ A $ again:

\begin{theorem}[{\cite{LPWZ-Koszul}}]
Let $ A $ be an augmented $ A_∞ $-algebra. If $ A $ is Adams connected, then $ \koszul{(\koszul A)} $ and $ A $ are quasi-isomorphic as $ A_∞ $-algebras.
\end{theorem}


The classical Koszul dual is automatically an ordinary graded algebra, without the need to pass to cohomology. From a modern perspective, Koszul duals are dg algebras and only passing to cohomology $ \H\koszul A $ and forgetting the $ A_∞ $-structure gives an ordinary graded algebra. Koszul algebras are certain types of algebras distinguished by the property that their Koszul dual $ \koszul A $ tends to be a formal dg algebra. This way, one can forget the higher structure on $ \H\koszul A $ and recover classical Koszul duality. For sake of completeness, we recall the definition here:

\begin{definition}[{\cite{BGS}}]
A graded associative algebra $ A $ is \emph{Koszul} if it is positively graded $ A = \bigoplus_{i ≥ 0} A^i $, we have $ A^0 = ℂ\id $ and $ ℂ = A / \bigoplus_{i > 0} A^i $ as an $ A $-module has a resolution of graded $ A $-modules
\begin{equation*}
… → P^2 → P^1 → P^0 → A → 0
\end{equation*}
in which every $ P^i $ is generated by its degree $ i $ component: $ P^i = A P_i^i $.
\end{definition}

In \cite{LPWZ-Koszul} a precise criterion was given for formality:

\begin{theorem}[{\cite[Corollary 2.7]{LPWZ-Koszul}}]
Let $ A $ be an $ (a, b) $-generated Koszul algebra in the sense of \cite[Definition 2.6]{LPWZ-Koszul}. Then we have an $ A_∞ $-quasi-isomorphism $ \koszul A ≅ \H^0\koszul A $.
\end{theorem}

\begin{remark}
Classical Koszul duality is full of correspondences between types of algebras. Folklore statements include that Koszul algebras correspond to Koszul algebras, Gorenstein corresponds to Gorenstein, Artin-Schelter regular corresponds to Frobenius \cite[Corollary D, E]{LPWZ-Koszul}.
\end{remark}

\begin{remark}
The classical analog of $ \Tw\H^0\koszul A $ is $ \Perf \H^0\koszul A $. This is the reason why Koszul duality is classically formulated as triangulated equivalence between categories of the form $ \D \Mod A $ and $ \Perf \H^0\koszul A $, see for instance \cite[Theorem B]{LPWZ-Koszul}. These classical Koszul duality functors are typically complicated for the reason that they do not take the possible higher structures on $ \H\koszul A $ into account. In the classical world, this is “solved” by restricting to Koszul algebras. Thanks to modern Koszul duality, it is possible to recover classical Koszul duality from the functor $ \Modfd A → \Tw \koszul A $. The idea is to replace $ \koszul A $ by its zeroth cohomology, which is an ordinary algebra. Abstractly, we aim for a functor $ \Tw\koszul A → \Tw\H^0 \koszul A $ in order to precompose it with the Koszul duality functor $ \Modfd A → \Tw\koszul A $. We follow this route in \autoref{sec:koszul-transfer}.
\end{remark}


Another classical statement of Koszul duality is that $ \koszul A $ can be interpreted as Ext algebra of the simple module $ ℂ = A /\bar A $ of $ A $. As we recall in \autoref{th:koszul-prop-Ext}, this is also the case if $ A $ is an $ A_∞ $-algebra. One can alternatively use this description of $ \koszul A $ as defining property. With regards to notation, let $ A $ be an augmented $ A_∞ $-algebra. Then we denote by $ ℂ $ the simple right $ A $-module $ ℂ = A /\bar A $. Its action map $ ℂ ¤ T(\bar A[1]) → ℂ $ is simply zero. In these terms, $ \koszul A $ is just the dg algebra $ \Hom_{rModfd A} (ℂ, ℂ) $:

\begin{lemma}
\label{th:koszul-prop-Ext}
Let $ A $ be an augmented $ A_∞ $-algebra. Then we have an isomorphism of dg algebras
\begin{equation*}
\koszul A ≅ \Hom_{\rModfd A} (ℂ, ℂ).
\end{equation*}
\end{lemma}

\begin{proof}
This follows easily from \autoref{th:koszul-koszul-functor}. Indeed, the Koszul duality functor $ F $ sends the simple module $ ℂ $ to the twisted complex $ (\koszul A, 0) ∈ \Tw\koszul A $ and therefore establishes a map between the two dg endomorphism algebras. It is easy to see that $ F $ is fully faithful. Since the endomorphism algebra of $ (\koszul A, 0) $ is just $ \koszul A $, we are done.
\end{proof}

\begin{remark}
Similar to \autoref{th:koszul-prop-Ext}, let $ (M, μ_M) $ be a finite-dimensional right $ A $-module. Then we have an isomorphism of right $ \koszul A $-modules
\begin{equation*}
M ¤ \koszul A ≅ \Hom_{\rModfd} (ℂ, M).
\end{equation*}
The differential on $ M ¤ \koszul A $ is induced from induced from $ \rModfd A $ and the action of $ \koszul A $ is by (signed) right-multiplication. The action of $ \koszul A $ on $ \Hom_{\rModfd} (ℂ, M) $ is by composing with $ \koszul A ≅ \Hom_{\rModfd} (ℂ, ℂ) $ on the right.
\end{remark}

%% file: prelim-koszul/cy.tex
\subsection{Calabi-Yau algebras}
\label{sec:koszul-cy}
In this section, we recall Calabi-Yau algebras. This class of algebras is now widely recognized as a noncommutative analog of Calabi-Yau manifolds. The definition brings several technical difficulties and correspondingly a wide range of adaptations and variants have been introduced in the literature. In the present section, we follow the original definition of Ginzburg \cite{Ginzburg-CY}, in particular we focus on the dg case.


We start by recalling opposite and enveloping algebras.

\begin{definition}
Let $ A $ be a dg algebra. Then the \emph{opposite algebra} $ A^{\algopp} $ of $ A $ is obtained by setting $ A^{\algopp} = A $ as vector space and defining
\begin{equation*}
d_{A^{\algopp}} (a) = - d_A (a), \quad a ·_{A^{\algopp}} b = (-1)^{|a||b|} ba.
\end{equation*}
The \emph{enveloping algebra} $ A^{\env} = A ¤ A^{\algopp} $ is the tensor product of $ A $ and $ A^{\algopp} $. It is an dg algebra itself with product $ (a ¤ b)(c ¤ d) = (-1)^{|b||c|} ac ¤ bd $ and differential $ d(a ¤ b) = da ¤ b + (-1)^{|a|} a ¤ db $.
\end{definition}


\begin{definition}
Let $ A $ be a dg algebra. Then an \emph{$ A $-bimodule} is the same as an $ A^{\env} $-module. The space $ A $ becomes naturally a bimodule over $ A $ by putting $ (a ¤ b)x = (-1)^{|b||x|} axb $ for $ a ¤ b ∈ A^{\oppalg} $ and $ x ∈ A $. The space $ A^{\env} $ is also an $ A $-bimodule, but in two different ways. The default action is the outer action given by
\begin{equation*}
(a ¤ b).(x ¤ y) = (-1)^{|b||x|} ax ¤ yb, \quad a ¤ b ∈ A^{\env}, x ¤ y ∈ A^{\env}.
\end{equation*}
The alternative is the inner action given by
\begin{equation*}
(a ¤ b).(x ¤ y) = (-1)^{|b||x| + |a||b| + |a||x|} xb ¤ ay, \quad a ¤ b ∈ A^{\env}, x ¤ y ∈ A^{\env}.
\end{equation*}
\end{definition}

The dg category of dg modules over a dg algebra can be defined similar to the module category over an $ A_∞ $-algebra. If $ D $ is a dg algebra, we denote this dg category by $ \Mod D $. It is not the same as the ordinary category with hom spaces the morphisms of dg modules, but rather the differential measures failure to be a dg morphism. To define $ \Mod D $, one turns the dg algebra into an honest $ A_∞ $ by means of the sign flip $ μ^1 (a) = (-1)^{|a|} $ and $ μ^2 (a, b) = (-1)^{|b|} ab $ and then forms $ \Mod D $ according to the recipe presented in \autoref{sec:koszul-bimodules}. One may express the result more explicitly which we will not attempt here.

In particular, we can form the category $ \Mod A^{\env} $. When $ M ∈ \Mod_{A^{\env}} $, we can regard the hom space
\begin{equation*}
\Mdual M = \Hom_{\Mod A^{\env}} (M, A^{\env}).
\end{equation*}
For us, it is most important that $ \Mdual M $ is again a dg $ A^{\env} $-module, often called the dual bimodule of $ M $. Its differential is the differential $ μ^1_{\Mod A^{\env}} $ and the $ A^{\env} $-action on $ \Mdual M $ is the inner action on the codomain $ A^{\env} $. It is elementary to check that $ \Mdual M $ is indeed a dg module. According to Kontsevich and Soibelman, the bimodule $ \Mdual A $ is to be viewed as “inverse dualizing bibundle” $ F ↦ F ¤ K_X^{-1} [-\vdimension X] $ of the noncommutative manifold defined by $ A $ \cite[Definition 8.1.6]{Kontsevich-Soibelman-NC}.

When $ M $ is a dg module, the $ n $-th left shift $ M[n] $ also becomes an object of $ \Mod A^{\env} $. Reflecting the definition of Calabi-Yau manifolds in the commutative world, the dg algebra $ A $ is called Calabi-Yau if the dual bimodule $ \Mdual A $ is quasi-isomorphic to a shift of $ A $:

\begin{definition}
Let $ A $ be a dg algebra. Then $ A $ is \emph{Calabi-Yau} of dimension $ n ≥ 1 $ (\CYn) if $ \Mdual A[n] $ and $ A $ are quasi-isomorphic in the category $ \Mod A^{\env} $.
\end{definition}

\begin{remark}
The original definition \cite[Definition 3.2.3]{Ginzburg-CY} requires that the quasi-isomorphism is a self-dual morphism. Van den Bergh has shown in \cite[Proposition C.1]{vdB-CY} that this condition is typically automatic.
\end{remark}

The definition of the hom space $ \Hom_{A^{\env}} (A, A^{\env}) $ takes the degrees of the chosen resolution of $ A $ and the degrees of $ A^{\env} $ into account. In particular, requiring $ \Hom_{A^{\env}} (A, A^{\env}) $ and $ A[n] $ to be quasi-isomorphic cannot be easily translated into a property regarding resolutions on $ A $. If $ A $ is however an ordinary algebra (concentrated in degree zero), then the definition simplifies as follows:

\begin{lemma}[{\cite[(3.2.5)]{Ginzburg-CY}}]
Let $ A $ be an associative algebra which has a finite projective $ A $-bimodule resolution of finitely generated bimodules. Then $ A $ is Calabi-Yau of dimension $ n ≥ 1 $ if and only if
\begin{equation}
\label{eq:koszul-cy-ordcrit}
\HH^k (A, A^{\env}) ≅ \begin{cases} A & \text{if } k = n, \\ 0 & \text{else}. \end{cases}
\end{equation}
Here $ \HH^k (A, A^{\env}) $ is equipped again with the inner $ A $-bimodule action and the isomorphism is meant as $ A $-bimodules.
\end{lemma}

\begin{proof}
By definition $ \HH^k (A, A^{\env}) $ is the cohomology of $ \Hom_{A^{\env}} (A, A^{\env}) $, together with the additional $ A $-bimodule action. We now prove both directions.

If $ A $ is \CYn\, then $ \Hom_{A^{\env}} (A, A^{\env}) $ and $ A[n] $ are quasi-isomorphic as objects of $ \Mod A^{\env} $. There exist closed morphisms $ f $ and $ g $ between them and a morphism $ h $ of degree $ -1 $ such that $ fg = \id + d(h) $. In particular $ f $ and $ g $ define quasi-inverse morphisms of complexes and we conclude the cohomology of $ \Hom_{A^{\env}} (A, A^{\env}) $ is $ A $, concentrated in degree $ n $.

Conversely, assume \eqref{eq:koszul-cy-ordcrit} holds. The conclusion is a general statement regarding complexes. Let $ P^• $ be a complex with homology concentrated in a single degree, namely $ φ: A \isoto \H^k P $. Then there is a quasi-isomorphism $ ψ: A → P^• $ of complexes given by $ a ↦ φ(a) ∈ \H^k P ⊂ P^k $. This map is indeed a map of chain complexes since the differential on $ A $ and on the image $ ψ(A) $ vanishes. The map is a quasi-isomorphism since on cohomology this map is just $ φ $. This finishes the proof.
\end{proof}

\begin{remark}
\label{rem:koszul-bimodules-AAopp}
In contrast to dg algebras, two $ A_∞ $-algebras can not be tensored easily to form a tensor $ A_∞ $-algebra. Instead, the construction is difficult and has attracted various literature \cite{Loday, Markl-Shnider}. This includes an advanced construction of the tensor product $ A ¤ A^{\algopp} $ as $ A_∞ $-algebra.

This way, $ A ¤ A^{\algopp} $ itself becomes an $ A $-bimodule. This construction should not be confused with the naive bimodule action of $ A $ mentioned in \cite[Theorem 1.8]{Mescher}. In this naive action, $ A $ acts on the two tensor components of the graded vector space $ A ¤ A^{\algopp} $ separately. In the present section, we avoid all difficulties by restricting to the case of $ A $ being a dg algebra whenever we need the algebra $ A^{\env} $ or $ A ¤ A^{\algopp} $ as $ A $-bimodule.
\end{remark}


In the definition of a category of modules $ \Mod A $, one is not limited to choosing the specific resolution $ T(A[1]) ¤ M $ of $ M $. Instead one may choose an arbitrary “projective replacement”. In case $ A $ is an $ A_∞ $-algebra, this is a topic of research, but if $ A $ is a dg algebra, then the correct notion of “projective replacement” is to be K-projective:

\begin{definition}[{\cite[Section 4.2]{LPWZ-Koszul}}]
Let $ A $ be a dg algebra. Then a dg module $ P $ is \emph{K-projective} if the dg hom space $ \Hom_{\Mod A} (P, Q) $ is acyclic whenever $ Q $ is an acyclic dg module. A \emph{K-projective replacement} of $ M ∈ \Modfd A $ is a K-projective dg module $ P $ together with a quasi-isomorphism $ P → M $.
\end{definition}

\begin{remark}
Keller shows in \cite[Section 3.2]{Keller-onDG} that K-projective replacements can be found by resolving a module into a complex of so-called dg-projective modules. A module $ P $ is called dg-projective if it is a cofibrant object with respect to the projective model structure on the category of dg categories. Explicitly, $ P $ is dg-projective if for every surjective quasi-isomorphism $ L \projects M $, every morphism $ P → M $ factors through $ L $.
\end{remark}

\begin{remark}
The hom spaces in $ \Modfd A $ enjoy various names in the literature, thanks to the fact that their cohomology can also be built as derived hom functor. Generally, the cohomology hom spaces $ \H^•\Hom_{\Modfd A} (M, N) $ may be denoted $ \RHom_A^• (M, N) $ or $ \Ext_A^• (M, N) $. When $ M = N $, the dg hom space $ \Hom_{\Modfd A} (M, M) $ is actually a dg algebra itself and its minimal model becomes an $ A_∞ $-algebra. This $ A_∞ $-algebra is commonly denoted $ \RHom_A^• (M, M) $ as well, despite the fact that classical derived hom functors do not retain homotopy information.
\end{remark}

\begin{remark}
If $ A $ is a dg algebra, then the Hochschild cohomology $ \HH^• (A, M) $ with coefficients in an $ A $-bimodule $ M $ is by definition equal to $ \H^•\Hom_{\Mod A^{\env}} (A, M) $.
\end{remark}

%% file: prelim-koszul/cy_ord.tex
\subsection{Van den Bergh and Serre duality}
\label{sec:koszul-cyord}
In this section, we recall Van den Bergh and Serre duality for ordinary (non-dg) Calabi-Yau algebras. As it turns out, the Calabi-Yau property for ordinary algebras is closely related to having a self-dual bimodule resolution. This fact has been observed and exploited in \cite{Bocklandt-CY}.

\begin{definition}[{\cite{Bocklandt-CY}}]
\label{def:prelim-cy-selfdual}
Let $ A $ be an algebra and $ 0 → P^n → … → P^0 → A $ be a projective $ A $-bimodule resolution of $ A $ by finitely generated bimodules. Then $ P^• $ is a \emph{self-dual} resolution if $ \Hom_{A^{\env}} (P^•, A^{\env}) \cong P^{n-•} $ as complexes of $ A $-bimodules:
\begin{equation*}
\begin{tikzcd}
0 \arrow[r] & \Hom_{A^{\env}} (P^0, A^{\env}) \arrow[d, "\rotatebox{90}{$ \sim $}"] \arrow[r] & … \arrow[r] & \Hom_{A^{\env}} (P^n, A^{\env}) \arrow[r] \arrow[d, "\rotatebox{90}{$ \sim $}"] & 0 \\
0 \arrow[r] & P^n \arrow[r] & … \arrow[r] & P^0 \arrow[r] & 0
\end{tikzcd}
\end{equation*}
\end{definition}

\begin{lemma}
If $ A $ has a selfdual projective bimodule resolution of length $ n ≥ 1 $ then $ A $ is \CYn.
\end{lemma}

\begin{proof}
Let $ P^• $ be a selfdual bimodule resolution. By definition, the Hochschild cohomology $ \HH^k (A, A^{\env}) $ is the homology in degree $ k $ of the complex $ \Hom_{A^{\env}} (P^•, A^{\env}) $. By selfduality, this homology is just the homology in degree $ n-k $ of $ P^• $, which is $ A $ if $ k = n $ and zero otherwise. This shows that $ A $ is \CYn.
\end{proof}


A bimodule resolution of $ A $ gives rise to resolutions for all left and right $ A $-modules. The idea is to tensor the resolution on the right or left side with $ M $, respectively:

\begin{lemma}[{\cite[Lemma 2.4]{Berger-Taillefer}}]
\label{th:koszul-cyord-tensorresolution}
Let $ A $ be an algebra and $ P^• → A $ a projective bimodule resolution. Let $ M $ be a left $ A $-module. Then $ P^• ¤_A M → M $ is a resolution for $ M $.
\end{lemma}


Van den Bergh duality is a Poincaré-style theorem for Hochschild homology and cohomology of \CYn\ algebras. This duality was first observed in \cite{vdB-duality} and we recall it as follows:

\begin{theorem}[{\cite{vdB-duality}}]
Let $ A $ be a \CYn\ algebra. If $ A $ has a finite projective resolution of finitely generated bimodules, then
\begin{equation*}
\HH^k (A, M) ≅ \HH_{n-k} (A, M).
\end{equation*}
\end{theorem}

\begin{proof}
We recall here the proof in the easy case where $ A $ has a selfdual resolution $ P^• $. We compute
\begin{align*}
\HH^k (A, M) &= \H^k\Hom_{A^{\env}} (P^•, M) \\
&= \H^k\big(\Hom_{A^{\env}} (P^•, A^{\env}) ¤_{A^{\env}} M\big) \\
&= \H^k (P^{n-•} ¤_{A^{\env}} M) = \HH_{n-k} (A, M).
\end{align*}
\end{proof}


As we recapitulate in \autoref{th:prelim-cy-ext}, a \CYn\ algebra has the property that the $ n $-th shift is a Serre functor for its derived category. In \cite{Bocklandt-CY}, \CYn\ algebras were precisely defined by this characteristic property. It is unclear to which extent the definitions are equivalent.

\begin{lemma}[{\cite[Proposition 2.3]{Berger-Taillefer}}]
\label{th:prelim-cy-ext}
Let $ A $ be an algebra with a finite projective resolution of finitely generated bimodules. If $ A $ is \CYn, then for all finite-dimensional $ A $-modules $ M, N $ there are natural isomorphisms of graded vector spaces
\begin{equation*}
\kdual{\Ext_A^• (M, N)} \cong \Ext_A^{n-•} (N, M).
\end{equation*}
\end{lemma}

\begin{proof}
We recall here the proof from \cite{Berger-Taillefer}. The first step of the proof is to realize that thanks to \autoref{th:koszul-cyord-tensorresolution}, $ \Ext^k_A (M, N) $ equals the Hochschild cohomology $ \HH^k (A, \Hom_{ℂ} (M, N)) $:
\begin{equation*}
\Ext^k_A (M, N) = \H^k\Hom_A (P^• ¤ M, N) = \H^k\Hom_{A^{\env}} (P^•, \Hom_{ℂ} (M, N)) = \HH^k (A, \Hom_{ℂ} (M, N)).
\end{equation*}
Here $ \Hom_{ℂ} (M, N) $ is the $ A $-bimodule with the left factor acting on $ N $ and the right factor acting on $ M $: $ (aφb)(m) = aφ(bm) $.

The second step is to show $ \HH^k (A, \kdual M) ≅ \kdual{\HH_k (A, M)} $ for finite-dimensional $ A $-bimodules $ M $. Here the vector space dual $ \kdual M $ is also an $ A $-bimodule by letting the left factor act from the right on $ M $ and the right factor act from the left on $ M $, explicitly $ (aφb)(m) = φ(bma) $.
We compute
\begin{equation*}
\HH^k (A, \kdual M) = \H^k\Hom_{A^{\env}} (P^•, \Hom_{ℂ} (M, ℂ)) = \Hom_{ℂ} (P^• ¤_{A^{\env}} M, ℂ) = \kdual{\HH_k (A, M)}.
\end{equation*}
Combining the first two steps with Van den Berg duality, we conclude
\begin{align*}
\Ext_A^k (M, N) &≅ \HH^k (A, \Hom_{ℂ} (M, N)) \\
&= \HH^k (A, \kdual{\Hom_{ℂ} (N, M)}) \\
&≅ \kdual{\HH_k (A, \Hom_{ℂ} (N, M))} \\
&≅ \kdual{\HH^{n-k} (A, \Hom_{ℂ} (N, M))} \\
&≅ \kdual{\Ext^{n-k}_A (N, M)}.
\end{align*}
This finishes the proof.
\end{proof}

\begin{remark}
The duality between the two Ext spaces can also be interpreted as a pairing
\begin{equation*}
⟨-, -⟩: \Ext_A^• (M, N) × \Ext_A^{n-•} (N, M) → ℂ.
\end{equation*}
This pairing can also be expressed by means of the traces $ \Tr_M $, $ \Tr_N $ on $ \Ext_A^• (M, M) $ and $ \Ext_A^• (N, N) $ as
\begin{equation*}
⟨x, y⟩ = \Tr_N (x ∘ y) = (-1)^{|x||y|} \Tr_M (y ∘ x), \quad x ∈ \Ext_A^• (M, N), \quad y ∈ \Ext_A^{n-•} (N, M).
\end{equation*}
Here $ ∘ $ denotes the composition of extensions, equivalently the product of $ \H\Mod A $. The trace $ \Tr_M: \Ext_A^• (M, M) → ℂ $ is induced from the trace on the complex $ \Hom_A (P^• ¤ M, P^• ¤ M) $ which computes $ \Ext^• (M, M) $. The correct signs were computed in \cite[Appendix]{Bocklandt-CY}.
\end{remark}

%% file: prelim-koszul/cy3.tex
\subsection{Jacobi algebras}
\label{sec:koszul-cy3}
In this section, we recall characterizations of ordinary (non-dg) CY3 algebras. The core idea is that most CY3 algebras are the Jacobi algebra of a quiver with superpotential. Conversely, most quivers with superpotential gives rise to a CY3 algebra. In the present section, we recall the precise criteria from \cite{Bocklandt-CY}. In particular, we recall superpotentials and their associated Jacobi algebras as well as a candidate bimodule resolution for Jacobi algebras. This section serves as a direct preparation for \autoref{sec:flatness}.

We start by fixing terminology for cyclic elements of quiver algebras as follows:

\begin{definition}
Let $ Q $ be a quiver. A path in $ Q $ is a cycle if it starts and ends at the same vertex. If $ p $ is a cycle in $ Q $, we denote by $ p_{\cyc} ∈ ℂQ $ the sum of its cyclic permutations. We extend this assignment linearly to $ ℂQ $ and denote it by $ p ↦ p_{\cyc} $ as well. An element $ W ∈ ℂQ $ is \emph{cyclic} if it lies in the image of this map. Explicitly, $ W $ is cyclic if it is a linear combination of cycles whose coefficients are invariant under cyclic permutation:
\begin{equation*}
W = \sum_{\text{cycles } a_k … a_1} λ_{a_k … a_1} a_k … a_1, \quad \text{with} \quad ∀i = 1, …, k: ~~ λ_{a_k … a_1} = λ_{a_{i-1} … a_{i+1} a_i}.
\end{equation*}
\end{definition}

A superpotential on a quiver $ Q $ is defined as a cyclic element $ W ∈ ℂQ $ which is a linear combination of paths of length at least two:

\begin{definition}
A \emph{superpotential} is a cyclic element $ W ∈ ℂQ_{≥2} $. Its \emph{relations} are the elements
\begin{equation*}
∂_a W = \sum_{\substack{\text{paths } a_k … a_1 \\ \text{with } a_k = a}} λ_{a_k … a_1} a_{k-1} … a_1, \quad a ∈ Q_1.
\end{equation*}
Its \emph{Jacobi algebra} is given by
\begin{equation*}
\Jac(Q, W) = \frac{ℂQ}{(∂_a W)}.
\end{equation*}
Here $ (∂_a W) $ denotes the two-sided ideal generated by the partial derivatives $ ∂_a W $ for $ a ∈ Q_1 $.
\end{definition}

\begin{remark}
A typical assumption in the literature is that the paths are of length at least three. In fact, length two term gives rise to a single arrow being contained in one relation $ ∂_a W $. The effect is that this arrow is killed in the Jacobi algebra.
\end{remark}

The original paper of Ginzburg \cite{Bocklandt-CY} formulated the expectation that all CY3 algebras “appearing in nature” are Jacobi algebras of quivers with superpotential. This expectation was largely verified in \cite{Bocklandt-CY}, with the core result that a quiver algebra with graded relations which is CY3 is necessarily of the form $ \Jac(Q, W) $:

\begin{theorem}[{\cite[Theorem 3.1]{Bocklandt-CY}}]
\label{th:koszul-cy3-graded}
Let $ Q $ be a quiver and let $ A = ℂQ / I $ be the quotient by a finitely generated graded ideal $ I ⊂ ℂQ_{≥2} $. If $ A $ is CY3, then there exists a superpotential $ W $ such that $ ℂQ/I ≅ \Jac(Q, W) $.
\end{theorem}

The idea of the proof is to explore the structure of $ I $ in terms of resolutions for the simple modules of $ A $. These first bits of these resolutions are standard and do not depend on the ideal $ I $. In contrast, the last bits depend on $ I $ but can be guessed by consideration on the dimension of the Ext spaces based on the assumption that $ A $ is CY3.

Conversely, not every Jacobi algebra is CY3. There is however a precise criterion due to \cite{Bocklandt-CY} as well. The criterion is formulated in terms of a “candidate” bimodule resolution for $ A $. We start with the following notation:

\begin{definition}
Let $ W ∈ ℂQ_{≥2} $ be a superpotential. Then the $ ℂQ_0 $-bimodule generated by $ W $ in $ ℂQ $ is denoted
\begin{equation*}
W = ℂ Q_0 W Q_0 = \bigoplus_{v ∈ Q_0} ℂvWv ⊂ ℂQ.
\end{equation*}
The relations space is denoted
\begin{equation*}
R = \vspan\{∂_a W \running a ∈ Q_1\}.
\end{equation*}
\end{definition}

To introduce the candidate bimodule resolution, let $ Q $ be a quiver, $ W ∈ ℂQ_{≥3} $ a superpotential. Let us temporarily write $ A = \Jac(Q, W) $ for the Jacobi algebra. The candidate bimodule resolution has the shape
\begin{equation}
\label{eq:koszul-cy3-resolution}
0 → A \tensor_{ℂQ_0} W \tensor_{ℂQ_0} A \overset{g_1}{→} A \tensor_{ℂQ_0} R \tensor_{ℂQ_0} A \overset{g_2}{→} A \tensor_{ℂQ_0} ℂQ_1 \tensor_{ℂQ_0} A \overset{g_3}{→} A \tensor_{ℂQ_0} A → A → 0.
\end{equation}

\begin{remark}
\label{rem:koszul-cy3-maps}
The maps in the sequence \eqref{eq:flatness-resolution} are described as follows:
\begin{itemize}
\item For the map $ g_1 $, let $ w = \sum_{i ∈ I} r_i a_i = \sum_{j ∈ J} b_i r_j ∈ W $, where all $ a_i $ an $ b_i $ are arrows and the $ r_i, r_j ∈ R $ are compatible elements. The map $ g_1 $ sends the element $ 1 ¤ w ¤ 1 $ to
\begin{equation*}
g_1 (1 ¤ w ¤ 1) = \sum_{i ∈ I} 1 ¤ r_i ¤ a_i - \sum_{j ∈ J} b_j ¤ r_j ¤ 1.
\end{equation*}
\item For the map $ g_2 $, let $ r ∈ R $ and for $ d ≥ 0 $ write $ r = \sum_{i ∈ I_d} p_i^{(d)} a_i^{(d)} q_i^{(d)} ∈ R $, where $ p_i^{(d)} $ are paths of length $ d $ and $ a_i^{(d)} $ are arrows. The map $ g_2 $ sends the element $ 1 ¤ r $ to
\begin{equation*}
g_2 (1 ¤ r ¤ 1) = \sum_{d ≥ 0} \sum_{i ∈ I_d} p_i^{(d)} ¤ a_i^{(d)} ¤ q_i^{(d)}.
\end{equation*}
\item The map $ g_3 $ is given by $ 1 ¤ a ¤ 1 ↦ a ¤ 1 - 1 ¤ a $.
\item The fourth map is simply the contraction map.
\end{itemize}
\end{remark}

The sequence \eqref{eq:flatness-resolution} is clearly a chain complex. Whether or not it is exact is an indicator on whether $ A $ is CY3 or not:

\begin{theorem}[{\cite[Theorem 4.3]{Bocklandt-CY}}]
Let $ Q $ be a quiver and $ W ∈ ℂQ_{≥3} $ a superpotential. Then the algebra $ \Jac(Q, W) $ is CY3 if and only if the sequence \eqref{eq:koszul-cy3-resolution} is a bimodule resolution for $ A $.
\end{theorem}

%% file: prelim-koszul/correspondence.tex
\subsection{Cyclic $ A_∞ $-algebras}
\label{sec:koszul-correspondence}
In this section, we recall the correspondence of cyclic $ A_∞ $-algebras and Calabi-Yau dg algebras via Koszul duality. In \autoref{th:prelim-cy-ext} and \ref{th:koszul-prop-Ext}, we have already seen that $ \koszul A $ is the endomorphism algebra of the simple module $ ℂ $ and that modules over a \CYn\ algebra enjoy Serre duality of degree $ n $. This is a strong indication that $ \H\koszul A $ has a cyclic structure of degree $ n $. As we recall in this section, this is the norm and cyclic $ A_∞ $-algebras correspond to Calabi-Yau algebras under Koszul duality:

\begin{center}
\begin{tikzpicture}
\path (0, 0) node[align=center] (A) {$ A $ \\ \textbf{Cyclic}} (8, 0) node[align=center] (B) {$ \koszul A $ \\ \textbf{Calabi-Yau}};
\path[draw, <->] ($ (A.east)!0.2!(B.west) $) -- ($ (A.east)!0.8!(B.west) $) node[midway, above] {Koszul};
\end{tikzpicture}
\end{center}

In this section, we start by recalling the notion of cyclic $ A_∞ $-algebras. Then we explain that their Koszul duals are so-called deformed dg-preprojective algebras, which are Calabi-Yau. This explains the correspondence of cyclic $ A_∞ $-algebras and Calabi-Yau dg algebras via Koszul duality. Our primary reference is \cite{vdB-CY}. We start by recalling cyclic $ A_∞ $-algebras, a generalizations of Frobenius algebras:

\begin{definition}
Let $ A $ be an $ A_∞ $-algebra. Then $ A $ is \emph{cyclic} of degree $ n $ if it is equipped with a bilinear nondegenerate pairing $ ⟨-, -⟩: A × A → ℂ $ of degree $ -n $ such that $ ⟨x, y⟩ = (-1)^{|x||y|} ⟨y, x⟩ $ and
\begin{equation*}
⟨μ^k (a_{k+1}, …, a_2), a_1⟩ = (-1)^{‖a_{k+1}‖ (‖a_k‖ + … + ‖a_1‖)} ⟨μ^k (a_k, …, a_1), a_{k+1}⟩.
\end{equation*}
\end{definition}

\begin{remark}
The pairing is a noncommutative and categorical manifestation of a symplectic form. This is one of the reasons that cyclic $ A_∞ $-algebras are dominant in the theory of Fukaya categories. In our context, the notion of cyclic Fukaya categories is reserved for the A-side of mirror symmetry, while Calabi-Yau algebras are reserved for the B-side of mirror symmetry.
\end{remark}

When discussing Koszul duality for cyclic $ A_∞ $-algebras, we frequently need to choose a basis compatible with the pairing $ ⟨-, -⟩ $. We recall this piece of linear algebra as follows:

\begin{lemma}
\label{th:koszul-correspondence-matrix}
Let $ V $ be a finite-dimensional graded vector space over $ ℂ $ and $ ⟨-, -⟩: V × V → ℂ $ a nondegenerate bilinear form of degree $ -n $ with $ ⟨x, y⟩ = (-1)^{|x||y|} ⟨y, x⟩ $. Then there is a graded basis for $ V $ in which $ ⟨-, -⟩ $ takes the form
\begin{align}
\label{eq:koszul-correspondence-matrix}
⟨-, -⟩ &= \left(\begin{array}{c|c|c}
\begin{matrix} & I_k \\ ± I_k & \end{matrix} && \\\hline
& I_l & \\\hline
&& \begin{matrix} & -I_m \\ I_m & \end{matrix}
\end{array}\right), \\
\nonumber ± I_k &= \diag((-1)^{(n+1)|x_1|}, …, (-1)^{(n+1)|x_k|}).
\end{align}
Here $ k = \vdimension V_{<n/2} $ and the first $ k $ basis elements $ x_1, …, x_k $ can be freely chosen. The matrix $ I_l $ appears only if $ n $ is even and $ n/2 $ is even. The symplectic part consisting of $ I_m $ and $ -I_m $ only appears if $ n $ is even and $ n/2 $ is odd. In these cases we have $ l = \vdimension V_{n/2} $ and $ m = \vdimension V_{n/2} / 2 $, respectively.
\end{lemma}

\begin{proof}
Let $ x_1, …, x_k $ be a given graded basis for $ V_{<n/2} $. The restricted pairing $ ⟨-, -⟩: V_{<n/2} × V_{>n/2} → ℂ $ gives an isomorphism $ φ: V_{>n/2} \isoto \kdual V_{<n/2} $ and the basis $ x_1, …, x_k $ gives rise to a dual basis $ \kdual x_1, …, \kdual x_i ∈ \kdual V_{<n/2} $. Pick $ x_j^* ∈ V_{>n/2} $ such that $ φ(x_j^*) = \kdual x_j $. Then
\begin{equation*}
⟨x_i, x_j^*⟩ = φ(x_j^*)(x_i) = \kdual x_j (x_i) = δ_{ij}, \quad ⟨x_j^*, x_i⟩ = (-1)^{(n-|x_j|)|x_i|} δ_{ij} = (-1)^{(n+1)|x_i|} δ_{ij}.
\end{equation*}
Picking $ x_1, …, x_k, x_1^*, …, x_k^* $ as first part of the basis sets up the first diagonal block of the matrix. If $ n $ is odd, then $ V = V_{<n/2} ⊕ V_{>n/2} $ and we are done. If $ n $ is even and $ n/2 $ is even, then $ ⟨-, -⟩ $ is a nondegenerate symmetric bilinear form on $ V_{n/2} $ and we can find a basis in which $ ⟨-, -⟩: V_{n/2} × V_{n/2} → ℂ $ is the identity matrix. If $ n $ is even and $ n/2 $ is odd, then $ ⟨-, -⟩: V_{n/2} × V_{n/2} → ℂ $ is a symplectic form on $ V_{n/2} $ and we can choose a symplectic basis. This finishes the proof.
\end{proof}

\begin{definition}
\label{def:koszul-correspondence-dualbasis}
Let $ V $ be a finite-dimensional graded vector space over $ ℂ $ and $ ⟨-, -⟩: V × V → ℂ $ a nondegenerate bilinear form of degree $ -n $. Choose a basis for $ V $ in which $ ⟨-, -⟩ $ takes the shape \eqref{eq:koszul-correspondence-matrix}:
\begin{equation*}
\begin{aligned}
& x_1, …, x_k, x_1^*, …, x_k^* && \quad \text{if } n \text{ odd}, \\
& x_1, …, x_k, x_1^*, …, x_k^*, x_{k+1}, …, x_{k+l} && \quad \text{if } n \text{ even}, ~ n/2 \text{ even}, \\
& x_1, …, x_k, x_1^*, …, x_k^*, x_{k+1}, …, x_{k+m}, x_{k+1}^*, …, x_{k+m}^* && \quad \text{if } n \text{ even}, ~ n/2 \text{ odd.}
\end{aligned}
\end{equation*}
Then we call the pair of sequences
\begin{equation*}
\begin{aligned}
& x_1, …, x_k && \text{and} \quad x_1^*, …, x_k^* && \text{if } n \text{ odd}, \\
& x_1, …, x_k, x_{k+1}, …, x_{k+l} && \text{and} \quad x_1^*, …, x_k^*, x_{k+1}^* ≔ x_{k+1}, …, x_{k+l}^* ≔ x_{k+l} && \text{if } n \text{ even}, ~ n/2 \text{ even}, \\
& x_1, …, x_k, x_{k+1}, …, x_{k+m} && \text{and} \quad x_1^*, …, x_{k+m}^* && \text{if } n \text{ even}, ~ n/2 \text{ odd}
\end{aligned}
\end{equation*}
a \emph{basis with dual basis} for $ V $ via $ ⟨-, -⟩ $.
\end{definition}

\begin{remark}
For any $ i $ we have $ |x_i^*| = n - |x_i| $, regardless of whether $ n $ is odd or even.
\end{remark}

Van den Bergh investigated the correspondence between cyclic $ A_∞ $-algebras and Calabi-Yau dg algebras via Koszul duality in \cite{vdB-CY}. The technical core result is a structure characterization of Calabi-Yau dg algebras. Roughly speaking, a dg algebra is \CYn\ if it is weakly equivalent to a so-called deformed dg-preprojective algebra $ Π(Q, n, W) $ for some quiver $ Q $ and superpotential $ W ∈ ℂQ_{≥3} $. Here $ W $ is required to satisfy $ \{W, W\} = 0 $ with respect to the so-called necklace bracket $ \{-, -\} $, on which we will not elaborate here. We recall the construction of the algebras $ Π(Q, n, W) $ as follows:

\begin{definition}
Let $ Q $ be a graded quiver, $ n ≥ 3 $ an integer and $ W ∈ ℂQ_{≥3} $ a superpotential homogeneous of degree $ -n+3 $ with respect to the grading on $ Q $. Assume that $ W $ is “graded cyclic”, instead of cyclic. Assume $ \{W, W\} = 0 $ with respect to the necklace bracket. Let $ \tilde Q $ be the double quiver, obtained from $ Q $ by inserting for every arrow $ a $ an arrow $ a^* $ in opposite direction. The arrow $ a^* $ is assigned degree $ -n+2-|a|_Q $. In the special case where an arrow $ a $ is a loop with odd degree $ |a|_Q = (2-n)/2 $, no arrow $ a^* $ is adjoined and one puts $ a^* = a $. Let $ \bar Q $ be the quiver obtained from $ \tilde Q $ by additionally inserting on every vertex a loop $ z_i $ of degree $ 1-n $. Then the \emph{deformed dg-preprojective algebra}
\begin{equation*}
Π(Q, n, W) = (\compl{ℂ\tilde Q}, d_Π)
\end{equation*}
is the dg algebra modeled on $ \compl{ℂ\tilde Q} $ with the following differential:
\begin{align*}
d_Π (a) &= (-1)^{(|a|+1)|a^*|} ∂_{a^*} W, \\
d_Π (a^*) &= (-1)^{|a|+1} ∂_a W, \\
d_Π (z_i) &= \sum_{h(a) = i} [a, a^*].
\end{align*}
\end{definition}

We recall Van den Bergh's structure result as follows:

\begin{theorem}[{\cite[Theorem 10.2.2]{vdB-CY}}]
\label{th:koszul-correspondence-classification}
Let $ Q $ be a graded quiver, $ n ≥ 3 $ an integer and $ W ∈ ℂ\tilde Q_{≥3} $ a superpotential with $ \{W, W\} = 0 $. Assume $ W $ is homogeneous of degree $ -n+3 $ with respect to the grading on $ Q $ and all arrows in $ Q $ lie in degree range $ [-n+2, 0] $. Then $ Π(Q, n, W) $ is \CYn.
\end{theorem}

Van den Bergh also shows the converse statement that every \CYn\ dg algebra under suitable finiteness conditions is equivalent to some deformed dg-preprojective algebra $ Π (Q, n, W) $. It is an easy consequence of \autoref{th:koszul-correspondence-classification} that the Koszul dual of a minimal cyclic $ A_∞ $-category is \CYn:

\begin{corollary}
Let $ A $ be a finite-dimensional augmented minimal $ A_∞ $-algebra concentrated in non-negative degrees with $ A^0 = ℂ\id $. Assume $ A $ is cyclic of degree $ n $. Then $ \koszul A $ is a deformed dg-preprojective algebra and is therefore \CYn.
\end{corollary}

\begin{proof}
The idea is to represent $ \koszul A $ as a deformed dg-preprojective algebra $ Π(Q, n, W) $ by letting the arrows of $ Q $ stand for basis elements of $ A $ and letting the superpotential $ W $ record the product $ μ $. We ignore signs.

The first step is to choose a basis for $ A $ by means of \autoref{th:koszul-correspondence-matrix}. Since $ A^0 = ℂ\id $, we can include the element $ \id $ in the the basis and obtain a basis with dual basis of the following form:
\begin{equation*}
A = ℂ\id ⊕ \vspan\{x_1, …, x_k\} ⊕ \vspan\{x_1^*, …, x_k^*\} ⊕ ℂ\coid.
\end{equation*}
Of course, we have to make adaptations in case $ n $ is even: There may be self-dual basis elements of degree $ n/2 $. This is not a problem and is adequately reflected in the definition of $ Π(Q, n, W) $ by the intricate extra condition on loops. We shall for simplicity proceed with the assumption that $ n $ is odd.

We define the quiver $ Q $ to be a single vertex with $ k $ arrows, such that the double quiver $ \bar Q $ and the extended quiver $ \tilde Q $ satisfy
\begin{align*}
\compl{ℂ\tilde Q} &= ℂ \ncpow{\kdual x_1, …, \kdual x_k, \kdual{(x_1^*)}, …, \kdual{(x_k^*)}}, \\
\compl{ℂ\bar Q} &= ℂ \ncpow{\kdual x_1, …, \kdual x_k, \kdual{(x_1^*)}, …, \kdual{(x_k^*)}, \kdual{(\coid)}}.
\end{align*}
In these graded algebras, the variables $ \kdual x_i $ and $ \kdual{(x_i^*)} $ have degrees $ 1 - |x_i| $ and $ 1-|x_i^*| $, respectively. The variable $ \kdual{(\coid)} $ has degree $ 1-n $.

We define the superpotential $ W ∈ \compl{ℂ\tilde Q} $ by Koszul transforming the $ A_∞ $-structure on $ μ $:
\begin{equation*}
W = \sum_{\substack{1 ≤ i_0, i_1, …, i_l ≤ k \\ 2^{l+1} \text{ star options}}} \underbrace{⟨μ(x_{i_l}^{[*]}, …, x_{i_1}^{[*]}), x_{i_0}^{[*]}⟩}_{∈ ℂ} · \underbrace{\kdual{(x_{i_l}^{[*]})} … \kdual{(x_{i_1}^{[*]})} \kdual{(x_{i_0}^{[*]})}}_{∈ ℂ\tilde Q}.
\end{equation*}
Here we sum over choices of indices $ i_0, …, i_l $ as well as choices of starred or non-starred basis elements and variables. Minimality of $ A $ ensures that $ W $ lies in $ ℂQ_{≥3} $ and cyclicity of $ A $ ensures that $ W $ is graded cyclic. The $ A_∞ $-relations for $ A $ ensure that $ \{W, W\} = 0 $.

The quiver $ Q $, the cyclicity degree $ n $ and the superpotential $ W $ now yield a deformed dg-projective algebra $ Π(Q, n, W) $. Explicitly, its differential reads
\begin{align*}
d_Π \kdual x_i &= (-1)^{(|\kdual x_i| + 1)|\kdual{(x_i^*)}} ∂_{\kdual{(x_i^*)}} W, \\
d_Π \kdual{(x_i^*)} &= (-1)^{|\kdual x_i| + 1} ∂_{\kdual x_i} W, \\
d_Π \kdual{(\coid)} &= \sum_{i = 1}^k \kdual x_i \kdual{(x_i^*)} - \kdual{(x_i^*)} \kdual x_i.
\end{align*}
In the remainder of the proof, we check that $ Π(Q, n, W) $ equals $ \koszul A $. In fact, we have already named all the variables $ 2k+1 $ variables in $ \compl{ℂQ} $ precisely the way in which variables for $ \koszul A $ are named.

To see that also the differentials $ d_Π $ and $ d_{\koszul A} $ agree, note that the derivative $ ∂_{\kdual{x_i}} W $ precisely records the $ x_i^* $ coefficient of all possible products $ μ(…) $ of which the input is not the identity $ \id $ or the co-identity $ \coid $. In principle, the differential of $ \koszul A $ also records those products $ μ(…) $ which contain a co-identity, however those products vanish:
\begin{equation*}
|μ^l(x_{i_l}^{[*]}, …, \coid, …, x_{i_1}^{[*]})| ≥ (2-l) + (l-1)+ n = n+1.
\end{equation*}
Here we have used the assumption that $ A $ is concentrated in non-negative degrees and the co-idenity is in degree $ n $. We draw the conclusion that such products vanish because $ A $ is limited to degrees at most $ n $. Ultimately, $ d \kdual{(x_i^*)} = ∂_{\kdual x_i} $ agrees with the differential of $ \kdual{(x_i^*)} $ in $ \koszul A $.

Also on the variable $ \kdual{(\coid)} $, the differentials of $ Π(Q, n, W) $ and $ \koszul A $ agree. Indeed, the differential of $ \koszul A $ records appearances of $ \coid $ as result of products $ μ(…) $. Thanks to cyclicity, we have
\begin{align*}
⟨μ^{l ≥ 3} (a_l, …, a_1), \id⟩ &= 0, \\
⟨μ^2 (a, b), \id⟩ + ⟨μ^2 (b, \id), a⟩ + ⟨μ^2 (\id, a), b⟩ &= 0.
\end{align*}
When plugging in $ a = x_i $ and $ b = x_j^* $, we precisely obtain $ d_{\koszul A} (\kdual{(\coid)}) = d_Π (\kdual{(\coid)}) $. Ultimately, we conclude that $ \koszul A = Π(Q, n, W) $. Thanks to \autoref{th:koszul-correspondence-classification}, the algebra $ Π(Q, n, W) $ is \CYn\ and therefore $ \koszul A $ is \CYn\ itself. This finishes the proof.
\end{proof}

%% file: prelim-koszul/chl.tex
\subsection{Cho-Hong-Lau roadmap}
\label{sec:koszul-transfer}
In this section, we motivate the Cho-Hong-Lau construction from the perspective of Koszul duality. The idea is to pass the dg algebra $ \koszul A $ to cohomology. Regarded as ordinary algebra, the cohomology $ \H^0\koszul A $ need not be a Calabi-Yau algebra itself. Under grading assumptions on $ A $, it is however a Jacobi algebra and therefore a candidate to be CY3. We explain how to drop the grading requirements so that one can start from $ A $ being $ ℤ/2ℤ $-graded. The Koszul duality functor only survives this drop of requirements when we replace the actual cohomology $ \H^0 \koszul A $ by a surrogate. At the end of the section, we provide this tweaked Koszul duality functor which comes close to the Cho-Hong-Lau construction.

Our first starting point is an $ A_∞ $-category $ A $ which is cyclic of degree $ n $, concentrated in non-negative degree and has degree zero part $ A^0 = ℂ \id $:
\begin{equation*}
A = ℂ\id ⊕ A^1 ⊕ … ⊕ A^{n-1} ⊕ ℂ\coid.
\end{equation*}
Our interest lies in the Koszul duality functor $ \rModfd A → \Tw\koszul A $ which we recalled in \autoref{sec:koszul-koszul}. The essential step leading to the Cho-Hong-Lau construction consists of replacing the dg algebra $ \koszul A $ by its zeroth cohomology $ \H^0 \koszul A $, which is an ordinary associative algebra. In this section, we provide a simple explanation why this works particularly well if $ A $ is cyclic of degree 3.

We start with the observation that passing to the minimal model is unproblematic if $ A $ is concentrated in non-negative degree. We denote the minimal model of the dg algebra $ \koszul A $ by $ \H\koszul A $. Its degree zero part is $ \H^0\koszul A $.

\begin{lemma}
\label{th:koszul-transfer-pass}
Let $ A $ be a non-negatively graded augmented $ A_∞ $-algebra such that $ A^0 = ℂ\id $. Then there is an $ A_∞ $-morphism $ \koszul A → \H^0 \koszul A $.
\end{lemma}

\begin{proof}
By definition of the minimal model $ \H\koszul A $, there is an $ A_∞ $-quasi-isomorphism $ \koszul A → \H\koszul A $. In the remainder of the proof, we show that the projection $ π_0: \H\koszul A → \H^0\koszul A $ is an $ A_∞ $-morphism. Regard homogeneous elements $ a_1, …, a_k ∈ \H\koszul A $. Writing out the definition of $ μ_{\H^0\koszul A} (π_0 a_k, …, π_0 a_1) $, it is our task to show
\begin{equation}
\label{eq:koszul-transfer-auxfunctor}
π_0 μ_{\H\koszul A} (a_k, …, a_1) = \begin{cases} μ_{\H\koszul A} (π_0 a_2, π_0 a_1) & \text{if } k = 2 \\ 0 &\text{if } k ≠ 2. \end{cases}
\end{equation}
The basic observation is that $ \koszul A $ is concentrated in non-positive degrees because $ \bar A $ is concentrated in positive degrees. The cohomology $ \H\koszul A $ is then also concentrated in non-positive degrees. To check \eqref{eq:koszul-transfer-auxfunctor}, we distinguish four cases according to the value of $ k $ and the degrees of the input elements $ a_1, …, a_k $.

If $ k = 1 $, both sides vanish. If $ k = 2 $ and $ a_1, a_2 $ are both of degree zero and \eqref{eq:koszul-transfer-auxfunctor} holds. If $ k = 2 $ and one of $ a_1, a_2 $ is of negative degree, then $ μ_{\H\koszul A} (a_2, a_1) $ is of negative degree and both sides of \eqref{eq:koszul-transfer-auxfunctor} vanish. If $ k ≥ 3 $, then $ μ_{\H\koszul A} (a_k, …, a_1) $ has degree at most $ 2-k ≤ -1 $ and therefore both sides vanish as well. This finishes the proof.
\end{proof}

We borrow the following notation from \autoref{def:flatness-ideals-IP}.

\begin{definition}
Let $ V $ be a vector space and $ R ⊂ V $ be a subspace. Then $ I(R)_{\compl{T(V)}} $ is the ideal of $ \compl{T(V)} $ given by the image of the map $ \compl{T(V)} \htensor R \htensor \compl{T(V)} → \compl{T(V)} $. Explicitly, $ I(R)_{\compl{T(V)}} $ consists of all elements that can be written as a series
\begin{equation*}
\sum_{i = 0}^∞ p_i r_i q_i, \quad p_i ∈ V^{¤ → ∞}, \quad r_i ∈ R, \quad q_i ∈ V^{¤ → ∞}.
\end{equation*}
\end{definition}

\begin{lemma}
\label{th:koszul-transfer-cohI}
Let $ A $ be an augmented $ A_∞ $-algebra concentrated in non-negative degrees with $ A^0 = ℂ\id $. Then we have an algebra identification
\begin{equation*}
\H^0\koszul A = \frac{\compl{T(\kdual{(A^1)})}}{I(\{d_{\koszul A} x \running x ∈ \kdual{(A^2)}\})_{\compl{T(\kdual{(A^1)})}}}.
\end{equation*}
\end{lemma}

\begin{proof}
Since $ \bar A $ is concentrated in positive degrees, its left-shift $ \bar A[1] $ is concentrated in non-negative degrees. We can write
\begin{equation*}
\koszul A = \prod_{i = 0}^∞ (\underbrace{\kdual{(A^1)}}_{\text{deg } 0} ⊕ \underbrace{\kdual{(A^2)}}_{\text{deg } -1} ⊕ …)^{¤ i}
\end{equation*}
In the remainder of the proof, we determine the degree zero cocycles and degree zero coboundaries of $ \koszul A $.

For the degree zero cocycles, we note that the degree zero part of $ \koszul A $ is $ \kdual{(A^1)} $ and the degree one part of $ \koszul A $ is zero. Together, the degree zero cocyles of $ \koszul A $ are just $ \kdual{T(A^1)} $.

For the degree zero coboundaries, we determine first the degree $ -1 $ part of $ \koszul A $. In fact, the part of degree $ -1 $ consists of products of degree zero elements with precisely one degree $ -1 $ element, more precisely
\begin{equation*}
\prod_{i = 0}^∞ \sum_{j = 0}^i \kdual{(A^1)}{}^{¤ j} ¤ \kdual{(A^2)} ¤ \kdual{(A^1)}{}^{¤ i-j}.
\end{equation*}
Now the degree zero coboundaries of $ \koszul A $ are precisely the differentials of elements in this space. Since $ d_{\koszul A} x = 0 $ for $ x ∈ \kdual{(A^1)} $, we immediately see that for $ x_1, …, x_i ∈ \kdual{(A^1)} $ and $ y ∈ \kdual{(A^2)} $ we have
\begin{equation*}
d_{\koszul A} (x_i ¤ … ¤ x_{i-j+1} ¤ y ¤ x_{i-j} ¤ … ¤ x_1) = x_i ¤ … ¤ x_{i-j+1} ¤ d_{\koszul A} y ¤ x_{i-1} ¤ … ¤ x_1.
\end{equation*}
We conclude that the degree zero coboundaries of $ \koszul A $ are precisely elements of $ I(\{d_{\koszul A} y \running y ∈ \kdual{(A^2)}\}) $. This finishes the proof.
\end{proof}

We shall continue writing $ I $ for the ideal defined in \autoref{th:koszul-transfer-cohI}:

\begin{definition}
\label{def:koszul-transfer-dropZ}
Let $ A $ be an augmented $ ℤ $-graded finite-dimensional $ A_∞ $-algebra cyclic of degree $ n ≥ 3 $. Denote by $ A^1 $ and $ A^2 $ the homogeneous subspaces of degree $ 1 $ and $ 2 $, respectively. Then we denote by $ I $ the ideal
\begin{equation*}
I = I(\{d_{\koszul A} x \running x ∈ \kdual{(A^2)}\})_{\compl{T(\kdual{(A^1)}}} ⊂ \compl{T(\kdual{(A^1)})}.
\end{equation*}
\end{definition}

It is in fact possible to say more about the ideal in case $ A $ is cyclic of degree $ 3 $:

\begin{lemma}
\label{th:koszul-transfer-3destiny}
Let $ A $ be an augmented finite-dimensional $ A_∞ $-algebra. Assume $ A $ is cyclic of degree $ 3 $ and concentrated in non-negative degrees with $ A^0 = ℂ\id $:
\begin{equation*}
A = ℂ\id ⊕ A^1 ⊕ A^2 ⊕ ℂ\coid.
\end{equation*}
Then we have a Koszul duality functor $ \rModfd A → \Tw\H^0\koszul A $. Upon choosing a basis $ x_1, …, x_k $ for $ A^1 $, the algebra $ \H^0\koszul A $ can be written as a Jacobi-type algebra and is naturally a candidate to be CY3:
\begin{equation}
\label{eq:koszul-koszul-quotientalg}
\H^0\koszul A = \frac{ℂ\ncpow{\kdual x_1, …, \kdual x_k}}{I(∂_{\kdual x_i} W)_{ℂ\ncpow{\kdual x_1, …, \kdual x_k}}}.
\end{equation}
\end{lemma}

\begin{proof}
We divide the proof into three steps. In the first step, we explain the Koszul duality functor. In the second step, we examine the algebra $ \H^0\koszul A $. In the third step, we comment on CY3-ness.

For the first step, pick the Koszul duality functor $ \rModfd A → \Tw\koszul A $ from \autoref{th:koszul-koszul-functor}. Thanks to \autoref{th:koszul-transfer-pass}, we have an additional morphism of $ A_∞ $-algebras $ \koszul A → \H^0 \koszul A $, which induces an $ A_∞ $-functor $ \Tw\koszul A → \Tw\H^0\koszul A $. Composing these two, we obtain the desired Koszul duality functor
\begin{equation*}
\rModfd A → \Tw\koszul A → \Tw\H^0 \koszul A.
\end{equation*}
For the second step of the proof, the starting point is the description $ \H^0\koszul A = \compl{T(\kdual{(A^1)})}/I $ from \autoref{th:koszul-transfer-cohI}. It is our task to examine the ideal $ I $. Choose a basis $ x_1, …, x_n $ for $ A^1 $ and denote by $ x_1^*, …, x_n^* $ the corresponding dual basis for $ A^2 $ via $ ⟨-, -⟩ $. We construct the superpotential $ W ∈ \compl{T(\kdual{(A^1)})} $ as follows:
\begin{equation*}
W = \sum_{1 ≤ i_1, …, i_k, i_0 ≤ n} ⟨μ(x_{i_k}, …, x_{i_1}), x_{i_0}⟩ \kdual x_{i_0} \kdual x_{i_1} … \kdual x_{i_k}.
\end{equation*}
This specific pairing has the chance of not vanishing only because $ n = 3 $. In fact, the degree of $ μ(…) $ is $ 2 $ while the degree of $ x_{i_0} $ is $ 1 $. The superpotential $ W $ is cyclic since $ μ $ is assumed to be cyclic. We now claim that $ ∂_{\kdual x_{i_0}} W = d_{\koszul A} \kdual{(x_{i_0}^*)} $. Indeed,
\begin{equation*}
∂_{\kdual x_{i_0}} W = \sum_{1 ≤ i_1, …, i_k ≤ n} ⟨μ(x_{i_k}, …, x_{i_1}), x_{i_0}⟩ \kdual x_{i_1} … \kdual x_{i_k} = d_{\koszul A} \kdual{(x_{i_0}^*)}.
\end{equation*}
This proves the desired description of $ \H^0\koszul A $. For the third part of the proof, we comment on the non-technical statement regarding the claimed CY3 candidate status of $ \H^0\koszul A $. As we have shown in the second part of the proof, the algebra $ \H^0 \koszul A $ is the quotient of a noncommutative power series ring by derivatives of a superpotential. This does technically not imply that $ \H^0\koszul A $ is CY3. However, it is known that if the number of variables and the degree of the superpotential are high enough, then the typical superpotential does turn the quotient into a CY3 algebra \cite[Corollary 4.4]{Bocklandt-CY}. This finishes the proof.
\end{proof}

\begin{remark}
It is essential that the cyclicity degree of $ A $ is $ 3 $: Assume $ A $ is cyclic of degree $ n $ instead. In order to represent $ d \kdual x $ as derivative of a superpotential $ W ∈ \compl{T(\kdual{(A^1)})} $, the only natural way is by $ d \kdual x $ being either the derivative $ ∂_{\kdual x} W $ or $ ∂_{\kdual{(x^*)}} W $. However, the variable $ \kdual x $ has degree $ -1 $ and the variable $ \kdual{(x^*)} $ has degree $ 1-(n-2) = 3-n $. We conclude that only for $ n = 3 $ any of the variables, namely $ \kdual{(x^*)} $, has a chance of appearing in $ W $. This shows that cyclicity in degree $ n = 3 $ is favorable for passing $ \koszul A $ to cohomology.
\end{remark}

The reduction of the Koszul dual to zeroth cohomology in \autoref{th:koszul-transfer-3destiny} makes room for further weakening of the assumptions on the side of the $ A_∞ $-algebra $ A $. As we show in \autoref{th:koszul-transfer-dropZ}, we can drop the requirement that $ A $ be $ ℤ $-graded and concentrated in non-negative degrees. The idea is to circumvent $ \koszul A $ and work directly with the quotient algebra displayed on the right-hand side of \eqref{eq:koszul-koszul-quotientalg}. When we restrict $ μ_A $ or $ μ_M $ to $ T(\bar A^1) $, we shall use the letter $ m $ instead of $ μ $:

\begin{definition}
Let $ A = ℂ\id ⊕ \bar A^1 ⊕ \bar A^2 ⊕ ℂ\coid $ be an augmented $ ℤ/2ℤ $-graded finite-dimensional $ A_∞ $-algebra cyclic of odd degree. We denote the restriction to $ T(A^1) $ of the map $ μ_A: M ¤ T(\bar A) → M ¤ T(\bar A) $ and its partial dual $ \kdual{μ}_A $ by
\begin{equation*}
\begin{aligned}
m_A & : & M ¤ T(\bar A^1) & → M ¤ T(\bar A^1), \\
\kdual m_{A, 0} & : & M & → M ¤ \kdual{T(\bar A^1)}.
\end{aligned}
\end{equation*}
Let $ M ∈ \rModfd A $ be a finite-dimensional right $ A $-module with action map $ μ_M: M ¤ T(A[1]) → M ¤ T(A[1]) $. Then we denote the restriction to $ T(A^1) $ of $ μ_M $ and $ \kdual{μ}_M $ by
\begin{equation*}
\begin{aligned}
m_M & : & M ¤ T(\bar A^1) & → M ¤ T(\bar A^1), \\
\kdual m_{M, 0} & : & M & → M ¤ \kdual{T(\bar A^1)}.
\end{aligned}
\end{equation*}
\end{definition}

We now formulate the Koszul duality where $ A $ is not required to be $ ℤ $-graded. From the standpoint of Koszul duality, this statement is the closest to the Cho-Hong-Lau construction that lies within the framework of finite-dimensional modules and twisted complexes:

\begin{lemma}
\label{th:koszul-transfer-dropZ}
Let $ A = ℂ\id ⊕ \bar A^1 ⊕ \bar A^2 ⊕ ℂ\coid $ be an augmented $ ℤ/2ℤ $-graded finite-dimensional $ A_∞ $-algebra cyclic of odd degree. Write $ J = \compl{T(\kdual{(A^1)})} / I $. Then we have a Koszul duality functor
\begin{align*}
F: \rModfd A &\verylongto \Tw J, \\
(M, μ_M) &\verylongmapsto (M ¤ J, \kdual m_{M, 0}), \\
f &\verylongmapsto \kdual f.
\end{align*}
\end{lemma}

\begin{proof}
We only provide a glimpse of the proof here, since the more general version is treated in \autoref{sec:CHL}. For instance, it is instructive to explain why $ (M, \kdual m_{M, 0}) $ is a twisted complex over $ J $.

In order to recognize $ (M, \kdual m_{M, 0}) $ as twisted complex over $ J $, we have to check its Maurer-Cartan identity. Since $ M $ is a module over $ A $, we have $ (m_M ∘ m_A)_0 + (m_M ∘ m_M)_0 = 0 $. The image of $ \kdual{(m_M ∘ m_A)_0} $ lies in $ M ¤ I $. This means that modulo $ M ¤ I $, we have
\begin{equation*}
μ^2_{\Add J} (\kdual m_{M, 0}, \kdual m_{M, 0}) = \kdual{(m_M ∘ m_M)_0} = 0.
\end{equation*}
This shows that $ \kdual m_{M, 0} $ satisfies the Maurer-Cartan equation in $ \Add J $ and $ (M ¤ J, \kdual m_{M, 0}) $ is indeed a twisted complex.
\end{proof}

\begin{remark}
The functor $ F $ in \autoref{th:koszul-transfer-dropZ} can be made explicit upon choice of basis for $ A $. Choose a basis of a consisting of elements $ x_i $ in odd degree and the dual elements $ x_i^* $ in even degree:
\begin{equation*}
A = \underbrace{ℂ\id}_{\text{even}} ⊕ \underbrace{\vspan\{x_1, …, x_k\}}_{\text{odd}} ⊕ \underbrace{\vspan\{x_1^*, …, x_k^*\}}_{\text{even}} ⊕ \underbrace{ℂ\coid}_{\text{odd}}.
\end{equation*}
The algebra $ J $ takes on the form
\begin{equation*}
J = \frac{ℂ \ncpow{\kdual x_1, …, \kdual x_k}}{I(∂_{\kdual x_i} W)_{ℂ \ncpow{\kdual x_1, …, \kdual x_k}}}.
\end{equation*}
Formulated in terms of the basis, the fact that $ \kdual m_M $ squares to zero is based on the observation that
\begin{align*}
0 =& \sum_{\substack{l ≥ 0 \\ 1 ≤ i_1, …, i_l ≤ k}} (μ · μ) (m, x_{i_l}, …, x_{i_1}) \\
=& \sum_{\substack{l ≥ 0 \\ 1 ≤ i_1, …, i_l ≤ k \\ 0 ≤ s ≤ r ≤ l}} μ(m, x_{i_l}, …, μ(x_{i_r}, …, x_{i_{s+1}}), …) + \sum_{\substack{l ≥ 0 \\ 1 ≤ i_1, …, i_l ≤ k \\ 0 ≤ r ≤ l}} μ(μ(m, x_{i_l}, …, x_{i_{r+1}}), …, x_1).
\end{align*}
Note that there are no signs since all $ x_i $ are odd. When dualizing the above equality, the first summand on the second row lands in the ideal $ I(∂_{\kdual x_i} W)_{ℂ\ncpow{x_1, …, x_k}} $ and vanishes in $ J $. Meanwhile, the second summand on the second row dualizes to $ (\kdual m_{M, 0})^2 $. We conclude that $ (\kdual m_{M, 0})^2 = 0 $, in other words $ (M ¤ J, \kdual m_M) $ is a twisted complex.
\end{remark}

The construction of \autoref{th:koszul-transfer-dropZ} leads directly to the Cho-Hong-Lau construction. The idea is to extend the construction of the Koszul duality functor $ \rModfd A → \Tw J $ to the case of non-augmented algebras. On the side of the Koszul dual, the required adaption consists of passing from the twisted completion to matrix factorizations. In \autoref{sec:CHL}, we recall this Cho-Hong-Lau construction and in particular amend all signs to the rules for $ A_∞ $-categories.

%% file: prelim-MS/intro.tex
\section{Preliminaries on mirror symmetry}
\label{sec:3prelim-MS}
In this section, we recollect preliminaries on mirror symmetry of punctured surfaces from \cite{Bocklandt}.

%% file: prelim-MS/gtl.tex
\subsection{Gentle algebras}
In this section, we recall the $ A_∞ $-gentle algebra from \cite{Bocklandt} and its deformed version from \paperone. In particular, we explain the philosophy of $ \Gtl_q Q $ as a discrete relative Fukaya category.

\begin{definition}
A \emph{punctured surface} is a closed oriented surface $ S $ with a finite set of punctures $ M ⊂ S $. We assume that $ |M| ≥ 1 $, or $ |M| ≥ 3 $ if $ S $ is a sphere.
\end{definition}

The assumptions on $ |M| $ are cosmetical and explained in \papertwoB.

\begin{definition}
Let $ (S, M) $ be a punctured surface. An \emph{arc} in $ S $ is a not necessarily closed curve $ γ: [0, 1] → S $ running from one puncture to another. An \emph{arc system} on a punctured surface is a finite collection of arcs such that the arcs meet only at the set $ M $ of punctures. Intersections and self-intersections are not allowed. The arc system satisfies the \emph{no monogons or digons} condition \emph{[NMD]} if
\begin{itemize}
\item No arc is a contractible loop in $ S \setminus M $.
\item No pair of distinct arcs is homotopic in $ S \setminus M $.
\end{itemize}
The arc system satisfies the \emph{no monogons or digons in the closed surface} condition \emph{[NMDC]} if
\begin{itemize}
\item No arc is a contractible loop in $ S $.
\item No pair of distinct arcs is homotopic in $ S $.
\end{itemize}
An arc system is \emph{full} if the arcs cut the surface into contractible pieces, which we call \emph{polygons}.
\end{definition}

\begin{definition}
A \emph{dimer} $ Q $ is a full arc system on a punctured surface such that every polygon is bounded by at least three arcs and the arcs along the boundary of a polygon are all oriented in the same direction. The letter $ Q_0 $ denotes the set of punctures of the dimer, the letter $ Q_1 $ denotes the set of arcs, and $ |Q| $ denotes the closed surface.
\end{definition}

\begin{remark}
All polygons in a dimer are either bounded entirely clockwise, or entirely anticlockwise. A polygon neighboring a clockwise polygon is anticlockwise, and the other way around. A dimer can be interpreted as a quiver embedded in a surface, therefore quiver terminology applies and we may for instance refer to paths of $ Q $.
\end{remark}

The \emph{gentle algebra} $ \Gtl Q $ is an $ A_∞ $-category defined as follows: Its objects are the arcs $ a ∈ Q_1 $. A basis for the hom space $ \Hom_{\Gtl Q} (a, b) $ is given by the set of all angles around punctures from $ a $ to $ b $. This includes \emph{empty angles}, which are the \emph{identities} on the arcs. The hom spaces of $ \Gtl Q $ are not finite-dimensional, in contrast to what is classically called a gentle algebra. The $ ℤ/2ℤ $-grading on $ \Gtl Q $ is given by declaring all interior angles of polygons to have odd degree. The differential $ μ^1 $ is plainly set to zero, and the product $ μ^2 (α, β) $ of two angles $ α, β $ is defined as the concatenation $ αβ $ of $ α $ and $ β $ if both angles wind around the same puncture and $ α $ starts where $ β $ ends:
\begin{equation*}
μ^1 ≔ 0, \quad μ^2 (α, β) ≔ (-1)^{|β|} αβ.
\end{equation*}
The higher products $ μ^{≥3} $ on $ \Gtl Q $ are defined in terms of what we will call discrete immersed disks. Roughly speaking, a discrete immersed disk may either be a polygon, or a sequence of polygons stitched together. We make this precise by regarding immersions of the standard polygon $ P_k $, depicted in \autoref{fig:prelim-gtl-P5}:

\begin{definition}
Let $ Q $ be dimer. A \emph{discrete immersed disk} in $ Q $ consists of an oriented immersion $ D: P_k → |Q| $ of a standard polygon $ P_k $ into the surface, such that
\begin{itemize}
\item The edges of the polygon are mapped to a sequence of arcs.
\item The immersion does not cover any punctures.
\end{itemize}
The immersion mapping itself is only taken up to reparametrization. The sequence of \emph{interior angles} of $ D $ is the sequence of angles in $ Q $ given as images of the interior angles of $ P_k $ under the map $ D $. An angle sequence $ α_1, …, α_k $ is a \emph{disk sequence} if it is the sequence of interior angles of some discrete immersed disk.
\end{definition}

\input{prelim-MS/fig_gtlq.tex}

We can now describe the higher products $ μ^{≥3} $ on $ \Gtl Q $ as follows:

\begin{definition}
Let $ Q $ be a dimer. Then the \emph{gentle algebra} $ \Gtl Q $ of $ Q $ is the $ A_∞ $-category with objects the arcs $ a ∈ Q_1 $, hom spaces spanned by angles, and $ A_∞ $-product $ μ $ defined by $ μ^1 = 0 $ and $ μ^2 (α, β) = (-1)^{|β|} αβ $. To define $ μ^{k≥3} $, let $ α_1, …, α_k $ be any disk sequence, let $ β $ be an angle composable with $ α_1 $, i.e.~$ β α_1 ≠ 0 $, and let $ γ $ be an angle post-composable with $ α_k $, i.e.~$ α_k γ ≠ 0 $. Then
\begin{equation*}
μ^k (β α_k, …, α_1) ≔ β, \quad μ^k (α_k, …, α_1 γ) ≔ (-1)^{|γ|} γ.
\end{equation*}
The higher products vanish on all angle sequences other than these.
\end{definition}

In \paperone\ we defined a deformation $ \Gtl_q Q $ of $ \Gtl Q $, hoping it to provide a discrete relative Fukaya category. In \papertwoB, we confirmed this hope. The starting point for the definition of $ \Gtl_q Q $ is a dimer which satisfies the [NMDC] condition. The category $ \Gtl_q Q $ is a curved $ A_∞ $-deformation of $ \Gtl Q $ over the deformation base $ B = ℂ⟦Q_0⟧ $, which has one deformation parameter per puncture. We denote the maximal ideal by $ \mathfrak{m} = (Q_0) ⊂ ℂ⟦Q_0⟧ $.

The curvature $ μ^0_q $ of $ \Gtl_q Q $ is defined as follows: For every puncture $ q ∈ Q_0 $, denote by $ ℓ_q $ the sum of all full turns around $ q $, summed over all arc ends at $ q $. The total curvature $ μ^0_q $ of $ \Gtl_q Q $ is defined as the sum over all puncture contributions:
\begin{equation*}
μ^0_q ≔ \sum_{q ∈ Q_0} q ℓ_q.
\end{equation*}
The product $ μ^1_q $ still vanishes and the product $ μ^2_q ≔ μ^2 $ is not deformed. The higher products $ μ^{≥3}_q $ however count discrete immersed disks which are now also allowed to cover punctures, weighting the result of every disk with the product $ ∈ ℂ⟦Q_0⟧ $ of the punctures covered. The definition of $ μ^0_q $ and $ μ^{≥3}_q $ is depicted in \autoref{fig:prelim-gtl-gtlq}.

%% file: prelim-MS/fig_gtlq.tex
\begin{figure}
\centering
\begin{subfigure}{0.3\linewidth}
\centering
\begin{tikzpicture}
\path[draw] (0, 0) node[below] {3} -- ++(right:1) node[midway, above] {\tiny 2} node[below] {2} -- ++(72:1) node[midway, left] {\tiny 1} node[right] {1} -- ++(144:1) node[midway, below] {\tiny 5} node[above] {5} -- ++(216:1) node[midway, below] {\tiny 4} node[left] {4} -- ++(-72:1) node[midway, right] {\tiny 3};
\end{tikzpicture}
\caption{Standard polygon $ P_5 $}
\label{fig:prelim-gtl-P5}
\end{subfigure}
\begin{subfigure}{0.3\linewidth}
\centering
\begin{tikzpicture}
\path[draw, ->] (0, 0) node[left] {$ q $} -- (2, 0) node[right] {$ p $} node[midway, above] {$ a $};
\path[fill] (0, 0) circle[radius=0.05] (2, 0) circle[radius=0.05];
\path[draw, ->] (2, 0) ++(-170:0.5) arc(-170:170:0.5) node[midway, right] {$ ℓ_p $};
\path[draw, ->] (0, 0) ++(10:0.5) arc(10:350:0.5) node[midway, left] {$ ℓ_q $};
\path (1, -1) node {$ μ^0_q = q ℓ_q + p ℓ_p $};
\end{tikzpicture}
\caption{Assigning curvature to an arc}
\end{subfigure}
\begin{subfigure}{0.3\linewidth}
\centering
\begin{tikzpicture}[scale=1]
\path[draw] (0, 0) -- ++(left:1) coordinate (A) -- ++(60:1) coordinate (B) coordinate[pos=0.3] (3-end) coordinate[pos=0.7] (4-start) -- ++(right:1) coordinate[pos=0.3] (4-end) coordinate[pos=0.7] (5-start) coordinate (C) -- ++(300:1) coordinate[pos=0.3] (5-end) coordinate[pos=0.7] (6-start) coordinate (D) -- ++(240:1) coordinate[pos=0.3] (6-end) coordinate[pos=0.7] (7-start) coordinate (E) -- ++(left:1) coordinate[pos=0.4] (1-end) coordinate[pos=0.3] (7-end) coordinate[pos=0.7] (2-start) coordinate[pos=0.6] (8-start) coordinate (F) -- ++(120:1) coordinate[pos=0.3] (2-end) coordinate[pos=0.7] (3-start);
\path[draw] (0, 0) -- (F) coordinate[pos=0.6] (8-end) coordinate[pos=0.3] (9-start);
\path[draw] (0, 0) -- (E) coordinate[pos=0.6] (1-start) coordinate[pos=0.3] (9-end);
\foreach \i in {2, 3, 4, 5, 6, 7} \path[draw, ->, bend right=60] (\i-start) to (\i-end);
\path[draw] (A) -- (D) (B) -- (E) (C) -- (F);
\path[fill] (0, 0) circle[radius=0.05];
\path (A) node[left] {$ α_1 $};
\path (B) node[left] {$ α_2 $};
\path (C) node[right] {$ α_3 $};
\path (D) node[right] {$ α_4 $};
\path (E) node[right] {$ α_5 $};
\path (F) node[left] {$ α_6 $};
\path (F) -- (A) node[midway, left] {$ a $};
\path (0, 0) node[above] {$ q $};
\path (0, -1.5) node {$ μ^6_q (α_6, …, α_1) = q \id_a $};
\end{tikzpicture}
\caption{Deformed product $ μ^6_q $}
\end{subfigure}
\caption{Illustration for $ \Gtl Q $ and $ \Gtl_q Q $}
\label{fig:prelim-gtl-gtlq}
\end{figure}
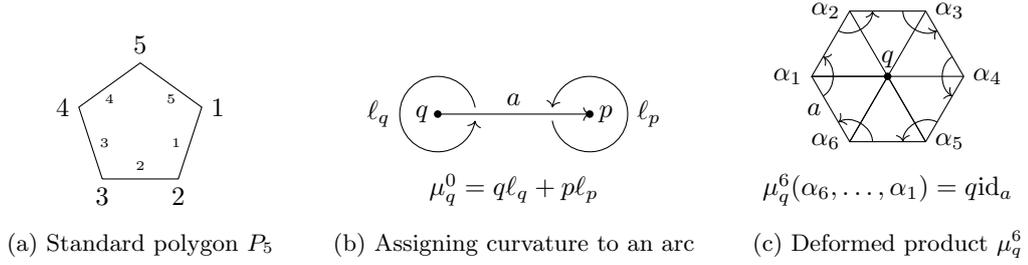

%% file: prelim-MS/zigzag.tex
\subsection{Zigzag paths}
\label{sec:prelim-zigzag}
In this section, we recall the notions of zigzag paths and geometric consistency for dimers. An exhaustive reference is \cite{Bocklandt-consistency}. More information is also found in \papertwoB.

\begin{definition}
Let $ Q $ be a dimer. A \emph{zigzag path} $ L $ is an infinite path $ … a_2 a_1 a_0 a_{-1} a_{-2} … $ of arcs in $ Q $ together with an alternating choice of “left” or “right” for every $ i ∈ ℕ $ such that
\begin{itemize}
\item $ a_{i+1} a_i $ lies in a clockwise polygon if $ i $ is assigned “right”,
\item $ a_{i+1} a_i $ lies in a counterclockwise polygon if $ i $ is assigned “left”.
\end{itemize}
We also say that $ L $ \emph{turns left} at $ a_i $ if $ i $ is assigned “left” and \emph{turns right} if $ a_i $ is assigned “right”. Two zigzag paths are identified if their paths including left/right indications differ only by integer shift.
\end{definition}

\begin{figure}
\centering
\begin{tikzpicture}
\path[draw, -{To[scale=2]}] (0, 0) -- ++(315:1.5) node[near end, left] {$ a_1 $} coordinate[midway] (alpha1-end) -- ++(45:1.5) node[near start, right] {$ a_2 $} coordinate[midway] (alpha1-start) -- ++(315:1.5) node[near end, left] {$ a_3 $} coordinate[midway] (alpha3-end) -- ++(45:1.5) node[near start, right] {$ a_4 $} coordinate[midway] (alpha3-start) -- ++(315:1.5) node[near end, left] {$ a_5 $} coordinate[midway] (alpha5-end) -- ++(45:1.5) node[near start, right] {$ a_6 $} coordinate[midway] (alpha5-start) -- ++(315:1.5) node[near end, left] {$ a_1 $} coordinate[midway] (additional);
\end{tikzpicture}
\caption{A zigzag path $ L $}
\label{fig:prelim-zigzag-sketch}
\end{figure}
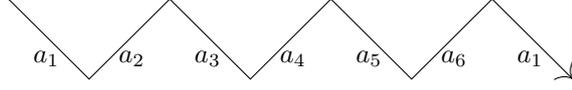

Slightly simplified, a zigzag path $ L $ is a path in $ Q $ that turns alternatingly maximally right and maximally left in $ Q $. The typical shape of a zigzag path is drawn in \autoref{fig:prelim-zigzag-sketch}. If every puncture of $ Q $ has valence at least 4, then the left/right indication is a superfluous datum. In this case, the left/right indication for zigzag paths is a superfluous part of the datum of a zigzag path. For other dimers $ Q $, the left/right indication is very important. An example is the $ M $-punctured sphere $ Q_M $ which we will recall in \autoref{sec:zigzag}. Our definition deviates slightly from the definition of \cite{Bocklandt}.

Geometric consistency is a specific instance of various consistency conditions which can be imposed on dimers. To define it, we denote by $ \tilde Q $ the lift of the arc system $ Q $ to the universal cover of the closed surface $ |Q| $. We define the auxiliary notion of zigzag rays as follows, depicted in \autoref{fig:prelim-zigzag-rays}:

\begin{definition}
Let $ a ∈ \tilde Q_1 $ be an arc. Then the four zigzag rays starting at $ a $ are the sequences of arcs $ (a_i^1)_{i≥0} $, $ (a_i^2)_{i≥0} $, $ (a_i^3)_{i≥0} $ and $ (a_i^4)_{i≥0} $ in $ \tilde Q $ determined by $ a_0^1 = a_0^2 = a_0^3 = a_0^4 = a $ and the following properties:
\begin{itemize}
\item The sequences $ (a_i^1) $ and $ (a_i^2) $ satisfy $ h(a_i^{1/2}) = t(a_{i+1}^{1/2}) $.
\item The sequences $ (a_i^3) $ and $ (a_i^4) $ satisfy $ t(a_i^{3/4}) = h(a_{i+1}^{3/4}) $.
\item The path $ a_{i+1}^{1/2} a_i^{1/2} $ lies in the boundary of a counterclockwise polygon when $ i $ is odd/even, and clockwise when $ i $ is even/odd.
\item The path $ a_i^{3/4} a_{i+1}^{3/4} $ lies in the boundary of a counterclockwise polygon when $ i $ is odd/even, and clockwise when $ i $ is even/odd.
\end{itemize}
\end{definition}

\begin{figure}
\centering
\begin{subfigure}{0.35\linewidth}
\centering
\begin{tikzpicture}[scale=2]
\path[draw] (0, 0) -- ++(45:1) node[midway, left] {$ a_2^3 $} -- ++(315:0.7) node[midway, left] {$ a_1^3 $} coordinate (B);
\path[draw, -{To[width=0.2cm, length=0.3cm]}] (B) -- ++(up:1.2) node[midway, left] {$ a $} coordinate (A);
\path[draw] (A) -- ++(315:0.7) node[midway, right] {$ a_1^2 $} -- ++(45:1) node[midway, left] {$ a_2^2 $};
\path[draw] (A) -- ++(225:0.7) node[midway, left] {$ a_1^1 $} -- ++(135:1) node[midway, right] {$ a_2^1 $};
\path[draw] (A) -- ++(down:1.2) node[pos=0.6, right] {} -- ++(45:0.7) node[midway, right] {$ a_1^4 $} -- ++(315:1) node[midway, right] {$ a_2^4 $};
\end{tikzpicture}
\caption{Zigzag rays starting at $ a ∈ \tilde Q_1 $}
\label{fig:prelim-zigzag-rays}
\end{subfigure}
\begin{subfigure}{0.3\linewidth}
\centering
\begin{tikzpicture}
\begin{scope}[dashed, gray]
\path[draw, <-] (0.1, 0) -- (1.9, 0);
\path[draw, ->] (2, 0.1) -- (2, 1.9);
\path[draw, <-] (1.9, 2) -- (0.1, 2);
\path[draw, <-] (0, 0.1) -- (0, 1.9);
\path[draw, ->] (0.1, 0.1) -- (0.9, 0.9);
\path[draw, <-] (1.9, 0.1) -- (1.1, 0.9);
\path[draw, ->] (1.9, 1.9) -- (1.1, 1.1);
\path[draw, <-] (0.1, 1.9) -- (0.9, 1.1);
\end{scope}
\path[draw, thick, rounded corners] (-0.1, 0) -- (-0.1, 2.2) -- (1, 1.1) -- (2.1, 2.2) -- (2.1, 0);
\path[draw, thick, rounded corners] (-0.2, 2) -- (-0.2, -0.3) -- (1, 0.9) -- (2.2, -0.3) -- (2.2, 2);
\end{tikzpicture}
\caption{Not geometrically consistent}
\label{fig:prelim-zigag-nonconsistent}
\end{subfigure}
\caption{On consistency}
\end{figure}
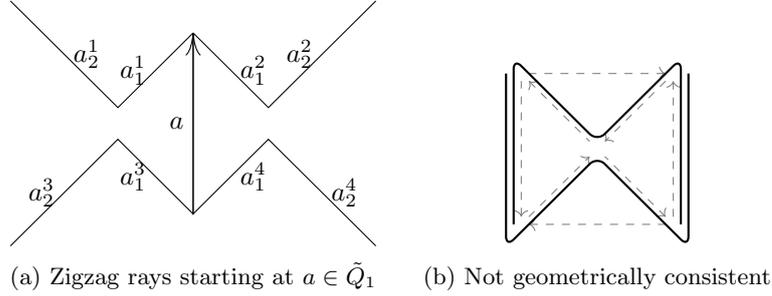

A dimer is geometrically consistent if the zigzag rays starting with an arc $ a $ in the universal cover intersect nowhere, except at $ a $ itself. The precise definition reads as follows:

\begin{definition}
Let $ Q $ be a dimer. Then $ Q $ is \emph{geometrically consistent} if for every $ a ∈ \tilde Q_1 $ the four zigzag rays $ (a_i^1) $, $ (a_i^2) $, $ (a_i^3) $ and $ (a_i^4) $ satisfy the following property: Whenever $ a_i^k = a_j^l $, then $ i = j $ and $ k = l $, or $ i = j = 0 $.
\end{definition}

\begin{remark}
Many dimers are geometrically consistent. In contrast, a dimer on a sphere is never geometrically consistent. A dimer which contains the pattern sketched in \autoref{fig:prelim-zigag-nonconsistent} is also not geometrically consistent. A geometrically consistent dimer satisfies the [NMDC] condition.
\end{remark}

%% file: prelim-MS/matrixfact.tex
\subsection{Matrix factorizations}
\label{sec:prelim-matrixfact}
In this section, we recall the notion of matrix factorizations. After recalling the definition, we focus on matrix factorization categories. Matrix factorizations serve as B-side in mirror symmetry. As such, our standard reference is \cite{Bocklandt}.

Matrix factorizations go back to 20th century work of Buchweitz and others. In a nutshell, the observation is as follows: Let $ A $ be an algebra and $ ℓ ∈ Z(A) $ a central element. Such a pair $ (A, ℓ) $ is also called a \emph{Landau-Ginzburg model}. If $ ℓ $ is prime in $ A $, there is no nontrivial factorization $ ℓ = ab $ in $ A $. However, there may be modules $ M, N ∈ \Mod A $ with maps $ f: M → N $ and $ g: N → M $ such that both $ g ∘ f: M → M $ and $ f ∘ g: N → N $ are multiplication by $ ℓ $. In other words, for typical $ ℓ $ there are more factorizations on the module level than in the algebra itself. This gives rise to the following definition:

\begin{definition}
\label{def:prelim-matrixfact-def}
Let $ A $ be an associative algebra and $ ℓ ∈ A $ a central element. A \emph{matrix factorization} of $ (A, ℓ) $ is a pair of finitely generated projective $ A $-modules $ (P, Q) $ together with $ A $-module morphisms $ f: P → Q $ and $ g: Q → P $ such that $ f ∘ g = ℓ \id_Q $ and $ g ∘ f = ℓ \id_P $.
\end{definition}

\begin{remark}
\label{th:prelim-matrixfact-alternative}
There is an alternative definition: A matrix factorization is a $ ℤ/2ℤ $-graded projective $ A $-module $ M $ together with an odd $ A $-module map $ δ: M → M $ such that $ δ^2 = ℓ \id_M $. This definition is equivalent to \autoref{def:prelim-matrixfact-def}. Given $ (M, δ) $, simply put $ P ≔ M^{\even} $ and $ Q ≔ M^{\odd} $. The odd map $ δ $ then automatically splits into two maps $ f: P → Q $ and $ g: Q → P $. Conversely given $ \matf PQfg $, put
\begin{equation*}
M = P ⊕ Q[1], \quad δ = \pmat{0 & g \\ f & 0}.
\end{equation*}
\end{remark}

The set of matrix factorizations of $ (A, ℓ) $ can be turned into a $ ℤ/2ℤ $-graded dg category $ \MF(A, ℓ) $. The intuition behind the dg structure is to interpret every matrix factorization $ (M, δ) $ as an almost chain complex, more precisely a twisted complex over the curved $ A_∞ $-category $ (A, ℓ) $. We use the notation $ \tilde{δ} $ to denote the tweaked version $ \tilde{δ}(m) = (-1)^{|m|} δ(m) $ of $ δ $.

\begin{definition}
Let $ (A, ℓ) $ be a Landau-Ginzburg model. The \emph{category of matrix factorizations} $ \MF(A, ℓ) $ is defined as follows:
\begin{itemize}
\item Objects are the matrix factorizations $ (M, δ) $ of $ (A, ℓ) $.
\item Hom spaces are given by $ \Hom((M, δ_M), (N, δ_N)) = \Hom_A (M, N) $, naturally $ ℤ/2ℤ $-graded.
\item The differential is given by $ μ^1 (f) = \tilde{δ}_N ∘ f - (-1)^{|f|} f ∘ \tilde{δ}_M $ for $ f ∈ \Hom((M, δ_M), (N, δ_N)) $.
\item The product is given by $ μ^2 (f, g) = (-1)^{‖f‖ |g|} f ∘ g $.
\end{itemize}
\end{definition}

More explicitly, regard two matrix factorizations $ \matf PQfg $ and $ \matf{P'}{Q'}{f'}{g'} $. Then their hom space is
\begin{equation*}
\underbrace{\Hom(P, P') ⊕ \Hom(Q, Q')}_{\text{even}} ⊕ \underbrace{\Hom(P, Q') ⊕ \Hom(Q, P')}_{\text{odd}}.
\end{equation*}
A morphism can be presented as a 2-by-2 matrix $ \pmat{A & B \\ C & D} $, where $ A: P → P' $ and $ B: Q → P' $ and so on. In these terms, we can write
\begin{equation*}
μ^1_{\MF(A, ℓ)} \pmat{A & B \\ C & D} = \pmat{-g'C+Bf & -g'D+Ag \\ f'A-Df & f'B-Cg}.
\end{equation*}
The product $ μ^2_{\MF(A, ℓ)} $ is simply given by signed matrix multiplication
\begin{equation*}
μ^2_{\MF(A, ℓ)} \left(\pmat{A & B \\ C & D}, \pmat{A' & B' \\ C' & D'}\right) = \pmat{AA'+BC' & -AB'+BD' \\ CA'-DC' & CB'+DD'}.
\end{equation*}

%% file: prelim-MS/mf.tex
\subsection{Jacobi algebras of dimers}
\label{sec:prelim-mf}
In this section, we recollect the Jacobi algebras of dimers, together with their associated categories of matrix factorizations. We also fix some notation in this section. For example, the Jacobi algebra of a dimer will be denoted $ \Jac Q $, its special central element will be denoted $ ℓ ∈ \Jac Q $. We would like to remind the reader that the full category of matrix factorizations is denoted $ \MF(\Jac Q, ℓ) $. Meanwhile, we recall in this section a specific small subcategory, denoted by lowercase letters $ \mf(\Jac Q, ℓ) ⊂ \MF(\Jac Q, ℓ) $.

Let us start by fixing terminology and notation for cyclicity:

\begin{definition}
Let $ Q $ be a quiver. If $ p $ is a path in $ Q $, we denote by $ p_{\cyc} ∈ ℂQ $ the sum of its cyclic permutations. We extend this assignment linearly to $ ℂQ $ and denote it by $ p ↦ p_{\cyc} $ as well. An element $ W ∈ ℂQ $ is \emph{cyclic} if it is a linear combination of cyclic paths in $ Q $ whose coefficients are invariant under cyclic permutation:
\begin{equation*}
W = \sum_{\text{cycles } a_k … a_1} λ_{a_k … a_1} a_k … a_1, \quad \text{with} \quad ∀i = 1, …, k: ~~ λ_{a_k … a_1} = λ_{a_{i-1} … a_{i+1} a_i}.
\end{equation*}
\end{definition}

We recall that superpotentials are defined as cyclic elements of length at least two:

\begin{definition}
A \emph{superpotential} is a cyclic element $ W ∈ ℂQ_{≥2} $. Its \emph{relations} are the elements
\begin{equation*}
∂_a W = \sum_{\substack{\text{paths } a_k … a_1 \\ \text{with } a_k = a}} λ_{a_k … a_1} a_{k-1} … a_1, \quad a ∈ Q_1.
\end{equation*}
Its \emph{Jacobi algebra} is given by
\begin{equation*}
\Jac(Q, W) = \frac{ℂQ}{(∂_a W)}.
\end{equation*}
Here $ (∂_a W) $ denotes the two-sided ideal generated by the partial derivatives $ ∂_a W $ for $ a ∈ Q_1 $.
\end{definition}

A dimer $ Q $ is nothing else than a specific type of quiver embedded in a surface. In particular it comes with an associated path algebra $ ℂQ $. The dimer structure of $ Q $ provides us with an additional central element $ W ∈ ℂQ $, given by the difference of the clockwise polygons of $ Q $ and the counterclockwise polygons, cyclically permuted:
\begin{equation*}
W = \sum_{\substack{a_1, …, a_k \\ \text{clockwise}}} (a_1 … a_k)_{\cyc} ~- \sum_{\substack{a_1, …, a_k \\ \text{counterclockwise}}} (a_1 … a_k)_{\cyc}.
\end{equation*}

\begin{definition}
Let $ Q $ be a dimer. Then its \emph{Jacobi algebra} is the associative algebra $ \Jac Q = ℂQ / (∂_a W) $.
\end{definition}

The relations $ ∂_a W $ equate two neighboring polygons: Flipping a path over an arc $ a $ is possible if the path follows all arcs of a neighboring polygon apart from $ a $. These flip moves are known as \emph{F-term} moves and the equivalence relation on the set of paths in $ Q $ is known as F-term equivalence. The terminology is depicted in \autoref{fig:prelim-mf-Fterm}. A good reference is \cite{Davison}.

\begin{figure}
\centering
\begin{subfigure}{0.25\linewidth}
\centering
\begin{tikzpicture}
\path[draw, ->, semithick, rounded corners] (0, 0) -- ++(left:1) coordinate (right) -- ++(135:1) -- ++(225:1) coordinate (left) -- ++(left:1);
\path[draw, semithick, rounded corners] (0, 0) -- ++(left:1) -- ++(down:0.7) -- ($ (left) + (0, -0.7) $) -- ++(up:0.7) -- ++(left:1);
\path[draw, ->, gray] (left) ++(0.1, 0) -- ($ (right) + (-0.1, 0) $);
\path[draw, ->] ($ (left)!0.5!(right) $) arc(0:20:1);
\path[draw, ->] ($ (left)!0.5!(right) $) arc(0:-20:1);
\end{tikzpicture}
\caption{An F-term flip}
\end{subfigure}
\begin{subfigure}{0.35\linewidth}
\centering
\begin{tikzpicture}
\foreach \y in {0, 1, 2}
{\foreach \x in {0, 1, 2, 3, 4}
\path[draw, ->, gray] (\x, \y+0.1) to (\x, \y+0.9);
\foreach \x in {0, 1, 2, 3}
\path[draw, ->, gray] (\x+0.9, \y+0.9) to (\x+0.1, \y+0.1);
\foreach \x in {0, 1, 2, 3}
\path[draw, ->, gray] (\x+0.1, \y) to (\x+0.9, \y);}
\foreach \x in {0, 1, 2, 3}
\path[draw, ->, gray] (\x+0.1, 3) to (\x+0.9, 3);
\path[draw, ->, semithick, rounded corners] (0, 3.3) -- (3.4, 3.3) -- (0.3, 0);
\path[draw, ->, semithick, rounded corners] (0, 3.2) -- (2.3, 3.2) -- (0.3, 1.2) -- (1.2, 1.2) -- (0.2, 0);
\path[draw, ->, semithick, rounded corners] (0, 3.1) -- (1.1, 3.1) -- (0.1, 2.1) -- (1.1, 2.1) -- (0.1, 1.1) -- (1, 1.1) -- (0.1, 0);
\end{tikzpicture}
\caption{These 3 paths are equivalent.}
\end{subfigure}
\caption{F-term equivalence}
\label{fig:prelim-mf-Fterm}
\end{figure}
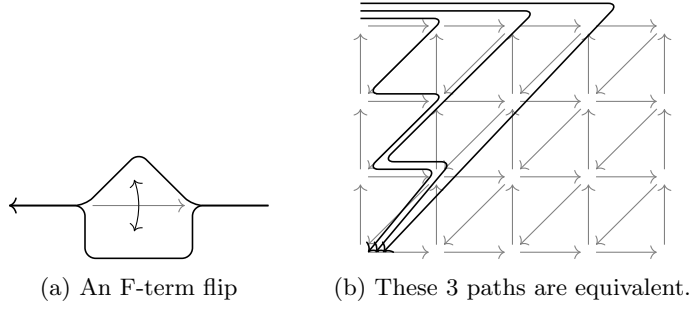

Regard the set of paths in $ Q $ modulo F-term equivalence. The set contains a special element $ ℓ_v $ for each vertex $ v ∈ Q_0 $, given by the boundary of a chosen polygon incident at $ v $. All boundaries of polygons incident at $ v $ are F-term equivalent, hence $ ℓ_v $ does not depend on the choice. In other words, it can be rotated around $ v $. We may drop the subscript from $ ℓ_v $ if it is clear from the context. The element $ ℓ $ commutes with all paths, that is, $ uℓ \sim ℓu $. Davison \cite{Davison} introduced the following consistency condition for dimers:

\begin{definition}[{\cite{Davison}}]
A dimer $ Q $ is \emph{cancellation consistent} if it has the following cancellation property:
\begin{equation*}
p ℓ \sim q ℓ ~\Longrightarrow~ p \sim q.
\end{equation*}
\end{definition}

The Jacobi algebra contains a special element which we denote by $ ℓ $ as well and which is called the \emph{potential}. It is given by the sum of the elements $ ℓ_v $ over $ v ∈ Q_0 $:
\begin{equation*}
ℓ = \sum_{v ∈ Q_0} ℓ_v ∈ \Jac Q.
\end{equation*}
The relations of $ \Jac Q $ ensure that $ ℓ $ is central and as an element of $ \Jac Q $ is independent of the choices of incident polygons.

We are ready to discuss matrix factorizations of the Landau-Ginzburg model $ (\Jac Q, ℓ) $. For every vertex $ v ∈ Q_0 $, the module $ (\Jac Q)v $ is projective. There are many further projectives, for instance given by taking direct sums of these elementary projectives. The hom space $ \Hom_{\Jac Q} ((\Jac Q)v, (\Jac Q)w) $ between two standard projectives is naturally identified with $ v (\Jac Q) w $, the subspace of paths in $ \Jac Q $ starting at $ w $ and ending at $ v $. A matrix factorization between two such projectives can be visualized as a bipartite graph consisting of vertices in $ Q $, connected by paths in $ Q $, such that all products sum up to $ ℓ $.

There is a special subcategory $ \mf(\Jac Q, ℓ) ⊂ \MF(\Jac Q, ℓ) $. The idea is that every polygon boundary can be factorized as the product of a single boundary arrow and all the other boundary arrows. More precisely, let $ a ∈ Q_1 $ be an arrow. There are precisely two polygons neighboring $ a $. The complements of their boundary are paths $ r_a^+ $ and $ r_a^- $. Within $ \Jac Q $, these two paths are identified and we simply denote them by $ \bar{a} = r_a^+ = r_a^- ∈ \Jac Q $. Since $ a \bar{a} = ℓ_{h(a)} $ and $ \bar{a} a = ℓ_{t(a)} $, we can build matrix factorizations from $ a $ and $ \bar{a} $:

\begin{definition}
Let $ Q $ be a dimer and $ a ∈ Q_1 $. Then $ \mf(\Jac Q, ℓ) ⊂ \MF(\Jac Q, ℓ) $ is the subcategory given by the matrix factorizations
\begin{equation*}
M_a = \dmatf{(\Jac Q) h(a)}{(\Jac Q) t(a)}{a}{\bar{a}}, \quad a ∈ Q_1.
\end{equation*}
\end{definition}

\begin{remark}
These matrix factorizations $ M_a $ are no factorizations of $ ℓ $ as element of $ \Jac Q $. Instead their behavior is “local”. We expect that the collection $ \{M_a\}_{a ∈ Q_1} $ already generates $ \HTw\MF(\Jac Q, ℓ) $ under shifts and cones. Such a result might be obtained by a local analysis or re-interpretation in terms of the commutative model of \cite{Ishii-Ueda} or \cite{Pascaleff-Sibilla}.
\end{remark}

%% file: prelim-MS/ms.tex
\subsection{Mirror symmetry for punctured surfaces}
\label{sec:prelim-ms}
In this section, we recall mirror symmetry for punctured surfaces from \cite{Bocklandt}. We start with an overview of the ingredients and the original proof. In particular, we exhibit A- and B-side of this mirror equivalence, explain the equivalence on object level and on hom spaces. Regarding terminology, we recall here the construction of the dual dimer $ \mirQ $ attached to $ Q $. To prepare the reader for the rest of the paper, we explain why the original setup makes it so hard to deform mirror symmetry and how the more recent Cho-Hong-Lau construction solves this issue.

The basic ingredient for mirror symmetry for punctured surfaces is a dimer $ Q $ and its dual dimer $ \mirQ $, defined as follows:

\begin{definition}
Let $ Q $ be a dimer. Then its \emph{dual dimer} $ \mirQ $ is the dimer obtained by cutting $ Q $ into its polygons, flipping over the counterclockwise polygons and inverting their arrows, and gluing everything together again along the arrows.
\end{definition}

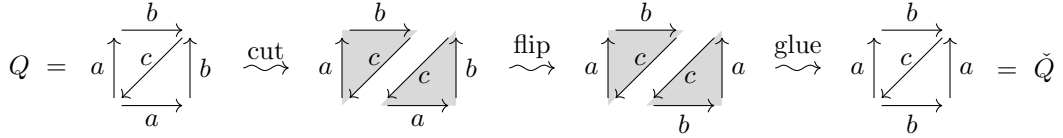
\begin{figure}
\centering
\begin{tikzpicture}
\path[draw, ->] (0, 0.1) -- (0, 0.9) node[midway, left] {$ a $};
\path[draw, ->] (0.1, 0) -- (0.9, 0) node[midway, below] {$ a $};
\path[draw, ->] (0.1, 1) -- (0.9, 1) node[midway, above] {$ b $};
\path[draw, ->] (1, 0.1) -- (1, 0.9) node[midway, right] {$ b $};
\path[draw, ->] (0.9, 0.9) -- (0.1, 0.1) node[pos=0.6, above] {$ c $};
\path[->, draw, decorate, decoration={snake, amplitude=0.1em}] (1.7, 0.5) to node[midway, above] {cut} (2.3, 0.5);
\path[->, draw, decorate, decoration={snake, amplitude=0.1em}] (5.2, 0.5) to node[midway, above] {flip} (5.8, 0.5);
\path[->, draw, decorate, decoration={snake, amplitude=0.1em}] (8.7, 0.5) to node[midway, above] {glue} (9.3, 0.5);
\path (-1, 0.5) node {$ Q ~= $} (12, 0.5) node {$ =~ \mirQ $};
\begin{scope}[shift={(3, 0)}]
\path[fill, color=gray!30] (0, 0) -- (0, 1) -- (1, 1) -- cycle;
\path[draw, ->] (0, 0.1) -- (0, 0.9) node[midway, left] {$ a $};
\path[draw, ->] (0.1, 1) -- (0.9, 1) node[midway, above] {$ b $};
\path[draw, ->] (0.9, 0.9) -- (0.1, 0.1) node[pos=0.6, above] {$ c $};
\path[fill, color=gray!30] (0.5, 0) -- (1.5, 0) -- (1.5, 1) -- cycle;
\path[draw, ->] (0.6, 0) -- (1.4, 0) node[midway, below] {$ a $};
\path[draw, ->] (1.5, 0.1) -- (1.5, 0.9) node[midway, right] {$ b $};
\path[draw, ->] (1.4, 0.9) -- (0.6, 0.1) node[pos=0.4, below] {$ c $};
\end{scope}
\begin{scope}[shift={(6.5, 0)}]
\path[fill, color=gray!30] (0, 0) -- (0, 1) -- (1, 1) -- cycle;
\path[draw, ->] (0, 0.1) -- (0, 0.9) node[midway, left] {$ a $};
\path[draw, ->] (0.1, 1) -- (0.9, 1) node[midway, above] {$ b $};
\path[draw, ->] (0.9, 0.9) -- (0.1, 0.1) node[pos=0.6, above] {$ c $};
\path[fill, color=gray!30] (0.5, 0) -- (1.5, 0) -- (1.5, 1) -- cycle;
\path[draw, ->] (0.6, 0) -- (1.4, 0) node[midway, below] {$ b $};
\path[draw, ->] (1.5, 0.1) -- (1.5, 0.9) node[midway, right] {$ a $};
\path[draw, ->] (1.4, 0.9) -- (0.6, 0.1) node[pos=0.4, below] {$ c $};
\end{scope}
\begin{scope}[shift={(10, 0)}]
\path[draw, ->] (0, 0.1) -- (0, 0.9) node[midway, left] {$ a $};
\path[draw, ->] (0.1, 0) -- (0.9, 0) node[midway, below] {$ b $};
\path[draw, ->] (0.1, 1) -- (0.9, 1) node[midway, above] {$ b $};
\path[draw, ->] (1, 0.1) -- (1, 0.9) node[midway, right] {$ a $};
\path[draw, ->] (0.9, 0.9) -- (0.1, 0.1) node[pos=0.6, above] {$ c $};
\end{scope}
\end{tikzpicture}
\caption{Three-punctured sphere and its mirror dimer}
\label{fig:prelim-ms-mirQ}
\end{figure}

\begin{example}
In \autoref{fig:prelim-ms-mirQ}, we have depicted the example of $ Q $ being the three-punctured sphere. Its mirror dimer $ \mirQ $ is a one-punctured torus. In this example, we have $ \Jac \mirQ = ℂ[a, b, c] $ and $ ℓ = abc $.
\end{example}

\begin{remark}
\label{rem:prelim-ms-Z}
A basic observation is that punctures in $ Q $ correspond to zigzag paths in $ \mirQ $. Moreover, let $ α: a → b $ be an angle in $ Q $. Denote by $ k ≥ 0 $ its “length”, or the number of indecomposable angles contained in $ α $. Then $ α $ can be reinterpreted in the mirror as a zigzag segment $ Z_α $ given by $ a = z_0, z_1, …, z_k = b $ between the arcs $ a $ and $ b $ in $ \mirQ $, with the property that $ h(z_i) = t(z_{i+1}) $.
\end{remark}

In the remainder of this section, we explain the following mirror symmetry of punctured surfaces:

\begin{theorem}[{\cite{Bocklandt}}]
Let $ Q $ be a dimer such that the dual dimer $ \mirQ $ is cancellation consistent. Then there exists an isomorphism of $ ℤ/2ℤ $-graded $ A_∞ $-categories
\begin{equation*}
F: \Gtl Q → \H\mf(\Jac \mirQ, ℓ).
\end{equation*}
\end{theorem}

Let us recollect this functor $ F $ on object level. The objects of $ \Gtl Q $ are the arcs of $ Q $. The arcs of $ Q $ are in bijection with those of $ \mirQ $. The category $ \H\mf(\Jac \mirQ, ℓ) $ is the minimal model of $ \mf(\Jac \mirQ, ℓ) $ and as such has the same objects as $ \mf(\Jac \mirQ, ℓ) $. Combining these three observations, it is easy to grasp the functor $ F $ on object level: It simply maps $ a ∈ Q_1 $ to the matrix factorization $ M_a ∈ \H\mf(\Jac \mirQ, ℓ) $.

We can also grasp $ F $ on the level of morphisms. Recall that the hom space $ \Hom_{\Gtl Q} (a, b) $ is spanned by angles starting on $ a $ and ending on $ b $, rotating around a common puncture. The corresponding basis for $ \H\Hom_{\mf(\Jac \mirQ, ℓ)} (M_a, M_b) $ was determined in \cite[Lemma 8.3]{Bocklandt}. Let $ α: a → b $ be an angle in $ Q $. By \autoref{rem:prelim-ms-Z}, the angle $ α $ comes with an associated zigzag segment $ Z_α $ in $ \mirQ $. Define $ \opp_1 $ to be the path from $ t(b) $ to $ t(a) $ and $ \opp_2 $ the path from $ h(b) $ to $ h(a) $ along $ Z_α $. This definition is depicted in \autoref{fig:prelim-ms-opposite}. Denote by $ ζ(Z_α) $ the morphism of matrix factorizations $ ζ(Z_α): M_a → M_b $ given by
\begin{equation*}
ζ(Z_α) = \pmat{\opp_2 & 0 \\ 0 & \opp_1}.
\end{equation*}
In case $ k $ is odd, let $ \opp_1 $ be the path from $ h(b) $ to $ t(a) $ and $ \opp_2 $ the path from $ t(b) $ to $ h(a) $. Denote by $ ζ(Z_α) $ the morphism of matrix factorizations $ ζ(Z): M_a → M_b $ given by
\begin{equation*}
ζ(Z_α) = \pmat{0 & \opp_1 \\ \opp_2 & 0}.
\end{equation*}
It is an easy check that for any parity of $ k $ the morphism $ ζ(Z_α) $ is a closed morphism in $ \mf(\Jac Q, ℓ) $ in the sense that $ μ^1_{\mf(\Jac \mirQ, ℓ)} (ζ(Z)) = 0 $. In terms of $ ζ(Z_α) $, the functor $ F $ on level of objects and on the level $ F^1 $ on hom spaces is given by
\begin{align*}
F(a) &= M_a, \quad ∀a ∈ Q_1, \\
F^1 (α) &= ζ(Z_α), \quad ∀α: a → b.
\end{align*}

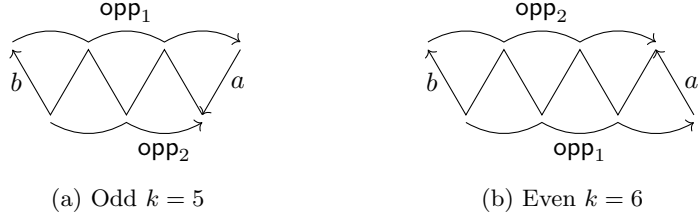
\begin{figure}
\centering
\begin{subfigure}{0.35\linewidth}
\centering
\begin{tikzpicture}
\path[draw, <-] (0, 0) -- ++(300:1) coordinate (1) -- ++(60:1) coordinate (2) -- ++(300:1) coordinate (3) -- ++(60:1) coordinate (4) -- ++(300:1) coordinate (5);
\path[draw, <-] (5) -- ++(60:1) coordinate (6);
\path[draw, ->] (0, 0.1) to[bend left] ($ (2) + (up:0.1) $) to[bend left] node[midway, above] {$ \opp_1 $} ($ (4) + (up:0.1) $) to[bend left] ($ (6) + (up:0.1) $);
\path[draw, ->] ($ (1) + (down:0.1) $) to[bend right] ($ (3) + (down:0.1) $) to[bend right] node[midway, below] {$ \opp_2 $} ($ (5) + (down:0.1) $);
\path (0, 0) -- (1) node[midway, left] {$ b $};
\path (5) -- (6) node[midway, right] {$ a $};
\end{tikzpicture}
\caption{Odd $ k = 5 $}
\end{subfigure}
\begin{subfigure}{0.35\linewidth}
\centering
\begin{tikzpicture}
\path[draw, <-] (0, 0) -- ++(300:1) coordinate (1) -- ++(60:1) coordinate (2) -- ++(300:1) coordinate (3) -- ++(60:1) coordinate (4) -- ++(300:1) coordinate (5) -- ++(60:1) coordinate (6);
\path[draw, <-] (6) -- ++(300:1) coordinate (7);
\path[draw, ->] (0, 0.1) to[bend left] ($ (2) + (up:0.1) $) to[bend left] node[midway, above] {$ \opp_2 $} ($ (4) + (up:0.1) $) to[bend left] ($ (6) + (up:0.1) $);
\path[draw, ->] ($ (1) + (down:0.1) $) to[bend right] ($ (3) + (down:0.1) $) to[bend right] node[midway, below] {$ \opp_1 $} ($ (5) + (down:0.1) $) to[bend right] ($ (7) + (down:0.1) $);
\path (0, 0) -- (1) node[midway, left] {$ b $};
\path (6) -- (7) node[midway, right] {$ a $};
\end{tikzpicture}
\caption{Even $ k = 6 $}
\end{subfigure}
\caption{Opposite paths for even and odd $ k $}
\label{fig:prelim-ms-opposite}
\end{figure}

\begin{remark}
This minimal model $ \H\mf(\Jac Q, ℓ) $ is by no means canonical. It can be calculated by the Kadeishvili construction, which however depends on the a choice of a so-called homological splitting for all hom spaces in $ \mf(\Jac \mirQ, ℓ) $. In \cite{Bocklandt}, it was observed that whatever splitting is chosen, the map $ F^1 $ alone is never an $ A_∞ $-functor. Instead, it is necessary to include higher components $ F^{≥2} $.
\end{remark}

The higher products $ F^{≥2} $ have been constructed in \cite[Appendix A]{Bocklandt}. The idea is to construct $ F^{k+1} $ inductively from $ F^1, …, F^k $. A first important ingredient is knowledge of the products
\begin{equation*}
μ^1_{\H\mf(\Jac \mirQ, ℓ)} (F^1 (α_k), …, F^1 (α_1)),
\end{equation*}
at least for sequences $ α_1, …, α_k $ which are consecutive interior angles of some polygon. The second ingredient is a temporary restriction to the case that $ Q $ is large enough to force certain unknown terms to vanish. With these two premises the component $ F^{k+1} $ can be constructed inductively.

\begin{remark}
The construction of $ F^{k+1} $ consists of solving a Hochschild cocycle equation, which of course has no unique solution. Correspondingly, the functor $ F $ cannot be computed explicitly from \cite{Bocklandt}.
\end{remark}

The lack of explicit functor $ F $ seemed to make it very difficult to deform mirror symmetry: Once we deform $ \Gtl Q $ to $ \Gtl_q Q $, there must be a deformation $ \mf_q (\Jac \mirQ, ℓ) $ of $ \mf(\Jac \mirQ, ℓ) $ such that $ \Gtl_q Q $ and $ \H\mf_q (\Jac \mirQ, ℓ) $ are still isomorphic. Actually finding this mirror deformation is a nontrivial task. The most basic approach is to write down the $ A_∞ $-functor equations and find manually a collection of deformed $ A_∞ $-products on $ \H\mf (\Jac \mirQ, ℓ) $ which still keep $ F $ a functor. This is impossible if $ F $ itself is unknown.

Fortunately, a modern explicit construction of a mirror functor $ F: \Gtl Q → \mf(\Jac \mirQ, ℓ) $ is available due to Cho, Hong and Lau \cite{CHL}. We shall explain here what makes this Cho-Hong-Lau construction so suited for deformations. The idea can be formulated in more generality: Let $ \cat C $ be an $ A_∞ $-category and $ \ZigzagCat ⊂ \cat C $ an subcategory. Then the Cho-Hong-Lau construction produces from the structure of the endomorphism algebra of $ \ZigzagCat $ a mirror Landau-Ginzburg model $ (J, ℓ) $ and a functor $ F: \cat C → \MF(J, ℓ) $:

\begin{center}
\begin{tikzpicture}
\path (0, 0) node[align=center] (A) {$ \ZigzagCat ⊂ \cat C $ \\ $ A_∞ $-category with subcategory} (7, 0) node[align=center] (B) {$ \cat C → \mf(J, ℓ) $ \\ mirror functor};
\path[draw, ->, decorate, decoration={snake, amplitude=0.2em, post length=0.5em}] ($ (A.east) + (right:1) $) to ($ (B.west) + (left:1) $);
\end{tikzpicture}
\end{center}

\begin{remark}
Let $ \cat C = \H\Tw\Gtl Q $ and $ \ZigzagCat ⊂ \H\Tw\Gtl Q $ be the category of zigzag paths in $ Q $, which we recall in \autoref{sec:zigzag-category}. Then the mirror $ (J, ℓ) $ is precisely $ (\Jac \mirQ, ℓ) $ and the construction gives an explicit functor $ F: \Gtl Q → \MF(\Jac \mirQ, ℓ) $. We recall this fact in more detail in \autoref{sec:MS-CHL}.
\end{remark}

In the context of deformations, the Cho-Hong-Lau construction is extraordinarily useful. Simply speaking, applying the construction to the deformed category $ \Gtl_q Q $ yields a deformed Landau-Ginzburg model $ (\Jac_q \mirQ, ℓ_q) $ and a deformed functor $ F_q: \Gtl_q Q → \MF(\Jac_q \mirQ, ℓ_q) $. In the rest of this paper, we make this rigorous. One bottleneck consists of proving that $ \Jac_q \mirQ $ is actually a deformation of $ \Jac \mirQ $.

%% file: prelim-zigzag/intro.tex
\section{Preliminaries on the category of zigzag paths}
\label{sec:zigzag}
In this section, we recollect the description of the deformed category of zigzag paths from \papertwoB. The aim is to translate the material in such a way that it becomes directly usable in \autoref{sec:MS}.

In \autoref{sec:zigzag-category}, we recall the three ways of thinking about zigzag paths: as zigzag paths in the dimer $ Q $, as zigzag curves in the surface $ |Q| $, and as twisted complexes lying in $ \Tw\Gtl Q $. In particular, we recall the interpretation of intersection points between two zigzag curves as basis of cohomology hom space of the corresponding twisted complexes. We recall the definition of the $ A_∞ $-category $ \ZigzagCat $ of zigzag paths. In \autoref{sec:zigzag-deformed}, we recall the deformed counterpart $ \DefZigzagCat $ of $ \ZigzagCat $. In \autoref{sec:zigzag-products}, we review the deformed $ A_∞ $-structure on $ \H\DefZigzagCat $ in terms of CR, ID, DS and DW disks. Overall, we try to introduce the reader to the translation presented in \autoref{tab:zigzag-intro-translation}. In \autoref{sec:zigzag-mirobjects}, we review another class of $ A_∞ $-products on the minimal model $ \HTw\Gtl_q Q $.

\begin{table}
\centering
\morearraystretch
\begin{tabular}{@{}ccc@{}}
\textbf{Gadget} & \textbf{Discrete} & \textbf{Smooth} \\\hline
{Input datum} & $ Q $ & $ |Q| \setminus Q_0 $ \\
(deformed) & $ Q $ & $ (|Q|, Q_0) $ \\\hline
{Wrapped category} & $ \Gtl Q $ & $ \wFuk (|Q|\setminus Q_0) $ \\
(deformed) & $ \Gtl_q Q $ & nonexistent \\\hline
{Zigzag object} & $ L $ & $ \smooth L ⊂ |Q|\setminus Q_0 $ \\
(deformed) & $ L $ & $ \smooth L ⊂ |Q| $ \\\hline
{Zigzag category} & $ \H\ZigzagCat $ & “$ \widetilde{\ZigzagCat} $” \\
(deformed) & $ \H\DefZigzagCat $ & “$ \widetilde{\DefZigzagCat} $” \\\hline
{Basis morphisms} & $ h ∈ \Hom_{\H\ZigzagCat} (L_1, L_2) $ & $ p ∈ \smooth L_1 ∩ \smooth L_2 $ \\
{$ A_∞ $-products} & CR/ID/DS/DW disks & Morse-Bott disks
\end{tabular}
\caption{Translation between discrete and smooth}
\label{tab:zigzag-intro-translation}
\end{table}

Throughout this section, $ Q $ denotes a fixed geometrically consistent dimer or a standard sphere dimer $ Q_M $ with $ M ≥ 3 $, depicted in \autoref{fig:zigzag-intro-Q5} and \ref{fig:zigzag-intro-Q6}. For \autoref{sec:zigzag-products} and \ref{sec:zigzag-mirobjects}, we assume additionally \autoref{conv:zigzag-category-convention}.

\input{prelim-zigzag/fig_Q5Q6.tex}

%% file: prelim-zigzag/fig_Q5Q6.tex
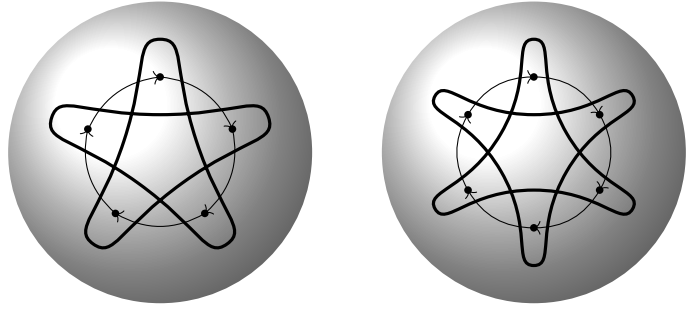
\begin{figure}
\centering
\begin{subfigure}{0.3\linewidth}
\centering
\begin{tikzpicture}
\path[shade, ball color=white] (0, 0) circle[radius=2];
\path[draw, very thick] plot[smooth cycle, tension=1.3] coordinates {(90:1.5) (18:0.5) (-54:1.5) (234:0.5) (162:1.5) (90:0.5) (18:1.5) (-54:0.5) (234:1.5) (162:0.5)};
\path[fill] (90:1) circle[radius=0.05];
\path[fill] (18:1) circle[radius=0.05];
\path[fill] (-54:1) circle[radius=0.05];
\path[fill] (234:1) circle[radius=0.05];
\path[fill] (162:1) circle[radius=0.05];
\path[draw, ->, bend left] (90:1) to (18:1);
\path[draw, ->, bend left] (18:1) to (-54:1);
\path[draw, ->, bend left] (-54:1) to (234:1);
\path[draw, ->, bend left] (234:1) to (162:1);
\path[draw, ->, bend left] (165:1) to (90:1);
\end{tikzpicture}
\caption{$ Q_5 $ with its single zigzag curve}
\label{fig:zigzag-intro-Q5}
\end{subfigure}
\begin{subfigure}{0.3\linewidth}
\centering
\begin{tikzpicture}
\path[shade, ball color=white] (0, 0) circle[radius=2];
\path[draw, very thick] plot[smooth cycle, tension=1.2] coordinates {(90:1.5) (30:0.5) (-30:1.5) (270:0.5) (210:1.5) (150:0.5)};
\path[draw, very thick] plot[smooth cycle, tension=1.2] coordinates {(30:1.5) (-30:0.5) (270:1.5) (210:0.5) (150:1.5) (90:0.5)};
\path[fill] (90:1) circle[radius=0.05];
\path[fill] (30:1) circle[radius=0.05];
\path[fill] (-30:1) circle[radius=0.05];
\path[fill] (270:1) circle[radius=0.05];
\path[fill] (210:1) circle[radius=0.05];
\path[fill] (150:1) circle[radius=0.05];
\path[draw, ->, bend left] (90:1) to (30:1);
\path[draw, ->, bend left] (30:1) to (-30:1);
\path[draw, ->, bend left] (-30:1) to (270:1);
\path[draw, ->, bend left] (270:1) to (210:1);
\path[draw, ->, bend left] (210:1) to (150:1);
\path[draw, ->, bend left] (150:1) to (90:1);
\end{tikzpicture}
\caption{$ Q_6 $ with its two zigzag curves}
\label{fig:zigzag-intro-Q6}
\end{subfigure}
\caption{The sphere dimers and their zigzag curves}
\label{fig:zigzag-intro-Q5Q6}
\end{figure}

%% file: prelim-zigzag/category.tex
\subsection{Category of zigzag paths}
\label{sec:zigzag-category}
In this section, we review the category $ \ZigzagCat $ of zigzag paths from \papertwoB. Throughout, $ Q $ is a geometrically consistent dimer or one of the standard sphere dimers $ Q_M $ with $ M ≥ 3 $.

There is a triad correspondence between zigzag paths, zigzag curves and corresponding twisted complexes. We have depicted this in \autoref{fig:zigzag-category-picture}. Let us explain their relation as follows:
\begin{itemize}
\item Zigzag paths are combinatorial gadgets in $ Q $.
\item Zigzag curves are defined as their smoothed analogs in the surface $ |Q| $.
\item A zigzag path $ L $ comes with a canonical twisted complex presentation also denoted $ L ∈ \Tw\Gtl Q $. The datum of $ L ∈ \Tw\Gtl Q $ includes a $ δ $-matrix consisting of angles between the arcs lying on the zigzag path.
\end{itemize}
In our setting, we allow additional signs in the entries of the $ δ $-matrix of the twisted complex. Once a specific choice of signs has been selected for every zigzag path, the sets of zigzag paths, zigzag curves and their twisted complexes are in one-to-one correspondence. We therefore allow ourselves to switch liberally between the three corresponding objects. For additional clarity, we may denote zigzag paths or their twisted complexes by letters $ L, L_1, … $ and associated zigzag curves by $ \smooth L, \smooth L_1, … $.

\begin{figure}
\centering
\begin{tikzpicture}
\begin{scope}[scale=0.5]
\path[draw, -{To[scale=2]}] (0, 0) -- ++(315:1.5) 
coordinate[midway] (alpha1-end) -- ++(45:1.5) 
coordinate[midway] (alpha1-start) -- ++(315:1.5) 
coordinate[midway] (alpha3-end) -- ++(45:1.5) 
coordinate[midway] (alpha3-start) -- ++(315:1.5) 
coordinate[midway] (alpha5-end) -- ++(45:1.5) 
coordinate[midway] (alpha5-start) -- ++(315:1.5) 
coordinate[midway] (additional);
\path (4, -2) node {zigzag path};
\end{scope}
\path (6, -0.5) node {\Large $ \longleftrightarrow $};
\begin{scope}[shift={(8, 0)}, scale=0.5]
\path[draw, gray!50, -{To[scale=2]}] (0, 0) -- ++(315:1.5) 
coordinate[midway] (alpha1-end) -- ++(45:1.5) 
coordinate[midway] (alpha1-start) -- ++(315:1.5) 
coordinate[midway] (alpha3-end) -- ++(45:1.5) 
coordinate[midway] (alpha3-start) -- ++(315:1.5) 
coordinate[midway] (alpha5-end) -- ++(45:1.5) 
coordinate[midway] (alpha5-start) -- ++(315:1.5) 
coordinate[midway] (additional);
\path[draw, very thick, rounded corners] ($ (alpha1-end)!-0.5!(alpha1-start) + (0, 0.2) $) -- (alpha1-end) -- ($ (alpha1-start)!0.5!(alpha1-end) + (0, -0.1) $) -- (alpha1-start) -- ($ (alpha1-start)!0.5!(alpha3-end) + (0, 0.1) $) -- (alpha3-end) -- ($ (alpha3-start)!0.5!(alpha3-end) + (0, -0.1) $) -- (alpha3-start) -- ($ (alpha3-start)!0.5!(alpha5-end) + (0, 0.1) $) -- (alpha5-end) -- ($ (alpha5-start)!0.5!(alpha5-end) + (0, -0.1) $) -- (alpha5-start) -- ($ (alpha5-start)!0.5!(additional) + (0, 0.1) $) -- (additional) -- ($ (additional)!-0.5!(alpha5-start) + (0, -0.2) $);
\path (4, -2) node {zigzag curve};
\end{scope}
\path (-1, -4) node {\Large $ \longleftrightarrow $};
\begin{scope}[shift={(0, -2.5)}]
\path (0, 0) node[anchor=west] (L) {$ L = (a_1 \oplus a_3 \oplus … \oplus a_k \oplus a_2 \oplus … \oplus a_{2k}, δ) $ with};
\path (0, -2) node[anchor=west] (delta) {$ δ = \left[
\begin{array}{c|c}
0 & \begin{array}{ccccc}
(-1)^{\#α_1} α_1 & 0 & … & 0 & (-1)^{\#α_{2k}} α_{2k} \\
(-1)^{\#α_2} α_2 & (-1)^{\#α_3} α_3 & … & 0 & 0 \\
0 & (-1)^{\#α_4} α_4 & … & 0 & 0 \\
… & … & … & … & … \\
0 & 0 & … & (-1)^{\#α_{2k-3}} α_{2k-3} & 0 \\
0 & 0 & … & (-1)^{\#α_{2k-2}} α_{2k-2} & (-1)^{\#α_{2k-1}} α_{2k-1}
\end{array} \\\hline
0 & 0
\end{array}\right] $};
\path[draw, decorate, decoration={brace, amplitude={8pt}, mirror}] ($ (L.north west) + (-0.1, 0) $) -- ($ (delta.south west) + (-0.1, 0) $);
\path (-1, -3) node[align=center] {twisted \\ complex};
\end{scope}
\end{tikzpicture}
\caption{Three descriptions of a zigzag path}
\label{fig:zigzag-category-picture}
\end{figure}
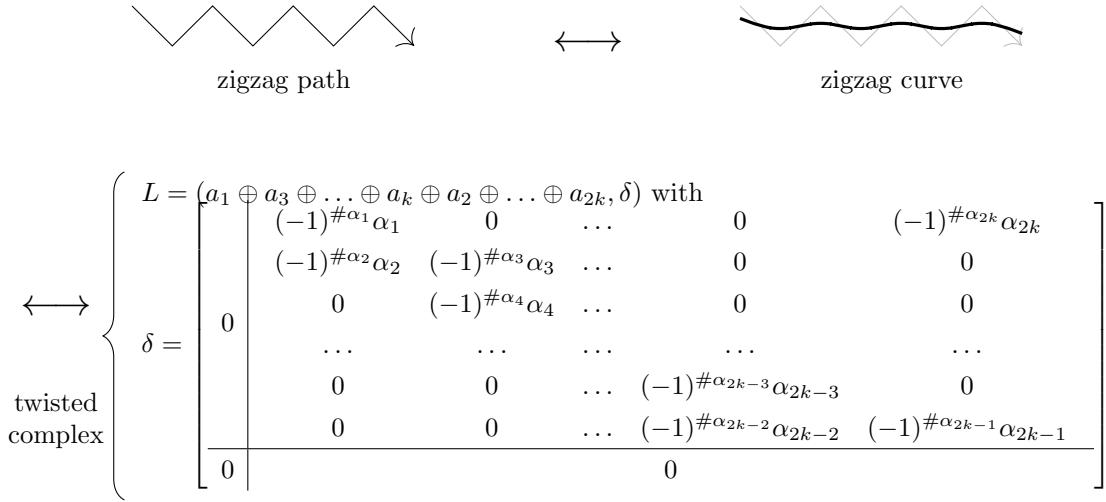


Interpreting a zigzag curve as object in the wrapped Fukaya category or as a twisted complex in $ \Tw\Gtl Q $ requires two more pieces of data. The first datum is the choice of spin structure. Choosing simultaneous spin structures for all zigzag paths in $ Q $ is equivalent to associating a sign $ (-1)^{\#α} $ to every internal angle of every polygon in $ Q $. Once the spin structure is chosen, we can form the twisted complex $ L ∈ \Tw\Gtl Q $ for every zigzag path $ L $, depicted with signs in \autoref{fig:zigzag-category-picture}. For the purpose of the present paper, we choose the $ \# $ signs in a very specific way, described in \autoref{conv:zigzag-category-convention}.

The second piece of data required is the choice of location for the identity and co-identity on every zigzag path. More precisely, this choice entails a choice of one of one indexed arc $ a_0 $ on $ L $ and the selection of one single angle $ α_0 $ out of all angles present in the $ δ $-matrix of $ L $. In the twisted complex presentation, the choice of identity and co-identity location are not visible. They are however needed as a choice to compute the products in the minimal model of $ \Tw\Gtl_q Q $, just like the products in the wrapped Fukaya category also depend on these choices.

We codify the convention on these choices as follows:

\begin{convention}
\label{conv:zigzag-category-convention}
The dimer $ Q $ is a geometrically consistent dimer or standard sphere dimer $ Q_M $ with $ M ≥ 3 $. Each zigzag path is supposed to come with a choice of an identity location $ a_0 $ and a co-identity location $ α_0 $. The co-identity $ α_0 $ shall be chosen to lie in a counterclockwise polygon. The spin structures is given by assigning to every interior angle $ α $ of a clockwise polygon the $ \# $ sign $ \#α = 0 $ and to every interior angle $ α $ of a counterclockwise polygon the $ \# $ sign $ \#α = 1 $.
\end{convention}

\begin{definition}
The \emph{category of zigzag paths} is the category $ \ZigzagCat ⊂ \Tw\Gtl Q $ given by the twisted complexes associated with all zigzag paths of $ Q $, each with its single associated choice of spin structure.
\end{definition}


In \papertwoB, we have investigated the minimal model $ \HTw\Gtl Q $. The objects of this minimal model are the same as those of $ \Tw\Gtl Q $. In particular, this category contains all twisted complexes associated with zigzag paths. However, the hom spaces are compressed in comparison to $ \Tw\Gtl Q $. We have provided explicit basis elements for these hom spaces $ \Hom_{\HTw\Gtl Q} (L_1, L_2) $ in \papertwoB. The basis elements can be identified with intersection points of $ \smooth L_1 $ and $ \smooth L_2 $:

\begin{lemma}
Let $ L_1 $ and $ L_2 $ be two zigzag paths. Then the intersection points between $ \smooth L_1 $ and $ \smooth L_2 $ naturally provide a basis for the hom space $ \Hom_{\H\ZigzagCat} (L_1, L_2) $. In case $ L_1 = L_2 $, this concerns only the transversal self-intersection points, which count double, plus the identity and co-identity points.
\end{lemma}

\begin{remark}
The description of the hom spaces by means of intersection points between $ \smooth L_1 $ and $ \smooth L_2 $ is expected from the derived equivalence of $ \Gtl Q $ and the wrapped Fukaya category \cite{Bocklandt}. Under this equivalence, the zigzag path $ L $ corresponds to the zigzag curve $ \smooth L $. As such, the hom space $ \Hom_{\wFuk (|Q|, Q_0)} (L_1, L_2) $ has basis given by intersections between $ \smooth L_1 $ and $ \smooth L_2 $. Since zigzag curves bound no digons in the punctured surface $ |Q| \setminus Q_0 $, the differential on $ \Hom_{\wFuk (|Q|, Q_0)} (\smooth L_1, \smooth L_2) $ vanishes and consequently intersection points also provide a basis for the cohomology $ \H\Hom_{\wFuk (|Q|, Q_0)} (\smooth L_1, \smooth L_2) $. In summary, this identifies intersections between $ \smooth L_1 $ and $ \smooth L_2 $ as hom space $ \Hom_{\HTw\Gtl Q} (L_1, L_2) $.
\end{remark}

Every basis morphism $ p: L_1 → L_2 $ comes with a degree $ |p| ∈ ℤ/2ℤ $ assigned. The degree depends on the orientation of the intersection between $ \smooth L_1 $ and $ \smooth L_2 $. The precise convention is depicted in \autoref{fig:zigzag-category-intersectiondeg}.

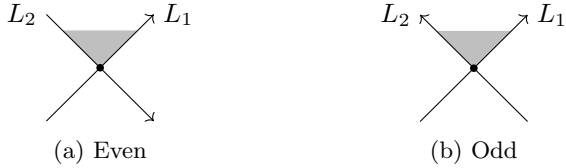
\begin{figure}
\centering
\begin{subfigure}{0.3\linewidth}
\centering
\begin{tikzpicture}
\path[fill, lightgray] (0, 0) -- (45:0.7) -- (135:0.7) -- cycle;
\path[draw] (0, 0) -- ++(135:1) node[left] {$ L_2 $} (0, 0) -- ++(225:1);
\path[draw, ->] (0, 0) -- ++(315:1);
\path[draw, ->] (0, 0) -- ++(45:1) node[right] {$ L_1 $};
\path[fill] (0, 0) circle[radius=0.05];
\end{tikzpicture}
\caption{Even}
\end{subfigure}
\begin{subfigure}{0.3\linewidth}
\centering
\begin{tikzpicture}
\path[fill, lightgray] (0, 0) -- (45:0.7) -- (135:0.7) -- cycle;
\path[draw, ->] (0, 0) -- ++(135:1) node[left] {$ L_2 $};
\path[draw] (0, 0) -- ++(225:1) (0, 0) -- ++(315:1);
\path[draw, ->] (0, 0) -- ++(45:1) node[right] {$ L_1 $};
\path[fill] (0, 0) circle[radius=0.05];
\end{tikzpicture}
\caption{Odd}
\end{subfigure}
\caption{Intersection degree}
\label{fig:zigzag-category-intersectiondeg}
\end{figure}

\begin{remark}
The basis of the hom space $ \Hom_{\HTw\Gtl Q} (L_1, L_2) $ is special in case $ L_1 = L_2 $. Its basis is then given by transversal intersections plus two special morphisms, namely the identity and co-identity. The transversal intersections are self-intersections, and they give in fact two morphisms in $ \HTw\Gtl Q $, of which one is odd and the other is even. Whenever we refer to self-intersections of a zigzag curve, it is understood that the datum of a self-intersection shall include the choice of whether we mean the odd or the even morphism. For more details we refer to \papertwoB\ or \cite[Chapter 9]{Bocklandt-book}.
\end{remark}

%% file: prelim-zigzag/deformed.tex
\subsection{Deformed category of zigzag paths}
\label{sec:zigzag-deformed}
In this section, we recall the deformation $ \DefZigzagCat $ of $ \ZigzagCat $ from \papertwoB. First, we recall the use of the deformed gentle algebra $ \Gtl_q Q $. We then explain that the zigzag curves survive upon deforming $ \Gtl Q $ to $ \Gtl_q Q $. Finally, we recall how the process of “uncurving” produces the deformed, yet curvature-free category $ \DefZigzagCat $.

We are interested in the minimal model of $ \Tw\Gtl_q Q $. For our purposes, it suffices in fact to look at the subcategory of $ \Tw\Gtl_q Q $ given by zigzag paths:

\begin{definition}
The category $ \DefZigzagCat^{\pre} ⊂ \Tw\Gtl_q Q $ is the subcategory consisting of the twisted complexes of all zigzag paths of $ Q $, each with their chosen spin structure.
\end{definition}

In \papertwoA, we showed how to compute a minimal model of a deformed $ A_∞ $-category $ \cat C_q $. The first step in the procedure consists of optimizing curvature according to a well-defined prescription. The result of the curvature optimization procedure is an $ A_∞ $-deformation that is gauge equivalent to $ \cat C_q $.

In the specific case of $ \DefZigzagCat^{\pre} $, we can explicitly describe the result $ \DefZigzagCat $ of the curvature optimization procedure:

\begin{definition}
Let $ L $ be a zigzag path of $ Q $, with associated twisted complex
\begin{equation*}
L = \left(a_1 ⊕ a_3 ⊕ … ⊕ a_k ⊕ a_2 ⊕ … ⊕ a_{2k}, δ\right)
\end{equation*}
as in \autoref{fig:zigzag-category-picture}. Then the corresponding \emph{deformed zigzag path} is the following object of $ \Tw'\Gtl_q $, still denoted $ L $:
\begin{align*}
L &= \left(a_1 ⊕ a_3 ⊕ … ⊕ a_k ⊕ a_2 ⊕ … ⊕ a_{2k}, δ\right), \\
δ &= \left[
\begin{array}{c|c}
0 & \text{ditto} \\\hline
\begin{array}{ccccc}
(-1)^{\# α_1} q_1 α_1' & (-1)^{\#α_2} q_2 α_2' & 0 & … & 0 \\
0 & (-1)^{\#α_3} q_3 α_3' & (-1)^{\#α_4} q_4 α_4' & … & 0 \\
… & … & … & … & … \\
0 & 0 & 0 & … & (-1)^{\#α_{2k-2}} q_{2k-2} α_{2k-2}' \\
(-1)^{\#α_{2k}} q_{2k} α_{2k}' & 0 & 0 & … & (-1)^{\#α_{2k-1}} q_{2k-1} α_{2k-1}'
\end{array} & 0
\end{array}\right].
\end{align*}
Here “ditto” denotes the same matrix entries as in \autoref{fig:zigzag-category-picture}. The letter $ q_i $ denotes the puncture around which $ α_i $ winds. The angle $ α_i' $ is defined as the complementary angle to $ α_i $. In other words, the angle $ α_i' $ is such that $ α_i' α_i $ comprises a single full turn around a puncture.

The \emph{category of deformed zigzag paths} is the subcategory $ \DefZigzagCat ⊂ \Tw'\Gtl_q Q $ consisting of the deformed zigzag paths.
\end{definition}

\begin{remark}
The definition of the deformed zigzag path is a specific application of the “complementary angle trick” laid out in \papertwoB. On the level of twisted complexes, the complementary angle trick infinitesimally changes the $ δ $-matrices of the zigzag paths. The resulting $ δ $-matrices are not upper triangular anymore, whence the notation $ \Tw'\Gtl_q Q $. On the level of the category $ \DefZigzagCat^{\pre} $ itself, the complementary angle trick consists of a specific gauge transformation.
\end{remark}

\begin{remark}
The category $ \DefZigzagCat $ of deformed zigzag paths is a deformation of the category $ \ZigzagCat $ of zigzag paths. In \papertwoB, we proved that $ \DefZigzagCat $ has “optimal curvature” in the sense of the deformed Kadeishvili theorem. In fact, if $ Q $ is geometrically consistent, then $ \DefZigzagCat $ is curvature-free. If $ Q = Q_M $ for $ M $ odd, then $ \DefZigzagCat $ is curvature-free as well (due to choice of spin structure). If $ Q = Q_M $ for $ M $ even, then the curvature of the two zigzag paths in $ \DefZigzagCat $ is an infinitesimal multiple of their identity morphisms.
\end{remark}

%% file: prelim-zigzag/products.tex
\subsection{Minimal model structure}
\label{sec:zigzag-products}
In this section, we recall the deformed $ A_∞ $-products on $ \H\DefZigzagCat $ from \papertwoB. As it turns out, these products have striking similarity with the products of the relative Fukaya category. Although we do not need Fukaya categories here, it is helpful to recall that their products are enumerated in terms of what we may call smooth immersed disks. In the present section, we recall that also the products of $ \H\DefZigzagCat $ can be enumerated in terms of certain types of smooth immersed disks. We recollect their precise rules.

The category $ \DefZigzagCat $ is a deformed $ A_∞ $-category and as such has no classical minimal model. In \papertwoA, we define a notion of minimal models for deformed $ A_∞ $-categories and show that every deformed $ A_∞ $ has a minimal model. Moreover, we show how to compute minimal models by means of a deformed Kadeishvili construction. In \papertwoB, we compute the minimal model $ \H\DefZigzagCat $ of $ \DefZigzagCat $ in its entirety.

\begin{remark}
The aim of our minimal model computation in \papertwoB\ was to identify the higher products of $ \H\DefZigzagCat $ as higher products of the relative Fukaya category. Here is a sketch of this computation: We provided an explicit homological splitting as well as corresponding deformed codifferential $ h_q $ and deformed projection $ π_q $ in terms of “tails of morphisms”. It remained to evaluate all Kadeishvili trees. The essential step was to analyze Kadeishvili trees by what we called “result components”. It turned out that every result component can be matched with a disk between the zigzag curves. Finally, we classified the disks that appeared this way into the types CR, ID, DS and DW. In other words, the minimal model $ \H\DefZigzagCat $ is given by intersection points of the zigzag curves, together with a deformed $ A_∞ $-structure which counts disks.
\end{remark}

The minimal model $ \H\DefZigzagCat $ can be expressed by means of counting CR, ID, DS and DW disks. The first step in this section is to recall these four types of disks. We give a verbal and visual characterization of these disks. As auxiliary disk type we recall the SL disks (shapeless disks), with a slightly abridged definition.

For the definition of these disk types, let $ L_1, …, L_{N+1} $ be a fixed sequence of $ N+1 ≥ 1 $ zigzag paths and $ h_i: L_i → L_{i+1} $ for $ i = 1, …, N $ some basis morphisms. Since identities among the inputs $ h_i $ yield well-known products $ μ_{\H\DefZigzagCat} (h_N, …, h_1) $ by the unitality property of $ \H\DefZigzagCat $, we assume that none of the inputs $ h_i $ is an identity. Note that co-identities are however allowed. Let $ t: L_1 → L_{N+1} $ be another intersection point, the letter standing for “target candidate”.

\begin{definition}
An \emph{SL disk} (shapeless disk) with \emph{inputs} $ h_1, …, h_N $ and \emph{output} $ t $ consists of an oriented immersion $ D: P_{N+1} → |Q| $ of the standard $ (N+1) $-gon $ P_{N+1} $ such that
\begin{itemize}
\item the $ i $-th edge of $ D $ is mapped to a segment of $ \smooth L_i $, for $ i = 1, …, N+1 $,
\item the $ i $-th corner of $ D $ is mapped to the intersection point corresponding with $ h_i $, for $ i = 1, …, N $,
\item the $ N+1 $-th corner of $ D $ is mapped to the intersection point corresponding with $ t $,
\item all corners of $ D $ are convex.
\end{itemize}
The immersion $ D $ itself is taken up to reparametrization.
\end{definition}

\begin{remark}
SL disks are allowed to be monogons ($ N = 0 $) or digons ($ N = 1 $). All the $ \smooth L_i $-segments are allowed to be empty. We can also imagine these empty segments as being infinitesimally short.
\end{remark}

The standard polygon $ P_{N+1} $ together with its numbering of corners and edges is depicted in \autoref{fig:prelim-gtl-P5}. Note that disk inputs are numbered in opposite direction as they would be in the standard definition \cite{Abouzaid}. The difference is necessary in order to match with the convention for gentle algebras \cite{Bocklandt}. With the definition of SL disks in mind, we are ready to recall the CR, ID, DS and DW disk types. Among these four disk types, only CR and ID disks are relevant for this paper and we have depicted their schematic in \autoref{fig:zigzag-products-CRdisk} and \ref{fig:zigzag-products-IDdisk}. For DS and DW disks we given an abridged definition and explain what makes them irrelevant.

\begin{definition}
A \emph{CR disk} (co-identity rule disk) is an SL disk all of whose segments are of non-empty, with the exception that multiple stacked co-identity inputs with empty segments in between are allowed, as long as their zigzag curve is oriented clockwise with the disk. We denote by $ \CRd $ the set of all CR disks, taking the union over arbitrary input sequence $ h_1, …, h_N $ and output $ t $. 
\end{definition}

\input{prelim-zigzag/fig_CRdisk.tex}

\begin{definition}
A \emph{ID disk} (identity degenerate disk) is an SL disk satisfying the following conditions:
\begin{itemize}
\item The output is an identity,
\item Precisely one input, the \emph{degenerate input}, is infinitesimally close to the output,
\item The degenerate input is an odd or even transversal intersection,
\item The disk becomes CR upon excision of the output and substitution of the output mark by the degenerate input,
\item In case the degenerate input is odd, it precedes respectively succeeds the output mark if $ \smooth L_1 $ is oriented clockwise respectively counterclockwise with the disk,
\item In case the degenerate input is even, then the source zigzag curve of the degenerate input is counterclockwise and the target zigzag curve of the degenerate input is clockwise.
\end{itemize}
We denote by $ \IDd $ the set of ID disks with arbitrary inputs and output.
\end{definition}

\input{prelim-zigzag/fig_IDdisk.tex}

\emph{DS and DW disks} (degenerate strip disks, degenerate wedge disks) are immersed strips fitting into one of the two digons bounded by a zigzag curve $ \smooth L $ and its Hamiltonian deformation. It is possible to make this more precise. However, every DS and DW disk necessarily includes at least one even input. This already renders DS and DW disks irrelevant for the present paper. Even without recalling the precise definition, we denote by $ \DSd $ and $ \DWd $ the set of DS and DW disks, respectively.

In order to give the description of the products $ μ_{\H\DefZigzagCat} $ in terms of disks, we have to introduce two pieces of notation here: the Abouzaid sign and the deformation parameter attached to a disk. We have chosen to name this sign rule after Abouzaid for the reason that it is the same as in \cite{Abouzaid}.

\begin{definition}
Let $ D $ be an SL disk. Then its \emph{Abouzaid sign} $ \Abouzaid(D) ∈ ℤ/2ℤ $ is the sum of all $ \# $ signs on the boundary of $ D $, plus the number of odd inputs $ h_i: L_i → L_{i+1} $ where $ \smooth L_{i+1} $ is oriented counterclockwise relative to $ D $, plus one if the output $ t: L_1 → L_{N+1} $ is odd and $ \smooth L_{N+1} $ is oriented counterclockwise. The deformation parameter $ \punctures(D) ∈ ℂ⟦Q_0⟧ $ is the total product of all punctures covered by $ D $, counted with multiplicity.
\end{definition}

We will now make the description of the product structure on $ \H\DefZigzagCat $ precise. The procedure is familiar: Let $ h_1, …, h_N $ be a sequence of inputs. The product $ μ_{\H\DefZigzagCat} (h_N, …, h_1) $ is given by enumerating all disks with inputs $ h_1, …, h_N $ and arbitrary output $ t: L_1 → L_{N+1} $. In our specific case, the types of disks involved are the CR, ID, DS and DW disks. The contribution from a disk $ D $ carries the Abouzaid sign $ (-1)^{\Abouzaid(D)} $ and is weighted by the deformation parameter $ \punctures(D) ∈ ℂ⟦Q_0⟧ $. For a disk $ D $, denote its output $ t: L_1 → L_{N+1} $ by $ \disktarget(D) $.

\begin{theorem}[{\papertwoB}]
\label{th:zigzag-products-th}
Let $ Q $ be a geometrically consistent dimer or standard sphere dimer $ Q_M $ with $ M ≥ 3 $. The $ A_∞ $-product $ μ_{\H\DefZigzagCat} $ is strictly unital. Let $ h_1, …, h_N $ be a sequence of $ N ≥ 0 $ non-identity basis morphisms with $ h_i: L_i → L_{i+1} $. Then their product is given by
\begin{equation*}
μ_{\H\DefZigzagCat}^N (h_N, …, h_1) = \sum_{\substack{D ∈ \CRd \disjoint \IDd \disjoint \DSd \disjoint \DWd \\ D \text{ has inputs } h_1, …, h_N}} (-1)^{\Abouzaid(D)} \punctures(D) \disktarget(D).
\end{equation*}
\end{theorem}

\begin{remark}
The description of \autoref{th:zigzag-products-th} also describes curvature $ μ^0_{\H\DefZigzagCat} $ and differential $ μ^1_{\H\DefZigzagCat} $ of the minimal model accurately. Let us explain this as follows:

In case $ Q $ is geometrically consistent, the curvature $ μ^0_{\H\DefZigzagCat} $ and the differential $ μ^1_{\H\DefZigzagCat} $ vanish. This is witnessed by the fact that there are no monogons or digons in $ Q $ bounded by zigzag curves.

In case $ Q = Q_M $ for odd $ M $, the curvature $ μ^0_{\H\DefZigzagCat} $ vanishes, but the differential $ μ^1_{\H\DefZigzagCat} $ is nonzero. This vanishing of curvature is witnessed by the fact that there are monogons in $ Q_M $ bounded by zigzag curves, but they cancel each other due to the spin structure. The differential is witnessed by the fact that there are digons in $ Q_M $ bounded by zigzag curves.

In case $ Q = Q_M $ for even $ M $, both curvatue $ μ^0_{\H\DefZigzagCat} $ and $ μ^1_{\H\DefZigzagCat} $ are nonzero. This is witnessed by the fact that there are both monogons and digons in $ Q_M $ bounded by zigzag paths.
\end{remark}

%% file: prelim-zigzag/fig_CRdisk.tex
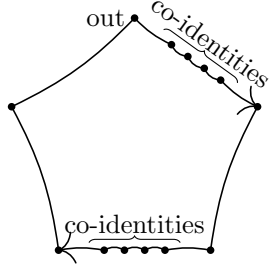
\begin{figure}
\centering
\begin{tikzpicture}
\path[draw, semithick, {To[scale=2]}-] (0, 0) coordinate(h3) to[bend left=10] ++(0.6, 0) coordinate(m1) to[bend left=30] ++(0.26666, 0) coordinate(m2) to[bend left=30] ++(0.26666, 0) coordinate (m3) to[bend left=30] ++(0.266666, 0) coordinate (m4) to[bend left=10] ++(0.6, 0) coordinate(h2) to[bend left=10] ++(72:2) coordinate(h1);
\path[draw, semithick, {To[scale=2]}-] (h1) to[bend left=10] ++(144:0.6) coordinate(m5) to[bend left=30] ++(144:0.26666) coordinate (m6) to[bend left=30] ++(144:0.26666) coordinate(m7) to[bend left=30] ++(144:0.266666) coordinate (m8) to[bend left=10] ++(144:0.6) coordinate(t) to[bend left=10] ++(216:2) coordinate(h4) to[bend left=10] (0, 0);
\path[fill] \foreach \i in {h1, h2, h3, h4, m1, m2, m3, m4, m5, m6, m7, m8, t} {(\i) circle[radius=0.05]};
\path (t) node[left] {out};
\path[draw, decorate, decoration={brace}] ($ (m1) + (-0.2, 0.1) $) to node[midway, above] {co-identities} ($ (m4) + (0.2, 0.1) $);
\path[draw, decorate, decoration={brace}] ($ (m8) + (144:0.2) + (54:0.1) $) to node[midway, sloped, above]  {co-identities} ($ (m5) + (144:-0.2) + (54:0.1) $);
\end{tikzpicture}
\caption{This pictures depicts a typical CR disk. The disk has twelve inputs of which eight are co-identities. There are two groups of co-identities, each consisting of four co-identities stacked together. By definition, the zigzag curves on which the two stacks of co-identities lie are required to run clockwise, as indicated by the arrows.}
\label{fig:zigzag-products-CRdisk}
\end{figure}

%% file: prelim-zigzag/fig_IDdisk.tex
\begin{figure}
\centering
\begin{subfigure}{0.3\linewidth}
\centering
\begin{tikzpicture}
\path (2, 0) -- ++(72:2) coordinate (A) -- ++(144:2) coordinate (B);
\path[draw, semithick, {To[scale=2]}-] (0, 0) coordinate (h3) to[bend left=5] ++(0.6, 0) coordinate (m1) to[bend left=30] ++(0.26666, 0) coordinate (m2) to[bend left=30] ++(0.2666666, 0) coordinate (m3) to[bend left=30] ++(0.26666, 0) coordinate (m4) to[bend left=5] ++(0.6, 0) coordinate (h2) to[bend left=5] ++(72:1.8) coordinate (h1);
\path[draw, semithick, {To[scale=2]}-] (h1) to[bend left=10] ($ (A)!0.9!(B) $) coordinate (t) to (B) coordinate (h5);
\path[draw, semithick, {To[scale=2]}-] (h5) to ++(216:2) coordinate (h4) to (0, 0);
\path[fill] \foreach \i in {h1, h2, h3, h4, h5, t, m1, m2, m3, m4} {(\i) circle[radius=0.05]};
\path (t) node[right] {out};
\path[draw, decorate, decoration={brace}] ($ (m1) + (-0.2, 0.1) $) to node[midway, above] {\small co-identities} ($ (m4) + (0.2, 0.1) $);
\end{tikzpicture}
\caption{Clockwise odd input}
\label{fig:zigzag-products-IDdiskCB}
\end{subfigure}
\begin{subfigure}{0.3\linewidth}
\centering
\begin{tikzpicture}
\path (2, 0) -- ++(72:2) -- ++(144:2) coordinate (A) -- ++(216:2) coordinate (B);
\path[draw, semithick, {To[scale=2]}-] (0, 0) coordinate (h3) to[bend left=5] ++(0.6, 0) coordinate (m1) to[bend left=30] ++(0.26666, 0) coordinate (m2) to[bend left=30] ++(0.2666666, 0) coordinate (m3) to[bend left=30] ++(0.26666, 0) coordinate (m4) to[bend left=5] ++(0.6, 0) coordinate (h2) to[bend left=5] ++(72:1.8) coordinate (h1);
\path[draw, semithick, -{To[scale=2]}] (h1) to[bend left=10] (A) coordinate (h0);
\path[draw, semithick, -{To[scale=2]}] (h0) to ($ (A)!0.1!(B) $) coordinate (t) to[bend left=5] (B) coordinate (h5);
\path[draw, semithick] (h5) to (0, 0);
\path[draw, semithick] (h4) to (0, 0);
\path[fill] \foreach \i in {h0, h1, h2, h3, h5, t, m1, m2, m3, m4} {(\i) circle[radius=0.05]};
\path (t) node[left] {out};
\path[draw, decorate, decoration={brace}] ($ (m1) + (-0.2, 0.1) $) to node[midway, above] {\small co-identities} ($ (m4) + (0.2, 0.1) $);
\end{tikzpicture}
\caption{Counterclockwise odd input}
\label{fig:zigzag-products-IDdiskCCB}
\end{subfigure}
\begin{subfigure}{0.3\linewidth}
\centering
\begin{tikzpicture}
\path (2, 0) -- ++(72:2) coordinate (A) -- ++(144:2) coordinate (B);
\path[draw, semithick, {To[scale=2]}-] (0, 0) coordinate (h3) to[bend left=5] ++(0.6, 0) coordinate (m1) to[bend left=30] ++(0.26666, 0) coordinate (m2) to[bend left=30] ++(0.2666666, 0) coordinate (m3) to[bend left=30] ++(0.26666, 0) coordinate (m4) to[bend left=5] ++(0.6, 0) coordinate (h2) to[bend left=5] ++(72:1.8) coordinate (h1);
\path[draw, semithick, {To[scale=2]}-] (h1) to[bend left=10] ($ (A)!0.9!(B) $) coordinate (t) to (B) coordinate (h5);
\path[draw, semithick, -{To[scale=2]}] (h5) to ++(216:2) coordinate (h4);
\path[draw, semithick] (h4) to (0, 0);
\path[fill] \foreach \i in {h1, h2, h3, h4, h5, t, m1, m2, m3, m4} {(\i) circle[radius=0.05]};
\path (t) node[right] {out};
\path[draw, decorate, decoration={brace}] ($ (m1) + (-0.2, 0.1) $) to node[midway, above] {\small co-identities} ($ (m4) + (0.2, 0.1) $);
\end{tikzpicture}
\caption{Even input}
\label{fig:zigzag-products-IDdiskC}
\end{subfigure}
\caption{These pictures depict typical ID disks, categorized according to the type of their degenerate input. All depicted ID disks have nine inputs, of which four consist of a stack of co-identities and one is the degenerate input. The degenerate input is the input located directly next to the output mark. The orientations of the zigzag curves near the output mark are enforced by the specific rules of ID disks. The orientation of the zigzag curve carrying the co-identities is enforced by the requirement that the disk becomes CR upon excision of the output mark.}
\label{fig:zigzag-products-IDdisk}
\end{figure}
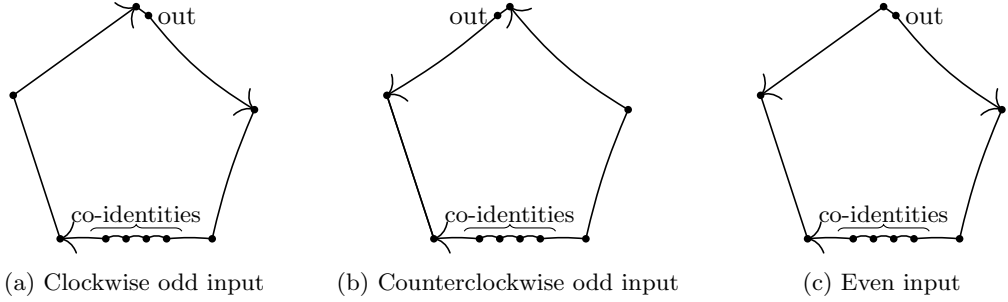

%% file: prelim-zigzag/mirobjects.tex
\subsection{Preparation for mirror objects}
\label{sec:zigzag-mirobjects}
In this section we review further selected products in the category $ \HTw\Gtl_q Q $ from \papertwoB. Namely, we regard products of the form $ μ(m, h_N, …, h_1) $, where $ h_1, …, h_N $ is a sequence of odd morphisms $ h_i: L_i → L_{i+1} $ not containing co-identities and $ m ∈ \Hom_{\HTw\Gtl_q Q} (L_{N+1}, a) $ is a morphism to an arc $ a ∈ \Gtl_q Q $. Ultimately, knowledge of these products serves the calculation of mirror objects in \autoref{sec:MS-mirobjects}.

We start with reviewing the hom space $ \Hom_{\HTw\Gtl Q} (L, a) $. Here $ L ∈ \HTw\Gtl Q $ is a zigzag path in $ Q $ and $ a ∈ \Gtl Q $ is an arc. The zigzag curve $ \smooth L $ and the arc $ a $ are both curves in the surface $ |Q| $. While $ \smooth L $ is a closed curve, the arc $ a $ is an interval. The curves $ \smooth L $ and $ a $ may be disjoint or intersect. If they intersect, they intersect in a single point, namely the midpoint of $ a $:

\begin{center}
\begin{tikzpicture}[yscale=0.7]
\path[use as bounding box] (0, 0) -- (3, 2);
\path[draw, gray, ->] (0, 0) -- ++(60:1) coordinate[midway] (1) -- ++(right:1) coordinate[midway] (2) coordinate (stop1);
\path[draw, gray] (stop1) -- ++(60:0.85) coordinate[midway] (3);
\path[draw, semithick, ->] (60:1)++(up:0.15) -- ++(right:1) node[midway, above] {$ a $} node[midway, below, shift={(down:0.3)}] {odd} coordinate[midway] (5);
\path[draw, semithick, rounded corners] ($ (1) + (-0.5, -0.1) $) to (1) to coordinate[pos=0.4] (7) (5) to coordinate[pos=0.6] (8) (3) to ($ (3) + (0.5, 0.1) $) node[below] {$ \smooth L $};
\path[fill] (5) circle[radius=0.05];
\end{tikzpicture}
\begin{tikzpicture}[yscale=0.7]
\path[use as bounding box] (0, 0) -- (3, 2);
\path[draw, gray, ->] (0, 1.6) -- ++(300:1) coordinate[midway] (4) -- ++(right:1) coordinate[midway] (5) coordinate (stop2);
\path[draw, gray] (stop2) -- ++(300:0.85) coordinate[midway] (6);
\path[draw, semithick, ->] (60:1) -- ++(right:1) node[midway, above] {$ a $} node[midway, below, shift={(down:0.3)}] {even} coordinate[midway] (2);
\path[draw, semithick, rounded corners] ($ (4) + (-0.5, 0.1) $) to (4) to coordinate[pos=0.4] (9) (2) to coordinate[pos=0.6] (10) (6) to ($ (6) + (0.5, -0.1) $) node[above] {$ \smooth L $};
\path[fill] (2) circle[radius=0.05];
\end{tikzpicture}
\begin{tikzpicture}[yscale=0.7]
\path[use as bounding box] (0, 0.1) -- (3, 2);
\path[draw, gray, ->] (0, 0) -- ++(60:1) coordinate[midway] (1) -- ++(right:1) coordinate[midway] (2) coordinate (stop1);
\path[draw, gray] (stop1) -- ++(60:0.85) coordinate[midway] (3);
\path[draw, gray, ->] (0, 2) -- ++(300:1) coordinate[midway] (4) -- ++(right:1) coordinate[midway] (5) coordinate (stop2);
\path[draw, gray] (stop2) -- ++(300:0.85) coordinate[midway] (6);
\path ($ (2)!0.5!(5) $) coordinate (m);
\path[draw, semithick, ->] (m) node[above, shift={(up:0.1)}] {$ a $} -- ++(left:0.5) -- ++(right:1);
\path[draw, semithick, rounded corners] ($ (1) + (-0.5, -0.1) $) to (1) to coordinate[pos=0.4] (7) (m) to coordinate[pos=0.6] (8) (3) to ($ (3) + (0.5, 0.1) $) node[above] {$ \smooth L $};
\path[draw, semithick, rounded corners] ($ (4) + (-0.5, 0.1) $) to (4) to coordinate[pos=0.4] (9) (m) to coordinate[pos=0.6] (10) (6) to ($ (6) + (0.5, -0.1) $) node[below] {$ \smooth L $};
\path[fill] (m) circle[radius=0.05];
\end{tikzpicture}
\end{center}

The picture on the left depicts a single odd intersection $ L → a $. The picture in the middle depicts an even intersection $ L → a $. It is also possible that $ L $ intersects the arc $ a $ twice, depicted on the right. Given that $ \HTw\Gtl Q $ and the wrapped Fukaya category are equivalent, we expect that intersections $ L → a $ provide a natural basis for the hom space $ \Hom_{\HTw\Gtl Q} (L, a) $. We confirmed this in \papertwoB:

\begin{lemma}
Let $ L $ be a zigzag path and $ a $ an arc in $ Q $. Then a natural basis for the hom space $ \Hom_{\HTw\Gtl Q} (L, a) $ is given by the intersections between $ a $ and $ \smooth L $.
\end{lemma}

Every intersection point $ m: L_1 → a $ comes with a partner $ m^*: L_2 → a $. To see this, let $ L_2 $ be the zigzag path departing from $ a $ on the opposite side of $ L_1 $. Then also $ \smooth L_2 $ intersects $ a $ at its midpoint. When $ m $ is even, its partner $ m^* $ is odd, and vice versa. It is possible that $ L_1 = L_2 $, namely in case the two zigzag paths departing from $ a $ are equal. Apart from the zigzag paths $ L_1 $ and $ L_2 $, there is not a single other zigzag path in $ Q $ that intersects $ a $. Let us review an additional type of disk, depicted in \autoref{fig:zigzag-mirobjects-MDMT}:

\begin{definition}
An \emph{MD disk} (mirror disk) is a CR disk whose
\begin{itemize}
\item inputs $ h_1, …, h_N $ are all odd and do not contain co-identities,
\item output is even and not an identity,
\item zigzag segments all run clockwise,
\end{itemize}
which has undergone the following surgery: The output mark, located at a certain arc $ a $, has been cut off. The odd morphism at $ a $ is added as final input, and the even morphism at $ a $ is indicated as new output.
\end{definition}

The Abouzaid sign $ \Abouzaid(D) $ and the deformation parameter $ \punctures(D) $ of an MD disk are defined in analogy to Abouzaid signs and deformation parameters for SL disks. Given an arc $ a ∈ Q_1 $, there are two zigzag paths departing from $ a $. In particular, there is one single odd basis morphism between these two zigzag paths located at $ a $. With this in mind, we are ready to recall the description of some products of the form $ μ_{\HTw\Gtl_q Q} (m, h_N, …, h_1) $ from \papertwoB:

\begin{lemma}
\label{th:zigzag-mirobjects-products}
Let $ h_1, …, h_N $ be a sequence of $ N ≥ 0 $ odd cohomology basis elements $ h_i: L_i → L_{i+1} $ such that none of them is the co-identity. Let $ a $ be an arc. Let $ m ∈ \Hom_{\HTw\Gtl_q Q} (L_{N+1}, a) $ be an odd intersection. Then we have
\begin{equation*}
μ_{\HTw\Gtl_q Q} (m, h_N, …, h_1) = \sum_{\substack{\text{MD disk } D \\ \text{with inputs } h_1, …, h_N, m}} (-1)^{\Abouzaid(D)} \punctures(D) m^*.
\end{equation*}
Let $ m ∈ \Hom_{\HTw\Gtl_q Q} (L_{N+1}, a) $ be an even intersection. Then the product $ μ(m, h_N, …, h_1) $ vanishes, except if $ N = 1 $ and $ h_1 $ is the odd intersection at $ a $ between the two zigzag paths departing from $ a $. In this case, we have
\begin{equation*}
μ_{\HTw\Gtl_q Q} (m, h_1) = - m^*.
\end{equation*}
\end{lemma}

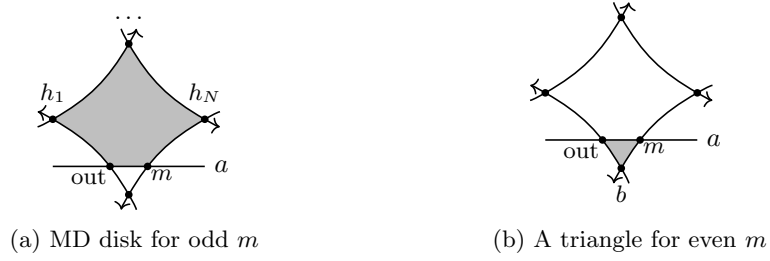
\begin{figure}
\centering
\begin{subfigure}{0.4\linewidth}
\centering
\begin{tikzpicture}
\path (0, 0) coordinate (A) (-1, 1) coordinate (B) (0, 2) coordinate (C) (1, 1) coordinate (D);
\path[draw, semithick] (A) to[out=300, in=100, looseness=0.5] ++(0.1, -0.2);
\path[draw, semithick, ->] (A) to[out=240, in=80, looseness=0.5] ++(-0.1, -0.2);
\path[draw, semithick, ->] (B) to[out=150, in=350, looseness=0.5] ++(-0.2, 0.1);
\path[draw, semithick] (B) to[out=210, in=10, looseness=0.5] ++(-0.2, -0.1);
\path[draw, semithick, ->] (C) to[out=60, in=260, looseness=0.5] ++(0.1, 0.2);
\path[draw, semithick] (C) to[out=120, in=280, looseness=0.5] ++(-0.1, 0.2);
\path[draw, semithick, ->] (D) to[out=330, in=170, looseness=0.5] ++(0.2, -0.1);
\path[draw, semithick] (D) to[out=30, in=190, looseness=0.5] ++(0.2, 0.1);
\path (A) to[out=120, in=330] coordinate[pos=0.3] (out) (B) to[out=30, in=240] (C) to[out=300, in=150] (D) to[out=210, in=60] coordinate[pos=0.7] (m) cycle;
\begin{scope}
\clip (out) -- (B) -- (C) -- (D) -- (m) -- cycle;
\path[fill=gray!50] (A) to[out=120, in=330] (B) to[out=30, in=240] (C) to[out=300, in=150] (D) to[out=210, in=60] cycle;
\end{scope}
\path[draw, semithick] (A) to[out=120, in=330] node[midway, above] {$ \phantom{L'} $} (B) to[out=30, in=240] (C) to[out=300, in=150] (D) to[out=210, in=60] node[midway, above] {$ \phantom{L} $} cycle;
\path[draw, semithick] ($ (out)!0.5!(m) $) ++(left:1) -- ++(right:2) node[pos=1, right] {$ a $};
\path[fill] \foreach \i in {A, B, C, D, out, m} {(\i) circle[radius=0.05]};
\path (out) node[below left, inner sep=0.1em] {\small out};
\path (m) node[below right, inner sep=0.1em] {\small $ m $};
\path (B) node[above, shift={(up:0.1)}] {\small $ h_1 $};
\path (C) node[above, shift={(up:0.2)}] {\small $ … $};
\path (D) node[above, shift={(up:0.1)}] {\small $ h_N $};
\end{tikzpicture}
\caption{MD disk for odd $ m $}
\end{subfigure}
\begin{subfigure}{0.4\linewidth}
\centering
\begin{tikzpicture}
\path (0, 0) coordinate (A) (-1, 1) coordinate (B) (0, 2) coordinate (C) (1, 1) coordinate (D);
\path[draw, semithick] (A) to[out=300, in=100, looseness=0.5] ++(0.1, -0.2);
\path[draw, semithick, ->] (A) to[out=240, in=80, looseness=0.5] ++(-0.1, -0.2);
\path[draw, semithick, ->] (B) to[out=150, in=350, looseness=0.5] ++(-0.2, 0.1);
\path[draw, semithick] (B) to[out=210, in=10, looseness=0.5] ++(-0.2, -0.1);
\path[draw, semithick, ->] (C) to[out=60, in=260, looseness=0.5] ++(0.1, 0.2);
\path[draw, semithick] (C) to[out=120, in=280, looseness=0.5] ++(-0.1, 0.2);
\path[draw, semithick, ->] (D) to[out=330, in=170, looseness=0.5] ++(0.2, -0.1);
\path[draw, semithick] (D) to[out=30, in=190, looseness=0.5] ++(0.2, 0.1);
\path (A) to[out=120, in=330] coordinate[pos=0.3] (out) (B) to[out=30, in=240] (C) to[out=300, in=150] (D) to[out=210, in=60] coordinate[pos=0.7] (m) cycle;
\begin{scope}
\clip (out) -- (m) -- (A) -- cycle;
\path[fill=gray!50] (A) to[out=120, in=330] (B) to[out=30, in=240] (C) to[out=300, in=150] (D) to[out=210, in=60] cycle;
\end{scope}
\path[draw, semithick] (A) to[out=120, in=330] node[midway, above] {$ \phantom{L'} $} (B) to[out=30, in=240] (C) to[out=300, in=150] (D) to[out=210, in=60] node[midway, above] {$ \phantom{L} $} cycle;
\path[draw, semithick] ($ (out)!0.5!(m) $) ++(left:1) -- ++(right:2) node[pos=1, right] {$ a $};
\path[fill] \foreach \i in {A, B, C, D, out, m} {(\i) circle[radius=0.05]};
\path (out) node[below left, inner sep=0.1em] {\small out};
\path (m) node[below right, inner sep=0.1em] {\small $ m $};
\path (A) node[below, shift={(down:0.1)}] {\small $ b $};
\end{tikzpicture}
\caption{A triangle for even $ m $}
\end{subfigure}
\caption{Disks contributing to products $ μ(m, b, …, b) $}
\label{fig:zigzag-mirobjects-MDMT}
\end{figure}

\begin{remark}
In case $ Q $ is geometrically consistent, the differential $ μ^1_{\HTw\Gtl_q Q} (m) $ vanishes. This is also the case if $ Q = Q_M $ for even $ M $. In case of $ Q = Q_M $ with $ M $ odd, the differential is instead given by counting digons which lie on the clockwise face of $ Q_M $. By definition of MD disks, this, corner case is included in \autoref{th:zigzag-mirobjects-products}.
\end{remark}

%% file: flatness/intro.tex
\section{Flatness of superpotential deformations}
\label{sec:flatness}
In this section, we show that a superpotential deformation of a CY3 Jacobi algebra is flat under assumption of a certain boundedness condition:

\begin{center}
\begin{tikzpicture}
\path (0, 0) node[align=center] (A) {\textbf{Jacobi algebra} \\ $ \Jac(Q, W) = \frac{ℂQ}{(∂_a W)} $} (10, 0) node[align=center] (B) {\textbf{Deformed Jacobi algebra} \\ $ \Jac(Q, W_q) = \frac{B \htensor ℂQ}{(∂_a W_q)} $};
\path[draw, ->] ($ (A.east)!0.1!(B.west) $) -- ($ (A.east)!0.9!(B.west) $) node[midway, above] {boundedness condition};
\end{tikzpicture}
\end{center}

Deformation theory for algebras has historically followed the question when a deformation of an ideal gives rise to a deformation of the algebra. The core indicators of this development are the type of algebra studied, the degree $ |R| $ of the relations in the algebra, and the degrees $ |R'| $ admitted for the deformation of the relations. Past development shows a continuous improvement on these two degrees:

\begin{center}
\begin{tabular}{cccccc}
Year & Authors & Type of algebra & Deformation & $ |R| $ & $ |R'| $ \\\hline
1994 & Braverman and Gaitsgory \cite{Braverman-Gaitsgory} & Koszul & PBW & 2 & 2 \\
2001 & Berger \cite{Berger} & $ N $-Koszul & — & $ N $ & — \\
2006 & Berger and Ginzburg \cite{Berger-Ginzburg} & $ N $-Koszul & PBW & $ N $ & $ < N $ \\
2006 & Berger and Taillefer \cite{Berger-Taillefer} & CY3 & PBW & $ N $ & $ < N $ \\
2023 & \emph{this paper} & CY3 & formal & bounded & any
\end{tabular}
\end{center}

In the present section, we examine this “flatness question” for the case of formal deformations of CY3 Jacobi algebras. We build on the work of Ginzburg-Berger, which shows that CY3-ness and homogeneity of the superpotential $ W $ with respect to a positive integer grading on $ Q $ give rise to an inductive argument that ensures flatness. The proof can however be continued without homogeneity requirement and we land in the completed path algebra $ \compl{ℂQ} $. In this section, we construct a boundedness condition which allows us to land back in $ ℂQ $ by means of a posteriori estimates. The result is a flatness theorem that goes beyond the work of Berger and Ginzburg. Finally, we show that Jacobi algebras of most dimers satisfy our boundedness condition.

The general context of our work is the theory of deformations of associative algebras. The main question can be sketched as follows: Let $ A $ be an algebra and $ I ⊂ A $ an ideal, then we have the quotient algebra $ A / I $. Let now $ B $ be a deformation base, for instance $ B = ℂ⟦q⟧ $. Let now $ I_q ⊂ B \htensor A $ be an ideal, for instance $ I_q ⊂ A⟦q⟧ $. We can then form the quotient algebra $ (B \htensor A) / I_q $, for instance $ A⟦q⟧ / I_q $. The main question is:

\begin{center}
How should $ I $ and $ I_q $ be related so that $ (B \htensor A) / I_q $ is a deformation of $ A / I $?
\end{center}

Our results are presented in \autoref{sec:flatness-flatness} and can be summarized as follows: Let $ A = \Jac(Q, W) $ be a CY3 algebra given by a quiver $ Q $ with superpotential $ W $. The relations of this algebra are the derivatives $ ∂_a W $ of the superpotential. Let $ W' ∈ \mathfrak{m} ℂQ $ be a deformation of the superpotential, that is, any additional cyclic terms lie in order $ \mathfrak{m} $. Then we can consider the deformed relations $ ∂_a (W + W') ∈ B \htensor ℂQ $. We show that if the relations $ ∂_a W $ satisfy a certain boundedness condition, then $ \Jac (Q, W + W') $ is automatically a deformation of $ \Jac(Q, W) $.

The boundedness condition for the relations $ ∂_a W $ is best described as follows: Each relation $ ∂_a W $ is a linear combination of paths in $ Q $. We define an equivalence relation on the set of paths in $ Q $ as follows:
\begin{itemize}
\item If $ p $ and $ q $ are paths appearing in the same relation $ ∂_a W $, then $ p \sim q $.
\item If $ p \sim q $, then $ apb \sim aqb $ for arbitrary paths $ a, b $.
\item Take the transitive hull.
\end{itemize}
This equivalence relation partitions the set of paths in $ Q $ into equivalence classes. Our boundedness condition demands that the path length in every equivalence class is bounded. In other words, we demand that path length cannot increase to infinity if we apply (parts of) relations.

We have structured this section as follows: In \autoref{sec:flatness-whatis}, we provide some intuition and terminology regarding flatness. In \autoref{sec:flatness-BGstandard}, we recollect the tools of Berger and Ginzburg \cite{Berger-Ginzburg}. In \autoref{sec:flatness-setup}, we set up notation and formulate a stronger versions of the tools of Berger and Ginzburg. In \autoref{sec:flatness-bounded}, we introduce our boundedness argument and demonstrate its strength, most of which can be read independent of the deformations context. In \autoref{sec:flatness-ideals}, we introduce more notation for ideal-like sets which allows us to apply the boundedness argument to the flatness question. In \autoref{sec:flatness-BGbounded}, we prove a bounded version of the tools of Berger and Ginzburg. In \autoref{sec:flatness-flatcompleted} we derive our first flatness result which concerns quasi-flatness of a certain auxiliary ideal in $ B \htensor \compl{ℂQ} $. In \autoref{sec:flatness-flatalgebraic}, we prove our second flatness result by tracing our way back to the non-completed path algebra. In \autoref{sec:flatness-closedness}, we provide an a posteriori interpretation of the auxiliary ideal, which promotes our first flatness result to a flatness result for completed ideals in $ \compl{ℂQ} $. In \autoref{sec:flatness-dimers}, we provide criteria for the Jacobi algebra of a dimer to satisfy our boundedness condition. In \autoref{sec:flatness-flatness}, we streamline our flatness theorems.

%% file: flatness/flatness_intro.tex
\subsection{Flatness and quasi-flatness}
\label{sec:flatness-whatis}
In this section, we recapitulate flatness of algebra deformations and quasi-flatness of ideal deformations. The core connection is that a quasi-flat ideal deformation typically makes the quotient algebra a flat algebra deformation:

\begin{center}
\begin{tikzpicture}
\path (0, 0) node[align=center] (A) {\textbf{Quasi-flat ideal deformation} \\ $ I_q ⊂ B \htensor A $ of $ I ⊂ A $} (8, 0) node[align=center] (B) {\textbf{Flat algebra deformation} \\ $ (B \htensor A) / I_q $ of $ A/I $};
\path[draw, ->] ($ (A.east)!0.2!(B.west) $) -- ($ (A.east)!0.8!(B.west) $);
\end{tikzpicture}
\end{center}

It is our aim to study deformations of ideals. We deploy notation and terminology from \autoref{sec:prelim-submodules}, in particular the distinct notation $ \mathfrak{m}^k Y \neq \mathfrak{m}^k · Y $. Let $ A $ be an algebra and $ I ⊂ A $ an ideal. Regard a deformation base $ B $ and let $ I_q ⊂ B \htensor A $ be an ideal. The question is in which cases $ (B \htensor A) / I_q $ can be identified with $ B \htensor (A/I) $. A first observation is that $ I_q $ needs to be close enough to $ I $. More precisely, we should have $ π(I_q) = I $, or in other words $ I_q + \mathfrak{m} A = I + \mathfrak{m} A $. This condition is not strong enough however. The basic problem is that $ I_q $ may be too large in higher orders of $ \mathfrak{m} $. Elements of $ A $ may get unexpectedly annihilated in the quotient $ (B \htensor A) / I_q $ when multiplied by elements of $ \mathfrak{m} $. The further condition that $ I_q $ therefore needs to satisfy is the quasi-flatness property $ I_q ∩ \mathfrak{m} A ⊂ \mathfrak{m} I_q $.

We start by fixing terminology for deformations of algebras and ideals as follows:

\begin{definition}
\label{def:flatness-whatis-Adefo}
Let $ A $ be an algebra and $ B $ a deformation base. Then a $ B $-algebra $ A_q $ is a \emph{deformation} of $ A $ if there is a $ B $-linear algebra isomorphism $ φ: A_q \isoto (B \htensor A, μ_q) $ where $ μ_q $ is a deformation of the product $ μ: A ¤ A → A $.
\end{definition}

\begin{definition}
\label{def:flatness-whatis-Idefo}
Let $ A $ be an algebra and $ I ⊂ A $ an ideal. Let $ B $ be a deformation base and $ I_q ⊂ B \htensor A $ an ideal. Then $ I_q $ is a \emph{deformation} of $ I $ if $ I + \mathfrak{m} A = I_q + \mathfrak{m} A $.
\end{definition}

\begin{remark}
Our terminology is slightly confusing: We have decided to call $ I_q $ a deformation of $ I $ already if it is loosely related to $ I $. In contrast, our notion for deformations of algebras is very strict.
\end{remark}

Recall from \autoref{sec:prelim-submodules} that $ I_q $ is \emph{quasi-flat} if $ I_q ∩ \mathfrak{m} A ⊂ \mathfrak{m} I_q $. The ideal $ I_q $ is \emph{pseudoclosed} if $ B I_q ⊂ I_q $. In the following, we present two sample deformations of the ideal $ I = (X) ⊂ A = ℂ[X] $ over $ B = ℂ\llbracket q \rrbracket $. One of the deformations is quasi-flat and the other is not:

\begin{center}
\begin{tabular}{ccc}
Object & Quasi-flat & Not quasi-flat \\\hline
Algebra $ A $ & $ ℂ[X] $ & $ ℂ[X] $ \\
Ideal $ I $ & $ (X) $ & $ (X) $ \\
Quotient $ A/I $ & $ ℂ $ & $ ℂ $ \\
Deformed ideal $ I_q $ & $ (X+q) $ & $ (X) + (q) $ \\
$ I_q ∩ qA $ & $ q I_q $ & $ q A $ \\
Quotient $ A⟦q⟧/I_q $ & $ \frac{ℂ[X]⟦q⟧}{(X+q)} = ℂ⟦q⟧ $ & $ \frac{ℂ[X]⟦q⟧}{(X) + (q)} = ℂ $ \\
\end{tabular}
\end{center}

Not every deformation of an algebra $ A/I $ can be described through a deformation of its ideal $ I $. Conversely, not every deformation $ I $ gives rise to a deformation of $ A/I $. We recall here a classification of those deformations $ I_q $ that give rise to deformations of $ A/I $:

\begin{proposition}
\label{th:flatness-whatis-equivalence}
Let $ A $ be an algebra and $ I ⊂ A $ an ideal. Let $ B $ be a deformation basis and $ I_q ⊂ B \htensor A $ a deformation of $ I $. Put $ A_q = (B \htensor A) / I_q $. Then the following are equivalent:
\begin{itemize}
\item $ I_q $ is quasi-flat and pseudoclosed.
\item $ A_q $ is a deformation of $ A/I $.
\end{itemize}
In this case, the $ \mathfrak{m} $-adic topology on $ A_q $ agrees with the quotient topology and any $ B $-linear isomorphism $ A_q \isoto B \htensor (A/I) $ is automatically continuous.
\end{proposition}

\begin{proof}
We divide the proof into three parts: We first show that $ I_q $ being quasi-flat and pseudoclosed implies that $ A_q $ is a deformation. Second we show the converse, and third we draw the topological conclusions.

For the first part, assume $ I_q $ is quasi-flat and pseudoclosed. It is our task to find an isomorphism of $ B $-algebras $ φ: (B \htensor A) / I_q \isoto (B \htensor (A/I), μ_q) $, where $ μ_q $ is a deformation of the algebra structure of $ A/I $. Pick a complement $ V ⊂ A $ of $ I $ such that $ A = I ⊕ V $. Both $ I_q ⊂ B \htensor A $ and $ BV ⊂ B \htensor A $ are quasi-flat and pseudoclosed. By \autoref{th:prelim-submodules-criterion} we conclude that $ B \htensor A = I_q ⊕ BV $. We obtain a $ B $-linear isomorphism
\begin{equation*}
φ: \frac{B \htensor A}{I_q} = \frac{I_q ⊕ BV}{I_q} \isoto BV \isoto \frac{BI ⊕ BV}{BI} \isoto B \htensor \frac{A}{I}.
\end{equation*}
This already provides $ φ $ as $ B $-linear map. It remains to check that the algebra structure $ μ_q $ induced on $ B \htensor (A/I) $ from $ (B \htensor A)/I_q $ via $ φ $ is a deformation of the natural algebra structure $ μ $ of $ A/I $. Pick $ a, b ∈ V ⊂ A $. Write $ ab = x + m + v $ with $ x+m ∈ I_q $, $ x ∈ I $, $ m ∈ \mathfrak{m} A $, $ v ∈ BV $. Projecting $ ab $ to $ BV $ along $ I_q ⊕ BV $ gives $ v $. Projecting $ ab $ to $ BV $ along $ BI ⊕ BV $ instead gives $ v + \landau(\mathfrak{m}) $. This shows that $ μ_q $ is a deformation of $ μ $.

For the second part, assume there is an isomorphism of $ B $-modules $ φ: (B \htensor A) / I_q → B \htensor (A/I) $. Regard the $ B $-linear projection map $ π: B \htensor A → (B \htensor A) / I_q $. The composition $ φπ: B \htensor A → B \htensor (A/I) $ is $ B $-linear and surjective. Pick a linear section $ A/I → B \htensor A $ of $ φπ $ and extend to a $ B $-linear and continuous map $ ψ: B \htensor (A/I) → B \htensor A $ following \autoref{rem:prelim-ainfty-continuityautomatic}. Thanks to continuity, $ ψ $ is a section of $ φπ $ in the sense that $ φπψ = \id $. In other words, we have the $ B $-linear map $ ψφ: (B \htensor A) / I_q → B \htensor A $ with $ π(ψφ) = \id $. This shows that the projection $ π $ has a $ B $-linear section, hence $ I_q $ is a direct summand of $ B \htensor A $. By \autoref{th:prelim-submodules-criterion}, we conclude that $ I_q $ is quasi-flat and pseudoclosed.

For the third part, let $ φ: A_q → B \htensor (A/I) $ be any $ B $-linear isomorphism. By \autoref{th:prelim-freemodules-presentingcontinuity}, this map is automatically continuous when $ A_q $ is equipped with the $ \mathfrak{m} $-adic topology. By \autoref{rem:prelim-ainfty-continuityautomatic}, the composition $ B \htensor A → A_q \isoto B \htensor (A/I) $ is automatically continuous as well, and the universal property of the quotient topology renders $ φ $ continuous when $ A_q $ is equipped with the quotient topology. This settles all claims.
\end{proof}

\begin{remark}
When $ A $ is unital, then the $ B $-algebra isomorphism $ φ: (B \htensor A) / I_q → (B \htensor A, μ_q) $ can be chosen to preserve the unit. Indeed, one simply chooses the complement $ V $ to contain the unit.
\end{remark}

\begin{remark}
\label{rem:flatness-whatis-semisimple}
The statement of \autoref{th:flatness-whatis-equivalence} still holds true when $ A $ is an algebra over a semisimple algebra $ Λ $, for instance $ Λ = ℂQ_0 $ with $ Q $ a quiver. The ideals $ I $ and $ I_q $ are automatically a $ Λ $-submodule as well, and the complement $ V ⊂ A $ can be chosen as a $ Λ $-submodule due to semisimplicity. The $ B $-algebra isomorphism $ φ: (B \htensor A) / I_q → (B \htensor (A/I), μ_q) $ then becomes $ Λ $-linear.
\end{remark}

Let us illustrate flatness in the case of deformations of a superpotential $ W $ on a quiver $ Q $. Our typical starting point is a cyclic deformation $ W_q ∈ B \htensor ℂQ $ of $ W $. This deformation gives rise to deformed relations $ ∂_a W_q $ for $ a ∈ Q_1 $. The question we try to settle is when the ideal $ (∂_a W_q) $ or its closure in $ B \htensor ℂQ $ is quasi-flat. On an intuitive level this means the following: We start from any element $ x ∈ ℂQ $ and continue adding up deformed relations $ ∂_a W_q $:
\begin{equation*}
x + p_1 ∂_{a_1} W_q q_1 + p_2 ∂_{a_2} W_q q_2 + …
\end{equation*}
Here $ p_i, q_i $ are arbitrary paths in $ Q $. If we arrive at some point at an element of the form $ x + \landau(\mathfrak{m}) $, does the infinitesimal part necessarily vanish or have we incurred additional terms in higher order? If the infinitesimal part always vanishes, then the deformed ideal $ (∂_a W_q) $ contains no unexpected new relations. If the infinitesimal part consists of higher-order multiples of relations, then the deformed ideal $ (∂_a W_q) $ still contains no new relations. If the infinitesimal part consists of strictly more than $ \mathfrak{m} (∂_a W_q) $, then the deformed ideal $ (∂_a W_q) $ contains new relations in higher order which cannot be made by combining existing relations in lower order.

Relations in higher order which cannot be made from combining existing relations in lower order indicate an ideal that is not quasi-flat. More precisely, these relations lie in $ ((∂_a W_q) ∩ \mathfrak{m} ℂQ) \setminus \mathfrak{m} (∂_a W_q) $. They kill more than expected in the higher-order part of the quotient $ (B \htensor ℂQ) / (∂_a W_q) $ and prevent the quotient from being isomorphic to $ B \htensor ℂQ / (∂_a W) $.

%% file: flatness/BGstandard.tex
\subsection{Berger-Ginzburg inclusion}
\label{sec:flatness-BGstandard}
In this section we recall superpotentials, the CY3 property and the Berger-Ginzburg inclusion. More precisely, we start by recalling superpotentials and their associated Jacobi algebras. In terms of a bimodule resolution, we recall what it means for an algebra to have the CY3 property. Jacobi algebras are sometimes CY3. If so, their superpotentials satisfies a certain condition which appeared in \cite{Berger-Ginzburg}. The reformulation of the CY3 property in terms of this condition is essential for us, and we shall refer to this condition as the \emph{Berger-Ginzburg inclusion}.

\begin{remark}
This section is meant as a motivational section. The reader who wishes to take the origin of the Berger-Ginzburg inclusion for granted is advised to skip to \autoref{sec:flatness-setup}. For instance, the original Berger-Ginzburg inclusion presented here is in fact not sufficient for our purposes, so we will build a stronger Berger-Ginzburg inclusion from scratch in \autoref{sec:flatness-setup}.
\end{remark}

To start with, let $ Q $ be a quiver. Recall from \autoref{sec:koszul-cy3} that a superpotential on $ Q $ is a cyclic element of $ ℂQ_{≥2} $. Recall also that we may denote by $ W $ the $ ℂQ_0 $-bimodule generated by $ W $ in $ ℂQ $:
\begin{equation*}
W = ℂ Q_0 W Q_0 = \bigoplus_{v ∈ Q_0} ℂvWv ⊂ ℂQ.
\end{equation*}
Recall that we denote the relations space by
\begin{equation*}
R = \vspan\{∂_a W \running a ∈ Q_1\}.
\end{equation*}

\begin{remark}
By abuse of notation, we have used the same notation for $ W ∈ ℂQ_{≥2} $ as for its associated bimodule in the above definition. We shall stick to this notation. The distinction should always be clear. In \autoref{sec:flatness-setup}, the terminology will be changed and improved.
\end{remark}

Recall that the Jacobi algebra associated with $ (Q, W) $ is the algebra $ A = \Jac(Q, W) = ℂQ / (∂_a W) $. Let us temporarily keep the shorthand $ A $ for the Jacobi algebra. Recall from \autoref{sec:koszul-cy3} the candidate resolution of $ A $ as $ A $-bimodule:
\begin{equation}
\label{eq:flatness-resolution}
0 → A \tensor_{ℂQ_0} W \tensor_{ℂQ_0} A \overset{g_1}{→} A \tensor_{ℂQ_0} R \tensor_{ℂQ_0} A \overset{g_2}{→} A \tensor_{ℂQ_0} ℂQ_1 \tensor_{ℂQ_0} A \overset{g_3}{→} A \tensor_{ℂQ_0} A → A → 0.
\end{equation}
We have described this sequence in \autoref{rem:koszul-cy3-maps}. As we recall in \autoref{sec:koszul-cy3}, the Jacobi algebra $ A = \Jac(Q, W) $ is CY3 if and only if the sequence is exact:

\begin{theorem}[{\cite[Theorem 4.3]{Bocklandt-CY}}]
Let $ Q $ be a quiver and $ W ∈ ℂQ_{≥3} $ a superpotential. Then the algebra $ \Jac(Q, W) $ is CY3 if and only if the sequence \eqref{eq:flatness-resolution} is exact.
\end{theorem}

Recall also from \autoref{th:koszul-cyord-tensorresolution} that the exact sequence can be used to produce projective resolutions for all left and right modules of $ A $. When tensoring the sequence on the left with $ ℂQ_0 $ over $ A $, then we obtain a resolution of $ ℂQ_0 $ as right $ A $-modules:
\begin{equation}
\label{eq:flatness-sequence}
0 → W \tensor_{ℂQ_0} A \overset{f_1}{→} R \tensor_{ℂQ_0} A \overset{f_2}{→} ℂQ_1 \tensor_{ℂQ_0} A \overset{f_3}{→} A → ℂQ_0 → 0.
\end{equation}

\begin{remark}
The maps $ f_1, f_2, f_3 $ are described as follows: For $ f_1 $, write an element $ w = \sum_{i ∈ I} r_i a_i $ in terms of relations $ r_i ∈ R $ and arrows $ a_i ∈ Q_1 $, then map $ f_1 (w) = \sum_{i ∈ I} r_i ¤ a_i $. For $ f_2 $, write an element $ r = \sum_{i ∈ I} a_i p_i $ with arrows $ a_i $, then map $ f_2 (r) = \sum_{i ∈ I} a_i ¤ p_i $. The map $ f_3 $ is the multiplication map and the fourth map is the projection onto the quotient of $ A $ by all arrows.
\end{remark}

Requiring an algebra to be CY3 is an algebro-geometric condition. Berger and Ginzburg \cite{Berger-Ginzburg} translate this condition for $ A = \Jac(Q, W) $ to a more tractable condition in terms of the superpotential and relations. More precisely, their approach is to evaluate exactness of the sequence \eqref{eq:flatness-sequence} at the module second from the left. The result is an inclusion in terms of $ R $ and $ W $:

\begin{lemma}[{\cite[Theorem 2.6]{Berger-Ginzburg}}]
\label{th:flatness-BGstandard-BGstandard}
If $ \Jac(Q, W) $ is CY3, then the \emph{Berger-Ginzburg inclusion} holds:
\begin{equation}
\label{eq:flatness-BGstandard}
R ℂQ ∩ ℂQ_1 R ⊂ W ℂQ + R I(R).
\end{equation}
\end{lemma}

\begin{remark}
The inclusion \eqref{eq:flatness-BGstandard} can be roughly interpreted as follows: Let $ p ∈ ℂQ $ be a path. Add any amount of ideal elements $ ry ∈ R ℂQ $ and $ ar ∈ ℂQ_1 R $ to $ p $. If we land at $ p $ again after the addition, the additions of the two types have been equal and therefore lie in the intersection $ R ℂQ ∩ ℂQ_1 R $. By the inclusion \eqref{eq:flatness-BGstandard}, the relations we have added are equivalent to applying the entire set of relations $ ∂_a W $ around some vertices and terms that “vanish quadratically” in $ A / I(R) $.
\end{remark}

Berger and Taillefer \cite{Berger-Taillefer} use the intermediate results of Berger and Ginzburg \cite{Berger-Ginzburg} to deduce flatness of PBW deformations. Apart from the CY3 condition in the form of the Berger-Ginzburg inclusion, their core assumption is that the superpotential $ W ∈ ℂQ_{≥2} $ is homogeneous in path length. Translated to the setting of formal deformations, their statement would read as follows:

\begin{proposition}
\label{th:flatness-BGstandard-BGflatness}
Let $ Q $ be a quiver and $ W ∈ ℂQ_{≥2} $ a homogeneous superpotential. If $ (B, \mathfrak{m}) $ is a deformation base and $ W' ∈ \mathfrak{m} ℂQ_{≥2} $ is a cyclic element, then the deformed ideal $ B (∂_a W + ∂_a W') $ is quasi-flat. In other words, the “deformed Jacobi algebra”
\begin{equation*}
\Jac(Q, W + W') ≔ \frac{B \htensor ℂQ}{B (∂_a W + ∂_a W')}
\end{equation*}
is a deformation of $ \Jac(Q, W) $.
\end{proposition}

We shall not prove \autoref{th:flatness-BGstandard-BGstandard} and \autoref{th:flatness-BGstandard-BGflatness} here. Namely, the core assumption that $ W $ is homogeneous is too restrictive for our purposes. We will prove stronger versions of both.

%% file: flatness/setup.tex
\subsection{Notation and conventions}
\label{sec:flatness-setup}
In this section, we define the setup in which we prove flatness properties. As a first step, we start from relation spaces instead of superpotentials. Second, we introduce a stronger version of the Berger-Ginzburg inclusion. We fix both items in a notational framework, which does not make explicit reference to superpotentials. Finally, we show that relations coming from CY3 superpotentials fall within the framework.

It may happen that one is interested in checking flatness of Calabi-Yau like algebras where the relation space $ R $ is given, without the relations coming from a single superpotential $ W $. In this case, one may form the space $ W = ℂQ_1 R ∩ R ℂQ_1 $ which serves as a replacement for the superpotential:

\begin{center}
\begin{tikzpicture}
\path (-1, 0) node {\textbf{CY3 viewpoint}} (3, 0) node[align=center] {$ W ∈ ℂQ_{≥2} $ \\ superpotential} (5, 0) node {\Large $ \rightsquigarrow $} (7.5, 0) node[align=center] {$ R = \vspan(∂_a W) $ \\ relation space};
\path (-1, -1.5) node {\textbf{Deformation viewpoint}} (3, -1.5) node[align=center] {$ R ⊂ ℂQ $ \\ relation space} (5, -1.5) node {\Large $ \rightsquigarrow $} (7.5, -1.5) node[align=center] {$ W = ℂQ_1 R ∩ R ℂQ_1 $ \\ superpotential space};
\end{tikzpicture}
\end{center}

The approach of starting with a relation space $ R $ is also the context used in \cite{Berger-Ginzburg}. We will codify the setup in \autoref{conv:flatness-relations}. During the remainder of \autoref{sec:flatness}, we stick to this convention.

\begin{convention}
\label{conv:flatness-relations}
The letter $ Q $ denotes a quiver. The space $ R ⊂ ℂQ_{≥1} $ is a finite-dimensional $ ℂQ_0 $-bimodule of \emph{relations}. We write $ W ≔ ℂQ_1 R ∩ R ℂQ_1 $. When regarding deformations, $ (B, \mathfrak{m}) $ is a deformation base and $ ψ: R → \mathfrak{m} \compl{ℂQ} $ is a $ ℂQ_0 $-bimodule map. The ideal generated by $ R $ is denoted $ I(R) = ℂQ R ℂQ ⊂ ℂQ $. The space of \emph{deformed relations} is denoted $ P ≔ (\Id + ψ) (R) ⊂ B \htensor \compl{ℂQ} $. Two additional properties may be assumed:
\begin{itemize}
\item[{[BG]}] The \emph{strong Berger-Ginzburg inclusion}
\begin{equation}
\label{eq:flatness-BGstrong}
R ¤ ℂQ ∩ \contraction^{-1} (ℂQ_1 I(R)) ⊂ f(W ¤ ℂQ) + R ¤ I(R).
\end{equation}
Here, $ \contraction: R ¤ ℂQ → R ℂQ $ denotes the contraction map and $ f: W ¤ ℂQ → R ¤ ℂQ_1 ℂQ $ denotes the map that splits $ W $ into relations and arrows.
\item[{[CP]}] The property $ (ℂQ_1 P + P ℂQ_1) ∩ \mathfrak{m} \compl{ℂQ} = 0 $.
\end{itemize}
\end{convention}

\begin{remark}
We drop subscripts from the tensor product $ ¤ $. When the two tensorands come from quiver algebras, the tensor product always refers to the tensor product over $ ℂQ_0 $:
\begin{align*}
R ¤ ℂQ = R ¤_{ℂQ_0} ℂQ, \quad R ¤ I(R) = R ¤_{ℂQ_0} I(R), \quad W ¤ ℂQ = W ¤_{ℂQ_0} ℂQ, \quad …
\end{align*}
\end{remark}

In the remainder of the section, we illustrate which algebras fall under the framework of \autoref{conv:flatness-relations}. In particular, we show that Jacobi algebras with the CY3 property satisfy the conditions [BG] and [CP].

\begin{example}
The algebra $ ℂ⟨A, B, C⟩ / (BC, CA, AB) $ falls into the framework of \autoref{conv:flatness-relations}. Its relations are given by $ R = \vspan(BC, CA, AB) $ and by definition we have
\begin{equation*}
W = ℂQ_1 R ∩ R ℂQ_1 = \vspan(ABC, BCA, CAB).
\end{equation*}
This space is three-dimensional and not spanned by a single superpotential on $ ℂQ $. The algebra is not CY3. However the strong Berger-Ginzburg inclusion is satisfied:
\begin{equation*}
R ¤ ℂQ ∩ \contraction^{-1} (ℂQ_1 I(R)) = BC ¤ AℂQ + CA ¤ BℂQ + AB ¤ CℂQ = f(W ¤ ℂQ).
\end{equation*}
The algebra therefore falls within our framework \autoref{conv:flatness-relations} and can therefore be treated with our approach.

For this specific algebra another approach to deformations is possible: In \cite{Barmeier-Wang}, Barmeier and Wang study quiver algebras whose relations are defined by so-called reduction systems. The deformations of these algebras are in combinatorial correspondence with deformations of their reduction system. The specific relations $ BC = CA = AB = 0 $ constitute a reduction system, therefore this specific algebra falls under the framework of \cite{Barmeier-Wang}.
\end{example}

Next, we show that the case of superpotentials falls within the framework of \autoref{conv:flatness-relations}. As a preparation, we need the following lemma, modeled after \cite{Berger-Ginzburg}.

\begin{lemma}
\label{th:flatness-setup-psi}
Assume [CP]. Then the two replacement maps
\begin{align*}
ψ^0: W &→ \mathfrak{m} ℂQ_1 \compl{ℂQ}, \quad \sum_{i ∈ I} a_i r_i ↦ \sum_{i ∈ I} a_i ψ(r_i), \\
ψ^1: W &→ \mathfrak{m} \compl{ℂQ} ℂQ_1, \quad \sum_{i ∈ I} r_i a_i ↦ \sum_{i ∈ I} ψ(r_i) a_i.
\end{align*}
are well-defined and equal.
\end{lemma}

\begin{proof}
For well-definedness, it suffices to note that $ ψ^0 $ and $ ψ^1 $ are nothing else than the splitting maps $ W ⊂ R ℂQ_1 → R ¤ ℂQ_1 $ and $ W ⊂ ℂQ_1 R → ℂQ_1 ¤ R $, composed with $ ψ $ acting either on the left or right factor. Now let us explain why $ ψ^0 = ψ^1 $. Let $ w ∈ W $, then
\begin{equation*}
ψ^0 (w) - ψ^1 (w) = (w + ψ^0 (w)) - (w + ψ^1 (w)) ∈ P ℂQ_1 + ℂQ_1 P.
\end{equation*}
The left-hand side simultaneously lies in $ \mathfrak{m} \compl{ℂQ} $. By [CP], the difference $ ψ^0 (w) - ψ^1 (w) $ vanishes. This proves $ ψ^0 = ψ^1 $.
\end{proof}

We are now ready to see how the case of relations coming from a superpotential fits into the framework of \autoref{conv:flatness-relations}. Let $ W $ be a superpotential and $ W' ∈ \mathfrak{m} ℂQ $ a deformation. We simply put $ R = \vspan\{∂_a W\} $ and $ ψ(∂_a W) ≔ ∂_a W' $. To make these definitions work, one needs to check that the relations $ ∂_a W $ are linearly independent and that $ ℂQ_1 R ∩ R ℂQ_1 $ is indeed the $ ℂQ_0 $-bimodule generated by the superpotential $ W ∈ ℂQ $. We verify this in the following lemma. Note that $ W $ is required to consist of paths of length $ ≥ 3 $ this time:

\begin{lemma}
\label{th:flatness-BGstrong-Wcase}
Let $ Q $ be a quiver with superpotential $ W ∈ ℂQ_{≥3} $. Assume $ \Jac(Q, W) $ is CY3. Then the relations $ ∂_a W $ for $ a ∈ Q_1 $ are linearly independent. Put $ R = \vspan\{∂_a W\} $. Then $ ℂQ_1 R ∩ R ℂQ_1 $ is equal to the $ ℂQ_0 $-bimodule generated by $ W $. The property [BG] holds.

Moreover, let $ W' ∈ \mathfrak{m} \compl{ℂQ} $ be a (cyclic) deformation. Define $ ψ: R → \mathfrak{m} \compl{ℂQ} $ by $ ψ(∂_a W) = ∂_a W' $. Then [CP] holds.
\end{lemma}

\begin{proof}
We prove all four statements one after another. For sake of simplicity, we use the notation $ W $ to mean both the superpotential and the $ ℂQ_0 $-bimodule $ ℂQ_0 W ℂQ_0 $ generated by the superpotential. In an expression like $ W ⊂ ℂQ_1 R $ or $ w ∈ W $, the letter $ W $ is to be interpreted as the $ ℂQ_0 $-bimodule $ ℂQ_0 W ℂQ_0 $. This way, no confusion should arise.

The first statement on linear independence of the relations $ ∂_a W $ is now widely known, due to Ginzburg, Bocklandt, Berger, Taillefer and others. For example, the CY3 property implies by \cite[Section 4.2]{Bocklandt-CY} that the sequence \eqref{eq:flatness-resolution} is a self-dual resolution of $ \Jac(Q, W) $. Hence $ R $ and $ ℂQ_1 $ are equal in dimension and the relations are linearly independent.

For the second part of the proof, we show that $ ℂQ_1 R ∩ R ℂQ_1 $ is the $ ℂQ_0 $-bimodule generated by $ W $. The inclusion $ W ⊂ ℂQ_1 R ∩ R ℂQ_1 $ is easy and follows from cyclicity: Indeed for any vertex $ v ∈ Q_0 $ we have
\begin{equation*}
vWv = \sum_{h(a) = v} a ∂_a W = \sum_{t(a) = v} ∂_a W a.
\end{equation*}
The first sum lies in $ ℂQ_1 R $ and the second sum lies in $ R ℂQ_1 $, hence $ vWv $ lies in their intersection.

The converse inclusion $ ℂQ_1 R ∩ R ℂQ_1 ⊂ W $ follows from inspection of the exact sequence \eqref{eq:flatness-sequence}. To see this, pick an element $ \sum r_i a_i $ with $ r_i ∈ R $ and $ a_i ∈ Q_1 $ that simultaneously lies in $ ℂQ_1 R $. We claim that $ \sum r_i ¤ a_i $ goes to zero under $ f_2: R ¤ \Jac(Q, W) → ℂQ_1 ¤ \Jac(Q, W) $. Indeed, the assumption that $ \sum r_i a_i ∈ ℂQ_1 R $ implies
\begin{equation*}
f_2 \left(\sum r_i ¤ a_i\right) = \sum r_i a_i = 0 ∈ \frac{ℂQ_1 ℂQ}{ℂQ_1 I(R)} \cong ℂQ_1 ¤ \Jac(Q, W).
\end{equation*}
By exactness of \eqref{eq:flatness-sequence}, we deduce that $ \sum r_i ¤ a_i $ lies in the image of $ f_1 $. Therefore within $ R ¤ \Jac(Q, W) $ we can write $ \sum r_i ¤ a_i = f_1 (\sum w_i ¤ p_i) $, where $ w_i ∈ W $ and $ p_i $ are paths. We shall now try to lift this equality to $ R ¤ ℂQ $. As a first step, recall that
\begin{equation*}
R ¤ \Jac(Q, W) = \frac{R ¤ ℂQ}{R ¤ I(R)}.
\end{equation*}
In consequence there exists a $ z ∈ R ¤ I(R) $ such that within $ R ¤ ℂQ $ we have
\begin{equation*}
\sum r_i ¤ a_i = f\left(\sum w_i ¤ p_i\right) + z.
\end{equation*}
Here $ f: W ¤ ℂQ → R ¤ ℂQ_1 ℂQ $ denotes the map that splits $ W $ into relations and arrows. The terms $ w_i ¤ p_i $ and $ z ∈ R ¤ I(R) $ look wild, but in reality we can vastly reduce the complexity: Collect in an index set $ S $ all indices $ i $ where the path $ p_i $ is not a vertex. For such indices $ i ∈ S $, we have that all terms in $ f(w_i ¤ p_i) $ have right tensor component of length at least 2. Note that all terms in $ R ¤ I(R) $ also have right tensor length at least 2. Nevertheless, the left-hand side $ \sum r_i ¤ a_i $ has terms only with right tensor length 1, since $ a_i $ are arrows. We deduce that $ f\left(\sum_{i \notin S} w_i ¤ p_i\right) + z = 0 $. Finally, we conclude
\begin{equation*}
\sum r_i ¤ a_i = f\left(\sum_{i ∈ S} w_i ¤ 1\right).
\end{equation*}
Contracting the tensors on both sides gives $ \sum r_i a_i ∈ W $ as desired. This proves the second statement.

For the third statement, it is our task to prove that the strong Berger-Ginzburg inclusion holds. We note that
\begin{equation*}
W ¤ \frac{ℂQ}{I(R)} \cong \frac{W ¤ ℂQ}{W ¤ I(R)}, \quad R ¤ \frac{ℂQ}{I(R)} \cong \frac{R ¤ ℂQ}{R ¤ I(R)} \quad ℂQ_1 ¤ \frac{ℂQ}{I(R)} \cong \frac{ℂQ_1 ℂQ}{ℂQ_1 I(R)}.
\end{equation*}
Inserting this into the exact sequence \eqref{eq:flatness-sequence}, we get the exact sequence
\begin{equation*}
\frac{W ¤ ℂQ}{W ¤ I(R)} \overset{[f]}{→} \frac{R ¤ ℂQ}{R ¤ I(R)} \overset{[\contraction]}{→} \frac{ℂQ_1 ℂQ}{ℂQ_1 I(R)}
\end{equation*}
Evaluating $ \Ker ⊂ \Im $ on this sequence yields the strong Berger-Ginzburg inclusion.

The fourth statement follows again from cyclicity. Any element in $ ℂQ_1 P + P ℂQ_1 $ can be written as $ x + ψ^0 (x) + y + ψ^1 (y) $, where $ x ∈ R ℂQ_1 $ and $ y ∈ ℂQ_1 R $. If such an element additionally lies in $ \mathfrak{m} ℂQ $, we get $ x + y = 0 $, since $ ψ^0 $ and $ ψ^1 $ only give terms in $ \mathfrak{m} ℂQ $. Hence $ x = -y ∈ ℂQ_1 R ∩ R ℂQ_1 = W $. Without loss of generality, we can assume $ x $ and $ y $ are in the form
\begin{equation*}
x = \sum_{t(a) = v} ∂_a W a = \sum_{h(a) = v} a ∂_a W = -y.
\end{equation*}
We deduce
\begin{equation*}
ψ^0 (x) = \sum_{t(a) = v} ∂_a W' a = \sum_{h(a) = v} a ∂_a W' = - ψ^1 (y).
\end{equation*}
In total, $ x + ψ^0 (x) + y + ψ^1 (y) $ vanishes. This demonstrates the fourth statement and finishes the proof.
\end{proof}

We have seen that the case of a CY3 superpotential is covered by \autoref{conv:flatness-relations}. From here on, we proceed in the context of \autoref{conv:flatness-relations}. We finish the present section with remarks on the difference between our strong Berger-Ginzburg inclusion and the original Berger-Ginzburg inclusion, as well as an explanation on how Berger and Ginzburg proceed in case of a homogeneous superpotential.

\begin{remark}
Let us compare the standard Berger-Ginzburg inclusion \eqref{eq:flatness-BGstandard} and strong Berger-Ginzburg inclusion \eqref{eq:flatness-BGstrong}. We claim that the strong inclusion implies the standard inclusion and the converse holds if $ W ∈ ℂQ $ is homogeneous. To show the strong-to-weak implication, lift any element $ x ∈ R ℂQ ∩ ℂQ_1 R $ to the tensor product $ R ¤ ℂQ $, apply the strong inclusion and contract again. To prove the weak-to-strong implication, assume $ W $ is homogeneous. Let $ x ∈ R ¤ ℂQ ∩ \contraction^{-1} (ℂQ_1 I(R)) $, then
\begin{equation*}
\contraction(x) ∈ R ℂQ ∩ ℂQ_1 I(R) ⊂ W ℂQ + R I(R) = \contraction(f(W ¤ ℂQ) + R ¤ I(R)).
\end{equation*}
Homogeneity makes the contraction map $ \contraction $ injective and hence $ x ∈ f(W ¤ ℂQ) + R ¤ I(R) $. This proves the weak-to-strong implication.
\end{remark}

The work of Berger and Ginzburg \cite{Berger-Ginzburg} takes place in the context of PBW deformations. We have translated their result to the case of formal deformations in \autoref{th:flatness-BGstandard-BGflatness}. Now that we have introduced better terminology, we are ready to explain how Berger and Ginzburg prove their result. Translated again to formal deformations, their core lemma is the following:

\begin{lemma}[Berger-Ginzburg]
\label{th:flatness-berger-ginzburg}
Assume the standard Berger-Ginzburg inclusion and [CP]. Then
\begin{equation}
\label{eq:flatness-BGstrong-toyflat}
I(P)_{ℂQ} ∩ \mathfrak{m} ℂQ ⊂ ℂQ_1 (I(P)_{ℂQ} ∩ \mathfrak{m} ℂQ) + \mathfrak{m} I(P)_{ℂQ}.
\end{equation}
\end{lemma}

After proving this lemma, Berger and Ginzburg assume that $ Q $ has a grading in which every arrow is positive and $ W $ is homogeneous. They continue as follows: Pick an element $ x $ on the left-hand side of homogeneous degree in zeroth order. By the lemma we can write $ x = y + z $, where $ y ∈ ℂQ_1 (I(P)_{ℂQ} ∩ \mathfrak{m} ℂQ) $ and $ z $ already lies in $ \mathfrak{m} I(P)_{ℂQ} $. In $ y $, we can split off arrows on the left. This results in reducing the degree and we can assume by induction that the part with an arrow less already lies in $ \mathfrak{m} I(P)_{ℂQ} $, hence $ y ∈ ℂQ_1 \mathfrak{m} I(P)_{ℂQ} ⊂ \mathfrak{m} I(P)_{ℂQ} $. It follows that $ I(P)_{ℂQ} ∩ \mathfrak{m} ℂQ ⊂ \mathfrak{m} I(P)_{ℂQ} $. This proves flatness in the setting of Berger and Ginzburg where $ W $ is homogeneous.

%% file: flatness/bounded_type.tex
\subsection{Relations of bounded type}
\label{sec:flatness-bounded}
In this section, we show how to circumvent the homogeneity requirement of Berger and Ginzburg. Our substitution for the homogeneity requirement is a simple yet powerful boundedness condition. In the present section, we first introduce the boundedness argument in high generality. Then we demonstrate its strength in a sequence of applications. In \autoref{sec:flatness-BGbounded}, we tailor it specifically to the case of deformed relations and derive a bounded version of the strong Berger-Ginzburg inclusion. The boundedness condition allows us in \autoref{sec:flatness-flatcompleted} to continue the proof of flatness without homogeneity assumption.

In the most general way, our boundedness argument is stated as follows:

\begin{lemma}
\label{th:flatness-basis-related}
Let $ V $ be a vector space with basis $ V_0 $ and an equivalence relation $ \sim $ on $ V_0 $. Let $ X ⊂ V $ be a subspace with spanning set $ X_0 $ such that the $ V_0 $-constituents of every $ x_0 ∈ X_0 $ are $ \sim $-related. Then for every set $ C ⊂ V_0 $ closed under $ \sim $ we have
\begin{equation*}
X ∩ \vspan(C) ⊂ \vspan(X_0 ∩ \vspan(C)).
\end{equation*}
\end{lemma}

\begin{proof}
The strategy is to write an element on the left-hand side in terms of the spanning set $ X_0 $ and then realize that the terms not related to $ C $ vanish collectively. Indeed, let
\begin{equation*}
x = \sum_{i ∈ I} λ_i x_i ∈ X ∩ \vspan(C), \quad x_i ∈ X_0, ~ |I| < ∞.
\end{equation*}
Now let $ S ⊂ I $ be the set of indices $ i ∈ I $ where the constituents of $ x_i $ lie in $ C $. Regard the decomposition
\begin{equation*}
x = \sum_{i ∈ S} λ_i x_i + \sum_{i \notin S} λ_i x_i.
\end{equation*}
The $ V_0 $-constituents of the first summand all lie in $ C $, while the constituents of the second summand all lie outside of $ C $. Since the left-hand side $ x $ lies in $ \vspan(C) $, the second summand necessarily vanishes. Since $ x_i ∈ X_0 ∩ \vspan(C) $ for $ i ∈ S $, this finishes the proof.
\end{proof}

We have a specific application to quiver algebras with relations $ R $ in mind: Declare paths $ pcq $ and $ pdq $ related if $ c $ and $ d $ appear in a relation of $ R $ together. Let us make this precise.

\begin{definition}
Let $ Q $ be a quiver and $ R ⊂ ℂQ $ a finite-dimensional $ ℂQ_0 $-bimodule. Assume a basis $ F $ for $ R $ is given. For each $ c ∈ F $, decompose $ c = \sum λ_i c_i $ as linear combination of paths ($ λ_i ≠ 0 $), then set $ p c_i q \sim p c_j q $ for any paths $ p, q $ and indices $ i, j $. Denote by $ \sim $ the transitive hull of this relation. Two paths $ p $ and $ q $ are \emph{$ F $-related} if $ p \sim q $. For $ N ∈ ℕ $ denote the supremum of all lengths of paths related to paths of lengths $ ≤ N $ by
\begin{equation*}
h(N) = \sup\{|q| \running ∃p: |p| ≤ N, ~ p \sim q\} ∈ ℕ ∪ \{∞\}.
\end{equation*}
The basis $ F $ is of \emph{bounded type} if $ h(N) < ∞ $ for all $ N ∈ ℕ $.
\end{definition}

\begin{remark}
The basis $ F $ is always assumed to be a basis for $ R $ as $ ℂQ_0 $-bimodule. Simply speaking, for every $ c ∈ F $ there shall be vertices $ v, w ∈ Q_0 $ such that every path in $ c $ runs from $ v $ to $ w $.
\end{remark}

During the present section, we provide applications of the bounded type condition. Typically, we will work with the basis $ F $ explicitly. In later sections it is only relevant that there exists a basis of bounded type. We therefore set up the following terminology:

\begin{definition}
\label{def:flatness-bounded-RWbounded}
Let $ Q $ be a quiver and $ R ⊂ ℂQ $ a finite-dimensional $ ℂQ_0 $-bimodule of relations. Then $ R $ is of \emph{bounded type} if it has a basis $ F $ of bounded type. A superpotential $ W $ is of \emph{bounded type} if its relation space $ R = \vspan\{∂_a W\} $ is of bounded type.
\end{definition}

\begin{example}
If $ R ⊂ ℂQ $ is graded space with respect to path length or any other grading that is positive on the arrows, then any homogeneous basis for $ R $ is of bounded type: Degree is then an invariant under $ F $-relatedness and length becomes bounded by a multiple of the degree. Here are three easy instances for the algebra $ ℂQ = ℂ⟨A, B, C⟩ $:

\begin{itemize}
\item Regard $ R = \vspan(AB-BA, BC-CB, CA-AC) $ with basis $ F = \{AB-BA, BC-CB, CA-AC\} $. We claim that $ F $ is of bounded type. Indeed, two paths $ p, q $ in $ Q $ are $ F $-related if and only if they differ by reordering of $ A $, $ B $ and $ C $. For instance, $ ABAC \sim AACB $. In particular, any two $ F $-related paths are of equal length and $ h(N) = N $. This shows that $ F $ is of bounded type.
\item Regard $ R = \vspan(AB-C^4, BA-C^2 B) $ and $ F = \{AB-C^4, BA-C^2 B\} $. This gives for instance $ B C^4 \sim BAB \sim C^2 B^2 $. To see that $ F $ is of bounded type, give $ A, B $ degree $ 2 $ and $ C $ degree $ 1 $. Then $ F $-related paths are equal in degree. We have $ h(0) = 0 $ and $ h(1) = 1 $ and $ h(N) = 2N $ for $ N ≥ 2 $.
\item Regard $ R = \vspan(AB, ABC) $. The basis $ F = \{AB, ABC\} $ is of bounded type and $ h(N) = N $. The basis $ F = \{AB + ABC, AB - ABC\} $ is however not of bounded type, because $ AB \sim ABC \sim ABCC $ etc.
\end{itemize}
\end{example}

\begin{remark}
Given a basis $ F ⊂ R $ and an integer $ N ∈ ℕ $, denote by $ l(N) $ the minimal length of paths $ F $-related to paths of length $ ≥ N $:
\begin{equation*}
l(N) = \min \{|q| \running ∃p: |p| ≥ N, ~ p \sim q\}.
\end{equation*}

Then $ h(l(N)) ≥ N $. Indeed, pick a path of length $ l(N) $. Then it is $ F $-related to a path of length $ ≥ N $ and hence the supremum $ h(l(N)) $ is at least $ N $.

Assume $ F $ is of bounded type. Then $ h $ is finite on every integer. The inequality $ h(l(N)) ≥ N $ ensures that $ l(N) $ converges to infinity as $ N → ∞ $. Simply speaking, if $ l(N) $ stays small, this means there are short paths equivalent to longer and longer paths, precisely the opposite of the assumption. In total, we conclude that the interval $ [l(M), h(N)] $ goes to infinity as $ [M, N] $ goes to infinity. Moreover, we can say
\begin{equation*}
p \sim q, ~ |q| ∈ [M, N] \quad \Longrightarrow \quad |p| ∈ [l(M), h(N)].
\end{equation*}
\end{remark}

To demonstrate the strength of the boundedness assumption, we apply \autoref{th:flatness-basis-related} to relations of bounded type. For instance, the next lemma brings elements $ x ∈ X $ known to satisfy length bounds back to a bounded part of the spanning set. We write $ \fourIdx{}{M≤}{}{≤N}{ℂQ} $ for the subspace of $ ℂQ $ spanned by paths of length between $ M $ and $ N $.

\begin{lemma}
\label{th:flatness-space-boundedness}
Let $ Q $ be a quiver and $ R ⊂ ℂQ $ a $ ℂQ_0 $-bimodule with basis $ F $. Let $ X ⊂ ℂQ $ be a subspace with spanning set $ X_0 $ such that for every $ x_0 ∈ X_0 $ all paths appearing in $ x_0 $ are $ F $-related. Then
\begin{equation}
\label{eq:flatness-space-boundedness}
∀M, N ∈ ℕ: \quad X ∩ \fourIdx{}{M≤}{}{≤N}{ℂQ} ⊂ \vspan(X_0 ∩ \fourIdx{}{l(M)≤}{}{≤h(N)}{ℂQ}).
\end{equation}
\end{lemma}

\begin{proof}
We apply \autoref{th:flatness-basis-related}. As ambient space use $ V = ℂQ $, as spanning set $ V_0 $ use the set of paths in $ Q $, as relation $ \sim $ use $ F $-relatedness, and as restriction $ C $ use the set of paths in $ Q $ that are $ F $-related to paths of length between $ M $ and $ N $. \autoref{th:flatness-basis-related} now yields $ X ∩ \vspan(C) ⊂ \vspan(X_0 ∩ \vspan(C)) $.

The desired inclusion \eqref{eq:flatness-space-boundedness} is slightly weaker than this. Let us check. The left side of \eqref{eq:flatness-space-boundedness} is contained in $ X ∩ \vspan(C) $, since $ C $ includes in particular all paths of length between $ M $ and $ N $. Finally $ \vspan(X_0 ∩ \vspan(C)) $ is contained in the right side of \eqref{eq:flatness-space-boundedness}, since paths in $ C $ always have length at least $ l(M) $ and at most $ h(N) $.
\end{proof}

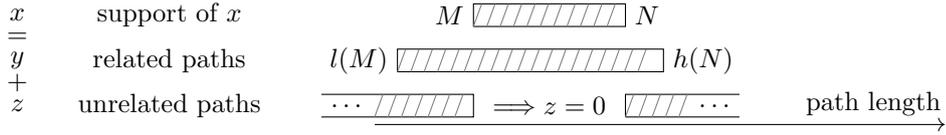
\begin{figure}
\centering
\begin{tikzpicture}
\path[draw] (0, 0) -- ++(up:0.3) node[midway, left] {$ M $} -- ++(right:2) -- ++(down:0.3) node[midway, right] {$ N $} coordinate (top-right) -- ++(left:2) coordinate (top-left);
\path[draw] (-1, -0.6) -- ++(up:0.3) node[midway, left] {$ l(M) $} -- ++(right:3.5) -- ++(down:0.3) node[midway, right] {$ h(N) $} coordinate (mid-right) -- ++(left:3.5) coordinate (mid-left);
\path[draw] (-2, -1.2) ++(up:0.3) coordinate (bot1-dots-top) -- ++(right:2) -- ++(down:0.3) coordinate[midway] (bot1-mid) coordinate (bot1-right) -- ++(left:2) coordinate[pos=0.65] (bot1-left);
\path[draw] (2, -1.2) -- ++(up:0.3) coordinate[midway] (bot2-mid) -- ++(right:1.5) coordinate (bot2-dots-top) ++(down:0.3) coordinate (bot2-dots-bot) [draw] -- ++(left:1.5) coordinate[pos=0.3] (bot2-right) coordinate (bot2-left);
\path ($ (-2, -1.2)!0.5!(bot1-dots-top) $) node[right] {…};
\path ($ (bot2-dots-top)!0.5!(bot2-dots-bot) $) node[left] {…};
\foreach \piece in {top, mid, bot1, bot2} \path[draw, decoration={border, amplitude=9pt, segment length=5pt, angle=70}, decorate, gray] (\piece-left) to (\piece-right);
\path[draw, ->] (-1.3, -1.3) to node[very near end, above] {path length} ++(right:7.5);
\path (-4, 0.15) node {support of $ x $};
\path (-4, -0.45) node {related paths};
\path (-4, -1.05) node {unrelated paths};
\path (-6, 0.15) node {$ x $};
\path (-6, -0.15) node {$ = $};
\path (-6, -0.45) node {$ y $};
\path (-6, -0.75) node {$ + $};
\path (-6, -1.05) node {$ z $};
\path ($ (bot1-mid)!0.5!(bot2-mid) $) node {$ \Longrightarrow z = 0 $};
\end{tikzpicture}
\caption{Decomposition into related and unrelated paths}
\label{fig:flatness-decomposition-length}
\end{figure}


We provide a further example of path length analysis, a toy version of our later flatness results.

\begin{lemma}
\label{th:flatness-toy-first-order}
Let $ Q $ be a quiver and $ R ⊂ ℂQ $ a $ ℂQ_0 $-bimodule with basis $ F $. Let $ (B, \mathfrak{m}) $ be a deformation base. Let $ X ⊂ B \htensor ℂQ $ be a subspace with spanning set $ X_0 $ such that for every $ x_0 ∈ X_0 $ all zeroth-order paths in $ x_0 $ are $ F $-related. Then
\begin{equation*}
X ∩ (ℂQ_{≤N} + \mathfrak{m}ℂQ) ⊂ \vspan(X_0 ∩ (ℂQ_{≤h(N)} + \mathfrak{m}ℂQ)) + X ∩ \mathfrak{m}ℂQ.
\end{equation*}
\end{lemma}

\begin{proof}
It is possible to build on \autoref{th:flatness-space-boundedness}, but we deploy \autoref{th:flatness-basis-related} instead. Denote by $ π: X → ℂQ $ the projection to zeroth order. Define $ Y_0 ≔ π(X_0) $. It is a spanning set for $ π(X) $. Now let $ x ∈ X ∩ (ℂQ_{≤N} + \mathfrak{m}ℂQ) $. Then
\begin{equation*}
π(x) ∈ π(X) ∩ ℂQ_{≤N} ⊂ \vspan(Y_0 ∩ ℂQ_{≤h(N)}) ⊂ \vspan(π(X_0 ∩ (ℂQ_{≤h(N)} + \mathfrak{m}ℂQ))).
\end{equation*}
Hence we can choose $ x' ∈ \vspan(X_0 ∩ (ℂQ_{≤h(N)} + \mathfrak{m}ℂQ)) $ such that $ π(x) = π(x') $. In particular, we have $ x - x' ∈ X ∩ \mathfrak{m} ℂQ $. We finish the proof with the observation that $ x = x' + (x - x') $.
\end{proof}

\autoref{th:flatness-toy-first-order} is best summarized as follows: When forming a large sum $ x $ over $ X_0 $ elements, once their zeroth-order length surpasses $ h(N) $ they cannot contribute anymore to $ x $, up to higher order terms. In other words, summands of zeroth-order length $ > h(N) $ contribute only to higher order terms.

\begin{remark}
The original idea behind our boundedness condition is best explained as follows: We are given an element $ x ∈ ℂQ_1 (I(P)_{ℂQ} ∩ \mathfrak{m} ℂQ) $ and are supposed to improve on this. Berger and Ginzburg proceed by splitting off arrows on the left and thereby reducing degree. Our idea is to iterate the inclusion instead which increases degree. The resulting elements $ x_i $ lie in higher an higher path length. If we can show that the terms in very high path length are not related at all with the low length terms we started with, then the $ x_i $ must vanish. This insight led to the boundedness condition presented in this section.
\end{remark}

%% file: flatness/ideals.tex
\subsection{Ideals tailored to boundedness}
\label{sec:flatness-ideals}
In this section, we collect and introduce notation for many ideal-like sets. For instance, we have already used the notation $ I(R) $ earlier to denote ideals generated by the relation space $ R $. In the present section, we define new ideal-like sets that are tailored to our boundedness argument of \autoref{sec:flatness-bounded}.

The setting of this section is \autoref{conv:flatness-relations}. The ideal-like sets we define here are depicted in \autoref{tab:flatness-ideals-overview}. In \autoref{def:flatness-bounded-ideals}, we start with simple subspaces of $ ℂQ $ where the lengths are bounded. In \autoref{def:flatness-ideals-embracing}, we define ideal-like spaces in which length bounds are imposed on the paths embracing a relation:

\begin{definition}
\label{def:flatness-bounded-ideals}
We define the following four subspaces of $ ℂQ $ and $ ℂQ ¤_{ℂQ_0} ℂQ $:
\begin{itemize}
\item The spaces
\begin{equation*}
\fourIdx{}{M≤}{}{≤N}{ℂQ}, \quad ℂQ_{≥N}, \quad ℂQ_{≤N}
\end{equation*}
are the subspaces of $ ℂQ $ spanned by paths of length in $ [M, N] $, $ [N, ∞) $ or $ [0, N] $, respectively.
\item The spaces 
\begin{equation*}
\fourIdx{}{M≤}{}{≤N}{(ℂQ ¤ ℂQ)}, \quad (ℂQ ¤ ℂQ)_{≥N}, \quad (ℂQ ¤ ℂQ)_{≤N}
\end{equation*}
are the subspaces of $ ℂQ ¤_{ℂQ_0} ℂQ $ spanned by pure tensors $ p ¤ p' $ where $ p $ and $ p' $ are paths with length bound $ |pp'| ∈ [M, N] $, $ |pp'| ≥ N $ or $ |pp'| ≤ N $, respectively.
\end{itemize}
\end{definition}

\begin{definition}
\label{def:flatness-ideals-embracing}
We define the following six subspaces of $ ℂQ $ and $ B \htensor \compl{ℂQ} $:
\begin{itemize}
\item The spaces
\begin{equation*}
\fourIdx{}{M≤}{}{≤N}{(ℂQ R ℂQ)}, \quad (ℂQ R ℂQ)_{≥M}, \quad (ℂQ R ℂQ)_{≤N}
\end{equation*}
are the subspaces of $ ℂQ $ consisting of elements that can be written in the form
\begin{equation*}
x = \sum_{\text{finite}} p_i r_i q_i,
\end{equation*}
where $ p_i, q_i $ are paths in $ Q $ and $ r_i ∈ R $ are relations such that for all $ i $ every path contained in $ p_i r_i q_i ∈ ℂQ $ has length in $ [M, N] $ or $ [M, ∞) $ or $ [0, N] $, respectively.
\item The spaces
\begin{equation*}
\fourIdx{}{M≤}{}{≤N}{(ℂQ P ℂQ)}, \quad (ℂQ P ℂQ)_{≥M}, \quad (ℂQ P ℂQ)_{≤N}
\end{equation*}
are the subspaces of $ B \htensor \compl{ℂQ} $ consisting of elements that can be written in the form
\begin{equation*}
x = \sum_{\text{finite}} p_i (r_i + ψ(r_i)) q_i,
\end{equation*}
where $ p_i, q_i $ are paths in $ Q $ and $ r_i ∈ R $ are relations such that for all $ i $ every path contained in $ p_i r_i q_i ∈ ℂQ $ has length in $ [M, N] $ or $ [M, ∞) $ or $ [0, N] $, respectively.
\end{itemize}
\end{definition}

\begin{table}
\centering
\morearraystretch
\begin{tabular}{@{}ccccc@{}}
\textbf{Notation} & & \textbf{Definition} & \textbf{Typical element} & \textbf{Requirement} \\\hline
$ I(R) $ & $ = $ & $ ℂQ R ℂQ $ & $ \sum\limits_{i = 0}^K p_i r_i q_i $ & finite sum \\
$ I(R)_{\compl{ℂQ}} $ & $ = $ & $ \substack{\Im(\compl{ℂQ} \htensor R \htensor \compl{ℂQ} \\ → \compl{ℂQ})} $ & $ \sum\limits_{i = 0}^∞ p_i r_i q_i $ & $ |p_i| + |q_i| → ∞ $ \\
$ \compl{ℂQ} R \compl{ℂQ} $ & $ = $ & $ \substack{\Im(\compl{ℂQ} ¤ R ¤ \compl{ℂQ} \\ → \compl{ℂQ})} $ & $ \sum\limits_{i = 0}^K \big(\sum\limits_{j = 0}^∞ λ_j p_j\big) r_i \big(\sum\limits_{j = 0}^∞ η_j q_j\big) $ & $ |p_i|, |q_i| → ∞ $ \\
$ (P) $ & $ = $ & {$ \substack{\Im(\compl{ℂQ} \htensor P \htensor \compl{ℂQ} \\ → B \htensor \compl{ℂQ})} $} & $ \sum\limits_{i = 0}^∞ p_i (r_i + ψ(r_i)) q_i $ & $ |p_i| + |q_i| → ∞ $ \\
$ I(P) $ & $ = $ & $ B ~ (P) $ & $ \sum\limits_{j = 0}^∞ m_j \sum\limits_{i = 0}^∞ p_i (r_i + ψ(r_i)) q_i $ & $ \substack{m_j ∈ \mathfrak{m}^{→∞} \\ |p_i| + |q_i| → ∞} $  \\
$ I(P)_{ℂQ} $ & $ = $ & $ B ℂQ P ℂQ $ & $ \sum\limits_{i = 0}^∞ m_i p_i (r_i + ψ(r_i)) q_i $ & $ m_i ∈ \mathfrak{m}^{→∞} $ \\
$ \fourIdx{}{M≤}{}{≤N}{ℂQ} $ & $ = $ & $ \underset{M ≤ |p| ≤ N}{\vspan} p $ & $ \sum\limits_{i = 0}^K p_i $ & $ M ≤ |p_i| ≤ N $ \\
$ \fourIdx{}{M≤}{}{≤N}{(ℂQ R ℂQ)} $ & $ = $ & $ \underset{M ≤ |prq| ≤ N}{\vspan} prq $ & $ \sum\limits_{i = 0}^K p_i r_i q_i $ & $ M ≤ |p_i r_i q_i| ≤ N $ \\
$ \fourIdx{}{M≤}{}{≤N}{(ℂQ P ℂQ)} $ & $ = $ & $ \underset{M ≤ |prq| ≤ N}{\vspan} p (r+ψ(r)) q $ & $ \sum\limits_{i = 0}^K p_i (r_i + ψ(r_i)) q_i $ & $ M ≤ |p_i r_i q_i| ≤ N $
\end{tabular}
\caption{Ideal-like sets with their definitions, typical elements and requirements. For sake of legibility, we have omitted double indices in the description of the typical elements. For instance $ p_j $ in double sums should read $ p_{i, j} $.}
\label{tab:flatness-ideals-overview}
\end{table}

We finally define several ideal-like spaces which are finely tuned to the purpose of bounding path lengths:

\begin{definition}
\label{def:flatness-ideals-IP}
We define the spaces $ I(R) $, $ I(R)_{\compl{ℂQ}} $, $ \compl{ℂQ} R \compl{ℂQ} $, $ (P) $, $ I(P) $ and $ I(P)_{ℂQ} $ as follows:
\begin{itemize}
\item The space $ I(R) = ℂQ R ℂQ ⊂ ℂQ $ is the ideal generated by $ R $. In other words, it contains elements of the form
\begin{equation*}
x = \sum\limits_{i = 0}^K p_i r_i q_i.
\end{equation*}
Here $ r_i ∈ R $ and $ p_i, q_i $ are paths in $ Q $.
\item The space $ I(R)_{\compl{ℂQ}} ⊂ \compl{ℂQ} $ is defined as the image of the multiplication map $ \compl{ℂQ} \htensor R \htensor \compl{ℂQ} \allowbreak → B \htensor \compl{ℂQ} $. In other words, its elements are of the form
\begin{equation*}
x = \sum_{i = 0}^∞ p_i r_i q_i.
\end{equation*}
Here $ r_i ∈ R $ and $ p_i, q_i $ are paths with $ |p_i| + |q_i| → ∞ $.
\item The space $ \compl{ℂQ} R \compl{ℂQ} $ is the ideal generated by $ R $ in $ \compl{ℂQ} $. In other words, its elements are of the form
\begin{equation*}
x = \sum_{j = 0}^K \big(\sum_{i = 0}^∞ λ_{i, j} p_{i, j}\big) r_{i, j} \big(\sum_{i = 0}^∞ η_{i, j} q_{i, j}\big).
\end{equation*}
Here $ r_i ∈ R $ and $ p_{i, j}, q_{i, j} $ are paths in $ Q $ such that for every $ j ∈ ℕ $ the lengths $ |p_{i, j}| $ and $ |q_{i, j}| $ converge to $ ∞ $ as $ i → ∞ $.
\item The space $ (P) ⊂ B \htensor \compl{ℂQ} $ is image of the multiplication map $ \compl{ℂQ} \htensor P \htensor \compl{ℂQ} → B \htensor \compl{ℂQ} $. In other words, it contains elements of the form
\begin{equation*}
x = \sum_{i = 0}^∞ p_i (r_i + ψ(r_i)) q_i.
\end{equation*}
Here $ r_i ∈ R $, and $ p_i, q_i $ are paths in $ Q $ with combined length $ |p_i| + |q_i| → ∞ $.
\item The space $ I(P) ⊂ B \htensor \compl{ℂQ} $ is defined as $ I(P) = B ~(P) $. In other words, it contains elements of the form
\begin{equation*}
x = \sum\limits_{j = 0}^∞ m_j \sum\limits_{i = 0}^∞ p_{i, j} (r_{i, j} + ψ(r_{i, j})) q_{i, j}.
\end{equation*}
Here $ m_j ∈ \mathfrak{m}^{→∞} $ and for every $ j ∈ ℕ $ the combined length $ |p_{i, j}| + |q_{i, j}| $ of $ p_{i, j} $ and $ q_{i, j} $ is supposed to converge to $ ∞ $ as $ i → ∞ $.
\item If $ \Im(ψ) ⊂ \mathfrak{m} ℂQ $, then the space $ I(P)_{ℂQ} ⊂ B \htensor ℂQ $ is defined as $ B (ℂQ P ℂQ) ⊂ B \htensor ℂQ $. In other words, it contains elements of the form
\begin{equation*}
x = \sum\limits_{i = 0}^∞ m_i p_i (r_i + ψ(r_i)) q_i.
\end{equation*}
Here $ m_i ∈ \mathfrak{m}^{→∞} $ and $ r_i ∈ R $. The elements $ p_i, q_i $ are paths in $ Q $.
\end{itemize}
\end{definition}

\begin{remark}
In \autoref{def:flatness-ideals-IP}, we have done our best to give both an abstract and a practical definition of all ideal-like sets. The abstract definitions come with the following two peculiarities: First, the notation $ B(P) $ in the definition of $ I(P) $ makes use of the shorthand notation of \autoref{def:prelim-submodules-shorthand}. Second, the given abstract definition of $ (P) $ makes use of the completed tensor product $ \compl{ℂQ} \htensor P \htensor \compl{ℂQ} $. The completion here is taken with respect to the Krull topology, which we actually only explain in \autoref{sec:flatness-closedness}. Just like $ B \htensor X $ consists of formal power series in elements of $ X $, the space $ \compl{ℂQ} \htensor P \htensor \compl{ℂQ} $ consists of formal two-sided power series of paths embracing elements of $ P $. For the present context, it suffices to accept the explicit description of the elements of $ (P) $ in terms of series.
\end{remark}

\begin{remark}
The individual definitions are tedious to memorize, but the hidden structure becomes apparent once we compare the definitions:

\begin{center}
\begin{tabular}{lll}
& $ R $ & $ P $ \\\hline
\textbf{formed in} $ ℂQ $ & $ I(R) $ & $ I(P)_{ℂQ} $ \\
\textbf{formed in} $ \compl{ℂQ} $ & $ I(R)_{\compl{ℂQ}} $ & $ I(P) $
\end{tabular}
\end{center}

The default objects are $ I(R) $ and $ I(P) $. The ideals $ I(R)_{\compl{ℂQ}} $ and $ I(P)_{ℂQ} $ only appear in \autoref{sec:flatness-flatalgebraic} and \ref{sec:flatness-closedness}.
\end{remark}

\begin{remark}
We list here a few warnings:
\begin{itemize}
\item $ I(R)_{\compl{ℂQ}} $ is an ideal in $ \compl{ℂQ} $, but not the ideal generated by $ R $ and not its closure either.
\item $ I(P) $ is an ideal in $ B \htensor \compl{ℂQ} $, but not the ideal generated by $ P $ and not its closure either.
\item $ \fourIdx{}{M≤}{}{≤N}{(ℂQ P ℂQ)} $ is not the same as $ ℂQ P ℂQ ∩ \fourIdx{}{M≤}{}{≤N}{ℂQ} $.
\end{itemize}
For instance, closedness of $ I(P) $ in the $ \mathfrak{m} $-adic topology would entail roughly the following: Whenever a sequence $ (x_n) ⊂ I(P) $ becomes concentrated in high powers of $ \mathfrak{m}^k $, the differences $ x_n - x_{n+1} $ are not only concentrated in $ \mathfrak{m}^k $ as a whole, but can be written as a sum over $ m_i p_i (r_i + ψ(r_i)) q_i $ where every single coefficient $ m_i $ lies in $ \mathfrak{m}^k $. This observation visualizes that $ I(P) $ is not necessarily closed. Once we prove $ I(P) $ quasi-flat however, it is also closed in the $ \mathfrak{m} $-adic topology according to \autoref{th:prelim-flatvariants-equivalence}.
\end{remark}

We have depicted abbreviated descriptions of the various sets in \autoref{tab:flatness-ideals-overview}.

%% file: flatness/BGbounded.tex
\subsection{Bounded strong Berger-Ginzburg inclusion}
\label{sec:flatness-BGbounded}
In this section, we prove a bounded version of the strong Berger-Ginzburg inclusion. The necessity of this bounded version is illustrated by the large amount of path length estimates we need to deploy in \autoref{sec:flatness-flatcompleted}. Indeed, the lack of homogeneity for $ W $ makes it necessary to estimate lengths of virtually every vector involved in the flatness argument. One of the sources of vectors in the flatness argument is the strong Berger-Ginzburg inclusion. As such, we need a version of the strong Berger-Ginzburg inclusion that is suited for the bounded world. The present section is meant to provide this bounded version. It is surprising that the bounded version follows directly from the ordinary strong Berger-Ginzburg inclusion. It is a kind of a posteriori estimate and works without additional assumptions on $ Q $ or $ W $:

\begin{center}
\begin{tikzpicture}
\path (0, 0) node[align=center] (A) {\textbf{Strong Berger-Ginzburg} \\ \scriptsize $ R ¤ ℂQ ∩ \contraction^{-1} (ℂQ_1 I(R)) $ \\ \scriptsize $ ⊂ f(W ¤ ℂQ) + R ¤ I(R) $} (9, 0) node[align=center] (B) {\textbf{Bounded strong Berger-Ginzburg} \\ %
\scriptsize $ R ¤ ℂQ ~∩~ \contraction^{-1} (ℂQ_1 I(R)) ~∩~ \fourIdx{}{M≤}{}{≤N}{(ℂQ ¤ ℂQ)} $ \\
\scriptsize $ ⊂ f(W ¤~ \fourIdx{}{l(M)-|W|≤}{}{≤h(N)}{ℂQ}) ~+~ R ¤~ \fourIdx{}{l(M) - |R|≤}{}{≤h(N)}{(ℂQ R ℂQ)} $};
\path[draw, -implies, double equal sign distance] ($ (A.east)!0.1!(B.west) $) -- ($ (A.east)!1.1!(B.west) $) node[midway, above] {a posteriori};
\end{tikzpicture}
\end{center}

Since the strong Berger-Ginzburg inclusion works with tensors instead of paths, our first step in this section is to introduce a notion of $ F $-relatedness for tensors:

\begin{definition}
Let $ F $ be a basis for $ R $. Then two pure tensors of paths $ p ¤ p' $ and $ q ¤ q' $ are $ F $-\emph{related} if $ pp' $ and $ qq' $ are $ F $-related.
\end{definition}

We shall now construct a spanning set $ W_0 $ for $ W $ with the property that for every $ w ∈ W_0 $ all pure tensors of paths appearing in $ f(w) $ are $ F $-related. Recall from \autoref{conv:flatness-relations} that $ f: W ¤ ℂQ → R ¤ ℂQ_1 ℂQ $ is the map which splits $ W $ into relations and arrows. When $ w ∈ W $, then we have $ f(w) ∈ R ¤ ℂQ_1 $.

\begin{lemma}
\label{th:flatness-W-relatedness}
Let $ F $ be a basis for $ R $. Then $ W = ℂQ_1 R ∩ R ℂQ_1 $ has a spanning set $ W_0 $ such that for every $ w ∈ W_0 $ the constituents of $ f(w) ∈ R ¤ ℂQ_1 $ are $ F $-related.
\end{lemma}

\begin{proof}
The strategy is to decompose an arbitrary $ w ∈ W $ into a sum $ \sum w_p $ such that the constituents of $ f(w_p) $ are $ F $-related. Pick $ w ∈ W $. Since $ W = ℂQ_1 R ∩ R ℂQ_1 $, we can write $ w $ in two ways as finite sums
\begin{equation*}
w = \sum_{i ∈ I} a_i r_i = \sum_{j ∈ J} r_j' b_j.
\end{equation*}
Here $ a_i $ and $ b_j $ are scalar multiples of arrows and $ r_i $ and $ r_j' $ lie in $ F $. Note that in both sums the constituents of every individual summand are $ F $-related. Denote by $ P $ the finite set of paths appearing anywhere in these sums, modulo $ F $-relatedness. Now $ P $ splits both $ I $ and $ J $ into classes. Namely for $ p ∈ P $, let $ I_p ⊂ I $ be the set of indices $ i $ where the constituents of $ a_i r_i $ are $ F $-related to $ p $. Similarly, let $ J_p ⊂ J $ be the set of indices where the constituents of $ r_j' b_j $ are $ F $-related to $ p $. For every $ p ∈ P $, both sums
\begin{equation*}
w_p = \sum_{i ∈ I_p} a_i r_i, \quad w_p' = \sum_{j ∈ J_p} r_j' b_j
\end{equation*}
have path support lying in the equivalence class of $ p $. Since the index sets $ (I_p)_{p ∈ P} $ and $ (J_p)_{p ∈ P} $ both exhaust disjointly $ I $ and $ J $, we conclude $ w_p = w_p' $ for every $ p ∈ P $. Furthermore $ w_p ∈ ℂQ_1 R $ and $ w_p' ∈ R ℂQ_1 $, hence $ w_p = w_p' ∈ W $. By construction, all constituents in
\begin{equation*}
f(w_p) = \sum_{j ∈ J_p} r_j' ¤ b_j.
\end{equation*}
contract to paths $ F $-related to $ p $. In particular, all constituents of $ f(w_p) $ are related to each other. We have decomposed $ w ∈ W $ into summands $ w_p ∈ W $ such that all constituents in $ f(w_p) $ are $ F $-related to each other. Running this algorithm for every $ w ∈ W $ provides a spanning set $ W_0 $ with the desired property. Naturally one can extract a finite one from it.
\end{proof}

\begin{remark}
The statement of \autoref{th:flatness-W-relatedness} is obvious in case $ R $ is the space of relations $ R = \vspan(∂_a W) $ coming from a CY3 superpotential. Indeed, simply use the spanning set
\begin{equation}
\label{eq:flatness-W-relatedness-superpotential}
W_0 = \left\{\sum_{t(a) = v} ∂_a W a ~\bigg|~ v ∈ Q_0\right\}.
\end{equation}
Let us check that the spanning set $ W_0 $ satisfies the claimed condition. For $ w = \sum_{t(a) = v} ∂_a W a $, we simply have $ f(w) = \sum_{t(a) = v} ∂_a W ¤ a $. By cyclicity of $ W $, all constituents of this sum are related. This makes that $ W_0 $ satisfies the requirements of \autoref{th:flatness-W-relatedness}.
\end{remark}

Before we devote ourselves to the bounded version of the strong Berger-Ginzburg inclusion, we need length bounds for the paths in $ R $ and $ W $. While $ R $ and $ W $ are not homogeneous, any element $ r ∈ R $ or $ w ∈ W $ can still be decomposed as a linear combination of paths in $ Q $:
\begin{equation*}
r = λ_1 p_1 + … λ_k p_k \quad \text{or} \quad w = ε_1 q_1 + … + ε_l q_l.
\end{equation*}
The paths $ p_1, …, p_k $ or $ q_1, …, q_1 $ are typically not of the same length. However, $ R $ and $ W $ are finite-dimensional by \autoref{conv:flatness-relations}. This implies that the path lengths encountered in $ R $ and $ W $ are bounded. We therefore fix the following notation:

\begin{definition}
\label{def:flatness-bounded-RWlength}
The maximum path length encountered in $ R $ and $ W $ is denoted by $ |R| ∈ ℕ $ and $ |W| ∈ ℕ $.
\end{definition}

We are ready to prove our bounded version of the strong Berger-Ginzburg inclusion:

\begin{lemma}
\label{th:flatness-berger-ginzburg-bounded}
Assume $ R $ is of bounded type. If the strong Berger-Ginzburg inclusion [BG] holds, then it also holds in the bounded form
\begin{multline}
\label{eq:flatness-berger-ginzburg-bounded}
R ¤ ℂQ ~∩~ \contraction^{-1} (ℂQ_1 I(R)) ~∩~ \fourIdx{}{M≤}{}{≤N}{(ℂQ ¤ ℂQ)} \\
⊂ f(W ¤~ \fourIdx{}{l(M)-|W|≤}{}{≤h(N)}{ℂQ}) ~+~ R ¤~ \fourIdx{}{l(M) - |R|≤}{}{≤h(N)}{(ℂQ R ℂQ)}
\end{multline}
Here $ N, M ∈ ℕ $ are arbitrary integers. The inclusion also holds when two-sided bounds are replaced by one-sided bounds from above or below.
\end{lemma}

\begin{proof}
The proof boils down to applying \autoref{th:flatness-space-boundedness} to the right side of the Berger-Ginzburg inclusion. Pick a spanning set $ W_0 ⊂ W $ as in \autoref{th:flatness-W-relatedness}. Put $ X = f(W ¤ ℂQ) + R ¤ I(R) $ and use splits $ f(w ¤ p) $ and products $ r ¤ p r' q $ as spanning set $ X_0 ⊂ X $, where $ w ∈ W_0 $, $ r, r' ∈ F $, and $ p, q $ are paths. Regard the ambient space $ V = ℂQ ¤ ℂQ $ and its basis $ V_0 $ consisting of pure tensors of paths.

By construction, all constituents in a split $ f(w ¤ p) $ are $ F $-related. Namely they consist of $ F $-related tensors, with additional $ p $ on the right side. Also all constituents in a product $ r ¤ p r' q $ are $ F $-related. Indeed, decomposing $ r = \sum λ_i c_i $ and $ r' = \sum λ_i' c_i' $ into scalar multiples of paths, we have
\begin{equation*}
∀i_0, j_0, i_1, j_1: \quad c_{i_0} ¤ p c_{j_0}' q ~\sim~ c_{i_1} ¤ p c_{j_0}' q ~\sim~ c_{i_1} ¤ p c_{j_1}' q.
\end{equation*}
Further choose $ C ⊂ V_0 $ as the set of path tensors $ p ¤ p' $ whose contraction $ pp' $ is related to a path of length at least $ M $ and at most $ N $. This set is closed under $ \sim $. Finally, all assumptions of \autoref{th:flatness-basis-related} are satisfied and we obtain
\begin{equation}
\label{eq:flatness-bg-bounded-application}
\big(f(W ¤ ℂQ) + R ¤ I(R)\big) ∩ \vspan(C) ⊂ \vspan(X_0 ∩ \vspan(C)).
\end{equation}
It remains to interpret both sides of this inclusion. Let us start with the right side. Tensors in $ \vspan(C) $ have length between $ l(M) $ and $ h(N) $. If an element $ f(w ¤ p) $ lies in $ \vspan(C) $, then $ p $ has length $ l(M) - |W| ≤ |p| ≤ h(N) $. If an element $ r ¤ p r' q $ lies in $ \vspan(C) $, then all paths $ c $ in $ p r' q $ have length $ l(M) - |R| ≤ |c| ≤ h(N) $. We conclude that the right side of \eqref{eq:flatness-bg-bounded-application} is contained in the right side of \eqref{eq:flatness-berger-ginzburg-bounded}.

We finish the proof with the remark that the left side of \eqref{eq:flatness-berger-ginzburg-bounded} is contained in the left side of \eqref{eq:flatness-bg-bounded-application}, since the strong Berger-Ginzburg inclusion holds by assumption and tensors of length between $ M $ and $ N $ lie in $ \vspan(C) $.
\end{proof}

%% file: flatness/flat_completed.tex
\subsection{Quasi-flatness in the completed path algebra}
\label{sec:flatness-flatcompleted}
In this section, we prove our first quasi-flatness result. The idea is to use our boundedness condition from \autoref{sec:flatness-bounded} to continue the line of Berger and Ginzburg without homogeneity assumption. The flatness result in this section deals with the ideals $ I(P) ⊂ B \htensor \compl{ℂQ} $ in the completed quiver algebra. In \autoref{sec:flatness-flatalgebraic} we will reduce the result of the present section to the case of the non-completed quiver algebra.

For convenience of the reader, we have sketched in \autoref{th:flatness-berger-ginzburg} the core idea of Berger and Ginzburg. It is however not necessary to be aware of the statement. Rather, we prove from scratch a bounded version of their statement. Before we proceed, recall from \autoref{def:flatness-bounded-RWlength} that $ |R| $ and $ |W| $ denote the maximum path length encountered in $ R $ and $ W $.

\begin{lemma}
\label{th:flatness-flatcompleted-bounds}
Assume $ R $ is of bounded type and [BG] and [CP] hold. Then
\begin{equation}
\label{eq:flatness-flatcompleted-boundsinclusion}
(ℂQ P ℂQ)_{≥N} ∩ \mathfrak{m} \compl{ℂQ} ⊂ ℂQ_1 (ℂQ P ℂQ)_{≥ l(N)-|R|-1} + \mathfrak{m} I(P)_{≥ l(N)-2|R|}
\end{equation}
\end{lemma}

\begin{proof}
By assumption, $ R $ has a basis $ F $ of bounded type. Now pick an element
\begin{equation*}
\tilde x = \sum_{i ∈ I} p_i (r_i + ψ(r_i)) q_i ∈ (ℂQ P ℂQ)_{≥N} ∩ \mathfrak{m} \compl{ℂQ}.
\end{equation*}
By definition of $ (ℂQ P ℂQ)_{≥N} $, we can assume $ p_i $ and $ q_i $ are scalar multiples of paths, $ r_i $ lies in $ F $ and all paths in $ p_i r_i q_i $ have length at least $ N $. Let us inspect the sum. The terms where $ p_i ∈ ℂQ_{≥1} $ already lie on the right-hand side of \eqref{eq:flatness-flatcompleted-boundsinclusion} and do not need further treatment. The terms where $ p_i $ is a vertex are nasty. Denote by $ I_0 ⊂ I $ the set of these nasty indices. Put
\begin{equation*}
x ≔ \sum_{i ∈ I_0} p_i r_i ¤ q_i ∈ R ¤ ℂQ.
\end{equation*}
Since $ \tilde x $ is supposed to vanish on zeroth order, we have
\begin{equation*}
\contraction(x) = \sum_{i ∈ I_0} p_i r_i q_i = - \sum_{i ∈ I \setminus I_0} p_i r_i q_i ∈ ℂQ_1 I(R).
\end{equation*}
At the same time, $ x $ lies in $ (ℂQ ¤ ℂQ)_{≥N} $ and hence on the left-hand side of the bounded strong Berger-Ginzburg inclusion \eqref{eq:flatness-berger-ginzburg-bounded}, using $ N $ for the lower bound and dropping the upper bound. We conclude
\begin{equation*}
x ∈ f(W ¤ ℂQ_{≥ l(N)-|W|}) + R ¤ (ℂQ R ℂQ)_{≥ l(N)-|R|}.
\end{equation*}
Split $ x = y + z $ according to this decomposition. We want to determine $ ψ (x) = ψ (y) + ψ (z) $, where $ ψ $ here acts on the left tensor factor of $ R ¤ ℂQ $. Let us regard $ y $ first. We can write $ y $ as a finite sum
\begin{equation*}
y = \sum_{i ∈ K} f(w_i ¤ p_i')
\end{equation*}
with $ w_i $ lying in $ W $ and $ p_i' $ scalar multiples of paths with $ |p_i'| ≥ l(N) - |W| $. By \autoref{th:flatness-setup-psi}, we have $ ψ^0 = ψ^1 $ on $ W $. Therefore we can swap $ ψ^0 (w_i) $ over:
\begin{equation*}
\contraction(y + ψ^0(y)) = \sum_{i ∈ K} (w_i + ψ^0 (w_i)) p_i' = \sum_{i ∈ K} (w_i + ψ^1 (w_i)) p_i'.
\end{equation*}
Note that $ w_i + ψ^1 (w_i) ∈ ℂQ_1 P $ by nature. We obtain the first intermediate result
\begin{equation*}
\contraction(y + ψ^0 (y)) ∈ ℂQ_1 P ℂQ_{≥ l(N)-|W|}.
\end{equation*}
Now regard $ z $. We can write
\begin{equation*}
z = \sum_{i ∈ L} r_i' ¤ s_i r_i'' t_i,
\end{equation*}
where $ r_i', r_i'' ∈ F $ and $ s_i, t_i $ are scalar multiples of paths, and all paths in $ s_i r_i'' t_i $ have length at least $ l(N) - |R| $.
We get
\begin{align*}
\contraction(z + ψ(z)) &= \sum_{i ∈ L} (r_i' + ψ(r_i')) s_i r_i'' t_i \\
&= \sum_{i ∈ L} r_i' s_i (r_i'' + ψ(r_i'')) t_i + \sum_{i ∈ L} ψ(r_i') s_i (r_i'' + ψ(r_i'')) t_i - \sum_{i ∈ L} (r_i' + ψ(r_i')) s_i ψ(r_i'') t_i \\
&∈ ℂQ_1 (ℂQ P ℂQ)_{≥ l(N)-|R|} + \mathfrak{m} I(P)_{≥ l(N)-2|R|}.
\end{align*}
In the last row, we have used that $ r_i' ∈ R ⊂ ℂQ_{≥1} $. In total, we get
\begin{align*}
\tilde x = \sum_{i ∈ I} p_i (r_i + ψ(r_i)) q_i &= \contraction(y + ψ(y)) + \contraction(z + ψ(z)) + \sum_{i ∈ I \setminus I_0} p_i (r_i + ψ(r_i)) q_i \\
&∈ ℂQ_1 (ℂQ P ℂQ)_{≥ l(N)-|R|-1} + \mathfrak{m} I(P)_{≥ l(N)-2|R|}.
\end{align*}
We have used that $ |W| ≤ |R|+1 $.
\end{proof}

The following is a continuation of the line of Berger and Ginzburg.

\begin{lemma}
\label{th:flatness-flatcompleted-crude}
Assume $ R $ is of bounded type and [BG] and [CP] holds. Then
\begin{equation*}
I(P) ∩ \mathfrak{m} \compl{ℂQ} ⊂ ℂQ_1 (I(P) ∩ \mathfrak{m} \compl{ℂQ}) + \mathfrak{m} I(P).
\end{equation*}
\end{lemma}

\begin{proof}
The strategy is to divide an element $ x $ on the left-hand side into chunks $ x_N $ that individually lie in $ (ℂQ P ℂQ)_{≥N} ∩ \mathfrak{m} $. We estimate that the lengths of the ideal paths used in $ x_N $ converge to $ ∞ $, so that summing up the chunks again lands on the right-hand side and not in completion of $ I(P) $.

Since $ \mathfrak{m} I(P) $ already lies on the right side, it suffices to regard elements of $ (P) $ on the left-hand side. Regard such an element
\begin{equation*}
x = \sum_{i = 0}^∞ p_i (r_i + ψ(r_i)) q_i ∈ (P) ∩ \mathfrak{m} \compl{ℂQ}.
\end{equation*}
For $ N ≥ 0 $, let $ S_N ⊂ ℕ $ be the set of indices $ i $ where all paths in $ p_i r_i q_i $ are related to a path of length $ ≤ N $. All sets $ S_N $ are finite and together exhaust $ ℕ $. In order to rewrite $ x $, let us make these sets disjoint by setting $ S_N' ≔ S_N \setminus S_{N-1} $ for $ N ≥ 1 $ and $ S_0' ≔ S_0 $. Put
\begin{equation*}
x_N = \sum_{i ∈ S_N'} p_i (r_i + ψ(r_i)) q_i.
\end{equation*}
Then we get $ x = \sum_{N = 0}^∞ x_N $. We show that each chunk $ x_N $ lies in $ \mathfrak{m} \compl{ℂQ} $. Let $ M < N $ and $ i ∈ S_M', j ∈ S_N' $. Then paths in $ p_i r_i q_i $ are related to paths of length $ ≤ M $, while paths in $ p_j r_j q_j $ are not related to paths of length $ ≤ M $, since $ j \notin S_M $. This implies that the zeroth order path supports of all chunks $ x_N $ are pairwise disjoint. However we know that $ x ∈ \mathfrak{m} \compl{ℂQ} $, hence $ x $ vanishes on zeroth order. We conclude $ x_N ∈ ℂQ P ℂQ ∩ \mathfrak{m} \compl{ℂQ} $.

Now note that for $ i ∈ S_N' $ the term $ p_i r_i q_i $ has length at least $ N $, for otherwise $ i $ would be contained in $ S_{N-1} $. We conclude that
\begin{equation*}
x_N ∈ (ℂQ P ℂQ)_{≥N} ∩ \mathfrak{m} \compl{ℂQ}.
\end{equation*}
Using \autoref{th:flatness-flatcompleted-bounds}, we deduce
\begin{equation*}
x_N ∈ ℂQ_1 (ℂQ P ℂQ)_{≥ l(N)-|R|-1} + \mathfrak{m} I(P)_{≥ l(N)-2|R|}.
\end{equation*}
Since $ R $ is of bounded type, we have $ l(N) → ∞ $ as $ N → ∞ $. When summing up $ x_N $ over $ N ∈ ℕ $, this estimate on path lengths embracing $ P $ ensures that we get two well-defined sums in $ I(P) $ and not e.g.~in the completion of $ I(P) $:
\begin{equation*}
x = \sum_{N = 0}^∞ x_N ∈ ℂQ_1 I(P) + \mathfrak{m} I(P).
\end{equation*}
Since $ x $ lies in $ \mathfrak{m} \compl{ℂQ} $ by assumption, we obtain
\begin{equation*}
x ∈ (ℂQ_1 I(P) ∩ \mathfrak{m} \compl{ℂQ}) + \mathfrak{m} I(P) = ℂQ_1 (I(P) ∩ \mathfrak{m} \compl{ℂQ}) + \mathfrak{m} I(P).
\end{equation*}
This finishes the proof.
\end{proof}

We now arrive at our first flatness result. Recall that quasi-flatness for $ I(P) $ refers to the inclusion $ I(P) ∩ \mathfrak{m} \compl{ℂQ} ⊂ \mathfrak{m} I(P) $.

\begin{proposition}[Quasi-flatness I]
\label{th:flatness-flatcompleted-th}
Assume $ R $ is of bounded type and [BG] and [CP] hold. Then $ I(P) $ is quasi-flat.
\end{proposition}

\begin{proof}
Iterating the statement of \autoref{th:flatness-flatcompleted-crude}, we get
\begin{align*}
I(P) ∩ \mathfrak{m} \compl{ℂQ} &⊂ ℂQ_1 (I(P) ∩ \mathfrak{m} \compl{ℂQ}) + \mathfrak{m} I(P) \\
&⊂ ℂQ_1 (ℂQ_1 (I(P) ∩ \mathfrak{m} \compl{ℂQ}) + \mathfrak{m} I(P)) + \mathfrak{m} I(P) ⊂ ….
\end{align*}
In other words, pick $ x ∈ I(P) ∩ \mathfrak{m} \compl{ℂQ} $. Then we get a sequence of elements $ x_N ∈ ℂQ_N (I(P) ∩ \mathfrak{m} \compl{ℂQ}) $ and $ y_N ∈ \mathfrak{m} ℂQ_{N-1} I(P) $ with
\begin{equation*}
∀N ∈ ℕ: \quad x = x_N + y_1 + … + y_N.
\end{equation*}
Now the series $ \sum_{N = 1}^∞ y_N $ defines an element in $ \mathfrak{m} I(P) $, since its summands $ y_N $ contain only paths of higher and higher length. We claim that this element is precisely $ x $. Indeed, disregard all paths longer than an arbitrary number. Then the series stabilizes to $ x $ after finitely many summands and does not change anymore thereafter. Finally, we conclude $ x $ is the sum of the series, which we already know lies in $ \mathfrak{m} I(P) $.
\end{proof}

%% file: flatness/flat_algebraic.tex
\subsection{Quasi-flatness in the path algebra}
\label{sec:flatness-flatalgebraic}
In this section, we build our second quasi-flatness result. Namely, we show that $ I(P)_{ℂQ} $ is quasi-flat if $ R $ is of bounded type and the conditions [BG] and [CP] hold. In \autoref{sec:flatness-flatcompleted} we have already seen that $ I(P) $ is quasi-flat. The idea of the present section is to use the boundedness argument to bring quasi-flatness from $ I(P) $ to $ I(P)_{ℂQ} $.

The first step is to refine the Berger-Ginzburg lemma \autoref{th:flatness-berger-ginzburg}. We deploy the bounded version of the Berger-Ginzburg inclusion, sticking to the lower bounds and forgetting the upper bounds.

\begin{lemma}
\label{th:flatness-deformed-estimate-crude}
Let $ R $ be of bounded type. Then for any natural numbers $ M ≤ N $ we have
\begin{equation*}
I(P) ∩ (\fourIdx{}{M≤}{}{≤N}{\compl{ℂQ}} + \mathfrak{m} \compl{ℂQ}) ⊂ \fourIdx{}{l(M)≤}{}{≤h(N)}{(ℂQ P ℂQ)} + \mathfrak{m} \compl{ℂQ} ∩ I(P).
\end{equation*}
The inclusion also holds once the upper bound or lower bound is dropped.
\end{lemma}

\begin{proof}
Since $ \mathfrak{m} I(P) $ already lies on the right-hand side, it suffices to prove the inclusion of
\begin{equation*}
(P) ∩ (\fourIdx{}{M≤}{}{≤N}{\compl{ℂQ}} + \mathfrak{m} \compl{ℂQ}).
\end{equation*}
Pick an element
\begin{equation*}
x = \sum_{i = 0}^∞ p_i (r_i + ψ(r_i)) q_i ∈ (P) ∩ (\fourIdx{}{M≤}{}{≤N}{\compl{ℂQ}} + \mathfrak{m} \compl{ℂQ}).
\end{equation*}
We can assume $ |p_i| + |q_i| → ∞ $. Our strategy is to decompose $ x $ into two parts. Let $ S ⊂ ℕ $ be the set of indices $ i $ where the constituents of $ p_i r_i q_i $ are related to paths of length in the interval $ [M, N] ⊂ ℕ $. Then for $ i ∈ S $ all paths in $ p_i r_i q_i $ are related to paths of length in $ [M, N] $ and are hence of length in $ [l(M), h(N)] $ themselves. In particular $ S $ is finite. Meanwhile for $ i \notin S $ no paths in $ p_i r_i q_i $ are related to paths of length in $ [M, N] $ and we conclude that
\begin{equation*}
\sum_{i ∈ ℕ \setminus S} p_i r_i q_i = 0.
\end{equation*}
In other words, we have
\begin{equation*}
x_{ℕ \setminus S} ≔ \sum_{i ∈ ℕ \setminus S} p_i (r_i + ψ(r_i)) q_i ∈ (P) ∩ \mathfrak{m}ℂQ,
\end{equation*}
while
\begin{equation*}
x_S ≔ \sum_{i ∈ S} p_i (r_i + ψ(r_i)) q_i ∈ \fourIdx{}{l(M)≤}{}{≤h(N)}{(ℂQ P ℂQ)}.
\end{equation*}
Recalling $ x = x_S + x_{ℕ \setminus S} $ gives that $ x $ lies in the right-hand side of the desired inclusion.
\end{proof}

We seek to apply \autoref{th:flatness-deformed-estimate-crude} iteratively. As we pull out more and more powers of the maximal ideal, we need to ensure that the remainder is still bounded in length at zeroth order. We can achieve this if we require from the very beginning that an element $ x ∈ I(P) $ has bounded length at order of $ \mathfrak{m} $. For $ B = ℂ⟦q⟧ $ we would require that the path lengths at every $ q $ power $ q^k $ are bounded. For general deformation basis $ B $, we shall require that the path lengths in $ x $ shall be bounded if we project to $ B / \mathfrak{m}^k $. This gives rise to a subset of $ B \htensor ℂQ $, namely the space of elements with length at most $ N_k $ at $ \mathfrak{m} $-order $ ≤ k $. Let us record this in the following definition:

\begin{definition}
Let $ N_0, N_1, … $ be an increasing sequence of integers. Then we set
\begin{align*}
ℂQ_{≤ (N_0, N_1, …)} &= \{x ∈ B \htensor ℂQ \running π_k (x) ∈ (B / \mathfrak{m}^k) ¤ ℂQ_{≤ N_k} ~∀k ∈ ℕ\} \\
& = ℂQ_{≤N_0} + \mathfrak{m}^1 ℂQ_{≤N_1} + \mathfrak{m}^2 ℂQ_{≤N_2} + ….
\end{align*}
Here $ π_k: B \htensor ℂQ → (B/\mathfrak{m}^k) ¤ ℂQ $ denotes the standard projection. 
\end{definition}

\begin{example}
For $ B = ℂ⟦q⟧ $ the space $ ℂQ_{≤ (N_0, N_1, …)} $ is simply the space of $ ℂQ $-valued power series in $
 q $ where the paths at order $ q^k $ are of length $ ≤ N_k $.
\end{example}

We are now getting closer to proving quasi-flatness of $ I(P)_{ℂQ} $. As we shall see, elements of $ I(P)_{ℂQ} $ are namely distinguished elements of $ I(P) $ in the sense that they simultaneously lie in $ ℂQ_{≤ (N_0, …)} $ for some sequence $ N_0 ≤ N_1 ≤ … $. To exploit this property, let us prove the following lemma.

\begin{lemma}
\label{th:flatness-flatalgebraic-termsout}
Assume $ R $ is of bounded type and $ I(P) $ is quasi-flat. Let $ N_1 ≤ N_2 ≤ … $ be a sequence. Then there exists a sequence $ N_1' ≤ N_2' ≤ … $ such that
\begin{align*}
I(P) ∩ ℂQ_{≤ (0, N_1, …)} &⊂ \sum_{k = 1}^∞ \mathfrak{m}^k (ℂQ P ℂQ)_{≤ N_k'} \\
& = \mathfrak{m}^1 (ℂQ P ℂQ)_{≤ N_1'} + \mathfrak{m}^2 (ℂQ P ℂQ)_{≤ N_2'} + …
\end{align*}
\end{lemma}

\begin{proof}
The idea is to iterate \autoref{th:flatness-flatalgebraic-crude} in combination with \autoref{th:flatness-deformed-estimate-crude} and quasi-flatness of $ I(P) $. In the first part of the proof, we construct the desired sequence $ (N_k') $. In the second part, we construct the desired decomposition for individual elements $ x ∈ I(P) ∩ ℂQ_{≤ (0, N_1, …)} $. In the third part, we wrap up and prove the desired inclusion.

For the first part of the proof, let us construct the sequence $ (N_k') $. Given arbitrary $ k ∈ ℕ $, let $ M_k $ be the maximum path length encountered in the image of $ π ∘ ψ: R → B / \mathfrak{m}^k ¤ ℂQ $. We construct an auxiliary sequence as
\begin{align*}
N_1'' &= N_1, \\
N_2'' &= \max(M_1 + h(N_1''), N_2), \\
N_3'' &= \max(M_2 + h(N_2''), N_3), \\
… &
\end{align*}
The sequence $ (N_k'') $ is increasing, since $ h(N_i'') ≥ N_i'' $. We construct the final desired sequence as
\begin{equation*}
N_k' = h(N_k'').
\end{equation*}

For the second part of the proof, we prove a finite version of the desired inclusion by iterating \autoref{th:flatness-deformed-estimate-crude}. Pick an arbitrary element $ x ∈ I(P) ∩ ℂQ_{≤ (0, N_1, …)} $. We shall inductively construct sequences $ (y_k)_{k ≥ 1} $, $ (x_k)_{k ≥ 0} $ such that
\begin{equation}
\label{eq:flatness-flatalgebraic-crudesequence}
\begin{aligned}
& ∀k ≥ 0: \quad x = y_1 + … + y_k + x_k, \\
& ∀k ≥ 0: \quad x_k ∈ \mathfrak{m}^{k+1} I(P), \\
& ∀k ≥ 1: \quad y_k ∈ \mathfrak{m}^k (ℂQ P ℂQ)_{≤ h(N_k'')}.
\end{aligned}
\end{equation}
Our induction starts at $ k = 0 $. Note that we are not required to construct an element $ y_0 $. For the induction base $ k = 0 $, note that quasi-flatness gives
\begin{equation*}
x ∈ I(P) ∩ ℂQ_{≤ (0, N_1, …)} ⊂ \mathfrak{m} I(P).
\end{equation*}
We simply put $ x_0 = x $. This satisfies the requirement \eqref{eq:flatness-flatalgebraic-crudesequence} for the induction base $ k = 0 $.

Towards the induction step, let $ k ≥ 0 $ and assume that $ x_0, …, x_k $ and $ y_1, …, y_k $ have been constructed with property \eqref{eq:flatness-flatalgebraic-crudesequence}. Regard the element $ x_k = x - y_1 - … - y_k $. Since $ x ∈ ℂQ_{≤ (0, N_1, …)} $ and $ y_i ∈ \mathfrak{m}^i (ℂQ P ℂQ)_{≤ h(N_k'')} $, the total length in $ x_k $ at level $ \mathfrak{m}^{k+1} $ is less or equal to
\begin{equation*}
\max(N_{k+1}, M_k + h(N_1''), M_{k-1} + h(N_2''), …, M_1 + h(N_k'')) ≤ \max(N_{k+1}, M_k + h(N_k'')) = N_{k+1}''.
\end{equation*}
In the above inequality, the term $ M_k + h(N_k'') $ in the second maximum only appears when $ k ≥ 1 $. In either case, the bound by $ N_{k+1} $ is valid. In combination with \autoref{th:flatness-flatalgebraic-crude}, \autoref{th:flatness-deformed-estimate-crude} and quasi-flatness we conclude
\begin{align*}
x_k &∈ \mathfrak{m}^{k+1} I(P) ∩ \mathfrak{m}^{k+1} (ℂQ_{≤ (N_{k+1}'')} + \mathfrak{m} \compl{ℂQ}) \\
& ⊂ \mathfrak{m}^{k+1} (I(P) ∩ (ℂQ_{≤ N_{k+1}''} + \mathfrak{m} \compl{ℂQ})) + \mathfrak{m}^{k+2} I(P) \\
& ⊂ \mathfrak{m}^{k+1} ((ℂQ P ℂQ)_{≤ h(N_{k+1}'')} + \mathfrak{m} I(P)) + \mathfrak{m}^{k+2} I(P) \\
& ⊂ \mathfrak{m}^{k+1} (ℂQ P ℂQ)_{≤ h(N_{k+1}'')} + \mathfrak{m}^{k+2} I(P).
\end{align*}
According to this sum decomposition, write $ x_k = y_{k+1} + x_{k+1} $. This finishes the induction step. Ultimately, we have constructed the sequences $ (x_k)_{k ≥ 0} $ and $ (y_k)_{k ≥ 1} $ with property \eqref{eq:flatness-flatalgebraic-crudesequence}.

For the third part of the proof, we wrap up and prove the desired inclusion. Regard an element $ x ∈ I(P) ∩ ℂQ_{≤ (0, N_1, …)} $ together with its sequences $ (x_k) $ and $ (y_k) $. With respect to the $ \mathfrak{m} $-adic topology on $ B \htensor A $, we have the converging series
\begin{equation*}
x = \sum_{k = 1}^∞ y_k.
\end{equation*}
Indeed, the summands lie in increasingly high powers of $ \mathfrak{m} $. The limit is $ x $, since the difference between $ y_1 + … + y_k $ and $ x $ is $ x_k $ which lies in increasingly high powers of $ \mathfrak{m} $. Even better, we conclude
\begin{equation*}
x ∈ \sum_{k = 1}^∞ \mathfrak{m}^k (ℂQ P ℂQ)_{≤ h(N_k'')}.
\end{equation*}
Since $ x $ was arbitrary and $ N_k' $ is defined as $ h(N_k'') $, this proves the claim.
\end{proof}

We are ready to prove quasi-flatness of $ I(P)_{ℂQ} $. The requirement is that $ I(P) $ is already quasi-flat. By \autoref{th:flatness-flatcompleted-th}, this happens for example $ R $ is of finite type and [BG] and [CP] hold.

\begin{proposition}[Quasi-flatness II]
\label{th:flatness-flatalgebraic-th}
Assume $ R $ is of bounded type and $ ψ $ maps to $ \mathfrak{m} ℂQ $. If $ I(P) $ is quasi-flat, then $ I(P)_{ℂQ} $ is quasi-flat.
\end{proposition}

\begin{proof}
It is our task to show $ I(P)_{ℂQ} ∩ \mathfrak{m} ℂQ ⊂ \mathfrak{m} I(P)_{ℂQ} $. Pick an element $ x ∈ I(P)_{ℂQ} ∩ \mathfrak{m} ℂQ $. Start with the observation that the relation space $ R $ is finite-dimensional, the image of $ ψ $ lies in $ \mathfrak{m} ℂQ $ instead of $ \mathfrak{m} \compl{ℂQ} $ and in $ I(P)_{ℂQ} $ only finitely many paths are multiplied to $ P $ at order $ ≤ k $. We see that $ x $ has bounded path length at order $ ≤ k $ for every $ k ∈ ℕ $. In other words, there is a sequence $ N = (N_0, N_1, …) $ such that $ x ∈ ℂQ_{≤ (0, N_1, …)} $. In total, we have $ x ∈ I(P) ∩ ℂQ_{≤ (0, N_1, …)} $. According to \autoref{th:flatness-flatalgebraic-termsout}
, we get
\begin{equation*}
x ∈ \sum_{k = 1}^∞ \mathfrak{m}^k (ℂQ P ℂQ)_{≤ N_k'}.
\end{equation*}
This shows $ x ∈ I(P)_{ℂQ} $ and finishes the proof.
\end{proof}

%% file: flatness/closedness.tex
\subsection{Closedness results}
\label{sec:flatness-closedness}
In this section, we prove additional closedness results. The idea is that a quotient of an algebra by an ideal can always be taken. However, if the algebra enjoys topological properties, they can be lost in the quotient if the ideal is not good enough. In the present section we devote ourselves to the study of this problem in the case of the deformed ideals $ I(P) ⊂ \compl{ℂQ} $. Essentially, we show that $ I(P) $ is good enough for the “Krull topology” on $ \compl{ℂQ} $ if $ R $ is of bounded type.

Let us start by recalling the classical topology on $ \compl{ℂQ} $.

\begin{definition}
The \emph{Krull topology} on $ ℂQ $ (or $ \compl{ℂQ} $) is the topology generated by the neighborhood basis
\begin{equation*}
x + ℂQ_{≥N} \quad \text{or} \quad x + \compl{ℂQ}_{≥N}, \quad \text{for } x ∈ ℂQ, ~ N ∈ ℕ.
\end{equation*}
\end{definition}

\begin{remark}
With the Krull topology, the space $ ℂQ $ is first countable and sequential. A sequence $ x_n ∈ ℂQ $ converges to some $ x ∈ ℂQ $ if $ x_n - x_{n+1} $ is concentrated in higher and higher path length. The space $ ℂQ $ is not complete: The space $ \compl{ℂQ} $ is in fact its completion.
\end{remark}

\begin{remark}
\label{rem:flatness-closedness-Iclosed}
A standard setup in algebra is as follows: Given is a path algebra $ ℂQ $ of a quiver and one is interested in dividing out an ideal $ I ⊂ ℂQ $. One accepts the quotient algebra $ ℂQ / I $ without questioning its topological properties. If $ I_q $ is a quasi-flat deformation of $ I $, then $ (B \htensor ℂQ) / I_q $ is a deformation of $ ℂQ / I $. For path algebras of quivers, this is all one typically desires.

In contrast, consider an ideal $ I ⊂ \compl{ℂQ} $ in the completed path algebra. Then one can still form the quotient $ \compl{ℂQ} / I $, but one is interested in the Krull topology on the quotient. Therefore one typically requires that $ I ⊂ \compl{ℂQ} $ is closed with respect to the Krull topology.
\end{remark}

\begin{remark}
Let us use the standard notation $ \compl{ℂQ} R \compl{ℂQ} $ for the ideal generated by $ R $ in $ \compl{ℂQ} $. We claim that
\begin{equation*}
\compl{ℂQ} R \compl{ℂQ} ⊂ I(R)_{\compl{ℂQ}} ⊂ \closure{I(R)}.
\end{equation*}
Indeed, the left-hand side is the finite span of elements of the form $ prq $ where $ r ∈ R $ and $ p, q ∈ \compl{ℂQ} $. A reordering, or counting argument, for the constituents of $ p $ and $ q $ shows that $ prq ∈ I(R)_{\compl{ℂQ}} $. On the other hand, pick an element $ x ∈ I(R)_{\compl{ℂQ}} $, presented as a series \autoref{def:flatness-ideals-IP}. Truncating the series at high indices immediately shows that $ x $ lies in the closure of $ I(R) $.
\end{remark}

We arrive at the following inclusion:

\begin{equation*}
R ~⊂~ I(R) ~⊂~ \compl{ℂQ} R \compl{ℂQ} ~⊂~ I(R)_{\compl{ℂQ}} ~⊂~ \closure{I(R)}.
\end{equation*}

In case $ R $ is of bounded type, we can improve on these inclusions: We shall see that $ I(R)_{\compl{ℂQ}} $ is the closure of $ I(R) $ and thereby also the closure of $ \compl{ℂQ} R \compl{ℂQ} $.

\begin{remark}
\label{rem:flatness-closedness-nonclosedCQRCQ}
At this point, we shall comment on the stark difference with the commutative case. In fact, $ \compl{ℂQ} R \compl{ℂQ} $ is not necessarily closed with respect to the Krull topology, while this would hold if $ A = \compl{ℂQ} $ were commutative. Indeed, let $ A $ be a commutative local ring with maximal ideal $ \mathfrak{m}_A $. Then $ (A, \mathfrak{m}_A) $ is a Zariski ring and hence any ideal $ I ⊂ A $ is automatically closed with respect to the $ \mathfrak{m}_A $-adic topology. If $ \compl{ℂQ} $ were commutative, this would imply that $ \compl{ℂQ} R \compl{ℂQ} $ is closed with respect to the Krull topology.

We shall here give an example of a quiver $ Q $ and a space $ R $ where $ \compl{ℂQ} R \compl{ℂQ} $ is clearly not closed:

\begin{center}
\begin{tikzpicture}
\path[fill] (0, 0) circle[radius=0.05] (2, 0) circle[radius=0.05];
\path[draw, ->] (0, 0.1) to[out=90, in=90] (-1, 0) node[left] {$ C $} to[out=270, in=270] (0, -0.1);
\path[draw, <-] (2, 0.1) to[out=90, in=90] (3, 0) node[right] {$ A $} to[out=270, in=270] (2, -0.1);
\path[draw, ->] (0.1, 0) -- (1.9, 0) node[midway, above] {$ B $};
\path (5, 0) node {$ R = \vspan(B) $.};
\end{tikzpicture}
\end{center}

We note that the series $ \sum_{i = 0}^∞ A^i B C^i $ lies in the closure of $ ℂQ R ℂQ $, and claim that it does not lie in $ \compl{ℂQ} B \compl{ℂQ} $. The clue is to analyze the structure of infinite sums of paths. For any element $ \sum_{i, j = 0}^∞ λ_{i, j} A^i B C^j $ let us call $ (λ_{i, j}) $ its coefficient matrix. An element of the form
\begin{equation*}
\left(\sum_{i = 0}^∞ λ_i A^i\right) B \left(\sum_{j = 0}^∞ η_j C^j\right)
\end{equation*}
has a coefficient matrix of column rank at most one, since all columns are multiples of each other. Any element in the ideal $ \compl{ℂQ} B \compl{ℂQ} $ is a finite sum of elements of this form and therefore has coefficient matrix with finite column rank. The coefficient matrix of the element $ x = \sum_{i = 0}^∞ A^i B C^i $ is however an infinite diagonal matrix which has infinite rank. We conclude $ x \notin \compl{ℂQ} B \compl{ℂQ} $. This illustrates a difference with the commutative case and shows how intricate closedness can be.
\end{remark}

\begin{lemma}
\label{th:flatness-closedness-IRclosed}
Let $ R $ be of bounded type. Then $ I(R)_{\compl{ℂQ}} $ is the closure of $ I(R) $ in the Krull topology.
\end{lemma}

\begin{proof}
It is clear that $ I(R)_{\compl{ℂQ}} $ is contained in the closure of $ I(R) $. Conversely, $ I(R) $ is naturally a subset of $ I(R)_{\compl{ℂQ}} $. It remains to be shown that $ I(R)_{\compl{ℂQ}} $ is closed.

Regard a sequence $ (x_n) ⊂ I(R)_{\compl{ℂQ}} $ with $ x_n → x ∈ \compl{ℂQ} $. It is our task to show that $ x $ can be written as a series
\begin{equation*}
x = \sum_{i = 0}^∞ p_i r_i q_i \quad \text{where} \quad r_i ∈ R, ~ |p_i| + |q_i| → ∞.
\end{equation*}
We can assume that $ x_{n+1} - x_n ∈ ℂQ_{≥ k_n} $ for a sequence $ (k_n) ⊂ ℕ $ with $ k_n → ∞ $. The standard boundedness argument shows that
\begin{equation*}
x_n - x_{n+1} ∈ I(R) ∩ \compl{ℂQ}_{≥ k_n} ⊂ (ℂQ R ℂQ)_{≥ l(k_n)}.
\end{equation*}
Summing up over $ n ≥ 0 $, we get that $ x $ is indeed an element of $ I(R)_{\compl{ℂQ}} $.
\end{proof}

Next, we shall prove that $ I(R)_{\compl{ℂQ}} $ has a closed complement in $ \compl{ℂQ} $.

\begin{example}
The existence of closed complements is intricate. It is not true that a complement of an ideal $ I ⊂ ℂQ $ gives rise to a closed complement of its closure $ \closure{I} ⊂ \compl{ℂQ} $. More precisely, if $ ℂQ = I + V $, then it does not necessarily hold that $ \compl{ℂQ} = \closure{I} + \closure{V} $. For instance, pick $ ℂQ = ℂ⟨A⟩ $ and regard
\begin{align*}
I &= (1+A) = \vspan(1+A, A+A^2, …), \\
V &= \vspan(1).
\end{align*}
We have $ ℂ⟨A⟩ = I ⊕ V $. We have $ 1 - A^n ∈ I $, hence $ 1 ∈ \closure{I} ∩ \closure{V} $.
\end{example}

\begin{lemma}
\label{th:flatness-closedness-IRcomplement}
Let $ R $ be of bounded type. Then $ I(R)_{\compl{ℂQ}} $ has a closed complement.
\end{lemma}

\begin{proof}
It is our task to find a closed subspace $ V ⊂ \compl{ℂQ} $ such that $ \compl{ℂQ} = I(R)_{\compl{ℂQ}} ⊕ V $. The naive idea is to fill $ V $ by picking paths of increasing length which complement $ I(R) $. However, the closure of the space spanned by these paths need not have vanishing intersection with $ I(R)_{\compl{ℂQ}} $ in general. The remedy is provided by the boundedness condition, which allows us to reduce the vanishing intersection property to checking it on finite lengths of paths. We divide the proof into three steps: In the first part of the proof, we construct the space $ V $. In the second part, we show $ I(R)_{\compl{ℂQ}} ∩ V = 0 $. In the third part, we show $ \compl{ℂQ} = I(R)_{\compl{ℂQ}} + V $.

For the first part of the proof, we construct the space $ V $. Pick for every $ N ∈ ℕ $ a set $ V_N $ of paths of length $ N $ such that all paths in $ V_n $ for $ n ≤ N $ are linearly independent of $ I(R) $, but span $ ℂQ_{≤N} $ when combined with $ I(R) $:
\begin{equation*}
ℂQ_{≤N} ⊂ \vspan(V_0 ∪ … ∪ V_N) ⊕ I(R).
\end{equation*}
Our candidate space for the closed complement is
\begin{equation*}
V = \prod_{N = 0}^∞ \vspan V_N ⊂ \compl{ℂQ}.
\end{equation*}
This is the subspace consisting of elements of $ \compl{ℂQ} $ which are supported on paths lying in the union $ ∪_{N ∈ ℕ} V_N $. The space $ V $ is closed with respect to the Krull topology. Indeed, sticking to our sequential viewpoint, whenever $ (x_n) ⊂ V $ converges in $ \compl{ℂQ} $, then coefficients of $ x_n $ on paths in low path length stabilize. From a purely topological viewpoint, if $ x ∈ \compl{ℂQ} \setminus V $, then $ x $ has support on some path $ p \notin ∪_{N ∈ ℕ} V_N $. Any element $ y ∈ x + ℂQ_{>|p|} $ then also has support on this path $ p $ and therefore $ y \notin V $. This shows $ V $ is closed.

For the second part of the proof, we show $ I(R)_{\compl{ℂQ}} ∩ V = 0 $. Pick any element $ x ∈ I(R)_{\compl{ℂQ}} ∩ V $. We shall prove that $ x ∈ ℂQ_{>N} $ for every $ N ∈ ℕ $. Indeed, let $ N ∈ ℕ $. Then decompose $ x = x_1 + x_2 $ into path support related to length $ ≤ N $ and path support not related to length $ ≤ N $. Since $ x $ lies in $ V $ and $ V $ is built from pure paths, we have that $ x_1, x_2 ∈ V $. More specifically, since all paths related to length $ ≤ N $ are of length $ ≤ h(N) $, we have $ x_1 ∈ \vspan(V_0 ∪ … ∪ V_{h(N)}) $.

We claim that $ x_1 ∈ I(R) $. Indeed, write
\begin{equation*}
x = \sum_{i = 0}^∞ p_i r_i q_i, \quad r_i ∈ R, ~ |p_i| + |q_i| → ∞.
\end{equation*}
Let $ S ⊂ ℕ $ be the set of indices where all paths in $ p_i r_i q_i $ are related to paths of length $ ≤ N $. Then
\begin{equation*}
x = \sum_{i ∈ S} p_i r_i q_i + \sum_{i ∈ ℕ \setminus S} p_i r_i q_i.
\end{equation*}
In the first summand, all paths are related to paths of length $ ≤ N $. In the second summand, no paths are related to length $ ≤ N $. In conclusion, we have
\begin{equation*}
x_1 = \sum_{i ∈ S} p_i r_i q_i, \quad x_2 = \sum_{i ∈ ℕ \setminus S} p_i r_i q_i.
\end{equation*}
We obtain $ x_1 ∈ I(R) $. In conclusion $ x_1 ∈ I(R) ∩ \vspan(V_0, …, V_{h(N)}) $. However by construction of the sets $ V_n $, the sum $ I(R) + \vspan(V_0, …, V_{h(N)}) $ is actually direct and hence $ x_1 = 0 $. We conclude $ x = x_2 $. Since all paths in the support of $ x_2 $ are not related to length $ ≤ N $, they are of length $ > N $. We conclude $ x ∈ \compl{ℂQ}_{>N} $. Since $ N $ was arbitrary, we conclude $ x = 0 $. This shows $ I(R)_{\compl{ℂQ}} ∩ V = 0 $.

For the third part of the proof, we show that $ \compl{ℂQ} = I(R)_{\compl{ℂQ}} + V $. The idea is to break down an element of $ \compl{ℂQ} $ into pieces which individually lie in the two spaces and make sure the pieces sum up appropriately. Let now $ x ∈ \compl{ℂQ} $. Split $ x $ into pieces according to path length:
\begin{equation*}
x = \sum_{n = 0}^∞ x_n, \quad |x_n| = n.
\end{equation*}
We claim that for every $ n ∈ ℕ $ we have $ x_n ∈ I(R)_{≥l(n)} + \vspan(V_{l(n)} ∪ … V_n) $. To prove this, fix $ n ∈ ℕ $. By construction, we can write $ x = y_n + z_n $ with $ y_n ∈ I(R) $ and $ z_n ∈ \vspan(V_0 ∪ … ∪ V_n) $. Write
\begin{equation*}
y_n = \sum_{i ∈ I}^∞ p_i r_i q_i, \quad r_i ∈ R, |I| < ∞.
\end{equation*}
Let $ S ⊂ I $ be the set of indices $ i ∈ I $ where all paths in $ p_i r_i q_i $ are related to a path of length $ n $. We have $ y_n = y_n^{(1)} + y_n^{(2)} $ with
\begin{equation*}
y_n^{(1)} = \sum_{i ∈ S} p_i r_i q_i, \quad y_n^{(2)} = \sum_{i ∈ I \setminus S} p_i r_i q_i.
\end{equation*}
For $ i ∈ S $, all paths in $ p_i r_i q_i $ are of length $ ≥ l(n) $. This implies $ y_n^{(1)} ∈ I(R)_{≥l(n)} $.

Next, split $ z_n = z_n^{(1)} + z_n^{(2)} $ with $ z_n^{(1)}, z_n^{(2)} ∈ \vspan(V_0 ∪ … ∪ V_n) $ such that $ z_n^{(1)} $ only contains paths related to paths of length $ n $ and $ z_n^{(2)} $ only contains paths not related to paths of length $ n $. Since all paths related to paths of length $ n $ have length at least $ l(n) $, we have $ z_n^{(1)} ∈ \vspan(V_{l(n)} ∪ … ∪ V_n) $.

We argue that $ y_n^{(2)} + z_n^{(2)} = 0 $. Indeed, all paths in $ y_n^{(1)} + z_n^{(1)} $ are related to paths of length $ n $, while all paths in $ y_n^{(2)} + z_n^{(2)} $ are not related to length $ n $. Since $ x_n $ itself has length $ n $, we conclude $ y_n^{(2)} + z_n^{(2)} = 0 $.

Finally, we get $ x_n = y_n^{(1)} + z_n^{(1)} $. Since $ y_n^{(1)} ∈ I(R)_{≥l(N)} $ and $ z_n^{(1)} ∈ \vspan(V_{l(n)} ∪ … ∪ V_n) $ and $ l(n) → ∞ $, we get
\begin{equation*}
x = \sum_{n = 0}^∞ y_n^{(1)} + z_n^{(1)} = \sum_{n = 0}^∞ y_n^{(1)} + \sum_{n = 0}^∞ z_n^{(1)} ∈ I(R)_{\compl{ℂQ}} + V.
\end{equation*}
This proves that $ \compl{ℂQ} = I(R)_{\compl{ℂq}} + V $. Together with the fact $ I(R)_{\compl{ℂQ}} ∩ V = 0 $ proven before, this proves the direct sum decomposition $ \compl{ℂQ} = I(R)_{\compl{ℂQ}} ⊕ V $ and finishes the proof.
\end{proof}


In the remainder of this section, we devote ourselves to the study of $ P $ and its associated ideal-like spaces. Let us give an indication of the additional topology we want to regard. As we recalled in \autoref{rem:flatness-closedness-Iclosed}, for a quotient $ \compl{ℂQ} / I $ one typically demands that $ I ⊂ \compl{ℂQ} $ be closed with respect to the Krull topology. The quotient $ \compl{ℂQ} / I $ then inherits a quotient topology. The tensor product $ B \htensor (\compl{ℂQ} / I) $ then obtains a combination of the Krull and $ \mathfrak{m} $-adic topology. We define this topology as follows:

\begin{definition}
\label{def:flatness-closedness-tensortop}
Let $ B $ be a deformation base. Then the \emph{tensor topology} on $ B \htensor \compl{ℂQ} $ is the limit topology induced from
\begin{equation*}
B \htensor \compl{ℂQ} = \lim (B/\mathfrak{m}^k ¤ \compl{ℂQ})
\end{equation*}
Here the individual spaces $ B/\mathfrak{m}^k ¤ \compl{ℂQ} $ are equipped with the Krull topology. An explicit neighborhood basis for the topology is given by the subspaces $ x + \mathfrak{m}^k \compl{ℂQ} + B \compl{ℂQ}_{≥N} $ for $ x ∈ B \htensor \compl{ℂQ} $.
\end{definition}

\begin{remark}
With respect to the tensor topology, a sequence $ x_n ⊂ \compl{ℂQ} $ converges if for every $ k ∈ ℕ $, the path lengths of the differences $ x_n - x_{n+1} $ go to infinity once $ \mathfrak{m}^k $ is divided out.
\end{remark}

We have the following chain of inclusions for $ P $ and its associated ideal-like spaces:

\begin{center}
\begin{tikzpicture}
\newcommand{\inclusionconnect}[2]{\path (#1.east) -- (#2.west) node[midway, sloped] {$ ⊂ $};}
\path (1, 0) node (A) {$ P $} (3, 0) node (B) {$ ℂQ P ℂQ $} (5, 0) node (C) {$ \compl{ℂQ} P \compl{ℂQ} $} (8, 0.5) node (D) {$ (P) $} (8, -0.5) node (E) {$ (B \htensor \compl{ℂQ}) P (B \htensor \compl{ℂQ}) $} (11, 0) node (F) {$ I(P) $} (12.5, 0) node (G) {$ \closure{I(P)} $} (14, 0) node (H) {$ \closure{I(P)}^{\tensor} $};
\inclusionconnect AB
\inclusionconnect BC
\inclusionconnect CD
\inclusionconnect CE
\inclusionconnect DF
\inclusionconnect EF
\inclusionconnect FG
\inclusionconnect GH
\end{tikzpicture}
\end{center}

Here $ (B \htensor \compl{ℂQ}) P (B \htensor \compl{ℂQ}) $ denotes the ideal generated by $ P $ in $ B \htensor \compl{ℂQ} $. If $ R $ is of bounded type and $ I(P) $ is quasi-flat, the inclusions simplify: We shall see that $ I(P) $ is the closure of $ (B \htensor \compl{ℂQ}) P (B \htensor \compl{ℂQ}) $.

\begin{lemma}
\label{th:flatness-completed-powerswitch}
Let $ R $ be of bounded type and $ I(P) $ quasi-flat. Write $ \tilde l (N) = l(N) - |R| $ for $ N ∈ ℕ $. Then for $ k ≥ 1 $ and $ N ≥ 0 $ we have
\begin{align*}
& I(P) ∩ (B \compl{ℂQ}_{≥N} + \mathfrak{m}^k \compl{ℂQ}) \\
& ⊂ \sum_{i = 0}^{k-1} \mathfrak{m}^i (ℂQ P ℂQ)_{≥l(\tilde l^i (N))} + \mathfrak{m}^k I(P) \\
& = (ℂQ P ℂQ)_{≥ l(N)} + \mathfrak{m} (ℂQ P ℂQ)_{≥ l(\tilde l(N))} + \mathfrak{m}^2 (ℂQ P ℂQ)_{≥ l(\tilde l(\tilde l(N)))} + … + \mathfrak{m}^k I(P).
\end{align*}
\end{lemma}

\begin{proof}
The proof is very similar to the proof of \autoref{th:flatness-flatalgebraic-termsout}. The idea is to iterate \autoref{th:flatness-flatalgebraic-crude} in combination with \autoref{th:flatness-deformed-estimate-crude} and quasi-flatness of $ I(P) $.

Let $ x ∈ I(P) ∩ (B \compl{ℂQ}_{≥N} + \mathfrak{m}^k \compl{ℂQ}) $. Then in particular
\begin{equation*}
x ∈ I(P) ∩ (\compl{ℂQ}_{≥N} + \mathfrak{m} \compl{ℂQ}) ⊂ (ℂQ P ℂQ)_{≥ l(N)} + \mathfrak{m} I(P).
\end{equation*}
According to this sum decomposition, write $ x = y_1 + x_1 $. We clearly have $ y_1 ∈ B \compl{ℂQ}_{≥ l(N) - |R|} $. Since $ x ∈ B \compl{ℂQ}_{≥N} + \mathfrak{m}^k \compl{ℂQ} $, we conclude $ x_1 ∈ \mathfrak{m} I(P) ∩ (B \compl{ℂQ}_{≥ \tilde l(N)} + \mathfrak{m}^k \compl{ℂQ}) $. We now continue this way, using \autoref{th:flatness-flatalgebraic-crude}:
\begin{align*}
x_1 &∈ \mathfrak{m} I(P) ∩ (\mathfrak{m} \compl{ℂQ}_{≥ \tilde l (N)} + \mathfrak{m}^2 \compl{ℂQ}) \\
& ⊂ \mathfrak{m} (I(P) ∩ (\compl{ℂQ}_{≥ \tilde l (N)} + \mathfrak{m} \compl{ℂQ})) + \mathfrak{m}^2 I(P) \\
& ⊂ \mathfrak{m} ((ℂQ P ℂQ)_{≥ l (\tilde l(N))} + \mathfrak{m} I(P)) + \mathfrak{m} I(P) \\
& ⊂ \mathfrak{m} (ℂQ P ℂQ)_{≥ l(\tilde l(N))} + \mathfrak{m}^2 I(P).
\end{align*}
Split $ x_1 = y_2 + x_2 $ according to this sum decomposition and continue. The result is immediate.
\end{proof}

\begin{proposition}
\label{th:flatness-completed-IPclosed}
Assume $ R $ is of bounded type and $ I(P) $ is quasi-flat. Then $ I(P) $ is the closure of the ideal $ (B \htensor \compl{ℂQ}) P (B \htensor \compl{ℂQ}) $ with respect to the tensor topology.
\end{proposition}

\begin{proof}
It is clear that $ (B \htensor \compl{ℂQ}) P (B \htensor \compl{ℂQ}) $ is contained in $ I(P) $ and that $ I(P) $ is contained in the closure of $ (B \htensor \compl{ℂQ}) P (B \htensor \compl{ℂQ}) $. It remains to show that $ I(P) $ is closed.

To prove $ I(P) $ closed, regard a series of $ (x_n) ⊂ I(P) $ converging in $ B \htensor \compl{ℂQ} $:
\begin{equation*}
x = \sum_{n = 0}^∞ x_n.
\end{equation*}
We now explain how to prove $ x ∈ I(P) $. It entails inspecting every $ x_n $ and dividing it into chunks in such a way it becomes evident that a reordering of the chunks sums up to an element of $ I(P) $. The reordering is unproblematic, keeping $ x $ intact as element of $ B \htensor \compl{ℂQ} $.

In order to define the chunks, we shall build a sequence $ (n_k)_{k ≥ 1} ⊂ ℕ $. Let $ k ≥ 1 $. Since we assume the series $ \sum x_n $ converges in the tensor topology, there exists for every $ n ∈ ℕ $ a (maximally chosen) $ N_n^{(k)} ∈ ℕ $ such that $ x_n ∈ B ℂQ_{≥N_n^{(k)}} + \mathfrak{m}^k \compl{ℂQ} $. We have that $ N_n^{(k)} → ∞ $ as $ n → ∞ $. In particular, we have $ l(\tilde l^{k-1} (N_n^{(k)})) → ∞ $ as $ n → ∞ $. Choose $ n_k ∈ ℕ $ so high that $ l(\tilde l^{k-1} (N_n^{(k)})) ≥ k $ for all $ n ≥ n_k $. Of course, we can enforce that $ n_k $ is an increasing sequence: $ n_k ≥ n_{k-1} $. We have now built the sequence $ (n_k)_{k ≥ 1} $ which is essential in the construction of the chunks.

Let now $ n ∈ ℕ $ and assume $ n ≥ n_1 $. Then there is a unique $ k ∈ ℕ $ such that $ n_k ≤ n < n_{k+1} $. We get
\begin{align*}
x_n &∈ I(P) ∩ [B ℂQ_{≥N_n^{(k)}} + \mathfrak{m}^k \compl{ℂQ}] \\
&⊂ \sum_{i = 0}^{k-1} \mathfrak{m}^i (ℂQ P ℂQ)_{≥l(\tilde l^i (N_n^{(k)}))} + \mathfrak{m}^k I(P) \\
&⊂ B (ℂQ P ℂQ)_{≥ k} + \mathfrak{m}^k I(P).
\end{align*}
In the second row, we have applied \autoref{th:flatness-completed-powerswitch}. In the third row, we have used that $ l(\tilde l^i (N_n^{(k)})) ≥ l(\tilde l^{k-1} (N_n^{(k)})) ≥ k $ minding $ n ≥ n_k $. With respect to the two summands, write now
\begin{equation*}
x_n = x_n^{(1)} + x_n^{(2)}, \quad x_n^{(1)} ∈ B (ℂQ P ℂQ)_{≥ k}, \quad x_n^{(2)} ∈ \mathfrak{m}^k I(P).
\end{equation*}
We can now write
\begin{align*}
x = \sum_{n = 0}^{n_1 - 1} x_n + \underbrace{\sum_{n = n_1}^{n_2 - 1} x_n^{(1)}}_{∈ B (ℂQ P ℂQ)_{≥ 1}} + \underbrace{\sum_{n = n_2}^{n_3 - 1} x_n^{(1)}}_{∈ B (ℂQ P ℂQ)_{≥2}} + …
 + \underbrace{\sum_{n = n_1}^{n_2 - 1} x_n^{(1)}}_{∈ \mathfrak{m}^1 I(P)} + \underbrace{\sum_{n = n_2}^{n_3 - 1} x_n^{(1)}}_{∈ \mathfrak{m}^2 I(P)} + … ~∈ I(P).
\end{align*}
The top row is a finite sum. The second row adds up to an element of $ I(P) $ since the path lengths increase. The third summand adds up to an element of $ I(P) $ since the powers of $ \mathfrak{m} $ increase. This finally shows $ x ∈ I(P) $ and we conclude $ I(P) $ is closed with respect to the tensor topology.
\end{proof}

We finish this section by commenting on deformations of algebras where the Krull topology is taken into account. Recall from \autoref{th:flatness-whatis-equivalence} that $ B \htensor A / I_q $ is a deformation of $ A / I $ if $ I_q $ is a quasi-flat deformation of $ I $. It would be delightful to have a version of this statement which takes into account the Krull topology in case $ A = \compl{ℂQ} $. This requires two steps: First we shall define what it means for a deformation to be tensor continuous. Second we shall prove that $ B \htensor \compl{ℂQ} / I_q $ is a tensor continuous deformation in case $ I_q $ is quasi-flat and closed with respect to the tensor topology. Let us start as follows:

\begin{definition}
Let $ I ⊂ \compl{ℂQ} $ be a closed ideal and $ I_q $ a deformation closed with respect to the tensor topology. Then $ (B \htensor \compl{ℂQ}) / I_q $ is a (tensor continuous) \emph{deformation} of $ \compl{ℂQ} / I $ if there exists a deformation $ μ_q $ of the product $ μ: (\compl{ℂQ} / I) ¤ (\compl{ℂQ} / I) → (\compl{ℂQ} / I) $ together with a $ B $-linear algebra isomorphism
\begin{equation*}
φ: \frac{B \htensor \compl{ℂQ}}{I_q} \isoto (B \htensor (\compl{ℂQ} / I), μ_q)
\end{equation*}
which is a homeomorphism with respect to the tensor topology.
\end{definition}

\begin{lemma}
\label{th:flatness-closedness-AlgVsIdeal}
Let $ Q $ be a quiver, $ I ⊂ \compl{ℂQ} $ a closed ideal which has a closed complement in $ \compl{ℂQ} $. Let $ I_q $ be a deformation closed with respect to the tensor topology. If $ I_q $ is quasi-flat, then $ (B \htensor \compl{ℂQ})/ I_q $ is a (tensor continuous) deformation of $ \compl{ℂQ} / I $.
\end{lemma}

\begin{proof}
The proof proceeds as in \autoref{th:flatness-whatis-equivalence}, with minor adaptions. The first step is to choose a complement $ V ⊂ \compl{ℂQ} $ such that $ \compl{ℂQ} = I ⊕ V $. In comparison with \autoref{th:flatness-whatis-equivalence}, we can choose $ V $ to be closed. As in \autoref{th:flatness-whatis-equivalence}, we obtain the direct sum decompositions
\begin{equation*}
\compl{ℂQ} = I ⊕ V, \quad B \htensor \compl{ℂQ} = I_q ⊕ BV.
\end{equation*}
The difference in our case is that the first is not only a direct sum of vector spaces and the second not only a direct sum of $ \mathfrak{m} $-adically closed subspaces. Instead, the direct summands of the first sum are closed with respect to the Krull topology and the summands of the second sum are closed with respect to the tensor topology. As in \autoref{th:flatness-whatis-equivalence}, we define the map
\begin{equation*}
φ: B \htensor \frac{\compl{ℂQ}}{I} = B \htensor \frac{V ⊕ I}{I} \isoto B \htensor V \isoto \frac{I_q ⊕ BV}{I_q} = \frac{B \htensor \compl{ℂQ}}{I_q}.
\end{equation*}
This map is immediately a homeomorphism with respect to the tensor topology. One then transfers the product of $ (B \htensor \compl{ℂQ}) / I_q $ onto a deformed product $ μ_q $ on $ B \htensor (\compl{ℂQ} / I) $. This finishes the proof.
\end{proof}

%% file: flatness/dimer_bounded.tex
\subsection{Dimers of bounded type}
\label{sec:flatness-dimers}
In this section, we introduce a boundedness condition of dimers. Our motivation comes from Jacobi algebras of dimers, which are a special case of Jacobi algebras of quivers with superpotential. It is our interest to show that a deformation of the superpotential leads to a flat deformation of the Jacobi algebra. Fortunately, Jacobi algebras of most dimers are CY3. In other words, a Jacobi algebra of a dimer falls under the framework developed in the past sections. The only question remaining is whether the superpotential is of bounded type.

In the present section, we show that the superpotential of a Jacobi algebra of a dimer is of bounded type in the following cases:
\begin{itemize}
\item All polygons in $ Q $ have equal number of edges (obvious),
\item $ Q $ is cancellation consistent and sits in a torus (\autoref{rem:dimer-bounded-torus}),
\item $ Q $ is cancellation consistent and has no triangles (\autoref{th:dimer-bounded-theorem}).
\end{itemize}
To start our investigation, let us recall the construction of this Jacobi algebra and the meaning of consistency. Let $ Q $ be a dimer. Then its superpotential $ W ∈ ℂQ $ is given by the difference of the clockwise polygons of $ Q $ and the counterclockwise polygons, cyclically permuted:
\begin{equation*}
W = \sum_{\substack{a_1, …, a_k \\ \text{clockwise}}} (a_1 … a_k)_{\cyc} ~- \sum_{\substack{a_1, …, a_k \\ \text{counterclockwise}}} (a_1 … a_k)_{\cyc}.
\end{equation*}
The relations $ ∂_a W $ equate two neighboring polygons: Flipping a path over an arc $ a $ is possible if the path follows all arcs of a neighboring polygon apart from $ a $. These flip moves are known as \emph{F-term} moves and the equivalence relation on the set of paths in $ Q $ is known as F-term equivalence. The terminology is depicted in \autoref{fig:flatness-dimers-Fterm}. A good reference is \cite{Davison}.

\input{flatness/fig_dimer_bounded.tex}

Regard the set of paths in $ Q $ modulo F-term equivalence. The set contains a special element $ ℓ_v $ for each vertex $ v ∈ Q_0 $, given by the boundary of a chosen polygon at $ v $. All boundaries of polygons incident at $ v $ are F-term equivalent, hence $ ℓ_v $ does not depend on the choice. In other words, it can be rotated around $ v $. We may drop the subscript from $ ℓ_v $ if it is clear from the context. The element $ ℓ $ commutes with all paths, that is, $ uℓ \sim ℓu $. Davison \cite{Davison} introduced the following consistency condition for dimers:

\begin{definition}[{\cite{Davison}}]
A dimer $ Q $ is \emph{cancellation consistent} if it has the following cancellation property:
\begin{equation*}
p ℓ \sim q ℓ ~\Longrightarrow~ p \sim q.
\end{equation*}
\end{definition}

\begin{remark}
An equivalent requirement is $ pr \sim qr ⇒ p \sim q $ for all compatible paths $ p, q, r $. Indeed assume $ pr \sim qr $. Then pick an “inverse” path $ r' $ in $ Q $ such that $ rr' \sim ℓ^k $ for some $ k ∈ ℕ $, and observe $ p ℓ^k \sim prr' \sim qrr' \sim q ℓ^k $, hence $ p \sim q $.
\end{remark}

The Jacobi algebra of $ Q $ is given by $ \Jac(Q) = ℂQ / (∂_a W) $. It has a special central element $ ℓ ∈ \Jac(Q) $, given by the sum of the $ ℓ_v $ at all vertices $ v ∈ Q_0 $:
\begin{equation*}
ℓ = \sum_{v ∈ Q_0} ℓ_v ∈ \Jac(Q),
\end{equation*}
The set of paths in $ Q $ modulo F-term equivalence is a basis for $ \Jac(Q) $. Davison shows in \cite[Theorem 8.1]{Davison} that $ \Jac(Q) $ is CY3 if $ Q $ is cancellation consistent. A complete survey of consistency conditions can be found in \cite{Bocklandt-consistency}. For example, zigzag consistency implies cancellation consistency.

The previous sections of this paper establish that CY3 algebras with relations of bounded type have favorable deformation theory: Whenever we deform their superpotential in a cyclic way, the deformation is flat. We want to apply this to Jacobi algebras of dimers. Since $ \Jac(Q) $ is already known to be CY3 if $ Q $ is cancellation consistent, it remains to ask if the relations $ ∂_a W $ are of bounded type. This is the case for some dimers, but not for all. In the present section, we provide criteria for this to happen. Let us give this notion a name.

\begin{definition}
\label{def:flatness-dimers-bounded}
A dimer $ Q $ is of \emph{bounded type} if all F-term equivalence classes are bounded in path length. Equivalently, all F-term equivalence classes are finite sets.
\end{definition}

In this new terminology, the Jacobi algebra of a dimer $ Q $ has favorable deformation theory if $ Q $ if cancellation consistent and of bounded type. Which dimers are of bounded type? An easy example are the dimers whose polygons are all of equal length. Indeed, in such case an F-term move preserves length. A more intricate case is the following:

\begin{lemma}
\label{rem:dimer-bounded-torus}
Let $ Q $ be a cancellation consistent torus dimer. Then $ Q $ is of bounded type.
\end{lemma}

\begin{proof}
In \cite[Theorem 7.6]{Bocklandt-consistency} it is shown that a cancellation consistent torus dimer has a so-called “consistent R-grading”, or “anomaly-free R-symmetry”. Broomhead shows in \cite[Section 2.3]{Broomhead} that the existence of an anomaly-free R-symmetry implies that for every arc there exists a perfect matching $ P: Q_1 → \{0, 1\} $ that is positive on that arc. Summing up all the perfect matching gradings, we obtain a positive integer grading $ d $ on $ ℂQ $ in which $ W $ is homogeneous. The degree $ d $ of a path is then preserved under F-term equivalence. Since length of a path $ p $ is bounded by its total degree $ d(p) $, we conclude that any cancellation consistent dimer on a torus is of bounded type.
\end{proof}

In this section, we give another wide class of cancellation consistent dimers of bounded type: those where all polygons are of length at least 4. The core observation is that whenever a path flips over a zigzag path for the first time, it includes a cycle $ ℓ $ right there, see \autoref{fig:dimer-bounded-zigzag-flip}. In short, this gives rise to the following line of proof: It suffices to regard paths equivalent to $ ℓ^k $ and proceed by induction over $ k ∈ ℕ $. Once F-term moves bring a path $ p \sim ℓ^k $ to cross a zigzag path for the first time, it includes a cycle $ ℓ $ right there. Stripping away $ ℓ $ from the path makes it equivalent to $ ℓ^{k-1} $ and by induction hypothesis such a path is bounded in length. In other words, zigzag paths at sufficient distance provide a “cage” for F-term equivalence classes, see \autoref{fig:dimer-bounded-zigzag-cage}.

The problem with this “cage proof” is that construction of an effective cage requires topological arguments and geometric consistency. We derive a more refined proof, focusing on the crossings between $ p $ and individual zigzag paths.

\input{flatness/fig_zigzag_cage.tex}

\begin{definition}
A \emph{crossing} of $ p $ over a zigzag path $ Z $ is a sequence of $ k ≥ 1 $ consecutive arcs in $ p $ that follow $ Z $, such that $ p $ leaves $ Z $ to the left or right before the sequence, and the right resp.~left after the sequence (\autoref{fig:dimer-bounded-zigzag-crossing}).
\end{definition}

Paths containing a full cycle $ ℓ $ around a polygon are easy to bound in length by induction. We therefore regard mainly paths that are $ ℓ $-free, that is, do not contain a full cycle $ ℓ $ around some polygon. In other words, a path is \emph{$ ℓ $-free} if it is not of the form $ pℓq $. Note this is not the same as being a minimal path, since minimality refers to F-term equivalence: A path is minimal if it is not F-term equivalent to a path of the form $ ℓq $. We are now ready to prove that crossings with zigzag paths are a partial invariant.

\begin{lemma}
\label{th:dimer-bounded-crossings-preserved}
Let $ Q $ be a dimer without triangles and let $ Z $ be a zigzag path. Let $ p $ and $ q $ be two closed $ ℓ $-free paths differing only by an F-term move. Then $ p $ and $ q $ have the same number of crossings with $ Z $.
\end{lemma}

\begin{proof}
The strategy is to inspect the crossings of $ p $ and $ q $ over $ Z $, and match them up. It is essential that $ Q $ has no triangles, because triangles bordering $ Z $ make it possible to create new crossings, see \autoref{fig:dimer-bounded-triangles}.

Let us inspect a crossing of $ p $ over $ Z $. Without loss of generality we can assume that $ p $ turns right at the end of the sequence and left at the beginning. We show that the crossing is preserved when $ p $ flips to $ q $. We scrutinize this by a case distinction on where the flip happens. Recall that a flip always involves precisely one polygon minus an arc.

Regard the sequence $ a, b_1, …, b_k, c $ of arcs on $ p $ crossing over $ Z $. Regard the case that the arcs involved in the flip are all before $ a $. Then $ a, b_1, …, b_k, c $ stay entirely part of the path and the crossing is preserved.

Regard the case that $ k ≥ 2 $. Then no three consecutive $ b_i $ arcs can be involved in the flip, because they follow a zigzag and do not circle around a polygon. Two consecutive $ b_i $ arcs are not enough for a flip, since all polygons are assumed to consist of at least 4 arcs. Hence arcs before $ a $, $ a $ itself, $ b_1 $ and $ b_2 $ remain as possible arcs involved in the flip. In all cases we check that a crossing at the same point still exists in $ q $. The case where arcs $ b_{k-1} $, $ b_k $, $ c $, $ … $ on the other side of the crossing are involved follows similarly.

We distinguish 4 cases, depicted in \autoref{fig:dimer-bounded-crossing-cases}. In cases 1–3, only arcs before $ a $ and $ a $ itself are involved in the flipping. Case 1 depicts the situation where $ a $ lies maximally left, case 2 depicts an average situation, case 3 depicts the situation where $ a $ lies maximally right. Due to arrow directions, case 1 and 3 differ in appearance.

It turns out in case 1 that there is a new arc, indicated by a checkmark in the figure, that leaves $ Z $ and everything between that arc and $ b_1 $ follows $ Z $. In other words $ q $ still has a crossing at the same location. In case 2, the arc leaving $ Z $ also changes through the flip, but the crossing as a whole remains. Case 3 is actually impossible: By assumption $ a $ leaves $ Z $, and in order to conclude a polygon $ P $ minus an arc, $ p $ needs to continue turning around $ P $. It ends precisely at the head of $ b_1 $, concluding an $ ℓ $-cycle $ …, a, b_1 $.

Case 4 is the situation where arcs before $ a $, the arc $ a $ itself and $ b_1 $ are involved in the flipping. Since $ a $ and $ b_1 $ are supposed to be part of the flipping, the polygon to be flipped is necessarily the one lying in the corner of the zigzag path at $ b_1 $, $ b_2 $. The first arc of $ p $ involved in the flip starts at the head of $ b_2 $ (or the corresponding vertex of $ Z $ in case $ k = 1 $). What comes before that arc in $ p $? The arc $ b_2 $ following $ b_1 $ on $ Z $, which we also label this way by abuse if $ k = 1 $, cannot come before it, because $ p $ is $ ℓ $-free. By arrow directions, it can also not concern the arc $ b_3 $, similarly labeled by abuse if $ k ≤ 2 $. Hence it must concern an arc that turns left of the zigzag path. In the figure this arc is depicted again by a checkmark. This demonstrates that also in case 4 the crossing is preserved.

Finally it is also easy to see that the crossing is preserved in case where paths before $ a $, $ a $ itself and $ b_1 $ and $ b_2 $ are involved in the flip. Moreover, no flip is possible that includes $ a $, $ b_1 $ and $ c $ if $ k = 1 $.

Let us scrutinize the conclusion. We have associated to each crossing of $ p $ over $ Z $, let us call it $ χ $, a crossing $ φ(χ) $ of $ q $ over $ Z $. Is this map $ χ ↦ φ(χ) $ a one-to-one correspondence? Swapping the roles of $ p $ and $ q $, we also have a map $ ψ $ from crossings of $ q $ over $ Z $ to crossings of $ p $ over $ Z $. Inspecting the case distinction \autoref{fig:dimer-bounded-crossing-cases} again, it becomes apparent that $ ψ $ has no other choice than associating to $ φ(χ) $ back $ χ $ again. For example, a crossing $ χ $ and its image $ φ(χ) $ always have an arc on $ Z $ in common, and similarly $ χ $ and $ ψ(χ) $ do. In other words $ φ ∘ ψ = \Id $ and similarly $ ψ ∘ φ = \Id $. We conclude that $ p $ and $ q $ have the same number of crossings over $ Z $.
\end{proof}

\input{flatness/fig_zigzag_crossing.tex}

Given a path $ p $, recall our plan is to utilize the crossings of $ p $ over arbitrary zigzag paths as a partial invariant to bound the length of $ p $. As announced, we do not construct an explicit cage, but rather argue as follows: The only way to avoid crossing zigzag paths is to follow the boundary of a polygon. If we assume $ p $ is $ ℓ $-free, then following the boundary of a polygon is possible for at most $ K $ consecutive arcs, where $ K $ is the maximum length of polygons in $ Q $. We conclude that $ p $ necessarily crosses a zigzag path at least once every $ K $ arcs. Let us make this precise.

\begin{lemma}
\label{th:dimer-bounded-observation}
Let $ Q $ be a cancellation consistent dimer. If $ p $ is an $ ℓ $-free path having $ C $ crossings with zigzag paths, then its length is bounded by $ K(C+1) $.
\end{lemma}

\begin{proof}
The strategy is to show that $ p $ necessarily crosses a zigzag path at least once every $ K $ arcs, and then apply \autoref{th:dimer-bounded-crossings-preserved}.

We claim that a path $ cba $ of length 3 either crosses a zigzag path at $ b $ or is part of a polygon. To check this, we take on the perspective of $ b $, allowing us to find words for where $ a $ and $ c $ turn at head and tail of $ b $. The generic case is when $ a $ and $ c $ are neither left-most nor right-most. Then, the path $ cba $ crosses both zigzag paths starting at $ b $.

Let us treat the special cases. If $ c $ is the left-most (resp.~right-most) at the head of $ b $ and $ a $ is the right-most (resp.~left-most) at the tail of $ a $, then $ cba $ is part of a counterclockwise (resp.~clockwise) polygon. If $ c $ is the left-most (resp.~right-most), but $ a $ is not the right-most (resp.~left-most), then the zigzag path starting at $ b $ and turning right (resp.~left) crosses $ p $ at $ b $. Similarly if $ a $ is the right-most (resp.~left-most), but $ c $ is not the left-most (resp.~right-most), then the zigzag path starting at $ b $ and turning left (resp.~right) crosses $ p $ at $ b $.

Either way, we conclude a path $ cba $ of length 3 either crosses a zigzag path at $ b $ or is part of a polygon. Now regard a path longer than 3 arrows. How many consecutive arcs are possible without crossing a zigzag path? By consistency, there are at least four polygons incident at every vertex. Hence if $ dcb $ lies in a polygon and $ cba $ lies in a polygon, then both lie in the same polygon. We conclude that after $ K $ arcs, a path has either completed a cycle around a polygon or crossed a zigzag path. If a path $ p $ contains no cycle at all, then it has crossed at least $ \lfloor |p| / K \rfloor $ many zigzag paths. Reading this inequality the other way around gives the desired bound.
\end{proof}

We recall some notions, before diving into the proof. Let $ p $ be a path in $ Q $. By consistency, $ p $ is equivalent to a composition $ ℓ^k q $ of a cycle power $ ℓ^k $ and a minimal path $ q $. That is, $ q $ cannot be written as a multiple of $ ℓ $. This decomposition $ p = ℓ^k q $ is unique up to equivalence of $ q $. Let us call $ k $ the “looseness” of $ p $.

\begin{remark}
\label{th:dimer-bounded-cycle-splitoff}
If $ p = p_2 ℓ p_1 $ is a path containing a cycle, then $ p_2 ℓ p_1 \sim p_2 p_1 ℓ $. If $ Q $ is cancellation consistent, it satisfies the cancellation condition: If $ p $ is equivalent to $ ℓ^k $, then $ p_2 p_1 $ is equivalent to $ ℓ^{k-1} $. Once established that paths equivalent to $ ℓ^{k-1} $ have bounded length, then $ p_2 p_1 $ and hence $ p_2 ℓ p_1 $ are also bounded.
\end{remark}

\begin{lemma}
\label{th:bounded-dimer-cycle}
Let $ Q $ be a cancellation consistent dimer without triangles. For any vertex $ v $ and integer $ k $, the paths equivalent to $ ℓ^k $ at $ v $ are of bounded length.
\end{lemma}

\begin{proof}
We have a partial invariant at hand: the number of times a given pass crosses a zigzag path. This number is not preserved under F-term equivalence in general, but becomes an invariant once we restrict to closed $ ℓ $-free paths.

We proceed by induction. Assume all closed paths equivalent to $ ℓ^{k-1} $ starting at vertex $ v $ have length $ ≤ N $. Let $ p $ be a path equivalent to $ ℓ^k $. If $ p $ contains a cycle $ ℓ $, then we can bound $ |p| ≤ N + K $ by \autoref{th:dimer-bounded-cycle-splitoff} and we are done. Therefore we can assume $ p $ is $ ℓ $-free.

Pick a sequence $ ℓ^k = p_1, …, p_n = p $ of paths, each related to its successor by an F-term move. Let $ m < n $ be the maximal number where $ p_m $ still contains a cycle $ ℓ $. Then $ |p_m| ≤ N + K $ by the induction hypothesis. An F-move changes length by at most $ K $, hence $ |p_{m+1}| ≤ N + 2K $. By \autoref{th:dimer-bounded-crossings-preserved}, the path $ p_{m+1} $ has the same total number of crossings with zigzag paths as $ p_{m+2}, …, p_n $ do.

Since $ |p_{m+1}| ≤ N + 2K $, the path $ p_{m+1} $ has at most $ 2N + 4K $ crossings with zigzag paths. We have seen this number stays constant and hence also $ p_n $ has at most $ 2N + 4K $ crossings with zigzag paths. By \autoref{th:dimer-bounded-observation} its length is bounded by $ K (2N + 4K + 1) $. This finishes the induction.
\end{proof}

\begin{theorem}
\label{th:dimer-bounded-theorem}
Any cancellation consistent dimer without triangles is of bounded type.
\end{theorem}

\begin{proof}
\autoref{th:bounded-dimer-cycle} already establishes the claim for the paths $ p = ℓ^k $. We deduce from this the general case where $ p $ is an arbitrary path. Fix some path $ p' $ from the end of $ p $ to the start of $ p $, such that $ p'p $ is contractible. Then $ p'p \sim ℓ^{k_0} $ for some $ k_0 $. Now let $ q $ be an arbitrary path equivalent to $ p $. We get
\begin{equation*}
q \sim p \quad \Longrightarrow \quad p' q \sim p' p \sim ℓ^{k_0}.
\end{equation*}
By \autoref{th:bounded-dimer-cycle}, the length of $ p' q $ is bounded. In particular, the length of $ q $ is bounded.
\end{proof}

\begin{remark}
The bound of \autoref{th:bounded-dimer-cycle} is exponential in $ k $:
\begin{equation*}
q \sim ℓ^k ~\Longrightarrow~ |q| ≤ \mathcal{O} ((2K)^k).
\end{equation*}
The bound is also exponential if we fix $ p $ and regard paths $ q \sim p ℓ^k $. Indeed, let $ p $ and $ p' $ be fixed paths with $ p' p \sim ℓ^{k_0} $, then
\begin{equation*}
q \sim ℓ^k p ~\Longrightarrow~ |q| ≤ |p'q| ≤ \mathcal{O} ((2K)^k),
\end{equation*}
since $ p'q \sim ℓ^{k+k_0} $. These bounds are far from sharp. Regard for example a relatively straight dimer like the one in \autoref{fig:dimer-bounded-equivalent-paths}. This figure convinces us that the expected bound is actually linear in $ k $.
\end{remark}

%% file: flatness/fig_dimer_bounded.tex
\begin{figure}
\centering
\begin{subfigure}{0.25\linewidth}
\centering
\begin{tikzpicture}
\path[draw, ->, semithick, rounded corners] (0, 0) -- ++(left:1) coordinate (right) -- ++(135:1) -- ++(225:1) coordinate (left) -- ++(left:1);
\path[draw, semithick, rounded corners] (0, 0) -- ++(left:1) -- ++(down:0.7) -- ($ (left) + (0, -0.7) $) -- ++(up:0.7) -- ++(left:1);
\path[draw, ->, gray] (left) ++(0.1, 0) -- ($ (right) + (-0.1, 0) $);
\path[draw, ->] ($ (left)!0.5!(right) $) arc(0:20:1);
\path[draw, ->] ($ (left)!0.5!(right) $) arc(0:-20:1);
\end{tikzpicture}
\caption{An F-term flip}
\end{subfigure}
\begin{subfigure}{0.35\linewidth}
\centering
\begin{tikzpicture}
\foreach \y in {0, 1, 2}
{\foreach \x in {0, 1, 2, 3, 4}
\path[draw, ->, gray] (\x, \y+0.1) to (\x, \y+0.9);
\foreach \x in {0, 1, 2, 3}
\path[draw, ->, gray] (\x+0.9, \y+0.9) to (\x+0.1, \y+0.1);
\foreach \x in {0, 1, 2, 3}
\path[draw, ->, gray] (\x+0.1, \y) to (\x+0.9, \y);}
\foreach \x in {0, 1, 2, 3}
\path[draw, ->, gray] (\x+0.1, 3) to (\x+0.9, 3);
\path[draw, ->, semithick, rounded corners] (0, 3.3) -- (3.4, 3.3) -- (0.3, 0);
\path[draw, ->, semithick, rounded corners] (0, 3.2) -- (2.3, 3.2) -- (0.3, 1.2) -- (1.2, 1.2) -- (0.2, 0);
\path[draw, ->, semithick, rounded corners] (0, 3.1) -- (1.1, 3.1) -- (0.1, 2.1) -- (1.1, 2.1) -- (0.1, 1.1) -- (1, 1.1) -- (0.1, 0);
\end{tikzpicture}
\caption{These 3 paths are equivalent.}
\label{fig:dimer-bounded-equivalent-paths}
\end{subfigure}
\begin{subfigure}{0.25\linewidth}
\centering
\begin{tikzpicture}
\path[draw, ->, bend right] (0, 0) to (0, 1);
\path[draw, ->, looseness=1.5, bend right=40] (0.1, -0.1) to (0.1, 1.1);
\path[draw, ->, looseness=2, bend right=60] (0.1, -0.2) to (0.1, 1.2);
\path[draw, ->, looseness=2.5, bend right=80] (0.2, -0.3) to coordinate[midway] (outer) (0.2, 1.3);
\path[fill] (0, -0.1) circle[radius=0.05];
\path[fill] (0, 1.1) circle[radius=0.05];
\path (outer) node[right] {… ?};
\end{tikzpicture}
\caption{Longer and longer?}
\end{subfigure}
\caption{Intuition on F-term equivalence}
\label{fig:flatness-dimers-Fterm}
\end{figure}
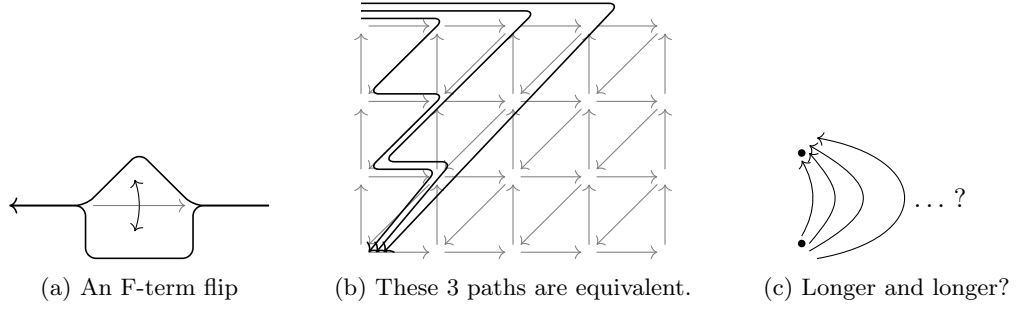

%% file: flatness/fig_zigzag_cage.tex
\begin{figure}
\centering
\begin{subfigure}{0.3\linewidth}
\centering
\begin{tikzpicture}
\path[draw, ->, gray] (0, 0) -- ++(225:1) -- ++(135:1) -- ++(225:1) coordinate (right) -- ++(135:1) -- ++(225:1) coordinate[midway] (flip) coordinate (left) -- ++(135:1);
\path[draw, ->, semithick, rounded corners] (left) ++(down:0.1) -- ++(300:0.5) -- ++(right:0.8) node[midway, above, shift={(0, 0.1)}] {$ ℓ $} -- ++(60:0.5) -- ++(135:0.9) coordinate (end);
\path[draw, semithick, dashed] (end) -- ++(225:0.8);
\path[draw, ->] (flip) arc(45:60:1);
\path[draw] (flip) arc(45:30:1);
\end{tikzpicture}
\caption{Initial flip always has an $ ℓ $}
\label{fig:dimer-bounded-zigzag-flip}
\end{subfigure}
\begin{subfigure}{0.3\linewidth}
\centering
\begin{tikzpicture}
\path[draw] (0, 0) -- ++(315:0.3) -- ++(45:0.3) -- ++(315:0.3) -- ++(45:0.3) -- ++(315:0.3) -- ++(45:0.3) -- ++(315:0.3) -- ++(45:0.3) -- ++(315:0.3) -- ++(45:0.3) -- ++(315:0.3) -- ++(45:0.3) -- ++(315:0.3) -- ++(45:0.3) coordinate (end1);
\path[draw] (end1) ++(0.1, -0.1) -- ++(225:0.3) -- ++(315:0.3) -- ++(225:0.3) -- ++(315:0.3) -- ++(225:0.3) -- ++(315:0.3) -- ++(225:0.3) -- ++(315:0.3) coordinate (end2);
\path[draw] (end2) ++(-0.1, -0.1) -- ++(135:0.3) -- ++(225:0.3) -- ++(135:0.3) -- ++(225:0.3) -- ++(135:0.3) -- ++(225:0.3) -- ++(135:0.3) -- ++(225:0.3) -- ++(135:0.3) -- ++(225:0.3) -- ++(135:0.3) -- ++(225:0.3) -- ++(135:0.3) -- ++(225:0.3) coordinate (end3);
\path[draw] (end3) ++(-0.1, 0.1) -- ++(45:0.3) -- ++(135:0.3) -- ++(45:0.3) -- ++(135:0.3) -- ++(45:0.3) -- ++(135:0.3) -- ++(45:0.3) -- ++(135:0.3);
\path[draw, ->, semithick, rounded corners] (1, -1) to[bend right=45, looseness=1] (1.5, -1.5) to[bend left=60, looseness=1.5] (2, -1.5) to (2.5, -1.25);
\path[draw, ->, semithick, rounded corners] (1, -0.9) to[bend left=60, looseness=1.5] (2.5, -1);
\path[fill] (1, -1) circle[radius=0.05];
\path[fill] (2.5, -1.15) circle[radius=0.05];
\end{tikzpicture}
\caption{Zigzag paths act as cage}
\label{fig:dimer-bounded-zigzag-cage}
\end{subfigure}
\begin{subfigure}{0.25\linewidth}
\centering
\begin{tikzpicture}[scale=0.7]
\path[draw, gray] (0, 0) -- ++(40:1.5) coordinate (left-top) -- ++(320:1.5) -- ++(40:1.5) coordinate (top);
\path[draw, semithick, gray] (top) -- ++(320:1.5) coordinate (bot);
\path[draw, ->, gray] (bot) -- ++(40:1.5);
\path (bot) -- ++(270:1) coordinate (down);
\path[draw, rounded corners, thick, ->] (left-top) ++(0, 0.1) ++(90:1) -- ++(270:1) node[midway, right] {$ a $} -- ++(320:1.4) node[near end, above] {$ b_1 $} -- ($ (top) + (0, 0.1) $) node[above] {…} -- ($ (bot) + (0.1, 0.1) $) node[near end, above] {$ b_k $} -- ($ (down) + (0.1, 0) $) node[midway, left] {$ c $};
\end{tikzpicture}
\caption{Crossing a zigzag path}
\label{fig:dimer-bounded-zigzag-crossing}
\end{subfigure}
\caption{Zigzag intuition}
\end{figure}
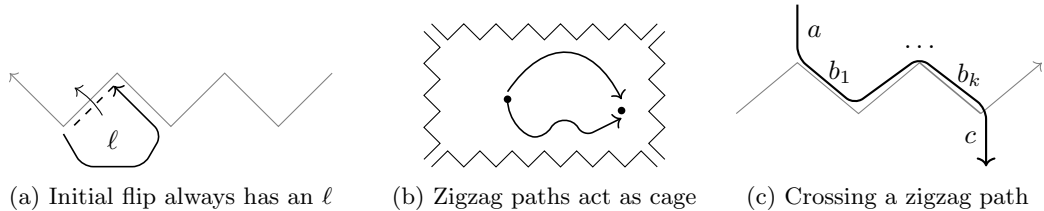

%% file: flatness/fig_zigzag_crossing.tex
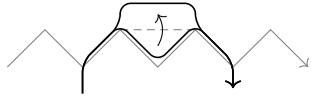
\begin{figure}
\centering
\begin{tikzpicture}[scale=0.7]
\path[draw, gray, ->] (0, 0) -- ++(45:1) -- ++(315:1) coordinate (1) -- ++(45:1) coordinate (2) -- ++(315:1) coordinate (3) -- ++(45:1) coordinate (4) -- ++(315:1) coordinate (5) -- ++(45:1) -- ++(315:1);
\path[draw, densely dashed, gray] (2) -- (4);
\path[draw, ->, semithick, rounded corners] (1) ++(down:0.5) -- ($ (1) + (0, 0.1) $) -- ($ (2) + (0, 0.1) $) -- ($ (2) + (0, 0.5) $) -- ($ (4) + (0, 0.5) $) -- ($ (4) + (0, 0.1) $) -- ($ (5) + (0, 0.1) $) -- ++(down:0.5);
\path[draw, ->, semithick, rounded corners] (1) ++(down:0.5) -- ($ (1) + (0, 0.1) $) -- ($ (2) + (0, 0.1) $) -- ($ (3) + (0, 0.1) $) -- ($ (4) + (0, 0.1) $) -- ($ (5) + (0, 0.1) $) -- ++(down:0.5);
\path[draw, ->, bend right] ($ (3) + (0, 0.4) $) to ($ (3) + (0, 1) $);
\end{tikzpicture}
\caption{Triangles ruin our invariant.}
\label{fig:dimer-bounded-triangles}
\end{figure}

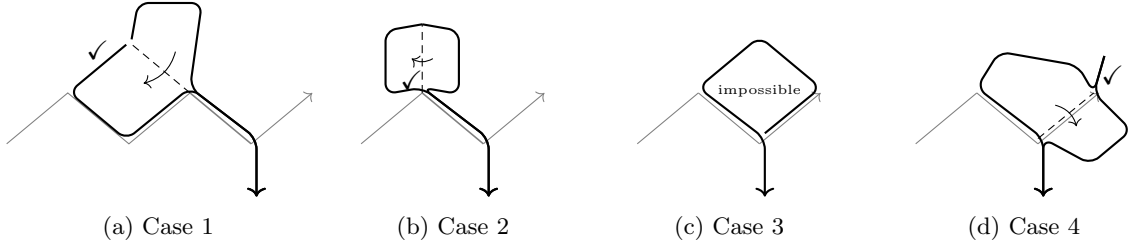
\begin{figure}
\centering
\begin{subfigure}{0.25\linewidth}
\centering
\begin{tikzpicture}[scale=0.7]
\path[draw, gray] (0, 0) -- ++(40:1.5) -- ++(320:1.5) -- ++(40:1.5) coordinate (top);
\path[draw, semithick, gray] (top) -- ++(320:1.5) coordinate (bot);
\path[draw, ->, gray] (bot) -- ++(40:1.5);
\path (bot) -- ++(270:1) coordinate (down);
\path[draw, densely dashed] (top) -- ++(140:1.5) coordinate (sup) coordinate[midway] (flip);
\path[draw, rounded corners, thick, ->] ($ (sup) + (-0.05, -0.05) $) -- ++(220:1.4) node[midway, above] {\checkmark} -- ++(320:1.4) -- ($ (top) + (0, 0.1) $) -- ($ (bot) + (0.1, 0.1) $) -- ($ (down) + (0.1, 0) $);
\path[draw, rounded corners, thick, ->] ($ (sup) + (0.05, 0.05) $) -- ++(80:0.7) -- ++(right:1.2) -- ($ (top) + (0, 0.1) $)  -- ($ (bot) + (0.1, 0.1) $) -- ($ (down) + (0.1, 0) $);
\path[draw, ->, bend left] ($ (flip) + (0.3, 0.3) $) to ($ (flip) + (-0.3, -0.3) $);
\end{tikzpicture}
\caption{Case 1}
\end{subfigure}
\begin{subfigure}{0.22\linewidth}
\centering
\begin{tikzpicture}[scale=0.7]
\path[draw, gray] (0, 0) -- ++(40:1.5) coordinate (top);
\path[draw, semithick, gray] (top) -- ++(320:1.5) coordinate (bot);
\path[draw, ->, gray] (bot) -- ++(40:1.5);
\path (bot) -- ++(270:1) coordinate (down);
\path[draw, densely dashed] (top) -- ++(90:1.3) coordinate (sup) coordinate[midway] (flip);
\path[draw, rounded corners, thick, ->] (sup) -- ++(190:0.7) -- ++(down:1.2) -- ($ (top) + (0, 0.1) $) node[midway, above, shift={(0.1, -0.1)}] {\checkmark} -- ($ (bot) + (0.1, 0.1) $) -- ($ (down) + (0.1, 0) $);
\path[draw, rounded corners, thick, ->] (sup) -- ++(350:0.7) -- ++(down:1.2) -- ($ (top) + (0, 0.1) $)  -- ($ (bot) + (0.1, 0.1) $) -- ($ (down) + (0.1, 0) $);
\path[draw, ->, bend left] ($ (flip) + (0.2, 0) $) to ($ (flip) + (-0.2, 0) $);
\end{tikzpicture}
\caption{Case 2}
\end{subfigure}
\begin{subfigure}{0.22\linewidth}
\centering
\begin{tikzpicture}[scale=0.7]
\path[draw, gray] (0, 0) -- ++(40:1.5) coordinate (top);
\path[draw, gray] (top) -- ++(320:1.5) coordinate (bot);
\path[draw, ->, gray] (bot) -- ++(40:1.5) coordinate (end);
\path (bot) -- ++(270:1) coordinate (down);
\path[draw, rounded corners, thick, ->] ($ (bot) + (0.1, 0.2) $) -- ($ (end) + (0, 0.1) $) -- ++(140:1.5) coordinate (high) -- ($ (top) + (0, 0.1) $) -- ($ (bot) + (0.1, 0.1) $) -- ($ (down) + (0.1, 0) $);
\path ($ (bot)!0.5!(high) $) node {\tiny impossible};
\end{tikzpicture}
\caption{Case 3}
\end{subfigure}
\begin{subfigure}{0.25\linewidth}
\centering
\begin{tikzpicture}[scale=0.7]
\path[draw, gray] (0, 0) -- ++(40:1.5) coordinate (top);
\path[draw, gray] (top) -- ++(320:1.5) coordinate (bot);
\path[draw, ->, gray] (bot) -- ++(40:1.5) coordinate (end);
\path (bot) -- ++(270:1) coordinate (down);
\path[draw, densely dashed] ($ (bot) + (0, 0.1) $) -- ($ (end) + (-0.1, 0) $) coordinate[midway] (flip);
\path[draw, ->, rounded corners, thick] ($ (end) + (0.1, 0.7) $) -- ($ (end) + (-0.1, 0) $) node[midway, right, shift={(-0.1, 0)}] {\checkmark} -- ++(120:0.6) -- ++(170:1.5) -- ($ (top) + (0, 0.1) $) -- ($ (bot) + (0.1, 0.1) $) -- ($ (down) + (0.1, 0) $);
\path[draw, ->, rounded corners, thick] ($ (end) + (0.1, 0.7) $) -- ($ (end) + (-0.1, 0) $) -- ++(315:1) -- ++(220:1) -- ($ (bot) + (0.1, 0.1) $) -- ($ (down) + (0.1, 0) $);
\path[draw, ->, bend left] ($ (flip) + (-0.2, 0.2) $) to ($ (flip) + (0.2, -0.2) $);
\end{tikzpicture}
\caption{Case 4}
\end{subfigure}
\caption{F-term moves of $ ℓ $-free paths preserve zigzag crossings.}
\label{fig:dimer-bounded-crossing-cases}
\end{figure}

%% file: flatness/main.tex
\subsection{Main theorems on flatness}
\label{sec:flatness-flatness}
In this section, we collect our main theorems on flatness. In particular, we return to the case where the relations come from a superpotential. For the statement of our theorem, the deformed relations are supposed to come from a deformation of the superpotential and the algebra is supposed to be CY3.

Let us state our flatness result first in the most general way, taking the setup from \autoref{conv:flatness-relations}.

\begin{remark}
Recall that $ I(R) $ is the ideal generated by $ R $ in $ ℂQ $. The spaces $ I(R)_{\compl{ℂQ}} $, $ I(P)_{ℂQ} $ and $ I(P) $ are a bit more complicated. We defined them in an intricate way in \autoref{sec:flatness-ideals}. Under the assumptions of \autoref{th:flatness-flatness-general}, the definitions however simplify: The space $ I(P) ⊂ B \htensor \compl{ℂQ} $ is quasi-flat by \autoref{th:flatness-flatcompleted-th} and $ I(P)_{ℂQ} ⊂ B \htensor ℂQ $ is quasi-flat if $ ψ $ only maps to $ \mathfrak{m} ℂQ $ by \autoref{th:flatness-flatalgebraic-th}.

In simplified terms, the space $ I(R)_{\compl{ℂQ}} $ is the closure of the ideal generated by $ R $ in $ \compl{ℂQ} $ by \autoref{th:flatness-closedness-IRclosed}. The space $ I(P)_{ℂQ} $ is the closure of the ideal generated by $ P $ in $ B \htensor ℂQ $ if $ ψ $ only maps to $ \mathfrak{m} ℂQ $, since $ I(P)_{ℂQ} $ is quasi-flat and hence closed. The space $ I(P) $ is the closure of the ideal generated by $ P $ in $ B \htensor \compl{ℂQ} $ with respect to the tensor topology by \autoref{th:flatness-completed-IPclosed}. Written out, we have
\begin{align*}
I(R) &= ℂQ R ℂQ ⊂ ℂQ, \\
I(R)_{\compl{ℂQ}} &= \closure{ℂQ R ℂQ} ⊂ \compl{ℂQ}, \\
I(P)_{ℂQ} &= \closure{(B \htensor ℂQ) P (B \htensor ℂQ)} ⊂ B \htensor ℂQ, \quad \text{if } ψ(R) ⊂ \mathfrak{m} ℂQ, \\
I(P) &= \closure{(B \htensor \compl{ℂQ})P(B \htensor \compl{ℂQ})}^{\tensor} ⊂ B \htensor \compl{ℂQ}.
\end{align*}
\end{remark}

With these preparations, we are ready to state our flatness result in the most general way.

\begin{proposition}
\label{th:flatness-flatness-general}
Under \autoref{conv:flatness-relations}, assume $ R $ is of bounded type and [BG] and [CP] hold. Then we have:
\begin{itemize}
\item $ \frac{B \htensor \compl{ℂQ}}{I(P)} $ is a (tensor continuous) deformation of $ \frac{\compl{ℂQ}}{I(R)_{\compl{ℂQ}}} $.
\item $ \frac{B \htensor ℂQ}{I(P)_{ℂQ}} $ is a deformation of $ \frac{ℂQ}{I(R)} $ if $ ψ $ only maps to $ \mathfrak{m} ℂQ $.
\end{itemize}
\end{proposition}

\begin{proof}
This is a culmination of what we have proved in the preceding sections. Regard the first statement. Since $ R $ is of bounded type and [BG] and [CP] hold, \autoref{th:flatness-flatcompleted-th} implies that $ I(P) $ is quasi-flat. By \autoref{th:flatness-closedness-IRclosed}, we have that $ I(R)_{\compl{ℂQ}} $ is closed and by \autoref{th:flatness-closedness-IRcomplement} it has a closed complement. Invoking \autoref{th:flatness-closedness-AlgVsIdeal} gives that $ (B \htensor \compl{ℂQ}) / I(P) $ is a deformation of $ \compl{ℂQ} / I(R)_{\compl{ℂQ}} $.

Regard the second statement. By \autoref{th:flatness-flatalgebraic-th}, also $ I(P)_{ℂQ} $ is quasi-flat. Invoking \autoref{th:flatness-whatis-equivalence} proves the second statement. This finishes the proof.
\end{proof}

Let us restate this proposition in case the relations come from a superpotential. Recall that a superpotential $ W $ gives a relation space $ R = \vspan\{∂_a W\} $ and a deformation $ W' $ of the superpotential gives a deformed relation space $ P = \vspan\{∂_a (W + W')\} $. All details are taken care of by \autoref{th:flatness-BGstrong-Wcase}. Let us use the following notation:
\begin{align*}
\Jac(Q, W) &= \frac{ℂQ}{(∂_a W)}, \\
\Jac(\compl{Q}, W) &= \frac{\compl{ℂQ}}{~~\closure{(∂_a W)}~~}, \\
\Jac(Q, W + W') &= \frac{B \htensor ℂQ}{~~\closure{(B \htensor ℂQ) (∂_a (W + W')) (B \htensor ℂQ)}~~}, \quad \text{if } W' ∈ \mathfrak{m} ℂQ, \\
\Jac(\compl{Q}, W + W') &= \frac{B \htensor \compl{ℂQ}}{~~\closure{(B \compl{ℂQ}) (∂_a (W + W')) (B \compl{ℂQ})}^{\tensor}~~}.
\end{align*}
With these considerations, the following theorem is an immediate consequence of \autoref{th:flatness-flatness-general}.

\begin{theorem}
\label{th:flatness-flatness-applied}
Let $ Q $ be a quiver, $ W ∈ ℂQ_{≥3} $ a superpotential and $ W' ∈ \mathfrak{m} \compl{ℂQ} $ be cyclic. If $ \Jac(Q, W) $ is CY3 and $ W $ is of bounded type, then
\begin{itemize}
\item $ \Jac(\compl{Q}, W + W') $ is a deformation of $ \Jac(\compl{Q}, W) $.
\item $ \Jac(Q, W + W') $ is a deformation of $ \Jac(Q, W) $ if $ W' ∈ \mathfrak{m} ℂQ $.
\end{itemize}
\end{theorem}

The theorem also applies to the standard Jacobi algebra $ \Jac Q $ of a dimer. Here cancellation consistency of $ Q $ already implies that $ \Jac Q $ is CY3 \cite[Theorem 8.1]{Davison}. We have investigated the specific superpotential $ W = W_{\cyc}^+ - W_{\cyc}^- $ in \autoref{sec:flatness-dimers}. It is the difference of the clockwise and the counterclockwise polygons in $ Q $. If all polygons in $ Q $ have equal length or $ Q $ sits in a torus or has no triangles, then the superpotential $ W $ is of bounded type.

\begin{theorem}
\label{th:flatness-flatness-dimer}
Let $ Q $ be a cancellation consistent dimer of bounded type. Denote by $ W $ the superpotential of $ Q $ and let $ W' ∈ \mathfrak{m} \compl{ℂQ} $ be cyclic. Then
\begin{itemize}
\item $ \Jac(\compl{Q}, W + W') $ is a deformation of $ \Jac(\compl{Q}, W) $.
\item $ \Jac(Q, W + W') $ is a deformation of $ \Jac (Q, W) $ if $ W' ∈ \mathfrak{m} ℂQ $.
\end{itemize}
\end{theorem}

%% file: CHL/intro.tex
\section{A deformed Cho-Hong-Lau construction}
\label{sec:CHL}
In this section, we recapitulate the mirror construction of Cho, Hong and Lau \cite{CHL} and formulate a deformed version. Let us sketch this procedure: The construction of Cho, Hong and Lau starts from a pair $ (\RefObjects, \cat C) $ consisting of an $ A_∞ $-category $ \cat C $ together with a designated subcategory $ \RefObjects ⊂ \cat C $ whose $ A_∞ $-products are cyclic. From this pair $ (\ZigzagCat, \cat C) $ they construct a Landau-Ginzburg model: an algebra $ \Jac(\chlQ, W) $ with a central element $ ℓ ∈ \Jac(\chlQ, W) $. They also construct a functor $ F: \cat C → \MF(\Jac(\chlQ, W), ℓ) $:

\begin{center}
\begin{tikzpicture}
\path (0, 0) node[align=center] (A) {\textbf{Cyclic subcategory} \\ $ \RefObjects ⊂ \cat C $} (8, 0) node[align=center] (B) {Mirror functor \\ $ F: \cat C → \MF(\Jac(\chlQ, W), ℓ) $};
\path[draw, ->] ($ (A.east)!0.2!(B.west) $) -- ($ (A.east)!0.8!(B.west) $);
\end{tikzpicture}
\end{center}

The idea to deform this construction is as easy as it can get: Once we change $ \cat C $ to a deformation $ \cat C_q $, the relations of $ \Jac(\chlQ, W) $ deform and the central element $ ℓ $ changes. As long as the subcategory $ \DefZigzagCat ⊂ \cat C_q $ is still cyclic, this gives a deformed Landau-Ginzburg model $ (\Jac(\chlQ, W_q), ℓ_q) $ together with a deformed functor $ F_q: \cat C_q → \MF(\Jac(\chlQ, W_q), ℓ_q) $:

\begin{center}
\begin{tikzpicture}
\path (0, 0) node[align=center] (A) {\textbf{Cyclic deformed subcategory} \\ $ \DefRefObjects ⊂ \cat C_q $} (8, 0) node[align=center] (B) {Deformed mirror functor \\ $ F: \cat C_q → \MF(\Jac(\chlQ, W_q), ℓ_q) $};
\path[draw, ->] ($ (A.east)!0.2!(B.west) $) -- ($ (A.east)!0.8!(B.west) $);
\end{tikzpicture}
\end{center}

In \autoref{sec:CHL-koszul}, we motivate the Cho-Hong-Lau construction via Koszul duality. In \autoref{sec:CHL-classical}, we recall the construction and fix notation. In \autoref{sec:CHL-defLG}, we start deforming the construction by building a deformed Landau-Ginzburg model. In \autoref{sec:CHL-projectives}, we prepare a category of projective modules for deformed algebras. In \autoref{sec:CHL-defMF}, we define categories of deformed matrix factorizations. In \autoref{sec:CHL-functor}, we construct the deformed mirror functor.

After discussing the Cho-Hong-Lau construction in \autoref{sec:CHL-classical}, we will assume \autoref{conv:CHL-category} throughout the rest of the section. From \autoref{sec:CHL-defLG} onwards, we assume its deformed version \autoref{conv:CHL-deformed}. In text and lemmas, we may typically omit mentioning the conventions, while for the actual results we will always remind the reader. We may refer to the non-deformed Cho-Hong-Lau construction as the “classical construction” and to the deformed version as the “deformed construction”.

%% file: CHL/roadmap.tex
\subsection{Perspective from Koszul duality}
\label{sec:CHL-koszul}
In this section, we motivate the construction of Cho, Hong and Lau from the perspective of Koszul duality. In fact, the Cho-Hong-Lau construction is a specialized variant of Koszul duality adapted to the case that $ \cat C $ is not an augmented $ A_∞ $-category. In \autoref{sec:koszul} we have already recalled Koszul duality and the connection between cyclic $ A_∞ $-algebras and Calabi-Yau algebras. We have also provided a series of direct tweaks to Koszul duality which motivate the Cho-Hong-Lau construction. In the present section, we compile explicitly a roadmap from Koszul duality to the Cho-Hong-Lau construction:

\begin{description}
\item[Starting from arbitrary $ \cat C $:] Koszul duality departs from an augmented finite-dimensional $ A_∞ $-algebra $ A $ and yields a functor $ \rModfd A → \Tw \koszul{A} $. Cho, Hong and Lau instead depart from an $ A_∞ $-category $ \cat C $ and subcategory of reference objects $ \RefObjects ⊂ \cat C $. They form the $ A_∞ $-algebra $ A = \Hom(\RefObjects, \RefObjects) $. Their mirror functor is then roughly the composition
\begin{equation*}
F: \cat C \xrightarrow{\Hom(\RefObjects, -)} \rModfd A \xrightarrow{\text{Koszul}} \Tw \koszul{A}.
\end{equation*}
\item[Multiple reference objects:] The Koszul dual $ \koszul A $ is always an algebra of noncommutative power series. When departing from a category $ \RefObjects $ with multiple objects, the algebra $ A $ becomes structured over the semisimple ring $ ℂQ_0 $. The Koszul dual $ \koszul A $ inherits the structuring and becomes a quiver algebra with vertex set $ Q_0 $. The dual element $ \kdual x_i $ runs in the opposite direction of $ x_i $.
\item[Passing to cohomology:] The Koszul dual of an $ A_∞ $-algebra is a dg algebra. Cho, Hong and Lau forget the dg algebra and pass to cohomology. The root cause for the success of this procedure consists of an $ A_∞ $-morphism $ \koszul A → \H^0 \koszul A $. As we demonstrate in \autoref{th:koszul-transfer-pass}, this $ A_∞ $-morphism exists naturally if $ A $ is positively graded.
\item[Restriction to odd morphisms:] The Koszul dual of an $ A_∞ $-algebra $ A $ spanned by $ x_1, …, x_n $ is generated by dual variables $ \kdual x_1, …, \kdual x_n $ of degree $ |\kdual x_i| = 1 - |x_i| $. Cho, Hong and Lau take the even part of $ A $ into account, but restrict the mirror $ \H^0 \koszul A $ to the degree zero generators only. The only relations divided out are $ d_{\koszul A} \kdual x $ for $ x ∈ A^2 $. As we demonstrate in \autoref{th:koszul-transfer-cohI}, this approach comes completely naturally if $ A $ is positively graded.
\item[Relaxing the grading:] $ ℤ $-gradedness of $ A $ is not a necessity for Koszul duality. As we demonstrate in \autoref{th:koszul-transfer-3destiny}, it is however essential for using $ \H^0\koszul A $ as codomain of Koszul duality. Cho, Hong and Lau admit $ A $ to be $ ℤ/2ℤ $-graded. To obtain a functioning Koszul duality in this case, the codomain of the functor is not based on the actual cohomology $ \H^0\koszul A $ but deploys a surrogate. As we demonstrate in \autoref{th:koszul-transfer-dropZ}, the functioning surrogate for $ \H^0\koszul A $ is a quotient of the tensor algebra generated by $ \kdual x_i $ for $ x_i $ odd. The surrogate ideal is generated by the restrictions of $ d_{\koszul A} \kdual x_i ∈ \koszul A $ to $ T(\bar A^{\odd} [1]) ⊂ T(\bar A[1]) $ for $ x_i $ even.
\item[Cyclicity of relations:] Cyclicity of the $ A_∞ $-algebra $ A $ typically makes its Koszul dual $ \koszul A $ a Calabi-Yau dg algebra. While its cohomology $ \H^0\koszul A $ is not necessarily a Calabi-Yau, the surrogate algebra used instead of $ \H^0\koszul A $ is always the Jacobi algebra $ \Jac(\chlQ, W) $ of a quiver with superpotential. It is a candidate for being a Calabi-Yau algebra of dimension 3.
\item[Non-augmented $ \cat C $:] Koszul duality requires that the algebra $ A $ is augmented in the sense that $ A_∞ $-products of non-identities never yield identities. Cho, Hong and Lau solve this by accumulating all products that yield identities in an element $ ℓ ∈ \koszul{A} $. The element $ ℓ $ is central in $ \koszul{A} $ and therefore forms a curved dg algebra $ (\koszul{A}, ℓ) $.
\item[Matrix factorizations:] Together with the curvature mentioned above, the result is a Landau-Ginzburg model $ (\Jac_W \chlQ, ℓ) $ consisting of the Jacobi algebra of a quiver with superpotential, together with the additional potential $ ℓ $ as curvature. The suitable analog of the codomain $ \Tw\Jac(\chlQ, W) $ in this curved setting is category $ \MF(\Jac_W \chlQ, ℓ) $ of matrix factorizations.
\item[Opposite construction:] Koszul duality yields a strictly defined Koszul dual $ \koszul A $ and a functor to $ \Tw\koszul A $. Cho, Hong and Lau instead construct a Jacobi algebra based on the opposite algebra of $ \koszul A $. In this algebra, generator $ \kdual x_i $ points in the same direction as $ x_i $. The codomain of the functor never has to be written $ \Tw(\koszul A)^{\algopp} $, because $ \Tw $ is replaced by $ \MF $. While $ \Tw $ is naturally a category of right modules, the category $ \MF $ is by definition a category of left modules and therefore already takes the opposite into account.
\end{description}

The most important ingredients of the Cho-Hong-Lau construction are the category $ \cat C $ and a choice of subcategory $ \RefObjects = \{L_1, …, L_N\} ⊂ \cat C $ which is largely cyclic with respect to a non-degenerate odd graded-symmetric pairing $ ⟨-, -⟩ $. We shall fix this terminology as follows:

\begin{definition}
Let $ \cat C $ be an $ ℤ/2ℤ $-graded $ A_∞ $-category. An \emph{odd non-degenerate graded-symmetric pairing} on $ \cat C $ consists of a family of non-degenerate odd bilinear pairings $ ⟨-, -⟩_{L_1, L_2} $ indexed by all pairs of objects $ L_1, L_2 ∈ \cat C $, with
\begin{equation*}
⟨-, -⟩_{L_1, L_2}: \Hom(L_1, L_2) × \Hom(L_2, L_1) → ℂ, \quad ⟨x, y⟩ = (-1)^{|x||y|} ⟨y, x⟩.
\end{equation*}
\end{definition}

\begin{remark}
We simply write $ ⟨-, -⟩ $ instead of $ ⟨-, -⟩_{L_1, L_2} $. We set $ ⟨x, y⟩ = 0 $ whenever $ x, y $ lie in incompatible hom spaces.
\end{remark}

When choosing a basis for the odd part of the hom spaces of $ \RefObjects $, one obtains a dual basis for the even part of the hom spaces of $ \RefObjects $ as in \autoref{def:koszul-correspondence-dualbasis}. We shall prepare here terminology for the precise type of basis that the Cho-Hong-Lau construction requires:

\begin{definition}
\label{def:CHL-classical-basis}
Let $ \RefObjects = \{L_1, …, L_N\} $ be a unital $ A_∞ $-category with non-degenerate odd graded-symmetric pairing $ ⟨-, -⟩ $. Let $ E_{ij} $ be disjoint index sets for every $ 1 ≤ i, j ≤ N $ and let
\begin{equation*}
\{X_e\}_{e ∈ E_{ij}} ⊂ \Hom^{\odd} (L_i, L_j), \quad \{Y_e\}_{e ∈ E_{ji}} ⊂ \Hom^{\even} (L_i, L_j), \quad \coid_{L_i} ∈ \Hom^{\odd} (L_i, L_i)
\end{equation*}
for every $ 1 ≤ i, j ≤ N $. Then the triple $ \{X_e\} $, $ \{Y_e\} $, $ \{\coid_{L_i}\} $ is a \emph{CHL basis} for $ \RefObjects $ if
\begin{enumerate}
\item These families of morphisms form a basis for the hom spaces of $ \RefObjects $ when combined with the identities $ \id_{L_i} $:
\begin{equation*}
\Hom(L_i, L_j) = \vspan\{X_e\}_{e ∈ E_{ij}} ⊕ \vspan\{Y_e\}_{e ∈ E_{ji}} \quad [~ ⊕ \vspan\{\id_{L_i}, \coid_{L_i}\} \text{ if } i = j].
\end{equation*}
\item We have the pairing identities
\begin{equation}
\label{eq:CHL-classical-basis}
\begin{aligned}
& ⟨Y_e, X_f⟩ = ⟨X_f, Y_e⟩ = δ_{ef}, && ⟨\coid_{L_i}, \id_{L_j}⟩ = ⟨\id_{L_j}, \coid_{L_i}⟩ = δ_{ij}, \\
& ⟨X_e, X_f⟩ = ⟨Y_e, Y_f⟩ = 0, && ⟨\id_{L_i}, \id_{L_j}⟩ = ⟨\coid_{L_i}, \coid_{L_j}⟩ = 0, \\
& ⟨X_e, \id_{L_i}⟩ = ⟨X_e, \coid_{L_i}⟩ = 0, && ⟨Y_e, \id_{L_i}⟩ = ⟨Y_e, \coid_{L_i}⟩ = 0.
\end{aligned}
\end{equation}
\end{enumerate}
The element $ \coid_{L_i} $ is the \emph{co-identity} of $ L_i $ in $ \RefObjects $.
\end{definition}

\begin{remark}
In contrast to the case of cyclic $ A_∞ $-algebras, the pairing pairs the opposite hom spaces $ \Hom(L_i, L_j) $ and $ \Hom(L_j, L_i) $. The dual basis element for $ X_e ∈ \Hom(L_i, L_j) $ is therefore an element $ Y_e ∈ \Hom(L_j, L_i) $. This is the reason why the basis elements $ Y_e $ for $ \Hom^{\even} (L_i, L_j) $ are indexed by the index set $ E_{ji} $ borrowed from the basis of the opposite hom space $ \Hom^{\odd} (L_j, L_i) $.
\end{remark}

\begin{remark}
A Cho-Hong-Lau basis for $ \RefObjects $ always exists if the hom spaces of $ \RefObjects $ are finite-dimensional. This is a simple consequence of \autoref{def:koszul-correspondence-dualbasis} which we shall briefly explain. First, the identity $ \id_{L_i} $ determines a dual odd element $ \coid_{L_i} ∈ \Hom^{\odd} (L_i, L_i) $. One now freely chooses further odd elements $ X_e ∈ \Hom^{\odd} (L_i, L_j) $, where the index $ e $ ranges over an arbitrary set $ E_{ij} $ for every $ 1 ≤ i, j ≤ N $. Together with the co-identities, the elements $ X_e $ are supposed to form a basis for the odd part of the hom spaces of $ \cat C $:
\begin{equation*}
\Hom^{\odd} (L_i, L_j) = \vspan\{X_e\}_{e ∈ E_{ij}} \quad [⊕ ℂ\coid_{L_i} \text{ if } i = j].
\end{equation*}
According to \autoref{def:koszul-correspondence-dualbasis}, we obtain a unique dual basis for the even hom spaces. It is of the form
\begin{equation*}
\Hom^{\even} (L_i, L_j) = \vspan\{Y_e\}_{e ∈ E_{ji}} \quad [⊕ ℂ\id_{L_i} \text{ if } i = j].
\end{equation*}
By construction, the triple $ \{X_e\} $, $ \{Y_e\} $ and $ \{\coid_{L_i}\} $ now satisfies the pairing identities \eqref{eq:CHL-classical-basis} and forms a CHL basis according to \autoref{def:CHL-classical-basis}.
\end{remark}

%% file: CHL/classical.tex
\subsection{The Cho-Hong-Lau construction}
\label{sec:CHL-classical}
In this section, we recall the construction of the noncommutative mirror functor due to Cho, Hong and Lau \cite{CHL}. The aim is to define the mirror functor as fast as possible. The mirror functor recalled here serves as leading term of the deformed mirror functor which we construct in the next sections. The present section also serves to fix notation and terminology as well as to fix sign conventions.

In \autoref{conv:CHL-category}, we record the complete list of input data and assumptions for the Cho-Hong-Lau construction. The input data include a category $ \cat C $ and a chosen subcategory $ \RefObjects ⊂ \cat C $. The input data also include a choice of CHL basis for $ \RefObjects $. The assumptions include that $ \RefObjects $ is “cyclic on the odd augmented part” of $ \RefObjects $. The precise list reads as follows:

\begin{convention}
\label{conv:CHL-category}
The $ A_∞ $-category $ \cat C $ is $ ℤ/2ℤ $-graded and unital. A subset of reference objects $ \RefObjects = \{L_1, …, L_N\} ⊂ \cat C $ is provided. The category $ \RefObjects $ is supposed to come with an odd non-degenerate graded-symmetric pairing $ ⟨-, -⟩ $. A CHL basis $ \{X_e\}_{e ∈ E_{ij}, 1 ≤ i, j ≤ N} $, $ \{Y_e\}_{e ∈ E_{ij}, 1 ≤ i, j ≤ N} $, $ \{\coid_{L_i}\}_{1 ≤ i ≤ N} $ for $ \RefObjects $ is provided. The category $ \RefObjects $ is required to be cyclic on the odd part with respect to $ ⟨-, -⟩ $:
\begin{equation*}
⟨μ(X_{e_{k+1}}, …, X_{e_2}), X_{e_1}⟩ = ⟨μ(X_{e_k}, …, X_{e_1}), X_{e_{k+1}}⟩.
\end{equation*}
The hom spaces $ \Hom(L_i, X) $ are assumed to be finite-dimensional for $ 1 ≤ i ≤ N $ and $ X ∈ \cat C $.
\end{convention}

The construction of mirror and mirror functor proceeds by Koszul transforming the $ A_∞ $-structure of $ \RefObjects $. In classical Koszul duality, one transforms the $ A_∞ $-structure of an $ A_∞ $-algebra by looking at which sequences of basis input elements produce a given basis element as output. In the construction of Cho, Hong and Lau, the role of the input sequences is played by sequences of basis elements $ X_{e_1}, …, X_{e_k} $, ranging over all hom spaces in $ \RefObjects $. To record the products on these sequences, one introduces a formal variable $ x_e $ for every $ e ∈ E_{ij} $ and $ 1 ≤ i,j ≤ N $. The variables are subject to constraints on composition, coming from a quiver structure:

\begin{definition}
The \emph{CHL quiver} $ \chlQ $ has one vertex $ L_i $ for every reference object $ L_i $ and an arrow $ x_e: L_i → L_j $ for every $ e ∈ E_{ij} $ and $ 1 ≤ i, j ≤ N $.
\end{definition}

With the definition of $ \chlQ $ in mind, we build the auxiliary formal element
\begin{equation}
\label{eq:CHL-classical-bdef}
b = \sum_{i, j = 1}^N \sum_{e ∈ E_{ij}} x_e X_e.
\end{equation}
In principle, the basis morphisms $ X_e $ lie in different hom spaces. If we view $ \RefObjects $ as a direct sum of its elements $ L_1, …, L_N $, we can interpret $ b $ as an element of $ \compl{ℂ\chlQ} ¤ \Hom(\RefObjects, \RefObjects) $. Our convention is that product of the type $ μ (m_k, …, m_1, b, …, b) $ are always to be understood as multlinear expansions of the product under use of the sum \eqref{eq:CHL-classical-bdef}. More background can be found in \cite[Chapter 2]{CHL}.

Summing up products of the type $ μ(m_k, …, m_1, b, …, b) $ over increasing number of $ b $-insertions gives an infinite series. The summands consist of paths in $ \chlQ $ multiplied by basis elements $ X_e $, $ Y_e $ or (co)identities. The coefficients series of a basis element $ X_e $ need not converge in $ ℂ\chlQ $, but generally only in the completed path algebra $ \compl{ℂ\chlQ} $. The special case where the coefficient series terminate is however relevant, as it allows one to obtain a Landau-Ginzburg model building on the quiver algebra $ ℂ\chlQ $ instead of its completion. We shall give this case a name:

\begin{definition}
$ \RefObjects $ is of \emph{bounded growth} if for all morphisms $ m_1, …, m_1 $ in $ \cat C $ there is an $ l_0 ∈ ℕ $ such that
\begin{equation*}
∀l ≥ l_0: \quad μ^{k+l} (m_k, …, m_1, b, …, b) = 0.
\end{equation*}
\end{definition}

With this in mind, we can define all relevant intermediates of the Cho-Hong-Lau construction as follows:

\begin{definition}
\label{def:CHL-classical-gadgetsdef}
The \emph{relations} $ R_e ∈ \compl{ℂ\chlQ} $ and the \emph{potential} $ ℓ ∈ \compl{ℂ\chlQ} $ are defined by
\begin{equation}
\label{eq:CHL-classical-Rdef}
\sum_{k ≥ 1} μ^k (b, …, b) = ℓ \id_{\RefObjects} + \sum_{i, j = 1}^N \sum_{e ∈ E_{ij}} R_e Y_e.
\end{equation}
The \emph{superpotential} is defined as
\begin{equation*}
W = ⟨\sum_{k ≥ 1} μ^k (b, …, b), b⟩ ∈ \compl{ℂ\chlQ}.
\end{equation*}
The \emph{Jacobi algebra} is defined as
\begin{equation}
\label{eq:CHL-classical-Jacdef}
\Jac (\compl{\chlQ}, W) = \frac{\compl{ℂ\chlQ}}{~\closure{(∂_{x_e} W)}~}.
\end{equation}
The \emph{Landau-Ginzburg model} is the pair $ (\Jac(\compl{\chlQ}, W), ℓ) $. If $ \RefObjects $ is of bounded growth, then $ R_e, ℓ, W $ are regarded as elements of $ ℂ\chlQ $, the Jacobi algebra is defined as $ \Jac (\chlQ, W) = ℂQ / (∂_{x_e} W) $, and the Landau-Ginzburg model is $ (\Jac(\chlQ, W), ℓ) $.
\end{definition}

\begin{remark}
The element $ W ∈ \compl{ℂ\chlQ} $ is cyclic, as we recall in \autoref{th:CHL-classical-centrality}. Its derivative $ ∂_{x_e} W $ is defined by stripping off $ x_e $ from the front (or rear) side of all terms in $ W $ that start (or end) with $ x_e $.
\end{remark}

\begin{remark}
The description of $ ℓ $ and $ R_e $ in \eqref{eq:CHL-classical-Rdef} is to be interpreted as follows: All the products $ μ(b, …, b) $ are even and can hence be written as a sum of the even basis elements. The even basis elements are by assumption of the form $ Y_e $ and $ \id_{L_i} $. The element $ R_e $ is formed by recording the coefficient of $ Y_e $ and the element $ ℓ $ is formed by recording the coefficients of the identities $ \id_{L_i} $ and summing up.
\end{remark}

\begin{remark}
We have used the notation $ \closure{X} $ for the closure of a set $ X ⊂ \compl{ℂ\chlQ} $ with respect to the Krull topology on $ \compl{ℂ\chlQ} $. More information on the Krull topology can be found in \autoref{sec:flatness-closedness}.
\end{remark}

\begin{remark}
In \autoref{def:CHL-classical-gadgetsdef}, the two uses of the symbol $ (∂_{x_e} W) $ differ slightly. Namely in \eqref{eq:CHL-classical-Jacdef}, the expression $ (∂_{x_e} W) $ denotes the ideal $ \compl{ℂ\chlQ} \vspan(∂_{x_e})_e \compl{ℂ\chlQ} $ generated by the relations $ ∂_{x_e} W $ in $ \compl{ℂ\chlQ} $, while in the case of bounded growth it denotes the ideal generated in $ ℂ\chlQ $. Alternatively, the shared definition $ (∂_{x_e} W) = ℂ\chlQ \vspan(∂_{x_e} W) ℂ\chlQ $ can be used in both cases since in \eqref{eq:CHL-classical-Jacdef} closure is taken. The notation for the two ideals also differs slightly from the notation of \autoref{sec:flatness}.
\end{remark}

\begin{remark}
We have decided to treat the case that $ \RefObjects $ is of bounded growth in parallel with the general case. In contrast, Cho, Hong and Lau \cite{CHL} specialize to the case of bounded growth only in Chapter 10. In fact, their construction departs from the even more general situation involving the Novikov ring. We shall try to make explicit every time whether we regard the general construction which uses the completed path algebra and the closure of the ideal, or the specific construction which uses the ordinary path algebra and ideal.
\end{remark}

\begin{remark}
Unfortunately the calculations of \cite{CHL} and \cite{CHL17} appear to contain at least two independent sign issues. We have therefore decided to repeat the calculations here and repair the signs. We shall here trace back the signs: The first dubious sign can be found in \cite[Theorem 2.19]{CHL17}: The expressions (2.10), (2.12), (2.14) indeed add up to zero due to the $ A_∞ $-relation for $ \cat C $, but we need to show that the difference $ +\text{(2.10)}-\text{(2.12)}-\text{(2.14)} $ vanishes. This issue breaks the functor equations for $ \compl F $, even if the matrix factorizations under consideration have vanishing $ δ $. The second dubious sign can be found in the combination of \cite[Definition 4.3]{CHL} and \cite[Definition 4.4]{CHL}. The specific combination of sign conventions for $ μ^1_{\MF} $ and for the endomorphism $ δ $ of $ \compl F $ seems to break the functor equations for $ \compl F $. A third issue is that even without regarding the functor relations the definition of $ \compl F $ immediately renders $ \compl F(\id) = -\id $, while it would be desirable to have $ \compl F(\id) = \id $. We have tried to repair all issues, even though it makes the sign convention for $ \compl F $ slightly unesthetic.
\end{remark}

From now on, the element $ ℓ $ is typically regarded as an element in the quotient $ \Jac(\chlQ, W) $. Its significance in the quotient and the relation between $ R_e $ and $ W $ is explained as follows:

\begin{lemma}
\label{th:CHL-classical-centrality}
The superpotential $ W ∈ \compl{ℂ\chlQ} $ is cyclic and we have $ R_e = ∂_{x_e} W $. The potential $ ℓ ∈ \Jac(\compl{\chlQ}, W) $ is central. The analogous statements hold if $ \RefObjects $ is of bounded growth.
\end{lemma}

\begin{proof}
We divide the proof into the four obvious parts. The first part of the proof is to check cyclicity of $ W $. To show this, regard the sum decomposition
\begin{equation}
\label{eq:CHL-classical-Wsum}
⟨μ^k (b, …, b), b⟩ = \sum_{e_1, …, e_{k+1}} ⟨μ^k (X_{e_{k+1}}, …, X_{e_2}), X_{e_1}⟩ x_{e_{k+1}} … x_{e_1}.
\end{equation}
We will group the collection of all summands into cyclic orbits. Namely, regard any term $ λ x_{e_{k+1}} … x_{e_1} $ appearing in this sum. Recall that by the cyclicity assumption of \autoref{conv:CHL-category} we have
\begin{equation}
λ = ⟨μ^k (X_{e_{k+1}}, …, X_{e_2}), X_{e_1}⟩ = ⟨μ^k (X_{e_k}, …, X_{e_1}), X_{e_{k+1}}⟩,
\end{equation}
This means the permuted version $ λ x_{e_k} … x_{e_1} x_{e_{k+1}} $ also appears in the sum \eqref{eq:CHL-classical-Wsum}, with equal coefficient $ λ $. This renders $ ⟨μ^k (b, …, b), b⟩ $ cyclic. Summing over $ k ≥ 1 $, we derive that the entire superpotential $ W $ is cyclic.

The second part of the proof is to check that $ R_e = ∂_{x_e} W $. Usually, derivatives of a cyclic superpotential are written in terms of those paths starting with the given variable $ x_e $. Since $ W $ is already cyclic, it suffices to extract all paths which instead end on $ x_e $. To avoid double indexing by $ e $, we will write $ x_f $ in the expansion of $ ⟨μ^k (b, …, b), b⟩ $. Ultimately, we calculate within $ \compl{ℂ\chlQ} $ that
\begin{align*}
∂_{x_e} W &= ∂_{x_e} \left⟨\sum_{k ≥ 1} μ^k (b, …, b), \sum_f x_f X_f\right⟩ \\ &= ∂_{x_e} \left⟨ℓ \id_{\RefObjects} + \sum_f R_f Y_f, \sum_f x_f X_f\right⟩
= ∂_{x_e} \left(\sum_f R_f x_f\right) = R_e.
\end{align*}
The third part of the proof is to check that $ ℓ ∈ \Jac(\compl{\chlQ}, W) $ is central. In this part of the proof, we shall regard $ ℓ $ interchangeably as element of $ \compl{ℂ\chlQ} $ or $ \Jac(\compl{Q}, W) $. We start with the observation that within $ \compl{ℂ\chlQ} ¤ \Hom(\RefObjects, \RefObjects) $ we have
\begin{align*}
0 = \sum_{k ≥ 1} \sum_{l ≥ 1} μ^k (b, …, μ^l (b, …, b), …, b)
= \sum_{k ≥ 1} μ^k (b, …, \sum_{i = 1}^N ℓ_i \id_{L_i} + \sum R_e Y_e, …, b).
\end{align*}
We have essentially performed a reordering of the double sum which is legitimate since the path lengths regarding $ \chlQ $ encountered in $ μ(b, …, b) $ increase as the number of inputs increases. We now claim that apart from $ μ^2 (b, ℓ_i \id_{L_i}) $ and $ μ^2 (ℓ_i \id_{L_i}, b) $, the entire sum on the right-hand side of the equation lies in $ \closure{(∂_{x_e} W)} ¤ \Hom(\RefObjects, \RefObjects) $. Indeed, all terms with identities vanish for $ k ≥ 3 $ and $ k = 1 $ and the products involving $ Y_e $ all embrace relations $ R_e $. We caution that the expression does in general not lie in $ \compl{ℂ\chlQ} \vspan(∂_{x_e} W) \compl{ℂ\chlQ} $, see \autoref{rem:flatness-closedness-nonclosedCQRCQ} for an illustration. Ultimately, we conclude within $ \Jac(\compl{\chlQ}, W) ¤ \Hom(\RefObjects, \RefObjects) $ that
\begin{equation*}
0 = μ^2 (b, ℓ \id_{\RefObjects}) + μ^2 (ℓ \id_{\RefObjects}, b).
\end{equation*}
Let $ 1 ≤ i, j ≤ N $ and $ e ∈ E_{ij} $. Then extracting the $ X_e $-component gives
\begin{equation*}
0 = x_e ℓ - ℓ x_e \text{ within } \Jac(\compl{\chlQ}, W).
\end{equation*}
We conclude that $ ℓ $ commutes with all arrows in $ \chlQ $ and hence with all finite paths. Let now $ x ∈ \compl{ℂ\chlQ} $ be an arbitrary element. If one wants to take Krull-continuity of the projection $ π: \compl{ℂ\chlQ} → \Jac(\compl{\chlQ}, W) $ for granted, simply choose a sequence $ (x_n) ∈ ℂ\chlQ $ with $ x_n → x $, then $ 0 = π(x_n) ℓ - ℓ π(x_n) → π(x)ℓ - ℓπ(x) $, hence $ xℓ = ℓx $. If not, write $ x = \sum_{k ≥ 0} x_k $ where $ x_k $ is homogeneous of length $ k $. Then
\begin{equation*}
\left(\sum_{k ≥ 0} π(x_k)\right) ℓ = π\left(\left(\sum_{k ≥ 0} x_k \right) ℓ\right) = π\left(\sum_{k ≥ 0} x_k ℓ\right) = π\left(\sum_{k ≥ 0} ℓ x_k + z_k \right) = ℓ \left(\sum_{k ≥ 0} π(x_k) \right).
\end{equation*}
Here $ z_k ∈ \closure{(∂_{x_e} W)} $ denotes the difference of $ x_k ℓ $ and $ ℓ x_k $ when regarded as elements of $ \compl{ℂ\chlQ} $. Note that $ z_k $ has length at least $ k $ since $ x_k $ does, hence $ \sum z_k ∈ \closure{(∂_{x_e} W)} $. We conclude that $ ℓ $ commutes with any element $ x ∈ \Jac(\compl{\chlQ}, W) $ and hence $ ℓ ∈ Z(\Jac(\compl{\chlQ}, W)) $. The fourth part of the proof consists of observing that the calculations still hold, even simplify, in case $ \RefObjects $ is of bounded growth. This finishes the proof.
\end{proof}

It is time to demonstrate how one calculates in the Jacobi algebra. For instance, within $ \Jac(\compl{\chlQ}, W) ¤ \Hom(\RefObjects, \RefObjects) $ the expression \eqref{eq:CHL-classical-Rdef} simplifies to
\begin{equation*}
\sum_{k ≥ 1} μ^k (b, …, b)  = \sum_{i = 1}^N ℓ_i \id_{L_i} ∈ \Jac(\compl{\chlQ}, W) ¤ \Hom(\RefObjects, \RefObjects).
\end{equation*}

We are ready to recall the construction of the mirror functor of Cho, Hong and Lau. The idea is to tweak the Koszul duality functor for the $ A_∞ $-algebra $ A = \End(\RefObjects) $. We shall construct two functors $ \compl F $ and $ F $, where $ \compl F $ serves the general case and $ F $ the case where $ \RefObjects $ is of bounded growth. The domain of both functors is $ \cat C $ and the codomains are the matrix factorization categories $ \MF(\Jac(\compl{\chlQ}, W), ℓ) $ and $ \MF(\Jac(\chlQ, W), ℓ) $, respectively. We start with the explicit descriptions on the level of objects:
\begin{align}
\label{eq:CHL-classical-Fobj}
\compl F(X) &= \left(\bigoplus_{i = 1}^N \Jac(\compl{\chlQ}, W) L_i \tensor \Hom(L_i, X), \quad δ(m) = \sum_{k ≥ 1} (-1)^{‖m‖} μ^k (m, b, …, b)\right), \\
F(X) &= \left(\bigoplus_{i = 1}^N \Jac(\chlQ, W) L_i ¤ \Hom(L_i, X), \quad δ(m) = \sum_{k ≥ 1} (-1)^{‖m‖} μ^k (m, b, …, b)\right).
\end{align}
Let us explain how to interpret these expressions as matrix factorizations. Recall from \autoref{th:prelim-matrixfact-alternative} that a matrix factorization can be defined as a $ ℤ/2ℤ $-graded module which together with an odd endomorphism which squares to the desired central element. In our case, the $ \Jac(\compl{\chlQ}, W) $-module $ \Jac(\compl{\chlQ}, W) L_i \tensor \Hom(L_i, X) $ shall have $ ℤ/2ℤ $-grading inherited from $ \Hom(L_i, X) $. Since $ b $ is odd, the map $ δ $ itself becomes odd. The module is projective and finitely generated since $ \Hom(L_i, X) $ is finite-dimensional by assumption. If we check that $ δ $ squares to $ ℓ $, then $ \compl F(X) $ is indeed a matrix factorization.

\begin{lemma}
For $ X ∈ \cat C $ the object $ \compl F(X) $ is indeed a matrix factorization of $ (\Jac(\compl{\chlQ}, W), ℓ) $. If $ \RefObjects $ is of bounded growth, then $ F(X) $ is a matrix factorization of $ (\Jac(\chlQ, W), ℓ) $.
\end{lemma}

\begin{proof}
We merely check the case of $ \compl F(X) $. It is our task to show that $ δ^2 $ equates to multiplying by $ ℓ $. Calculating in $ \Jac(\compl{\chlQ}, W) ¤ \Hom(\RefObjects, X) $, we get
\begin{align*}
δ(δ(m)) &= δ(\sum_{k ≥ 1} (-1)^{‖m‖} μ^k (m, b, …, b)) \\
&= \sum_{l ≥ 1} \sum_{k ≥ 1} (-1)^{‖m‖ + ‖m‖ - 1} μ^l (μ^k (m, b, …, b), b, …, b) \\
&= - \sum_{n ≥ 1} (μ · μ)^n (m, b, …, b) + \sum_{k ≥ 2} \sum_{l ≥ 1} μ^k (m, b, …, μ^l (b, …, b), …, b) \\
&= μ^2 (m, ℓ \id_{\RefObjects}) = ℓ m ∈ \Jac(\compl{\chlQ}, W) ¤ \Hom(\RefObjects, X).
\end{align*}
Here we have used the $ A_∞ $-relation $ μ · μ = 0 $ and $ \sum_{l ≥ 1} μ^l (b, …, b) = ℓ \id_{\RefObjects} $. We conclude that $ δ^2 = ℓ $. This shows that $ \compl F(X) $ is a matrix factorization. The analogous calculations show that $ F(X) $ is a matrix factorization when $ \RefObjects $ is of bounded growth.
\end{proof}

We are finally ready to write down the functors $ \compl F $ and $ F $. On objects, these functors map $ X ∈ \cat C $ to their associated matrix factorizations $ \compl F(X) $ and $ F(X) $ defined in \eqref{eq:CHL-classical-Fobj}. On morphisms, the functors are intuitively constructed through viewing $ \Hom(L_i, X) $ as a module over $ \End(\RefObjects) $ and composing with the Koszul duality functors. We explain the origin of the functors in more detail in \autoref{rem:CHL-classical-origin}.

\begin{definition}
The \emph{CHL functor} $ \compl F $ is the mapping
\begin{align}
\compl F: \cat C &\verylongto \MF(\Jac(\compl{\chlQ}, W), ℓ), \\
X &\verylongmapsto \compl F(X), \\
\compl F(m_k, …, m_1)(m) &= (-1)^{(‖m_1‖ + … + ‖m_k‖) ‖m‖ + 1} \sum_{l ≥ 0} μ^{k+l+1} (m_k, …, m_1, m, b, …, b) \\
& \text{for } m_i: X_i → X_{i+1}, \quad m ∈ \compl F(X_1).
\end{align}
In case $ \RefObjects $ is of bounded growth, the functor $ F: \cat C → \MF(\Jac(\chlQ, W), ℓ) $ is defined analogously.
\end{definition}

\begin{remark}
Let us elaborate and make sense of the definition of $ \compl F(m_k, …, m_1) $. To start with, the sequence $ m_1, …, m_k $ is an arbitrary sequence of morphisms $ m_i: X_i → X_{i+1} $ in $ \cat C $. The morphism $ \compl F(m_k, …, m_1) $ is supposed to be a $ \Jac(\compl{\chlQ}, W) $-module map $ \compl F(X_1) → \compl F(X_{k+1}) $. Let us make sense of its definition: The element $ m $ used to define this map lies in the direct sum
\begin{equation*}
\compl F(X_1) = \bigoplus_{i ∈ \chlQ_0} \Jac(\compl{\chlQ}, W) L_i ¤ \Hom(L_i, X_1).
\end{equation*}
Its image $ \compl F(m_k, …, m_1) (m) $ is defined by the products $ μ(m_k, …, m_1, m, b, …, b) $, with an arbitrary amount of $ b $-insertions. The result of each of these products lives in
\begin{equation*}
\compl F(X_{k+1}) = \bigoplus_{j ∈ \chlQ_0} \Jac(\compl{\chlQ}, W) L_j ¤ \Hom(L_j, X_{k+1}).
\end{equation*}
To see this, recall that the formal parameters $ x_e $ in $ b $ simply get multiplied up as we evaluate the product. Simply speaking, a product where the right-most $ b $-summand $ x_e X_e $ is consumed lands in $ \Jac(\compl{\chlQ}, W) t(e) ¤ \allowbreak \Hom(L_{t(e)}, X_{k+1}) $. Finally, note that $ \compl F(m_k, …, m_1) $ is indeed a module map: If $ a ∈ \Jac(\compl{\chlQ}, W) $, then $ F(m_k, …, m_1)(am) = a F(m_k, …, m_1)(m) $ since the factor $ a $ can be pulled to the front.
\end{remark}

The main algebraic result of Cho, Hong and Lau entails that $ \compl F $ indeed defines an $ A_∞ $-functor:

\begin{lemma}
\label{th:CHL-classical-functor}
The CHL functor is a unital $ A_∞ $-functor $ \compl F: \cat C → \MF(\Jac(\compl{\chlQ}, W), ℓ) $. In case $ \RefObjects $ is of bounded growth, then $ F $ is a unital $ A_∞ $-functor as well.
\end{lemma}

\begin{proof}
We have that $ \compl F(\id)(m) = (-1)^{‖m‖ + 1} μ^2 (\id, m) = m $, therefore $ \compl F(\id) = \id $. Similarly, we have $ \compl F(m_k, …, m_1) = 0 $ whenever $ k ≥ 1 $ and one $ m_i $ is an identity. For the functor relations, we need to check
\begin{multline*}
\sum_{i, j} (-1)^{‖m_1‖ + … + ‖m_j‖} \compl F(m_k, …, μ(m_i, …, m_{j+1}), …, m_1) \\
= μ^2_{\MF} (\compl F(m_k, …, m_{i+1}), \compl F(m_i, …, m_1)) + μ^1_{\MF} (\compl F(m_k, …, m_1)).
\end{multline*}
Both sides are homomorphisms $ \compl F(X) → \compl F(Y) $ of certain $ \Jac(\compl{\chlQ}, W) $-modules. In order to equate both sides, we plug in an arbitrary element $ m ∈ \compl F(X) $ and start evaluating from the left-hand side. Minding the notation $ \tilde{δ} (m) = (-1)^{|m|} δ(m) $, we get:
\begin{align*}
& \sum_{0 ≤ i ≤ j ≤ k} (-1)^{‖m_1‖ + … + ‖m_j‖} \compl F(m_k, …, μ(m_i, …, m_{j+1}), …, m_1)(m) \\
&= \sum_{\substack{0 ≤ i ≤ j ≤ k \\ 0 ≤ l}} (-1)^{‖m_1‖ + … + ‖m_j‖ + (‖m_1‖ + … + ‖m_k‖ + 1)‖m‖ + 1} μ^{k+l} (m_k, …, μ(m_i, …, m_{j+1}), …, m_1, m, b, …, b) \\
&= \sum_{\substack{1 ≤ i ≤ k-1 \\ j, l ≥ 0}} (-1)^{‖m‖ + (‖m_1‖ + … + ‖m_k‖ + 1)‖m‖} μ^{k-i+l+1} (m_k, …, μ^{i+1+l} (m_i, …, m_1, m, b, …, b), b, …, b) \\
& \quad + \sum_{\substack{i, j ≥ 0}} (-1)^{‖m‖ + (‖m_1‖ + … + ‖m_k‖ + 1)‖m‖} μ^{i+1} (μ^{k+j+1} (m_k, …, m_1, m, b, …, b), b, …, b) \\
& \quad + \sum_{\substack{i, j ≥ 0}} (-1)^{‖m‖ + (‖m_1‖ + … + ‖m_k‖ + 1)‖m‖} μ^{k+i+1} (m_k, …, m_1, μ^{j+1} (m, b, …, b), b, …, b) \\
&= \sum_{1 ≤ i ≤ k-1} (-1)^{\maltese} \compl F(m_k, …, m_{i+1})(\compl F(m_i, …, m_1)(m)) \\
& \quad + (-1)^{‖m‖ + (‖m_1‖ + … + ‖m_k‖ + 1)‖m‖ + (‖m_1‖ + … + ‖m_k‖)‖m‖ + 1 + 1} \tilde{δ}(\compl F(m_k, …, m_1)(m)) \\
& \quad + (-1)^{‖m‖ + (‖m_1‖ + … + ‖m_k‖ + 1)‖m‖ + 1 + (‖m_1‖ + … + ‖m_k‖)‖\tilde{δ}(m)‖ + 1} \compl F(m_k, …, m_1)(\tilde{δ}(m)) \\
& \qquad \scalebox{0.7}{$ \maltese = ‖m‖ + (‖m_1‖ + … + ‖m_k‖ + 1)‖m‖ + (‖m_1‖ + … + ‖m_i‖)‖m‖ + 1 + (‖m_{i+1}‖ + … + ‖m_k‖) (‖\compl F(m_i, …, m_1)(m)‖) + 1 $} \\
&= \sum_{1 ≤ i ≤ k-1} (-1)^{(‖m_{i+1}‖ + … + ‖m_k‖) (‖m_1‖ + … + ‖m_i‖ + 1)} \compl F(m_k, …, m_{i+1})(\compl F(m_i, …, m_1)(m)) \\
& \quad + \tilde{δ}(\compl F(m_k, …, m_1)(m)) - (-1)^{|\compl F(m_k, …, m_1)|} \compl F(m_k, …, m_1)(\tilde{δ}(m)) \\
&= μ^2_{\MF} (\compl F(m_k, …, m_{i+1}), \compl F(m_i, …, m_1)) (m) + μ^1_{\MF} (\compl F(m_k, …, m_1)) (m).
\end{align*}
This calculation deserves a few comments. In the first equality, we have unraveled the definition of $ \compl F $. In the second equality, we have used the $ A_∞ $-relations for $ \cat C $, which adds an absolute flip to the sign. In the third equality, we have reinterpreted the inner and outer $ μ $ applications as $ \compl F $ or $ \tilde{δ} $, using that $ \tilde{δ} (m) = - \sum_{l ≥ 1} μ^l (m, b, …, b) $. In the fourth equality, we have rewritten the expressions in terms of the products $ μ_{\MF} $. Since $ m $ was arbitrary, we conclude that the $ A_∞ $-functor relations hold. The very same calculations apply for $ F $ in case $ \RefObjects $ is of bounded growth.
\end{proof}

\begin{remark}
\label{rem:CHL-classical-origin}
The approach of Cho, Hong and Lau seems to originate from Yoneda functors instead of Koszul duality. We sketch here their line of reasoning \cite{CHL17, CHL}, focusing on the case of a single reference object, $ \RefObjects = \{L\} $. Let $ \Ch $ denote the dg category of chain complexes.

A typical Yoneda functor in the $ A_∞ $-world takes the shape $ F = \Hom(L, -): \cat C → \Ch $ and sends $ X ∈ \cat C $ to the chain complex $ (\Hom(L, X), μ^1) $. The functor $ F $ is not strict, but has higher components $ F^{≥1} $ given by $ F(m_k, …, m_1)(m) = ± μ(m_k, …, m_1, m) $.

Twisting an $ A_∞ $-category $ \cat C $ by any element $ b $ gives a new possibly curved $ A_∞ $-category. Cho, Hong and Lau simply decided to enlarge the category $ \cat C $ by introducing formal variables $ x_e $ for every odd basis element $ X_e $ of the hom space $ \Hom(L, L) $. Their twist is given by the element $ b = \sum x_e X_e $. This gives an enlarged and twisted category $ \cat C_b = (\cat C \ncpow{x_e}, b) $. Its products $ μ^k_b $ are given by twisting $ μ $ with $ b $. This category in principle has a Yoneda functor $ F_b: \cat C_b → \Ch $ itself, given by
\begin{equation}
\label{eq:CHL-classical-twistedfunctor}
\begin{aligned}
F_b (X) &= (\Hom(L, X), μ^1_b), \\
F_b (m_k, …, m_1)(m) &= μ_b (m_k, …, m_1, m).
\end{aligned}
\end{equation}
At this point, adaptations have to be made. For instance the map $ μ^1_b $, explicitly $ μ^1_b (m) = μ(b, …, m, …, b) $, is not a differential anymore because $ \cat C_b $ has curvature $ μ(b, …, b) $. Before we proceed, simplify $ F_b $ slightly by inserting $ b $ only at the right-most positions in \eqref{eq:CHL-classical-twistedfunctor}.

To make $ μ^1_b $ a differential and restore the functor $ F_b $, the crucial idea is to enforce relations among the variables $ x_e $ to at least render the curvature a multiple of the identity. This precisely explains the choice of the relations $ R_e $. It be noted that $ μ^1_b $ is still not a differential since $ μ(b, …, b) $ may contain an identity term $ ℓ \id_L $. This in mind, $ μ^1_b $ however squares to $ ℓ $ and we are in the realm of matrix factorizations. The functor $ F_q $ can be adapted accordingly. This explains the original motivation of Cho, Hong and Lau.
\end{remark}

%% file: CHL/defLG.tex
\subsection{Deformed Landau-Ginzburg model}
\label{sec:CHL-defLG}
In this section, we start deforming the construction of Cho, Hong and Lau. The idea is as easy as it can get: We apply the Cho-Hong-Lau construction formally to the whole formal family of categories $ (\cat C_q) $. Intuitively, we obtain a family of matrix factorization categories. More precisely, we get a new algebra $ \Jac(\compl{\chlQ}, W_q) $, a new central element $ ℓ_q ∈ Z(\Jac(\compl{\chlQ}, W_q)) $, a deformation $ \MF(\Jac(\compl{\chlQ}, W_q), ℓ_q) $ of $ \MF(\Jac(\compl{\chlQ}, W) ℓ) $ and a functor $ \compl F_q $ running from $ \cat C_q $ to $ \MF(\Jac(\compl{\chlQ}, W_q), ℓ_q) $. In other words: Thanks to the generality the Cho-Hong-Lau construction, the mirror gets deformed when the input category $ \cat C $ gets deformed.

In the present section, we devote ourselves to the construction of $ W_q $, $ ℓ_q $ and the Jacobi algebra $ \Jac(\compl{\chlQ}, W_q) $. In \autoref{sec:CHL-defMF}, we define the deformed matrix factorizations category and in \autoref{sec:CHL-functor} we construct our deformed mirror functor. We document our setup in the following convention:

\begin{convention}
\label{conv:CHL-deformed}
The category $ \cat C $, the reference objects $ \RefObjects = \{L_1, …, L_N\} $, the basis elements $ X_e $, $ Y_e $, $ \coid_{L_i} $ and the pairing $ ⟨-, -⟩ $ are as in \autoref{conv:CHL-category}. Let $ \cat C_q $ be a deformation of $ \cat C $ over a deformation base $ (B, \mathfrak{m}) $. We assume $ \cat C_q $ is (deformed) cyclic on odd elements:
\begin{equation*}
⟨μ_q (X_{e_{k+1}}, …, X_{e_2}), X_{e_1}⟩ = ⟨μ_q (X_{e_k}, …, X_{e_1}), X_{e_{k+1}}⟩.
\end{equation*}
We assume that $ \cat C_q $ is strictly unital with the same identities $ \id_X $ as in $ \cat C $:
\begin{equation*}
μ_q^1 (\id_X) = 0, \quad μ_q^2 (a, \id_X) = a, \quad μ_q^2 (\id_X, a) = (-1)^{|a|} a, \quad μ_q^{≥3} (…, \id_X, …) = 0.
\end{equation*}
The letter $ \DefRefObjects $ denotes the subcategory of $ \cat C_q $ given by the objects $ L_1, …, L_N ∈ \RefObjects $.
\end{convention}

In overview, our construction of the deformed Landau-Ginzburg model $ (\Jac(\compl{\chlQ}, W_q), ℓ_q) $ proceeds as follows: We model the algebra $ \Jac(\compl{\chlQ}, W_q) $ still via the CHL quiver $ \chlQ $ and relations. More precisely, we regard the enlarged version $ B \htensor \compl{ℂ\chlQ} $ instead of $ \compl{ℂ\chlQ} $. The deformed superpotential $ W_q $ lives in this enlarged algebra instead of $ \compl{ℂ\chlQ} $. The deformed Jacobi algebra $ \Jac(\compl{\chlQ}, W_q) $ is defined as quotient of $ B \htensor \compl{ℂ\chlQ} $ instead of $ \compl{ℂ\chlQ} $. The deformed potential $ ℓ_q $ lives in the center of $ \Jac(\compl{\chlQ}, W_q) $. We develop in parallel a variant in case $ \RefObjects $ is of bounded type. For the variant in case of bounded type, we need an additional condition on the deformed products $ μ_q $. We refer to this condition as being of slow growth, and the condition already implies that $ \RefObjects $ is of bounded growth. All terminology is collected in \autoref{tab:CHL-defLG-notions}.

\begin{table}
\centering
\morearraystretch
\begin{tabular}{@{}ccc@{}}
\textbf{Gadget} & \textbf{Classical} & \textbf{Deformed} \\\hline
Completed quiver algebra & $ \compl{ℂ\chlQ} $ & $ B \htensor \compl{ℂ\chlQ} $ \\
(slow growth) & $ ℂ\chlQ $ & $ B \htensor ℂ\chlQ $ \\\hline
Superpotential & $ W ∈ \compl{ℂ\chlQ} $ & $ W_q ∈ B \htensor \compl{ℂ\chlQ} $ \\
(slow growth) & $ W ∈ ℂ\chlQ $ & $ W_q ∈ B \htensor ℂ\chlQ $ \\\hline
Relations & $ R_e ∈ \compl{ℂ\chlQ} $ & $ R_{q, e} ∈ B \htensor \compl{ℂ\chlQ} $ \\
(slow growth) & $ R_e ∈ ℂ\chlQ $ & $ R_{q, e} ∈ B \htensor ℂ\chlQ $ \\\hline
Jacobi algebra & $ \Jac(\compl{\chlQ}, W) = \frac{\compl{ℂ\chlQ}}{~\closure{(∂_{x_e} W)}~} $ & $ \Jac(\compl{\chlQ}, W_q) = \frac{B \htensor \compl{\chlQ}}{~\closure{(∂_{x_e} W)}^{\tensor}~} $ \\
(slow growth) & $ \Jac(\chlQ, W) = \frac{ℂ\chlQ}{(∂_{x_e} W)} $ & $ \Jac(\chlQ, W_q) = \frac{B \htensor ℂ\chlQ}{~\closure{(∂_{x_e} W)}~} $ \\\hline
Potential & $ ℓ ∈ \Jac(\compl{\chlQ}, W) $ & $ ℓ_q ∈ \Jac(\compl{\chlQ}, W_q) $ \\
(slow growth) & $ ℓ ∈ \Jac(\chlQ, W) $ & $ ℓ_q ∈ \Jac(\chlQ, W_q) $
\end{tabular}
\caption{Terminology of deformed Cho-Hong-Lau construction}
\label{tab:CHL-defLG-notions}
\end{table}

\begin{definition}
\label{def:CHL-defLG-slowgrowth}
$ \DefRefObjects $ is of \emph{slow growth} if for all morphisms $ m_1, …, m_k $ in $ \cat C $ and for every $ n ∈ ℕ $ there exists an $ l_0 ∈ ℕ $ such that
\begin{equation*}
∀l ≥ l_0: \quad μ_q^{k+l} (m_k, …, m_1, b, …, b) ∈ \mathfrak{m}^n \Hom(\RefObjects, X_{k+1}).
\end{equation*}
\end{definition}

\begin{remark}
If $ \DefRefObjects $ is of slow growth, then $ \RefObjects $ is automatically of bounded growth.
\end{remark}

We are now ready to define the deformed relations $ R_{q, e} $ and potential $ ℓ_q $. As the reader may expect, these are simply read off from $ μ_q (b, …, b) $.

\begin{definition}
\label{def:CHL-defLG-def}
The \emph{deformed relations} $ R_{q, e} ∈ B \htensor \compl{ℂ\chlQ} $ and the \emph{deformed potential} are defined by
\begin{equation}
\label{eq:CHL-defLG-Rdef}
\sum_{k ≥ 0} μ_q^k (b, …, b) = ℓ_q \id + \sum_{i, j = 1}^N \sum_{e ∈ E_{ij}} R_{q, e} Y_e.
\end{equation}
The \emph{deformed superpotential} is defined as
\begin{equation*}
W_q = ⟨\sum_{k ≥ 0} μ_q^k (b, …, b), b⟩ ∈ B \htensor \compl{ℂ\chlQ}.
\end{equation*}
The \emph{deformed Jacobi algebra} is defined as
\begin{equation*}
\Jac(\compl{\chlQ}, W_q) = \frac{B \htensor \compl{ℂ\chlQ}}{~\closure{(∂_{x_e} W)}^{\tensor}~}.
\end{equation*}
The \emph{deformed Landau-Ginzburg model} is the pair $ (\Jac(\compl{\chlQ}, W_q), ℓ_q) $. If $ \DefRefObjects $ is of slow growth, then $ R_{q, e} $, $ ℓ_q $, $ W_q $ are regarded as elements of $ B \htensor ℂ\chlQ $, the deformed Jacobi algebra is defined as $ \Jac(\chlQ, W_q) = (B \htensor ℂ\chlQ) / \closure{(∂_{x_e} W)} $, and the deformed Landau-Ginzburg model is $ (\Jac(\compl{\chlQ}, W_q), ℓ_q) $.
\end{definition}

\begin{remark}
The element $ W_q ∈ B \htensor \compl{ℂ\chlQ} $ is cyclic, as we shall see in \autoref{th:CHL-defLG-centrality}. Cyclicity for $ W_q $ is to be understood in the sense that $ W_q ∈ B \htensor \compl{ℂ\chlQ}_{\cyc} $, where $ \compl{ℂ\chlQ}_{\cyc} ⊂ \compl{ℂ\chlQ} $ is the subspace of elements invariant under cyclic permutation. More explicitly, $ W_q $ is of the form $ \sum m_i p_i $ with $ m_i ∈ \mathfrak{m}^{→∞} $ and $ p_i ∈ \compl{ℂ\chlQ}_{\cyc} $. This explains the sense in which $ W_q $ is cyclic.

The derivatives $ ∂_{x_e} W_q $ are elements of $ B \htensor \compl{ℂ\chlQ} $. More precisely, the derivative $ ∂_{x_e} W_q $ is defined as $ \lim_k ∂_{x_e} π_k (W_q) $ where $ π_k $ denotes the projection $ π_k: B \htensor \compl{ℂ\chlQ}_{\cyc} → B/\mathfrak{m}^k ¤ \compl{ℂ\chlQ}_{\cyc} $. More explicitly, if $ W_q = \sum m_k p_k $ with $ m_k ∈ \mathfrak{m}^{→∞} $ and $ p_k ∈ \compl{ℂ\chlQ}_{\cyc} $, then $ ∂_{x_e} W_q = \sum m_k ∂_{x_e} p_k $. This explains how to understand $ ∂_{x_e} W_q $.
\end{remark}

\begin{remark}
The definition of the relations $ R_{q, e} $ is understood as follows: The chunk of $ R_{q, e} $ read off from $ μ_q^k (b, …, b) $ for a certain $ k ≥ 0 $ consists of paths in $ \chlQ $ of length $ k $, weighted with deformation parameters. More precisely, this chunk lies in $ B \htensor ℂ\chlQ_k $. The total relation $ R_{q, e} $ is the sum of these chunks over $ k ≥ 0 $. As such, the deformed relations $ R_{q, e} $ all lie in $ B \htensor \compl{ℂ\chlQ} $. As in the classical case, the relation $ R_{q, e} $ only contain paths running from $ h(e) $ to $ t(e) $ for every $ e ∈ E_{ij} $.

Similarly, the potential $ ℓ_q $ lies in $ B \htensor \compl{ℂ\chlQ} $. We shall from now on typically denote by $ ℓ_q $ its projection to $ \Jac(\compl{\chlQ}, W_q) $. As in the classical case, the only paths contained in $ ℓ_q $ are loops of $ \chlQ $.
\end{remark}

\begin{remark}
In contrast to the classical case, the category $ \cat C_q $ may also have infinitesimal curvature. This is not a problem. We simply start counting from $ k = 0 $ instead of $ k = 1 $ in \eqref{eq:CHL-defLG-Rdef}.
\end{remark}

\begin{remark}
Within the present \autoref{sec:CHL}, we have decided to stick to the following closure notations: If $ X ⊂ \compl{ℂ\chlQ} $, then $ \closure{X} $ is the closure with respect to the Krull topology. If $ X ⊂ B \htensor \compl{ℂ\chlQ} $ or $ X ⊂ B \htensor ℂ\chlQ $, then $ \closure{X} $ denotes the closure with respect to the $ \mathfrak{m} $-adic topology. If $ X ⊂ B \htensor \compl{ℂ\chlQ} $, then $ \closure{X}^{\tensor} $ denotes the closure with respect to the tensor topology. For more information on the tensor topology, we refer to \autoref{sec:flatness-closedness}.
\end{remark}

In the rest of this section, we check properties of the deformed Landau-Ginzburg model. Our first milestone is a deformed version of \autoref{th:CHL-classical-centrality}:

\begin{lemma}
\label{th:CHL-defLG-centrality}
Assume \autoref{conv:CHL-deformed}. Then the superpotential $ W_q ∈ B \htensor \compl{ℂ\chlQ} $ is cyclic and we have $ R_{q, e} = ∂_{x_e} W_q $. The potential $ ℓ_q ∈ \Jac(\compl{\chlQ}, W_q) $ is central. The analogous statements hold if $ \DefRefObjects $ is of slow growth.
\end{lemma}

\begin{proof}
The proof is analogous to the proof of the classical version \autoref{th:CHL-classical-centrality}. Cautious is due when working with the completed ideals. We shall therefore spell out a few details. By comparison, the fact that $ \cat C_q $ is allowed to have curvature is technically unproblematic.

Cyclicity of $ W_q $ and the property that $ R_{q, e} = ∂_{x_e} W_q $ follow immediately from the cyclicity assumption in \autoref{conv:CHL-deformed} and are unproblematic in the sense that these properties hold in $ B \htensor \compl{ℂ\chlQ} $.

We shall rather comment in detail on the centrality of the deformed potential $ ℓ_q $: Within $ B \htensor \compl{ℂ\chlQ} ¤ \Hom(\RefObjects, \RefObjects) $, we have
\begin{equation}
\label{eq:CHL-defLG-centralityreason}
\begin{aligned}
0 &= \sum_{k ≥ 0} \sum_{0 ≤ l ≤ k} μ_q^{k-l+1} (b, …, μ_q^l (b, …, b), b, …, b) \\
&= \sum_{k ≥ 1} \sum_{l ≥ 0} μ_q^k (b, …, μ_q^l (b, …, b), …, b) \\
&= \sum_{k ≥ 1} μ_q^k (b, …, \sum_{i = 1}^N ℓ_{q, i} \id_{L_i} + \sum R_e Y_e, …, b).
\end{aligned}
\end{equation}
In the first row, we have used the curved $ A_∞ $-relations. In the second row, we have applied a reordering of the double sum with increasing path length. In the third row, we have reproduced \eqref{eq:CHL-defLG-Rdef}. We claim that up to the two terms $ μ_q^2 (b, ℓ_{q, i} \id_{L_i}) $ and $ μ_q^2 (ℓ_{q, i} \id_{L_i}, b) $, the entire sum lies in $ \closure{\vspan(∂_{x_e} W_q)}^{\tensor} ¤ \Hom(\RefObjects, \RefObjects) $. Indeed, for every $ k≥3 $ and for $ k = 1 $ the summand lies in the intersection of $ B \htensor ℂ\chlQ \vspan(R_e) ℂ\chlQ $ and $ B \htensor ℂ\chlQ_{≥k-1} $. Therefore the series converges, apart from the two special summands, to an element of $ \closure{\vspan(∂_{x_e} W_q)}^{\tensor} ¤ \Hom(\RefObjects, \RefObjects) $. Within $ \Jac(\compl{\chlQ}, W_q) ¤ \Hom(\RefObjects, \RefObjects) $, we conclude that
\begin{equation*}
0 = μ_q^2 (b, ℓ_q \id) + μ_q^2 (ℓ_q \id, b).
\end{equation*}
Let $ 1 ≤ i, j ≤ N $ and $ e ∈ E_{ij} $. Then extracting the $ X_e $-component gives
\begin{equation*}
0 = x_e ℓ_q - ℓ_q x_e \text{ within }  \Jac(\compl{\chlQ}, W_q).
\end{equation*}
We conclude that $ ℓ $ commutes with all arrows in $ \chlQ $, hence with all finite paths and the image of $ B ¤ ℂ\chlQ $. Let now $ x ∈ B \htensor \compl{ℂ\chlQ} $ be an arbitrary element. We shall prove that $ π(x) ℓ = ℓ π(x) $, where $ π: B \htensor \compl{ℂ\chlQ} → \Jac(\compl{\chlQ}, W_q) $ is the projection. The easiest way to achieve this is by approximating $ x $ by finite paths via the tensor topology. Pick any sequence $ (x_n) ⊂ B ¤ ℂ\chlQ $ such that $ x_n → x $ in the tensor topology. Since $ \closure{(∂_{x_e} W_q)}^{\tensor} $ is closed with respect to the tensor topology, the projection $ π $ is continuous. We obtain
\begin{equation*}
0 = π(x_n) ℓ - ℓ π(x_n) → π(x) ℓ - ℓ π(x) \text{ within } \Jac(\compl{\chlQ}, W_q).
\end{equation*}
We conclude that $ ℓ $ commutes with any element $ x ∈ \Jac(\compl{\chlQ}, W_q) $ and hence $ ℓ ∈ Z(\Jac(\compl{\chlQ}, W_q)) $. This finishes the third part.

The fourth part of the proof consists of observing that the calculations still hold in case $ \DefRefObjects $ is of slow growth. Indeed, in this case the double sums in the calculation \eqref{eq:CHL-defLG-centralityreason} all become finite and apart from the two special terms lie in $ \closure{\vspan(∂_{x_e} W_q)} $, noting that the closure is taken this time only with respect to the $ \mathfrak{m} $-adic topology. The approximation of $ x ∈ B \htensor ℂ\chlQ $ happens by a sequence $ x_n ∈ B ¤ ℂ\chlQ $ which converges to $ x $ in the $ \mathfrak{m} $-adic topology. This proves the fourth step and finishes the proof.
\end{proof}

%% file: CHL/projectives.tex
\subsection{Projectives of deformed algebras}
\label{sec:CHL-projectives}
In this section, we show that projectives of an algebra and any deformation are in one-to-one correspondence. The purpose of this section is to build a rigorous foundation for our deformed category of matrix factorizations.

\begin{center}
\begin{tikzpicture}
\path (0, 0) node (A) {Projectives of $ A_q $} (8, 0) node[align=center] (B) {Projectives of $ A $};
\path[draw, <->] ($ (A.east)!0.2!(B.west) $) to ($ (A.east)!0.8!(B.west) $);
\end{tikzpicture}
\end{center}

The core idea of matching projectives is as follows: Projectives of $ A $ are direct $ A $-module summands of a free $ A $-module. Given such a decomposition $ A^{⊕n} = P_1 ⊕ … P_k $ into projectives, intuition says that the modules $ P_i $ can be extended to $ A_q $-modules $ P_{q, i} $ such that $ A_q^{⊕n} = P_{q, 1} ⊕ … ⊕ P_{q, k} $. The extensions $ P_{q, i} $ are direct summands of $ A_q^{⊕n} $ and therefore automatically projective. This procedure allows us to turn projectives of $ A $ into projectives of $ A_q $. Conversely, our expectation is that a projective module $ P_q $ of $ A_q $ gives rise to a projective $ P $ of $ A $ by merely setting $ P = P_q / \mathfrak{m} P_q $. This allows us to match projectives of $ A_q $ with projectives of $ A $.

If $ P $ is a projective of $ A $, then $ B \htensor P $ is not a projective of $ A_q $. Instead, the module needs to be tweaked a little. This is best accomplished by realizing $ P $ as image of an $ A $-linear idempotent $ p: A^{\oplus n} → A^{\oplus n} $ and lifting $ p $ to an $ A_q $-linear idempotent $ p_q: A_q^{\oplus n} → A_q^{\oplus n} $. We achieve this by an adaption of the classical idempotent lifting argument to the case of formal deformations:

\begin{lemma}
Let $ A $ be an algebra and $ A_q = (B \htensor A, μ_q) $ a deformation. Let $ p: A^{\oplus n} → A^{\oplus n} $  be an $ A $-linear idempotent in the sense that $ p^2 = p $. Then there exists an $ A_q $-linear idempotent $ p_q: A_q^{\oplus n} → A_q^{\oplus n} $ with leading term $ p $.
\end{lemma}

\begin{proof}
The idea is to lift $ p $ to an arbitrary $ A_q $-module map and then turn it iteratively into a projection. As a first step, let $ p_0: A_q^{⊕n} → A_q^{⊕n} $ be any $ A_q $-linear continuous lift of $ p $, for instance given by setting $ p_0 (e_i) = p(e_i) ∈ A^{⊕n} ⊂ A_q^{⊕n} $ where $ e_1, …, e_n $ are the basis vectors of $ A_q^{⊕n} $.

As a second step, we construct a sequence of $ A_q $-module maps $ p_k: A_q^{⊕n} → A_q^{⊕n} $ such that $ p_k^2 - p_k ∈ \mathfrak{m}^{2^k} \End(A_q^{⊕n}) $ and $ p_{k+1} - p_k ∈ \mathfrak{m}^{2^k} \End(A_q^{⊕n}) $. The first item in the sequence is the already constructed map $ p_0 $. Assume for induction that the sequence has already been constructed up to $ p_k $. Then put $ ε = p_k^2 - p_k $. By induction hypothesis, $ ε: A_q^{⊕n} → A_q^{⊕n} $ is an $ A_q $-module map which lies in $ \mathfrak{m}^{2^k} \End(A_q^{⊕n}) $.

The trick is to set $ p_{k+1} = p_k ± ε $ in order to render $ p_{k+1}^2 - p_{k+1} $ of lesser order than $ p_k^2 - p_k = ε $. The correct sign is neither plus or minus in general, but differs for two parts of $ ε $. More precisely, we split $ ε $ into a part mapping to the image of $ p_k $ and a part almost mapping to the kernel of $ p_k $:
\begin{equation*}
ε = p_k ∘ ε + (p_k - \id) ∘ ε.
\end{equation*}
The right sign for the first part is $ -1 $, while the sign for the second part is $ +1 $. In consequence, we put
\begin{equation*}
p_{k+1} = p_k - p_k ∘ ε + (\id - p_k) ∘ ε.
\end{equation*}
Since $ p_k $ is a $ A_q $-linear, so is $ p_{n+1} $. Noting that $ ε $ commutes with $ p_k $ by definition, we get
\begin{align*}
p_{k+1}^2 - p_{k+1} &= (p_k - p_k ε + (\id - p_k) ε)^2 - (p_k - p_k ε + (\id - p_k) ε) \\
&= (p_k^2 - p_k) - 2 p_k^2 ε + 2 (p_k - p_k^2) ε + p_k ε - (\id - p_k) ε + \landau(ε^2) \\
&= p_k ε + (\id - p_k) ε - 2 p_k^2 ε + 2 (p_k - p_k^2) ε + p_k ε - (\id - p_k) ε + \landau(ε^2) \\
&= 4 (p_k - p_k^2) ε + \landau(ε^2) = \landau(ε^2).
\end{align*}
In the calculation, we have written $ \landau(ε^2) $ for any expressions containing an $ ε^2 $ factor. We conclude that $ p_{k+1} ∈ (\mathfrak{m}^{2^k})^2 \End(A_q^{⊕n}) = \mathfrak{m}^{2^{k+1}} \End(A_q^{⊕n}) $. This finishes the construction of $ p_{k+1} $ and thereby the inductive construction of the sequence $ (p_k) $.

Finally, the sequence $ (p_k) $ converges to a limit map $ p_q: A_q^{⊕n} → A_q^{⊕n} $. It can be explicitly described as limit of $ p_k $ in every matrix entry, where matrix representation is taken with respect to the direct sum description $ A_q^{⊕n} $. We conclude that $ p_q $ is an $ A_q $-linear map with $ p_q^2 = p_q $. Since $ x_{k + 1} - x_k ∈ \mathfrak{m} \End(A_q^{⊕n}) $ for $ k ≥ 0 $, the map $ p_q $ reduces to $ p $ modulo $ \mathfrak{m} $. This finishes the proof.
\end{proof}

Let $ P ⊂ A_q^{⊕n} $ be a submodule, not necessarily projective. Regard the projection map $ π: B \htensor A → A $. Then $ π(P) ⊂ A^{⊕n} $ is an $ A $-submodule of $ A^{⊕n} $. To avoid confusion, let us write $ μ_q $ for the deformed product of $ A_q $ and $ μ $ for the product of $ A $. In these terms, we have for $ a ∈ A $ and $ x ∈ B \htensor A^{⊕n} $ that $ μ(a, π(x)) = π(μ_q (a, x)) $. Now if $ x ∈ P $, then $ μ_q (a, x) ∈ P $ and we get $ μ(a, π(x)) = π(μ_q (a, x)) ∈ π(P) $. This shows that $ π(P) $ is an $ A $-submodule. With this observation in mind, we can make the following statement:

\begin{lemma}
\label{th:CHL-projectives-piprojective}
Let $ P ⊂ A_q^{⊕n} $ be a projective module. Then the $ A $-module $ π(P) $ is projective.
\end{lemma}

\begin{proof}
Write $ A_q^{⊕n} = P ⊕ Q $ for some $ A_q $-module $ Q $. As kernels of projections, both are automatically closed with respect to the $ \mathfrak{m} $-adic topology, in particular pseudoclosed. According to \autoref{th:prelim-submodules-criterion}, we get that $ A^{⊕n} = π(P) ⊕ π(Q) $ as vector spaces. As observed before, both $ π(P) $ and $ π(Q) $ are actually $ A $-submodules of $ A^{⊕n} $. We conclude that $ π(P) $ is a projective $ A $-module. This finishes the proof.
\end{proof}

The endomorphism space $ \Hom_{A_q} (A_q^{⊕n}, A_q^{⊕n}) $ is the same as $ A_q^{⊕n} = B \htensor A^{⊕n} $ as $ B $-module. If $ A_q^{⊕n} = P_1 ⊕ … ⊕ P_k $ is a decomposition into $ A_q $-submodules, then every hom space $ \Hom_{A_q} (P_i, P_j) $ can be naturally interpreted as $ B $-linear subspace of $ B \htensor A^{⊕n} $. In fact, $ B \htensor A^{⊕n} $ is their direct sum.

\begin{lemma}
\label{th:CHL-projectives-homspaces}
Let $ A_q^{⊕n} = P_1 ⊕ … ⊕ P_k $ be a decomposition as $ A_q $-module. Then $ \Hom_{A_q} (P_i, P_j) ⊂ B \htensor A^{⊕n} $ is pseudoclosed and quasi-flat, and we have $ π(\Hom_{A_q} (P_i, P_j)) = \Hom_A (π(P_i), π(P_j)) $.
\end{lemma}

\begin{proof}
We have $ A_q^{⊕n^2} = \Hom_{A_q} (A_q^{⊕n}, A_q^{⊕n}) = \bigoplus_{i, j} \Hom_{A_q} (P_i, P_j) $. By \autoref{th:prelim-submodules-criterion}, we conclude that every $ \Hom_{A_q} (P_i, P_j) $ is quasi-flat and pseudoclosed.

It remains to show that $ π (\Hom_{A_q} (P_i, P_j)) = \Hom_A (π(P_i), π(P_j)) ⊂ A_q^{⊕n^2} $. Pick any $ A_q $-linear morphism $ φ: P_i → P_j $. We claim that $ π(φ): π(P_i) → π(P_j) $ is an $ A $-module map. Indeed, for $ a ∈ A $ and $ x ∈ P_i $ we have
\begin{equation*}
π(φ) (μ(a, π(x))) = π(φ(μ_q (a, x))) = π(μ_q(a, φ(x))) = μ(a, π(φ(x))) = μ(a, π(φ)(x)).
\end{equation*}
Conversely, let $ φ: π(P_i) → π(P_j) $ be an $ A $-module morphism. We shall construct an $ A_q $-module map $ φ_q: P_i → P_j $ such that $ π(φ_q) = φ $. The idea is to view the composition $ \tilde{φ}: A^{⊕n} \projects π(P_i) → π(P_j) \embeds A^{⊕n} $. We can write $ \tilde{φ} $ as an element of $ A^{⊕n^2} $. In particular, we obtain a map of $ A_q $-modules $ φ_q: P_i \embeds A_q^{⊕n^2} → A_q^{⊕n^2} \projects P_j $ with $ π(φ_q) = φ $. This shows $ π (\Hom_{A_q} (P_i, P_j)) = \Hom_A (π(P_i), π(P_j)) $ and finishes the proof.
\end{proof}

Let us set up more terminology. We denote by $ \Proj A $ the category of finitely generated projective $ A $-modules and by $ \Proj A_q $ the category of finitely generated $ A_q $-modules. Both are ordinary unital $ ℂ $-linear categories. We use the terminology of $ A_∞ $-deformations for these two categories, even though the only product of both is an ordinary associative composition with different sign rule. In terms of \autoref{def:prelim-freemodules-objectcloning}, we shall prove that $ \Proj A_q $ is a loose object-cloning deformation of $ \Proj A $.

The first step is to define the mapping $ F: \Ob\Proj A_q → \Ob\Proj A $ as given by $ F(P) = P/(\mathfrak{m} · P) $. Obviously, the quotient $ P/ (\mathfrak{m} · P) $ carries a natural action of $ A = A_q / \mathfrak{m} A_q $, rendering it an $ A $-module. This mapping $ F $ is the generalization of the mapping $ P ↦ π(P) $ to projectives not presented as submodules of free modules. We shall now prove that $ F(P) $ is actually projective:

\begin{lemma}
Let $ P ∈ \Proj A_q $. Then the $ A $-module $ F(P) = P / (\mathfrak{m} · P) $ is projective.
\end{lemma}

\begin{proof}
It suffices to regard the case of a submodule $ P ⊂ A_q^{⊕n} $. We claim that $ P / (\mathfrak{m} · P) \cong π(P) $ as $ A $-modules. Regard the $ A $-linear projection map $ π: P → π(P) $. The map is surjective. Its kernel is $ P ∩ \mathfrak{m} A^{⊕n} $. Since $ P $ is a direct summand of $ B \htensor A^{⊕n} $ as $ B $-module, $ P ⊂ B \htensor A^{⊕n} $ is quasi-flat and pseudoclosed. In particular we have $ P ∩ \mathfrak{m} A^{⊕n} = \mathfrak{m} P = \mathfrak{m} · P $. We obtain an isomorphism of $ A $-modules $ P / (\mathfrak{m} · P) \isoto π(P) $. By \autoref{th:CHL-projectives-piprojective}, $ π(P) $ is a projective $ A $-module and we finish the proof by concluding that $ P/ (\mathfrak{m} · P) ∈ \Proj A $.
\end{proof}

\begin{proposition}
\label{th:CHL-projectives-correspondence}
The category $ \Proj A_q $ is an essentially surjective loose object-cloning deformation of $ \Proj A $ over $ B $ via the mapping $ F: \Ob\Proj A_q → \Ob\Proj A $.
\end{proposition}

\begin{proof}
The proof consists of two steps: First we show that $ F $ is essentially surjective. Second, we show that $ \Proj A_q $ becomes an object cloning deformation of $ \Proj A $.

To see that $ F $ is essentially surjective, let $ P ∈ \Proj A $ be any finitely generated projective $ A $ module. There exists an isomorphic module $ P' $ and a decomposition $ A^{\oplus n} = P' \oplus Q $ for some $ n \in ℕ $ and some module $ Q $. Pick a lift $ A_q^{\oplus n} = P'_q \oplus Q_q $. Then $ P'_q / (\mathfrak{m} · P'_q) \cong P $, since $ \mathfrak{m} P'_q = \mathfrak{m} · P'_q $. This shows that $ P $ is reached by $ F $ up to isomorphism. In other words, $ F $ is essentially surjective.

To see that $ \Proj A_q $ is a loose object-cloning deformation of $ \Proj A $, we have to provide linear isomorphisms
\begin{equation*}
ψ_{P, Q}: \Hom_{A_q} (P, Q) / (\mathfrak{m} · \Hom_{A_q} (P, Q)) \isoto \Hom_A (F(P), F(Q)) \quad \text{for } P, Q ∈ \Proj A_q.
\end{equation*}
Let $ φ ∈ \Hom_{A_q} (P, Q) $. Then $ φ $ descends to a map $ F(P) → F(Q) $. We define $ ψ_{P, Q} (φ) $ as this descended map. We need to check that $ ψ_{P, Q} $ is an isomorphism. For this, we note that $ ψ_{P, Q} $ is defined purely in terms of the module structure of $ P $ and $ Q $. It therefore suffices to show bijectivity of $ ψ_{P, Q} $ only in case $ P $ and $ Q $ are embedded as submodules of a free module $ A_q^{⊕n} $. In this case, the hom space $ \Hom_{A_q} (P, Q) $ has an interpretation as pseudoclosed quasi-flat $ B $-submodule of $ B \htensor A^{⊕n^2} $. The map $ ψ_{P, Q} $ is in this case the map $ \Hom_{A_q} (P, Q) / (\mathfrak{m} · \Hom_{A_q} (P, Q)) → \Hom_A (π(P), π(Q)) $ simply induced from the projection to zeroth order $ \Hom_{A_q} (P, Q) → \Hom_A (π(P), π(Q)) $. The map $ ψ_{P, Q} $ is surjective since $ \Hom_A (π(P), π(Q)) = π(\Hom_{A_q} (P, Q)) $ by \autoref{th:CHL-projectives-homspaces}, and injective since $ \Hom_{A_q} (P, Q) $ is quasi-flat. This shows that $ ψ_{P, Q} $ is an isomorphism.

Finally, we have to check that composition in the category $ \Proj A_q $ reduces to composition in $ \Proj A $ via $ \{ψ_{P, Q}\}_{P, Q} $ once $ \mathfrak{m} $ is divided out. This is however immediate, since composition commutes with descending to the quotient by $ \mathfrak{m} $. This finishes the proof.
\end{proof}

\begin{remark}
\label{rem:CHL-projectives-explicit}
The correspondence between projectives of $ A_q $ and $ A $ can typically be made more explicit. Assume for instance, an element $ v ∈ A $ is an idempotent in the sense that $ v^2 = v $. Then $ Av $ is a projective of $ A $ since $ Av $ is the image of the $ A $-linear idempotent map $ (-) v: A → A $. Moreover, if $ v^2 = v $ within $ A_q $, then $ A_q v $ is a projective of $ A_q $, since it is the image of the idempotent $ A_q $-linear map $ (-) v: A_q → A_q $.

The $ A $-projective $ F(A_q v) = A_q v / (\mathfrak{m} · A_q v) $ can be naturally identified with $ Av $ via the map $ φ: A_q v / (\mathfrak{m} · A_q v) \isoto Av $ given by $ φ(a) = π(a)v $ for $ a ∈ A_q v ⊂ B \htensor A $, where $ π: B \htensor A → A $ is the standard projection. Since $ μ_q $ is a deformation of $ μ $, we have $ φ(μ_q (a, b)) = π(μ_q(a, b))v = μ(π(a), π(b)) v $ and conclude that $ φ $ is $ A $-linear. The map $ φ $ is surjective since for $ a ∈ A $ we have $ φ(a) = av $. The map $ φ $ is injective, since $ π(a) v = 0 $ for $ a ∈ A_q v $ implies $ π(μ_q (a, v)) = 0 $, hence $ μ_q (a, v) ∈ \mathfrak{m} ∩ A_q v ⊂ \mathfrak{m} · A_q v $, since $ A_q v ⊂ B \htensor A $ is quasi-flat.

Hom spaces between projectives can also be identified. Let $ v, w $ be two idempotents of $ A $ still idempotent in $ A_q $. Then we have the natural identification
\begin{equation*}
\frac{\Hom_{A_q} (A_q v, A_q w)}{m · \Hom_{A_q} (A_q v, A_q w)} \isoto \Hom_A (F(A_q v), F(A_q w)) \isoto \Hom_A (Av, Aw).
\end{equation*}
\end{remark}

%% file: CHL/defMF.tex
\subsection{Deformed matrix factorizations}
\label{sec:CHL-defMF}
In this section, we define deformed categories of matrix factorizations. There are several issues on the road to their definition. In order to satisfy our demand to serve as mirror model in the deformed Cho-Hong-Lau construction, the definition furthermore needs to deviate from what one expects. The first step in this section is to explain these issues. We then provide a successful construction of the deformed category of matrix factorizations for any deformed Landau-Ginzburg model. We finish this section with an explanation how this category becomes an deformation of the ordinary category of matrix factorizations.

Our starting point is a pair $ (A, ℓ) $ consisting of an associative algebra $ A $ with a central element $ ℓ ∈ A $. Recall from \autoref{sec:prelim-mf} that the category of matrix factorizations $ \MF(A, ℓ) $ is a dg category which has as objects pairs of finitely generated projective $ A $-modules $ M, N $ together with $ A $-module morphisms $ f: M → N $ and $ g: N → M $ such that $ f ∘ g = ℓ \id_N $ and $ g ∘ f = ℓ \id_M $.

The starting point for the deformed setup is a pair $ (A_q, ℓ_q) $ of a deformation $ A_q $ of $ A $ and a deformed central element $ ℓ_q ∈ Z(A_q) $. More precisely, $ A_q = (B \htensor A, μ_q) $ shall be an associative deformation of the algebra $ A $ over a deformation base $ B $, in the sense that $ μ_q: (B \htensor A) ¤ (B \htensor A) → B \htensor A $ is a $ B $-linear associative product that reduces to the product $ μ: A ¤ A → A $ once $ \mathfrak{m} $ is divided out. The element $ ℓ_q ∈ Z(A_q) $ is required to be a deformation of $ ℓ $ in the sense that $ ℓ_q - ℓ ∈ \mathfrak{m} A $. We fix this terminology as follows:

\begin{definition}
A \emph{Landau-Ginzburg model} $ (A, ℓ) $ is a pair of an associative algebra and a central element. A \emph{deformation} $ (A_q, ℓ_q) $ \emph{of} $ (A, ℓ) $ consists of an algebra deformation $ A_q = (B \htensor A, μ_q) $ of $ A $ together with a central element $ ℓ_q ∈ Z(A_q) $ which is a deformation of $ ℓ $. Disregarding the reference to $ (A, ℓ) $, we may call $ (A_q, ℓ_q) $ a \emph{deformed Landau-Ginzburg model}.
\end{definition}

The following is a naive candidate for a deformed matrix factorization category: Objects are projective $ A_q $-modules $ M, N $ together with $ A_q $-module maps $ f: M → N $ and $ g: N → M $ such that $ f ∘ g = ℓ_q \id_N $ and $ g ∘ f = ℓ_q \id_M $. The hom space of two such matrix factorizations $ \matf MNfg $ and $ \matf{M'}{N'}{f'}{g'} $ would be defined as $ \Hom_{A_q} (M, M') ⊕ … $ similar to the classical case. The differential $ μ^1 $ on this category would be given by commuting a morphism with $ f, g $ and $ f', g' $ as in the classical case. The product $ μ^2 $ would be given by standard matrix composition. This defines a dg category and a naive candidate for a deformed category of matrix factorizations.

Let us now sketch the issues associated with this naive definition. The leading question is how to interpret this category as a deformation of $ \MF(A, ℓ) $, both on object and morphism level.

The first issue consists of matching projective modules of $ A_q $ with projective modules of $ A $. Even more, we also need that the hom spaces between projective modules of the two kinds match. We have resolved this in \autoref{sec:CHL-projectives}.

The second issue consists of matching the factorization morphisms $ f, g $ for $ ℓ_q $ with factorization morphisms for $ ℓ $. In fact, multiple matrix factorizations of $ (A_q, ℓ_q) $ reduce to the same matrix factorization of $ (A, ℓ) $. A simple example is to multiply $ f $ by an element $ 1+ε $ with $ ε ∈ \mathfrak{m} $ and multiply $ g $ by $ (1+ε)^{-1} = \sum (-ε)^i $. Conversely, it is unclear whether every matrix factorization of $ (A, ℓ) $ extends to a matrix factorization of $ (A_q, ℓ_q) $. This makes a precise correspondence between matrix factorizations of $ (A, ℓ) $ and of $ (A_q, ℓ_q) $ impossible. We resolve this by building $ \MF(A_q, ℓ_q) $ as an object-cloning deformation of $ \MF(A, ℓ) $.

The third issue consists of liberalizing the category $ \MF(A_q, ℓ_q) $ enough so that it can serve as codomain of the deformed mirror functor. As we shall see, our deformed mirror functor does not map to objects $ (M, N, f, g) $ where all compositions $ f ∘ g $ and $ g ∘ f $ are equal to each other. Instead, the compositions will only agree with the predefined central element $ ℓ_q $ on zeroth order and differ per object. This phenomenon is inevitable when starting form a curved $ A_∞ $-deformation. It requires us to admit objects very liberally into the category $ \MF(A_q, ℓ_q) $.

We resolve the third issue as follows: The objects of $ \MF(A_q, ℓ_q) $ are pairs $ (M, N, f, g) $ of projective modules and module maps, but $ f ∘ g $ and $ g ∘ f $ need only equal $ ℓ_q $ up to zeroth order. The difference of $ f ∘ g $ and $ ℓ_q $, and of $ g ∘ f $ and $ ℓ_q $, serves as curvature of $ \MF(A_q, ℓ_q) $. Since our deformed mirror functor typically requires the codomain to carry curvature as well, this resolves the third issue sufficiently.

At this point, we stress the crucial importance of $ A_q $ being a deformation of $ A $ in the sense of \autoref{def:flatness-whatis-Adefo}. Intuitively, if $ A_q $ is smaller than $ B \htensor A $, then its modules have smaller hom spaces as well, which breaks all chances to make the category of matrix factorizations of $ (A_q, ℓ_q) $ a deformation of $ \MF(A, ℓ) $.

\begin{definition}
Let $ (A_q, ℓ_q) $ be a deformed Landau-Ginzburg model. A \emph{deformed matrix factorization} of $ (A_q, ℓ_q) $ consists of two finitely generated projective $ A_q $-modules $ P, Q $ together with $ A_q $-module maps $ f: P → Q $ and $ g: Q → P $ such that $ f ∘ g - ℓ_q \id_Q ∈ \mathfrak{m} · \Hom_{A_q} (Q, Q) $ and $ g ∘ f - ℓ_q \id_P ∈ \mathfrak{m} · \Hom_{A_q} (P, P) $.
\end{definition}

\begin{example}
Let $ A = ℂ[X, Y] $ and regard the trivial deformation $ A_q = (A⟦q⟧, μ) $. Regard the central element $ ℓ = XY $ and the deformed central element $ ℓ_q = XY+q $. The object $ (A, A, X, Y) $ is a matrix factorization of $ (A, ℓ) $. Both $ (A_q, A_q, X, Y) $ and $ (A_q, A_q, X+5q, Y+qX) $ are deformed matrix factorizations of $ (A_q, ℓ_q) $. Both however do not factor to $ ℓ_q $ precisely, but only on zeroth order. In fact, there are no single elements $ f, g ∈ A_q $ such that $ fg = XY + q $ and $ f - X ∈ (q) $ and $ g - Y ∈ (q) $. To see this, assume $ f = X + qz $ and $ g = Y + qw $, then $ fg = XY + q(zY+Xw) + q^2 zw $. In order for this to be equal to $ ℓ_q = XY+q $, we need that $ zY+Xw $ is $ 1 $ on first order, which is impossible. This shows that matrix factorizations need not have strict lifts, but deformed matrix factorizations are rather abundant.
\end{example}

\begin{remark}
As in the classical case, the pair of modules $ (P, Q) $ can also be described as a $ ℤ/2ℤ $-graded module $ M = P ⊕ Q[1] $ (where both graded parts are projective). The pair of morphisms $ (f, g) $ can be described as an odd morphism $ δ: P ⊕ Q[1] → P ⊕ Q[1] $. We shall liberally switch between the two kinds of notation, identifying
\begin{equation*}
\matf PQfg \quad \longleftrightarrow \quad \left(P ⊕ Q[1], \pmat{0 & g \\ f & 0}\right).
\end{equation*}
\end{remark}

We are now ready to define $ \MF(A_q, ℓ_q) $. As in the classical case, if $ (M, δ) $ is a deformed matrix factorization, we denote by $ \tilde{δ} $ the tweaked differential given by $ \tilde{δ}(m) = (-1)^{|m|} δ(m) $.

\begin{definition}
\label{def:CHL-defMF-cat}
Let $ (A_q, ℓ_q) $ be a deformed Landau-Ginzburg model. The \emph{deformed category of matrix factorizations} $ \MF(A_q, ℓ_q) $ is defined as follows:
\begin{itemize}
\item Objects are the deformed matrix factorizations $ (M, δ) $ of $ (A_q, ℓ_q) $.
\item Hom spaces are given by $ \Hom((M, δ_M), (N, δ_N)) = \Hom_{A_q} (M, N) $, naturally $ ℤ/2ℤ $-graded.
\item The curvature of an object $ (M, δ) $ is given by $ μ^0_{(M, δ)} = ℓ_q \id_M - δ^2 $.
\item The differential is given by $ μ^1 (f) = \tilde{δ}_N ∘ f - (-1)^{|f|} f ∘ \tilde{δ}_M $ for $ f ∈ \Hom((M, δ_M), (N, δ_N)) $.
\item The product is given by $ μ^2 (f, g) = (-1)^{‖f‖ |g|} f ∘ g $.
\end{itemize}
\end{definition}

\begin{remark}
In writing $ \MF(A_q, ℓ_q) $, we have abused notation: The category $ \MF(A_q, ℓ_q) $ is not the same as (classical) matrix factorizations of the pair $ (A_q, ℓ_q) $.
\end{remark}

We aim at showing that $ \MF(A_q, ℓ_q) $ is an object-cloning deformation of $ \MF(A, ℓ) $. To make this true, we provide a map $ \Ob\MF(A_q, ℓ_q) → \Ob\MF(A, ℓ) $. The construction of this map is easy and consists of taking the leading term of any matrix factorization:

\begin{definition}
\label{def:CHL-defMF-leadingterm}
Let $ (A_q, ℓ_q) $ be a deformed Landau-Ginzburg model. Let $ (M, δ) $ be a deformed matrix factorization of $ (A, ℓ) $. Then the \emph{leading term} of $ (M, δ) $ is the matrix factorization $ \mfconst (M, δ) $ of $ (A, ℓ) $ given by
\begin{equation*}
\mfconst (M, δ) = (M / (\mathfrak{m} · M), π(δ)).
\end{equation*}
\end{definition}

In this definition, $ π(δ) $ denotes the induced map $ M / (\mathfrak{m} · M) → M / (\mathfrak{m} · M) $ and is automatically an $ A $-module map. The quotient $ M / (\mathfrak{m} · M) $ is understood to be performed in even and odd degree separately. We are now ready to show that $ \MF(A_q, ℓ_q) $ is an object-cloning deformation of $ \MF(A, ℓ) $.

\begin{proposition}
\label{th:CHL-defMF-correspondence}
The category $ \MF(A_q, ℓ_q) $ is a loose object-cloning deformation of $ \MF(A, ℓ) $ along the map $ \mfconst: \Ob\MF(A_q, ℓ_q) → \Ob\MF(A, ℓ) $.
\end{proposition}

\begin{proof}
We divide the proof into three steps: The first step is to show that $ \MF(A_q, ℓ_q) $ satisfies the (curved) $ A_∞ $-axioms, merely regarding its hom spaces as $ ℤ/2ℤ $-graded vector spaces instead of $ B $-modules. The second step is to investigate the $ B $-linear structure on its hom spaces and the shape of the products. The third step is to draw the conclusion that $ \MF(A_q, ℓ_q) $ is indeed an object-cloning deformation of $ \MF(A, ℓ) $.

For the first step, we check all curved $ A_∞ $-relations for $ \MF(A_q, ℓ_q) $ one after another. The first relation reads
\begin{equation*}
μ^1 (μ^0_{(M, δ)}) = \tilde{δ} ∘ (ℓ_q \id_M - δ^2) - (ℓ_q \id_M - δ^2) ∘ \tilde{δ} = 0.
\end{equation*}
Note we have used that $ δ^2 = - \tilde{δ}^2 $ and $ ℓ_q $ is central. The second relation reads
\begin{align*}
μ^1 (μ^1 (f)) + (-1)^{‖f‖} μ^2 (μ^0, f) + μ^2 (f, μ^0) &= \tilde{δ} ∘ (\tilde{δ} ∘ f - (-1)^{|f|} f ∘ \tilde{δ}) \\
& \quad - (-1)^{|f| + 1} (\tilde{δ} ∘ f - (-1)^{|f|} f ∘ \tilde{δ}) ∘ \tilde{δ} \\
& \quad + (-1)^{‖f‖ + |f|} (ℓ_q \id_M - δ^2) ∘ f + f ∘ (ℓ_q \id_M - δ^2) \\
&= - δ^2 ∘ f + f ∘ δ^2 - ℓ_q f + f ℓ_q + δ^2 ∘ f - f ∘ δ^2 = 0.
\end{align*}
We have used again that $ ℓ_q $ is central. The third relation is analogous to the classical case and reads
\begin{align*}
μ^1 (μ^2 (f, g)) + (-1)^{‖g‖} μ^2 (μ^1 (f), g) + μ^2 (f, μ^1 (g)) &= (-1)^{‖f‖ |g|} (\tilde{δ} ∘ f ∘ g - (-1)^{|fg|} f ∘ g ∘ \tilde{δ}) \\
& \quad + (-1)^{|f| |g| + ‖g‖} (\tilde{δ} ∘ f - (-1)^{|f|} f ∘ \tilde{δ}) ∘ g \\
& \quad + (-1)^{‖f‖ ‖g‖} f ∘ (\tilde{δ} ∘ g - (-1)^{|g|} g ∘ \tilde{δ}) = 0.
\end{align*}
The fourth relation is associativity and all other relations vanish.

For the second part of the proof, we give $ \MF(A_q, ℓ_q) $ the structure of loose object-cloning deformation. This entails providing for every two deformed matrix factorizations $ (M, δ_M), (N, δ_N) $ a linear $ ℤ/2ℤ $-graded isomorphism
\begin{equation*}
ψ_{(M, δ_M), (N, δ_N)}: \frac{\Hom_{\MF(A_q, ℓ_q)} ((M, δ_M), (N, δ_N))}{\mathfrak{m} · \Hom_{\MF(A_q, ℓ_q)} ((M, δ_M), (N, δ_N))} \isoto \Hom_{\MF(A, ℓ)} (\mfconst (M, δ_M), \mfconst (N, δ_N)).
\end{equation*}
The hom space $ \Hom_{\MF(A_q, ℓ_q)} ((M, δ_M), (N, δ_N)) $ merely consists of the direct sum of four hom spaces from the category $ \Proj A_q $. Thanks to \autoref{th:CHL-projectives-correspondence}, every of these four hom spaces, quotiented by $ \mathfrak{m} $, comes with a natural isomorphism to the corresponding hom space from the category $ \Proj A $. We now construct the isomorphism $ ψ_{(M, δ_M), (N, δ_N)} $ by simply combining these four isomorphisms on their respective domains.

For the third part of the proof, we explain why $ \MF(A_q, ℓ_q) $ is a loose object-cloning deformation of $ \MF(A, ℓ) $ via $ \mfconst $. Indeed, composition in $ \MF(A_q, ℓ_q) $ is merely given by matrix composition and we have seen in \autoref{th:CHL-projectives-correspondence} that it is compatible via $ ψ $ with composition in $ \Proj A $. Similarly, the differential in $ \MF(A_q, ℓ_q) $ is defined via commuting with $ δ $, which reduces via $ ψ $ to commuting with the induced $ δ $ in $ \Proj A $ and hence to the differential of $ \MF(A, ℓ) $. Finally, the curvature in $ \MF(A_q, ℓ_q) $ is infinitesimal and reduces to zero after dividing out $ \mathfrak{m} $. This proves $ \MF(A_q, ℓ_q) $ a loose object-cloning deformation of $ \MF(A, ℓ) $ via $ \mfconst $ and finishes the proof.
\end{proof}

%% file: CHL/functor.tex
\subsection{Deformed mirror functor}
\label{sec:CHL-functor}
In this section, we construct our deformed mirror functor. The idea is to simply repeat the Cho-Hong-Lau construction using the deformed category $ \cat C_q $ as input instead of $ \cat C $. As we have seen in \autoref{sec:CHL-defLG}, we obtain a deformed Landau-Ginzburg model $ (\Jac(\compl{\chlQ}, W_q), ℓ_q) $. The target of the mirror functor is then the deformed category of matrix factorizations $ \MF(\Jac(\compl{\chlQ}, W_q), ℓ_q) $. We deploy notation and assumptions from \autoref{conv:CHL-deformed}. As in the classical case, we construct functors in the general case and the case of slow growth in parallel.

The first step of this section is to define the deformed functors $ \compl F_q: \cat C_q → \MF(\Jac(\compl{\chlQ}, W_q), ℓ_q) $ and $ F_q: \cat C_q → \MF(\Jac(\chlQ, W_q), ℓ_q) $ on object level. The second step is to define them on morphism level. Finally we check that $ \compl F_q $ and $ F_q $ satisfy the deformed $ A_∞ $-functor relations and their leading terms are the classical Cho-Hong-Lau functors $ \compl{F} $ and $ F $.

An important assumption in this section is that $ \Jac(\compl{\chlQ}, W_q) $ be a deformation of $ \Jac(\compl{\chlQ}, W) $, or $ \Jac(\chlQ, W_q) $ be a deformation of $ \Jac(\chlQ, W) $, in the sense of \autoref{def:flatness-whatis-Adefo}. As such, these algebras come with $ B $-linear algebra isomorphisms
\begin{align*}
\compl{φ}_{\Jac}: \Jac(\compl{\chlQ}, W_q) &\isoto (B \htensor \Jac(\compl{\chlQ}, W), μ_{\Jac, q}), \\
φ_{\Jac}: \Jac(\chlQ, W_q) &\isoto (B \htensor \Jac(\chlQ, W), μ_{\Jac, q}).
\end{align*}

We fix these isomorphisms and will always view $ \Jac(\compl{\chlQ}, W_q) $ and $ \Jac(\chlQ, W_q) $ as deformations of $ \Jac(\compl{\chlQ}, W) $ and $ \Jac(\chlQ, W) $.

\begin{remark}
\label{rem:CHL-functor-projectives}
By \autoref{rem:flatness-whatis-semisimple}, the isomorphisms $ \compl{φ}_{\Jac} $ and $ φ_{\Jac} $ can be assumed to be unital and be $ ℂ\chlQ_0 $-bimodule morphisms. For instance, any vertex element $ L_i ∈ \Jac(\compl{\chlQ}, W_q) $ is mapped to the vertex $ L_i ∈ \Jac(\compl{\chlQ}, W) $, without getting deformed. We also have $ L_i^2 = L_i $ in $ \Jac(\compl{\chlQ}, W_q) $ and $ \Jac(\chlQ, W_q) $, since the same holds in $ \Jac(\compl{\chlQ}, W) $ and $ \Jac(\chlQ, W) $. The two module maps $ (-) L_i: \Jac(\compl{\chlQ}, W_q) → \Jac(\compl{\chlQ}, W_q) $ and $ (-) L_i: \Jac(\chlQ, W_q) → \Jac(\chlQ, W_q) $ are therefore idempotents. We conclude that $ \Jac(\compl{\chlQ}, W_q) L_i $ and $ \Jac(\chlQ, W_q) L_i $ are projectives.
\end{remark}

We start with the explicit description of $ \compl F_q $ and $ F_q $ on object level.
\begin{align*}
\compl F_q (X) &= \left(\bigoplus_{i = 1}^N \Jac(\compl{\chlQ}, W_q) L_i \tensor \Hom_{\cat C} (L_i, X), \quad δ(m) = \sum_{k ≥ 1} (-1)^{‖m‖} μ_q^k (m, b, …, b)\right), \\
F_q (X) &= \left(\bigoplus_{i = 1}^N \Jac(\chlQ, W_q) L_i ¤ \Hom_{\cat C} (L_i, X), \quad δ(m) = \sum_{k ≥ 1} (-1)^{‖m‖} μ_q^k (m, b, …, b)\right).
\end{align*}
Note that each map $ δ $ is well-defined since $ \DefRefObjects $ is of slow growth.

\begin{lemma}
\label{th:CHL-functor-ismf}
For $ X ∈ \cat C $ the object $ \compl F_q (X) $ is indeed a deformed matrix factorization of $ (\Jac(\compl{\chlQ}, W_q), ℓ_q) $. If $ \DefRefObjects $ is of slow growth, then $ F_q (X) $ is a deformed matrix factorization of $ (\Jac(\chlQ, W_q), ℓ_q) $.
\end{lemma}

\begin{proof}
We merely check the case of $ \compl F_q (X) $. It is our task to show that $ ℓ_q \id_{\compl F_q (X)} - δ^2 $ is infinitesimal. Calculating in $ \Jac(\compl{\chlQ}, W_q) ¤ \Hom(\RefObjects, X) $, we find
\begin{align*}
ℓ_q m - δ^2 (m) &= ℓ_q m - δ^2 (m) \\
&= ℓ_q m + \sum_{i, j ≥ 0} μ_q^{i+1} (μ_q^{j+1} (m, b …, b), b, …, b) \\
&= ℓ_q m - \sum_{i, j, k ≥ 0} μ_q (m, \underbrace{b, …, b}_{i}, μ_q (\underbrace{b, …, b}_{j}), \underbrace{b, …, b}_k) - (-1)^{‖m‖} \sum_{i ≥ 0} μ_q^{i+2} (μ_q^0, m, b, …, b) \\
&= (-1)^{|m|} \sum_{i ≥ 0} μ_q^{i+2} (μ_q^0, m, b, …, b) \\
& ∈ \mathfrak{m} · (\Jac(\compl{\chlQ}, W_q) ¤ \Hom(\RefObjects, X)).
\end{align*}
Here we have used the curved $ A_∞ $-relation for $ μ_q $ and $ \sum_{l ≥ 0} μ_q^l (b, …, b) = ℓ_q \id_{\DefRefObjects} $. We conclude that $ ℓ_q \id_{\compl F_q (X)} - δ^2 $ is infinitesimal and $ \compl F_q (X) $ is a deformed matrix factorization. The analogous calculations show that $ F_q (X) $ is a matrix factorization when $ \DefRefObjects $ is of bounded growth.
\end{proof}

\begin{remark}
Explicitly, the curvature of the object $ \compl F_q (X) $ is the endomorphism of $ \compl F_q (X) $ given by
\begin{equation*}
μ^0_{\MF, \compl F_q (X)} (m) = ℓ_q m - δ^2 (m) = (-1)^{|m|} \sum_{l ≥ 0} μ_q^{l+2} (μ_q^0, m, b, …, b).
\end{equation*}
\end{remark}

We now define the deformed CHL functor in analogy to the classical CHL functor:

\begin{definition}
The \emph{deformed CHL functor} $ \compl F_q $ is the mapping
\begin{align*}
\compl F_q: \cat C_q &\verylongto \MF(\Jac(\compl{\chlQ}, W_q), ℓ_q), \\
X &\verylongmapsto \compl F_q(X), \\
\compl F_q (m_k, …, m_1)(m) &= (-1)^{(‖m_1‖ + … + ‖m_k‖) ‖m‖ + 1} \sum_{l ≥ 0} μ_q^{k+l+1} (m_k, …, m_1, m, b, …, b) \\
& \qquad \text{for } m_i: X_i → X_{i+1}, \quad m ∈ \compl F_q (X_1).
\end{align*}
In case $ \DefRefObjects $ is of slow growth, the functor $ F_q: \cat C_q → \MF(\Jac(\chlQ, W_q), ℓ_q) $ is defined analogously.
\end{definition}

\begin{remark}
Note that we put $ F_q^0 = 0 $, as opposed to $ F^0_q ≔ \sum_{l ≥ 0} μ_q^{l+1} (-, b, …, b) $.
\end{remark}

We shall now prove that $ \compl F_q $ and $ F_q $ are actually functors. We shall also show that their leading terms are the classical functors $ \compl F $ and $ F $. The categories $ \MF(\Jac(\compl{\chlQ}, W_q), ℓ_q) $ and $ \MF(\Jac(\chlQ, W_q), ℓ_q) $ are only object-cloning deformations of $ \MF(\Jac(\compl{\chlQ}, W), ℓ) $ and $ \MF(\Jac(\chlQ, W), ℓ) $, so making the statement on leading terms rigorous we need to provide an identification map on object level. In \autoref{def:CHL-defMF-leadingterm}, we have already provided such identification maps
\begin{align*}
\mfconst: \Ob\MF(\Jac(\compl{\chlQ}, W_q), ℓ_q) &→ \Ob\MF(\Jac(\compl{\chlQ}, W), ℓ), \\
\mfconst: \Ob\MF(\Jac(\chlQ, W_q), ℓ_q) &→ \Ob\MF(\Jac(\chlQ, W), ℓ).
\end{align*}
The map $ \mfconst $ sends $ \compl F_q (X) $ and $ F_q (X) $ to
\begin{align*}
\mfconst (\compl F_q (X)) &= \left(\bigoplus_{i = 1}^N \frac{\Jac(\compl{\chlQ}, W_q) L_i}{\mathfrak{m} · \Jac(\compl{\chlQ}, W_q) L_i} \tensor \Hom_{\cat C} (L_i, X), \quad δ(m) = \sum_{k ≥ 1} (-1)^{‖m‖} (π ¤ \id) (μ_q^k (m, b, …, b))\right), \\
\mfconst (F_q (X)) &= \left(\bigoplus_{i = 1}^N \frac{\Jac(\chlQ, W_q) L_i}{\mathfrak{m} · \Jac(\chlQ, W_q) L_i} \tensor \Hom_{\cat C} (L_i, X), \quad δ(m) = \sum_{k ≥ 1} (-1)^{‖m‖} (π ¤ \id) (μ_q^k (m, b, …, b))\right).
\end{align*}
Here $ π $ denotes the projection from the Jacobi algebra to its quotient by $ \mathfrak{m} $. Meanwhile, the object $ \compl F (X) $ is given by a sum of projectives of the form $ \Jac(\compl{\chlQ}, W) L_i $. As discussed in \autoref{rem:CHL-projectives-explicit}, the difference is entirely cosmetic and we shall naturally identify
\begin{equation}
\label{eq:CHL-functor-objectidentification}
\begin{aligned}
\mfconst (\compl F_q (X)) \isoto \compl F (X), \\
\mfconst (F_q (X)) \isoto F(X).
\end{aligned}
\end{equation}
A similar statement holds for the hom spaces. Regard a hom space between two image objects $ \compl F_q (X) $ and $ \compl F_q (Y) $, divide out $ \mathfrak{m} $, and identify the quotient via the map $ ψ $ from \autoref{sec:CHL-projectives}. According to \autoref{rem:CHL-projectives-explicit}, the result can further be identified naturally with the hom space in $ \MF(\Jac(\compl{\chlQ}, W), ℓ) $ or $ \MF(\Jac(\chlQ, W), ℓ) $:
\begin{equation}
\label{eq:CHL-functor-homidentification}
\begin{aligned}
\frac{\Hom(\compl F_q (X), \compl F_q (Y))}{\mathfrak{m} · \Hom(\compl F_q (X), \compl F_q (Y))} &\xrightarrow[~~\sim~~]{ψ} \Hom\left(\frac{\compl F_q (X)}{\mathfrak{m} · \compl F_q (X)}, \frac{\compl F_q (Y)}{\mathfrak{m} · \compl F_q (Y)}\right) \isoto \Hom(\compl F (X), \compl F (Y)), \\
\frac{\Hom(F_q (X), F_q (Y))}{\mathfrak{m} · \Hom(F_q (X), F_q (Y))} &\xrightarrow[~~\sim~~]{ψ} \Hom\left(\frac{F_q (X)}{\mathfrak{m} · F_q (X)}, \frac{F_q (Y)}{\mathfrak{m} · F_q (Y)}\right) \isoto \Hom(F (X), F (Y)).
\end{aligned}
\end{equation}
We claim that with respect to this identification, the leading terms of the two functors $ \compl F_q $ and $ F_q $ are the classical Cho-Hong-Lau functors $ \compl F $ and $ F $:

\begin{theorem}
\label{th:CHL-functor-th}
Assume \autoref{conv:CHL-deformed} and that $ \Jac(\compl{\chlQ}, W_q) $ is a deformation of $ \Jac(\compl{\chlQ}, W) $. Then the mapping $ \compl F_q $ defines a functor of loose object-cloning $ A_∞ $-deformations
\begin{equation*}
\compl F_q: \cat C_q → \MF(\Jac(\compl{\chlQ}, W_q), ℓ_q).
\end{equation*}
The leading term of $ \compl F_q $ via the identifications \eqref{eq:CHL-functor-objectidentification} and \eqref{eq:CHL-functor-homidentification} is the classical Cho-Hong-Lau functor $ \compl F $.

Assume instead \autoref{conv:CHL-deformed}, that $ \DefRefObjects $ is of slow growth and that $ \Jac(\chlQ, W_q) $ is a deformation of $ \Jac(\chlQ, W) $. Then the mapping $ F_q $ defines a functor of loose object-cloning $ A_∞ $-deformations
\begin{equation*}
F_q: \cat C_q → \MF(\Jac(\chlQ, W_q), ℓ_q).
\end{equation*}
The leading term of $ F_q $ via the identifications \eqref{eq:CHL-functor-objectidentification} and \eqref{eq:CHL-functor-homidentification} is the classical Cho-Hong-Lau functor $ F $.
\end{theorem}

\begin{proof}
First we check the $ A_∞ $-functor relations, then we comment on the leading term. To start with, for the functor relations it is our task to check that
\begin{multline*}
\sum_{0 ≤ j ≤ i ≤ k} (-1)^{‖m_1‖ + … + ‖m_j‖} \compl F_q (m_k, …, μ_q (m_i, …, m_{j+1}), …, m_1) \\
= μ^2_{\MF} (\compl F_q (m_k, …, m_{i+1}), \compl F_q (m_i, …, m_1)) + μ^1_{\MF} (\compl F_q (m_k, …, m_1)).
\end{multline*}
Checking these relations is similar to the classical case \autoref{th:CHL-classical-functor}. The calculation does not use explicitly that $ δ^2 $ vanishes, therefore remains intact. However, due to the possible curvature of $ \cat C_q $ a few new terms appear on one side of the $ A_∞ $-functor equation. We shall check these terms in more detail.

Assume there are $ k ≥ 1 $ input morphisms $ m_1, …, m_k $. Comparing with the calculation in \autoref{th:CHL-classical-functor}, the new terms on the left-hand side of the functor relation are
\begin{align*}
&\sum_{0 ≤ n ≤ k} (-1)^{‖m_1‖ + … + ‖m_n‖} \compl F_q (m_k, …, m_{n+1}, μ^0_q, m_n, …, m_1) (m) \\
&= \sum_{\substack{0 ≤ n ≤ k \\ 0 ≤ l}} (-1)^{‖m_1‖ + … + ‖m_n‖ + ‖m‖ (‖m_1‖ + … + ‖m_k‖ + 1) + 1} μ_q^{k+l+1} (m_k, …, m_{n+1}, μ^0_q, m_n, …, m_1, m, b, …, b).
\end{align*}
There are no new terms on the right-hand side of the functor relation. When applying the curved $ A_∞ $-relation as in the proof of \autoref{th:CHL-classical-functor}, the terms on the left-hand side disappear and the terms $ μ_q (m_k, …, m_1, m, b, …, μ_q^{≥0} (b, …, b), …, b) $ come in, which still vanish since $ \sum μ_q (b, …, b) = ℓ_q \id_{\DefRefObjects} $. This proves the functor relations for $ k ≥ 1 $.

Also for $ k = 0 $ the functor relation is still satisfied:
\begin{equation*}
\compl F_q^1 (μ^0_X) (m) = (-1)^{‖m‖ + 1} \sum_{l ≥ 0} μ_q^{l+2} (μ^0_q, m, b, …, b) = μ^0_{\compl F_q (X)}.
\end{equation*}
The computations for $ F_q $ are analogous. This finishes the checks of the functor relations.

Finally, let us now comment on the leading terms of $ \compl F_q $ and $ F_q $. Indeed, every functor component $ \compl F_q^k $ or $ F_q^k $ is constructed via the deformed products of $ \cat C_q $.

We need to check the $ \Jac(\compl{\chlQ}, W_q) $-linear map $ \compl F_q^k (m_k, …, m_1): \compl F_q (X_1) → \compl F_q (X_{k+1}) $. Up to terms of $ \mathfrak{m} · \compl F_q (X_{k+1}) $, this map is equal to
\begin{equation*}
\sum_{l ≥ 0} μ^{k+l+1} (m_k, …, m_1, -, b, …, b): \compl F_q (X_1) → \compl F_q (X_{k+1}).
\end{equation*}
Passing along $ ψ $ gives
\begin{equation*}
\sum_{l ≥ 0} μ^{k+l+1} (m_k, …, m_1, -, b, …, b): \compl F_q (X_1) / (\mathfrak{m} · \compl F_q (X_1)) → \compl F_q (X_{k+1}) / (\mathfrak{m} · \compl F_q (X_{k+1})).
\end{equation*}
This map is not the same as $ \compl F^k (m_k, …, m_1) $ yet. But now identify $ \compl F_q (X_1) / (\mathfrak{m} · \compl F_q (X_1)) \isoto \compl F (X_1) $ and similarly for $ X_{k+1} $. Recalling \autoref{rem:CHL-projectives-explicit}, the induced map $ \compl F (X_1) → \compl F (X_{k+1}) $ is merely given by the composition $ \compl F (X_1) → \compl F_q (X_1) / (\mathfrak{m} · \compl F_q (X_1)) → \compl F_q (X_{k+1}) / (\mathfrak{m} · \compl F_q (X_{k+1})) → \compl F (X_{k+1}) $, yielding the map
\begin{equation*}
\sum_{l ≥ 0} μ^{k+l+1} (m_k, …, m_k, -, b, …, b): \compl F (X_1) → \compl F (X_{k+1}).
\end{equation*}
This is precisely $ \compl F^k (m_k, …, m_1) $. We have shown that the leading term of $ \compl F_q $ is $ \compl F $. The same argument holds for $ F_q $. This finishes the proof.
\end{proof}

%% file: MS/intro.tex
\section{Deformed mirror symmetry}
\label{sec:MS}
In this section, we collect all preliminary results and prove deformed mirror symmetry, our main theorem. The procedure is as follows: We start with a geometrically consistent dimer $ Q $ under \autoref{conv:zigzag-category-convention}. By \autoref{sec:zigzag}, the deformed category of zigzag curves $ \H\DefZigzagCat ⊂ \H\Tw\Gtl_q Q $ is deformed cyclic. By \autoref{sec:CHL}, we obtain an algebra $ \Jac_q \mirQ $ plus central element $ ℓ_q ∈ \Jac_q \mirQ $ together with a deformed functor
\begin{equation*}
F_q: \quad \Gtl_q Q → \mf(\Jac_q \mirQ, ℓ_q).
\end{equation*}
Because of the slow growth requirements, this functor is restricted to the domain category $ \Gtl_Q $ instead of $ \HTw\Gtl_q Q $. By \autoref{sec:flatness}, the algebra $ \Jac_q \mirQ $ is a deformation of $ \Jac \mirQ $. The necessary requirement for this last step is that $ \mirQ $ is cancellation consistent and of bounded type. Finally, we interpret the category $ \mf(\Jac_q \mirQ, ℓ_q) $ as a deformation of the classical mirror $ \mf(\Jac \mirQ, ℓ) $, and the functor $ F_q $ as a deformation of the classical $ A_∞ $ mirror quasi-isomorphism
\begin{equation*}
F: \quad \Gtl Q → \mf(\Jac \mirQ, ℓ).
\end{equation*}
In particular, we conclude that $ F_q $ itself is a quasi-isomorphism.

%% file: MS/CHL.tex
\subsection{Mirror symmetry by Cho-Hong-Lau}
\label{sec:MS-CHL}
In this section, we recollect how a gentle algebra as specific instance of the Cho-Hong-Lau construction yields mirror symmetry for punctured surfaces. The first step in this section is to provide some details on this specific instance. Second, we read off the specific properties for the Cho-Hong-Lau construction. Third, we realize that the resulting mirror functor indeed recovers mirror symmetry for punctured surfaces. This section is an integrated summary of \cite[Chapter 10]{CHL}.

\begin{remark}
Cho, Hong and Lau depart from a folklore version of wrapped Fukaya category, where the complete set of products is unclear. We have opted in the present paper to use the very rigorous description of $ \HTw\Gtl Q $ and $ \HTw\Gtl_q Q $ from \papertwoB. In particular, we restrict to the case that $ Q $ is geometrically consistent or a standard sphere dimer.
\end{remark}

Our starting point is a dimer $ Q $ whose zigzag paths are equipped with a specific choice of spin structure, which we have codified in \autoref{conv:zigzag-category-convention}. We shall start describing the specific instance of the Cho-Hong-Lau construction needed for mirror symmetry. The first step is to choose $ \cat C = \HTw\Gtl Q $. The subcategory of reference objects, denoted $ \RefObjects $ in \autoref{sec:CHL}, is the category of zigzag paths $ \H\ZigzagCat ⊂ \HTw\Gtl Q $.

As a second step, we describe the CHL basis that needs to be chosen. The right choice of basis elements $ \{X_e\} $ for $ \Hom_{\H\ZigzagCat} (L_1, L_2) $ is the collection of all transversal intersection points of $ L_1 $ and $ L_2 $ which are odd as morphisms $ L_1 → L_2 $. In other words, every single transversal intersection points between two arbitrary zigzag curves appears as basis element for precisely one hom space, namely the one in which it is odd. The set of transversal intersection points between zigzag curves in $ Q $ is precisely the same as the set of arcs $ Q_1 $. We shall therefore fix notation as follows:

\begin{definition}
For every arc $ a ∈ Q_1 $, we denote by $ X_a $ the odd morphism in $ \ZigzagCat $ which is located at the midpoint of $ a $. Similarly, we denote by $ Y_a $ the even morphism in $ \ZigzagCat $ which is located at the midpoint of $ a $.
\end{definition}

\begin{remark}
Visually speaking, we have the correspondence
\begin{center}
\begin{tikzpicture}
\path (0, 0) node[align=center] {odd morphisms \\ (transversal only)} (2.5, 0) node {\Large $ \longleftrightarrow $} (5, 0) node {arrows $ a ∈ Q_1 $} (7.5, 0) node {\Large $ \longleftrightarrow $} (10, 0) node[align=center] {even morphisms \\ (transversal only)};
\end{tikzpicture}
\end{center}
\end{remark}

\begin{table}
\centering
\morearraystretch
\begin{tabular}{ccc}
\textbf{Gadget} & \textbf{General} & \textbf{Specific} \\\hline
Category & $ \cat C $ & $ \HTw\Gtl Q $ \\
Reference objects & $ \ZigzagCat $ & $ \H\ZigzagCat = \{L_1, …, L_N\} $ \\
Cohomology basis & $ \{X_e, Y_e, \id_L, \id_L^*\} $ & intersection points \\
Quiver & $ \chlQ $ & $ \mirQ $ \\
Superpotential $ W $ & $ ⟨μ(b, …, b), b⟩ $ & $ W = \mirQ^+_{\cyc} - \mirQ^-_{\cyc} $ \\
Relations $ R_e $ & $ ⟨μ(b, …, b), X_e⟩ $ & $ r_a^+ - r_a^- $ \\
Landau-Ginzburg model & $ (\Jac(\chlQ, W), ℓ) $ & $ (\Jac \mirQ, ℓ) $ \\
Matrix factorization $ F(a) $ & $ (M, δ) $ & $ \dmatf{(\Jac \mirQ) h(a)}{(\Jac \mirQ) t(a)}{a}{\bar a} $
\end{tabular}
\caption{For every gadget involved in the Cho-Hong-Lau construction, this overview exhibits the general definition and its specific shape in the case of the pair $ (\H\ZigzagCat, \HTw\Gtl Q) $.}
\label{tab:MS-CHL-gadgets}
\end{table}

With this in mind, let us describe the specific instance of the Cho-Hong-Lau construction which yields mirror symmetry for punctured surfaces. The specific instance consists of a specific pair $ (\RefObjects, \cat C) $, a specific choice of CHL basis and a specific choice of odd pairing. These data are given as follows:
\begin{itemize}
\item The category $ \cat C $ is the derived category $ \HTw\Gtl Q $ of the gentle algebra $ \Gtl Q $.
\item The subcategory $ \RefObjects ⊂ \cat C $ is the subcategory $ \H\ZigzagCat ⊂ \HTw\Gtl Q $ of zigzag paths.
\item The CHL basis consists of all odd cohomology basis elements $ X_e $ between zigzag paths, all even cohomology basis elements $ Y_e $ between zigzag paths, and the co-identity elements $ \coid_{L_i} $.
\item The odd pairing $ ⟨-, -⟩ $ on $ \H\ZigzagCat $ is defined by enforcing the pairing identities \eqref{eq:CHL-classical-basis}. Explicitly, one sets $ ⟨-, -⟩ $ to zero on all pairs of basis elements except
\begin{equation*}
⟨X_e, Y_f⟩ = δ_{ef}, \quad ⟨\coid_{L_i}, \id_{L_j}⟩ = δ_{ij}.
\end{equation*}
\end{itemize}

Application of the Cho-Hong-Lau construction to the pair $ (\H\ZigzagCat, \HTw\Gtl Q) $ yields precisely the dual dimer $ \mirQ $, which we have recalled in \autoref{sec:prelim-ms}:

\begin{lemma}
The CHL quiver of the pair $ (\H\ZigzagCat, \HTw\Gtl Q) $ is the quiver $ \chlQ = \mirQ $.
\end{lemma}

\begin{proof}
By general definition, the vertices of the quiver $ \chlQ $ are the reference objects $ L_i ∈ \RefObjects $ and the arrows from $ L_i $ to $ L_j $ are given by the index set $ E_{ij} $. For the specific case of the pair $ (\H\ZigzagCat, \HTw\Gtl Q) $, the reference objects are the zigzag paths of $ Q $. Since zigzag paths of $ Q $ are in correspondence with vertices of $ \mirQ $, this identifies the vertices of $ \chlQ $ and $ \mirQ $. For given vertices $ i, j $, the set of arrows $ E_{ij} $ from $ i $ to $ j $ is equal to the set of odd transversal intersection points $ L_i → L_j $. Every odd transversal intersection is located at the midpoint of an arc $ a ∈ Q_1 $. The corresponding arc $ a ∈ \mirQ_1 $ runs from $ i $ to $ j $ as well, as illustrated in \autoref{fig:MS-CHL-arccorrespondence}. We conclude that $ \chlQ = \mirQ $.
\end{proof}

\begin{figure}
\centering
\begin{tikzpicture}
\path[draw, gray, ->] (0, 0) -- ++(60:1) coordinate[midway] (1) -- ++(right:1) coordinate[midway] (2) coordinate (stop1);
\path[draw, gray] (stop1) -- ++(60:0.85) coordinate[midway] (3);
\path[draw, gray, ->] (0, 1.6) -- ++(300:1) coordinate[midway] (4) -- ++(right:1) coordinate[midway] (5) coordinate (stop2);
\path[draw, gray] (stop2) -- ++(300:0.85) coordinate[midway] (6);
\path ($ (2)!0.5!(5) $) coordinate (m);
\path[draw, semithick, rounded corners] ($ (1) + (-0.5, -0.1) $) to (1) to coordinate[pos=0.4] (7) (m) to coordinate[pos=0.6] (8) (3) to ($ (3) + (0.5, 0.1) $) node[below] {$ L_1 $};
\path[draw, semithick, rounded corners] ($ (4) + (-0.5, 0.1) $) to (4) to coordinate[pos=0.4] (9) (m) to coordinate[pos=0.6] (10) (6) to ($ (6) + (0.5, -0.1) $) node[below] {$ L_2 $};
\path[fill] (m) circle[radius=0.05] node[above] {$ a $};
\path[draw, semithick, -{To[scale=2]}, bend right=60] (7) to (10);
\path[draw, semithick, -{To[scale=2]}, bend right=60] (8) to (9);
\path (1, -0.5) node {in $ Q $} (5, -0.5) node {in $ \mirQ $};
\begin{scope}[shift={(4, -0.1)}]
\path[draw, gray, ->] (0, 0) -- ++(60:1);
\path[draw, gray, ->] (60:1) -- ++(120:1);
\path[draw, semithick, ->] (60:1) -- ++(right:1);
\path[draw, semithick, ->] (60:1)++(up:0.12) -- ++(right:1) coordinate (right) node[midway, above] {$ a $};
\path[draw, gray, ->] (right) -- ++(60:1);
\path[draw, gray, ->] (right) -- ++(300:1);
\path (60:1) node[left] {$ L_1 $} (right) node[right] {$ L_2 $};
\end{scope}
\end{tikzpicture}
\caption{This picture depicts the correspondence between odd intersections in $ Q $ and arrows in $ \mirQ $. The odd basis element $ X_a: L_1 → L_2 $, given by the intersection of $ \smooth L_1 $ and $ \smooth L_2 $ located at the midpoint of the arc $ a $, corresponds to an arrow $ L_1 → L_2 $ in the dual dimer $ \mirQ $.}
\label{fig:MS-CHL-arccorrespondence}
\end{figure}
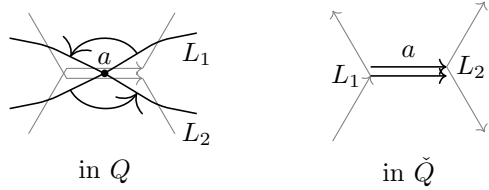

\begin{lemma}[{\cite[Lemma 10.13]{CHL}}]
\label{th:MS-CHL-W}
Application of the Cho-Hong-Lau construction to the pair $ (\H\ZigzagCat, \HTw\Gtl Q) $ yields the familiar superpotential
\begin{equation*}
W = \sum\limits_{\substack{a_k … a_1 \\ \text{clockwise}}} (a_k … a_1)_{\cyc} ~-~ \sum\limits_{\substack{a_k … a_1 \\ \text{counterclockwise}}} (a_k … a_1)_{\cyc} ∈ (ℂ \mirQ)_{≥3}.
\end{equation*}
In particular, application yields the familiar relations $ R_a = r_a^+ - r_a^- $ and the familiar Jacobi algebra $ \Jac(\chlQ, W) = \Jac \mirQ $.
\end{lemma}

\begin{proof}
It is our task to determine the superpotential $ W $. After that, the conclusion on the relations $ R_a $ and the Jacobi algebra is immediate.

The general superpotential is given by $ W = ⟨μ(b, …, b), b⟩ $. Specifically for the pair $ (\H\ZigzagCat, \HTw\Gtl Q) $ this description boils down to calculating the products
\begin{equation*}
μ_{\HTw\Gtl Q} (X_{e_k}, …, X_{e_1})
\end{equation*}
The result of such a product is a linear combination of the even basis elements $ Y_e $ and possibly an identity. It is our task to extract the coefficient of every $ Y_e $ in the product. We claim that this coefficient is precisely
\begin{equation}
\label{eq:MS-CHL-Wcalc}
⟨μ_{\HTw\Gtl Q} (X_{e_k}, …, X_{e_1}), X_e⟩ = \begin{cases}
+1 & \text{if } e_k … e_1 e \text{ is a clockwise polygon in } \mirQ \\
-1 & \text{if } e_k … e_1 e \text{ is a counterclockwise polygon in } \mirQ \\
0 & \text{else} \end{cases}.
\end{equation}
To compute the product, we use our explicit description of the minimal model $ \H\ZigzagCat $, see \autoref{sec:zigzag-products}. Recall that \autoref{sec:zigzag-products} actually describes the deformed version $ \H\DefZigzagCat $ of $ \H\ZigzagCat $. The $ A_∞ $-structure $ \H\ZigzagCat $ is obtained from the deformed $ A_∞ $-structure of $ \H\DefZigzagCat $ by extracting the leading terms of every product. Spelling this out, the higher products in $ \H\ZigzagCat $ are simply determined by only those CR, ID, DS and DW disks that do not cover any punctures.

With this in mind, we are ready to calculate the products \eqref{eq:MS-CHL-Wcalc}. Let $ D $ be a (CR, ID, DS or DW) disk contributing to the product. Note that all input morphisms $ X_{e_i} $ are odd and the output is not an identity, so $ D $ is necessarily a CR disk.

\input{MS/fig_checkerboard.tex}

We claim that the disk $ D $ precisely bounds the interior of an (elementary) polygon, as depicted in \autoref{fig:MS-CHL-Wcalc}. Indeed, the zigzag curves of $ Q $ split $ D $ into pieces of two types: the first type comprises an interior of a polygon of $ Q $, and the second type comprises a polygonal neighborhood of a puncture bounded by neighboring zigzag curves. Two pieces of $ D $ bounding each other are of opposite type. An example of this checkerboard coloring is depicted in \autoref{fig:MS-CHL-checkerboard}. Since the specific disk $ D $ is supposed to cover no puncture, it only consists of pieces of the first type. In fact, there it consists of only one such piece, since any second bordering piece would be of the second type. This shows that $ D $ precisely fits a polygon of $ Q $.

Let us conclude the product formula \eqref{eq:MS-CHL-Wcalc}. If $ e_k … e_1 e $ is a clockwise polygon in $ \mirQ $, then the same path forms a clockwise polygon in $ Q $. As we have just seen, there is a single disk $ D $ contributing the product, and we conclude
\begin{equation*}
μ(X_{e_k}, …, X_{e_1}) = Y_e.
\end{equation*}
The sign is positive because all inputs of $ D $ are odd and run clockwise with the disk, the output is odd, and no $ \# $ signs appear on the zigzag curve segments. If $ e_k … e_1 e $ is a counterclockwise polygon in $ \mirQ $, then the path $ e_1 … e_k e $ is a counterclockwise polygon in $ Q $. As we have just seen, there is a single disk $ D $ contributing to the product, and we conclude
\begin{equation*}
μ(X_{e_k}, …, X_{e_1}) = - Y_e.
\end{equation*}
The sign is negative because $ D $ has $ k $ odd inputs running counterclockwise with $ D $, while there is also a $ \# $-sign in every of the $ k + 1 $ corners of the polygon.

If $ e_k … e_1 e $ is a path in $ \mirQ $ that is not a polygon, there is no disk $ D $ with inputs $ X_{e_1}, …, X_{e_k} $ and output $ Y_e $ at all. In summary, we have shown \eqref{eq:MS-CHL-Wcalc}. Ultimately, the superpotential is given by
\begin{align*}
W &= ⟨μ_{\HTw\Gtl Q} (b, …, b), b⟩ \\
&= \sum_{e_k, …, e_1, e} e_k … e_1 e ⟨μ(X_{e_k}, …, X_{e_1}), Y_e⟩ \\
&= \sum_{\substack{a_k … a_1 \\ \text{clockwise}}} (a_k … a_1)_{\cyc} - \sum_{\substack{a_k … a_1 \\ \text{counterclockwise}}} (a_k … a_1)_{\cyc}.
\end{align*}
This finishes the calculation of the superpotential $ W $. The final statement on $ R_a $ and $ \Jac(\chlQ, W) $ is now immediate: By \autoref{th:CHL-classical-centrality}, we have $ R_a = ∂_a W $ which simplifies to $ R_a = r_a^+ - r_a^- $. In particular, we have $ \Jac(\chlQ, W) = \Jac \mirQ $. This finishes the proof.
\end{proof}

\begin{remark}
\label{rem:MS-CHL-specialpoly}
As preparation for the proof of \autoref{th:MS-CHL-l}, let us fix some terminology regarding the identity location of a zigzag path. Let $ L_i ∈ \ZigzagCat $ be one of the zigzag paths of $ Q $. The identity location of $ L_i $ is a certain arc $ a_0 $. This arc borders precisely two polygons in $ Q $. One of the two is clockwise, the other counterclockwise. The \emph{special polygon} $ P_i $ of $ L_i $ is the one lying on the left or right side of the arc $ a_0 $, depending on whether where $ L_i $ turns left or right at the head of $ a_0 $, respectively. If $ Q $ has no punctures of valence 2, this condition can also be expressed as follows: $ P $ is the polygon which has the identity arc $ a_0 $ and its successor within $ L_i $ as part of the boundary.

In case $ P_i $ is clockwise, we denote by $ p_1, …, p_l $ be the boundary arcs of $ P_i $, in clockwise order and ending with the identity arc $ p_l = a_0 $. In case $ P_i $ is counterclockwise, let $ p_1, …, p_l $ be the boundary arcs of $ P_i $, in clockwise order and starting with $ p_1 = a_0 $. We may call $ p_1, …, p_l $ the \emph{special polygon sequence} of $ L_i $. Note that $ p_l … p_1 $ is only a path in $ Q $ if $ P_i $ is clockwise, while it is always a path in $ \mirQ $.
\end{remark}

\begin{lemma}[{\cite[Lemma 10.21]{CHL}}]
\label{th:MS-CHL-l}
Application of the Cho-Hong-Lau construction to the pair $ (\H\ZigzagCat, \HTw\Gtl Q) $ yields the familiar potential $ ℓ = \sum ℓ_v ∈ \Jac \mirQ $. Altogether, application yields the familiar Landau-Ginzburg model $ (\Jac \mirQ, ℓ) $.
\end{lemma}

\begin{proof}
The general potential $ ℓ = \sum_{i ∈ \chlQ_0} ℓ_i $ is given by
\begin{equation*}
ℓ_i = ⟨μ(b, …, b), \id_{L_i}⟩.
\end{equation*}
Specifically for the pair $ (\H\ZigzagCat, \HTw\Gtl Q) $ this description boils down to extracting the identities from the products
\begin{equation*}
μ_{\HTw\Gtl Q} (X_{e_k}, …, X_{e_1}).
\end{equation*}

Denote by $ p_1, …, p_l $ the special polygon sequence of $ L_i $, defined in \autoref{rem:MS-CHL-specialpoly}. We claim the coefficient of $ \id_{L_i} $ in the product $ μ(X_{e_k}, …, X_{e_1}) $ is precisely
\begin{equation}
\label{eq:MS-CHL-lcalc}
⟨μ_{\HTw\Gtl Q} (X_{e_k}, …, X_{e_1}), \coid_{L_i}⟩ = \begin{cases}
+1 & \text{if } e_k … e_1 = p_l … p_1 \\
0 & \text{else}. \end{cases}
\end{equation}
As in the proof of \autoref{th:MS-CHL-W}, we use our explicit description of the minimal model $ \H\ZigzagCat $. Again, the only disks that count are those not covering any punctures. Let $ D $ be a (CR, ID, DS or DW) disk contributing to the product \eqref{eq:MS-CHL-lcalc}. Note that all input morphisms $ X_{e_i} $ are odd and the output is an identity, so $ D $ is necessarily a CR or ID disk. Similar to the case of the products presented in the proof of \autoref{th:MS-CHL-W}, we conclude that $ D $ precisely bounds a certain polygon $ P $. More precisely, the length $ l $ of the polygon is equal to the number $ k $ of inputs of $ D $, and the boundary arcs of $ P $ are equal to $ e_1, …, e_k $ in this order.

Let us analyze the properties of $ D $. By assumption, the output of $ D $ is the identity of $ L_i $. This identity is therefore located on one of the boundary arcs of $ P $. Since every single odd intersection located at arcs of the polygon is used as input of $ D $, the identity location necessarily of $ L_i $ necessarily lies infinitesimally close to one of the inputs of $ D $. We conclude that $ D $ is necessarily an ID disk.

We claim that $ P $ is the special polygon $ P_i $ of $ L_i $. To show this, assume $ P $ is clockwise. By the definition of ID disks, the input lying close to the output precedes the output. In other words, the identity arc is $ e_k $ and the zigzag path $ L_i $ is the one turning right at the head of $ e_k $. Assume now $ P $ is counterclockwise. By definition of ID disks, the input lying close to the output succeeds the output. In other words, the identity arc is $ e_1 $ and the zigzag path $ L_i $ is the one turning left at the head of $ e_1 $. In both cases we conclude that $ P $ is precisely the special polygon $ P_i $.

We are ready to conclude the product formula \eqref{eq:MS-CHL-lcalc}. In case $ e_1, …, e_k = p_1, …, p_k $ within $ Q $ is the boundary of the special polygon $ P_i $, then there is a single disk contributing to the product as we have just seen. We conclude
\begin{equation*}
⟨μ(X_{e_k}, … X_{e_1}), \coid_{L_i}⟩ = +1.
\end{equation*}
The sign is always positive. Indeed, if $ D $ lies in a clockwise polygon, then all of its $ k $ inputs are odd but clockwise with $ D $ and no $ \# $ signs appear on the boundary of $ D $. If $ D $ lies in a counterclockwise polygon, then all of its $ k $ inputs are odd and counterclockwise, while there is also a $ \# $-sign in every of the $ k $ corners of the polygon

In case $ e_1, …, e_k $ is not $ p_1, …, p_k $, then there is no disk contributing to the product as we have just seen. We conclude
\begin{equation*}
⟨μ(X_{e_k}, …, X_{e_1}), \coid_{L_i}⟩ = 0.
\end{equation*}
Ultimately, the potential reads
\begin{align*}
ℓ &= \sum_{i ∈ \mirQ_0} ⟨μ_{\HTw\Gtl Q} (b, …, b), \coid_{L_i}⟩ \\
&= \sum_{i ∈ \mirQ_0} ℓ_i ∈ \Jac \mirQ.
\end{align*}
Here for every $ i $ the letter $ ℓ_i $ denotes an arbitrary boundary cycle of a polygon starting at vertex $ i ∈ \mirQ_0 $. This finishes the proof.
\end{proof}

\begin{lemma}[{\cite[Proposition 10.30]{CHL}}]
Application of the Cho-Hong-Lau construction to the pair $ (\H\ZigzagCat, \HTw\Gtl Q) $ yields the familiar matrix factorizations
\begin{equation*}
F(a) = \matf{(\Jac \mirQ)h(a)}{(\Jac\mirQ)t(a)}{a}{\bar a} \text{ for } a ∈ \Ob\Gtl Q = Q_1.
\end{equation*}
\end{lemma}

\begin{proof}
By definition, the matrix factorization $ F(a) $ is given by $ (M, δ) $ with module given by
\begin{equation*}
M = \bigoplus_{L_i ∈ \mirQ_0} (\Jac \mirQ) L_i \tensor \Hom_{\HTw\Gtl Q} (L_i, a).
\end{equation*}
and differential given by
\begin{equation*}
δ(m) = (-1)^{‖m‖} \sum μ_{\HTw\Gtl Q} μ(m, b, …, b).
\end{equation*}
Let us evaluate the module. As seen earlier, there are precisely two indices $ i, j $ whose hom space $ \Hom_{\HTw\Gtl Q} (L_i, a) $ is non-empty. These indices correspond to the two zigzag paths that cross the arc $ a $. Of course, the two zigzag paths leaving $ a $ might actually be equal, in which case we set $ i = j $. If we set $ i $ to be the zigzag path with the even intersection and $ j $ the zigzag path with the odd intersection, we have $ i = h(a) $ and $ j = t(a) $. This already gives the shape
\begin{equation*}
(M, δ) = \matf{(\Jac \mirQ) h(a)}{(\Jac \mirQ) t(a)}{*}{*}.
\end{equation*}
It remains to evaluate the differential $ δ $. Denote the even intersection point by $ p ∈ \Hom_{\HTw\Gtl Q} (L_{h(a)}, a) $ and the odd intersection point by $ q ∈ \Hom_{\HTw\Gtl Q} (L_{t(a)}, a) $. As we have seen earlier, the two products $ μ_{\HTw\Gtl Q} (p, b, …, b) $ and $ μ_{\HTw\Gtl Q} (q, b, …, b) $ are computed by MT and MD disks, respectively. For the present non-deformed case, this description boils down to the equations
\begin{align*}
δ(h(a) ¤ p) = (-1)^{‖p‖} μ_{\HTw\Gtl Q} (p, b, …, b) = (-1)^{1+1} a ¤ q, \\
δ(t(a) ¤ q) = (-1)^{‖q‖} μ_{\HTw\Gtl Q} (q, b, …, b) = + \bar a ¤ p.
\end{align*}
This shows that the module map $ (\Jac \mirQ) h(a) → (\Jac \mirQ) t(a) $ is given by left multiplication with $ a $ and the module map $ (\Jac \mirQ) t(a) → (\Jac \mirQ) h(a) $ is given by left multiplication with $ \bar a $.
\end{proof}

We are now ready to grasp the mirror functor associated with the specific case of the pair $ (\H\ZigzagCat, \HTw\Gtl Q) $. Indeed, we have computed the specific Landau-Ginzburg model and the mirror objects. The functor itself is not a quasi-equivalence, but it becomes one if we restrict to domain $ \Gtl Q ⊂ \HTw\Gtl Q $ and codomain $ \mf(\Jac \mirQ, ℓ) ⊂ \MF(\Jac \mirQ, ℓ) $.

\begin{corollary}[{\cite[Proposition 10.32]{CHL}}]
\label{th:MS-CHL-qi}
If $ \mirQ $ is zigzag consistent, the CHL functor $ F: \Gtl Q → \mf(\Jac \mirQ, ℓ) $ is a quasi-isomorphism.
\end{corollary}

\begin{proof}
On the level of objects, $ F $ maps an arc $ a ∈ Q_1 $ to the matrix factorization $ F(a) $ which is bijective on the level of objects. On the level of hom spaces, we are forced to cheat a little and believe the following fact: Let $ L_1, …, L_{N+1} $ be a sequence of $ N+1 ≥ 1 $ zigzag paths and let $ a, b ∈ Q_1 $ be arcs. Let $ h_1, …, h_N $ be odd morphisms with $ h_i ∈ \Hom_{\H\ZigzagCat} (L_i, L_{i+1}) $ not containing co-identities, let $ m ∈ \Hom_{\HTw\Gtl Q} (L_{N+1}, a) $ and $ α ∈ \Hom_{\HTw\Gtl Q} (a, b) $ be further morphisms. Then we shall assume without proof that the product
\begin{equation*}
μ_{\HTw\Gtl Q} (α, m, h_N, …, h_1)
\end{equation*}
is computed by smooth immersed disks with Abouzaid sign. A good illustration of these disks can be found in \cite[Figure 20]{CHL}. This illustration makes it easy to check that $ F^1 $ sends an angle $ α ∈ \Hom_{\Gtl Q} (a, b) $ to its associated morphism of matrix factorizations
\begin{equation*}
±\pmat{0 & - \mathrm{opp}_1 \\ \mathrm{opp}_2 & 0} \text{ or } ±\pmat{\mathrm{opp}_2 & 0 \\ 0 & \mathrm{opp}_1},
\end{equation*}
depending on whether $ α $ is odd or even. It was shown in \cite{Bocklandt} that these morphisms of matrix factorizations provide a basis for $ \H\Hom_{\mf(\Jac \mirQ, ℓ)} (F(a), F(b)) $. This proves the functor $ F $ a quasi-isomorphism.
\end{proof}

The original mirror functor for punctured surfaces \cite{Bocklandt} was defined in a nonconstructive and nonunique way. Its properties known precisely are its values on objects and the first component $ F^1 $. All higher components $ F^{≥2} $ are defined nonconstructively. In contrast, the CHL functor $ F: \Gtl Q → \mf(\Jac \mirQ, ℓ) $ has explicitly defined higher components, at least if one counts the involved higher products of $ \HTw\Gtl Q $ as explicit. The CHL functor should therefore be viewed as a modern constructive incarnation of the original mirror functor.

\begin{remark}
Cho, Hong and Lau use a different convention for Fukaya categories: In their convention, a disk $ D $ contributing to a product $ μ(h_N, …, h_1) $ is supposed to hit the intersection points $ h_1, …, h_N $ in counterclockwise order. We have decided to stick with the original definition of gentle algebras and let products be given by counting disks hitting the intersection points in clockwise order.

Correspondingly, the conventions for the dual dimer $ \mirQ $ also differ: In their convention, the dual dimer $ \mirQ $ is obtained from $ Q $ by flipping over the clockwise faces. In our convention, the dual dimer $ \mirQ $ is obtained by flipping over the counterclockwise faces instead. The results differ by a flip of orientation and an inversion of the arrows.
\end{remark}

%% file: MS/fig_checkerboard.tex
\begin{figure}
\centering
\begin{subfigure}{0.4\linewidth}
\centering
\begin{tikzpicture}
\path[draw] (0, 0) -- (2, 0) -- (2, 1) -- (0, 1) -- (0, 0);
\path[draw] (1, 0) -- (1, 1);
\path[fill=gray!50, draw, semithick] (0.5, 0) to[bend right] (0, 0.5) to[bend right] (0.5, 1) to[bend right] (1, 0.5) to[bend right] (0.5, 0);
\path[fill=gray!50, draw, semithick] (1.5, 0) to[bend right] (1, 0.5) to[bend right] (1.5, 1) to[bend right] (2, 0.5) to[bend right] (1.5, 0);
\path[fill] (1, 0.5) circle[radius=0.03] node[below, shift={(left:0.1)}] {$ e $};
\path[fill] (0.5, 0) circle[radius=0.03] node[below] {$ e_1 $};
\path[fill] (0.5, 1) circle[radius=0.03] node[above] {$ e_k $};
\path[fill] (0, 0.5) circle[radius=0.03];
\path (0, 0) -- (0, 1) node[midway, sloped, above] {$ … $};
\path[fill] (1.5, 0) circle[radius=0.03];
\path[fill] (2, 0.5) circle[radius=0.03];
\path[fill] (1.5, 1) circle[radius=0.03];
\end{tikzpicture}
\caption{The two disks contributing to $ ∂_e W $}
\label{fig:MS-CHL-Wcalc}
\end{subfigure}
\begin{subfigure}{0.4\linewidth}
\centering
\begin{tikzpicture}[scale=0.5]
\path[draw] (0, 0) -- ++(right:1) coordinate[midway] (1-0) -- ++(right:1) coordinate[midway] (3-0) -- ++(right:1) coordinate[midway] (5-0) -- ++(right:1) coordinate[midway] (7-0);
\path[draw] (0, 1) -- ++(right:1) coordinate[midway] (1-2) -- ++(right:1) coordinate[midway] (3-2) -- ++(right:1) coordinate[midway] (5-2) -- ++(right:1) coordinate[midway] (7-2);
\path[draw] (0, 2) -- ++(right:1) coordinate[midway] (1-4) -- ++(right:1) coordinate[midway] (3-4) -- ++(right:1) coordinate[midway] (5-4) -- ++(right:1) coordinate[midway] (7-4);
\path[draw] (0, 3) -- ++(right:1) coordinate[midway] (1-6) -- ++(right:1) coordinate[midway] (3-6) -- ++(right:1) coordinate[midway] (5-6) -- ++(right:1) coordinate[midway] (7-6);
\path[draw] (0, 4) -- ++(right:1) coordinate[midway] (1-8) -- ++(right:1) coordinate[midway] (3-8) -- ++(right:1) coordinate[midway] (5-8) -- ++(right:1) coordinate[midway] (7-8);
\path[draw] (0, 0) -- ++(up:1) coordinate[midway] (0-1) -- ++(up:1) coordinate[midway] (0-3) -- ++(up:1) coordinate[midway] (0-5) -- ++(up:1) coordinate[midway] (0-7);
\path[draw] (1, 0) -- ++(up:1) coordinate[midway] (2-1) -- ++(up:1) coordinate[midway] (2-3) -- ++(up:1) coordinate[midway] (2-5) -- ++(up:1) coordinate[midway] (2-7);
\path[draw] (2, 0) -- ++(up:1) coordinate[midway] (4-1) -- ++(up:1) coordinate[midway] (4-3) -- ++(up:1) coordinate[midway] (4-5) -- ++(up:1) coordinate[midway] (4-7);
\path[draw] (3, 0) -- ++(up:1) coordinate[midway] (6-1) -- ++(up:1) coordinate[midway] (6-3) -- ++(up:1) coordinate[midway] (6-5) -- ++(up:1) coordinate[midway] (6-7);
\path[draw] (4, 0) -- ++(up:1) coordinate[midway] (8-1) -- ++(up:1) coordinate[midway] (8-3) -- ++(up:1) coordinate[midway] (8-5) -- ++(up:1) coordinate[midway] (8-7);
\path[draw, ultra thick] plot[smooth] coordinates{($ (0-1)!-0.5!(1-2) $) (0-1) ($ (0-1)!0.5!(1-2) + (315:0.1) $) (1-2) ($ (1-2)!0.5!(2-3) + (135:0.1) $) (2-3) ($ (2-3)!0.5!(3-4) + (315:0.1) $) (3-4) ($ (3-4)!0.5!(4-5) + (135:0.1) $) (4-5) ($ (4-5)!0.5!(5-6) + (315:0.1) $) (5-6) ($ (5-6)!0.5!(6-7) + (135:0.1) $) (6-7) ($ (6-7)!0.5!(7-8) + (315:0.1) $) (7-8) ($ (7-8)!-0.5!(6-7) $)};
\path[draw, ultra thick] plot[smooth] coordinates{($ (0-3)!-0.5!(1-4) $) (0-3) ($ (0-3)!0.5!(1-4) + (315:0.1) $) (1-4) ($ (1-4)!0.5!(2-5) + (135:0.1) $) (2-5) ($ (2-5)!0.5!(3-6) + (315:0.1) $) (3-6) ($ (3-6)!0.5!(4-7) + (135:0.1) $) (4-7) ($ (4-7)!0.5!(5-8) + (315:0.1) $) (5-8) ($ (5-8)!-0.5!(4-7) $)};
\path[draw, ultra thick] plot[smooth] coordinates{($ (0-5)!-0.5!(1-6) $) (0-5) ($ (0-5)!0.5!(1-6) + (315:0.1) $) (1-6) ($ (1-6)!0.5!(2-7) + (135:0.1) $) (2-7) ($ (2-7)!0.5!(3-8) + (315:0.1) $) (3-8) ($ (3-8)!-0.5!(2-7) $)};
\path[draw, ultra thick] plot[smooth] coordinates{($ (0-7)!-0.5!(1-8) $) (0-7) ($ (0-7)!0.5!(1-8) + (315:0.1) $) (1-8) ($ (1-8)!-0.5!(0-7) $)};
\path[draw, ultra thick] plot[smooth] coordinates{($ (1-0)!-0.5!(2-1) $) (1-0) ($ (1-0)!0.5!(2-1) + (135:0.1) $) (2-1) ($ (2-1)!0.5!(3-2) + (315:0.1) $) (3-2) ($ (3-2)!0.5!(4-3) + (135:0.1) $) (4-3) ($ (4-3)!0.5!(5-4) + (315:0.1) $) (5-4) ($ (5-4)!0.5!(6-5) + (135:0.1) $) (6-5) ($ (6-5)!0.5!(7-6) + (315:0.1) $) (7-6) ($ (7-6)!0.5!(8-7) + (135:0.1) $) (8-7) ($ (8-7)!-0.5!(7-6) $)};
\path[draw, ultra thick] plot[smooth] coordinates{($ (3-0)!-0.5!(4-1) $) (3-0) ($ (3-0)!0.5!(4-1) + (135:0.1) $) (4-1) ($ (4-1)!0.5!(5-2) + (315:0.1) $) (5-2) ($ (5-2)!0.5!(6-3) + (135:0.1) $) (6-3) ($ (6-3)!0.5!(7-4) + (315:0.1) $) (7-4) ($ (7-4)!0.5!(8-5) + (135:0.1) $) (8-5) ($ (8-5)!-0.5!(7-4) $)};
\path[draw, ultra thick] plot[smooth] coordinates{($ (5-0)!-0.5!(6-1) $) (5-0) ($ (5-0)!0.5!(6-1) + (135:0.1) $) (6-1) ($ (6-1)!0.5!(7-2) + (315:0.1) $) (7-2) ($ (7-2)!0.5!(8-3) + (135:0.1) $) (8-3) ($ (8-3)!-0.5!(7-2) $)};
\path[draw, ultra thick] plot[smooth] coordinates{($ (7-0)!-0.5!(8-1) $) (7-0) ($ (7-0)!0.5!(8-1) + (135:0.1) $) (8-1) ($ (8-1)!-0.5!(7-0) $)};
\path[draw, ultra thick] plot[smooth] coordinates{($ (0-1)!-0.5!(1-0) $) (0-1) ($ (0-1)!0.5!(1-0) + (45:0.1) $) (1-0) ($ (1-0)!-0.5!(0-1) $)};
\path[draw, ultra thick] plot[smooth] coordinates{($ (0-3)!-0.5!(1-2) $) (0-3) ($ (0-3)!0.5!(1-2) + (45:0.1) $) (1-2) ($ (1-2)!0.5!(2-1) + (225:0.1) $) (2-1) ($ (2-1)!0.5!(3-0) + (45:0.1) $) (3-0) ($ (3-0)!-0.5!(2-1) $)};
\path[draw, ultra thick] plot[smooth] coordinates{($ (0-5)!-0.5!(1-4) $) (0-5) ($ (0-5)!0.5!(1-4) + (45:0.1) $) (1-4) ($ (1-4)!0.5!(2-3) + (225:0.1) $) (2-3) ($ (2-3)!0.5!(3-2) + (45:0.1) $) (3-2) ($ (3-2)!0.5!(4-1) + (225:0.1) $) (4-1) ($ (4-1)!0.5!(5-0) + (45:0.1) $) (5-0) ($ (5-0)!-0.5!(4-1) $)};
\path[draw, ultra thick] plot[smooth] coordinates{($ (0-7)!-0.5!(1-6) $) (0-7) ($ (0-7)!0.5!(1-6) + (45:0.1) $) (1-6) ($ (1-6)!0.5!(2-5) + (225:0.1) $) (2-5) ($ (2-5)!0.5!(3-4) + (45:0.1) $) (3-4) ($ (3-4)!0.5!(4-3) + (225:0.1) $) (4-3) ($ (4-3)!0.5!(5-2) + (45:0.1) $) (5-2) ($ (5-2)!0.5!(6-1) + (225:0.1) $) (6-1) ($ (6-1)!0.5!(7-0) + (45:0.1) $) (7-0) ($ (7-0)!-0.5!(6-1) $)};
\path[draw, ultra thick] plot[smooth] coordinates{($ (1-8)!-0.5!(2-7) $) (1-8) ($ (1-8)!0.5!(2-7) + (225:0.1) $) (2-7) ($ (2-7)!0.5!(3-6) + (45:0.1) $) (3-6) ($ (3-6)!0.5!(4-5) + (225:0.1) $) (4-5) ($ (4-5)!0.5!(5-4) + (45:0.1) $) (5-4) ($ (5-4)!0.5!(6-3) + (225:0.1) $) (6-3) ($ (6-3)!0.5!(7-2) + (45:0.1) $) (7-2) ($ (7-2)!0.5!(8-1) + (225:0.1) $) (8-1) ($ (8-1)!-0.5!(7-2) $)};
\path[draw, ultra thick] plot[smooth] coordinates{($ (3-8)!-0.5!(4-7) $) (3-8) ($ (3-8)!0.5!(4-7) + (225:0.1) $) (4-7) ($ (4-7)!0.5!(5-6) + (45:0.1) $) (5-6) ($ (5-6)!0.5!(6-5) + (225:0.1) $) (6-5) ($ (6-5)!0.5!(7-4) + (45:0.1) $) (7-4) ($ (7-4)!0.5!(8-3) + (225:0.1) $) (8-3) ($ (8-3)!-0.5!(7-4) $)};
\path[draw, ultra thick] plot[smooth] coordinates{($ (5-8)!-0.5!(6-7) $) (5-8) ($ (5-8)!0.5!(6-7) + (225:0.1) $) (6-7) ($ (6-7)!0.5!(7-6) + (45:0.1) $) (7-6) ($ (7-6)!0.5!(8-5) + (225:0.1) $) (8-5) ($ (8-5)!-0.5!(7-6) $)};
\path[draw, ultra thick] plot[smooth] coordinates{($ (7-8)!-0.5!(8-7) $) (7-8) ($ (7-8)!0.5!(8-7) + (225:0.1) $) (8-7) ($ (8-7)!-0.5!(7-8) $)};
\path[fill=gray, fill opacity=0.5] plot[smooth] coordinates{(1-0) ($ (1-0)!0.5!(2-1) + (135:0.1) $) (2-1)} -- plot[smooth] coordinates{(2-1) ($ (1-2)!0.5!(2-1) + (225:0.1) $) (1-2)} -- plot[smooth] coordinates{(1-2) ($ (0-1)!0.5!(1-2) + (315:0.1) $) (0-1)} -- plot[smooth] coordinates{(0-1) ($ (0-1)!0.5!(1-0) + (45:0.1) $) (1-0)};
\path[fill=gray, fill opacity=0.5] plot[smooth] coordinates{(1-2) ($ (1-2)!0.5!(2-3) + (135:0.1) $) (2-3)} -- plot[smooth] coordinates{(2-3) ($ (1-4)!0.5!(2-3) + (225:0.1) $) (1-4)} -- plot[smooth] coordinates{(1-4) ($ (0-3)!0.5!(1-4) + (315:0.1) $) (0-3)} -- plot[smooth] coordinates{(0-3) ($ (0-3)!0.5!(1-2) + (45:0.1) $) (1-2)};
\path[fill=gray, fill opacity=0.5] plot[smooth] coordinates{(1-4) ($ (1-4)!0.5!(2-5) + (135:0.1) $) (2-5)} -- plot[smooth] coordinates{(2-5) ($ (1-6)!0.5!(2-5) + (225:0.1) $) (1-6)} -- plot[smooth] coordinates{(1-6) ($ (0-5)!0.5!(1-6) + (315:0.1) $) (0-5)} -- plot[smooth] coordinates{(0-5) ($ (0-5)!0.5!(1-4) + (45:0.1) $) (1-4)};
\path[fill=gray, fill opacity=0.5] plot[smooth] coordinates{(1-6) ($ (1-6)!0.5!(2-7) + (135:0.1) $) (2-7)} -- plot[smooth] coordinates{(2-7) ($ (1-8)!0.5!(2-7) + (225:0.1) $) (1-8)} -- plot[smooth] coordinates{(1-8) ($ (0-7)!0.5!(1-8) + (315:0.1) $) (0-7)} -- plot[smooth] coordinates{(0-7) ($ (0-7)!0.5!(1-6) + (45:0.1) $) (1-6)};
\path[fill=gray, fill opacity=0.5] plot[smooth] coordinates{(3-0) ($ (3-0)!0.5!(4-1) + (135:0.1) $) (4-1)} -- plot[smooth] coordinates{(4-1) ($ (3-2)!0.5!(4-1) + (225:0.1) $) (3-2)} -- plot[smooth] coordinates{(3-2) ($ (2-1)!0.5!(3-2) + (315:0.1) $) (2-1)} -- plot[smooth] coordinates{(2-1) ($ (2-1)!0.5!(3-0) + (45:0.1) $) (3-0)};
\path[fill=gray, fill opacity=0.5] plot[smooth] coordinates{(3-2) ($ (3-2)!0.5!(4-3) + (135:0.1) $) (4-3)} -- plot[smooth] coordinates{(4-3) ($ (3-4)!0.5!(4-3) + (225:0.1) $) (3-4)} -- plot[smooth] coordinates{(3-4) ($ (2-3)!0.5!(3-4) + (315:0.1) $) (2-3)} -- plot[smooth] coordinates{(2-3) ($ (2-3)!0.5!(3-2) + (45:0.1) $) (3-2)};
\path[fill=gray, fill opacity=0.5] plot[smooth] coordinates{(3-4) ($ (3-4)!0.5!(4-5) + (135:0.1) $) (4-5)} -- plot[smooth] coordinates{(4-5) ($ (3-6)!0.5!(4-5) + (225:0.1) $) (3-6)} -- plot[smooth] coordinates{(3-6) ($ (2-5)!0.5!(3-6) + (315:0.1) $) (2-5)} -- plot[smooth] coordinates{(2-5) ($ (2-5)!0.5!(3-4) + (45:0.1) $) (3-4)};
\path[fill=gray, fill opacity=0.5] plot[smooth] coordinates{(3-6) ($ (3-6)!0.5!(4-7) + (135:0.1) $) (4-7)} -- plot[smooth] coordinates{(4-7) ($ (3-8)!0.5!(4-7) + (225:0.1) $) (3-8)} -- plot[smooth] coordinates{(3-8) ($ (2-7)!0.5!(3-8) + (315:0.1) $) (2-7)} -- plot[smooth] coordinates{(2-7) ($ (2-7)!0.5!(3-6) + (45:0.1) $) (3-6)};
\path[fill=gray, fill opacity=0.5] plot[smooth] coordinates{(5-0) ($ (5-0)!0.5!(6-1) + (135:0.1) $) (6-1)} -- plot[smooth] coordinates{(6-1) ($ (5-2)!0.5!(6-1) + (225:0.1) $) (5-2)} -- plot[smooth] coordinates{(5-2) ($ (4-1)!0.5!(5-2) + (315:0.1) $) (4-1)} -- plot[smooth] coordinates{(4-1) ($ (4-1)!0.5!(5-0) + (45:0.1) $) (5-0)};
\path[fill=gray, fill opacity=0.5] plot[smooth] coordinates{(5-2) ($ (5-2)!0.5!(6-3) + (135:0.1) $) (6-3)} -- plot[smooth] coordinates{(6-3) ($ (5-4)!0.5!(6-3) + (225:0.1) $) (5-4)} -- plot[smooth] coordinates{(5-4) ($ (4-3)!0.5!(5-4) + (315:0.1) $) (4-3)} -- plot[smooth] coordinates{(4-3) ($ (4-3)!0.5!(5-2) + (45:0.1) $) (5-2)};
\path[fill=gray, fill opacity=0.5] plot[smooth] coordinates{(5-4) ($ (5-4)!0.5!(6-5) + (135:0.1) $) (6-5)} -- plot[smooth] coordinates{(6-5) ($ (5-6)!0.5!(6-5) + (225:0.1) $) (5-6)} -- plot[smooth] coordinates{(5-6) ($ (4-5)!0.5!(5-6) + (315:0.1) $) (4-5)} -- plot[smooth] coordinates{(4-5) ($ (4-5)!0.5!(5-4) + (45:0.1) $) (5-4)};
\path[fill=gray, fill opacity=0.5] plot[smooth] coordinates{(5-6) ($ (5-6)!0.5!(6-7) + (135:0.1) $) (6-7)} -- plot[smooth] coordinates{(6-7) ($ (5-8)!0.5!(6-7) + (225:0.1) $) (5-8)} -- plot[smooth] coordinates{(5-8) ($ (4-7)!0.5!(5-8) + (315:0.1) $) (4-7)} -- plot[smooth] coordinates{(4-7) ($ (4-7)!0.5!(5-6) + (45:0.1) $) (5-6)};
\path[fill=gray, fill opacity=0.5] plot[smooth] coordinates{(7-0) ($ (7-0)!0.5!(8-1) + (135:0.1) $) (8-1)} -- plot[smooth] coordinates{(8-1) ($ (7-2)!0.5!(8-1) + (225:0.1) $) (7-2)} -- plot[smooth] coordinates{(7-2) ($ (6-1)!0.5!(7-2) + (315:0.1) $) (6-1)} -- plot[smooth] coordinates{(6-1) ($ (6-1)!0.5!(7-0) + (45:0.1) $) (7-0)};
\path[fill=gray, fill opacity=0.5] plot[smooth] coordinates{(7-2) ($ (7-2)!0.5!(8-3) + (135:0.1) $) (8-3)} -- plot[smooth] coordinates{(8-3) ($ (7-4)!0.5!(8-3) + (225:0.1) $) (7-4)} -- plot[smooth] coordinates{(7-4) ($ (6-3)!0.5!(7-4) + (315:0.1) $) (6-3)} -- plot[smooth] coordinates{(6-3) ($ (6-3)!0.5!(7-2) + (45:0.1) $) (7-2)};
\path[fill=gray, fill opacity=0.5] plot[smooth] coordinates{(7-4) ($ (7-4)!0.5!(8-5) + (135:0.1) $) (8-5)} -- plot[smooth] coordinates{(8-5) ($ (7-6)!0.5!(8-5) + (225:0.1) $) (7-6)} -- plot[smooth] coordinates{(7-6) ($ (6-5)!0.5!(7-6) + (315:0.1) $) (6-5)} -- plot[smooth] coordinates{(6-5) ($ (6-5)!0.5!(7-4) + (45:0.1) $) (7-4)};
\path[fill=gray, fill opacity=0.5] plot[smooth] coordinates{(7-6) ($ (7-6)!0.5!(8-7) + (135:0.1) $) (8-7)} -- plot[smooth] coordinates{(8-7) ($ (7-8)!0.5!(8-7) + (225:0.1) $) (7-8)} -- plot[smooth] coordinates{(7-8) ($ (6-7)!0.5!(7-8) + (315:0.1) $) (6-7)} -- plot[smooth] coordinates{(6-7) ($ (6-7)!0.5!(7-6) + (45:0.1) $) (7-6)};
\end{tikzpicture}
\caption{Checkerboard coloring}
\label{fig:MS-CHL-checkerboard}
\end{subfigure}
\caption{One can put a checkerboard coloring on the area of the dimer $ Q $. The checkerboard coloring makes it easy to check that there are only two disks which contribute to $ ∂_e W $.}
\end{figure}
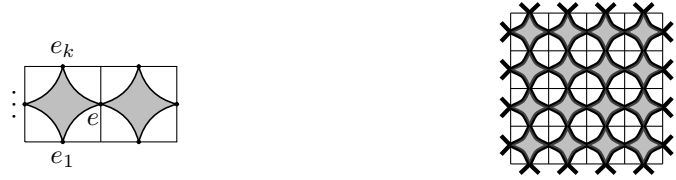

%% file: MS/midpoint.tex
\subsection{Midpoint polygons}
\label{sec:MS-midpoint}
In this section, we introduce an auxiliary tool for the description of the deformed superpotential and deformed potential in \autoref{sec:MS-defsuperpotential} and \ref{sec:MS-defpotential}. The deformed Cho-Hong-Lau construction namely expresses superpotential and potential in term of the products on $ \H\DefZigzagCat $ which are in turn enumerated in terms of CR, ID, DS and DW disks. In the present section, we overhaul this description and provide one single type of disks, which we call midpoint polygons.

\begin{definition}
\label{def:MS-defsuperpotential-midpointpolygon}
A \emph{midpoint polygon} is an immersion of the standard polygon $ D: P_N → |Q| $ with $ N ≥ 1 $, such that
\begin{itemize}
\item The boundary of $ P_N $ is mapped to (nonempty) zigzag curve segments.
\item The corners are convex and lie on intersections points of zigzag curves.
\item The corners point into the interior of a polygon.
\end{itemize}
The map $ D $ itself is taken only up to reparametrization. The corners $ h_1, …, h_N $ of $ D $ lie on intersection points of zigzag curves, in other words on the midpoints of arcs $ a_1, …, a_N ∈ Q_1 $. We refer to the sequence of arcs as the \emph{arc sequence} of $ D $. The midpoint polygon \emph{starts} at arc $ a_1 $ and \emph{ends} at arc $ a_N $. It starts \emph{at the left/right side} of $ a_1 $ if the interior of $ D $ at $ a_1 $ lies at the left/right of the arc $ a_1 $ in the arc's natural orientation. It ends \emph{at the left/right side} of $ a_N $ if the interior of $ D $ lies at the left/right of the arc $ a_N $ in the arc's natural orientation. A \emph{arc crossing} of $ D $ is the datum of a single (indexed) arc crossed by one of the zigzag segments of $ D $, not counting the corner arcs. Arc crossings are always supposed to come with the datum of their location on the boundary of $ D $.
\end{definition}

Midpoint polygons are our most general container format for enumerating products in $ \H\DefZigzagCat $. Midpoint polygons with specific properties or specific additional data can be used to enumerate specific products. By requiring that the corners point into the interior of a polygon, we have incarnated the fact that we only regard disks with odd inputs.

\begin{definition}
\label{def:zigzag-defsuperpotential-polysign}
Let $ D $ be a midpoint polygon with sides lengths $ n_1, …, n_k $. Then the \emph{sign} of $ D $ is $ |D| = \sum (n_i - 1)/2 ∈ ℤ/2ℤ $. The \emph{deformation parameter} $ \punctures(D) ∈ ℂ⟦Q_0⟧ $ is the product of the punctures covered by $ D $, counting multiplicities. Let $ e_1, …, e_k $ be the arc sequence of $ D $. Then the \emph{path recording} of $ D $ is the path $ \arcsequence(D) = e_k … e_1 ∈ ℂ\mirQ $.
\end{definition}

A sample midpoint polygon is depicted in \autoref{fig:MS-midpoint-polygon}.

\begin{figure}
\centering
\begin{subfigure}[b]{0.4\linewidth}
\centering
\begin{tikzpicture}[scale=1.4]
\path[draw] (0, 0) coordinate (1-start1) -- ++(120:0.5) coordinate (1-start2) -- ++(0:0.5) -- ++(120:0.5) -- ++(0:0.5) -- ++(120:0.5) -- ++(0:0.5) -- ++(120:0.5) coordinate (1-end1) -- ++(0:0.5) coordinate (1-end2);
\path[draw] (1-end1) ++(0, 0.05) coordinate (2-start1) -- ++(0:0.6) coordinate (2-start2) -- ++(240:0.5) --  ++(0:0.5) -- ++(240:0.5) --  ++(0:0.5) -- ++(240:0.5) --  ++(0:0.5) coordinate (2-end1) -- ++(240:0.5) coordinate (2-end2);
\path[draw] (2-end1) ++(0.05, 0) coordinate (3-start1) -- ++(240:0.6) coordinate (3-start2) -- ++(120:0.5)  -- ++(240:0.5) -- ++(120:0.5)  -- ++(240:0.5) -- ++(120:0.5) -- ++(240:0.5) coordinate (3-end1) -- ++(120:0.5) coordinate (3-end2);
\path ($ 0.25*(1-start1)+0.25*(1-start2)+0.25*(3-end1)+0.25*(3-end2) $) coordinate (m1);
\path ($ 0.25*(2-start1)+0.25*(2-start2)+0.25*(1-end1)+0.25*(1-end2) $) coordinate (m2);
\path ($ 0.25*(3-start1)+0.25*(3-start2)+0.25*(2-end1)+0.25*(2-end2) $) coordinate (m3);
\path[draw, thick, ->] ($ (m1)!-0.1!(m2) $) -- ($ (m1)!1.1!(m2) $);
\path[draw, thick, ->] ($ (m2)!-0.1!(m3) $) -- ($ (m2)!1.1!(m3) $);
\path[draw, thick, ->] ($ (m3)!-0.1!(m1) $) -- ($ (m3)!1.1!(m1) $);
\path[fill] \foreach \i in {1, 2, 3} {(m\i) circle[radius=0.05]};
\path ($ (1-start1)!0.5!(1-start2) $) node[below left, shift={(-0.1, -0.1)}] {$ e_1 $};
\path ($ (2-start1)!0.5!(2-start2) $) node[above, shift={(0, 0.1)}] {$ e_2 $};
\path ($ (3-start1)!0.5!(3-start2) $) node[below right, shift={(0.1, -0.1)}] {$ e_3 $};
\end{tikzpicture}
\caption{Midpoint polygon}
\label{fig:MS-midpoint-polygon}
\end{subfigure}
\begin{subfigure}[b]{0.4\linewidth}
\begin{tikzpicture}[scale=1.2]
\path[draw, gray, -{To[scale=1.5]}] (0, 0) -- ++(right:1) coordinate[midway] (1) -- ++(down:1) -- ++(right:1) -- ++(down:1) coordinate[midway] (2);
\path[draw, semithick] (1)++(225:1) -- (1) coordinate[pos=0.3] (0) -- (2) node[midway, above right] {$ n_i = 3 $} -- ++(225:1) coordinate[pos=0.7] (3);
\filldraw[semithick, draw=black, fill=gray, fill opacity=0.3] (0) -- (1) -- (2) -- (3);
\path[fill] (1) circle[radius=0.05] (2) circle[radius=0.05];
\path ($ (0)!0.3!(1) $) node[above] {\small $ + $} ($ (1)!0.25!(2) $) node[above] {\small $ + $};
\path ($ (1)!0.45!(2) $) node[below] {\small $ - $} ($ (1)!0.9!(2) $) node[above] {\small $ + $};
\begin{scope}[shift={(3, 0)}]
\path[draw, gray, {To[scale=1.5]}-] (0, 0) -- ++(right:1) coordinate[midway] (1) -- ++(down:1) -- ++(right:1) -- ++(down:1) coordinate[midway] (2);
\path[draw, semithick] (1)++(225:1) -- (1) coordinate[pos=0.3] (0) -- (2) node[midway, above right] {$ n_i = 3 $} -- ++(225:1) coordinate[pos=0.7] (3);
\filldraw[semithick, draw=black, fill=gray, fill opacity=0.3] (0) -- (1) -- (2) -- (3);
\path[fill] (1) circle[radius=0.05] (2) circle[radius=0.05];
\path ($ (0)!0.3!(1) $) node[above] {\small $ - $} ($ (1)!0.25!(2) $) node[above] {\small $ - $};
\path ($ (1)!0.45!(2) $) node[below] {\small $ + $} ($ (1)!0.9!(2) $) node[above] {\small $ - $};
\end{scope}
\end{tikzpicture}
\caption{Lengths and signs of zigzag segments}
\label{fig:zigzag-defsuperpot-length}
\end{subfigure}
\caption{These pictures illustrate midpoint polygons and how we measure the lengths of their boundary segments. In the midpoint polygon on the left, the intersection points are located at arcs $ e_1 $, $ e_2 $, $ e_3 $ of $ Q $. To this midpoint polygon $ D $, we assign the path recording $ \arcsequence(D) = e_3 e_2 e_1 $ and the sign $ |D| = 3·(7-1)/2 = 1 ∈ ℤ/2ℤ $.}
\end{figure}
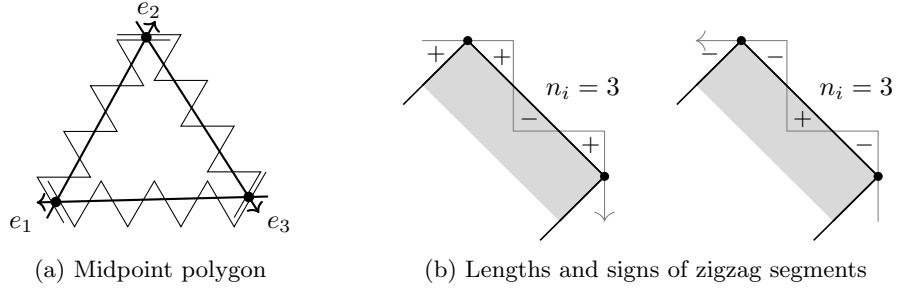

%% file: MS/technical.tex
\subsection{Cyclicity and slow growth}
\label{sec:MS-technical}
In this section, we check that $ \H\DefZigzagCat $ is deformed cyclic and of slow growth. These two properties are technical requirements for the deformed Cho-Hong-Lau construction according to \autoref{th:CHL-functor-th}. We check the cyclicity property directly in terms of midpoint polygons. By contrast, the slow growth property makes reference also to products between arcs and zigzag paths in $ \HTw\Gtl_q Q $ and must therefore be dealt with from scratch. We do not prove slow growth with respect to the entire category, but restrict to arcs and zigzag paths.

\begin{lemma}
\label{th:MS-defsuperpotential-cyclicity}
The $ A_∞ $-deformation $ \H\DefZigzagCat $ is deformed cyclic on odd sequences containing no co-identities.
\end{lemma}

\begin{proof}
Rotating a midpoint polygon contributing to $ ⟨μ_{\H\DefZigzagCat} (X_{e_{k-1}}, …, X_{e_1}), Y_{e_k}⟩ $ gives a midpoint polygon contributing to $ ⟨μ_{\H\DefZigzagCat} (X_{e_k}, …, X_{e_2}), X_{e_1}⟩ $. Their Abouzaid signs and deformation parameters agree. We conclude that both pairings agree, which finishes the proof.
\end{proof}

Let us now check the slow growth requirement property. The requirement is slightly problematic in that we are supposed to evaluate products of the form $ μ_{\HTw\Gtl_q Q} (m_k, …, m_1, b, …, b) $. Here $ m_1, …, m_k $ are morphisms between arbitrary objects of $ \HTw\Gtl_q Q $ and we do not have these products under control. However, for our purposes it suffices to construct the mirror functor merely on the subcategory $ \Gtl_q Q ⊂ \HTw\Gtl_q Q $. Inspecting the construction of the functor $ F_q $ and the proof of \autoref{th:CHL-functor-ismf} and \autoref{th:CHL-functor-th}, we conclude that for constructing the functor $ F_q: \Gtl_q Q → \MF(\Jac(\chlQ, W_q), ℓ_q) $ it suffices to check the slow growth property of \autoref{def:CHL-defLG-slowgrowth} only for morphisms with $ m_1: L → a_1 $ and $ m_i: a_{i-1} → a_i $ with $ L ∈ \DefZigzagCat $ and $ a_1, …, a_k ∈ \Gtl_q Q $. We record this as follows:

\begin{lemma}
\label{th:MS-defsuperpotential-slowgrowth}
Let $ k ≥ 0 $ and $ a_1, …, a_k ∈ \Gtl_q Q $. Let $ L ∈ \DefZigzagCat $ and $ m_1: L → a_1 $ and $ m_i: a_{i-1} → a_i $ for $ i = 2, …, k $. Then for every $ n ∈ ℕ $ there exists an $ l_0 ∈ ℕ $ such that
\begin{equation*}
∀l ≥ l_0: \quad μ_{\HTw\Gtl_q Q}^{k+l} (m_k, …, m_1, b, …, b) ∈ \mathfrak{m}^n \Hom(\RefObjects, a_k).
\end{equation*}
\end{lemma}

\begin{proof}
The intuition is that the number of $ b $ insertions is a lower bound for the size of a disk contributing to the product. In other words, the number of punctures covered by a disk is at least as large as the number of $ b $ insertions, up to a multiplicative constant. If one is willing to assume that the product $ μ_{\HTw\Gtl_q Q} (m_k, …, m_1, b, …, b) $ is computed by counting disks (which we have not shown), this argument should suffice.

Otherwise, the property can be checked rigorously as follows: According to the minimal model calculation procedure detailed in \papertwoA, the product can be computed by summing over Kadeishvili trees. In \papertwoB, we also provide a description of the applicable codifferential $ h_q $ for hom spaces between zigzag paths and for hom spaces from zigzag paths to arcs.

The idea is to measure subtrees containing purely $ b $ inputs separately. Thanks to the dedicated subdisk construction in \papertwoB, we see that result components of $ h $-trees only consisting of $ b $ inputs immediately lie in order $ ≥ n $ as soon as they have a certain amount of $ K ∈ ℕ $ inputs. The $ h $-trees consisting only of $ b $ inputs only yield $ β $ (A) result components of order $ ≥1 $. Apart from the fact that this already establishes the claim in case $ k = 0 $, this observation is important for what follows.

If $ m $ is any morphism between zigzag paths or arcs, then we call the maximum length of angles contained in $ m $ simply the length of $ m $. Let $ I $ be the sum of the lengths of $ m_1, …, m_k $. Let $ F $ the maximum length of a full turn around a puncture in $ Q $.

Pick a result component of the product $ μ(m_k, …, m_1, b, …, b) $ of order $ <n $. Let $ T $ be the Kadeishvili $ π $-tree which the result component is derived from. We call a node of $ T $ pure if its subtree purely consumes $ b $ inputs (instead of $ m_i $) and if it is maximal with this property $ T $ (in that its parent does not have this property). We view a subtree at a pure node of $ T $ as an indecomposable unit. We call all nodes of $ T $ typical that are non-leaf and not contained in a pure subtree.

Regard a typical node $ N ∈ T $. We claim the result at $ N $ has length at most one less than its total inputs including outer $ δ $ insertions and excluding inner $ δ $ inputs and direct $ b $ inputs. Indeed, the result component at the node is a morphism from a zigzag path to an arc by construction. When $ m $ is an $ H $ or $ R $ basis morphism from a zigzag path to an arc, it can be checked easily that the following products vanish:
\begin{equation*}
h_q μ^2_{\Add\Gtl_q Q} (m, α_3/α_4), ~ h_q μ^2_{\Add\Gtl_q Q} (m, β\text{(A)}), ~ π_q μ^2_{\Add\Gtl_q Q} (m, α_3/α_4), ~ π_q μ^2_{\Add\Gtl_q Q} (m, β\text{(A)}).
\end{equation*}
Therefore the node $ N $ is necessarily decorated with $ h_q μ^{≥3}_{\Add\Gtl_q Q} $ or $ π_q μ^{≥3}_{\Add\Gtl_q Q} $. A first-out or final-out disk definitely reduces the length of the inputs. This shows the claim.

According to the claim just proven, the total length of the output of a typical node $ N ∈ T $ is at most $ I+nF-s $, where $ s $ is the number of typical nodes in the subtree at $ N $. This bounds the number of typical nodes in $ T $ by $ I+nF $. The subtree at every pure node has at most $ K $ nodes by assumption. Since there are at most $ n $ pure nodes, the tree $ T $ in total has at most $ I+nF+nK $ nodes. This bounds the total number of nodes in $ T $.

Regard a typical node $ N $. Among the child nodes of $ N $, there are at most $ k $ typical children and at most $ n $ pure children. Moreover, $ N $ may have at most $ n $ direct outer $ δ $ insertions. This gives rise to already $ 2n+k $ inputs of the $ h_q μ_{\Add\Gtl_q Q} $ at $ N $. Every string of direct $ b $ and $ δ $ insertions between these $ 2n+k $ inputs is limited to at most $ F $ items, since otherwise no discrete immersed disk can be made. This means that $ N $ consumes at most $ (2n+k+1)F $ many direct $ b $ inputs. This number is bounded.

The total number of direct $ b $ inputs at typical nodes is therefore bounded by $ (2n+k+1)F(I+nF) $. The total number of $ b $ inputs used for pure trees is at most $ nK $. Therefore the total number of $ b $ inputs in the tree $ T $ is bounded by $ (2n+k+1)F(I+nF)+nK $.

We have shown that the product $ μ_{\HTw\Gtl_q Q} (m_k, …, m_1, b, …, b) $ can only have nonzero result components of order $ <n $ if the number of $ b $ insertions is bounded. In consequence, once the number of $ b $ insertions exceeds this bound, all result components are necessarily of order $ ≥n $. This finishes the proof.
\end{proof}

\begin{remark}
\label{rem:MS-technical-restricted}
The construction of the deformed CHL functor $ F_q: \HTw\Gtl_q Q → \MF(\Jac_q Q, ℓ_q) $ along the procedure of \autoref{sec:CHL} is impossible in the specific case of the pair $ (\H\DefZigzagCat, \HTw\Gtl_q Q) $. Indeed, we need to restrict the domain of the functor to the subcategory $ \Gtl_q Q $ due to the slow-growth requirements. Let us explain this in a slightly more elaborate fashion.

If the slow-growth property held for the entire category $ \HTw\Gtl_q Q $, then we could build a CHL functor $ F_q: \HTw\Gtl_q Q → \MF(\Jac_q \mirQ, ℓ_q) $ and restrict afterwards to the subcategory $ \Gtl_q Q ⊂ \HTw\Gtl_q Q $. The construction and subsequent restriction would be analogous to the process described in \autoref{sec:MS-CHL} in case of the undeformed pair $ (\H\ZigzagCat, \HTw\Gtl Q) $.

In reality, the slow-growth property is not guaranteed for the entire category $ \HTw\Gtl_q Q $. For this reason, we cannot define the deformed mirror functor with domain $ \HTw\Gtl_q Q $. Fortunately, we have proven a restricted slow-growth property \autoref{th:MS-defsuperpotential-slowgrowth}. With this help, the construction of the deformed mirror functor succeeds if from its very beginning we restrict the domain of the functor to consist only of the subcategory $ \Gtl_q Q ⊂ \HTw\Gtl_q Q $.
\end{remark}

%% file: MS/defsuperpot.tex
\subsection{Deformed superpotential}
\label{sec:MS-defsuperpotential}
In this section, we start applying the deformed Cho-Hong-Lau construction to the specific pair $ (\H\DefZigzagCat, \HTw\Gtl_q Q) $. The result is a deformed superpotential $ W_q ∈ ℂ\mirQ ⟦Q_0⟧ $, which we describe explicitly in terms of midpoint polygons. After that, we describe the resulting deformed Jacobi algebra. In order to obtain a mirror functor according to \autoref{th:CHL-functor-th}, we are required to show that the deformed Jacobi algebra is a deformation of the classical Jacobi algebra. In the present section, we invoke the flatness result \autoref{th:flatness-flatness-dimer} to show that this is the case. Recall that we work under \autoref{conv:zigzag-category-convention}.

\begin{lemma}
\label{th:MS-defsuperpotential-W}
The deformed superpotential $ W_q ∈ ℂ\mirQ⟦Q_0⟧ $ can be expressed as
\begin{equation*}
W_q = \sum_{\substack{D \text{ clockwise} \\ \text{midpoint polygon}}} (-1)^{|D|} \punctures(D) \arcsequence(D) ~-~ \sum_{\substack{D \text{ counterclockwise} \\ \text{midpoint polygon}}} (-1)^{|D|} \punctures(D) \arcsequence(D).
\end{equation*}
\end{lemma}

\begin{proof}
Let us digest the statement. It is our task to evaluate the products $ μ_{\H\DefZigzagCat} (X_{e_{k-1}}, …, X_{e_1}) $ and extract the coefficient of $ Y_{e_k} $. By our explicit description of the products from \autoref{sec:zigzag-products}, these products boil down to counting CR, ID, DS and DW disks. Since ID disks have identity outputs and DS and DW disks are already irrelevant, we are left with the task of counting CR disks.

Let $ D $ be a CR disk with inputs $ X_{e_1}, …, X_{e_{k-1}} $ and output $ Y_{e_k} $. Since $ D $ is CR and its inputs contain no co-identities, all of the zigzag segments of $ D $ are actually non-empty. This immediately renders $ D $ a midpoint polygon. More precisely, we associate with $ D $ the midpoint polygon given by the one with the same shape as $ D $ and arc sequence $ e_1, …, e_k $. Note that this sets up a bijection
\begin{align*}
&\{\text{CR disks with inputs } X_{e_1}, …, X_{e_{k-1}} \text{ and output } Y_{e_k}\} \\
&\qquad\longleftrightarrow\quad \{\text{midpoint polygons with arc sequence } e_1, …, e_k\}.
\end{align*}
We can therefore write
\begin{align*}
⟨μ(X_{e_{k-1}}, …, X_{e_1}), Y_{e_k}⟩ &= \sum_{\substack{\text{CR disks } D \\ \text{with inputs } X_{e_1}, …, X_{e_{k-1}} \\ \text{and output } Y_{e_k}}} (-1)^{\Abouzaid(D)} \punctures(D) \\
&= \sum_{\substack{\text{midpoint polygons } D \\ \text{with arc sequence } e_1, …, e_k}} (-1)^{\Abouzaid(D)} \punctures(D).
\end{align*}
Here we have already anticipated that the Abouzaid sign of a CR disk $ D $ with inputs $ X_{e_1}, …, X_{e_{k-1}} $ and output $ Y_{e_k} $ is equal to the Abouzaid sign attached to its associated midpoint polygon by \autoref{def:zigzag-defsuperpotential-polysign}.

In the remainder of the proof, we check that these two signs are indeed equal. Regard one of the CR disks. Its boundary cuts a number of angles of clockwise and counterclockwise polygons of $ Q $. Let $ n_1, …, n_k $ denote the lengths of the zigzag segments, as depicted in \autoref{fig:zigzag-defsuperpot-length}. The $ \# $ signs alternate along the zigzag segments and the individual lengths $ n_i $ are all odd. At every corner $ X_{e_i} $, the two neighboring signs are equal and in fact determined by whether $ D $ is clockwise or counterclockwise.

If $ D $ is clockwise, the sign at every corner is $ 0 ∈ ℤ/2ℤ $ and it alternates $ n_i $ times until the next corner. This gives a sign contribution of $ (n_i - 1)/2 $. The total sign becomes
\begin{equation}
\label{eq:zigzag-defsuperpot-hashsign}
\sum_{i = 1}^k \frac{n_i - 1}{2} ∈ ℤ/2ℤ.
\end{equation}
Since all zigzag curves lie clockwise with $ D $, this is already the Abouzaid sign of $ D $. If $ D $ is counterclockwise, the sign at every corner is $ 1 ∈ ℤ/2ℤ $. In comparison with \eqref{eq:zigzag-defsuperpot-hashsign}, we incur a sign flip for all $ n_1 + … + n_k $ angles that $ L $ runs through. Since every $ n_i $ is odd, this changes the sign by $ k ∈ ℤ/2ℤ $. Since all zigzag paths are counterclockwise relative to $ D $, the Abouzaid sign also incurs an increase by $ k-1 ∈ ℤ/2ℤ $. This makes the sign precisely opposite to \eqref{eq:zigzag-defsuperpot-hashsign} and finishes the proof.
\end{proof}

The deformed Cho-Hong-Lau construction proceeds by defining the deformed Jacobi algebra. The construction comes in two variants, depending on whether the category of reference objects $ \RefObjects $ is of slow growth with respect to $ \cat C $ or not. In the case of the specific pair $ (\H\DefZigzagCat, \HTw\Gtl_q Q) $, we have seen in \autoref{th:MS-defsuperpotential-slowgrowth} that $ \H\DefZigzagCat $ is of slow growth, at least in a way sufficient for the Cho-Hong-Lau construction. We can therefore apply the slow growth variant of the deformed Cho-Hong-Lau construction and obtain a deformed Jacobi algebra $ \Jac(\mirQ, W_q) $ according to \autoref{def:CHL-defLG-def}. We shall abbreviate this algebra by $ \Jac_q \mirQ $:

\begin{definition}
We denote the \emph{deformed Jacobi algebra} by
\begin{equation*}
\Jac_q \mirQ = \Jac(\mirQ, W_q) = \frac{ℂ\mirQ ⟦Q_0⟧}{~\closure{(∂_e W_q)_{e ∈ \mirQ_1}}~}.
\end{equation*}
\end{definition}

\begin{remark}
Recall that the closure refers to the $ \mathfrak{m} $-adic topology on $ ℂ\mirQ ⟦Q_0⟧ $.
\end{remark}

Let us now explain that the deformed Jacobi algebra $ \Jac_q \mirQ $ is a deformation of the classical Jacobi algebra $ \Jac \mirQ $ in the sense of \autoref{def:flatness-whatis-Adefo}. We have investigated this question in detail during \autoref{sec:flatness} in the general context of superpotential deformations of CY3 Jacobi algebras. We record the specific case of $ \Jac_q \mirQ $ as follows:

\begin{corollary}
\label{th:MS-defsuperpotential-flatness}
If $ \mirQ $ is cancellation consistent and of bounded type, then $ \Jac_q \mirQ $ is a deformation of $ \Jac \mirQ $.
\end{corollary}

\begin{proof}
The goal is to invoke \autoref{th:flatness-flatness-dimer}. We have to check two conditions. First, $ W_q $ and therefore also $ W' = W_q - W $ is indeed cyclic as shown in \autoref{th:MS-defsuperpotential-cyclicity} or equivalently \ref{th:MS-defsuperpotential-W}. Second, $ W_q $ indeed lies in $ ℂ\mirQ ⟦Q_0⟧ $ instead of only $ \compl{ℂ\mirQ} ⟦Q_0⟧ $, since $ \H\DefZigzagCat $ is of slow growth by \autoref{th:MS-defsuperpotential-slowgrowth}.
%
%
Finally, we conclude that \autoref{th:flatness-flatness-dimer} applies and $ \Jac_q \mirQ $ is a deformation of $ \Jac \mirQ $. This finishes the proof.
\end{proof}

\begin{remark}
The dual dimer $ \mirQ_M $ of the standard sphere dimer $ Q_M $ is cancellation consistent. The dimer $ \mirQ_M $ is also of bounded type because its two polygons have the same length, namely $ M $. Therefore \autoref{th:MS-defsuperpotential-flatness} particularly applies to the case $ Q = Q_M $.
\end{remark}

%% file: MS/defpotential.tex
\subsection{Deformed potential}
\label{sec:MS-defpotential}
In this section, we compute the deformed potential. More precisely, we evaluate the definition of the deformed potential from the deformed Cho-Hong-Lau construction in the case of $ \H\DefZigzagCat $. The result is an element $ ℓ_q ∈ \mirQ⟦Q_0⟧ $. We describe this element in terms of midpoint polygons. Recall that we work under \autoref{conv:zigzag-category-convention}.

Recall that the deformed Cho-Hong-Lau construction in general requires us to compute a deformed potential $ ℓ_q ∈ B \htensor \chlQ $, given by counting identity outputs of products of the form $ μ(X_{e_k}, …, X_{e_1}) $. More precisely, the definition reads
\begin{equation*}
ℓ_q = \sum_{i ∈ \chlQ_0} ⟨μ(b, …, b), \coid_i⟩.
\end{equation*}
It is our task to evaluate this deformed potential for the specific pair $ (\H\DefZigzagCat, \HTw\Gtl_q Q) $. In this specific case, the products $ μ_{\H\DefZigzagCat} (b, …, b) $ are given by an enumeration of CR, ID, DS and DW disks. As announced, DS and DW disks are irrelevant. In fact, the enumeration boils down to counting midpoint polygons of specific type.

Regard a midpoint polygon. Its boundary crosses many arcs. Indeed, it consists of segments of zigzag curves, which run between midpoints of arcs. The longer the segments, the more arcs are crossed. A midpoint polygon may or may not cross arcs that are identity arcs $ a_0 $ of zigzag paths. If a midpoint polygon starts at, ends at or crosses an arc that is the identity arc $ a_0 $ of a zigzag path, then the midpoint polygon together with the datum of this crossing determines an identity contribution to the product of the intersection points associated with the polygon's corners. To make this precise, we set up the following terminology:

\begin{definition}
\label{def:zigzag-identity-polygons}
Let $ L $ be a zigzag path in $ Q $ with identity location $ a_0 $. An \emph{$ L $-polygon} is one of the following:
\begin{itemize}
\item If $ L $ turns left at the head of $ a_0 $: a midpoint polygon starting at the left side of $ a_0 $.
\item If $ L $ turns right at the head of $ a_0 $: a midpoint polygon ending on the right side of $ a_0 $.
\item A midpoint polygon $ D $ whose segment between the last and first arc is an $ L $ segment, together with the datum of a crossing of this $ L $-segment with the arc $ a_0 ∈ L $.
\end{itemize}
\end{definition}

\input{MS/fig_defpotential.tex}

We have illustrated $ L $-polygons in \autoref{fig:MS-defpotential-Lpolygons}. With this in mind, we can express the potential $ ℓ_q ∈ ℂ\mirQ⟦Q_0⟧ $ in terms of $ L $-polygons:

\begin{lemma}
\label{th:MS-defpotential-th}
The deformed potential $ ℓ_q = \sum_{i ∈ \mirQ_0} ℓ_{q, i} ∈ ℂ\mirQ⟦Q_0⟧ $ can be expressed as
\begin{equation*}
ℓ_{q, i} = \sum_{L_i\text{-polygons } D} (-1)^{|D|} \punctures(D) \arcsequence(D).
\end{equation*}
\end{lemma}

\begin{proof}
It is our task to evaluate identity terms in the products $ μ_{\H\DefZigzagCat} (b, …, b) $ and match them with $ L $-polygons of the three types.

According to the description recapitulated in \autoref{sec:zigzag-products}, the products $ μ_{\H\DefZigzagCat} (b, …, b) $ are computed by enumerating CR and ID disks. More precisely, a counterclockwise CR disk with inputs $ X_{e_1}, …, X_{e_k} $ and output $ \id_{L_i} $ corresponds precisely to an $ L_i $-polygon with arc sequence $ e_1, …, e_k $ of the first type. A clockwise CR disk with inputs $ X_{e_1}, …, X_{e_k} $ and output $ \id_{L_i} $ corresponds precisely to an $ L_i $-polygon with arc sequence $ e_1, …, e_k $ of the second type. An ID disk with inputs $ X_{e_1}, …, X_{e_k} $ and output $ \id_{L_i} $ corresponds precisely to an $ L_i $-polygon with arc sequence $ e_1, …, e_k $ of the third type. This already matches all identity terms in the products $ μ_{\H\DefZigzagCat} (b, …, b) $ with $ L $-polygons.

It remains to check the signs. More precisely, we have to check that the Abouzaid sign of a CR or ID disk contributing to $ μ_{\H\DefZigzagCat} (b, …, b) $ is equal to the sign $ |D| ∈ ℤ/2ℤ $ of its associated midpoint polygon $ D $. To see that both are equal, we distinguish whether $ D $ is clockwise or counterclockwise. If the disk is clockwise, its Abouzaid sign is simply the sum of the $ \# $ signs along $ D $, which by the proof of \autoref{th:MS-defsuperpotential-W} is equal to $ |D| $. If $ D $ is counterclockwise and has $ k $ corners, its Abouzaid sign is the sum of the $ \# $ signs plus $ k $, which by the proof of \autoref{th:MS-defsuperpotential-W} is still $ |D| $. This shows that the Abouzaid sign of a disk is equal to the sign $ |D| $ of the associated midpoint polygon $ D $. This proves the desired formula for $ ℓ_{q, i} $.
\end{proof}

The deformed potential $ ℓ_q ∈ ℂ\mirQ ⟦Q_0⟧ $ depends on the choice of identity locations. However, we claim that it becomes independent of this choice when we pass to the deformed Jacobi algebra $ \Jac_q \mirQ $:

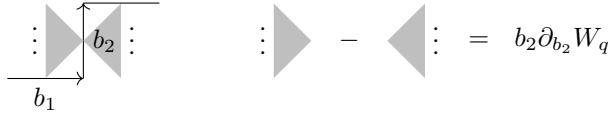
\begin{figure}
\centering
\begin{tikzpicture}
\path[draw, ->] (0, 0) -- (1, 0) node[midway, below] {$ b_1 $};
\path[draw, ->] (1, 0) -- (1, 1) coordinate[midway] (X);
\path[fill=gray!50] (X) -- ++(45:0.7) -- ($ (X) + (315:0.7) $) node[midway, above, sloped] {$ … $} -- cycle;
\path[fill=gray!50] (X) -- ++(135:0.7) -- ($ (X) + (225:0.7) $) node[midway, below, sloped] {$ … $} -- cycle;
\path[draw] (1, 1) -- (2, 1);
\path (X) node[right] {$ b_2 $};
\begin{scope}[shift={(4, 0.5)}]
\path[fill=gray!50] (0, 0) -- (135:0.7) -- (225:0.7) node[midway, below, sloped] {$ … $} -- cycle;
\path (0.5, 0) node {$ - $};
\path[fill=gray!50] (1, 0) -- ++(45:0.7) -- ($ (1, 0) + (315:0.7) $) node[midway, above, sloped] {$ … $} -- cycle;
\path (3, 0) node {$ = ~~ b_2 ∂_{b_2} W_q $};
\end{scope}
\end{tikzpicture}
\caption{The difference between two choices of $ ℓ_q $ is a relation.}
\label{fig:MS-defpotential-choices}
\end{figure}

\begin{lemma}
\label{th:MS-defpotential-independence}
The deformed potential $ ℓ_q ∈ \Jac_q \mirQ = ℂ\mirQ ⟦Q_0⟧ / \closure{(∂_a W_q)} $ is independent of the choice of identity locations.
\end{lemma}

\begin{proof}
The idea is to describe the difference of two possible choices explicitly in terms of the relations $ ∂_a W_q $. Denote by $ L_1, …, L_N $ the zigzag paths of $ Q $. Recall that the element $ ℓ_q $ is the sum of elements $ ℓ_{q, i} $ over all $ 1 ≤ i ≤ N $. The identity location on a certain zigzag path $ L_i $ only impacts the element $ ℓ_{q, i} $ and no other summands. It therefore suffices to study only a single element $ ℓ_{q, i} $. It also suffices to assume that the two choices of identity locations to be studied are arcs which are neighbors along $ L_i $. The strategy is to enumerate $ L_i $-polygons with respect to the two choices.

Let us denote the two neighboring arcs are $ b_1 $ and $ b_2 $, such that $ h(b_1) = t(b_2) $ and $ L_i $ turns left at the head of $ b_1 $ and right at the head of $ b_2 $. The situation is depicted in \autoref{fig:MS-defpotential-choices}. We claim that $ L_i $-polygons with respect to $ a_0 = b_1 $ and $ L_i $-polygons with respect to $ a_0 = b_2 $ are almost the same. Indeed, the typical $ L_i $-polygon with respect to $ a_0 = b_1 $ has a very long $ \smooth L_i $-segment and many crossings with $ b_1 $. Since $ b_2 $ is the neighbor of $ b_1 $, all these $ L_i $-polygons can be interpreted alternatively as $ L_i $-polygons with respect to $ a_0 = b_2 $. The only difference lies in the corner cases.

To be very precise, the only $ L_i $-polygons with respect to $ a_0 = b_1 $ which cannot be interpreted as $ L_i $-polygons with respect to $ a_0 = b_2 $ are those which start at $ b_1 $ and end at $ b_2 $, with just the shortest possible $ \smooth L_i $-segment in between, consisting of a single angle. The only $ L_i $-polygons with respect to $ a_0 = b_2 $ which cannot be interpreted as $ L_i $-polygons with respect to $ a_0 = b_1 $ are those which end at $ b_2 $. Let us denote by $ ℓ_{q, i}^{(b_1)} $ and $ ℓ_{q, i}^{(b_2)} $ the two potentials in $ ℂ\mirQ ⟦Q_0⟧ $ obtained from the two choices of identity locations. In these terms, we conclude
\begin{align*}
ℓ_{q, i}^{(b_2)} - ℓ_{q, i}^{(b_1)} &= \sum_{\substack{\text{clockwise} \\ \text{midpoint polygons } D \\ \text{ending at } b_2}} \punctures(D) \arcsequence(D) - \sum_{\substack{\text{counterclockwise} \\ \text{midpoint polygons } D \\ \text{ending at } b_2}} \punctures(D) \arcsequence(D) \\
&= b_2 ∂_{b_2} W_q.
\end{align*}
This shows that the difference of $ ℓ_{q, i}^{(b_1)} $ and $ ℓ_{q, i}^{(b_2)} $ is a relation in the deformed Jacobi algebra. In conclusion, the potential is independent of the choice of identity locations.
\end{proof}

%% file: MS/fig_defpotential.tex
\begin{figure}
\centering
\begin{subfigure}{0.3\linewidth}
\centering
\begin{tikzpicture}[scale=1.4]
\path[draw] (0, 0) coordinate (1-start1) -- ++(120:0.5) coordinate (1-start2) -- ++(0:0.5) -- ++(120:0.5) -- ++(0:0.5) -- ++(120:0.5) -- ++(0:0.5) -- ++(120:0.5) coordinate (1-end1) -- ++(0:0.5) coordinate (1-end2);
\path[draw] (1-end1) ++(0, 0.05) coordinate (2-start1) -- ++(0:0.6) coordinate (2-start2) -- ++(240:0.5) --  ++(0:0.5) -- ++(240:0.5) --  ++(0:0.5) -- ++(240:0.5) --  ++(0:0.5) coordinate (2-end1) -- ++(240:0.5) coordinate (2-end2);
\path[draw] (2-end1) ++(0.05, 0) coordinate (3-start1) -- ++(240:0.6) coordinate (3-start2) -- ++(120:0.5)  -- ++(240:0.5) -- ++(120:0.5)  -- ++(240:0.5) -- ++(120:0.5) -- ++(240:0.5) coordinate (3-end1) -- ++(120:0.5) coordinate (3-end2);
\path ($ 0.25*(1-start1)+0.25*(1-start2)+0.25*(3-end1)+0.25*(3-end2) $) coordinate (m1);
\path ($ 0.25*(2-start1)+0.25*(2-start2)+0.25*(1-end1)+0.25*(1-end2) $) coordinate (m2);
\path ($ 0.25*(3-start1)+0.25*(3-start2)+0.25*(2-end1)+0.25*(2-end2) $) coordinate (m3);
\path[draw, thick, <-] ($ (m1)!-0.1!(m2) $) -- ($ (m1)!1.1!(m2) $);
\path[draw, thick, <-] ($ (m2)!-0.1!(m3) $) -- ($ (m2)!1.1!(m3) $) node[midway, above right] {$ L $};
\path[draw, thick, <-] ($ (m3)!-0.1!(m1) $) -- ($ (m3)!1.1!(m1) $);
\path[fill] \foreach \i in {1, 2, 3} {(m\i) circle[radius=0.05]};
\path ($ (1-start1)!0.5!(1-start2) $) node[below left, shift={(-0.1, -0.1)}] {$ e_2 $};
\path ($ (2-start1)!0.5!(2-start2) $) node[above, shift={(0, 0.1)}] {$ e_3 $};
\path ($ (3-start1)!0.5!(3-start2) $) node[below, shift={(down:0.2)}] {$ a_0 = e_1 $};
\end{tikzpicture}
\caption{First type}
\label{fig:MS-defpotential-Lpolygons1}
\end{subfigure}
\begin{subfigure}{0.3\linewidth}
\centering
\begin{tikzpicture}[scale=1.4]
\path[draw] (0, 0) coordinate (1-start1) -- ++(120:0.5) coordinate (1-start2) -- ++(0:0.5) -- ++(120:0.5) -- ++(0:0.5) -- ++(120:0.5) -- ++(0:0.5) -- ++(120:0.5) coordinate (1-end1) -- ++(0:0.5) coordinate (1-end2);
\path[draw] (1-end1) ++(0, 0.05) coordinate (2-start1) -- ++(0:0.6) coordinate (2-start2) -- ++(240:0.5) --  ++(0:0.5) -- ++(240:0.5) --  ++(0:0.5) -- ++(240:0.5) --  ++(0:0.5) coordinate (2-end1) -- ++(240:0.5) coordinate (2-end2);
\path[draw] (2-end1) ++(0.05, 0) coordinate (3-start1) -- ++(240:0.6) coordinate (3-start2) -- ++(120:0.5)  -- ++(240:0.5) -- ++(120:0.5)  -- ++(240:0.5) -- ++(120:0.5) -- ++(240:0.5) coordinate (3-end1) -- ++(120:0.5) coordinate (3-end2);
\path ($ 0.25*(1-start1)+0.25*(1-start2)+0.25*(3-end1)+0.25*(3-end2) $) coordinate (m1);
\path ($ 0.25*(2-start1)+0.25*(2-start2)+0.25*(1-end1)+0.25*(1-end2) $) coordinate (m2);
\path ($ 0.25*(3-start1)+0.25*(3-start2)+0.25*(2-end1)+0.25*(2-end2) $) coordinate (m3);
\path[draw, thick, ->] ($ (m1)!-0.1!(m2) $) -- ($ (m1)!1.1!(m2) $) node[midway, above left] {$ L $};
\path[draw, thick, ->] ($ (m2)!-0.1!(m3) $) -- ($ (m2)!1.1!(m3) $);
\path[draw, thick, ->] ($ (m3)!-0.1!(m1) $) -- ($ (m3)!1.1!(m1) $);
\path[fill] \foreach \i in {1, 2, 3} {(m\i) circle[radius=0.05]};
\path ($ (1-start1)!0.5!(1-start2) $) node[below, shift={(down:0.2)}] {$ a_0 = e_3 $};
\path ($ (2-start1)!0.5!(2-start2) $) node[above, shift={(0, 0.1)}] {$ e_1 $};
\path ($ (3-start1)!0.5!(3-start2) $) node[below right, shift={(0.1, -0.1)}] {$ e_2 $};
\end{tikzpicture}
\caption{Second type}
\label{fig:MS-defpotential-Lpolygons2}
\end{subfigure}
\begin{subfigure}{0.3\linewidth}
\centering
\begin{tikzpicture}[scale=1.4]
\path[draw] (0, 0) coordinate (1-start1) -- ++(120:0.5) coordinate (1-start2) -- ++(0:0.5) -- ++(120:0.5) -- ++(0:0.5) -- ++(120:0.5) -- ++(0:0.5) -- ++(120:0.5) coordinate (1-end1) -- ++(0:0.5) coordinate (1-end2);
\path[draw] (1-end1) ++(0, 0.05) coordinate (2-start1) -- ++(0:0.6) coordinate (2-start2) -- ++(240:0.5) --  ++(0:0.5) -- ++(240:0.5) --  ++(0:0.5) -- ++(240:0.5) --  ++(0:0.5) coordinate (2-end1) -- ++(240:0.5) coordinate (2-end2);
\path[draw] (2-end1) ++(0.05, 0) coordinate (3-start1) -- ++(240:0.6) coordinate (3-start2) -- ++(120:0.5)  -- ++(240:0.5) -- ++(120:0.5) -- ++(240:0.5) node[near end, right] {$ a_0 $} coordinate[pos=0.4] (cross) -- ++(120:0.5) -- ++(240:0.5) coordinate (3-end1) -- ++(120:0.5) coordinate (3-end2);
\path ($ 0.25*(1-start1)+0.25*(1-start2)+0.25*(3-end1)+0.25*(3-end2) $) coordinate (m1);
\path ($ 0.25*(2-start1)+0.25*(2-start2)+0.25*(1-end1)+0.25*(1-end2) $) coordinate (m2);
\path ($ 0.25*(3-start1)+0.25*(3-start2)+0.25*(2-end1)+0.25*(2-end2) $) coordinate (m3);
\path[draw, thick, ->] ($ (m1)!-0.1!(m2) $) -- ($ (m1)!1.1!(m2) $);
\path[draw, thick, ->] ($ (m2)!-0.1!(m3) $) -- ($ (m2)!1.1!(m3) $);
\path[draw, thick, ->] ($ (m3)!-0.1!(m1) $) -- ($ (m3)!1.1!(m1) $) node[pos=0.25, below] {$ L $};
\path[fill] \foreach \i in {1, 2, 3} {(m\i) circle[radius=0.05]};
\path[draw, very thick] (cross) ++(115:0.1) -- ++(295:0.2) (cross) ++(25:0.1) -- ++(205:0.2);
\path ($ (1-start1)!0.5!(1-start2) $) node[below left, shift={(-0.1, -0.1)}] {$ e_1 $};
\path ($ (2-start1)!0.5!(2-start2) $) node[above, shift={(0, 0.1)}] {$ e_2 $};
\path ($ (3-start1)!0.5!(3-start2) $) node[below right, shift={(0.1, -0.1)}] {$ e_3 $};
\end{tikzpicture}
\caption{Third type}
\label{fig:MS-defpotential-Lpolygons3}
\end{subfigure}
\caption{These pictures illustrate $ L $-polygons of all three types. According to the definition, the midpoint polygon underlying an $ L $-polygon of the first type starts on the left side of $ a_0 $. In the first picture, we have depicted this by writing $ e_1 = a_0 $. The midpoint polygon underlying an $ L $-polygon of the second type ends on the right side of $ a_0 $. We have depicted the ending in the second picture as $ e_3 = a_0 $. The midpoint polygon underlying an $ L $-polygon of the third type requires that the zigzag segment between the first and the last arc is $ L $ and comes with the datum of a crossing with $ a_0 $. In the third picture, we have indicated this crossing by a thick cross.}
\label{fig:MS-defpotential-Lpolygons}
\end{figure}
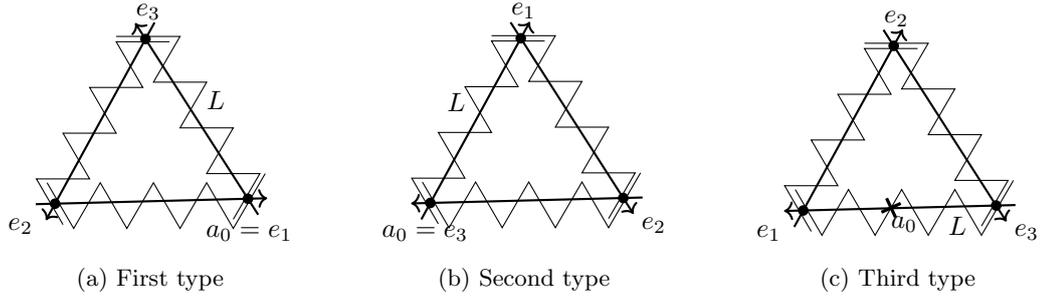

%% file: MS/mirobjects.tex
\subsection{Deformed mirror objects}
\label{sec:MS-mirobjects}
In this section, we compute the deformed mirror objects. More precisely, we evaluate the definition of the deformed mirror functor for the arc objects $ a ∈ \Gtl_q Q ⊂ \HTw\Gtl_q Q $. We describe these objects $ F_q (a) ∈ \MF(\Jac_q \mirQ, ℓ_q) $ in terms of midpoint polygons. Recall that we work under \autoref{conv:zigzag-category-convention}. As we have seen in \autoref{sec:MS-CHL}, the mirror functor $ F: \HTw\Gtl Q → \MF(\Jac \mirQ, ℓ) $ sends the arcs $ a ∈ Q_1 $ to very explicit matrix factorizations $ F(a) $. We provide in this section an explicit computation of the deformed matrix factorizations $ F_q (a) $, which is in fact likewise explicit:

\begin{center}
\begin{tikzpicture}
\path (0, 0) node (A) {arc $ a ∈ \Ob\Gtl Q = Q_1 $} (7, 0) node (B) {$ \dmatf{(\Jac \mirQ) h(a)}{(\Jac \mirQ) t(a)}{a}{\bar a} $};
\path ($ (A.east)!0.5!(B.west) + (up:0.15) $) node {\large $ \xmapsto{~~~F~~~} $};
\path (0, -1) node (C) {arc $ a ∈ \Ob\Gtl_q Q = Q_1 $} (7, -1) node (D) {$ \dmatf{(\Jac_q \mirQ) h(a)}{(\Jac \mirQ) t(a)}{a}{\bar a_q} $};
\path ($ (C.east)!0.5!(D.west) + (up:0.15) $) node {\large $ \xmapsto{~~~F_q~~~} $};
\end{tikzpicture}
\end{center}

To compute the object $ F_q (X) $ for $ X ∈ \cat C_q $, we are in general required to find for every $ L_i ∈ \RefObjects $ the hom space $ \Hom_{\cat C_q} (L_i, X) $. We also need to compute the products $ μ_q (m, b, …, b) $ where $ m ∈ \Hom_{\cat C_q} (L_i, X) $ is a morphism from one of the reference objects to $ X $.

For the specific instance $ \H\DefZigzagCat ⊂ \HTw\Gtl_q Q $ of the deformed Cho-Hong-Lau construction, we have described the hom spaces and the relevant products already \autoref{sec:zigzag-mirobjects}. Computing the mirror objects $ F_q (a) $ for $ a ∈ \Gtl_q Q $ now boils down to evaluating the products $ μ_{\HTw\Gtl_q Q} (m, b, …, b) $, where $ m ∈ \Hom_{\HTw\Gtl_q Q} (L, a) $ is an even or odd intersection point between $ L $ and $ a $. Recall that the products for even $ m $ are easy to express, while the products for odd $ m $ are computed by MD disks. We shall simplify the description of these products here in terms of midpoint polygons.

\begin{definition}
Let $ a ∈ Q_1 $ be an arc. Define the \emph{deformed complement} of $ a $ as
\begin{equation*}
\bar a_q = \sum_{\substack{\text{clockwise} \\ \text{midpoint polygons } D \\ \text{ending at } a}} \punctures(D) ~ ∂_a \arcsequence(D) ∈ \Jac_q \mirQ.
\end{equation*}
\end{definition}

\begin{corollary}
The deformed mirror objects are the deformed matrix factorizations given by
\begin{equation*}
F_q (a) = \dmatf{(\Jac_q \mirQ) h(a)}{(\Jac_q \mirQ) t(a)}{a}{\bar a_q}.
\end{equation*}
\end{corollary}

\begin{proof}
It is our task to evaluate the products $ μ_{\HTw\Gtl_q Q} (m, b, …, b) $. Here $ m $ is an even or odd intersection point between a zigzag path and an arc. We shall regard the two cases separately and explain the relevant products.

First of all, fix an arc $ a ∈ Q_1 $. Let $ L $ be the zigzag path turning left at the head of $ a $ and $ L' $ the zigzag path turning right at the head of $ a $. Let $ m: L → a $ be the odd intersection point and $ m^*: L' → a $ the even intersection point. Regard the odd morphism $ X_a: L → L' $ of zigzag paths. According to \autoref{th:zigzag-mirobjects-products}, the product $ μ_{\HTw\Gtl_q Q} (m^*, X_a) $ is equal to $ - m $ and it is the only nonvanishing term in $ μ(m^*, b, …, b) $. We conclude
\begin{equation*}
δ(m^*) = (-1)^{‖m^*‖} \sum_{l ≥ 0} μ_{\HTw\Gtl_q Q}^{l+1} (m^*, b, …, b) = a m.
\end{equation*}
Regard now the odd intersection point $ m: L → a $ and regard a sequence $ X_{e_1}, …, X_{e_N} $ of odd basis morphisms with $ X_{e_i}: L_i → L_{i+1} $ and $ L_{N+1} = L $. According to \autoref{th:zigzag-mirobjects-products}, we can describe the product $ μ_{\HTw\Gtl_q Q} (m, X_{e_k}, …, X_{e_1}) $ by means of MD disks with output $ m^* $:
\begin{equation*}
μ_{\HTw\Gtl_q Q} (m, X_{e_k}, …, X_{e_1}) = \sum_{\substack{D \text{ MD disk} \\ \text{with inputs } X_{e_1}, …, X_{e_k}, m}} \punctures(D) m^*.
\end{equation*}
Unraveling the definition, an MD disk with inputs $ X_{e_1}, …, X_{e_k}, m $ is the same as a midpoint polygon with arc sequence $ e_1, …, e_k, a $. We can therefore write
\begin{equation*}
δ(m) = (-1)^{‖m‖} \sum_{l ≥ 0} μ_{\HTw\Gtl_q Q}^{l+1} (m, b, …, b) = \sum_{\substack{\text{clockwise} \\ \text{midpoint polygons } D \\ \text{ending at } a}} \punctures(D) ~ ∂_a \arcsequence(D) m^* = \bar a_q m^*.
\end{equation*}
In summary, we have computed the map $ δ $ both on the odd hom space $ \Hom(L, a) $ and the even hom space $ \Hom(L', a) $. This finishes the proof.
\end{proof}

\input{MS/fig_curvature.tex}

\begin{remark}
The mirror objects $ F_q (a) ∈ \mf(\Jac_q \mirQ, ℓ_q) $ have curvature. Explicitly, their curvature reads
\begin{equation*}
μ^0_{\MF, F_q (a)} (m) = ℓ_q m - δ^2 (m) = (-1)^{|m|} \sum_{l ≥ 0} μ_{\HTw\Gtl_q Q}^{l+2} (μ_{\HTw\Gtl_q Q}^0, m, b, …, b).
\end{equation*}
In \autoref{fig:MS-mirobjects-curvature}, we have depicted how to interpret this curvature geometrically.
\end{remark}

%% file: MS/fig_curvature.tex
\begin{figure}
\centering
\begin{tikzpicture}
\begin{scope}[every path/.style={gray}]
\path[draw] (0, 0) -- ++(right:1) coordinate[midway] (1-0) -- ++(right:1) coordinate[midway] (3-0) -- ++(right:1) coordinate[midway] (5-0);
\path[draw] (0, 1) -- ++(right:1) coordinate[midway] (1-2) -- ++(right:1) coordinate[midway] (3-2) -- ++(right:1) coordinate[midway] (5-2);
\path[draw] (0, 2) -- ++(right:1) coordinate[midway] (1-4) -- ++(right:1) coordinate[midway] (3-4) -- ++(right:1) coordinate[midway] (5-4);
\path[draw] (0, 3) -- ++(right:1) coordinate[midway] (1-6) -- ++(right:1) coordinate[midway] (3-6) -- ++(right:1) coordinate[midway] (5-6);
\path[draw] (0, 0) -- ++(up:1) coordinate[midway] (0-1) -- ++(up:1) coordinate[midway] (0-3) -- ++(up:1) coordinate[midway] (0-5);
\path[draw] (1, 0) -- ++(up:1) coordinate[midway] (2-1) -- ++(up:1) coordinate[midway] (2-3) -- ++(up:1) coordinate[midway] (2-5);
\path[draw] (2, 0) -- ++(up:1) coordinate[midway] (4-1) -- ++(up:1) coordinate[midway] (4-3) -- ++(up:1) coordinate[midway] (4-5);
\path[draw] (3, 0) -- ++(up:1) coordinate[midway] (6-1) -- ++(up:1) coordinate[midway] (6-3) -- ++(up:1) coordinate[midway] (6-5);
\end{scope}
\path[draw, fill=gray, fill opacity=0.5] plot[smooth] coordinates{(1-0) ($ (1-0)!0.5!(2-1) + (135:0.1) $) (2-1) ($ (2-1)!0.5!(3-2) + (315:0.1) $) ($ (3-2)+(down:0.1) $) (1, 0.9) (0.9, 1) (1, 1.1) ($ (3-2)+(up:0.1) $) ($ (3-2)!0.5!(4-3) + (135:0.1) $) (4-3) ($ (4-3)!0.5!(5-4) + (315:0.1) $) (5-4) ($ (5-4)!0.5!(6-5) + (135:0.1) $) (6-5)} -- plot[smooth] coordinates{(6-5) ($ (5-6)!0.5!(6-5) + (225:0.1) $) (5-6)} -- plot[smooth] coordinates{(5-6) ($ (4-5)!0.5!(5-6) + (315:0.1) $) (4-5) ($ (3-4)!0.5!(4-5) + (135:0.1) $) (3-4) ($ (2-3)!0.5!(3-4) + (315:0.1) $) (2-3) ($ (1-2)!0.5!(2-3) + (135:0.1) $) (1-2) ($ (0-1)!0.5!(1-2) + (315:0.1) $) (0-1)} -- plot[smooth] coordinates{(0-1) ($ (0-1)!0.5!(1-0) + (45:0.1) $) (1-0)};
\path[draw] (1, 1) -- (2, 1);
\path[fill] (1, 1) circle[radius=0.05];
\path[fill] (1.5, 1.1) circle[radius=0.05] node[right] {$ m $};
\path[fill] (1.5, 0.9) circle[radius=0.05] node[right] {out};
\path[draw, ->] (1.2, 1.1) arc(30:330:0.2) node[pos=0.7, left] {$ μ^0_q $};
\path[draw, ->] (1.5, 1)++(60:0.3) arc(60:180:0.2) node[at start, right] {};
\path[draw, ->] (1.3, 0.9) arc(180:240:0.2);
\path[fill] (1-0) circle[radius=0.05] node[below] {$ b $};
\path[fill] (0-1) circle[radius=0.05] node[left] {$ b $};
\path[fill] (5-6) circle[radius=0.05] node[above] {$ b $};
\path[fill] (6-5) circle[radius=0.05] node[right] {$ b $};
\end{tikzpicture}
\caption{The picture depicts the products $ μ(μ^0_q, m, b, …, b) $, which determine the curvature of $ F_q (a) $. The products can be interpreted as smooth immersed disks between odd intersections of zigzag curves, with a small embayment given by intermediately tracing an arc. The longer a disk, the more arcs can be embayed and the more it contributes to curvature of mirror objects. The less punctures a disk is supposed to cover, the shorter it is, which shows nicely why the classical mirror objects $ F(a) $ are curvature-free.}
\label{fig:MS-mirobjects-curvature}
\end{figure}
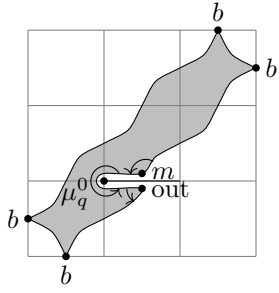

%% file: MS/main.tex
\subsection{Main result}
\label{sec:MS-main}
In this section, we state our main result. It consists of a wide range of deformed mirror equivalences for punctured surfaces. The A-side is the deformed gentle algebra $ \Gtl_q Q $ and the B-side is the deformed category of matrix factorizations $ \mf(\Jac_q \mirQ, ℓ_q) $. When dividing out the maximal ideal, one recovers the classical mirror symmetry originally described in \cite{Bocklandt}. More specifically, one recovers the classical mirror functor of Cho, Hong and Lau defined in \cite[Chapter 10]{CHL}.

\begin{remark}
For sake of completeness, we provide here a brief overview of the setting. The starting point is a dimer $ Q $ which is geometrically consistent or a standard sphere dimer $ Q_M $ with $ M ≥ 3 $. We assume that the dual dimer $ \mirQ $ is zigzag consistent and of bounded type. The bounded type requirement for dimers is detailed in \autoref{sec:flatness-dimers}. For instance, a cancellation consistent dimer on a torus or without triangles is automatically of bounded type.

Our mirror functor is a specific instance of the deformed Cho-Hong-Lau construction of \autoref{sec:CHL}. The input pair $ (\H\DefZigzagCat, \HTw\Gtl_q Q) $ for this specific instance consists of derived category $ \HTw\Gtl_q Q $ of the deformed gentle algebra together with the minimal model $ \H\DefZigzagCat $ of the deformed category of zigzag paths. In \autoref{sec:MS-defsuperpotential}, we have started evaluating the deformed Cho-Hong-Lau construction for this specific instance and obtained an explicit description of the deformed superpotential $ W_q $ and deformed Jacobi algebra $ \Jac_q \mirQ $. In \autoref{sec:MS-defpotential}, we have obtained an explicit description of the deformed potential $ ℓ_q $. The general deformed Cho-Hong-Lau construction of \autoref{sec:CHL-functor} together with \autoref{rem:MS-technical-restricted} about the restriction to the subcategory $ \Gtl_q Q ⊂ \HTw\Gtl_q Q $ provides us with a functor
\begin{equation*}
F_q: \Gtl_q Q → \MF(\Jac_q \mirQ, ℓ_q).
\end{equation*}
In \autoref{sec:MS-mirobjects}, we have computed explicitly the mirror objects $ F_q (a) $ for $ a ∈ \mirQ_1 $. We denote by $ \mf(\Jac_q \mirQ, ℓ_q) $ the subcategory consisting of these objects:
\begin{equation*}
\mf(\Jac_q \mirQ, ℓ_q) = \{F_q (a)\}_{a ∈ \mirQ_1} ⊂ \MF(\Jac_q \mirQ, ℓ_q).
\end{equation*}
\end{remark}

\begin{remark}
\label{rem:MS-main-reinterpretation}
In our main result, we wish to present a mirror functor of $ A_∞ $-deformations. At the present point, both the Jacobi algebra $ \Jac_q \mirQ $ and the category $ \MF(\Jac_q \mirQ, ℓ_q) $ however do not come with a canonical identification as completed tensor product of $ \Jac \mirQ $ and $ \MF(\Jac \mirQ, ℓ) $ with $ B $. We amend this by the following procedure:
\begin{itemize}
\item The algebra $ \Jac_q \mirQ $ is a deformation of $ \Jac \mirQ $ in the sense that there exists a deformation $ μ_{\Jac, q} $ of the product of $ \Jac \mirQ $ and a $ ℂ⟦Q_0⟧ $-linear algebra isomorphism
\begin{equation*}
φ_{\Jac, q}: \Jac_q \mirQ \isoto (ℂ⟦Q_0⟧ \htensor \Jac \mirQ, μ_{\Jac, q}).
\end{equation*}
This identification is not canonical and depends on choice. For sake of the construction, we shall fix one identification and view $ \Jac_q \mirQ $ as $ (ℂ⟦Q_0⟧ \htensor \Jac \mirQ, μ_{\Jac, q}) $.

\item The category $ \MF(\Jac_q \mirQ, ℓ_q) $ is by construction only a loose object-cloning deformation of the category $ \MF(\Jac \mirQ, ℓ) $. This means that its objects are not the same as the objects of $ \MF(\Jac \mirQ, ℓ) $ and its hom spaces cannot be naturally identified with the completed tensor product of the hom spaces of $ \MF(\Jac \mirQ, ℓ) $. Instead, such identification depends on choice. We have elaborated on this phenomenon in \autoref{sec:CHL-functor}. We amend this by regarding the subcategory $ \mf(\Jac_q \mirQ, ℓ_q) $ instead of $ \MF(\Jac_q \mirQ, ℓ_q) $.

\item The category $ \mf(\Jac_q \mirQ, ℓ_q) $ is by construction only a loose deformation of $ \mf(\Jac \mirQ, ℓ) $. This entails that the objects $ F_q (a) $ are in correspondence with the objects $ F(a) $, but the hom spaces have not been identified. We amend this by noting that the specific projective modules contained in $ F_q (a) $ are of the form $ (\Jac_q \mirQ)L_i $. These projective modules come with natural isomorphisms
\begin{equation*}
\Hom_{\Jac_q \mirQ} ((\Jac_q \mirQ)L_i, (\Jac_q \mirQ)L_j) \isoto L_i (\Jac_q \mirQ) L_j \isoto B \htensor L_i (\Jac \mirQ) L_j.
\end{equation*}
We can therefore identify the hom spaces of $ \mf(\Jac_q \mirQ, ℓ_q) $ as
\begin{equation}
\label{eq:MS-main-homidentification}
\Hom_{\mf(\Jac_q \mirQ, ℓ_q)} (F_q (a), F_q (b)) \isoto B \htensor \Hom_{\mf(\Jac \mirQ, ℓ)} (F (a), F(b)).
\end{equation}
This way, we can view $ \mf(\Jac_q \mirQ, ℓ_q) $ as an actual $ A_∞ $-deformation of $ \mf(\Jac \mirQ, ℓ) $.
\item The mirror functor $ F_q: \Gtl_q Q → \MF(\Jac_q \mirQ, ℓ_q) $ is by construction only a functor of loose $ A_∞ $-deformations. This means that it does not have a naturally defined leading term. However, once we interpret $ \mf(\Jac_q \mirQ, ℓ_q) $ as an actual $ A_∞ $-deformation of $ \mf(\Jac \mirQ, ℓ) $, the functor $ F_q: \Gtl_q Q → \mf(\Jac_q \mirQ, ℓ_q) $ becomes a functor of actual $ A_∞ $-deformations and does have a well-defined leading term.
\end{itemize}
\end{remark}

In \autoref{th:MS-main-th} we present the main result. It is a specific instance of \autoref{th:CHL-functor-th} and yields a wide range of deformed mirror equivalences. In the statement of the result, we have already applied the reinterpretation of $ \mf(\Jac_q \mirQ, ℓ_q) $ as actual $ A_∞ $-deformation of $ \mf(\Jac \mirQ, ℓ) $, according to \autoref{rem:MS-main-reinterpretation}.

\begin{theorem}
\label{th:MS-main-th}
Let $ Q $ be a geometrically consistent dimer or standard sphere dimer $ Q_M $ with $ M ≥ 3 $. Assume that the dual dimer $ \mirQ $ is zigzag consistent and of bounded type. Denote by $ (\Jac_q \mirQ, ℓ_q) $ the deformed Landau-Ginzburg model associated with $ Q $. Then:
\begin{enumerate}
\item The algebra $ \Jac_q \mirQ $ is a deformation of $ \Jac \mirQ $.
\item The category $ \mf(\Jac_q \mirQ, ℓ_q) $ is an $ A_∞ $-deformation of $ \mf(\Jac \mirQ, ℓ) $.
\item The deformed Cho-Hong-Lau functor provides a quasi-isomorphism of deformed $ A_∞ $-categories
\begin{equation*}
F_q: \Gtl_q Q \verylongisoto \mf(\Jac_q \mirQ, ℓ_q).
\end{equation*}
\item The leading term of $ F_q $ is the classical Cho-Hong-Lau functor $ F: \Gtl Q → \mf(\Jac \mirQ, ℓ) $.
\end{enumerate}
\end{theorem}

\begin{proof}
This is essentially a restatement of \autoref{th:CHL-functor-th}, applied to the specific pair $ (\H\DefZigzagCat, \HTw\Gtl_q Q) $ under the help of \autoref{rem:MS-technical-restricted}. Although we have already worked with the functor $ F_q $ in \autoref{sec:MS-mirobjects}, we shall recapitulate here why the conditions of \autoref{th:CHL-functor-th} are satisfied. After that, we comment on the first, second, fourth and third claimed statement, in this order.

Let us explain that all requirements of \autoref{th:CHL-functor-th} are satisfied. To start with, we briefly trace the requirements according to \autoref{conv:CHL-deformed}. The category $ \H\DefZigzagCat $ is strictly unital with the same identities as $ \H\ZigzagCat $. It comes with a non-degenerate odd pairing and a choice of CHL basis. As shown in \autoref{th:MS-defsuperpotential-cyclicity}, the deformed $ A_∞ $-structure is cyclic on the odd part. As we explained in \autoref{rem:MS-technical-restricted}, we restrict the construction of the deformed Cho-Hong-Lau functor to the domain $ \Gtl_q Q $. We have seen in \autoref{sec:zigzag} that the hom spaces between zigzag paths and arcs are finite-dimensional. This establishes all requirements of \autoref{conv:CHL-deformed}, adapted to the restriction of the functor to $ \Gtl_q Q $.

The other requirements of \autoref{th:CHL-functor-th} are that $ \H\DefZigzagCat $ is of slow growth and that $ \Jac_q \mirQ $ is a deformation of $ \Jac \mirQ $. In \autoref{th:MS-defsuperpotential-slowgrowth} we have indeed verified that $ \H\DefZigzagCat $ is of slow growth, at least to an extent sufficient for application of the deformed Cho-Hong-Lau construction when restricting the functor to $ \Gtl_q Q $. In \autoref{th:MS-defsuperpotential-flatness}, we have verified that $ \Jac_q \mirQ $ is a deformation of $ \Jac \mirQ $. This already proves the first claimed statement.

We see that all requirements of \autoref{th:CHL-functor-th} are satisfied, adapted to the restriction of the functor to $ \Gtl_q Q $. Invoke the slow growth version of \autoref{th:CHL-functor-th} and conclude that the large category $ \MF(\Jac_q \mirQ, ℓ_q) $ is a loose object-cloning deformation of $ \MF(\Jac \mirQ, ℓ) $. Furthermore, the deformed Cho-Hong-Lau functor defines a functor of loose object-cloning deformations
\begin{equation*}
F_q: \Gtl_q Q → \MF(\Jac_q \mirQ, ℓ_q).
\end{equation*}
Note that the functor is restricted to $ \Gtl_q Q ⊂ \HTw\Gtl_q Q $ because we have only partially verified that $ \H\DefZigzagCat $ is of slow growth. We now explain the second, third and fourth claimed statements of the theorem.

For the second statement of the theorem, we explain why $ \mf(\Jac_q \mirQ, ℓ_q) $ is a deformation of $ \mf(\Jac \mirQ, ℓ) $. Indeed, the subcategory $ \mf(\Jac_q \mirQ, ℓ_q) $ consists of the objects $ F_q (a) $. As such, it is a loose object-cloning deformation of $ \mf(\Jac \mirQ, ℓ) $ via the object-cloning map
\begin{align*}
O: \Ob(\mf(\Jac_q \mirQ, ℓ_q)) &\verylongto \Ob(\mf(\Jac \mirQ, ℓ)), \\
F_q (a) &\verylongmapsto F(a).
\end{align*}
Since the object-cloning map $ O $ is in fact a bijection, we can call $ \mf(\Jac_q \mirQ, ℓ_q) $ simply a loose deformation of $ \mf(\Jac Q, ℓ) $ instead of a loose object-cloning deformation. Upon the further identification \eqref{eq:MS-main-homidentification}, we can naturally view $ \mf(\Jac_q \mirQ, ℓ_q) $ as an actual deformation of $ \mf(\Jac \mirQ, ℓ) $.

For the fourth statement, regard the deformed Cho-Hong-Lau functor $ F_q $ provided to us by \autoref{th:CHL-functor-th}. We merely restrict the codomain of $ F_q $ to the subcategory $ \mf(\Jac_q Q, ℓ_q) ⊂ \MF(\Jac_q Q, ℓ_q) $:
\begin{equation*}
F_q: \Gtl_q Q → \mf(\Jac_q \mirQ, ℓ_q).
\end{equation*}
This restriction is a bijection on object level. In \autoref{th:CHL-functor-th}, we have investigated its leading term. By definition, the leading term entails dividing out $ \mathfrak{m} $ on hom spaces on both sides. We have shown the leading term is $ F $ when the following subsequent identification is applied:
\begin{equation}
\label{eq:MS-main-quotientidentification}
\frac{\Hom_{\MF(\Jac_q \mirQ, ℓ_q)} (F_q (a), F_q (b))}{(Q_0) · \Hom_{\MF(\Jac_q \mirQ, ℓ_q)} (F_q (a), F_q (b))} \isoto \Hom_{\MF(\Jac \mirQ, ℓ)} (F(a), F(b)),
\end{equation}
Under the present interpretation of $ \mf(\Jac_q \mirQ, ℓ_q) $ as actual deformation of $ \mf(\Jac \mirQ, ℓ) $, we can go a step further. Indeed, passing $ F_q $ to the quotient by $ \mathfrak{m} $ and subsequently applying \eqref{eq:MS-main-quotientidentification} is equivalent to interpreting $ \mf(\Jac_q \mirQ, ℓ_q) $ as actual deformation of $ \mf(\Jac \mirQ, ℓ) $ by means of \eqref{eq:MS-main-homidentification} and taking the leading term of the functor of actual $ A_∞ $-deformations $ F_q: \Gtl_q Q → \mf(\Jac_q \mirQ, ℓ_q) $. This shows that $ F $ is the leading term of the functor $ F_q $ of actual deformed $ A_∞ $-categories. This proves the fourth statement.

For the third statement, recall from \autoref{th:MS-CHL-qi} that the classical Cho-Hong-Lau functor $ F: \Gtl Q → \mf(\Jac Q, ℓ) $ is a quasi-isomorphism. Since $ F $ is the leading term of $ F_q $, the functor $ F_q $ is then also a quasi-isomorphism in the sense of \papertwoA. This finishes the proof.
\end{proof}

\begin{remark}
The classical mirror functor $ F: \HTw\Gtl Q → \MF(\Jac \mirQ, ℓ) $ is not quasi-fully-faithful on the entire category $ \HTw\Gtl Q $. For instance, a narrow loop around a puncture is mapped to the zero object, because the loop does not intersect any zigzag curves. While we have formally not verified the slow growth condition for $ \H\DefZigzagCat $ with respect to the entire category $ \HTw\Gtl_q Q $, suppose we obtain a deformed Cho-Hong-Lau functor $ F_q: \HTw\Gtl_q Q → \MF(\Jac_q \mirQ, ℓ_q) $. This functor would then not be quasi-fully-faithful either, in the sense of \papertwoA.
\end{remark}

\begin{remark}
We can build another functor $ \tilde F_q: \HTw\Gtl_q Q → \HTw\MF(\Jac_q \mirQ, ℓ_q) $ by extending the deformed Cho-Hong-Lau functor $ F_q: \Gtl_q Q → \MF(\Jac_q \mirQ, ℓ_q) $ to the deformed twisted completion and passing to the deformed minimal model, in the sense of \papertwoA. In contrast to $ F_q $, the functor $ \tilde F_q $ is quasi-fully-faithful.
\end{remark}

%% file: examples/intro.tex
\section{Examples}
\label{sec:3examples}
In this section, we present two examples of deformed mirror symmetry. We consider the 3-punctured sphere and a 4-punctured torus. In both cases we determine the deformed Jacobi algebra $ \Jac_q \mirQ $ and the deformed potential $ ℓ_q $ explicitly.

%% file: examples/sphere.tex
\subsection{3-punctured sphere}
In this section, we exhibit mirror symmetry and deformed mirror symmetry for the 3-punctured sphere. We start by explaining the dimer involved together with its zigzag curve. Then we turn to the Jacobi algebra and the mirror objects. We explain all midpoint polygons and compute the deformed Landau-Ginzburg model and the deformed matrix factorizations.

\begin{figure}
\centering
\begin{subfigure}{0.3\linewidth}
\centering
\begin{tikzpicture}
\path[shade, ball color=white] (0, 0) circle[radius=2];
\path[draw, very thick] plot[smooth cycle, tension=1.3] coordinates {(90:1.5) (210:0.5) (330:1.5) (90:0.5) (210:1.5) (330:0.5)};
\path[fill] (90:1) circle[radius=0.05] node[above] {$ q_c $};
\path[fill] (210:1) circle[radius=0.05] node[below] {$ q_a $};
\path[fill] (330:1) circle[radius=0.05] node[below] {$ q_b $};
\path[draw, ->, bend left] (90:1) to node[midway, right] {$ a $} (330:1);
\path[draw, ->, bend left] (330:1) to node[midway, below] {$ c $} (210:1);
\path[draw, ->, bend left] (210:1) to node[midway, left] {$ b $} (90:1);
\end{tikzpicture}
\caption{3-punctured sphere}
\label{fig:examples-sphere-arcsys}
\end{subfigure}
\begin{subfigure}{0.2\linewidth}
\centering
\begin{tikzpicture}
\path[draw, ->] (0, 0.1) -- (0, 0.9) node[midway, left] {$ b $};
\path[draw, ->] (0.1, 0) -- (0.9, 0) node[midway, below] {$ b $};
\path[draw, ->] (0.1, 1) -- (0.9, 1) node[midway, above] {$ a $};
\path[draw, ->] (1, 0.1) -- (1, 0.9) node[midway, right] {$ a $};
\path[draw, ->] (0.9, 0.9) -- (0.1, 0.1) node[pos=0.6, above] {$ c $};
\end{tikzpicture}
\caption{$ Q $}
\end{subfigure}
\begin{subfigure}{0.2\linewidth}
\centering
\begin{tikzpicture}
\path[draw, ->] (0, 0.1) -- (0, 0.9) node[midway, left] {$ b $};
\path[draw, ->] (0.1, 0) -- (0.9, 0) node[midway, below] {$ a $};
\path[draw, ->] (0.1, 1) -- (0.9, 1) node[midway, above] {$ a $};
\path[draw, ->] (1, 0.1) -- (1, 0.9) node[midway, right] {$ b $};
\path[draw, ->] (0.9, 0.9) -- (0.1, 0.1) node[pos=0.6, above] {$ c $};
\end{tikzpicture}
\caption{$ \mirQ $}
\label{fig:examples-sphere-mirQ}
\end{subfigure}
\caption{The 3-punctured sphere and its mirror}
\end{figure}
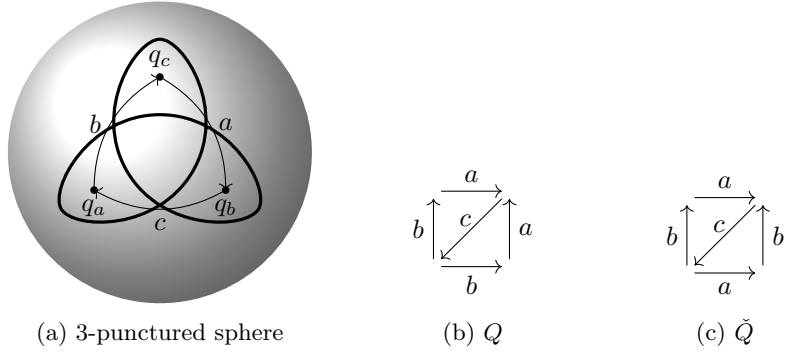

\paragraph*{The dimer and its mirror} The sphere we regard is depicted in \autoref{fig:examples-sphere-arcsys}. The figures depicts the three punctures $ q_a $, $ q_b $, $ q_c $, the three arcs $ a, b, c $ and the zigzag curve. The zigzag curve, as an object of the Fukaya category, is also known as the Seidel Lagrangian.

The mirror dimer $ \mirQ $ is a 1-punctured torus depicted in \autoref{fig:examples-sphere-mirQ}. The superpotential $ W ∈ ℂ\mirQ $ equals $ W = (abc)_{\cyc} - (cba)_{\cyc} $. The Jacobi algebra is $ \Jac \mirQ = ℂ⟨a, b, c⟩ / (ab-ba, ca-ac, bc-cb) = ℂ[a, b, c] $. The potential is $ ℓ = abc ∈ \Jac \mirQ $. The mirror objects are the matrix factorizations
\begin{equation*}
M_a = \dmatf{\Jac \mirQ}{\Jac \mirQ}{a}{bc}, \quad M_b = \dmatf{\Jac \mirQ}{\Jac \mirQ}{b}{ca}, \quad M_c = \dmatf{\Jac \mirQ}{\Jac \mirQ}{c}{ab}.
\end{equation*}

\paragraph*{The deformed superpotential} The deformed gentle algebra $ \Gtl_q Q $ has three deformation parameters $ q_a $, $ q_b $ and $ q_c $. To compute the deformed superpotential, we have to enumerate all midpoint polygons in $ Q $. There are in total 12 midpoint polygons contributing to $ W_q $, depicted in \autoref{fig:examples-sphere-Wmain} and \ref{fig:examples-sphere-Wside}. They fall into the following four groups:
Three of them are 3-gons on the front side and have arc sequences $ abc $, $ bca $ and $ cab $ (sign $ +1 $). They are depicted in \autoref{fig:examples-sphere-Wmainfront}.
Three of them are 3-gons on the rear side and have arc sequences $ cba $, $ bac $ and $ acb $ (sign $ -1 $). They are depicted in \autoref{fig:examples-sphere-Wmainrear}.
Three of them are monogons mainly lying on the front side with arc sequences $ a $, $ b $, $ c $ (sign $ +1 $, $ q $-parameters $ q_a $, $ q_b $, $ q_c $, respectively). They are depicted in \autoref{fig:examples-sphere-Wsidefront_a}, \ref{fig:examples-sphere-Wsidefront_b} and \ref{fig:examples-sphere-Wsidefront_c}.
Three of them are monogons mainly lying on the rear side with arc sequences $ a $, $ b $, $ c $ (sign $ -1 $, $ q $-parameters $ q_a $, $ q_b $, $ q_c $, respectively). They are depicted in \autoref{fig:examples-sphere-Wsiderear_a}, \ref{fig:examples-sphere-Wsiderear_b} and \ref{fig:examples-sphere-Wsiderear_c}.
We conclude that the deformed terms cancel each other and we are left with
\begin{equation*}
W_q = W = (abc)_{\cyc} - (cba)_{\cyc} ∈ ℂ⟨a, b, c⟩.
\end{equation*}
The deformed Jacobi algebra is $ \Jac_q \mirQ = ℂ[a, b, c]⟦q_a, q_b, q_c⟧ $.

\newcommand{\threepspheredraw}{%
\path[draw] (270:0.5) to[out=45, in=255] (30:0.5) to[out=75, in=0] (90:1.5);
\path[draw, <-] (90:1.5) to[out=180, in=105] (150:0.5) to[out=285, in=135] (270:0.5) to[out=315, in=240] (-30:1.5);
\path[draw, <-] (-30:1.5) to[out=60, in=-15] (30:0.5) to[out=165, in=15] (150:0.5) to[out=195, in=120] (210:1.5);
\path[draw, <-] (210:1.5) to[out=300, in=225] (270:0.5);
\path[fill] (90:0.8) circle[radius=0.05] node[above] {$ q_c $};
\path[fill] (210:0.8) circle[radius=0.05] node[below] {$ q_a $};
\path[fill] (-30:0.8) circle[radius=0.05] node[below] {$ q_b $};
\path (150:0.5) node[above left] {$ b $};
\path (30:0.5) node[above right] {$ a $};
\path (-90:0.5) node[below] {$ c $};}

\begin{figure}
\centering
\begin{subfigure}{0.3\linewidth}
\centering
\begin{tikzpicture}
\path[draw] (0, 0) circle[radius=2];
\path[fill, color=gray!50] (30:0.5) to[out=155, in=15] (150:0.5) to[out=285, in=135] (-90:0.5) to[out=45, in=255] cycle;
\threepspheredraw
\end{tikzpicture}
\caption{$ \arcsequence(D) = abc, bca, cab $}
\label{fig:examples-sphere-Wmainfront}
\end{subfigure}
\begin{subfigure}{0.35\linewidth}
\centering
\begin{tikzpicture}
\path[draw] (0, 0) circle[radius=2];
\path[fill, color=gray!50] (90:2) arc(90:270:2) arc(-90:90:2) to (90:1.5) to[out=0, in=75] (30:0.5) to[out=-15, in=60] (330:1.5) to[out=240, in=315] (-90:0.5) to[out=225, in=300] (210:1.5) to[out=120, in=195] (150:0.5) to[out=105, in=180] (90:1.5) to cycle;
\threepspheredraw
\end{tikzpicture}
\caption{$ \arcsequence(D) = cba, bac, acb $}
\label{fig:examples-sphere-Wmainrear}
\end{subfigure}
\caption{Six midpoint polygons contributing to $ W_q $}
\label{fig:examples-sphere-Wmain}
\end{figure}
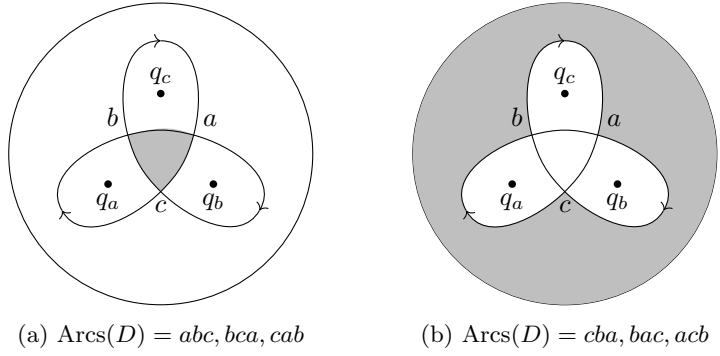

\begin{figure}
\centering
\begin{subfigure}{0.3\linewidth}
\centering
\begin{tikzpicture}
\path[fill, color=gray!50] (30:0.5) to[out=165, in=15] (150:0.5) to[out=195, in=120] (210:1.5) to[out=300, in=225] (-90:0.5) to[out=45, in=255] cycle;
\threepspheredraw
\end{tikzpicture}
\caption{$ \arcsequence(D) = a $}
\label{fig:examples-sphere-Wsidefront_a}
\end{subfigure}
\begin{subfigure}{0.3\linewidth}
\centering
\begin{tikzpicture}
\path[fill, color=gray!50] (150:0.5) to[out=285, in=135] (-90:0.5) to[out=315, in=240] (-30:1.5) to[out=60, in=-15] (30:0.5) to[out=165, in=15] cycle;
\threepspheredraw
\end{tikzpicture}
\caption{$ \arcsequence(D) = b $}
\label{fig:examples-sphere-Wsidefront_b}
\end{subfigure}
\begin{subfigure}{0.3\linewidth}
\centering
\begin{tikzpicture}
\path[fill, color=gray!50] (-90:0.5) to[out=45, in=255] (30:0.5) to[out=75, in=0] (90:1.5) to[out=180, in=105] (150:0.5) to[out=285, in=135] cycle;
\threepspheredraw
\end{tikzpicture}
\caption{$ \arcsequence(D) = c $}
\label{fig:examples-sphere-Wsidefront_c}
\end{subfigure}
\begin{subfigure}{0.32\linewidth}
\centering
\begin{tikzpicture}
\path[draw] (0, 0) circle[radius=2];
\path[fill, color=gray!50] (90:2) to (90:1.5) to[out=180, in=105] (150:0.5) to[out=285, in=135] (270:0.5) to[out=315, in=240] (-30:1.5) to[out=60, in=-15] (30:0.5) to[out=75, in=0] (90:1.5) to (90:2) arc(90:-90:2) arc(-90:-270:2);
\threepspheredraw
\end{tikzpicture}
\caption{$ \arcsequence(D) = a $}
\label{fig:examples-sphere-Wsiderear_a}
\end{subfigure}
\begin{subfigure}{0.32\linewidth}
\centering
\begin{tikzpicture}
\path[draw] (0, 0) circle[radius=2];
\path[fill, color=gray!50] (90:2) to (90:1.5) to[out=180, in=105] (150:0.5) to[out=195, in=120] (210:1.5) to[out=300, in=225] (-90:0.5) to[out=45, in=255] (30:0.5) to[out=75, in=0] (90:1.5) to (90:2) arc(90:-270:2);
\threepspheredraw
\end{tikzpicture}
\caption{$ \arcsequence(D) = b $}
\label{fig:examples-sphere-Wsiderear_b}
\end{subfigure}
\begin{subfigure}{0.32\linewidth}
\centering
\begin{tikzpicture}
\path[draw] (0, 0) circle[radius=2];
\path[fill, color=gray!50] (30:0.5) to[out=165, in=15] (150:0.5) to[out=195, in=120] (210:1.5) to[out=300, in=225] (270:0.5) to[out=315, in=240] (-30:1.5) to[out=60, in=-15] (30:0.5) to (90:2) arc(90:-90:2) arc(270:90:2) to (30:0.5);
\threepspheredraw
\end{tikzpicture}
\caption{$ \arcsequence(D) = c $}
\label{fig:examples-sphere-Wsiderear_c}
\end{subfigure}
\caption{Six canceling midpoint polygons contributing to $ W_q $}
\label{fig:examples-sphere-Wside}
\end{figure}
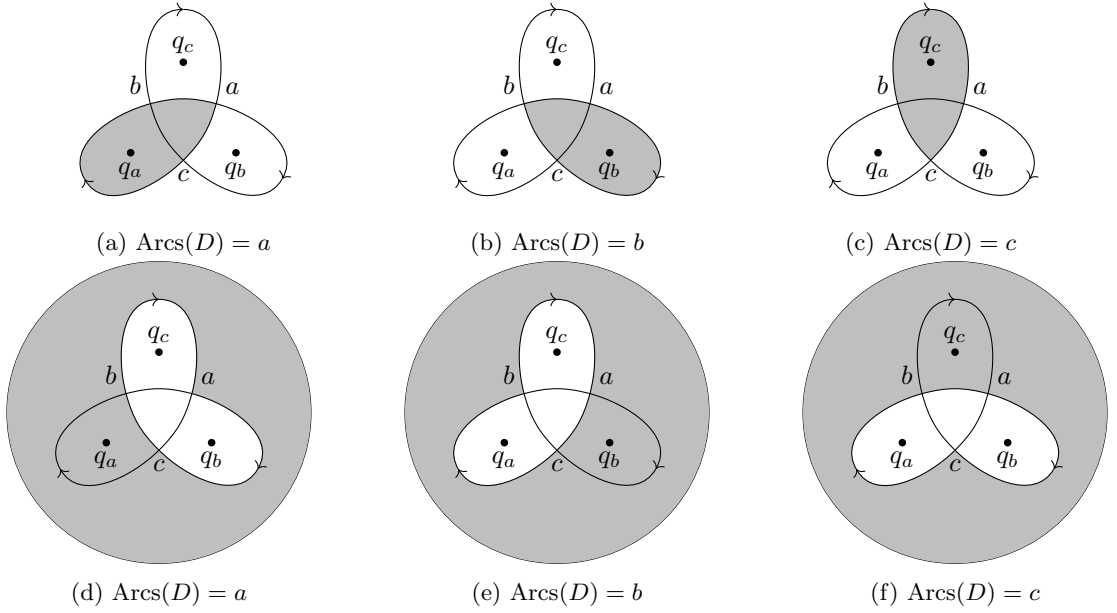

\begin{figure}
\centering
\begin{tikzpicture}
\threepspheredraw
\path[draw] (-90:0.5) to[out=45, in=255] coordinate[pos=0.8] (id) (30:0.5);
\path[fill] (id) circle[radius=0.05] node[left] {$ a_0 $};
\end{tikzpicture}
\caption{Identity location $ a_0 $}
\label{fig:examples-sphere-identity}
\end{figure}

\paragraph*{The deformed potential}
The expression for the deformed potential depends on the choice of identity location for the zigzag path $ L $. Let us choose the identity to lie on the copy of $ a $ which turns right at head and tail of $ a $. This choice is depicted in \autoref{fig:examples-sphere-identity}. The deformed potential $ ℓ_q $ is determined by enumerating $ L $-polygons. The specific choice of $ a_0 $ gives the following four $ L $-polygons. We list these polygons here and recapitulate in technical terms from which part of the minimal model $ \H\DefZigzagCat $ they come from:
\begin{itemize}
\item
One $ L $-polygon is a 3-gon lying on the front side and ending at the right side of $ a $. The underlying midpoint polygon is the same as the one depicted in \autoref{fig:examples-sphere-Wmainfront}. Its contribution to $ ℓ_q $ is $ abc $. In terms of $ \H\DefZigzagCat $, it concerns an ID disk.
\item
One $ L $-polygon is a monogon mainly lying on the front side and ending at the right side of $ a $. The underlying midpoint polygon is the same as the one depicted in \autoref{fig:examples-sphere-Wsidefront_a}. Its contribution to $ ℓ_q $ is $ - q_a a $. In terms of $ \H\DefZigzagCat $, it concerns an ID disk.
\item
One $ L $-polygon is a monogon mainly lying on the rear side and having a crossing with $ a_0 $. The underlying midpoint polygon is the same as the one depicted in \autoref{fig:examples-sphere-Wsiderear_b}. Its contribution to $ ℓ_q $ is $ - q_b b $. In terms of $ \H\DefZigzagCat $, it concerns a CR disk.
\item
One $ L $-polygon is a monogon mainly lying on the front side and having a crossing with $ a_0 $. The underlying midpoint polygon is the same as the one depicted in \autoref{fig:examples-sphere-Wsidefront_c}. Its contribution to $ ℓ_q $ is $ - q_c c $. In terms of $ \H\DefZigzagCat $, it concerns a CR disk.
\end{itemize}
The total deformed potential amounts to
\begin{equation*}
ℓ_q = abc - q_a a - q_b b - q_c c ∈ \Jac_q \mirQ.
\end{equation*}
We immediately see that $ ℓ_q ∈ \Jac_q \mirQ = ℂ[a, b, c]⟦q_a, q_b, q_c⟧ $ is still central.

\paragraph*{The deformed mirror objects}
The deformed mirror objects $ F_q (a) $, $ F_q (b) $, $ F_q (c) $ can be determined by enumerating clockwise midpoint polygons ending at $ a $, $ b $ and $ c $, respectively. For the arc $ a $, there are two such midpoint polygons. The first is a 3-gon and has $ ∂_a \arcsequence(D) = bc $, sign $ +1 $ and no $ q $-parameters. The second is a monogon and has $ ∂_a \arcsequence(D) = 1 $, sign $ -1 $ and $ q $-parameter $ q_a $. Both add up to a contribution of $ bc - q_a $. The deformed mirror object is therefore
\begin{equation*}
F_q (a) = \dmatf{\Jac_q \mirQ}{\Jac_q \mirQ}{a}{bc - q_a}.
\end{equation*}
Similar considerations hold for $ F_q (b) $ and $ F_q (c) $. The deformed mirror objects are listed in \autoref{tab:examples-sphere-mirobjects}. Note that none of the “factors” actually factor to $ ℓ_q $. Instead, the failure to factor $ ℓ_q $ is infinitesimal and serves as curvature of the deformed matrix factorizations.

\begin{table}
\centering
\begin{tabular}{ccc}
Object & Description & Curvature \\\hline
$ F_q (a) $ & $ \dmatf{\Jac_q \mirQ}{\Jac_q \mirQ}{a}{bc-q_a} $ & $ - q_b b - q_c c $ \\
$ F_q (b) $ & $ \dmatf{\Jac_q \mirQ}{\Jac_q \mirQ}{b}{ca-q_b} $ & $ - q_a a - q_c c $ \\
$ F_q (c) $ & $ \dmatf{\Jac_q \mirQ}{\Jac_q \mirQ}{c}{ab-q_c} $ & $ - q_a a - q_b b $
\end{tabular}
\caption{The deformed mirror objects}
\label{tab:examples-sphere-mirobjects}
\end{table}

%% file: examples/torus.tex
\subsection{4-punctured torus}
In this section, we exhibit mirror symmetry and deformed mirror symmetry for a 4-punctured torus dimer. We start by explaining the dimer involved together with its zigzag curve. Then we turn to the Jacobi algebra and the mirror objects. We explain all midpoint polygons and compute the deformed Landau-Ginzburg model and the deformed matrix factorizations. Centrality of the deformed Landau-Ginzburg potential obtained is not obvious, and we provide a manual check.

\paragraph*{The dimer and its mirror} The torus dimer we regard is depicted in \autoref{fig:examples-torus-fig}. The dimer has four punctures $ q_1 $, $ q_2 $, $ q_3 $, $ q_4 $, eight arcs $ a_1, …, a_4 $ and $ b_1, …, b_4 $ and four elementary polygons. It has four zigzag paths and is geometrically consistent. The associated zigzag curves are depicted in \autoref{fig:examples-torus-zigzag}. The figure also depicts a few sample midpoint polygons of different sizes (without indication which is the first and the last midpoint).

The mirror dimer $ \mirQ $ is a 4-punctured torus as well, depicted in \autoref{fig:examples-torus-mirror}. In the numbering of the figure, the correspondence between the punctures $ 1, 2, 3, 4 $ of $ \mirQ $ and zigzag paths of $ Q $ is as follows:
\begin{align*}
1 ∈ \mirQ_0 & \longleftrightarrow L_1 = …, b_1, a_2, b_4, a_3, … \\
2 ∈ \mirQ_0 &\longleftrightarrow L_2 = …, b_2, a_3, b_3, a_2, … \\
3 ∈ \mirQ_0 &\longleftrightarrow L_3 = …, a_1, b_2, a_4, b_3, … \\
4 ∈ \mirQ_0 &\longleftrightarrow L_4 = …, a_1, b_4, a_4, b_1, …
\end{align*}
For instance, the two associated zigzag curves $ \smooth L_1 $ and $ \smooth L_2 $ are depicted in \autoref{fig:examples-torus-Xa2}. The superpotential $ W ∈ ℂ\mirQ $ equals
\begin{equation*}
W = (b_1 a_4 b_2 a_2)_{\cyc} + (b_4 a_1 b_3 a_3)_{\cyc} - (b_2 a_3 b_1 a_1)_{\cyc} - (b_3 a_2 b_4 a_4)_{\cyc}.
\end{equation*}
The Jacobi algebra $ \Jac \mirQ = ℂ\mirQ / (∂_a W)_{a ∈ Q_1} $ is a noncommutative quiver algebra with relations. The dimer $ \mirQ $ is zigzag consistent and therefore $ \Jac \mirQ $ is CY3. Its 12 relations read
\begin{align*}
a_4 b_2 a_2 &= a_1 b_2 a_3, \\
a_2 b_1 a_4 &= a_3 b_1 a_1, \\
… &= ….
\end{align*}
The potential $ ℓ ∈ \Jac \mirQ $ reads
\begin{equation*}
ℓ = (b_1 a_4 b_2 a_2)_{\cyc}.
\end{equation*}
The mirror objects are the eight matrix factorizations
\begin{align*}
M_{a_1} = \dmatf{(\Jac \mirQ) L_3}{(\Jac \mirQ) L_4}{a_1}{b_3 a_3 b_4}, \quad
M_{a_2} = \dmatf{(\Jac \mirQ) L_2}{(\Jac \mirQ) L_1}{a_2}{b_1 a_4 b_2}, \\
M_{a_3} = \dmatf{(\Jac \mirQ) L_2}{(\Jac \mirQ) L_1}{a_1}{b_4 a_1 b_3}, \quad
M_{a_4} = \dmatf{(\Jac \mirQ) L_3}{(\Jac \mirQ) L_4}{a_4}{b_2 a_2 b_1}, \\
M_{b_1} = \dmatf{(\Jac \mirQ) L_1}{(\Jac \mirQ) L_3}{b_1}{a_4 b_2 a_2}, \quad
M_{b_2} = \dmatf{(\Jac \mirQ) L_4}{(\Jac \mirQ) L_2}{b_2}{a_2 b_1 a_4}, \\
M_{b_3} = \dmatf{(\Jac \mirQ) L_4}{(\Jac \mirQ) L_2}{b_3}{a_3 b_4 a_1}, \quad
M_{b_4} = \dmatf{(\Jac \mirQ) L_1}{(\Jac \mirQ) L_3}{b_4}{a_1 b_3 a_3}.
\end{align*}

\input{examples/fig_torus.tex}

\paragraph*{The deformed superpotential}
The deformed gentle algebra $ \Gtl_q Q $ has four deformation parameters $ q_1 $, $ q_2 $, $ q_3 $, $ q_4 $. To compute the deformed superpotential $ W_q $, we have to enumerate all midpoint polygons in $ Q $. As can be seen from \autoref{fig:examples-torus-zigzag}, there are infinitely many midpoint polygons. All midpoint polygons are rectangular, in particular for the sign $ |D| $ we have $ |D| = 0 ∈ ℤ/2ℤ $ for every midpoint polygon $ D $. The midpoint polygons can be classified into 64 different types according to their arc sequence, or 16 if one takes arc sequences up to cyclic permutation. All 16 types are listed in \autoref{tab:examples-torus-polygonsenumeration}.

The midpoint polygons of any type are a family indexed by natural numbers $ k, l ∈ ℕ $, standing for the width and height of the polygon. More precisely, define side lengths of a midpoint polygon to be the number of angles cut by the zigzag segments. In these terms, every family comes with a minimally small midpoint polygon and all larger midpoint polygons in the family are derived from this minimal version by extending side lengths by multiples of $ 4 $. For every of the 16 types, we have indicated in \autoref{tab:examples-torus-polygonsenumeration} the $ q $-parameters of the midpoint polygon of size $ (k, l) $. This way, we have efficiently enumerated all midpoint polygons and their properties in 16 types and two parameters.

We use the following abbreviations:
\begin{equation*}
q = q_1 q_2 q_3 q_4, ~ q_{14} = q_1 q_4, ~ q_{23} = q_2 q_3.
\end{equation*}
With this in mind, the deformed superpotential can be expressed as a sum over sixteen terms, eight of which actually cancel out due to symmetry of the 4-punctured torus:
\begin{align*}
W_q =& \sum_{\substack{\text{8 clockwise} \\ \text{types}}} \sum_{k, l ≥ 0} \punctures(D) \arcsequence(D)_{\cyc} - \sum_{\substack{\text{8 counterclockwise} \\ \text{types}}} \sum_{k, l ≥ 0} \punctures(D) \arcsequence(D)_{\cyc} \\
= & \left(\sum_{k, l ≥ 0} q_{23}^l q_{14}^k q^{2kl} - \sum_{k, l ≥ 0} q (q_{14} q)^l (q_{23} q)^k q^{2kl}\right) (b_2 a_2 b_1 a_4)_{\cyc} \\
& + \left(\sum_{k, l ≥ 0} q (q_{23} q)^l (q_{14} q)^k q^{2kl} - \sum_{k, l ≥ 0} q_{14}^l q_{23}^k q^{2kl}\right) (b_3 a_2 b_4 a_4)_{\cyc} \\
& + \left(\sum_{k, l ≥ 0} q_{23}^l q_{14}^k q^{2kl} - \sum_{k, l ≥ 0} q (q_{14} q)^l (q_{23} q)^k q^{2kl}\right) (b_3 a_3 b_4 a_1)_{\cyc} \\
& + \left(\sum_{k, l ≥ 0} q (q_{23} q)^l (q_{14} q)^k q^{2kl} - \sum_{k, l ≥ 0} q_{14}^l q_{23}^k q^{2kl}\right) (b_2 a_3 b_1 a_1)_{\cyc}.
\end{align*}
The deformed Jacobi algebra is $ \Jac_q \mirQ = ℂ\mirQ ⟦q_1, q_2, q_3, q_4⟧ / \closure{(∂_a W_q)_{a ∈ \mirQ_1}} $. Here $ \closure{(∂_a W_q)_{a ∈ \mirQ_1}} $ denotes the closure of the ideal generated by the derivatives $ ∂_a W_q ∈ ℂ\mirQ ⟦q_1, q_2, q_3, q_4⟧ $ with respect to the $ (q_1, q_2, q_3, q_4) $-adic topology.

\begin{table}
\centering
\begin{tabular}{ccc}
Orientation & $ q $-parameter $ \punctures(D) $ & arc sequence $ \arcsequence(D) $ \\\hline
clockwise & $ q_{23}^l q_{14}^k q^{2kl} $ & $ b_1 a_4 b_2 a_2 $ \\ 
clockwise & $ q_3 q_{23}^l (q_{14} q)^k q^{2kl} $ & $ b_4 a_1 b_2 a_2 $ \\
clockwise & $ q_4 (q_{23} q)^l q_{14}^k q^{2kl} $ & $ b_1 a_1 b_3 a_2 $ \\
clockwise & $ q (q_{23} q)^l (q_{14} q)^k q^{2kl} $ & $ b_4 a_4 b_3 a_2 $ \\\hline
clockwise & $ q_{23}^l q_{14}^k q^{2kl} $ & $ b_4 a_1 b_3 a_3 $ \\ 
clockwise & $ q_2 q_{23}^l (q_{14} q)^k q^{2kl} $ & $ b_1 a_4 b_3 a_3 $ \\
clockwise & $ q_1 (q_{23} q)^l q_{14}^k q^{2kl} $ & $ b_4 a_4 b_2 a_3 $ \\
clockwise & $ q (q_{23} q)^l (q_{14} q)^k q^{2kl} $ & $ b_1 a_1 b_2 a_3 $ \\\hline
counterclockwise & $ q_{14}^l q_{23}^k q^{2kl} $ & $ b_2 a_3 b_1 a_1 $ \\ 
counterclockwise & $ q_4 q_{14}^l (q_{23} q)^k q^{2kl} $ & $ b_3 a_2 b_1 a_1 $ \\
counterclockwise & $ q_3 (q_{14} q)^l q_{23}^k q^{2kl} $ & $ b_2 a_2 b_4 a_1 $ \\
counterclockwise & $ q (q_{14} q)^l (q_{23} q)^k q^{2kl} $ & $ b_3 a_3 b_4 a_1 $ \\\hline
counterclockwise & $ q_{14}^l q_{23}^k q^{2kl} $ & $ b_3 a_2 b_4 a_4 $ \\ 
counterclockwise & $ q_1 q_{14}^l (q_{23} q)^k q^{2kl} $ & $ b_2 a_3 b_4 a_4 $ \\
counterclockwise & $ q_2 (q_{14} q)^l q_{23}^k q^{2kl} $ & $ b_3 a_3 b_1 a_4 $ \\
counterclockwise & $ q (q_{14} q)^l (q_{23} q)^k q^{2kl} $ & $ b_2 a_2 b_1 a_4 $ \\
\end{tabular}
\caption{Enumeration of midpoint polygons}
\label{tab:examples-torus-polygonsenumeration}
\end{table}

\paragraph*{The deformed potential}
The deformed potential $ ℓ_q $ as element of $ ℂ\mirQ ⟦q_1, q_2, q_3, q_4⟧ $ is the sum of four potentials, one for each of the four zigzag paths:
\begin{equation*}
ℓ_q = ℓ_{q, 1} + ℓ_{q, 2} + ℓ_{q, 3} + ℓ_{q, 4}.
\end{equation*}
The expression for the deformed potential $ ℓ_{q, i} $ as element of $ ℂ\mirQ ⟦q_1, q_2, q_3, q_4⟧ $ depends on the choice of identity location for the zigzag path $ L_i $. For instance, let us choose the identity of the zigzag path $ L_1 = …, a_2, b_1, a_3, b_4, … $ to be the arc $ a_2 $. Then the summands $ ℓ_{q, 1} $ and $ ℓ_{q, 2} $ read
\begin{align*}
ℓ_{q, 1} =& \left(\sum_{k, l ≥ 0} l q_{23}^l q_{14}^k q^{2kl} b_1 a_4 b_2 a_2 + \sum_{k, l ≥ 0} (l+1) q (q_{14} q)^l (q_{23} q)^k q^{2kl}\right) b_1 a_4 b_2 a_2 \\
& + \left(\sum_{k, l ≥ 0} l q_4 (q_{23} q)^l q_{14}^k q^{2kl} + \sum_{k, l ≥ 0} (l+1) q_4 q_{14}^l (q_{23} q)^k q^{2kl}\right) b_1 a_1 b_3 a_2 \\
& + \left(\sum_{k, l ≥ 0} l q_3 q_{23}^l (q_{14} q)^k q^{2kl} + \sum_{k, l ≥ 0} (l+1) q_3 (q_{14} q)^l q_{23}^k q^{2kl} \right) b_4 a_1 b_2 a_2 \\
& + \left(\sum_{k, l ≥ 0} l q (q_{23} q)^l (q_{14} q)^k q^{2kl} + \sum_{k, l ≥ 0} (l+1) q_{14}^l q_{23}^k q^{2kl}\right) b_4 a_4 b_3 a_2 \\
& + \left(\sum_{k, l ≥ 0} (l+1) q_2 q_{23}^l (q_{14} q)^k q^{2kl} + \sum_{k, l ≥ 0} l q_2 (q_{14} q)^l q_{23}^k q^{2kl}\right) b_1 a_4 b_3 a_3 \\
& + \left(\sum_{k, l ≥ 0} (l+1) q (q_{23} q)^l (q_{14} q)^k q^{2kl} + \sum_{k, l ≥ 0} l q_{14}^l q_{23}^k q^{2kl}\right) b_1 a_1 b_2 a_3 \\
& + \left(\sum_{k, l ≥ 0} l q_{23}^l q_{14}^k q^{2kl} + \sum_{k, l ≥ 0} (l+1) q (q_{14} q)^l (q_{23} q)^k q^{2kl}\right) b_4 a_1 b_3 a_3 \\
& + \left(\sum_{k, l ≥ 0} l q_1 (q_{23} q)^l q_{14}^k q^{2kl} + \sum_{k, l ≥ 0} (l+1) q_1 q_{14}^l (q_{23} q)^k q^{2kl}\right) b_4 a_4 b_2 a_3
\end{align*}
and
\begin{align*}
ℓ_{q, 2} =& \left(\sum_{k, l ≥ 0} (k+1) q_{23}^l q_{14}^k q^{2kl} + \sum_{k, l ≥ 0} k q (q_{14} q)^l (q_{23} q)^k q^{2kl}\right) a_2 b_1 a_4 b_2 \\
& + \left(\sum_{k, l ≥ 0} (k+1) q_4 (q_{23} q)^l q_{14}^k q^{2kl} + \sum_{k, l ≥ 0} k q_4 q_{14}^l (q_{23} q)^k q^{2kl}\right) a_2 b_1 a_1 b_3 \\
& + \left(\sum_{k, l ≥ 0} (k+1) q_3 (q_{23})^l (q_{14} q)^k q^{2kl} + \sum_{k, l ≥ 0} k q_3 (q_{14} q)^l q_{23}^k q^{2kl}\right) a_2 b_4 a_1 b_2 \\
& + \left(\sum_{k, l ≥ 0} (k+1) q (q_{23} q)^l (q_{14} q)^k q^{2kl} + \sum_{k, l ≥ 0} k q_{14}^l q_{23}^k q^{2kl}\right) a_2 b_4 a_4 b_3 \\
& + \left(\sum_{k, l ≥ 0} k (q_{23})^l (q_{14})^k q^{2kl} + \sum_{k, l ≥ 0} (k+1) q (q_{14} q)^l (q_{23} q)^k q^{2kl}\right) a_3 b_4 a_1 b_3 \\
& + \left(\sum_{k, l ≥ 0} (k+1) q_1 (q_{23} q)^l q_{14}^k q^{2kl} + \sum_{k, l ≥ 0} k q_1 q_{14}^l (q_{23} q)^k q^{2kl}\right) a_3 b_4 a_4 b_2 \\
& + \left(\sum_{k, l ≥ 0} k q_2 q_{23}^l (q_{14} q)^k q^{2kl} + \sum_{k, l ≥ 0} (k+1) q_2 (q_{14} q)^l q_{23}^k q^{2kl}\right) a_3 b_1 a_4 b_3 \\
& + \left(\sum_{k, l ≥ 0} (k+1) q (q_{23} q)^l (q_{14} q)^k q^{2kl} + \sum_{k, l ≥ 0} k q_{14}^l q_{23}^k q^{2kl}\right) a_3 b_1 a_1 b_2.
\end{align*}

\paragraph*{Centrality}
The general mirror construction guarantees that the deformed potential $ ℓ_q $ is central in the deformed Jacobi algebra $ \Jac_q \mirQ $. In what follows, we verify centrality manually in the case of the 4-punctured torus. Our starting point is the explicit description of $ W_q $ and $ ℓ_q $ in terms of midpoint polygons. It turns out that centrality is not as obvious as in the classical case of $ ℓ ∈ \Jac \mirQ $. The checks we present incarnate particular cases of the $ A_∞ $-relations for $ \H\DefZigzagCat $, in fact going beyond the transversal case. Checking centrality therefore presents evidence that the description of $ \H\DefZigzagCat $ in terms of CR, ID, DS and DW disks is accurate and is therefore evidence for the correctness of \papertwoB.

To get started, recall that proving $ ℓ_q $ central entails finding for every arc $ a ∈ \mirQ_1 $ an element $ x ∈ \closure{(∂_a W_q)} $ such that $ a ℓ_q + x = ℓ_q a $ within $ ℂ\mirQ $. In the classical case of $ ℓ ∈ \Jac \mirQ $, this element $ x $ can be described as a sequence of F-term flips of the path $ ℓ a ∈ ℂ\mirQ $, see \autoref{sec:flatness-dimers}. In the deformed case of $ ℓ_q ∈ \Jac_q \mirQ $, the potential $ ℓ_q $ is much more complicated and the element $ x $ can only be obtained by inspecting and Koszul dualizing the $ A_∞ $-relations of $ \H\DefZigzagCat $.

\input{examples/fig_torus_Xa2.tex}

For symmetry reasons, it suffices to prove $ ℓ_q a - a ℓ_q ∈ \closure{(∂_a W_q)} $ for a single arc $ a ∈ \mirQ_1 $, which we choose to be the arc $ a = a_2 $. Recall that $ t(a_2) = L_1 $ and $ h(a_2) = L_2 $. The arrow set $ \mirQ_1 $ identifies with odd transversal intersections of zigzag paths in $ Q $, and we write $ X_a $ for the odd basis morphism whose corresponding intersection point is located at the midpoint of the arc $ a $. For instance, the two zigzag paths $ L_1 $ and $ L_2 $ and the morphism $ X_{a_2}: L_1 → L_2 $ are depicted in \autoref{fig:examples-torus-Xa2}. Recall that the element $ b $ is a formal sum of the odd transversal intersections $ X_a $, weighted by their formal parameter $ a ∈ \mirQ_1 $. In our case, we have explicitly
\begin{equation*}
b = \sum_{i = 1}^4 a_i X_{a_i} + \sum_{i = 1}^4 b_i X_{b_i}.
\end{equation*}
The first step to centrality is to obtain a candidate expression for $ x = ℓ_{q, 2} a_2 - a_2 ℓ_{q, 1} $ in terms of the relations $ ∂_a W_q $. The idea is to inspect those $ A_∞ $-relations which lead to centrality in the proof of \autoref{th:CHL-defLG-centrality}:
\begin{align}
\label{eq:examples-torus-aid}
\begin{split}
& μ_{\H\DefZigzagCat} (μ^{≥0} (b, …, b)) + μ_{\H\DefZigzagCat} (b, μ^{≥0}_{\H\DefZigzagCat} (b, …, b)) \\
+ & μ_{\H\DefZigzagCat} (μ^{≥0} (b, …, b), b) + μ_{\H\DefZigzagCat}^{≥3} (b, …, μ^{≥0} (b, …, b), …, b) = 0.
\end{split}
\end{align}
The expression on the left-hand side of \eqref{eq:examples-torus-aid} is a sum of multiple odd morphisms weighted by paths in $ \mirQ $. To guess $ x $, we extract the coefficient of the odd morphism $ X_{a_2} $.

Let us inspect all four terms of \eqref{eq:examples-torus-aid}. The first term vanishes in the case of the 4-punctured torus, since it is geometrically consistent. For the second and third term, recall that $ μ^{≥0} (b, …, b) $ is the sum of the idententities of the zigzag paths $ L_i $ weighted by $ ℓ_{q, i} $, and the even intersections $ Y_a $ weighted by the relations $ ∂_a W_q $:
\begin{equation*}
μ^{≥0} (b, …, b) = \sum_{i = 1}^4 ℓ_{q, i} \id_{L_i} + \sum_{i = 1}^4 ∂_{a_i} W_q Y_{a_i} + \sum_{i = 1}^4 ∂_{b_i} W_q Y_{b_i}.
\end{equation*}
We insert this expression into the second and third terms of \eqref{eq:examples-torus-aid} and extract the coefficient of $ X_{a_2} $. Since there are no 3-gons among zigzag curves in $ Q $, the $ X_{a_2} $ coefficient of the second term is $ a_2 ℓ_{q, 1} $ and the $ X_{a_2} $ coefficient of the third term is $ - ℓ_{q, 2} a_2 $.

For the fourth term in \eqref{eq:examples-torus-aid}, we have to count CR, ID, DS and DW disks with output $ X_{a_2} $ and input sequence consisting of odd elements of the form $ X_{a_i} $ or $ X_{b_i} $, mixed with one single even input of the form $ Y_{a_i} $ or $ Y_{b_i} $. The $ X_{a_2} $ coefficient from every CR, ID, DS or DW disk is then the path in $ \mirQ $ given by the concatenation of the $ a_i $ or $ b_i $ symbols, inserting the relation $ ∂_{a_i} W_q $ or $ ∂_{b_i} W_q $ instead at the even input.

It remains to enumerate all relevant CR, ID, DS and DW disks explicitly in the case of the 4-punctured torus:
\begin{itemize}
\item CR disks are 4-gons, in fact rectangles due to the shape of $ Q $. CR disks contribute both to the inner and to the outer $ μ_{\H\DefZigzagCat}^3 $ in \eqref{eq:examples-torus-aid}.
\item ID disks are 5-gons, in fact rectangles of which one side has an output marking on the identity location $ a_2 $. They are only relevant for the inner $ μ_{\H\DefZigzagCat} $ since their output is an identity and not $ X_{a_2} $.
\item DS disks are only relevant for the outer $ μ_{\H\DefZigzagCat} $ since they require at least one even input. According to the intricate rules for DS disks, the only DS disk with output $ X_{a_2} $ is the DS disk with infinitesimally small side lengths contributing to $ μ^3_{\H\DefZigzagCat} (X_{a_2}, Y_{a_2}, X_{a_2}) $. The existence of this DS disk is independent of the choice of co-identity location. In terms of the Kadeishvili construction from \papertwoB, this DS disk corresponds to the π-tree with associated DS result component depicted in \autoref{fig:examples-torus-DStree}.
\item DW disks are irrelevant. In fact, a DW disk has at least one even input and has a co-identity output and can therefore neither contribute to the inner $ μ_{\H\DefZigzagCat} $ nor contribute $ X_{a_2} $ to the outer $ μ_{\H\DefZigzagCat} $.
\end{itemize}
This finishes our evaluation of \eqref{eq:examples-torus-aid}. All in all, the $ A_∞ $-relations for $ \H\DefZigzagCat $ claim that the $ X_{a_2} $ coefficients of \eqref{eq:examples-torus-aid} add up to zero. This provides us with a candidate expression for $ x = ℓ_{q, 2} a_2 - a_2 ℓ_{q, 1} $ in terms of the relations $ ∂_a W_q $. We record and verify this guess as follows:

\begin{figure}
\centering
\begin{tikzpicture}
\path (0, 0) node (A) {$ α_4 $} (2, 0) node (B) {$ \id $(C)} (4, 0) node (C) {$ α_4 $}
(3, -1.3) node[align=center] (D) {$ μ^2 = α_4 $ \\ $ h_q = \id_{b_1} $} edge (B) edge (C)
(2, -2.6) node[align=center] {$ μ^2 = α_4 $ \\ $ π_q = α_3 + α_4 $} edge (A) edge (D);
\end{tikzpicture}
\caption{DS result component contributing to $ μ^3 (X_{a_2}, Y_{a_2}, X_{a_2}) $}
\label{fig:examples-torus-DStree}
\end{figure}
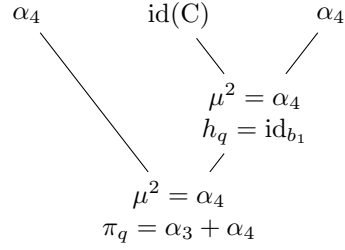

\begin{lemma}
Regard the 4-punctured torus $ Q $ with its zigzag paths $ L_1, L_2, L_3, L_4 $. Denote by $ ℓ_{q, 1} $ and $ ℓ_{q, 2} $ the potentials in $ ℂ\mirQ⟦q_1, q_2, q_3, q_4⟧ $ associated with $ L_1 $ and $ L_2 $ under the choice of $ a_2 $ as identity location. Then within $ ℂ\mirQ ⟦q_1, q_2, q_3, q_4⟧ $ we have
\begingroup
\allowdisplaybreaks
\begin{align}
\label{eq:examples-torus-centrality}
ℓ_{q, 2} a_2 - a_2 ℓ_{q, 1} = & ~ a_2 (∂_{a_2} W_q) a_2 \\
\nonumber & + \bigg(\sum_{k, l ≥ 0} q_1 q^l q_{14}^k q^{2kl} - \sum_{k, l ≥ 0} q_1 q_{23} q^l (q_{23} q)^k q^{2kl}\bigg) a_3 b_4 (∂_{b_1} W_q) \\
\nonumber & + \bigg(\sum_{k, l ≥ 0} q_{14} q^{2l} q_{14}^k q^{2kl} - \sum_{k, l ≥ 0} q_{23} q q^{2l} (q_{23} q)^k q^{2kl}\bigg) a_2 b_1 (∂_{b_1} W_q) \\
\nonumber & + \bigg(\sum_{k, l ≥ 0} q_1 q_2 q_4 q^l (q_{14} q)^k q^{2kl} - \sum_{k, l ≥ 0} q_2 q^l q_{23}^k q^{2kl} \bigg) a_3 b_1 (∂_{b_4} W_q) \\
\nonumber & + \bigg(\sum_{k, l ≥ 0} q_{14} q q^{2l} (q_{14} q)^k q^{2kl} - \sum_{k, l ≥ 0} q_{23} q^{2l} q_{23}^k q^{2kl} \bigg) a_2 b_4 (∂_{b_4} W_q) \\
\nonumber & + \bigg(\sum_{k, l ≥ 0} q_1 q_2 q^l q^k q^{2kl} - \sum_{k, l ≥ 0} q_1 q_2 q^l q^k q^{2kl} \bigg) a_3 (∂_{a_2} W_q) a_3 \\
\nonumber & + \bigg(\sum_{k, l ≥ 0} q^2 q^{2l} q^{2k} q^{2kl} - \sum_{k, l ≥ 0} q^2 q^{2l} q^{2k} q^{2kl} \bigg) a_2 (∂_{a_2} W_q) a_2 \\
\nonumber & + \bigg(\sum_{k, l ≥ 0} q q^l q^{2k} q^{2kl} - \sum_{k, l ≥ 0} q q^l q^{2k} q^{2kl} \bigg) a_3 (∂_{a_3} W_q) a_2 \\
\nonumber & + \bigg(\sum_{k, l ≥ 0} q q^{2l} q^k q^{2kl} - \sum_{k, l ≥ 0} q q^{2l} q^k q^{2kl} \bigg) a_2 (∂_{a_3} W_q) a_3 \\
\nonumber & + \bigg(\sum_{k, l ≥ 0} q_2 q_{14} (q_{14} q)^l q^k q^{2kl} - \sum_{k, l ≥ 0} q_2 q_{23}^l q^k q^{2kl} \bigg) (∂_{b_2} W_q) b_3 a_3 \\
\nonumber & + \bigg(\sum_{k, l ≥ 0} q_{14} q (q_{14} q)^l q^{2k} q^{2kl} - \sum_{k, l ≥ 0} q_{23} q_{23}^l q^{2k} q^{2kl} \bigg) (∂_{b_2} W_q) b_2 a_2 \\
\nonumber & + \bigg(\sum_{k, l ≥ 0} q_1 q_{14}^l q^k q^{2kl} - \sum_{k, l ≥ 0} q_1 q_{23} (q_{23} q)^l q^k q^{2kl} \bigg) (∂_{b_3} W_q) b_2 a_3 \\
\nonumber & + \bigg(\sum_{k, l ≥ 0} q_{14} q_{14}^l q^{2k} q^{2kl} - \sum_{k, l ≥ 0} q_{23} q (q_{23} q)^l q^{2k} q^{2kl} \bigg) (∂_{b_3} W_q) b_3 a_2.
\end{align}
\endgroup
In particular, we have $ ℓ_q a_2 = a_2 ℓ_q $ within $ \Jac_q \mirQ $.
\end{lemma}

\begin{proof}
The strategy is to cancel terms on the right-hand side against each other. The correct means of finding cancellation partners is by tracing back all terms to their respective geometric origins. More precisely, recall that the claimed identity \eqref{eq:examples-torus-centrality} is supposed to reflect the $ A_∞ $-relations for $ \H\DefZigzagCat $. The $ A_∞ $-relations again reflect the geometric property that nonconvex disks can be divided into convex disks in two different ways. Once we trace back every term on the right-hand side of \eqref{eq:examples-torus-centrality} to its nonconvex disk origin, we find the correct cancellation partner. A few corner cases remain in which the disk is actually convex, and these correspond to the left-hand side.

It might be tempting to cancel the eight middle terms on the right-hand side of \eqref{eq:examples-torus-centrality} first. However, the cancellation of these eight terms is specific to the case of the 4-punctured torus and we in fact need to preserve these terms in order to make the other cancellations transparent.

To get started, note that the identity \eqref{eq:examples-torus-centrality} consists of $ ℂ⟦q_1, q_2, q_3, q_4⟧ $-linear combinations of paths of length five. Checking the identity entails comparing coefficients of every of the possible paths. In the remainder of the proof, we merely focus on two example paths: $ a_2 b_1 a_1 b_3 a_2 $ and $ a_2 b_1 a_1 b_2 a_3 $. Out of these two, the case of the first path is easily settled because the right-hand side contains no such paths. The case of the second path is much harder and truly illustrates the geometric reason for equality. The authors also checked manually the coefficients of all 14 other paths and used the results to eradicate all errors in \eqref{eq:examples-torus-centrality}. For the present demonstration, we merely restrict to the two paths mentioned.

Let us comment on the first path $ a_2 b_1 a_1 b_3 a_2 $. We determine the coefficient of this path in all four entities of \eqref{eq:examples-torus-centrality}. The coefficient coming from the term $ ℓ_{q, 2} a_2 $ is
\begin{equation*}
\sum_{k, l ≥ 0} (k+1) q_4 (q_{23} q)^l q_{14}^k q^{2kl} + \sum_{k, l ≥ 0} k q_4 q_{14}^l (q_{23} q)^k q^{2kl}.
\end{equation*}
The coefficient coming from the term $ a_2 ℓ_{q, 1} $ is
\begin{equation*}
\sum_{k, l ≥ 0} l q_4 (q_{23} q)^l q_{14}^k q^{2kl} + \sum_{k, l ≥ 0} (l+1) q_4 q_{14}^l (q_{23} q)^k q^{2kl}.
\end{equation*}
The coefficient coming from the term $ a_2 (∂_{a_2} W_q) a_2 $ is
\begin{equation*}
\sum_{k, l ≥ 0} q_4 q_{14}^k (q_{23} q)^l q^{2kl} - \sum_{k, l ≥ 0} q_4 q_{14}^l (q_{23} q)^k q^{2kl} = 0.
\end{equation*}
Note that this vanishing is specific to the 4-punctured torus. The coefficient in the rest of the right-hand side of \eqref{eq:examples-torus-centrality} vanishes because $ ∂_{b_1} W_q $ has vanishing $ a_1 b_3 a_2 $ coefficient, because $ ∂_{a_2} W_q $ has vanishing $ b_1 a_1 b_3 $ coefficient, and because $ ∂_{b_3} W_q $ has vanishing $ a_2 b_1 a_1 $ coefficient. All in all, both the left-hand side and the right-hand side of \eqref{eq:examples-torus-centrality} cancel out. This proves the identity \eqref{eq:examples-torus-centrality} on the level of $ a_2 b_1 a_1 b_3 a_2 $ terms.

We now focus on the second path $ a_2 b_1 a_1 b_2 a_3 ∈ ℂ\mirQ $ and prove that its coefficient in \eqref{eq:examples-torus-centrality} vanishes. Regarding the left-hand side, the coefficient in $ ℓ_{q, 2} a_2 $ vanishes and the coefficient in $ - a_2 ℓ_{q, 1} $ consists of the two sums
\begin{equation}
\label{eq:examples-torus-lcontrib}
- \sum_{k, l ≥ 0} (l+1) q (q_{23} q)^l (q_{14} q)^k q^{2kl} - \sum_{k, l ≥ 0} l q_{14}^l q_{23}^k q^{2kl}.
\end{equation}
Regarding the right-hand side, the coefficient in $ a_2 (∂_{a_2} W_q) a_2 $ vanishes. The relevant contribution of the 24 remaining relation terms is
\begin{align}
\label{eq:examples-torus-centralityRHS}
\begin{split}
& + \sum_{k, l ≥ 0} q_1 q_4 q^{2l} q_{14}^k q^{2kl} a_2 b_1 (∂_{b_1} W_q) \\ 
& - \sum_{k, l ≥ 0} q_2 q_3 q q^{2l} (q_{23} q)^k q^{2kl} a_2 b_1 (∂_{b_1} W_q) \\ 
& - \sum_{k, l ≥ 0} q q^{2l} q^k q^{2kl} a_2 (∂_{a_3} W_q) a_3 \\ %
& + \sum_{k, l ≥ 0} q q^{2l} q^k q^{2kl} a_2 (∂_{a_3} W_q) a_3 \\
& + \sum_{k, l ≥ 0} q_1 q_{14}^l q^k q^{2kl} (∂_{b_3} W_q) b_2 a_3 \\
& - \sum_{k, l ≥ 0} q_1 q_2 q_3 (q_{23} q)^l q^k q^{2kl} (∂_{b_3} W_q) b_2 a_3
\end{split}
\end{align}
We count the $ a_1 b_2 a_3 $ coefficients in $ ∂_{b_1} W_q $, $ ∂_{a_3} W_q $ and $ ∂_{b_3} W_q $ as follows:
\begin{align*}
∂_{b_1} W_q \restr{a_1 b_2 a_3} = & + \sum_{k, l ≥ 0} q_{14}^{2k+1+l+2kl} q_{23}^{2l+1+k+2kl} \\
& - \sum_{k, l ≥ 0} q_{14}^{l+2kl} q_{23}^{k+2kl} \\
∂_{a_3} W_q \restr{b_1 a_1 b_2} = & + \sum_{k, l ≥ 0} q_{14}^{2k+l+1+2kl} q_{23}^{2l+k+1+2kl} \\
& - \sum_{k, l ≥ 0} q_{14}^{l+2kl} q_{23}^{k+2kl} \\
∂_{b_3} | a_2 b_1 a_1 = & + \sum_{k, l ≥ 0} q_4 q_{14}^{k+l+2kl} q_{23}^{2l+2kl} \\
& - \sum_{k, l ≥ 0} q_4 q_{14}^{k+l+2kl} q_{23}^{2k+2kl}
\end{align*}
Plugging these expressions into \eqref{eq:examples-torus-centralityRHS}, we get twelve terms which we name as follows:
\begin{align*}
\text{1A:} & \quad + \sum_{k, l, k', l' ≥ 0} q_{14}^{(k+1)(2l+1) + (2k'+1)(l'+1)} q_{23}^{2(k+1)l + (k'+1)(2l'+1)} \\
\text{1B:} & \quad - \sum_{k, l, k', l' ≥ 0} q_{14}^{(k+1)(2l+1) + l'(2k'+1)} q_{23}^{2(k+1)l+k'(2l'+1)}
\end{align*}
\begin{align*}
\text{2A:} & \quad - \sum_{k, l, k', l' ≥ 0} q_{14}^{(k+1)(2l+1) + (2k'+1)(l'+1)} q_{23}^{2(k+1)(l+1) + (k'+1)(2l'+1)} \\
\text{2B:} & \quad + \sum_{k, l, k', l' ≥ 0} q_{14}^{(k+1)(2l+1) + (2k'+1)l'} q_{23}^{2(k+1)(l+1) + k'(2l'+1)}
\end{align*}
\begin{align*}
\text{3A:} & \quad - \sum_{k, l, k', l' ≥ 0} q_{14}^{(k+1)(2l+1) + (2k'+1)(l'+1)} q_{23}^{(k+1)(2l+1) + (k'+1)(2l'+1)} \\
\text{3B:} & \quad + \sum_{k, l, k', l' ≥ 0} q_{14}^{(k+1)(2l+1) + (2k'+1) l'} q_{23}^{(k+1)(2l+1) + k'(2l'+1)}
\end{align*}
\begin{align*}
\text{4A:} & \quad + \sum_{k, l, k', l' ≥ 0} q_{14}^{(k+1)(2l+1) + (2k'+1)(l'+1)} q_{23}^{(k+1)(2l+1) + (k'+1)(2l'+1)} \\
\text{4B:} & \quad - \sum_{k, l, k', l' ≥ 0} q_{14}^{(k+1)(2l+1) + (2k'+1)l'} q_{23}^{(k+1)(2l+1) + k'(2l'+1)}
\end{align*}
\begin{align*}
\text{5A:} & \quad + \sum_{k, l, k', l' ≥ 0} q_{14}^{(1+l+k+2kl) + (k'+l'+2k'l')} q_{23}^{k(2l+1) + 2l'(k'+1)} \\
\text{5B:} & \quad - \sum_{k, l, k', l' ≥ 0} q_{14}^{(1+l+k+2kl) + (k'+l'+2k'l')} q_{23}^{k(2l+1) + 2k'(l'+1)}
\end{align*}
\begin{align*}
\text{6A:} & \quad - \sum_{k, l, k', l' ≥ 0} q_{14}^{(1+l+k+2kl) + (k'+l'+2k'l')} q_{23}^{(k+1)(2l+1) + 2(k'+1)l'} \\
\text{6B:} & \quad + \sum_{k, l, k', l' ≥ 0} q_{14}^{(1+l+k+2kl) + (k'+l'+2k'l')} q_{23}^{(k+1)(2l+1) + 2k'(l'+1)}
\end{align*}
It is our task to sum up these twelve terms (1A–6B) and prove the sum equal to \eqref{eq:examples-torus-lcontrib}. Again, note that 3A and 3B cancel with 4A and 4B. However, performing this easy cancellation would obfuscate all other cancellations. Instead, we shall obtain the correct cancellations by inspecting the geometric origin of every of these twelve terms.

We give a catalog of all cancellations below. To understand the reasons behind the cancellations, we shall demonstrate here how one finds the correct cancellation for the term 1A in case $ l ≥ l' + 1 $. Its geometric origin is depicted in \autoref{fig:examples-torus-cancellation1A}. In short, this figure depicts a pair of disks contributing to
\begin{equation}
\label{eq:examples-torus-centrality_Ainfty1}
μ_{\H\DefZigzagCat} (X_{a_2}, X_{b_1}, μ_{\H\DefZigzagCat} (X_{a_1}, X_{b_2}, X_{a_3})).
\end{equation}
More precisely, the disk labeled $ ∂_{b_1} W_q $ contributes to the inner $ μ_{\H\DefZigzagCat} $, while the disk labeled “coeff” contributes to the outer $ μ_{\H\DefZigzagCat} $. The side lengths of the disk labeled $ ∂_{b_1} W_q $ are $ 4k' + 3 $ and $ 4l' + 3 $. Correspondingly, the example depicted has $ k' = l' = 0 $. The side lengths of the disk labeled “coeff” are $ 4k+4 $ and $ 4l+1 $. Correspondingly, the example depicted has $ k = 0 $ and $ l = 1 $. The nonconvex shape can be split into two other pieces, depicted in \autoref{fig:examples-torus-cancellation3A}. That figure depicts a pair of disks contributing to
\begin{equation}
\label{eq:examples-torus-centrality_Ainfty2}
μ_{\H\DefZigzagCat} (X_{a_2}, μ_{\H\DefZigzagCat} (X_{b_1}, X_{a_1}, X_{b_2}), X_{a_3}).
\end{equation}
The two individual contributions to \eqref{eq:examples-torus-centrality_Ainfty1} and \eqref{eq:examples-torus-centrality_Ainfty2} cover the same area and therefore have the same deformation parameters. A quick check shows that they have equal sign. Noting that 1A and 3A come with opposite overall signs, this makes 3A our candidate for cancellation with 1A in case $ l ≥ l'+1 $.

The cancellation of 1A and 3A involves a change of indices. More precisely, the term 1A with indices $ (k, l, k', l') $ cancels with 3A with indices $ (k, l-l'-1, k+k'+1, l') $. To obtain this correspondence between indices, let us understand the disks in \autoref{fig:examples-torus-cancellation3A}. The disk labeled $ ∂_{a_3} W_q $ contributes to the inner $ μ_{\H\DefZigzagCat} $, while the disk labeled “coeff” contributes to the outer $ \H\DefZigzagCat $. The side lengths of the disk labeled $ ∂_{a_3} W_q $ are $ 4(k+k'+1) + 3 $ and $ 4l' + 3 $. The side lengths of the disk labeled “coeff” are $ 4k+4 $ and $ 4(l-l'-1) + 2 $. This means we expect to find the partner of 1A with indices $ (k, l, k', l') $ in 3A with indices $ (k, l-l'-1, k+k'+1, l') $.

Similar considerations for all other pairs of terms give rise to a list of cancellation candidates. In what follows, we compile this list and check for every entry individually that it indeed cancels out with its partner:
\begin{itemize}
\item 1A in case $ l ≥ l'+1 $ cancels with 3A with indices $ (k, l-l'-1, k+k'+1, l') $:
\begin{equation*}
\begin{aligned}
(q_{14}): && (k+1)(2l+1) + (2k'+1)(l'+1) &= (k+1)(2(l-l'-1)+1) + (2(k+k'+1)+1)(l'+1), \\
(q_{23}): && 2(k+1)l + (k'+1)(2l'+1) &= (k+1)(2(l-l'-1)+1) + ((k+k'+1)+1)(2l'+1).
\end{aligned}
\end{equation*}
\item 1A in case $ l ≤ l' $ cancels with 6A with indices $ (k', l'-l, k+k'+1, l) $:
\begin{equation*}
\begin{aligned}
(q_{14}): && (k+1)(2l+1) + (2k'+1)(l'+1) & \\
&& & \hspace{-10em} = (1+(l'-l)+k'+2k'(l'-l)) + ((k+k'+1)+l+2(k+k'+1)l), \\
(q_{23}): && 2(k+1)l + (k'+1)(2l'+1) &= (k'+1)(2(l'-l)+1) + 2((k+k'+1)+1)l.
\end{aligned}
\end{equation*}
\item 1B in case $ k ≥ k' $ cancels with 5A with indices $ (k', l+l', k-k', l) $:
\begin{equation*}
\begin{aligned}
(q_{14}): && (k+1)(2l+1) + l'(2k'+1) &= (1+(l+l')+k'+2k'(l+l')) + ((k-k')+l+2(k-k')l), \\
(q_{23}): && 2(k+1)l+k'(2l'+1) &= k'(2(l+l')+1) + 2l((k-k')+1).
\end{aligned}
\end{equation*}
\item 1B in case $ k+1 ≤ k' $ cancels with 3B with indices $ (k, l+l', k'-k-1, l') $:
\begin{equation*}
\begin{aligned}
(q_{14}): && (k+1)(2l+1) + l'(2k'+1) &= (k+1)(2(l+l')+1) + (2(k'-k-1)+1) l', \\
(q_{23}): && 2(k+1)l+k'(2l'+1) &= (k+1)(2(l+l')+1) + (k'-k-1)(2l'+1).
\end{aligned}
\end{equation*}
\item 2A in case $ k+1 ≤ k' $ cancels with 4A with indices $ (k, l+l'+1, k'-k-1, l') $:
\begin{equation*}
\begin{aligned}
(q_{14}): && (k+1)(2l+1) + (2k'+1)(l'+1) &= (k+1)(2(l+l'+1)+1) + (2(k'-k-1)+1)(l'+1), \\
(q_{23}): && 2(k+1)(l+1) + (k'+1)(2l'+1) &= (k+1)(2(l+l'+1)+1) + ((k'-k-1)+1)(2l'+1).
\end{aligned}
\end{equation*}
\item 2A in case $ k ≥ k' $ cancels with 6B with indices $ (k', l+l'+1, k-k', l) $:
\begin{equation*}
\begin{aligned}
(q_{14}): && (k+1)(2l+1) + (2k'+1)(l'+1) & \\
&& & \hspace{-10em} = (1+(l+l'+1)+k'+2k'(l+l'+1)) + ((k-k')+l+2(k-k')l) \\
(q_{23}): && 2(k+1)(l+1) + (k'+1)(2l'+1) &= (k'+1)(2(l+l'+1)+1) + 2(k-k')(l+1).
\end{aligned}
\end{equation*}
\item 2B in case $ l ≥ l' $ cancels with 4B with indices $ (k, l-l', k+k'+1, l') $:
\begin{equation*}
\begin{aligned}
(q_{14}): (k+1)(2l+1) + (2k'+1)l' &= (k+1)(2(l-l')+1) + (2(k+k'+1)+1)l'. \\
(q_{23}): 2(k+1)(l+1) + k'(2l'+1) &= (k+1)(2(l-l')+1) + (k+k'+1)(2l'+1).
\end{aligned}
\end{equation*}
\item 2B in case $ l+1 ≤ l' $ cancels with 5B with indices $ (k', l'-l-1, k+k'+1, l) $:
\begin{equation*}
\begin{aligned}
(q_{14}): && (k+1)(2l+1) + (2k'+1)l' & \\
&& & \hspace{-10em} = (1+(l'-l-1)+k'+2k'(l'-l-1)) + ((k+k'+1)+l+2(k+k'+1)l), \\
(q_{23}): && 2(k+1)(l+1) + k'(2l'+1) &= k'(2(l'-l-1)+1) + 2(k+k'+1)(l+1).
\end{aligned}
\end{equation*}
\item 3A in case $ k' ≤ k $ cancels with 5A with indices $ (k-k', l, k', l+l'+1) $:
\begin{equation*}
\begin{aligned}
(q_{14}): && (k+1)(2l+1) + (2k'+1)(l'+1) & \\
&& & \hspace{-10em} = (1+l+(k-k')+2(k-k')l) + (k'+(l+l'+1)+2k'(l+l'+1)), \\
(q_{23}): && (k+1)(2l+1) + (k'+1)(2l'+1) &= (k-k')(2l+1) + 2(l+l'+1)(k'+1).
\end{aligned}
\end{equation*}
\item 3B in case $ l' ≥ l+1 $ cancels with 5B with indices $ (k+k'+1, l, k', l'-l-1) $:
\begin{equation*}
\begin{aligned}
(q_{14}): && (k+1)(2l+1) + (2k'+1) l' & \\
&& & \hspace{-10em} = (1+l+(k+k'+1)+2(k+k'+1)l) + (k'+(l'-l-1)+2k'(l'-l-1)), \\
(q_{23}): && (k+1)(2l+1) + k'(2l'+1) &= (k+k'+1)(2l+1) + 2k'((l'-l-1)+1).
\end{aligned}
\end{equation*}
\item 4A in case $ l' ≥ l $ cancels with 6A with indices $ (k+k'+1, l, k', l'-l) $:
\begin{equation*}
\begin{aligned}
(q_{14}): && (k+1)(2l+1) + (2k'+1)(l'+1) & \\
&& & \hspace{-10em} = (1+l+(k+k'+1)+2(k+k'+1)l) + (k'+(l'-l)+2k'(l'-l)), \\
(q_{23}): && (k+1)(2l+1) + (k'+1)(2l'+1) &= ((k+k'+1)+1)(2l+1) + 2(k'+1)(l'-l).
\end{aligned}
\end{equation*}
\item 4B in case $ k ≥ k' $ cancels with 6B with indices $ (k-k', l, k', l+l') $:
\begin{equation*}
\begin{aligned}
(q_{14}): && (k+1)(2l+1) + (2k'+1)l' &= (1+l+(k-k')+2(k-k')l) + (k'+(l+l')+2k'(l+l')), \\
(q_{23}): && (k+1)(2l+1) + k'(2l'+1) &= ((k-k')+1)(2l+1) + 2k'((l+l')+1).
\end{aligned}
\end{equation*}
\end{itemize}
We have shown that most terms of \eqref{eq:examples-torus-centralityRHS} cancel out. It remains to analyze the corner cases which did not cancel and match them with \eqref{eq:examples-torus-lcontrib}. In fact, the only remaining corner terms are 6A with indices $ (k, l, k, l') $ and 5B with indices $ (k, l, k, l') $. We conclude that the sum of the 12 terms (1A–6B) is
\begin{align*}
& - \sum_{k, l, l' ≥ 0} q_{14}^{1+(l+l')+2k+2k(l+l')} q_{23}^{(k+1)(2(l+l')+1)}
- \sum_{k, l, l' ≥ 0} q_{14}^{1+(l+l')+2k+2k(l+l')} q_{23}^{k(2(l+l')+3)} \\
&=
- \sum_{k, s ≥ 0} (s+1) q_{14}^{1+s+2k+2ks} q_{23}^{1+k+2s+2ks}
- \sum_{k, s ≥ 0} s q_{14}^{s+2ks} q_{23}^{k+2ks}.
\end{align*}
We recognize this expression as equal to \eqref{eq:examples-torus-lcontrib}. In other words, the coefficient of $ a_2 b_1 a_1 b_2 a_3 $ in the centrality identity \eqref{eq:examples-torus-centrality} agrees on both sides. This finishes the checks for the coefficients of $ a_2 b_1 a_1 b_2 a_3 $, and thereby finishes our selected calculations aimed at demonstrating \eqref{eq:examples-torus-centrality}.
\end{proof}

\input{examples/fig_cancellations.tex}

%% file: examples/fig_torus.tex
\begin{figure}
\centering
\begin{subfigure}{0.3\linewidth}
\centering
\begin{tikzpicture}[scale=2]
\newcommand{\arrowbetween}[2]{($ (#1)!0.1!(#2) $) -- ($ (#2)!0.1!(#1) $)}
\path[draw, ->] \arrowbetween{0, 0}{1, 0} node[midway, above] {$ a_1 $};
\path[draw, ->] \arrowbetween{2, 0}{1, 0} node[midway, above] {$ a_2 $};
\path[draw, ->] \arrowbetween{0, 2}{1, 2} node[midway, above] {$ a_1 $};
\path[draw, ->] \arrowbetween{2, 2}{1, 2} node[midway, above] {$ a_2 $};
\path[draw, ->] \arrowbetween{1, 1}{0, 1} node[midway, above] {$ a_3 $};
\path[draw, ->] \arrowbetween{1, 1}{2, 1} node[midway, above] {$ a_4 $};
\path[draw, ->] \arrowbetween{0, 1}{0, 0} node[midway, left] {$ b_1 $};
\path[draw, ->] \arrowbetween{0, 1}{0, 2} node[midway, left] {$ b_3 $};
\path[draw, ->] \arrowbetween{2, 1}{2, 0} node[midway, left] {$ b_1 $};
\path[draw, ->] \arrowbetween{2, 1}{2, 2} node[midway, left] {$ b_3 $};
\path[draw, ->] \arrowbetween{1, 0}{1, 1} node[midway, left] {$ b_2 $};
\path[draw, ->] \arrowbetween{1, 2}{1, 1} node[midway, left] {$ b_4 $};
\path (0, 0) node {\small $ 1 $};
\path (2, 0) node {\small $ 1 $};
\path (2, 2) node {\small $ 1 $};
\path (0, 2) node {\small $ 1 $};
\path (1, 0) node {\small $ 2 $};
\path (1, 2) node {\small $ 2 $};
\path (0, 1) node {\small $ 3 $};
\path (2, 1) node {\small $ 3 $};
\path (1, 1) node {\small $ 4 $};
\end{tikzpicture}
\caption{Four-punctured torus $ Q $}
\label{fig:examples-torus-fig}
\end{subfigure}
\begin{subfigure}{0.3\linewidth}
\centering
\begin{tikzpicture}[scale=2]
\newcommand{\arrowbetween}[2]{($ (#1)!0.1!(#2) $) -- ($ (#2)!0.1!(#1) $)}
\path[draw, ->] \arrowbetween{0, 0}{1, 0} node[midway, above] {$ a_3 $};
\path[draw, ->] \arrowbetween{2, 0}{1, 0} node[midway, above] {$ a_2 $};
\path[draw, ->] \arrowbetween{0, 2}{1, 2} node[midway, above] {$ a_3 $};
\path[draw, ->] \arrowbetween{2, 2}{1, 2} node[midway, above] {$ a_2 $};
\path[draw, ->] \arrowbetween{1, 1}{0, 1} node[midway, above] {$ a_1 $};
\path[draw, ->] \arrowbetween{1, 1}{2, 1} node[midway, above] {$ a_4 $};
\path[draw, ->] \arrowbetween{0, 1}{0, 0} node[midway, left] {$ b_1 $};
\path[draw, ->] \arrowbetween{0, 1}{0, 2} node[midway, left] {$ b_4 $};
\path[draw, ->] \arrowbetween{2, 1}{2, 0} node[midway, left] {$ b_1 $};
\path[draw, ->] \arrowbetween{2, 1}{2, 2} node[midway, left] {$ b_4 $};
\path[draw, ->] \arrowbetween{1, 0}{1, 1} node[midway, left] {$ b_2 $};
\path[draw, ->] \arrowbetween{1, 2}{1, 1} node[midway, left] {$ b_3 $};
\path (0, 0) node {\small $ 1 $};
\path (2, 0) node {\small $ 1 $};
\path (2, 2) node {\small $ 1 $};
\path (0, 2) node {\small $ 1 $};
\path (1, 0) node {\small $ 2 $};
\path (1, 2) node {\small $ 2 $};
\path (0, 1) node {\small $ 3 $};
\path (2, 1) node {\small $ 3 $};
\path (1, 1) node {\small $ 4 $};
\end{tikzpicture}
\caption{Its mirror $ \mirQ $}
\label{fig:examples-torus-mirror}
\end{subfigure}
\begin{subfigure}{0.3\linewidth}
\centering
\begin{tikzpicture}[scale=0.66666667]
\path[draw] (0, 0) -- ++(right:1) coordinate[midway] (1-0) -- ++(right:1) coordinate[midway] (3-0) -- ++(right:1) coordinate[midway] (5-0) -- ++(right:1) coordinate[midway] (7-0) -- ++(right:1) coordinate[midway] (9-0) -- ++(right:1) coordinate[midway] (11-0);
\path[draw] (0, 1) -- ++(right:1) coordinate[midway] (1-2) -- ++(right:1) coordinate[midway] (3-2) -- ++(right:1) coordinate[midway] (5-2) -- ++(right:1) coordinate[midway] (7-2) -- ++(right:1) coordinate[midway] (9-2) -- ++(right:1) coordinate[midway] (11-2);
\path[draw] (0, 2) -- ++(right:1) coordinate[midway] (1-4) -- ++(right:1) coordinate[midway] (3-4) -- ++(right:1) coordinate[midway] (5-4) -- ++(right:1) coordinate[midway] (7-4) -- ++(right:1) coordinate[midway] (9-4) -- ++(right:1) coordinate[midway] (11-4);
\path[draw] (0, 3) -- ++(right:1) coordinate[midway] (1-6) -- ++(right:1) coordinate[midway] (3-6) -- ++(right:1) coordinate[midway] (5-6) -- ++(right:1) coordinate[midway] (7-6) -- ++(right:1) coordinate[midway] (9-6) -- ++(right:1) coordinate[midway] (11-6);
\path[draw] (0, 4) -- ++(right:1) coordinate[midway] (1-8) -- ++(right:1) coordinate[midway] (3-8) -- ++(right:1) coordinate[midway] (5-8) -- ++(right:1) coordinate[midway] (7-8) -- ++(right:1) coordinate[midway] (9-8) -- ++(right:1) coordinate[midway] (11-8);
\path[draw] (0, 5) -- ++(right:1) coordinate[midway] (1-10) -- ++(right:1) coordinate[midway] (3-10) -- ++(right:1) coordinate[midway] (5-10) -- ++(right:1) coordinate[midway] (7-10) -- ++(right:1) coordinate[midway] (9-10) -- ++(right:1) coordinate[midway] (11-10);
\path[draw] (0, 6) -- ++(right:1) coordinate[midway] (1-12) -- ++(right:1) coordinate[midway] (3-12) -- ++(right:1) coordinate[midway] (5-12) -- ++(right:1) coordinate[midway] (7-12) -- ++(right:1) coordinate[midway] (9-12) -- ++(right:1) coordinate[midway] (11-12);
\path[draw] (0, 0) -- ++(up:1) coordinate[midway] (0-1) -- ++(up:1) coordinate[midway] (0-3) -- ++(up:1) coordinate[midway] (0-5) -- ++(up:1) coordinate[midway] (0-7) -- ++(up:1) coordinate[midway] (0-9) -- ++(up:1) coordinate[midway] (0-11);
\path[draw] (1, 0) -- ++(up:1) coordinate[midway] (2-1) -- ++(up:1) coordinate[midway] (2-3) -- ++(up:1) coordinate[midway] (2-5) -- ++(up:1) coordinate[midway] (2-7) -- ++(up:1) coordinate[midway] (2-9) -- ++(up:1) coordinate[midway] (2-11);
\path[draw] (2, 0) -- ++(up:1) coordinate[midway] (4-1) -- ++(up:1) coordinate[midway] (4-3) -- ++(up:1) coordinate[midway] (4-5) -- ++(up:1) coordinate[midway] (4-7) -- ++(up:1) coordinate[midway] (4-9) -- ++(up:1) coordinate[midway] (4-11);
\path[draw] (3, 0) -- ++(up:1) coordinate[midway] (6-1) -- ++(up:1) coordinate[midway] (6-3) -- ++(up:1) coordinate[midway] (6-5) -- ++(up:1) coordinate[midway] (6-7) -- ++(up:1) coordinate[midway] (6-9) -- ++(up:1) coordinate[midway] (6-11);
\path[draw] (4, 0) -- ++(up:1) coordinate[midway] (8-1) -- ++(up:1) coordinate[midway] (8-3) -- ++(up:1) coordinate[midway] (8-5) -- ++(up:1) coordinate[midway] (8-7) -- ++(up:1) coordinate[midway] (8-9) -- ++(up:1) coordinate[midway] (8-11);
\path[draw] (5, 0) -- ++(up:1) coordinate[midway] (10-1) -- ++(up:1) coordinate[midway] (10-3) -- ++(up:1) coordinate[midway] (10-5) -- ++(up:1) coordinate[midway] (10-7) -- ++(up:1) coordinate[midway] (10-9) -- ++(up:1) coordinate[midway] (10-11);
\path[draw] (6, 0) -- ++(up:1) coordinate[midway] (12-1) -- ++(up:1) coordinate[midway] (12-3) -- ++(up:1) coordinate[midway] (12-5) -- ++(up:1) coordinate[midway] (12-7) -- ++(up:1) coordinate[midway] (12-9) -- ++(up:1) coordinate[midway] (12-11);
%
\path[fill, color=gray, fill opacity=0.5] plot[smooth] coordinates{(2-7) ($ (2-7)!0.5!(3-6) + (45:0.1) $) (3-6) ($ (3-6)!0.5!(4-5) + (225:0.1) $) (4-5) ($ (4-5)!0.5!(5-4) + (45:0.1) $) (5-4) ($ (5-4)!0.5!(6-3) + (225:0.1) $) (6-3) ($ (6-3)!0.5!(7-2) + (45:0.1) $) (7-2) ($ (7-2)!0.5!(8-1) + (225:0.1) $) (8-1) ($ (8-1)!0.5!(9-0) + (45:0.1) $) (9-0)} -- plot[smooth] coordinates{(9-0) ($ (9-0)!0.5!(10-1) + (135:0.1) $) (10-1) ($ (10-1)!0.5!(11-2) + (315:0.1) $) (11-2) ($ (11-2)!0.5!(12-3) + (135:0.1) $) (12-3)} -- plot[smooth] coordinates{(12-3) ($ (11-4)!0.5!(12-3) + (225:0.1) $) (11-4) ($ (10-5)!0.5!(11-4) + (45:0.1) $) (10-5) ($ (9-6)!0.5!(10-5) + (225:0.1) $) (9-6) ($ (8-7)!0.5!(9-6) + (45:0.1) $) (8-7) ($ (7-8)!0.5!(8-7) + (225:0.1) $) (7-8) ($ (6-9)!0.5!(7-8) + (45:0.1) $) (6-9) ($ (5-10)!0.5!(6-9) + (225:0.1) $) (5-10)} -- plot[smooth] coordinates{(5-10) ($ (4-9)!0.5!(5-10) + (315:0.1) $) (4-9) ($ (3-8)!0.5!(4-9) + (135:0.1) $) (3-8) ($ (2-7)!0.5!(3-8) + (315:0.1) $) (2-7)};
\path[fill, color=gray, fill opacity=0.5] plot[smooth] coordinates{(1-0) ($ (1-0)!0.5!(2-1) + (135:0.1) $) (2-1) ($ (2-1)!0.5!(3-2) + (315:0.1) $) (3-2) ($ (3-2)!0.5!(4-3) + (135:0.1) $) (4-3)} -- plot[smooth] coordinates{(4-3) ($ (3-4)!0.5!(4-3) + (225:0.1) $) (3-4)} -- plot[smooth] coordinates{(3-4) ($ (2-3)!0.5!(3-4) + (315:0.1) $) (2-3) ($ (1-2)!0.5!(2-3) + (135:0.1) $) (1-2) ($ (0-1)!0.5!(1-2) + (315:0.1) $) (0-1)} -- plot[smooth] coordinates{(0-1) ($ (0-1)!0.5!(1-0) + (45:0.1) $) (1-0)};
\path[fill, color=gray, fill opacity=0.5] plot[smooth] coordinates{(10-9) ($ (10-9)!0.5!(11-8) + (45:0.1) $) (11-8)} -- plot[smooth] coordinates{(11-8) ($ (11-8)!0.5!(12-9) + (135:0.1) $) (12-9)} -- plot[smooth] coordinates{(12-9) ($ (11-10)!0.5!(12-9) + (225:0.1) $) (11-10)} -- plot[smooth] coordinates{(11-10) ($ (10-9)!0.5!(11-10) + (315:0.1) $) (10-9)};
%
\path[draw, ultra thick] plot[smooth] coordinates{($ (0-1)!-0.5!(1-2) $) (0-1) ($ (0-1)!0.5!(1-2) + (315:0.1) $) (1-2) ($ (1-2)!0.5!(2-3) + (135:0.1) $) (2-3) ($ (2-3)!0.5!(3-4) + (315:0.1) $) (3-4) ($ (3-4)!0.5!(4-5) + (135:0.1) $) (4-5) ($ (4-5)!0.5!(5-6) + (315:0.1) $) (5-6) ($ (5-6)!0.5!(6-7) + (135:0.1) $) (6-7) ($ (6-7)!0.5!(7-8) + (315:0.1) $) (7-8) ($ (7-8)!0.5!(8-9) + (135:0.1) $) (8-9) ($ (8-9)!0.5!(9-10) + (315:0.1) $) (9-10) ($ (9-10)!0.5!(10-11) + (135:0.1) $) (10-11) ($ (10-11)!0.5!(11-12) + (315:0.1) $) (11-12) ($ (11-12)!-0.5!(10-11) $)};
\path[draw, ultra thick] plot[smooth] coordinates{($ (0-3)!-0.5!(1-4) $) (0-3) ($ (0-3)!0.5!(1-4) + (315:0.1) $) (1-4) ($ (1-4)!0.5!(2-5) + (135:0.1) $) (2-5) ($ (2-5)!0.5!(3-6) + (315:0.1) $) (3-6) ($ (3-6)!0.5!(4-7) + (135:0.1) $) (4-7) ($ (4-7)!0.5!(5-8) + (315:0.1) $) (5-8) ($ (5-8)!0.5!(6-9) + (135:0.1) $) (6-9) ($ (6-9)!0.5!(7-10) + (315:0.1) $) (7-10) ($ (7-10)!0.5!(8-11) + (135:0.1) $) (8-11) ($ (8-11)!0.5!(9-12) + (315:0.1) $) (9-12) ($ (9-12)!-0.5!(8-11) $)};
\path[draw, ultra thick] plot[smooth] coordinates{($ (0-5)!-0.5!(1-6) $) (0-5) ($ (0-5)!0.5!(1-6) + (315:0.1) $) (1-6) ($ (1-6)!0.5!(2-7) + (135:0.1) $) (2-7) ($ (2-7)!0.5!(3-8) + (315:0.1) $) (3-8) ($ (3-8)!0.5!(4-9) + (135:0.1) $) (4-9) ($ (4-9)!0.5!(5-10) + (315:0.1) $) (5-10) ($ (5-10)!0.5!(6-11) + (135:0.1) $) (6-11) ($ (6-11)!0.5!(7-12) + (315:0.1) $) (7-12) ($ (7-12)!-0.5!(6-11) $)};
\path[draw, ultra thick] plot[smooth] coordinates{($ (0-7)!-0.5!(1-8) $) (0-7) ($ (0-7)!0.5!(1-8) + (315:0.1) $) (1-8) ($ (1-8)!0.5!(2-9) + (135:0.1) $) (2-9) ($ (2-9)!0.5!(3-10) + (315:0.1) $) (3-10) ($ (3-10)!0.5!(4-11) + (135:0.1) $) (4-11) ($ (4-11)!0.5!(5-12) + (315:0.1) $) (5-12) ($ (5-12)!-0.5!(4-11) $)};
\path[draw, ultra thick] plot[smooth] coordinates{($ (0-9)!-0.5!(1-10) $) (0-9) ($ (0-9)!0.5!(1-10) + (315:0.1) $) (1-10) ($ (1-10)!0.5!(2-11) + (135:0.1) $) (2-11) ($ (2-11)!0.5!(3-12) + (315:0.1) $) (3-12) ($ (3-12)!-0.5!(2-11) $)};
\path[draw, ultra thick] plot[smooth] coordinates{($ (0-11)!-0.5!(1-12) $) (0-11) ($ (0-11)!0.5!(1-12) + (315:0.1) $) (1-12) ($ (1-12)!-0.5!(0-11) $)};
\path[draw, ultra thick] plot[smooth] coordinates{($ (1-0)!-0.5!(2-1) $) (1-0) ($ (1-0)!0.5!(2-1) + (135:0.1) $) (2-1) ($ (2-1)!0.5!(3-2) + (315:0.1) $) (3-2) ($ (3-2)!0.5!(4-3) + (135:0.1) $) (4-3) ($ (4-3)!0.5!(5-4) + (315:0.1) $) (5-4) ($ (5-4)!0.5!(6-5) + (135:0.1) $) (6-5) ($ (6-5)!0.5!(7-6) + (315:0.1) $) (7-6) ($ (7-6)!0.5!(8-7) + (135:0.1) $) (8-7) ($ (8-7)!0.5!(9-8) + (315:0.1) $) (9-8) ($ (9-8)!0.5!(10-9) + (135:0.1) $) (10-9) ($ (10-9)!0.5!(11-10) + (315:0.1) $) (11-10) ($ (11-10)!0.5!(12-11) + (135:0.1) $) (12-11) ($ (12-11)!-0.5!(11-10) $)};
\path[draw, ultra thick] plot[smooth] coordinates{($ (3-0)!-0.5!(4-1) $) (3-0) ($ (3-0)!0.5!(4-1) + (135:0.1) $) (4-1) ($ (4-1)!0.5!(5-2) + (315:0.1) $) (5-2) ($ (5-2)!0.5!(6-3) + (135:0.1) $) (6-3) ($ (6-3)!0.5!(7-4) + (315:0.1) $) (7-4) ($ (7-4)!0.5!(8-5) + (135:0.1) $) (8-5) ($ (8-5)!0.5!(9-6) + (315:0.1) $) (9-6) ($ (9-6)!0.5!(10-7) + (135:0.1) $) (10-7) ($ (10-7)!0.5!(11-8) + (315:0.1) $) (11-8) ($ (11-8)!0.5!(12-9) + (135:0.1) $) (12-9) ($ (12-9)!-0.5!(11-8) $)};
\path[draw, ultra thick] plot[smooth] coordinates{($ (5-0)!-0.5!(6-1) $) (5-0) ($ (5-0)!0.5!(6-1) + (135:0.1) $) (6-1) ($ (6-1)!0.5!(7-2) + (315:0.1) $) (7-2) ($ (7-2)!0.5!(8-3) + (135:0.1) $) (8-3) ($ (8-3)!0.5!(9-4) + (315:0.1) $) (9-4) ($ (9-4)!0.5!(10-5) + (135:0.1) $) (10-5) ($ (10-5)!0.5!(11-6) + (315:0.1) $) (11-6) ($ (11-6)!0.5!(12-7) + (135:0.1) $) (12-7) ($ (12-7)!-0.5!(11-6) $)};
\path[draw, ultra thick] plot[smooth] coordinates{($ (7-0)!-0.5!(8-1) $) (7-0) ($ (7-0)!0.5!(8-1) + (135:0.1) $) (8-1) ($ (8-1)!0.5!(9-2) + (315:0.1) $) (9-2) ($ (9-2)!0.5!(10-3) + (135:0.1) $) (10-3) ($ (10-3)!0.5!(11-4) + (315:0.1) $) (11-4) ($ (11-4)!0.5!(12-5) + (135:0.1) $) (12-5) ($ (12-5)!-0.5!(11-4) $)};
\path[draw, ultra thick] plot[smooth] coordinates{($ (9-0)!-0.5!(10-1) $) (9-0) ($ (9-0)!0.5!(10-1) + (135:0.1) $) (10-1) ($ (10-1)!0.5!(11-2) + (315:0.1) $) (11-2) ($ (11-2)!0.5!(12-3) + (135:0.1) $) (12-3) ($ (12-3)!-0.5!(11-2) $)};
\path[draw, ultra thick] plot[smooth] coordinates{($ (11-0)!-0.5!(12-1) $) (11-0) ($ (11-0)!0.5!(12-1) + (135:0.1) $) (12-1) ($ (12-1)!-0.5!(11-0) $)};
\path[draw, ultra thick] plot[smooth] coordinates{($ (0-1)!-0.5!(1-0) $) (0-1) ($ (0-1)!0.5!(1-0) + (45:0.1) $) (1-0) ($ (1-0)!-0.5!(0-1) $)};
\path[draw, ultra thick] plot[smooth] coordinates{($ (0-3)!-0.5!(1-2) $) (0-3) ($ (0-3)!0.5!(1-2) + (45:0.1) $) (1-2) ($ (1-2)!0.5!(2-1) + (225:0.1) $) (2-1) ($ (2-1)!0.5!(3-0) + (45:0.1) $) (3-0) ($ (3-0)!-0.5!(2-1) $)};
\path[draw, ultra thick] plot[smooth] coordinates{($ (0-5)!-0.5!(1-4) $) (0-5) ($ (0-5)!0.5!(1-4) + (45:0.1) $) (1-4) ($ (1-4)!0.5!(2-3) + (225:0.1) $) (2-3) ($ (2-3)!0.5!(3-2) + (45:0.1) $) (3-2) ($ (3-2)!0.5!(4-1) + (225:0.1) $) (4-1) ($ (4-1)!0.5!(5-0) + (45:0.1) $) (5-0) ($ (5-0)!-0.5!(4-1) $)};
\path[draw, ultra thick] plot[smooth] coordinates{($ (0-7)!-0.5!(1-6) $) (0-7) ($ (0-7)!0.5!(1-6) + (45:0.1) $) (1-6) ($ (1-6)!0.5!(2-5) + (225:0.1) $) (2-5) ($ (2-5)!0.5!(3-4) + (45:0.1) $) (3-4) ($ (3-4)!0.5!(4-3) + (225:0.1) $) (4-3) ($ (4-3)!0.5!(5-2) + (45:0.1) $) (5-2) ($ (5-2)!0.5!(6-1) + (225:0.1) $) (6-1) ($ (6-1)!0.5!(7-0) + (45:0.1) $) (7-0) ($ (7-0)!-0.5!(6-1) $)};
\path[draw, ultra thick] plot[smooth] coordinates{($ (0-9)!-0.5!(1-8) $) (0-9) ($ (0-9)!0.5!(1-8) + (45:0.1) $) (1-8) ($ (1-8)!0.5!(2-7) + (225:0.1) $) (2-7) ($ (2-7)!0.5!(3-6) + (45:0.1) $) (3-6) ($ (3-6)!0.5!(4-5) + (225:0.1) $) (4-5) ($ (4-5)!0.5!(5-4) + (45:0.1) $) (5-4) ($ (5-4)!0.5!(6-3) + (225:0.1) $) (6-3) ($ (6-3)!0.5!(7-2) + (45:0.1) $) (7-2) ($ (7-2)!0.5!(8-1) + (225:0.1) $) (8-1) ($ (8-1)!0.5!(9-0) + (45:0.1) $) (9-0) ($ (9-0)!-0.5!(8-1) $)};
\path[draw, ultra thick] plot[smooth] coordinates{($ (0-11)!-0.5!(1-10) $) (0-11) ($ (0-11)!0.5!(1-10) + (45:0.1) $) (1-10) ($ (1-10)!0.5!(2-9) + (225:0.1) $) (2-9) ($ (2-9)!0.5!(3-8) + (45:0.1) $) (3-8) ($ (3-8)!0.5!(4-7) + (225:0.1) $) (4-7) ($ (4-7)!0.5!(5-6) + (45:0.1) $) (5-6) ($ (5-6)!0.5!(6-5) + (225:0.1) $) (6-5) ($ (6-5)!0.5!(7-4) + (45:0.1) $) (7-4) ($ (7-4)!0.5!(8-3) + (225:0.1) $) (8-3) ($ (8-3)!0.5!(9-2) + (45:0.1) $) (9-2) ($ (9-2)!0.5!(10-1) + (225:0.1) $) (10-1) ($ (10-1)!0.5!(11-0) + (45:0.1) $) (11-0) ($ (11-0)!-0.5!(10-1) $)};
\path[draw, ultra thick] plot[smooth] coordinates{($ (1-12)!-0.5!(2-11) $) (1-12) ($ (1-12)!0.5!(2-11) + (225:0.1) $) (2-11) ($ (2-11)!0.5!(3-10) + (45:0.1) $) (3-10) ($ (3-10)!0.5!(4-9) + (225:0.1) $) (4-9) ($ (4-9)!0.5!(5-8) + (45:0.1) $) (5-8) ($ (5-8)!0.5!(6-7) + (225:0.1) $) (6-7) ($ (6-7)!0.5!(7-6) + (45:0.1) $) (7-6) ($ (7-6)!0.5!(8-5) + (225:0.1) $) (8-5) ($ (8-5)!0.5!(9-4) + (45:0.1) $) (9-4) ($ (9-4)!0.5!(10-3) + (225:0.1) $) (10-3) ($ (10-3)!0.5!(11-2) + (45:0.1) $) (11-2) ($ (11-2)!0.5!(12-1) + (225:0.1) $) (12-1) ($ (12-1)!-0.5!(11-2) $)};
\path[draw, ultra thick] plot[smooth] coordinates{($ (3-12)!-0.5!(4-11) $) (3-12) ($ (3-12)!0.5!(4-11) + (225:0.1) $) (4-11) ($ (4-11)!0.5!(5-10) + (45:0.1) $) (5-10) ($ (5-10)!0.5!(6-9) + (225:0.1) $) (6-9) ($ (6-9)!0.5!(7-8) + (45:0.1) $) (7-8) ($ (7-8)!0.5!(8-7) + (225:0.1) $) (8-7) ($ (8-7)!0.5!(9-6) + (45:0.1) $) (9-6) ($ (9-6)!0.5!(10-5) + (225:0.1) $) (10-5) ($ (10-5)!0.5!(11-4) + (45:0.1) $) (11-4) ($ (11-4)!0.5!(12-3) + (225:0.1) $) (12-3) ($ (12-3)!-0.5!(11-4) $)};
\path[draw, ultra thick] plot[smooth] coordinates{($ (5-12)!-0.5!(6-11) $) (5-12) ($ (5-12)!0.5!(6-11) + (225:0.1) $) (6-11) ($ (6-11)!0.5!(7-10) + (45:0.1) $) (7-10) ($ (7-10)!0.5!(8-9) + (225:0.1) $) (8-9) ($ (8-9)!0.5!(9-8) + (45:0.1) $) (9-8) ($ (9-8)!0.5!(10-7) + (225:0.1) $) (10-7) ($ (10-7)!0.5!(11-6) + (45:0.1) $) (11-6) ($ (11-6)!0.5!(12-5) + (225:0.1) $) (12-5) ($ (12-5)!-0.5!(11-6) $)};
\path[draw, ultra thick] plot[smooth] coordinates{($ (7-12)!-0.5!(8-11) $) (7-12) ($ (7-12)!0.5!(8-11) + (225:0.1) $) (8-11) ($ (8-11)!0.5!(9-10) + (45:0.1) $) (9-10) ($ (9-10)!0.5!(10-9) + (225:0.1) $) (10-9) ($ (10-9)!0.5!(11-8) + (45:0.1) $) (11-8) ($ (11-8)!0.5!(12-7) + (225:0.1) $) (12-7) ($ (12-7)!-0.5!(11-8) $)};
\path[draw, ultra thick] plot[smooth] coordinates{($ (9-12)!-0.5!(10-11) $) (9-12) ($ (9-12)!0.5!(10-11) + (225:0.1) $) (10-11) ($ (10-11)!0.5!(11-10) + (45:0.1) $) (11-10) ($ (11-10)!0.5!(12-9) + (225:0.1) $) (12-9) ($ (12-9)!-0.5!(11-10) $)};
\path[draw, ultra thick] plot[smooth] coordinates{($ (11-12)!-0.5!(12-11) $) (11-12) ($ (11-12)!0.5!(12-11) + (225:0.1) $) (12-11) ($ (12-11)!-0.5!(11-12) $)};
\end{tikzpicture}
\caption{A few midpoint polygons}
\label{fig:examples-torus-zigzag}
\end{subfigure}
\caption{Illustration of 4-punctured torus}
\end{figure}

%% file: examples/fig_torus_Xa2.tex
\begin{figure}
\centering
\begin{subfigure}{0.8\linewidth}
\centering
\begin{tikzpicture}[scale=2]
\newcommand{\arrowbetween}[2]{($ (#1)!0.1!(#2) $) -- ($ (#2)!0.1!(#1) $)}
\newcommand{\thetorus}{%
\path[draw, ->] \arrowbetween{0, 0}{1, 0} node[midway, above] {$ a_1 $};
\path[draw, ->] \arrowbetween{2, 0}{1, 0} node[midway, above] {$ a_2 $};
\path[draw, ->] \arrowbetween{0, 2}{1, 2} node[midway, above] {$ a_1 $};
\path[draw, ->] \arrowbetween{2, 2}{1, 2} node[midway, above] {$ a_2 $};
\path[draw, ->] \arrowbetween{1, 1}{0, 1} node[midway, above] {$ a_3 $};
\path[draw, ->] \arrowbetween{1, 1}{2, 1} node[midway, above] {$ a_4 $};
\path[draw, ->] \arrowbetween{0, 1}{0, 0} node[midway, left] {$ b_1 $};
\path[draw, ->] \arrowbetween{0, 1}{0, 2} node[midway, left] {$ b_3 $};
\path[draw, ->] \arrowbetween{2, 1}{2, 0} node[midway, left] {$ b_1 $};
\path[draw, ->] \arrowbetween{2, 1}{2, 2} node[midway, left] {$ b_3 $};
\path[draw, ->] \arrowbetween{1, 0}{1, 1} node[midway, left] {$ b_2 $};
\path[draw, ->] \arrowbetween{1, 2}{1, 1} node[midway, left] {$ b_4 $};
\path (0, 0) node {\small $ 1 $};
\path (2, 0) node {\small $ 1 $};
\path (2, 2) node {\small $ 1 $};
\path (0, 2) node {\small $ 1 $};
\path (1, 0) node {\small $ 2 $};
\path (1, 2) node {\small $ 2 $};
\path (0, 1) node {\small $ 3 $};
\path (2, 1) node {\small $ 3 $};
\path (1, 1) node {\small $ 4 $};}
\begin{scope}[shift={(2, 0)}, every path/.style={gray}] \thetorus \end{scope}
\begin{scope}[shift={(2, 2)}, every path/.style={gray}] \thetorus \end{scope}
\begin{scope}[shift={(4, 0)}, every path/.style={gray}] \thetorus \end{scope}
\begin{scope}[shift={(4, 2)}, every path/.style={gray}] \thetorus \end{scope}
\path[draw, thick] plot[smooth] coordinates{(2, 0.5) ($ (2.25, 0.75) + (315:0.1) $) (2.5, 1) ($ (2.75, 1.25) + (135:0.1) $) (3, 1.5) ($ (3.25, 1.75) + (315:0.1) $) (3.5, 2) ($ (3.75, 2.25) + (135:0.1) $) (4, 2.5) ($ (4.25, 2.75) + (315:0.1) $) (4.5, 3) ($ (4.75, 3.25) + (135:0.1) $) (5, 3.5) ($ (5.25, 3.75) + (315:0.1) $) (5.5, 4)};
\path[draw, thick] plot[smooth] coordinates{(5.5, 0) ($ (5.25, 0.25) + (45:0.1) $) (5, 0.5) ($ (4.75, 0.75) + (225:0.1) $) (4.5, 1) ($ (4.25, 1.25) + (45:0.1) $) (4, 1.5) ($ (3.75, 1.75) + (225:0.1) $) (3.5, 2) ($ (3.25, 2.25) + (45:0.1) $) (3, 2.5) ($ (2.75, 2.75) + (225:0.1) $) (2.5, 3) ($ (2.25, 3.25) + (45:0.1) $) (2, 3.5)};
\path[fill] (3.5, 2) circle[radius=0.03];
\path[draw, bend right=45, ->] ($ (3.75, 2.25) + (135:0.1) $) to node[above] {$ X_{a_2} $} node[at start, below right] {$ \smooth L_1 $} node[at end, below left] {$ \smooth L_2 $} ($ (3.25, 2.25) + (45:0.1) $);
\path ($ (4.75, 0.75) + (225:0.1) $) node[left] {$ \smooth L_2 $};
\path ($ (2.25, 0.75) + (315:0.1) $) node[right] {$ \smooth L_1 $};
\end{tikzpicture}
\caption{The zigzag paths $ L_1 $ and $ L_2 $ and the morphism $ X_{a_2}: L_1 → L_2 $}
\label{fig:examples-torus-Xa2}
\end{subfigure}
\end{figure}

%% file: examples/fig_cancellations.tex
\begin{figure}
\newcommand{\eightbasicgrid}{%
\begin{scope}[every path/.style={gray, draw opacity=0.5}]
\path[draw] (0, 0) -- ++(right:1) coordinate[midway] (1-0) -- ++(right:1) coordinate[midway] (3-0) -- ++(right:1) coordinate[midway] (5-0) -- ++(right:1) coordinate[midway] (7-0) -- ++(right:1) coordinate[midway] (9-0) -- ++(right:1) coordinate[midway] (11-0) -- ++(right:1) coordinate[midway] (13-0) -- ++(right:1) coordinate[midway] (15-0);
\path[draw] (0, 1) -- ++(right:1) coordinate[midway] (1-2) -- ++(right:1) coordinate[midway] (3-2) -- ++(right:1) coordinate[midway] (5-2) -- ++(right:1) coordinate[midway] (7-2) -- ++(right:1) coordinate[midway] (9-2) -- ++(right:1) coordinate[midway] (11-2) -- ++(right:1) coordinate[midway] (13-2) -- ++(right:1) coordinate[midway] (15-2);
\path[draw] (0, 2) -- ++(right:1) coordinate[midway] (1-4) -- ++(right:1) coordinate[midway] (3-4) -- ++(right:1) coordinate[midway] (5-4) -- ++(right:1) coordinate[midway] (7-4) -- ++(right:1) coordinate[midway] (9-4) -- ++(right:1) coordinate[midway] (11-4) -- ++(right:1) coordinate[midway] (13-4) -- ++(right:1) coordinate[midway] (15-4);
\path[draw] (0, 3) -- ++(right:1) coordinate[midway] (1-6) -- ++(right:1) coordinate[midway] (3-6) -- ++(right:1) coordinate[midway] (5-6) -- ++(right:1) coordinate[midway] (7-6) -- ++(right:1) coordinate[midway] (9-6) -- ++(right:1) coordinate[midway] (11-6) -- ++(right:1) coordinate[midway] (13-6) -- ++(right:1) coordinate[midway] (15-6);
\path[draw] (0, 4) -- ++(right:1) coordinate[midway] (1-8) -- ++(right:1) coordinate[midway] (3-8) -- ++(right:1) coordinate[midway] (5-8) -- ++(right:1) coordinate[midway] (7-8) -- ++(right:1) coordinate[midway] (9-8) -- ++(right:1) coordinate[midway] (11-8) -- ++(right:1) coordinate[midway] (13-8) -- ++(right:1) coordinate[midway] (15-8);
\path[draw] (0, 5) -- ++(right:1) coordinate[midway] (1-10) -- ++(right:1) coordinate[midway] (3-10) -- ++(right:1) coordinate[midway] (5-10) -- ++(right:1) coordinate[midway] (7-10) -- ++(right:1) coordinate[midway] (9-10) -- ++(right:1) coordinate[midway] (11-10) -- ++(right:1) coordinate[midway] (13-10) -- ++(right:1) coordinate[midway] (15-10);
\path[draw] (0, 6) -- ++(right:1) coordinate[midway] (1-12) -- ++(right:1) coordinate[midway] (3-12) -- ++(right:1) coordinate[midway] (5-12) -- ++(right:1) coordinate[midway] (7-12) -- ++(right:1) coordinate[midway] (9-12) -- ++(right:1) coordinate[midway] (11-12) -- ++(right:1) coordinate[midway] (13-12) -- ++(right:1) coordinate[midway] (15-12);
\path[draw] (0, 0) -- ++(up:1) coordinate[midway] (0-1) -- ++(up:1) coordinate[midway] (0-3) -- ++(up:1) coordinate[midway] (0-5) -- ++(up:1) coordinate[midway] (0-7) -- ++(up:1) coordinate[midway] (0-9) -- ++(up:1) coordinate[midway] (0-11);
\path[draw] (1, 0) -- ++(up:1) coordinate[midway] (2-1) -- ++(up:1) coordinate[midway] (2-3) -- ++(up:1) coordinate[midway] (2-5) -- ++(up:1) coordinate[midway] (2-7) -- ++(up:1) coordinate[midway] (2-9) -- ++(up:1) coordinate[midway] (2-11);
\path[draw] (2, 0) -- ++(up:1) coordinate[midway] (4-1) -- ++(up:1) coordinate[midway] (4-3) -- ++(up:1) coordinate[midway] (4-5) -- ++(up:1) coordinate[midway] (4-7) -- ++(up:1) coordinate[midway] (4-9) -- ++(up:1) coordinate[midway] (4-11);
\path[draw] (3, 0) -- ++(up:1) coordinate[midway] (6-1) -- ++(up:1) coordinate[midway] (6-3) -- ++(up:1) coordinate[midway] (6-5) -- ++(up:1) coordinate[midway] (6-7) -- ++(up:1) coordinate[midway] (6-9) -- ++(up:1) coordinate[midway] (6-11);
\path[draw] (4, 0) -- ++(up:1) coordinate[midway] (8-1) -- ++(up:1) coordinate[midway] (8-3) -- ++(up:1) coordinate[midway] (8-5) -- ++(up:1) coordinate[midway] (8-7) -- ++(up:1) coordinate[midway] (8-9) -- ++(up:1) coordinate[midway] (8-11);
\path[draw] (5, 0) -- ++(up:1) coordinate[midway] (10-1) -- ++(up:1) coordinate[midway] (10-3) -- ++(up:1) coordinate[midway] (10-5) -- ++(up:1) coordinate[midway] (10-7) -- ++(up:1) coordinate[midway] (10-9) -- ++(up:1) coordinate[midway] (10-11);
\path[draw] (6, 0) -- ++(up:1) coordinate[midway] (12-1) -- ++(up:1) coordinate[midway] (12-3) -- ++(up:1) coordinate[midway] (12-5) -- ++(up:1) coordinate[midway] (12-7) -- ++(up:1) coordinate[midway] (12-9) -- ++(up:1) coordinate[midway] (12-11);
\path[draw] (7, 0) -- ++(up:1) coordinate[midway] (14-1) -- ++(up:1) coordinate[midway] (14-3) -- ++(up:1) coordinate[midway] (14-5) -- ++(up:1) coordinate[midway] (14-7) -- ++(up:1) coordinate[midway] (14-9) -- ++(up:1) coordinate[midway] (14-11);
\path[draw] (8, 0) -- ++(up:1) coordinate[midway] (16-1) -- ++(up:1) coordinate[midway] (16-3) -- ++(up:1) coordinate[midway] (16-5) -- ++(up:1) coordinate[midway] (16-7) -- ++(up:1) coordinate[midway] (16-9) -- ++(up:1) coordinate[midway] (16-11);
\end{scope}
}
\centering
\begin{subfigure}{0.45\linewidth}
\centering
\begin{tikzpicture}[scale=0.8]
\eightbasicgrid
\path[fill=gray!50, draw, thick] plot[smooth] coordinates{(9-0) ($ (9-0)!0.5!(10-1) + (135:0.1) $) (10-1) ($ (10-1)!0.5!(11-2) + (315:0.1) $) (11-2) ($ (11-2)!0.5!(12-3) + (135:0.1) $) (12-3) ($ (12-3)!0.5!(13-4) + (315:0.1) $) (13-4) ($ (13-4)!0.5!(14-5) + (135:0.1) $) (14-5)} -- plot[smooth] coordinates{(14-5) ($ (13-6)!0.5!(14-5) + (225:0.1) $) (13-6) ($ (12-7)!0.5!(13-6) + (45:0.1) $) (12-7) ($ (11-8)!0.5!(12-7) + (225:0.1) $) (11-8) ($ (10-9)!0.5!(11-8) + (45:0.1) $) (10-9) ($ (9-10)!0.5!(10-9) + (225:0.1) $) (9-10) ($ (8-11)!0.5!(9-10) + (45:0.1) $) (8-11) ($ (7-12)!0.5!(8-11) + (225:0.1) $) (7-12)} -- plot[smooth] coordinates{(7-12) ($ (6-11)!0.5!(7-12) + (315:0.1) $) (6-11) ($ (5-10)!0.5!(6-11) + (135:0.1) $) (5-10) ($ (4-9)!0.5!(5-10) + (315:0.1) $) (4-9)} -- plot[smooth] coordinates{(4-9) ($ (4-9)!0.5!(5-8) + (45:0.1) $) (5-8) ($ (5-8)!0.5!(6-7) + (225:0.1) $) (6-7) ($ (6-7)!0.5!(7-6) + (45:0.1) $) (7-6)} -- plot[smooth] coordinates{(7-6) ($ (6-5)!0.5!(7-6) + (315:0.1) $) (6-5) ($ (5-4)!0.5!(6-5) + (135:0.1) $) (5-4)} -- plot[smooth] coordinates{(5-4) ($ (5-4)!0.5!(6-3) + (225:0.1) $) (6-3) ($ (6-3)!0.5!(7-2) + (45:0.1) $) (7-2) ($ (7-2)!0.5!(8-1) + (225:0.1) $) (8-1) ($ (8-1)!0.5!(9-0) + (45:0.1) $) (9-0)};
\eightbasicgrid
\path[draw, thick] plot[smooth] coordinates{(7-6) ($ (7-6)!0.5!(8-7) + (135:0.1) $) (8-7) ($ (8-7)!0.5!(9-8) + (315:0.1) $) (9-8) ($ (9-8)!0.5!(10-9) + (135:0.1) $) (10-9)};
\path (5-4) node[below left] {$ a_2 $ (out)};
\path (9-0) node[below] {$ a_2 $};
\path (14-5) node[right] {$ b_1 $};
\path (7-12) node[above] {$ a_1 $};
\path (4-9) node[left] {$ b_2 $};
\path (7-6) node[left] {$ a_3 $};
\path (10-9) node[above right] {$ b_1 $};
\path[draw, decorate, decoration={brace, amplitude=10pt, raise=5pt, mirror, pre=moveto, post=moveto, pre length=7pt, post length=7pt}] (5-4) to node[midway, shift={(225:0.7)}] {$ k $} (9-0);
\path[draw, decorate, decoration={brace, amplitude=10pt, raise=5pt, mirror, pre=moveto, post=moveto, pre length=7pt, post length=7pt}] (4-9) to node[midway, shift={(225:0.9)}] {$ k' $} (7-6);
\path[draw, decorate, decoration={brace, amplitude=10pt, raise=9pt, mirror, pre=moveto, post=moveto, pre length=7pt, post length=7pt}] (9-0) to node[midway, shift={(315:0.9)}] {$ l $} (14-5);
\path[draw, decorate, decoration={brace, amplitude=10pt, raise=5pt, mirror, pre=moveto, post=moveto, pre length=7pt, post length=7pt}] (7-12) to node[midway, shift={(135:0.7)}] {$ l' $} (4-9);
\path (6-9) -- (8-9) node[midway] {\Large $ ∂_{b_1} W_q $};
\path (9-4) node {\Large coeff};
\path[use as bounding box] (0, -0.7) -- (8, 6.5);
\end{tikzpicture}
\caption{The origin of 1A}
\label{fig:examples-torus-cancellation1A}
\end{subfigure}
\begin{subfigure}{0.45\linewidth}
\centering
\begin{tikzpicture}[scale=0.8]
\eightbasicgrid
\path[fill=gray!50, draw, thick] plot[smooth] coordinates{(9-0) ($ (9-0)!0.5!(10-1) + (135:0.1) $) (10-1) ($ (10-1)!0.5!(11-2) + (315:0.1) $) (11-2) ($ (11-2)!0.5!(12-3) + (135:0.1) $) (12-3) ($ (12-3)!0.5!(13-4) + (315:0.1) $) (13-4) ($ (13-4)!0.5!(14-5) + (135:0.1) $) (14-5)} -- plot[smooth] coordinates{(14-5) ($ (13-6)!0.5!(14-5) + (225:0.1) $) (13-6) ($ (12-7)!0.5!(13-6) + (45:0.1) $) (12-7) ($ (11-8)!0.5!(12-7) + (225:0.1) $) (11-8) ($ (10-9)!0.5!(11-8) + (45:0.1) $) (10-9) ($ (9-10)!0.5!(10-9) + (225:0.1) $) (9-10) ($ (8-11)!0.5!(9-10) + (45:0.1) $) (8-11) ($ (7-12)!0.5!(8-11) + (225:0.1) $) (7-12)} -- plot[smooth] coordinates{(7-12) ($ (6-11)!0.5!(7-12) + (315:0.1) $) (6-11) ($ (5-10)!0.5!(6-11) + (135:0.1) $) (5-10) ($ (4-9)!0.5!(5-10) + (315:0.1) $) (4-9)} -- plot[smooth] coordinates{(4-9) ($ (4-9)!0.5!(5-8) + (45:0.1) $) (5-8) ($ (5-8)!0.5!(6-7) + (225:0.1) $) (6-7) ($ (6-7)!0.5!(7-6) + (45:0.1) $) (7-6)} -- plot[smooth] coordinates{(7-6) ($ (6-5)!0.5!(7-6) + (315:0.1) $) (6-5) ($ (5-4)!0.5!(6-5) + (135:0.1) $) (5-4)} -- plot[smooth] coordinates{(5-4) ($ (5-4)!0.5!(6-3) + (225:0.1) $) (6-3) ($ (6-3)!0.5!(7-2) + (45:0.1) $) (7-2) ($ (7-2)!0.5!(8-1) + (225:0.1) $) (8-1) ($ (8-1)!0.5!(9-0) + (45:0.1) $) (9-0)};
\eightbasicgrid
\path[draw, thick] plot[smooth] coordinates{(7-6) ($ (7-6)!0.5!(8-5) + (225:0.1) $) (8-5) ($ (8-5)!0.5!(9-4) + (45:0.1) $) (9-4) ($ (9-4)!0.5!(10-3) + (225:0.1) $) (10-3) ($ (10-3)!0.5!(11-2) + (45:0.1) $) (11-2)};
\path (5-4) node[below left] {$ a_2 $ (out)};
\path (9-0) node[below] {$ a_2 $};
\path (14-5) node[right] {$ b_1 $};
\path (7-12) node[above] {$ a_1 $};
\path (4-9) node[left] {$ b_2 $};
\path (7-6) node[left] {$ a_3 $};
\path (11-2) node[below right] {$ a_3 $};
\path[draw, decorate, decoration={brace, amplitude=10pt, raise=5pt, mirror, pre=moveto, post=moveto, pre length=7pt, post length=7pt}] (5-4) to node[midway, shift={(225:0.7)}] {$ k $} (9-0);
\path[draw, decorate, decoration={brace, amplitude=10pt, raise=5pt, mirror, pre=moveto, post=moveto, pre length=7pt, post length=7pt}] (14-5) to node[midway, shift={(45:1)}] {$ k+k' $} (7-12);
\path[draw, decorate, decoration={brace, amplitude=10pt, raise=9pt, mirror, pre=moveto, post=moveto, pre length=7pt, post length=7pt}] (9-0) to node[midway, shift={(315:0.9)}] {$ l-l' $} (11-2);
\path[draw, decorate, decoration={brace, amplitude=10pt, raise=5pt, mirror, pre=moveto, post=moveto, pre length=7pt, post length=7pt}] (7-12) to node[midway, shift={(135:0.7)}] {$ l' $} (4-9);
\path (8-7) -- (10-7) node[midway] {\Large $ ∂_{a_3} W_q $};
\path (8-3) node {\Large coeff};
\path[use as bounding box] (0, -0.7) -- (8, 6.5);
\end{tikzpicture}
\caption{The origin of 3A}
\label{fig:examples-torus-cancellation3A}
\end{subfigure}
\caption{Term 1A cancels with term 3A if $ l ≥ l' + 1 $}
\label{fig:examples-torus-cancellation}
\end{figure}
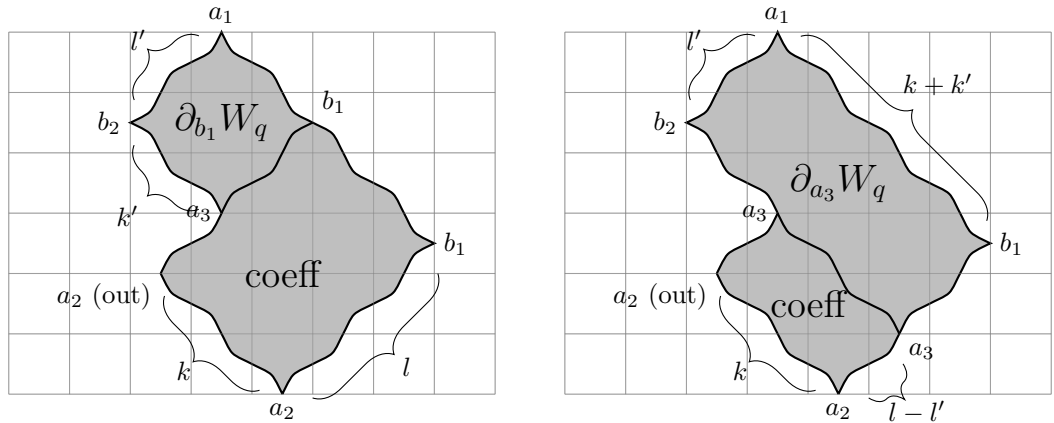